\documentclass[11pt]{article}
\usepackage[margin=1in]{geometry}
 
\usepackage{listings}
\usepackage{geometry}

\usepackage{amsmath,amsthm,amsfonts,amssymb,amscd}
\usepackage[colorlinks=true, pdfstartview=FitV, linkcolor=blue,citecolor=blue, urlcolor=blue]{hyperref}
\usepackage[usenames,dvipsnames]{color}
\usepackage{times}

\usepackage{mathtools}
\usepackage{mathabx}
\usepackage{float}
\usepackage{makeidx}
\usepackage{stmaryrd}
\usepackage{mathrsfs}
\usepackage{emptypage}
\usepackage{xfrac}
\usepackage{bbm}
\usepackage[doi=false,isbn=false,url=false,eprint=false,maxbibnames=99]{biblatex}
\renewcommand{\epsilon}{\varepsilon}

\usepackage{graphicx}

\definecolor{labelkey}{rgb}{0,0,1}

\newcommand{\pa}{\partial}
\newcommand{\g}{\gamma}
\newcommand{\al}{\alpha}

\newcommand*{\supp}{\ensuremath{\mathrm{supp\,}}}

\renewcommand*{\tilde}{\widetilde}
\renewcommand*{\hat}{\widehat}
\renewcommand*{\bar}{\overline}

\newcommand{\ep}{\epsilon}

\newtheorem{theorem}{Theorem}[section]
\newtheorem{lemma}[theorem]{Lemma}

\newtheorem{corollary}[theorem]{Corollary}
\theoremstyle{definition}

\newtheorem{rem}[theorem]{Remark}
\numberwithin{equation}{section}

\def\f1r{{\frac{1}{r}}  }
\def\f1r{{\frac{1}{r}}  }

\allowdisplaybreaks[4]
\title{The regularity of the solutions to the Muskat equation: the degenerate regularity near the turnover points}
\author{Jia Shi}
\date{}
\begin{document}

\maketitle
\begin{abstract}
  In this paper, we prove that if a solution to the Muskat problem with different densities and the same viscosity is 
 sufficiently smooth, the solution is analytic in a region that degenerates at the turnover points, 
 provided some additional conditions are satisfied. This paper studies the analyticity of the solution near turnover points, complementing the result in \cite{muskatregulatiryJia}.
\end{abstract}
\tableofcontents
\section{Introduction}
The Muskat problem is a free boundary problem studying the interface between fluids in the porous media \cite{MuskatPhysics}. It can also describe the Hele-Shaw cell \cite{shaw1898motion}. The density function $\rho$ follows the active scalar equation:
\begin{equation}\label{physicalequation}
 \frac{d\rho}{dt}+(v\cdot \nabla)\rho=0,   
\end{equation}
with 
\begin{align*}
\rho(x,t)=\left\{\begin{array}{ccc}
\rho_1 && x\in D_1(t),\\
\rho_2 && x\in D_2(t). \end{array}\right.
\end{align*}
Here $D_1(t)$ and $D_2(t)$ are open domains with $D_1(t)\cup D_2(t) \cup \partial D_{1}(t)=\mathbb{R}^2$. The velocity field $v$ in \eqref{physicalequation} satisfies Darcy's law:
\begin{equation}\label{Darcylaw}
    \frac{\mu}{\kappa}v=-\nabla p -(0,g\rho),
\end{equation}
and the incompressibility condition: 
\[
\nabla \cdot v=0,
\]
where $p$ is the pressure and $\mu$ is the viscosity.  $\kappa$, $g$ are the permeability constant and the gravity force.

We focus on the problem where two fluids have different densities $\rho_1, \rho_2$ and the same viscosity $\mu$. 

After scaling and a suitable choice of parametrization, the equation for the boundary $\partial D_{1}(t)$  in the periodic setting reads:

\begin{equation}\label{muskat equation0}
\frac{\partial f_i}{\partial t}(\alpha,t)=\frac{\rho_2-\rho_1}{2}\int_{-\pi}^{\pi} \frac{\sin(f_1(\alpha)-f_1(\alpha-\beta))(\partial_{\alpha}f_i(\alpha)-\partial_{\alpha}f_i(\alpha-\beta))}{\cosh(f_2(\alpha)-f_2(\alpha-\beta))-\cos(f_{1}(\alpha)-f_{1}(\alpha-\beta))}d\beta,
\end{equation}
for $i=1,2$ (See \cite{Castro-Cordoba-Fefferman-Gancedo-LopezFernandez:rayleigh-taylor-breakdown}). Here $f(\alpha,t)=(f_1(\alpha,t),f_2(\alpha,t))$ is a parameterization of the boundary curve. $f(\alpha,t)-(\alpha,0)$ is periodic in $\alpha$. 

 Given an initial interface at time 0, \eqref{muskat equation0} is divided into three regimes. When the interface is a graph and the heavier fluid is on the bottom as in Figure \ref{fig_turning over1}a, it is in a stable regime. When heavier fluid is above the boundary as in Figure \ref{fig_turning over1}b, it is in a stable regime when time flows backward. Thus, given any initial data, \eqref{muskat equation0} can be solved for small negative time $t$. In both regimes, shown in Figures \ref{fig_turning over1}a, \ref{fig_turning over1}b, \eqref{muskat equation0} can not be solved in the wrong direction unless the initial interface is real analytic. %(\cite{CordobaGancedocontourdynamicsmuskat},\cite{Castro-Cordoba-Fefferman-Gancedo-LopezFernandez:rayleigh-taylor-breakdown} ). 
 The third regime, shown in Figure \ref{fig_turning over1}c, it is highly unstable because the heavier fluid lies on top near point $S_1$ while the lighter fluid lies on top near point $S_2$. Note two turnover points $T_1$ and $T_2$ where the interface has a vertical tangent. For generic initial data in the turnover regime, \eqref{muskat equation0} has no solutions either as time flows forward or backward. %is an unstable regime with turnover points in the boundary.
\begin{figure}[h!]
\begin{center}
\includegraphics[scale=0.8]{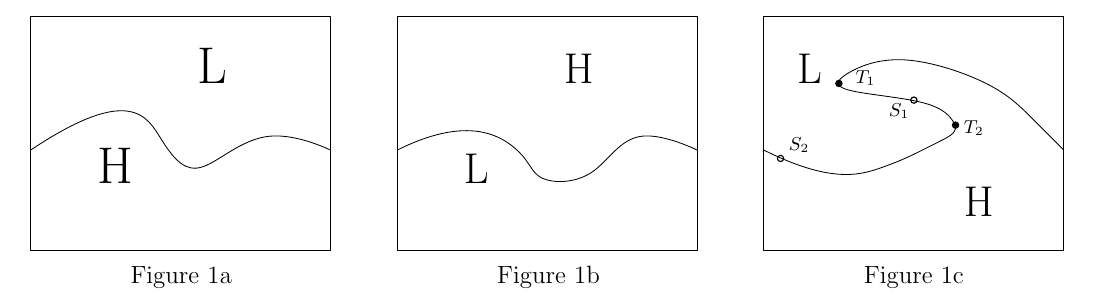}
\caption{\label{fig_turning over1}Three regimes of the Muskat equation.}
\end{center}
\end{figure}

In the third regime, there are several examples from the literature (eg. \cite{Castro-Cordoba-Fefferman-Gancedo-LopezFernandez:rayleigh-taylor-breakdown}, \cite{Castro-Cordoba-Fefferman-Gancede}, \cite{Cordoba-GomezSerrano-Zlatos:stability-shifting-muskat}, \cite{Cordoba-GomezSerrano-Zlatos:stability-shifting-muskat-II}), but they are all real analytic solutions.  Without the real analytic assumption, due to the spatially non-consistent parabolic behavior, the existence is usually false and the uniqueness is unknown. It is natural to ask: In which class of initial data in the turnover  
regime do Muskat solutions exist in an open time interval about $t=0$?  
In which class of initial data is the solution unique?

It was proven in \cite{Castro-Cordoba-Fefferman-Gancedo-LopezFernandez:rayleigh-taylor-breakdown} that solutions exist if the initial contour is  
real analytic. Conversely, our earlier paper \cite{muskatregulatiryJia} shows that a sufficiently smooth 
solution can exist only if the initial contour is real analytic except 
possibly at the turnover points. Thus it remains to understand how the  
initial contour behaves at the turnover points.

In this paper we begin to address this issue. We prove under a generic  
assumption that a sufficiently smooth solution exists only if the initial contour $f(\alpha,  
0)$ continues analytically to a region in the complex $\alpha$-plane that  
degenerates at a specific rate at values of $\alpha$ corresponding to the  
turnover points.

Moreover, for the analytic solutions, one can prove an energy estimate on an analyticity region that shrinks when time increases (\cite{Castro-Cordoba-Fefferman-Gancedo-LopezFernandez:rayleigh-taylor-breakdown}). That energy estimate implies uniqueness in the class of analytic solutions. Therefore,  the investigation towards analyticity can serve as a first step to deal with the uniqueness.

Let us explain this result in more details.
In \cite{muskatregulatiryJia}, we introduce a new way to prove that any sufficiently smooth solution is analytic except at the turnover points. We get
 \begin{theorem}\label{thm1}\cite{muskatregulatiryJia}
 Let $f(\alpha,t)=(f_1(\alpha,t),f_2(\alpha,t))\in C^{1}([-t_0,t_0],(H^{6}[-\pi,\pi])^2)$ be a solution of the Muskat equation \eqref{muskat equation0} satisfying the arc-chord condition. If $\partial_{\alpha}f_1(\alpha_0,t)\neq 0$, and $-t_0< t<t_0$, then $f(\cdot,t)$ is analytic at $\alpha_0$.
 \end{theorem}
%\color{blue}
%add something here, discuss the muskat equation
%\color{black}

In this paper, we focus on the degenerate analyticity near the turnover points. The existence and uniqueness are crucially related to the way the real-analyticity degenerates at those points.
 We have the following theorem:
\begin{theorem}\label{muskatnear1}
 Let $f(\alpha,t)=(f_1(\alpha,t),f_2(\alpha,t))\in C^{1}([-t_0,t_0], (C^{100}[-\pi,\pi])^2)$ be a solution of the Muskat equation with two turnover points satisfying the arc-chord condition \eqref{arcchord00}. $Z_1(t)$, $Z_2(t)$ are values of $\alpha$ of these two turnover points. If we assume the solution satisfies the following three conditions:
\begin{align}\label{extracondition1}
\partial_{\alpha}^2f_1(Z_1(t),t)\neq 0,
\end{align}
\begin{align}\label{extracondition2}
\partial_{\alpha}f_1(\alpha,t) \neq 0 \text{  except at $Z_1(t)$, $Z_2(t)$}, 
\end{align}
and
\begin{align}\label{extracondition3}
(\frac{d Z_1}{dt}(t)+\frac{\rho_2-\rho_1}{2}p.v.\int_{-\pi}^{\pi}\frac{\sin(f_1(\alpha)-f_1(\alpha-\beta))}{\cosh(f_2(\alpha)-f_2(\alpha-\beta))-\cos(f_1(\alpha)-f_1(\alpha-\beta))}d\beta) \neq 0,
\end{align}
then when $-t_0<t<t_0$, $f(\cdot,t)$ can be analytically extended to region 
\begin{equation}\label{Omegatdefi}
\Omega(t)=\{x+iy|-\epsilon_1(t)+Z_{1}(t)\leq x\leq Z_{1}(t)+\epsilon_1(t),|y|\leq \epsilon_2(t)(x-Z_{1}(t))^2,
\end{equation}
with $\epsilon_1(t)>0$, $\epsilon_2(t)>0$.
\end{theorem}
\begin{rem}
 For a generic solution with turnover points, the quantity \eqref{extracondition3} is not 0. 
 \end{rem}
 \begin{rem}
The space $ C^{1}([-t_0,t_0], (C^{100}[-\pi,\pi])^2)$ is chosen sufficiently smooth for the sake of convenience and is very loose.
 \end{rem}
 \begin{figure}[h!]
\begin{center}
\includegraphics[scale=1.2]{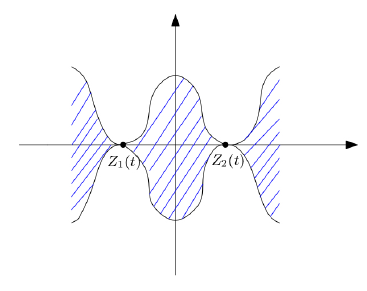}
\caption{\label{Degenerate}Degenerate analyticity region}
\end{center}
\end{figure} 
\subsection{Background}
In order to make the equation well-defined, the arc-chord condition is introduced, saying
\begin{equation}\label{arcchord00}
F(f)=|\frac{\beta^2}{\cosh(f_2(\alpha)-f_2(\alpha-\beta))-\cos(f_{1}(\alpha)-f_{1}(\alpha-\beta))}|
\end{equation}
is in $L^{\infty}$.

The Rayleigh-Taylor coefficient $\sigma$ is used to characterize the three regimes in Figure \ref{fig_turning over1} and is defined as 
\begin{equation}
\sigma=\frac{\rho_2-\rho_1}{2}\frac{\partial_{\alpha}f_{1}(\alpha,t)}{(\partial_{\alpha}f_1(\alpha,t))^2+(\partial_{\alpha}f_2(\alpha,t))^2}.
\end{equation}
$\sigma\geq 0$ is corresponding to the stable regime and $\sigma\leq 0$ the backward stable regime. When $\sigma$ changes sign, it is in the unstable regime.

In the stable regime (heavier liquid is below the lighter liquid), local well-posedness and the global well-posedness with constraints on the initial data have been widely studied, with the lowest Sobolev space $H^{\frac{3}{2}}$. 
(\cite{Yifahuaimuskatlocalexistence},\cite{YI2003442},\cite{Ambrose2004WellposednessOT},\cite{MichaelRusselHowisonmuskat},\cite{CordobaGancedocontourdynamicsmuskat},\cite{CordobaGancedomuskatmaximum},\cite{Constantin2013OnTG},\cite{CordobaCordobaGancedomuskatexistence},\cite{CHENG201632},\cite{constantincordobagancedostrainmuskatexistence},\cite{CONSTANTIN20171041},\cite{stephenmuskatexistence},\cite{Deng2016OnTT},\cite{DiegoOmarmuskatexistence},\cite{ Matioc:local-existence-muskat-hs},\cite{NguyenBenoitparadifferentialmuskat},\cite{alazardomarparalinearizationmuskat},\cite{abels_matioc_2022},\cite{muskatc12021chenquocxu},\cite{NGUYEN2022108122},\cite{alazardthomascritialspacemuskat},\cite{Alazard-Nguyennonlipshitzmuskat}, \cite{Alazard2020EndpointST}). Examples of self-similar solutions have been found in \cite{2021selfsimilar} and the dynamics of solutions having initial data with multiple small corners have been proved \cite{2023desingularizationGGHP}.   Interested readers can see \cite{2021selfsimilar} and \cite{muskatc12021chenquocxu} for detailed reviews. Due to the parabolic behavior, instant analyticity has been proved in the stable regime. Castro--Córdoba--Fefferman--Gancedo--López-Fernández \cite{Castro-Cordoba-Fefferman-Gancedo-LopezFernandez:rayleigh-taylor-breakdown} proved the $H^4$ solutions become instantly analytic if the solutions remain to be in the stable region for a short time. In \cite{Matioc:local-existence-muskat-hs}, also in the stable region,  Matioc improved the instant analyticity to $H^{s}$, where  $s\in (\frac{3}{2},3)$. In \cite{GANCEDO2019552}, 
Gancedo--García-Juárez--Patel--Strain showed that in the stable regime, a medium size initial data in $\dot{\mathcal{F}^{1,1}}\cap L^2$ with $\|f\|_{\mathcal{F}^{1,1}}=\int|\zeta||\hat{f}(\zeta)|d\zeta$ becomes instantly analytic. Their result also covers the different viscosities case and the 3D case. 

When the heavier liquid is above the lighter liquid,  the equation is ill-posed when time flows forward \cite{CordobaGancedocontourdynamicsmuskat}.

A solution that starts from a stable regime and develops turnover points was first discovered in \cite{Castro-Cordoba-Fefferman-Gancedo-LopezFernandez:rayleigh-taylor-breakdown}. That solution still exists for a short time after turnover due to the analyticity when the turnover happens. Moreover, breakdown of smoothness can happen \cite{Castro-Cordoba-Fefferman-Gancede}. There are also examples where the solutions transform from stable to unstable and go back to stable \cite{Cordoba-GomezSerrano-Zlatos:stability-shifting-muskat} and vice versa \cite{Cordoba-GomezSerrano-Zlatos:stability-shifting-muskat-II} .

Weak solutions and a special kind of weak solutions: mixing solutions of \eqref{physicalequation} have also been studied. They do not satisfy \eqref{muskat equation0} and can develop a mixing zone. Weak solutions do not have uniqueness \cite{cordoba2011lack}. In all three regimes, there are infinitely many mixing solutions (\cite{szekelyhidi2012relaxation},\cite{castro2016mixing},\cite{forster2018piecewise},\cite{castro2019degraded},\cite{noisette2021mixing},\cite{castro2022localized}).

\subsection{The outline of the proof of Theorem \ref{thm1}}
Without loss of generality, we assume 
\begin{equation}\label{assumption0}
\frac{\rho_2-\rho_1}{2}=1,
\end{equation}
and for all $t\in(-t_0,t_0)$,
\begin{align}\label{assumption}
&\eqref{extracondition2}, \eqref{extracondition3}, \text{ holds, and} \quad
\partial_{\alpha}^{2}f_1(Z_1(t),t)>0.
\end{align}
If $\partial_{\alpha}^{2}f_1(Z_1(t),t)<0$, we could change $t$ to $-t$ and the same proof holds.

It is enough to show the following theorem:
\begin{theorem}\label{muskathalf}
Under the assumptions in theorem \ref{muskatnear1}, and \eqref{extracondition2}, for $t_*\in(-t_0,t_0)$, there exists $\epsilon$ sufficiently small such that when $t_*<t\leq t_*+\epsilon$, the solution $f(\alpha,t)$ can be analytically extended to $\Omega\cap\{\alpha|\alpha>Z_1(t)\}$, with $\Omega$ from \eqref{Omegatdefi}. Additionally when $t_{*}-\epsilon\leq t<t_*,$  $f(\alpha,t)$ can analytically extended to $\Omega\cap\{\alpha|\alpha<Z_1(t)\}$. $\epsilon$ can be chosen uniformly within a compact set in $(-t_0,t_0)$.
\end{theorem}
Without loss of generality, we will choose $t_*=0$ and only show the proof for $t>0$. The other part can be done in the same way. 

     We use a similar structure as in \cite{muskatregulatiryJia}: 
      we make a $C^1$ continuation of the parametrized interface $\alpha\to(f_1(\alpha,t),f_2(\alpha,t))$ to complex $\alpha$ and then prove the $C^1$ continuation satisfies the Cauchy-Riemann equation. To do so, we break the complex region into curves $\alpha+ic(\alpha)\gamma t$ with $\gamma \in [-1,1]$. On each such curve, we solve an equation for $(f_1,f_2).$ We then show when $\gamma$ varies, that our solutions on the curve fit together into an $C^1$ function of $\alpha+i\beta$. Finally, we prove that $C^1$ function satisfies the Cauchy-Riemann equation, thus producing the desired analytic continuation.

       Because the Rayleigh-Taylor coefficient changes sign across turnover point $Z_1(t)$, the equation only behaves well on half of the real axis, when $t$ flows in one direction. In order to get the existence of the extended solution on curve $\alpha+ic(\alpha)\gamma t$, our localization must cut the equation into right/left parts. Since the equation is non-local, we need to find an appropriate way to cut the equation and deal with extra high-order terms and non-smoothness caused by the cut boundary.

         More precisely, if turnover happens at $\alpha=0$,  a linear model for the behavior near the turnover point is 
          \begin{align}\label{new linear model}
          \frac{dg(\alpha,t)}{dt}=-\lambda(\alpha)1_{\alpha\geq 0}\alpha (-\Delta)^{\frac{1}{2}}g(\alpha,t) + \eta \partial_{\alpha}g(\alpha,t),
          \end{align}
          with $\lambda $ is a smooth cut-off near 0. 
      Compared to a linear model  for the behavior far away from turnover points:
          \begin{align}\label{original linear model}
          \frac{dg}{dt}(\alpha,t)=-\lambda(\alpha-1)\alpha(-\Delta)^{\frac{1}{2}}g(\alpha,t)
          \end{align}
   the coefficient function $-\lambda(\alpha)1_{\alpha\geq 0}\alpha$ before $(-\Delta)^{\frac{1}{2}}$ is non-smooth and the effect of transport term is not negligible.

        In order to avoid the corner-singularity of the coefficient, we only study the solution when $\alpha>0$. This causes problems in the energy estimate. For the $L^2$ estimate, we can not use the G\r{a}rding's inequality that we cited:\cite[Section 2.4]{Castro-Cordoba-Fefferman-Gancede} because $\lambda(\alpha)1_{\alpha\geq 0}$ is not in $C^2$. This lemma is done in \ref{Garding}, showing the high-order term caused by the corner-singularity can be absorbed by the dominant term in the G\r{a}rding's inequality.
        
         The higher-order norm estimate does not go through because the non-local integration operator $(-\Delta)^{\frac{1}{2}}$. It can be written as a composition of Hilbert transform $H$ and derivative. Hilbert transform is a bounded operator on $H^{k}(-\infty,\infty)$ for any $k$, but after the cut $\tilde{H}(g)=p.v. \frac{1}{\pi}\int_{0}^{\infty}\frac{g(\beta)}{\beta-\alpha}d\beta$ is no longer a bounded operator on $H^{k}(0,\infty)$ when $k\geq 1$. In order to fix this, we modify the equation for solutions on curve:$\alpha+ic(\alpha)\gamma t$. We only keep the integration of the extended solution ($\gamma\neq 0$) on the symmetric integral area $[0,2\alpha]$. On the integration area $[0,2\alpha]^c$, we modify the integration to be dependent only on the non-extended solution ($\gamma= 0$).  We derive our modification equation in section \ref{modifiedequationsection} and the well-posedness is shown in section \ref{kappa1sectionexistence}, \ref{kappa2sectionexistence}.

            The transport term can also be added in the model \eqref{original linear model} but is harmless because of the integration by parts on $(-\infty,\infty)$.
      For the new model \eqref{new linear model}, since $0$ is the critical point separating the backward-parabolic region and the parabolic region, even small transportation is very crucial. Technically, this leads to the boundary term at $0$ after integration by parts.  Depending on the sign of $\eta$, we treat the two cases separately. The rough ideas are the following: for a transport equation on $H^{1}_{\alpha}[0,\infty)$, 
         \[
         \frac{dg}{dt}=\eta\partial_{\alpha}g(\alpha,t),
         \]
      when $\eta>0$, if we let time go forward and study the behavior in the region $\alpha>0$, the energy estimate is good:
       
       \[
       \Re<g,\frac{dg}{dt}>_{H^{1}_{\alpha}[0,\infty)}=-\eta g^2(0,t)\leq 0,
       \]
       because the transportation direction is moving out of the half region. The energy estimate for this case is done in section \ref{kappa1sectionexistence}. Additionally, when $\eta<0$, if an initial value is defined on $(-\infty,\infty)$, with support in $[0,\infty)$, then when $t>0$, the solution always vanishes near 0. Thus $\Re<g,\frac{dg}{dt}>_{H^{1}_{\alpha}[0,\infty)}=0$ for $t>0$. In order to apply this good boundary behavior, we introduce a translation $\tau(t)$ \eqref{defkappa} while we derive the modified equation in section \ref{modifiedequationsection}. The corresponding boundary behavior is shown in section \ref{kappa2sectiongene} and the existence in section \ref{kappa2sectionexistence}. This sign of $\eta$ in the linear model \eqref{new linear model} is connected to the sign of the function in condition \eqref{extracondition3}.

          Given the existence of the $C^1$ extension of the solution on the curve, we verify the extension satisfies the Cauchy- Riemann equation in section \ref{analyticity+section}, \ref{analyticity-section}. This verification heavily relies on the fact that the modified equation still keeps the analyticity structure, which leads to some commuting lemmas in the Appendix that allow us to do the energy estimate.
        \subsection{Organization of the paper}
      In Section \ref{modifiedequationsection}, we use a parameterization to fix the turnover point $Z_1(t)$ be $0$. Then we modify the equation and derive the equation for the solution that is analytically extended to $\Omega\cap\{\alpha|\alpha>0\}$. The modified equation is the same for both cases ($\eta>0$ or $\eta<0$ in the corresponding linear model \eqref{new linear model}) except for a translation from $\alpha$ to $\alpha+\tau(t).$
        
     In Section \ref{kappa1sectiongene}, section \ref{kappa2sectiongene}, we show that the modified equation satisfies two generalized equations depending on the sign of $\eta$.
         
In section \ref{kappa1sectionexistence},\ref{kappa2sectionexistence} we show the well-posedness and uniqueness of those generalized equations. 

In section \ref{analyticity+section}, \ref{analyticity-section}, we prove our extended solution satisfies the Cauchy-Riemann equation. 
\section{Notation}
\color{black}

We will use the following notations:

$Z_1(t)$, $Z_2(t)$: values of $\alpha$ of the turnover points.

$\delta$: a sufficiently small number.

$\delta_c$: a sufficiently small number depending on $\delta$.

$\lambda_0(\alpha)$: $\lambda(\alpha)\geq 0$ and in $C^{100}(-\infty,\infty)$, satisfying
\begin{align*}
    \lambda_0(\alpha)=\begin{cases}1&|\alpha|\leq 10\delta,\\0&|\alpha|\geq 20\delta.
\end{cases}
\end{align*}

$\lambda(\alpha)$:$\lambda(\alpha)=\lambda_0(\frac{\alpha}{10})$.

$c(\alpha)$:
\begin{align*}
\Bigg\{\begin{array}{cc}
          c(\alpha)=\delta_c \alpha^2,  \text{ when }0\leq \alpha \leq \frac{\delta}{32},\\
          \supp c(\alpha) \subset [0,\frac{\delta}{8}],\\
          c(\alpha)\geq 0,\  c(\alpha) \in C^{100}[0,\infty)\cap C^{1,1}(-\infty,\infty),\  \|c(\alpha)\|_{C^{100}[0,\infty)\cap C^{1,1}(-\infty,\infty)}\leq \delta.
    \end{array}
    \end{align*}

$x(\alpha,t)$: a new variable such that $-\frac{\pi}{2}$, $0$ be two fixed turnover points.

$f(\alpha,t)=(f_1(\alpha,t),f_2(\alpha,t))$: the original solution of the Muskat equation.

 $\tilde{f}(\alpha,t)$:

\[
\tilde{f}(\alpha,t)=f(x(\alpha,t),t).
\]

$\tilde{f}^{+}(\alpha,t)$:

\[
\tilde{f}^{+}(\alpha,t)=(\tilde{f}(\alpha,t)-\sum_{i=0}^{60}\frac{\tilde{f}^{<i>}(0)}{i!}\alpha^i)\lambda(\alpha)1_{\alpha\geq 0}.
\]

$\tilde{f}^{L}(\alpha,t)$:
\[
\tilde{f}^{L}(\alpha,t)=\tilde{f}(\alpha,t)-\tilde{f}^{+}(\alpha,t).
\]
 
 $K(\cdot)$:

\[
K(f(\alpha,t)-f(\beta,t))=\frac{\sin(f_1(\alpha,t)-f_1(\beta,t))}{\cosh(f_2(\alpha,t)-f_2(\beta,t))-\cos(f_1(\alpha,t)-f_1(\beta,t))}.
\]

$X_i(\alpha,t)$: $X_i(\alpha,t)=(\partial_{\alpha}^{l_i}\tilde{f}^{L}(\alpha,t))^{\phi_{i}}(\partial_{\alpha}^{l_i'}(\frac{1}{\frac{dx}{d\alpha}(\alpha,t)}))^{\phi_{i}'}(\partial_{\alpha}^{l_i''}\frac{d}{d\alpha}x(\alpha,t))^{\phi_{i''}},
$
with $l_i,l_i',l_1''\leq 61$, $\phi_i,\phi_{i}',\phi_{i''} \in \{0,1\}$.

$V_X(\alpha,t)$: $V_X(\alpha,t)$ is the vector with components of all $X_i(\alpha,t)$.

$ \tilde{X}_i(\alpha,t)$:
\begin{equation*}
  \tilde{X}_i(\alpha,t)=\partial_{\alpha}^{b_i}(\frac{\frac{dx(\alpha,t)}{dt}}{\frac{dx(\alpha,t)}{d\alpha}}), 
\end{equation*}
with $b_i\leq 60$.

$D_{A}$:
$D_{A}=\{\alpha+iy|0<\alpha<\frac{\delta}{4}, |y|<c(\alpha) t\}$.

$t_0$: the original solution exists when $t\in (-t_0,t_0)$.

$\kappa(t)$:
\[
\kappa(t)=(\frac{\partial_{t}x}{\partial_{\alpha}x}(0,t)+\frac{1}{(\partial_{\alpha}x)(0,t)}\int_{-\pi}^{\pi}K(\tilde{f}^{+}(0,t)-\tilde{f}^{+}(\beta,t)+\tilde{f}^{L}(0,t)-\tilde{f}^{L}(\beta,t))(\partial_{\beta}x)(\beta,t)d\beta).
\]

 $\tau(t)$:
\begin{align*}
\tau(t)=\begin{cases}
0, & \text{when }\ \kappa(0)>0,\\
-\int_{0}^{t}\kappa(\zeta)d\zeta, & \text{when }\ \kappa(0)<0.
\end{cases}
\end{align*}

$D^{-i}(h)$: For any $h\in L_{\al}^{2}[-\tau(t),\pi]$, we could define
\begin{align*}\label{D-formularnew}
&\qquad D^{-i}(\tilde{h})(\alpha,\gamma,t)\\\nonumber
&=\bigg\{\begin{array}{cc}
         \int_{-\tau(t)}^{\alpha}(1+ic'(\alpha_1+\tau(t))\gamma t) ...\int_{-\tau(t)}^{\alpha_{i-1}}(1+ic'(\alpha_i+\tau(t))\gamma t)h(\alpha_i)d\alpha_i d\alpha_{i-1}...d\alpha_1 &  -\tau(t)< \alpha \leq \pi\\\nonumber
          0 & -\pi\leq\alpha\leq -\tau(t).
    \end{array}
\end{align*}

$V^{k}$: we set
\begin{align*}
V^{k}(\alpha,\gamma,t) =(\tilde{f}^{+}_1(\alpha,t), \tilde{f}^{+}_2(\alpha,t), \partial_{\alpha}\tilde{f}^{+}_1(\alpha,t),\partial_{\alpha}\tilde{f}^{+}_2(\alpha,t), ...,\partial_{\alpha}^{k}\tilde{f}^{+}_1(\alpha,t),\partial_{\alpha}^{k}\tilde{f}^{+}_2(\alpha,t)).
\end{align*}

$V^{k}_h$: we set
\begin{align*}
V^{k}_{h}(\alpha,\gamma,t) =(D^{-60}h_1(\alpha,\gamma,t), D^{-60}h_2(\alpha,\gamma,t), ...,D^{-60+k}h_1(\alpha,\gamma,t),D^{-60+k}h_2(\alpha,\gamma,t)).
\end{align*}

$V^{k}_{\tilde{f}^{L}}$: we set
\begin{align*}
V^{k}_{\tilde{f}^{L}}(\alpha,\gamma,t) =(\tilde{f}^{L}_1(\alpha,t), \tilde{f}^{L}_2(\alpha,t), \partial_{\alpha}\tilde{f}^{L}_1(\alpha,t),\partial_{\alpha}\tilde{f}^{L}_2(\alpha,t), ...,\partial_{\alpha}^{k}\tilde{f}^{L}_1(\alpha,t),\partial_{\alpha}^{k}\tilde{f}^{L}_2(\alpha,t)).
\end{align*}

$A(h)$:
\begin{align*}
A(h)=\frac{ic(\alpha)t}{1+ic'(\alpha)\gamma t}\partial_{\alpha}h-\partial_{\gamma}h.
\end{align*}

\section{The modified equation}\label{modifiedequationsection}
Recall that under the assumption \eqref{assumption0}, we have the equation:
\begin{equation}\label{muskat equation}
\frac{\partial f_i}{\partial t}(\alpha,t)=\int_{-\pi}^{\pi} \frac{\sin(f_1(\alpha)-f_1(\alpha-\beta))(\partial_{\alpha}f_i(\alpha)-\partial_{\alpha}f_i(\alpha-\beta))}{\cosh(f_2(\alpha)-f_2(\alpha-\beta))-\cos(f_{1}(\alpha)-f_{1}(\alpha-\beta))}d\beta,
\end{equation}
for $i=1,2$. $Z_1(t)$, $Z_2(t)$ are values of $\alpha$ of these two turnover points.

Without loss of generality, we set $Z_1(0)=0$, $Z_2(0)=-\frac{\pi}{2}$. 
We now change the variable so that the turning points happen at $0$ and $-\frac{\pi}{2}$ for all the time. Let 
\begin{equation}\label{initialchangevariable}
    x(\alpha, t)=\alpha - \sin(\alpha)(Z_2(t)+\frac{\pi}{2}-Z_1(t))+Z_1(t).
\end{equation}
Due to the smallness of $Z_1(t)$ and $Z_2(t)-\frac{\pi}{2}$ when $t$ is sufficiently small, $x(\cdot,t)$ is a diffeomorphism from $\mathcal{R}\to \mathcal{R}$ and $\mathcal{T}\to\mathcal{T}$. Let 
\begin{equation}\label{tildefdefi}
\tilde{f}(\alpha,t)=f(x(\alpha,t),t).
\end{equation}
Then we have 
\begin{equation}
\frac{d\tilde{f}_{\mu}(\alpha,t)}{dt}=\frac{d\tilde{f}_{\mu}}{d\alpha}(\alpha,t)(\frac{\frac{dx(\alpha,t)}{dt}}{\frac{dx(\alpha,t)}{d\alpha}})+\int_{-\pi}^{\pi}K(\tilde{f}(\alpha,t)-\tilde{f}(\beta,t))(\frac{\partial_{\alpha}\tilde{f}_{\mu}(\alpha,t)}{\frac{dx(\alpha,t)}{d\alpha}}-\frac{\partial_{\alpha}\tilde{f}_{\mu}(\beta,t)}{\frac{dx(\beta,t)}{d\beta}})(\frac{dx(\beta,t)}{d\beta})d\beta.
\end{equation}
with 
\begin{equation}\label{kdefi}
K(x)=K((x_1,x_2))=\frac{\sin(x_1)}{\cosh(x_2)-\cos(x_1)}.
\end{equation}
\begin{equation}\label{turnoverconditionnew}
\partial_{\alpha}\tilde{f}_1(0,t)=0, \ \partial_{\alpha}\tilde{f}_1(-\frac{\pi}{2},t)=0.
\end{equation}
\subsection{Localization}\label{nearlocal}
Next we separate $\tilde{f}$ and let
\begin{equation}\label{cut function0}
    \tilde{f}=\tilde{f}^{+}+\tilde{f}^{L}
\end{equation}
Here 
\begin{equation}\label{fLequation}
\tilde{f}^{+}(\alpha,t)=(\tilde{f}(\alpha,t)-\sum_{i=0}^{99}\frac{\tilde{f}^{<i>}(0)}{i!}\alpha^i)\lambda_0(\alpha)1_{\alpha\geq 0},
\end{equation} where $\lambda_0(\alpha)\geq 0$ and in $C^{100}(-\infty,\infty)$, satisfying
\begin{align}\label{lambda0defi}
    \lambda_0(\alpha)=\begin{cases}1&|\alpha|\leq \delta,\\0&|\alpha|\geq 2\delta.
\end{cases}
\end{align}
From the separation, $\tilde{f}^{+}, \tilde{f}^{L}$ satisfy following properties :\begin{align}\label{fLspace}
&\tilde{f}^{+}, \tilde{f}^{L}(\alpha,t) \in C^{1}((-t_0,t_0), C^{99}(-\infty,\infty)),\\\nonumber
&\tilde{f}^{L}(\alpha,t) \in \{g| g \text{ can be analytically extended to the set} \quad \{\alpha+iy|0<\alpha<\frac{\delta}{2}\}\}\\\nonumber
&\tilde{f}^{+(j)}(0,t)=0 , \text{ for } j\leq 99, \supp \tilde{f}\subset[0,2\delta]
\end{align} 
\color{black}Then we only need to focus on $\tilde{f}^{+}$ and study the analyticity behavior when $\alpha>0$. \color{black} 
We have
\begin{align}\label{equation1.2}
    &\quad\frac{d\tilde{f}_{\mu}^{+}(\alpha,t)}{dt}\\\nonumber
    &=\underbrace{\frac{d\tilde{f}_{\mu}^{+}}{d\alpha}(\alpha,t)(\frac{\frac{dx(\alpha,t)}{dt}}{\frac{dx(\alpha,t)}{d\alpha}})}_{Term_1}\underbrace{-\frac{d\tilde{f}_{\mu}^{L}(\alpha,t)}{dt}+\frac{d\tilde{f}_{\mu}^{L}}{d\alpha}(\alpha,t)(\frac{\frac{dx(\alpha,t)}{dt}}{\frac{dx(\alpha,t)}{d\alpha}})}_{Term_2}\\\nonumber
    &\quad\underbrace{+\int_{-\pi}^{\pi}K(\tilde{f}^{+}(\alpha,t)-\tilde{f}^{+}(\beta,t)+\tilde{f}^{L}(\alpha,t)-\tilde{f}^{L}(\beta,t))(\frac{\partial_{\alpha}\tilde{f}_{\mu}^{+}(\alpha,t)}{\frac{dx(\alpha,t)}{d\alpha}}-\frac{\partial_{\beta}\tilde{f}_{\mu}^{+}(\beta,t)}{\frac{dx(\beta,t)}{d\beta}})(\frac{dx(\beta,t)}{d\beta})d\beta}_{Term_{3}}\\\nonumber
    &\quad\underbrace{+\int_{-\pi}^{\pi}K(\tilde{f}^{+}(\alpha,t)-\tilde{f}^{+}(\beta,t)+\tilde{f}^{L}(\alpha,t)-\tilde{f}^{L}(\beta,t))(\frac{\partial_{\alpha}\tilde{f}_{\mu}^{L}(\alpha,t)}{\frac{dx(\alpha,t)}{d\alpha}}-\frac{\partial_{\beta}\tilde{f}_{\mu}^{L}(\beta,t)}{\frac{dx(\beta,t)}{d\beta}})(\frac{dx(\beta,t)}{d\beta})d\beta}_{Term_4}\\\nonumber
    &=Term_{1}+Term_{2}+Term_{3}+\\\nonumber
    &\quad\underbrace{+\int_{-\pi}^{\pi}\int_{0}^{1}\nabla K(\zeta\tilde{f}^{+}(\alpha,t)-\zeta\tilde{f}^{+}(\beta,t)+\tilde{f}^{L}(\alpha,t)-\tilde{f}^{L}(\beta,t))\cdot (\tilde{f}^{+}(\alpha,t)-\tilde{f}^{+}(\beta,t))}_{Term_{4,1}}\\\nonumber
    &\qquad\underbrace{(\frac{\partial_{\alpha}\tilde{f}_{\mu}^{L}(\alpha,t)}{\frac{dx(\alpha,t)}{d\alpha}}-\frac{\partial_{\beta}\tilde{f}_{\mu}^{L}(\beta,t)}{\frac{dx(\beta,t)}{d\beta}})(\frac{dx(\beta,t)}{d\beta})d\zeta d\beta}_{Term_{4,1}}\\\nonumber
     &\quad+\underbrace{\int_{-\pi}^{\pi}K(\tilde{f}^{L}(\alpha,t)-\tilde{f}^{L}(\beta,t))(\frac{\partial_{\alpha}\tilde{f}_{\mu}^{L}(\alpha,t)}{\frac{dx(\alpha,t)}{d\alpha}}-\frac{\partial_{\beta}\tilde{f}_{\mu}^{L}(\beta,t)}{\frac{dx(\beta,t)}{d\beta}})(\frac{dx(\beta,t)}{d\beta})d\beta}_{Term_{4,2}}.
\end{align}
where we separate the last term into the part depending on $\tilde{f}^{+}$ and the part not depending on $\tilde{f}^{+}$.
\subsection{Take derivative}
In order to make the structure of the equation more linear, we take the 60th derivative and
study the equation of $\tilde{f}_{\mu}^{+(60)}(\alpha,t)$. We have
\begin{align}\label{tildefderiequation}
      &\quad\frac{d\tilde{f}_{\mu}^{+(60)}(\alpha,t)}{dt}\\\nonumber
      &=\tilde{f}_{\mu}^{+(61)}(\alpha,t)(\frac{\frac{dx(\alpha,t)}{dt}}{\frac{dx(\alpha,t)}{d\alpha}})\\\nonumber
    &\quad+\int_{-\pi}^{\pi}K(\tilde{f}^{+}(\alpha,t)-\tilde{f}^{+}(\beta,t)+\tilde{f}^{L}(\alpha,t)-\tilde{f}^{L}(\beta,t))(\frac{\tilde{f}_{\mu}^{+(61)}(\alpha,t)}{\frac{dx(\alpha,t)}{d\alpha}}-\frac{\tilde{f}_{\mu}^{+(61)}(\beta,t)}{\frac{dx(\beta,t)}{d\beta}})(\frac{dx(\beta,t)}{d\beta})d\beta\\\nonumber
    &\quad+\sum_{i}{O_0^{i}}(\alpha,t)+T_{fixed}(\alpha,t)\\\nonumber
    &=\tilde{f}_{\mu}^{+(61)}(\alpha,t)(\frac{\frac{dx(\alpha,t)}{dt}}{\frac{dx(\alpha,t)}{d\alpha}}+\frac{1}{\frac{dx(\alpha,t)}{d\alpha}}p.v.\int_{-\pi}^{\pi}K(\tilde{f}^{+}(\alpha,t)-\tilde{f}^{+}(\beta,t)+\tilde{f}^{L}(\alpha,t)-\tilde{f}^{L}(\beta,t))(\frac{dx(\beta,t)}{d\beta})d\beta)\\\nonumber
    &\quad-p.v.\int_{-\pi}^{\pi}K(\tilde{f}^{+}(\alpha,t)-\tilde{f}^{+}(\beta,t)+\tilde{f}^{L}(\alpha,t)-\tilde{f}^{L}(\beta,t))(\frac{d\tilde{f}_{\mu}^{+(60)}(\beta,t)}{d\beta})d\beta\\\nonumber
    &\quad+\sum_{i}{O_0^{i}}(\alpha,t)+T_{fixed}(\alpha,t).
\end{align}
Here $T_{fixed}(\alpha,t)$ are terms not depending on $\tilde{f}^{+}$, coming from $Term_2+Term_{4,2}$. We have
\begin{align}\label{Tfixed}
&\quad T_{fixed}(\alpha,t)\\\nonumber
&=\partial_{\alpha}^{60}(-\frac{d\tilde{f}_{\mu}^{L}(\alpha,t)}{dt}+\frac{d\tilde{f}_{\mu}^{L}}{d\alpha}(\alpha,t)(\frac{\frac{dx(\alpha,t)}{dt}}{\frac{dx(\alpha,t)}{d\alpha}})\\\nonumber&\quad+\int_{-\pi}^{\pi}K(\tilde{f}^{L}(\alpha,t)-\tilde{f}^{L}(\beta,t))(\frac{\partial_{\alpha}\tilde{f}_{\mu}^{L}(\alpha,t)}{\frac{dx(\alpha,t)}{d\alpha}}-\frac{\partial_{\beta}\tilde{f}_{\mu}^{L}(\beta,t)}{\frac{dx(\beta,t)}{d\beta}})(\frac{dx(\beta,t)}{d\beta})d\beta).
\end{align}
Because it does not depend on $\tilde{f}^{+}$. From the property \eqref{fLspace}, this term is analytic in $\{x+iy|0<x<\frac{\delta}{2}\}$. The $O_0^{i}$  are the terms with at most 60th derivative hit on $\tilde{f}^{+}$.

Before we write down the explicit form of $O^i_0$, we introduce some notations. 
Let 
\begin{align}\label{notation01}
V^{k}(\alpha,t) =(\tilde{f}^{+}_1(\alpha,t), \tilde{f}^{+}_2(\alpha,t), \partial_{\alpha}\tilde{f}^{+}_1(\alpha,t),\partial_{\alpha}\tilde{f}^{+}_2(\alpha,t), ...,\partial_{\alpha}^{k}\tilde{f}^{+}_1(\alpha,t),\partial_{\alpha}^{k}\tilde{f}^{+}_2(\alpha,t)).
\end{align}
\begin{align}\label{notation02}
V^{k}_{\tilde{f}^{L}}(\alpha,t) =(\tilde{f}^{L}_1(\alpha,t), \tilde{f}^{L}_2(\alpha,t), \partial_{\alpha}\tilde{f}^{L}_1(\alpha,t),\partial_{\alpha}\tilde{f}^{L}_2(\alpha,t), ...,\partial_{\alpha}^{k}\tilde{f}^{L}_1(\alpha,t),\partial_{\alpha}^{k}\tilde{f}^{L}_2(\alpha,t)).
\end{align}
\begin{align}\label{notation03}
V_{X}(\alpha,t) =\{X_i(\alpha,t)\} .
\end{align}
When we write $X_i(\alpha,t)$, we mean 
\begin{equation}\label{Xifunction}
X_i(\alpha,t)=(\partial_{\alpha}^{l_i}\tilde{f}^{L}(\alpha,t))^{\phi_{i}}(\partial_{\alpha}^{l_i'}(\frac{1}{\frac{dx}{d\alpha}(\alpha,t)}))^{\phi_{i}'}(\partial_{\alpha}^{l_i''}\frac{d}{d\alpha}x(\alpha,t))^{\phi_{i''}},
\end{equation}
with $l_i,l_i',l_1''\leq 61$, $\phi_i,\phi_{i}',\phi_{i''} \in \{0,1\}$, $x(\al,t)$ from \eqref{initialchangevariable}. $V_X(\alpha,t)$ is the vector with components of all $X_i(\alpha,t)$. 

When we write $\tilde{X}_i(\alpha,t)$, we mean 
\begin{equation}
   \tilde{X}_i(\alpha,t)=\partial_{\alpha}^{b_i}(\frac{\frac{dx(\alpha,t)}{dt}}{\frac{dx(\alpha,t)}{d\alpha}}), 
\end{equation}
with $b_i\leq 61$.

\color{black}
 A function $K_{-\sigma}^{j}(A,B,C)$ is of $-\sigma$ type if for $A$, $B$, $C$ in $R^n$, it has the form
\begin{align}\label{k-sigma}
K_{-\sigma}^{j}(A,B,C)=&c_j\frac{\sin(A_1+B_1)^{m_1}\cos(A_1+B_1)^{m_2}}{(\cosh(A_2+B_2)-\cos(A_1+B_1))^{m_0}}\\\nonumber
    &\times(\sinh(A_2+B_2))^{m_3}(\cosh(A_2+B_2))^{m_4}\Pi_{j=1}^{m_5}(A_{\lambda_{j}})\Pi_{j=1}^{m_6}(B_{\lambda_{j,2}})\Pi_{j=1}^{m_7}(C_{\lambda_{j,3}}),
\end{align}
with $m_1+m_3+m_5+m_6+m_7-2m_0\geq -\sigma$. $c_j$ is a constant.

 A function $K_{-\sigma,\zeta}^{j}(A,B,C)$ is of $-\sigma$ type if for $A$, $B$, $C$ in $R^n$, it has the form
\begin{align}\label{k-sigmazeta}
K_{-\sigma,\zeta}^{j}(A,B,C)=&c_j\frac{\sin(\zeta( A_1)+B_1)^{m_1}\cos(\zeta(A_1)+B_1)^{m_2}}{(\cosh(\zeta(A_2)+B_2)-\cos(\zeta(A_1)+B_1)^{m_0}}\\\nonumber
    &\times(\sinh(\zeta(A_2)+B_2))^{m_3}(\cosh(\zeta(A_2)+B_2))^{m_4}\Pi_{j=1}^{m_5}(A_{\lambda_{j}})\Pi_{j=1}^{m_6}(B_{\lambda_{j,2}})\Pi_{j=1}^{m_7}(C_{\lambda_{j,3}})\zeta^{m_8},
\end{align}
with $m_1+m_3+m_5+m_6+m_7-2m_0\geq -\sigma$, $m_8\geq 0$. $c_j$ is a constant.

By using the above notations, we have 
\begin{align}\label{Xispace}
&X_i(\alpha,t), \tilde{X}_i(\alpha,t),  \in C^{1}([0,t_0], C^{38}(-\infty,\infty)),\  \tilde{f}^{L}(\alpha,t) \in C^{1}([0,t_0], C^{99}(-\infty,\infty)),\\\nonumber
& X_i(\alpha,t), \tilde{X}_i(\alpha,t), \tilde{f}^{L}(\alpha,t) \in \{g| g \text{ can be analytically extended to the set}  \{\alpha+iy|0<\alpha<\frac{\delta}{2}\}\}.
\end{align}
\color{black} We claim that $O_{0}^{i}$ can be written as following five types, by separating the highest order term in the derivative. Since we take 60th derivatives, only one term can be hit by more than 31th derivatives. When we write $\sum_{i}O_0^{i}$, we always mean $\sum_{1\leq i'\leq 5}\sum_{i'}O_0^{1,i}$.

Here we have:
\begin{align*}
O_0^{1,i}(\alpha,t)=&\int_{-\pi}^{\pi}K_{-1}^{i}(V^{30}(\alpha,t)-V^{30}(\beta,t),V^{61}_{\tilde{f}^{L}}(\alpha,t)-V^{61}_{\tilde{f}^{L}}(\beta,t),V_{X}(\alpha,t)-V_{X}(\beta,t))\\
&\cdot X_{i'}(\beta,t)(\partial_{\alpha}^{b_{i}}\tilde{f}_{v}^{+}(\alpha,t)-\partial_{\alpha}^{b_{i}}\tilde{f}_{v}^{+}(\beta,t))d\beta.
\end{align*}
\begin{align*}
O_{0}^{2,i}(\alpha,t)=&\int_{-\pi}^{\pi}K_{-1}^{i}(V^{30}(\alpha,t)-V^{30}(\beta,t),V^{61}_{\tilde{f}^{L}}(\alpha,t)-V^{61}_{\tilde{f}^{L}}(\beta,t),V_{X}(\alpha,t)-V_{X}(\beta,t))X_{i'}(\beta,t){\partial_{\alpha}^{s_i}}\tilde{f}_{\mu}^{+}(\alpha,t)\\
&(\partial_{\alpha}^{b_{i}}\tilde{f}_{v}^{+}(\alpha,t)-\partial_{\alpha}^{b_{i}}\tilde{f}_{v}^{+}(\beta,t))d\beta.
\end{align*}
\begin{align*}
O_0^{3,i}(\alpha,t)=\int_{-\pi}^{\pi}K_{0}^{i}(V^{30}(\alpha,t)-V^{30}(\beta,t),V^{61}_{\tilde{f}^{L}}(\alpha,t)-V^{61}_{\tilde{f}^{L}}(\beta,t),V_{X}(\alpha,t)-V_{X}(\beta,t))X_{i'}(\beta,t){\partial_{\alpha}^{b_i}}\tilde{f}_{\mu}^{+}(\alpha,t)d\beta.
\end{align*}
\begin{align*}
O_0^{4,i}(\alpha,t)=&\int_{-\pi}^{\pi}\int_{0}^{1}K_{-1}^{i}(\zeta V^{30}(\alpha,t)-\zeta V^{30}(\beta,t),V^{61}_{\tilde{f}^{L}}(\alpha,t)-V^{61}_{\tilde{f}^{L}}(\beta,t),V_{X}(\alpha,t)-V_{X}(\beta,t))\\
&X_{i'}(\beta,t)\zeta^{q_i}(\partial_{\alpha}^{b_i}\tilde{f}_{v}^{+}(\alpha,t)-\partial_{\alpha}^{b_i}\tilde{f}_{v}^{+}(\beta,t))d\zeta d\beta\\
=&\int_{-\pi}^{\pi}\int_{0}^{1}K_{-1,\zeta}^{i}( V^{30}(\alpha,t)- V^{30}(\beta,t),V^{61}_{\tilde{f}^{L}}(\alpha,t)-V^{61}_{\tilde{f}^{L}}(\beta,t),V_{X}(\alpha,t)-V_{X}(\beta,t))\\
&X_{i'}(\beta,t)\zeta^{q_i}(\partial_{\alpha}^{b_i}\tilde{f}_{v}^{+}(\alpha,t)-\partial_{\alpha}^{b_i}\tilde{f}_{v}^{+}(\beta,t))d\zeta d\beta.
\end{align*}
\begin{align*}
O_0^{5,i}(\alpha,t)=\partial_{\alpha}^{b_i}\tilde{f}_{u}^{+}(\alpha,t)\tilde{X}_i(\alpha,t),
\end{align*}
with $0\leq b_i\leq 60$, $s_i\leq 30$, $q_i\geq 0$.

$O_{0}^{1,i}$, $O_0^{2,i}$ and $O_0^{3,i}$ come from $Term_{3}$. There are three terms because $Term_3$ can be written as 
\begin{align*}
    &\quad\int_{-\pi}^{\pi}K(\tilde{f}^{+}(\alpha,t)-\tilde{f}^{+}(\beta,t)+\tilde{f}^{L}(\alpha,t)-\tilde{f}^{L}(\beta,t))(\frac{\partial_{\alpha}\tilde{f}_{\mu}^{+}(\alpha,t)}{\frac{dx(\alpha,t)}{d\alpha}}-\frac{\partial_{\alpha}\tilde{f}_{\mu}^{+}(\beta,t)}{\frac{dx(\beta,t)}{d\beta}})(\frac{dx(\beta,t)}{d\beta})d\beta\\
    &=\int_{-\pi}^{\pi}K(\tilde{f}^{+}(\alpha,t)-\tilde{f}^{+}(\beta,t)+\tilde{f}^{L}(\alpha,t)-\tilde{f}^{L}(\beta,t))(\partial_{\alpha}\tilde{f}_{\mu}^{+}(\alpha,t)-\partial_{\alpha}\tilde{f}_{\mu}^{+}(\beta,t))d\beta\\
    &\quad+\int_{-\pi}^{\pi}K(\tilde{f}^{+}(\alpha,t)-\tilde{f}^{+}(\beta,t)+\tilde{f}^{L}(\alpha,t)-\tilde{f}^{L}(\beta,t))(\frac{\frac{dx(\beta,t)}{d\beta}}{\frac{dx(\alpha,t)}{d\alpha}}-1)\partial_{\alpha}\tilde{f}_{\mu}^{+}(\alpha,t)d\beta.
\end{align*}

$O_0^{4,i}$ comes from $Term_{4,1}$. 

$O_0^{5,i}$ comes from $Term_1$.

Notice that $O_0^{2,i}$ can be written as a product of $O_0^{1,i}$ type times $\partial_{\alpha}^{s_i}\tilde{f}_{\mu}^{+}$ with $s_i\leq 30$. And for each $\zeta$ in the inner integral, $O_0^{4,i}$ is similar as $O_{0}^{1,i}$ except that we change $K_{-1}^{i}$ to $K_{-1,\zeta}^{i}$ and times $\zeta^{q_i}$. These two terms behave similarly as $O_0^{1,i}$.
\subsection{Change the contour}
Let $c(\alpha)$ satisfy
\begin{align}\label{cdefi}
\begin{cases}
          c(\alpha)=\delta_c \alpha^2,  \text{ when }0\leq \alpha \leq \frac{\delta}{32}\\
          \supp c(\alpha) \subset [0,\frac{\delta}{8}]\\
          c(\alpha) \in C^{100}[0,\infty)\cap C^{1,1}(-\infty,\infty),\  \|c(\alpha)\|_{C^{100}[0,\infty)\cap C^{1,1}(-\infty,\infty)}\leq \delta.
\end{cases}
\end{align}

Then we assume that $\tilde{f}(\cdot,t)$ can be extended to an analytic function in the region $D_{A}=\{\alpha+iy|0<\alpha<\frac{\delta}{4}, |y|<c(\alpha)t\}$. We
will drop the assumption after we get the desired equation \eqref{extendedequation01}. We remark that $\tilde{f}^{L}(\alpha,t)$, $x(\alpha,t)$ and $T_{fixed}(\al,t)$ can be analytically extended  to the region $\tilde{D}_A$ without any assumption. 

With this assumption, we change the contour. From now on we will omit the dependency of t in $\tilde{f}$ for simplicity.
Let $\alpha_{\gamma,0}^{t}=\alpha+ic(\alpha)\gamma t$.
For $-1\leq \gamma \leq 1$, from \eqref{tildefderiequation}, we have
\begin{align*}
    &\quad\frac{d\tilde{f}_{\mu}^{+(60)}(\alpha_{\gamma,0}^{t})}{dt}\\
    &=\frac{d\tilde{f}_{\mu}^{+(60)}(\alpha_{\gamma,0}^{t})}{d\alpha}\frac{1}{1+ic'(\alpha)\gamma t}(ic(\alpha)\gamma+\frac{\partial_{t}x}{\partial_{\alpha}x}(\alpha_{\gamma,0}^{t})\\
    &\qquad+\frac{1}{(\partial_{\alpha}x)(\alpha_{\gamma,0}^t)}p.v.\int_{-\pi}^{\pi}K(\tilde{f}^{+}(\alpha_{\gamma,0}^{t})-\tilde{f}^{+}(\beta_{\gamma,0}^{t})+\tilde{f}^{L}(\alpha_{\gamma,0}^{t})-\tilde{f}^{L}(\beta_{\gamma,0}^{t}))(\partial_{\beta}x)(\beta_{\gamma,0}^{t})(1+ic'(\beta)\gamma t)d\beta)\\
    &\qquad-p.v.\int_{-\pi}^{\pi}K(\tilde{f}^{+}(\alpha_{\gamma,0}^{t})-\tilde{f}^{+}(\beta_{\gamma,0}^{t})+\tilde{f}^{L}(\alpha_{\gamma,0}^{t})-\tilde{f}^{L}(\beta_{\gamma,0}^{t}))\frac{d\tilde{f}_{\mu}^{+(60)}(\beta_{\gamma,0}^{t})}{d\beta}d\beta\\
    &\quad+\sum_{i}O_0^{i}(\alpha^{t}_{\gamma,0})+T_{fixed}(\alpha_{\gamma,0}^{t}).
\end{align*}
The "p.v." shows up because in the last step of \eqref{tildefderiequation} we write  $(\frac{\tilde{f}_{\mu}^{+(61)}(\alpha,t)}{\frac{dx(\alpha,t)}{d\alpha}}-\frac{\tilde{f}_{\mu}^{+(61)}(\beta,t)}{\frac{dx(\beta,t)}{d\beta}})$ into two terms. There is no singularity if we add those terms with $p.v$ up, and we could change the contour safely. 
Now $O_0^i$ becomes
\begin{align*}
O_0^{1,i}(\alpha_{\gamma,0}^{t})=&\int_{-\pi}^{\pi}K_{-1}^{i}(V^{30}(\alpha_{\gamma,0}^{t})-V^{30}(\beta_{\gamma,0}^{t}),V^{61}_{\tilde{f}^{L}}(\alpha_{\gamma,0}^{t})-V^{61}_{\tilde{f}^{L}}(\beta_{\gamma,0}^{t}),V_X(\alpha_{\gamma,0}^{t})-V_X(\beta_{\gamma,0}^{t}))X_{i'}(\beta_{\gamma,0}^{t})\\
&((\partial_{\alpha}^{b_{i}}\tilde{f}_{v}^{+})(\alpha_{\gamma,0}^{t})-(\partial_{\alpha}^{b_{i}}\tilde{f}_{v}^{+})(\beta_{\gamma,0}^{t}))(1+ic'(\beta)\gamma t)d\beta,
\end{align*}
\begin{align*}
&O_0^{2,i}(\alpha_{\gamma,0}^{t})=\int_{-\pi}^{\pi}K_{-1}^{i}(V^{30}(\alpha_{\gamma,0}^{t})-V^{30}(\beta_{\gamma,0}^{t}),V^{61}_{\tilde{f}^{L}}(\alpha_{\gamma,0}^{t})-V^{61}_{\tilde{f}^{L}}(\beta_{\gamma,0}^{t}),V_X(\alpha_{\gamma,0}^{t})-V_X(\beta_{\gamma,0}^{t}))X_{i'}(\beta_{\gamma,0}^{t})\\
&({\partial_{\alpha}^{s_i}}\tilde{f}_{\mu}^{+})(\alpha_{\gamma,0}^{t})((\partial_{\alpha}^{b_{i}}\tilde{f}_{v}^{+})(\alpha_{\gamma,0}^{t})-(\partial_{\alpha}^{b_{i}}\tilde{f}_{v}^{+})(\beta_{\gamma,0}^{t}))(1+ic'(\beta)\gamma t)d\beta,
\end{align*}
\begin{align*}
&O_0^{3,i}(\alpha_{\gamma,0}^{t})=\int_{-\pi}^{\pi}K_{0}^{i}(V^{30}(\alpha_{\gamma,0}^{t})-V^{30}(\beta_{\gamma,0}^{t}),V^{61}_{\tilde{f}^{L}}(\alpha_{\gamma,0}^{t})-V^{61}_{\tilde{f}^{L}}(\beta_{\gamma,0}^{t}),V_X(\alpha_{\gamma,0}^{t})-V_X(\beta_{\gamma,0}^{t}))X_{i'}(\beta_{\gamma,0}^{t})\\
&{(\partial_{\alpha}^{b_i}}\tilde{f}_{\mu}^{+})(\alpha_{\gamma,0}^{t})(1+ic'(\beta)\gamma t ) d\beta,
\end{align*}
\begin{align*}
&O_0^{4,i}(\alpha_{\gamma,0}^{t})=\int_{-\pi}^{\pi}\int_{0}^{1}K_{-1,\zeta}^{i}( V^{30}(\alpha_{\gamma,0}^{t})-V^{30}(\beta_{\gamma,0}^{t}),V^{61}_{\tilde{f}^{L}}(\alpha_{\gamma,0}^{t})-V^{61}_{\tilde{f}^{L}}(\beta_{\gamma,0}^{t}),V_X(\alpha_{\gamma,0}^{t})-V_X(\beta_{\gamma,0}^{t}))\\
&X_{i'}(\beta_{\gamma,0}^{t})\zeta^{q_i}((\partial_{\alpha}^{b_i}\tilde{f}_{v}^{+})(\alpha_{\gamma,0}^{t})-(\partial_{\alpha}^{b_i}\tilde{f}_{v}^{+})(\beta_{\gamma,0}^{t}))d\zeta(1+ic'(\beta)\gamma t ) d\beta,
\end{align*}
\begin{align*}
    O_0^{5,i}(\alpha_{\gamma,0}^{t})=\partial_{\alpha}^{b_i}\tilde{f}_{u}^{+}(\alpha_{\gamma,0}^{t})\tilde{X}_i(\alpha_{\gamma,0}^t).
\end{align*}
\subsection{The modified equation}\label{nearmodified}
\subsubsection{The sign of the coefficient before transport term}
The sign of the coefficient before $\frac{d}{d\alpha}\tilde{f}^{+(60)}(\alpha_{\gamma,0}^{t})$  when $\alpha =0$ is related with condition \ref{extracondition3} in the theorem, which corresponds to the positive/negative $\eta$ in \eqref{new linear model}. In fact, we have 
\begin{equation}\label{3coefficient}
\begin{split}
    &\qquad\frac{1}{1+ic'(\alpha)\gamma t}(ic(\alpha)\gamma+\frac{\partial_{t}x}{\partial_{\alpha}x}(\alpha_{\gamma,0}^{t})+\frac{1}{(\partial_{\alpha} x)(\alpha_{\gamma,0}^{t})}\\
    &\quad\quad\cdot p.v.\int_{-\pi}^{\pi}K(\tilde{f}^{+}(\alpha_{\gamma,0}^{t})-\tilde{f}^{+}(\beta_{\gamma,0}^{t})+\tilde{f}^{L}(\alpha_{\gamma,0}^{t})-\tilde{f}^{L}(\beta_{\gamma,0}^{t})(\partial_{\beta}x)(\beta_{\gamma,0}^{t})(1+ic'(\beta)\gamma t)d\beta)\\
    &=\frac{1}{1+ic'(\alpha)\gamma t}(ic(\alpha)\gamma+\frac{\partial_{t}x}{\partial_{\alpha}x}(\alpha_{\gamma,0}^{t})+\frac{1}{(\partial_{\alpha}x)(\alpha_{\gamma,0}^{t})}\\
    &\qquad\cdot p.v.\int_{-\pi}^{\pi} K(\tilde{f}^{+}(\alpha_{\gamma,0}^{t})-\tilde{f}^{+}(\beta_{\gamma,0}^{t})+\tilde{f}^{L}(\alpha_{\gamma,0}^{t})-\tilde{f}^{L}(\beta_{\gamma,0}^{t}))(\partial_{\beta}x)(\beta_{\gamma,0}^{t})(1+ic'(\beta)\gamma t)d\beta\\
    &\quad-\frac{\partial_{t}x}{\partial_{\alpha}x}(0)-\frac{1}{(\partial_{\alpha}x)(0)}p.v. \int_{-\pi}^{\pi}K(\tilde{f}^{+}(0)-\tilde{f}^{+}(\beta_{\gamma,0}^{t})+\tilde{f}^{L}(0)-\tilde{f}^{L}(\beta_{\gamma,0}^{t}))(\partial_{\beta}x)(\beta_{\gamma,0}^t)(1+ic'(\beta)\gamma t)d\beta\\
    &\quad+\frac{\partial_{t}x}{\partial_{\alpha}x}(0)+\frac{1}{(\partial_{\alpha}x)(0)}p.v.\int_{-\pi}^{\pi}K(\tilde{f}^{+}(0)-\tilde{f}^{+}(\beta)+\tilde{f}^{L}(0)-\tilde{f}^{L}(\beta))(\partial_{\beta}x)(\beta)d\beta).
    \end{split}
\end{equation}
Here we use the analyticity assumption of $\tilde{f}^{+}$ and the property of $c(\alpha)$. Since $c(\alpha)=\delta_c\alpha^2$, we could change the contour in the last step. 
Let 
\begin{equation}\label{kappadefi}
\kappa(t)=(\frac{\partial_{t}x}{\partial_{\alpha}x}(0)+\frac{1}{(\partial_{\alpha}x)(0)}\int_{-\pi}^{\pi}K(\tilde{f}^{+}(0)-\tilde{f}^{+}(\beta)+\tilde{f}^{L}(0)-\tilde{f}^{L}(\beta))(\partial_{\beta}x)(\beta)d\beta).
\end{equation}
\color{black} 

 Then from  \eqref{3coefficient}, we have
 \begin{equation}\label{3coefficient2}
\begin{split}
    &\qquad\frac{1}{1+ic'(\alpha)\gamma t}(ic(\alpha)\gamma+\frac{\partial_{t}x}{\partial_{\alpha}x}(\alpha_{\gamma,0}^{t})+\frac{1}{(\partial_{\alpha} x)(\alpha_{\gamma,0}^{t})}\\
    &\quad\quad\cdot p.v.\int_{-\pi}^{\pi}K(\tilde{f}^{+}(\alpha_{\gamma,0}^{t})-\tilde{f}^{+}(\beta_{\gamma,0}^{t})+\tilde{f}^{L}(\alpha_{\gamma,0}^{t})-\tilde{f}^{L}(\beta_{\gamma,0}^{t})(\partial_{\beta}x)(\beta_{\gamma,0}^{t})(1+ic'(\beta)\gamma t)d\beta)\\
    &=[\frac{1}{1+ic'(\alpha)\gamma t}(ic(\alpha)\gamma+\frac{\partial_{t}x}{\partial_{\alpha}x}(\alpha_{\gamma,0}^{t})\\
    &\quad+\frac{1}{(\partial_{\alpha}x)(\alpha_{\gamma,0}^{t})}
    \cdot p.v.\int_{-\pi}^{\pi}K(\tilde{f}^{+}(\alpha_{\gamma,0}^{t})-\tilde{f}^{+}(\beta_{\gamma,0}^{t})+\tilde{f}^{L}(\alpha_{\gamma,0}^{t})-\tilde{f}^{L}(\beta_{\gamma,0}^{t}))\frac{dx(\beta_{\gamma,0}^{t})}{d\beta}d\beta\\&\qquad
    -\frac{\partial_{t}x}{\partial_{\alpha}x}(0)-\frac{1}{(\partial_{\alpha}x)(0)}p.v.\int_{-\pi}^{\pi}K(\tilde{f}^{+}(0)-\tilde{f}^{+}(\beta_{\gamma,0}^{t})+\tilde{f}^{L}(0)-\tilde{f}^{L}(\beta_{\gamma,0}^{t}))\frac{dx(\beta_{\gamma,0}^{t})}{d\beta}d\beta)]\\
    &\quad+(\frac{1}{1+ic'(\alpha)\gamma t} -1)\kappa(t) +\kappa (t).
    \end{split}
\end{equation}
Moreover, we have the following lemma 
\begin{lemma}\label{conditionchange}
For sufficiently small t, the inequality below is true:
\begin{equation}\label{modifiedcondition}
\kappa(t)(\frac{d Z_1}{dt}(t)+p.v.\int_{-\pi}^{\pi}K(f(Z_1(t),t)-f(\beta,t))d\beta) >0
\end{equation}
\end{lemma}
\begin{proof}
\eqref{initialchangevariable} gives: 
\[
\frac{dx(0,t)}{dt}=\frac{dZ_1(t)}{dt},
\]
and 
\[
\frac{dx(\alpha,t)}{d\alpha}=1-\cos(\alpha)(Z_2(t)+\frac{\pi}{2}-Z_1(t))> 0,
\]
when $t$ is sufficiently small.
Moreover, we can change the variable and get
\begin{align*}
&\quad p.v.\int_{-\pi}^{\pi}K(f(Z_1(t),t)-f(\beta,t))d\beta\\
&=p.v.\int_{-\pi}^{\pi}K(\tilde{f}^{+}(0,t)-\tilde{f}^{+}(\beta,t)+\tilde{f}^{L}(0,t)-\tilde{f}^{L}(\beta,t))\frac{dx(\beta,t)}{d\beta}d\beta.
\end{align*}
Therefore,
we have
\begin{align}\label{kappadefinition}
    \kappa(t)&= \frac{\partial_{t}x}{\partial_{\alpha}x}(0,t)+\frac{1}{(\partial_{\alpha}x)(0,t)}\int_{-\pi}^{\pi}K(\tilde{f}^{+}(0,t)-\tilde{f}^{+}(\beta)+\tilde{f}^{L}(0,t)-\tilde{f}^{L}(\beta,t))(\partial_{\beta}x)(\beta,t)d\beta)\\\nonumber
    &=\frac{1}{\frac{dx}{d\alpha}(\alpha,t)}(\frac{d Z_1}{dt}(t)+p.v.\int_{-\pi}^{\pi}K(f(Z_1(t),t)-f(\beta,t))d\beta)
\end{align} when $t$ is sufficiently small.
Then we have the result since from the condition \eqref{extracondition3},
\[
(\frac{d Z_1}{dt}(t)+p.v.\int_{-\pi}^{\pi}K(f(Z_1(t),t)-f(\beta,t))d\beta)\neq 0,
\]
and $\frac{dx}{d\alpha}(\alpha,t)>0$ when $t$ is sufficiently small. 
\end{proof}
Hence $\kappa(t)$ has the same sign as $(\frac{d Z_1}{dt}(t)+p.v.\int_{-\pi}^{\pi}K(f(Z_1(t),t)-f(\beta,t))d\beta)$.
Therefore, by replacing the coefficient in the corresponding term with $\kappa(t)$, we have 
\begin{align}\label{changecontourequation}
    &\frac{d\tilde{f}_{\mu}^{+(60)}(\alpha_{\gamma,0}^{t})}{dt}\\\nonumber
    =&\kappa(t)\frac{d\tilde{f}_{\mu}^{+(60)}(\alpha_{\gamma,0}^{t})}{d\alpha}\\\nonumber
    &+B_{\tilde{f}}(\alpha,\gamma,t)\frac{d\tilde{f}_{\mu}^{+(60)}(\alpha_{\gamma,0}^{t})}{d\alpha}\\\nonumber
    &-p.v.\int_{-\pi}^{\pi} K(\tilde{f}^{+}(\alpha_{\gamma,0}^{t})-\tilde{f}^{+}(\beta_{\gamma,0}^{t})+\tilde{f}^{L}(\alpha_{\gamma,0}^{t})-\tilde{f}^{L}(\beta_{\gamma,0}^{t}))\frac{d\tilde{f}_{\mu}^{+(60)}(\beta_{\gamma,0}^{t})}{d\beta}d\beta\\\nonumber
    &+\sum_{i}O^i_{0}(\alpha_{\gamma,0}^{t})+T_{fixed}(\alpha_{\gamma,0}^{t}).
\end{align}
where 
\begin{align*}
    &\quad B^{0}_{\tilde{f}}(\alpha,\gamma,t)\\
    &=[\frac{1}{1+ic'(\alpha)\gamma t}(ic(\alpha)\gamma+\frac{\partial_{t}x}{\partial_{\alpha}x}(\alpha_{\gamma,0}^{t})\\
    &\qquad+\frac{1}{(\partial_{\alpha}x)(\alpha_{\gamma,0}^{t})}
    \cdot p.v.\int_{-\pi}^{\pi}K(\tilde{f}^{+}(\alpha_{\gamma,0}^{t})-\tilde{f}^{+}(\beta_{\gamma,0}^{t})+\tilde{f}^{L}(\alpha_{\gamma,0}^{t})-\tilde{f}^{L}(\beta_{\gamma,0}^{t}))\frac{dx(\beta_{\gamma,0}^{t})}{d\beta}d\beta\\&
    \qquad-\frac{\partial_{t}x}{\partial_{\alpha}x}(0)-\frac{1}{(\partial_{\alpha}x)(0)}p.v.\int_{-\pi}^{\pi}K(\tilde{f}^{+}(0)-\tilde{f}^{+}(\beta_{\gamma,0}^{t})+\tilde{f}^{L}(0)-\tilde{f}^{L}(\beta_{\gamma,0}^{t}))\frac{dx(\beta_{\gamma,0}^{t})}{d\beta}d\beta)]\\
    &\quad+(\frac{1}{1+ic'(\alpha)\gamma t} -1)\kappa(t).
\end{align*}
Now we change the variable again.
Let
\begin{equation}\label{alphadefi}
\alpha_{\gamma}^{t}=\alpha+\tau(t)+ic(\alpha+\tau(t))\gamma t
\end{equation}
and take
\begin{align}\label{defkappa}
\tau(t)=\begin{cases}
0, & \text{when }\ \kappa(0)>0,\\
-\int_{0}^{t}\kappa(t)dt, & \text{when }\ \kappa(0)<0.
\end{cases}
\end{align} 
Notice that from $\kappa(t)$ definition \eqref{kappadefinition}, we have $\kappa(t)$ is continuous. From lemma \ref{conditionchange}, $\kappa(t)\neq 0$. Thus $\kappa(t)$ will not change sign. We also have \begin{equation}\label{con:tautneg}
\tau(t)>0, \quad \tau'(t)>0, \text{ when }\kappa(0)<0.
\end{equation}
From \eqref{changecontourequation}, we get
\begin{align}\label{aftertranslationequation}
     &\frac{d\tilde{f}_{\mu}^{+(60)}(\alpha_{\gamma}^{t})}{dt}\\\nonumber
    &=(\kappa(t)+\tau'(t))\frac{d\tilde{f}_{\mu}^{+(60)}(\alpha_{\gamma}^{t})}{d\alpha}+\frac{d\tilde{f}_{\mu}^{+(60)}(\alpha_{\gamma}^{t})}{d\alpha}B_{\tilde{f}}(\alpha,\gamma,t)\\\nonumber
    &-p.v.\int_{-\pi}^{\pi}K(\tilde{f}^{+}(\alpha_{\gamma}^{t})-\tilde{f}^{+}(\beta_{\gamma}^{t})+\tilde{f}^{L}(\alpha_{\gamma}^{t})-\tilde{f}^{L}(\beta_{\gamma}^{t}))\frac{d\tilde{f}_{\mu}^{+(60)}(\beta_{\gamma}^{t})}{d\beta}d\beta\\\nonumber
    &+\sum_i O_0^{i}(\alpha_{\gamma}^{t})+T_{fixed}(\alpha_{\gamma}^{t}),
\end{align}
where 
\begin{align*}
    &\quad B_{\tilde{f}}(\alpha,\gamma,t)\\
    &=[\frac{1}{1+ic'(\alpha+\tau(t))\gamma t}(ic(\alpha+\tau(t))\gamma+\frac{\partial_{t}x}{\partial_{\alpha}x}(\alpha_{\gamma}^{t})\\
    &\qquad+\frac{1}{(\partial_{\alpha}x)(\alpha_{\gamma}^{t})}
    \cdot p.v.\int_{-\pi}^{\pi}K(\tilde{f}^{+}(\alpha_{\gamma}^{t})-\tilde{f}^{+}(\beta_{\gamma}^{t})+\tilde{f}^{L}(\alpha_{\gamma}^{t})-\tilde{f}^{L}(\beta_{\gamma}^{t}))\frac{dx(\beta_{\gamma}^{t})}{d\beta}d\beta\\&\qquad
    -\frac{\partial_{t}x}{\partial_{\alpha}x}(0)-\frac{1}{(\partial_{\alpha}x)(0)}p.v.\int_{-\pi}^{\pi}K(\tilde{f}^{+}(0)-\tilde{f}^{+}(\beta_{\gamma}^{t})+\tilde{f}^{L}(0)-\tilde{f}^{L}(\beta_{\gamma}^{t}))\frac{dx(\beta_{\gamma}^{t})}{d\beta}d\beta)]\\
    &\quad+(\frac{1}{1+ic'(\alpha+\tau(t))\gamma t} -1)\kappa(t).
\end{align*}

And $O_0^{i}$ terms become
\begin{align}\label{O1inew}
O^{1,i}_0(\alpha_{\gamma}^{t})=&\int_{-\pi}^{\pi}K_{-1}^{i}(V^{30}(\alpha_{\gamma}^{t})-V^{30}(\beta_{\gamma}^{t}),V^{61}_{\tilde{f}^{L}}(\alpha_{\gamma}^{t})-V^{61}_{\tilde{f}^{L}}(\beta_{\gamma}^{t}),V_X(\alpha_{\gamma}^{t})-V_X(\beta_{\gamma}^{t}))X_{i'}(\beta_{\gamma}^{t})\\\nonumber
&((\partial_{\alpha}^{b_{i}}\tilde{f}_{v}^{+})(\alpha_{\gamma}^{t})-(\partial_{\alpha}^{b_{i}}\tilde{f}_{v}^{+})(\beta_{\gamma}^{t}))(1+ic'(\beta+\tau(t))\gamma t)d\beta.
\end{align}
\begin{align}\label{O2inew}
&O^{2,i}_0(\alpha_{\gamma}^{t})=\int_{-\pi}^{\pi}K_{-1}^{i}(V^{30}(\alpha_{\gamma}^{t})-V^{30}(\beta_{\gamma}^{t}),V^{61}_{\tilde{f}^{L}}(\alpha_{\gamma}^{t})-V^{61}_{\tilde{f}^{L}}(\beta_{\gamma}^{t}),V_X(\alpha_{\gamma}^{t})-V_X(\beta_{\gamma}^{t}))X_{i'}(\beta_{\gamma}^{t})\\\nonumber
&({\partial_{\alpha}^{s_i}}\tilde{f}_{\mu}^{+})(\alpha_{\gamma}^{t})((\partial_{\alpha}^{b_{i}}\tilde{f}_{v}^{+})(\alpha_{\gamma}^{t})-(\partial_{\alpha}^{b_{i}}\tilde{f}_{v}^{+})(\beta_{\gamma}^{t}))(1+ic'(\beta+\tau(t))\gamma t)d\beta.
\end{align}
\begin{align}\label{O3inew}
&O^{3,i}_0(\alpha_{\gamma}^{t})=\int_{-\pi}^{\pi}K_{0}^{i}(V^{30}(\alpha_{\gamma}^{t})-V^{30}(\beta_{\gamma}^{t}),V^{61}_{\tilde{f}^{L}}(\alpha_{\gamma}^{t})-V^{61}_{\tilde{f}^{L}}(\beta_{\gamma}^{t}),V_X(\alpha_{\gamma}^{t})-V_X(\beta_{\gamma}^{t}))X_{i'}(\beta_{\gamma}^{t})\\\nonumber
&{(\partial_{\alpha}^{b_i}}\tilde{f}_{\mu}^{+})(\alpha_{\gamma}^{t})(1+ic'(\beta+\tau(t))\gamma t ) d\beta.
\end{align}
\begin{align}\label{O4inew}
&O^{4,i}_0(\alpha_{\gamma}^{t})=\int_{-\pi}^{\pi}\int_{0}^{1}K_{-1,\zeta}^{i}( V^{30}(\alpha_{\gamma}^{t})- V^{30}(\beta_{\gamma}^{t}),V^{61}_{\tilde{f}^{L}}(\alpha_{\gamma}^{t})-V^{61}_{\tilde{f}^{L}}(\beta_{\gamma}^{t}),V_X(\alpha_{\gamma}^{t})-V_X(\beta_{\gamma}^{t}))\\\nonumber
&X_{i'}(\beta_{\gamma}^{t})\zeta^{q_i}((\partial_{\alpha}^{b_i}\tilde{f}_{v}^{+})(\alpha_{\gamma}^{t})-(\partial_{\alpha}^{b_i}\tilde{f}_{v}^{+})(\beta_{\gamma}^{t}))d\zeta(1+ic'(\beta+\tau(t))\gamma t ) d\beta.
\end{align}
\begin{align}\label{O5inew}
    O^{5,i}_0(\alpha_{\gamma}^{t})=\partial_{\alpha}^{b_i}\tilde{f}_{u}^{+}(\alpha_{\gamma}^{t})\tilde{X}_i(\alpha_{\gamma}^t),
\end{align}
with $s_i\leq 30$, $b_i\leq 60$, $q_i\geq 0$.
\subsubsection{Times a cut-off function}
Now we want to make sure the solution of our new equation is supported in a small region. Let $\lambda$ defined on the union of a real line and a disk $D_{\delta}=\{|\alpha+iy|\leq \delta\}$. On the real line, let 
\begin{equation}
\lambda(\alpha)=\lambda_0(\frac{\alpha}{10}).
\end{equation} 
Hence  we have
\begin{align}\label{lambdadefi}
    \lambda(\alpha)=\begin{cases}1&|\alpha|\leq 10\delta,\\0&|\alpha|\geq 20\delta.
\end{cases}
\end{align}
On the disk, let $\lambda=1$. Therefore we have $\lambda\lambda_0(\alpha)=\lambda_0(\alpha)=\lambda\lambda_0(\alpha_{\gamma}^{t})$, and
\[
\lambda\partial_{\alpha}^{i}\tilde{f}^{+}(\alpha_{\gamma}^{t},t)=\partial_{\alpha}^{i}\tilde{f}^{+}(\alpha_{\gamma}^{t},t),
\]
when $t$ is sufficiently small.
We also let
\begin{equation}\label{3oinew}
    O^{i}=\lambda(\alpha_{\gamma}^{t})O^{i}_0.
\end{equation}

Then from \eqref{aftertranslationequation} we get
\begin{align}\label{3mainequationchange}
 &\frac{d\tilde{f}_{\mu}^{+(60)}(\alpha_{\gamma}^{t})}{dt}\\\nonumber
     =&\frac{d\tilde{f}_{\mu}^{+(60)}(\alpha_{\gamma}^{t})\lambda(\alpha_{\gamma}^{t})}{dt}\\\nonumber
    =&\lambda^{'}(\alpha_{\gamma}^{t})(1+ic'(\alpha+\tau(t))\gamma t)\tilde{f}_{\mu}^{+(60)}(\alpha_{\gamma}^{t})+\lambda(\alpha_{\gamma}^{t})\frac{d}{dt}(\tilde{f}_{\mu}^{+(60)}(\alpha_{\gamma}^{t}))\\\nonumber
    =&\underbrace{\lambda(\alpha_{\gamma}^{t})(\kappa(t)+\tau'(t))\frac{d\tilde{f}_{\mu}^{+(60)}(\alpha_{\gamma}^{t})}{d\alpha}}_{M_{1,1}}+\underbrace{\lambda(\alpha_{\gamma}^{t})\frac{d\tilde{f}_{\mu}^{+(60)}(\alpha_{\gamma}^{t})}{d\alpha}B_{\tilde{f}}(\alpha,\gamma,t)}_{M_{1,2}}\\\nonumber
    &\underbrace{-\lambda(\alpha_{\gamma}^{t})p.v.\int_{-\pi}^{\pi}K(\tilde{f}^{+}(\alpha_{\gamma}^{t})-\tilde{f}^{+}(\beta_{\gamma}^{t})+\tilde{f}^{L}(\alpha_{\gamma}^{t})-\tilde{f}^{L}(\beta_{\gamma}^{t}))\frac{d\tilde{f}_{\mu}^{+(60)}(\beta_{\gamma}^{t})}{d\beta} d\beta}_{F_2}\\\nonumber
    &+\lambda(\alpha_{\gamma}^{t})\cdot \sum_{i}O^i_{0}(\alpha_{\gamma}^{t})+\lambda(\alpha_{\gamma}^{t})T_{fixed}(\alpha_{\gamma}^{t})\\\nonumber
    =&M_{1,1}(\tilde{f}^{+(60)})+M_{1,2}(\tilde{f}^{+(60)})+F_{2}(\tilde{f}^{+(60)})+\sum_{i} O^i(\tilde{f}^{+(60)})+\lambda(\alpha_{\gamma}^{t})T_{fixed}(\alpha_{\gamma}^{t}).
\end{align}
Here we use the fact that the intersection of the support of $\lambda^{'}$ and $\tilde{f}^{+}$ is empty (from \eqref{fLspace} and \eqref{lambdadefi}). We also abuse the notation here by writing $O^i(\alpha_{\gamma}^{t})$ as $O^i(\tilde{f}^{+(60)})$
.
\subsubsection{Modify the terms to preserve the behavior near $\alpha =-\tau(t)$}

In order to avoid the singularity caused by the boundary, we need to do further changes to $F_2$ and $O^i$. We claim the equation holds at $\gamma=0$ even without any analyticity assumption.

For $F_2$, we separate the integral into the "near $-\tau(t)$ part" and the "far away from $-\tau(t)$" part and do integration by parts. We have
\begin{align*}
    &\quad F_2(\tilde{f}^{+(60)})\\
    &=-\lambda(\alpha_{\gamma}^{t})p.v.\int_{-\pi}^{\pi}K(\tilde{f}^{+}(\alpha_{\gamma}^{t})-\tilde{f}^{+}(\beta_{\gamma}^{t})+\tilde{f}^{L}(\alpha_{\gamma}^{t})-\tilde{f}^{L}(\beta_{\gamma}^{t}))\frac{d\tilde{f}_{\mu}^{+(60)}(\beta_{\gamma}^{t})}{d\beta}d\beta\\
    &=-\lambda(\alpha_{\gamma}^{t})p.v.\int_{-\tau(t)}^{2\alpha+\tau(t)}K(\tilde{f}^{+}(\alpha_{\gamma}^{t})-\tilde{f}^{+}(\beta_{\gamma}^{t})+\tilde{f}^{L}(\alpha_{\gamma}^{t})-\tilde{f}^{L}(\beta_{\gamma}^{t}))d(\tilde{f}_{\mu}^{+(60)}(\beta_{\gamma}^{t})-\tilde{f}_{\mu}^{+(60)}(\alpha_{\gamma}^{t}))\\
    &\quad-\lambda(\alpha_{\gamma}^{t})\int_{2\alpha+\tau(t)}^{\pi}K(\tilde{f}^{+}(\alpha_{\gamma}^{t})-\tilde{f}^{+}(\beta_{\gamma}^{t})+\tilde{f}^{L}(\alpha_{\gamma}^{t})-\tilde{f}^{L}(\beta_{\gamma}^{t}))d(\tilde{f}_{\mu}^{+(60)}(\beta_{\gamma}^{t}))\\
    &=-\lambda(\alpha_{\gamma}^{t})K(\tilde{f}^{+}(\alpha_{\gamma}^{t})-\tilde{f}^{+}((2\alpha+\tau(t))_{\gamma}^{t})+\tilde{f}^{L}(\alpha_{\gamma}^{t})-\tilde{f}^{L}((2\alpha+\tau(t))_{\gamma}^{t}))(\tilde{f}_{\mu}^{+(60)}((2\alpha+\tau(t))_{\gamma}^{t})-\tilde{f}_{\mu}^{+(60)}(\alpha_{\gamma}^{t}))\\
    &\quad+\lambda(\alpha_{\gamma}^{t})K(\tilde{f}^{+}(\alpha_{\gamma}^{t})-\tilde{f}^{+}((-\tau(t))_{\gamma}^{t})+\tilde{f}^{L}(\alpha_{\gamma}^{t})-\tilde{f}^{L}((-\tau(t))_{\gamma}^{t}))(\tilde{f}_{\mu}^{+(60)}((-\tau(t))_{\gamma}^{t})-\tilde{f}_{\mu}^{+(60)}(\alpha_{\gamma}^{t}))\\
    &\quad+\lambda(\alpha_{\gamma}^{t})p.v.\int_{-\tau(t)}^{2\alpha+\tau(t)}\frac{d}{d\beta}(K(\tilde{f}^{+}(\alpha_{\gamma}^{t})-\tilde{f}^{+}(\beta_{\gamma}^{t})+\tilde{f}^{L}(\alpha_{\gamma}^{t})-\tilde{f}^{L}(\beta_{\gamma}^{t})))(\tilde{f}_{\mu}^{+(60)}(\beta_{\gamma}^{t})-\tilde{f}_{\mu}^{+(60)}(\alpha_{\gamma}^{t}))d\beta\\
    &\quad+\lambda(\alpha_{\gamma}^{t})K(\tilde{f}^{+}(\alpha_{\gamma}^{t})-\tilde{f}^{+}((2\alpha+\tau(t))_{\gamma}^{t})+\tilde{f}^{L}(\alpha_{\gamma}^{t})-\tilde{f}^{L}((2\alpha+\tau(t))_{\gamma}^{t}))\tilde{f}_{\mu}^{+(60)}((2\alpha+\tau(t))_{\gamma}^{t})\\
    %&+\lambda(\alpha_{\gamma}^{t})K(\tilde{h}(\alpha,\gamma)-\tilde{h}(\pi,\gamma)+\tilde{f}^{L}(\alpha_{\gamma}^{t})-\tilde{f}^{L}((\pi)_{\gamma}^{t}))h(\pi,\gamma)\\
    &\quad+\lambda(\alpha_{\gamma}^{t})\int_{2\alpha+\tau(t)}^{\pi}\frac{\frac{d}{d\beta}(K(\tilde{f}^{+}(\alpha_{\gamma}^{t})-\tilde{f}^{+}(\beta_{\gamma}^{t})+\tilde{f}^{L}(\alpha_{\gamma}^{t})-\tilde{f}^{L}(\beta_{\gamma}^{t})))}{1+ic'(\beta+\tau(t))\gamma t}\tilde{f}_{\mu}^{+(60)}(\beta_{\gamma}^{t})(1+ic'(\beta+\tau(t))\gamma t)d\beta\\
    &=Term_1 + Term_2 + Term_3+ Term_4 + Term_5.
\end{align*}
We further do the change to the $Term_5$. We change the contour from the curve to a vertical line and a horizontal line starting from $2\alpha+\tau(t)$. Here $1+ic'(2\alpha+2\tau(t))\eta t$ in the denominator of $Term_{5,2}$ coming from $1+ic'(\beta+\tau(t))\gamma t$ by changing $\beta$ to $2\alpha+\tau(t)$ and $\gamma$ to $\eta$.
\begin{align*}
    &\quad Term\  5\\
    &= -\lambda(\alpha_{\gamma}^{t})\int_{0}^{\gamma}\frac{\frac{d}{d\beta}(K(\tilde{f}^{+}(\alpha_{\gamma}^{t})-\tilde{f}^{+}(\beta_{\eta}^{t})+\tilde{f}^{L}(\alpha_{\gamma}^{t})-\tilde{f}^{L}(\beta_{\eta}^{t})))|_{\beta=2\alpha+\tau(t)}}{1+ic'(2\alpha+2\tau(t))\eta t}\tilde{f}_{\mu}^{+(60)}((2\alpha+\tau(t))_{\eta}^{t})ic(2\alpha+2\tau(t)) t d\eta\\
    &\quad+\lambda(\alpha_{\gamma}^{t})\int_{2\alpha+\tau(t)}^{\pi}\frac{d}{d\beta}(K(\tilde{f}^{+}(\alpha_{\gamma}^{t})-\tilde{f}^{+}(\beta_{0}^{t})+\tilde{f}^{L}(\alpha_{\gamma}^{t})-\tilde{f}^{L}(\beta_{0}^{t})))\tilde{f}_{\mu}^{+(60)}(\beta_0^{t})d\beta\\
    %&=\lambda(\alpha_{\gamma}^{t})\int_{0}^{\gamma}\frac{\frac{d}{d\beta}(K(\tilde{h}(\alpha,\gamma)-\tilde{h}(\beta,\eta)+\tilde{f}^{L}(\alpha_{\gamma}^{t})-\tilde{f}^{L}(\beta_{\eta}^{t})))|_{\beta=2\alpha}}{1+ic'(2\alpha)\eta t}h(2\alpha,\eta)ic(2\alpha) t d\eta\\
      %&- \lambda(\alpha_{\gamma}^{t})\int_{2\alpha}^{\pi}\frac{d}{d\beta}(K(-\tilde{f}^{+}(\beta_{0}^{t})+\tilde{f}^{L}(\alpha_{\gamma}^{t})-\tilde{f}^{L}(\beta_{0}^{t})))\tilde{f}^{+<9>}(\beta,t)d\beta\\
    %&- \lambda(\alpha_{\gamma}^{t})\int_{2\alpha}^{\pi}\int_{0}^{1}\frac{d}{d\beta}(K(\eta_2\tilde{h}(\alpha,\gamma)-\tilde{f}^{+}(\beta_{0}^{t})+\tilde{f}^{L}(\alpha_{\gamma}^{t})-\tilde{f}^{L}(\beta_{0}^{t}))\tilde{f}^{+<9>}(\beta,t)d\eta_2d\beta\tilde{h}(\alpha,\gamma)\\
    &=Term_{5,1} +Term_{5,2}.
\end{align*}
Here by changing the contour, we change the $\tilde{f}^+(\beta_{\gamma}^t)$ to $\tilde{f}^+(\beta_0^t)=\tilde{f}^+(\beta+\tau(t))$ when $\beta>2\alpha+\tau(t)$ in the integral. Notice that the $\tilde{f}^+(\beta)$ is known and well-defined even without the analyticity assumption. In the later section, we will change $\tilde{f}^+(\alpha_{\gamma}^{t})$ to an unknown solution $h(\alpha,\gamma,t)$, but all $\tilde{f}^+(\beta)$ will be kept same.
In conclusion, let
\begin{align}\label{M2c}
F_2(\tilde{f}^{+(60)})=Term_1 + Term_2 + Term_3 + Term_4 + Term_{5,1} +Term_{5,2}.
\end{align}
For $O^{1,i}$, we change it in the following way. First we separate the terms in the bracket $((\partial_{\alpha}^{b_{i}}\tilde{f}_{v}^{+})(\alpha_{\gamma}^{t})-(\partial_{\alpha}^{b_{i}}\tilde{f}_{v}^{+})(\beta_{\gamma}^{t}))$ and use the fact that $\tilde{f}^{+}=0$ when $\alpha<0$. We then have
\begin{align}\label{O1iterm1}
O^{1,i}(\tilde{f}^{+(60)})=&\lambda(\alpha_{\gamma}^{t})p.v.\int_{-\pi}^{\pi}K_{-1}^{i}(V^{30}(\alpha_{\gamma}^{t})-V^{30}(\beta_{\gamma}^{t}),V^{61}_{\tilde{f}^{L}}(\alpha_{\gamma}^{t})-V^{61}_{\tilde{f}^{L}}(\beta_{\gamma}^{t}),V_X(\alpha_{\gamma}^{t})-V_X(\beta_{\gamma}^{t}))X_{i'}(\beta_{\gamma}^{t})\\\nonumber
&(\partial_{\alpha}^{b_{i}}\tilde{f}_{v}^{+})(\alpha_{\gamma}^{t})(1+ic'(\beta+\tau(t))\gamma t)d\beta\\\nonumber
&-\lambda(\alpha_{\gamma}^{t})p.v.\int_{-\tau(t)}^{\pi}K_{-1}^{i}(V^{30}(\alpha_{\gamma}^{t})-V^{30}(\beta_{\gamma}^{t}),V^{61}_{\tilde{f}^{L}}(\alpha_{\gamma}^{t})-V^{61}_{\tilde{f}^{L}}(\beta_{\gamma}^{t}),V_X(\alpha_{\gamma}^{t})-V_X(\beta_{\gamma}^{t}))X_{i'}(\beta_{\gamma}^{t})\\\nonumber
&(\partial_{\alpha}^{b_{i}}\tilde{f}_{v}^{+})(\beta_{\gamma}^{t})(1+ic'(\beta+\tau(t))\gamma t)d\beta\\\nonumber
&=Term_1+Term_2.
\end{align}
We do not change the first term. For the $Term_2$, we separate the integral into the "near $-\tau(t)$ part" and the "far away from $-\tau(t)$" part like in $F_2$. But we do not do the integration by parts and just change the contour directly. Because the term $O^{1,i}$ has less singularity than $F_2$. Here we have
\begin{align*}
    &Term_2\\
    =&-\lambda(\alpha_{\gamma}^{t})p.v.\int_{-\tau(t)}^{\pi}K_{-1}^{i}(V^{30}(\alpha_{\gamma}^{t})-V^{30}(\beta_{\gamma}^{t}),V^{61}_{\tilde{f}^{L}}(\alpha_{\gamma}^{t})-V^{61}_{\tilde{f}^{L}}(\beta_{\gamma}^{t}),V_X(\alpha_{\gamma}^{t})-V_X(\beta_{\gamma}^{t}))X_{i'}(\beta_{\gamma}^{t})\\
&\qquad(\partial_{\alpha}^{b_{i}}\tilde{f}_{v}^{+})(\beta_{\gamma}^{t})(1+ic'(\beta+\tau(t))\gamma t)d\beta\\
&=-\lambda(\alpha_{\gamma}^{t})p.v.\int_{-\tau(t)}^{2\alpha+\tau(t)}K_{-1}^{i}(V^{30}(\alpha_{\gamma}^{t})-V^{30}(\beta_{\gamma}^{t}),V^{61}_{\tilde{f}^{L}}(\alpha_{\gamma}^{t})-V^{61}_{\tilde{f}^{L}}(\beta_{\gamma}^{t}),V_X(\alpha_{\gamma}^{t})-V_X(\beta_{\gamma}^{t}))X_{i'}(\beta_{\gamma}^{t})\\
&\qquad(\partial_{\alpha}^{b_{i}}\tilde{f}_{v}^{+})(\beta_{\gamma}^{t})(1+ic'(\beta+\tau(t))\gamma t)d\beta\\
&\quad+\lambda(\alpha_{\gamma}^{t})\int_{0}^{\gamma}K_{-1}^{i}(V^{30}(\alpha_{\gamma}^{t})-V^{30}((2\alpha+\tau(t))_{\eta}^{t}),V^{61}_{\tilde{f}^{L}}(\alpha_{\gamma}^{t})-V^{61}_{\tilde{f}^{L}}((2\alpha+\tau(t))_{\eta}^{t}),V_{X}(\alpha_{\gamma}^{t})-V_{X}((2\alpha+\tau(t))_{\eta}^{t})\\&\qquad X_{i'}((2\alpha+\tau(t))_{\eta}^{t})(\partial_{\alpha}^{b_{i}}\tilde{f}_{v}^{+})((2\alpha+\tau(t))_{\eta}^{t})(ic(2\alpha+2\tau(t))t)d\eta\\
&\quad-\lambda(\alpha_{\gamma}^{t})\int_{2\alpha+\tau(t)}^{\pi}K_{-1}^{i}(V^{30}(\alpha_{\gamma}^{t})-V^{30}(\beta_{0}^{t}),V^{61}_{\tilde{f}^{L}}(\alpha_{\gamma}^{t})-V^{61}_{\tilde{f}^{L}}(\beta_{0}^{t}),V_X(\alpha_{\gamma}^{t})-V_X(\beta_{0}^{t}))X_{i'}(\beta_{0}^{t})\\
&\qquad(\partial_{\alpha}^{b_{i}}\tilde{f}_{v}^{+})(\beta_{0}^{t})d\beta\\
&=Term_{2,1}+ Term_{2,2} + Term_{2,3}.
\end{align*}
Hence 
\begin{align}\label{o1ic}
O^{1,i}(\tilde{f}^{+(60)})=Term_{1} + Term_{2,1} + Term_{2,2} + Term_{2,3}.
\end{align}
For the $O^{2,i}$ terms, from \eqref{O1inew}, \eqref{O2inew} and \eqref{3oinew}, notice that it is a $O^{1,i}$ type terms times an extra factor $\partial_{\alpha}^{s_i}\tilde{f}_{\mu}^{+}$. we do the same change as $O^{1,i}$ and do not change the factor $\partial_{\alpha}^{s_i}\tilde{f}_{\mu}^{+}$. 
Then we have
\begin{align}\label{3O2ic}
O^{2,i}(\tilde{f}^{+(60)})=O^{1,\tilde{i}}(\tilde{f}^{+(60)})\partial_{\alpha}^{s_i}\tilde{f}_{\mu}^{+},
\end{align}
with $s_i\leq 30$.
For the type $O^{4,i}$, from \eqref{O1inew}, \eqref{O4inew}, \eqref{3oinew} we have
\begin{align}\label{3O4ic}
O^{4,i}(\tilde{f}^{+(60)})=\int_{0}^{1}O^{1,\tilde{i}}_{\zeta}(\tilde{f}^{+(60)})\zeta^{q_i}d\zeta.
\end{align}
Here $O^{1,\tilde{i}}$ means a $O^{1,i}$ type term and $O^{1,\tilde{i}}_{\zeta}$ means a $O^{1,i}$ type term except changing  kernel $K_{-1}^{i}$ to $K_{-1,\zeta}^{i}$ with $q_i\geq 0$.

We do not change $O^{3,i}$ and $O^{5,i}$, from \eqref{O3inew}, \eqref{O5inew}, \eqref{3oinew}, we get
\begin{equation}\label{3O3ic}
\begin{split}
O^{3,i}(\tilde{f}^{+(60)})=&\lambda(\alpha_{\gamma}^{t})\int_{-\pi}^{\pi}K_{0}^{i}(V^{30}(\alpha_{\gamma}^{t})-V^{30}(\beta_{\gamma}^{t}),V^{61}_{\tilde{f}^{L}}(\alpha_{\gamma}^{t})-V^{61}_{\tilde{f}^{L}}(\beta_{\gamma}^{t}),V_X(\alpha_{\gamma}^{t})-V_X(\beta_{\gamma}^{t}))X_{i'}(\beta_{\gamma}^{t})\\
&{(\partial_{\alpha}^{b_i}}\tilde{f}_{\mu}^{+})(\alpha_{\gamma}^{t})(1+ic'(\beta+\tau(t))\gamma t ) d\beta.
\end{split}
\end{equation}
\begin{align}\label{3O5ic}
O^{5,i}(\tilde{f}^{+(60)})=\lambda(\alpha_{\gamma}^{t})\partial_{\alpha}^{b_i}\tilde{f}_{u}^{+}(\alpha_{\gamma}^{t})\tilde{X}_i(\alpha_{\gamma}^t),
\end{align}
with $b_i\leq 60.$
%In conclusion,
%we have
%\begin{align*}
    %\frac{d\tilde{f}_{\mu}^{+(60)}(\alpha_{\gamma}^{t})}{dt}=M_{1,1}(\tilde{f}^{+(60)})-M_{1,2}(\tilde{f}^{+(60)})+F_2(\tilde{f}^{+(60)})+\sum_{i}\lambda(\alpha_{\gamma}^{t}) O^{i}(\tilde{f}^{+}^{(60)})+\lambda(\alpha_{\gamma}^{t})T_{fixed}(\alpha_{\gamma}^{t}).
%\end{align*}
%We claim when $\gamma=0$ the equation holds without the analyticity condition.
\subsection{Conclusion: the modified equation $T$ in the complex plane  }
We now let $\tilde{f}^{+(60)}(\alpha_{\gamma}^{t})=h(\alpha,\gamma,t)$ and write the equation in $h$. The first step is to recover the lower order derivative of $\tilde{f}^{+}$ from $h$.

For any $\tilde{h}\in L_{\al}^{2}[-\tau(t),\pi]$, we could define
\begin{align}\label{D-formularnew}
&\quad D^{-i}(\tilde{h})(\alpha,\gamma,t)\\\nonumber
&=\bigg\{\begin{array}{cc}
         \int_{-\tau(t)}^{\alpha}(1+ic'(\alpha_1+\tau(t))\gamma t) ...\int_{-\tau(t)}^{\alpha_{i-1}}(1+ic'(\alpha_i+\tau(t))\gamma t)h(\alpha_i)d\alpha_i d\alpha_{i-1}...d\alpha_1 &  -\tau(t)< \alpha \leq \pi\\\nonumber
          0 & -\pi\leq\alpha\leq -\tau(t).
    \end{array}
\end{align}
It is clear that
\begin{equation}\label{3d-ipro}
D^{-i}\tilde{h}\in H_{\alpha}^{i}[-\pi,\pi].
\end{equation}
Moreover, for $h(\alpha,\gamma,t) \in C_{\gamma}^{i}([-1,1],L_{\alpha}^{2}[-\tau(t),\pi])$,  we have 
\begin{equation}\label{3d-ipro}
D^{-i}h\in C_{\gamma}^{i}([-1,1], H_{\alpha}^{i}[-\pi,\pi]).
\end{equation}
 Due to the fact that up to 98 order of derivative of $\tilde{f}^{+}$ at $\alpha=-\tau(t)$ is 0,  we claim for all $1\leq i\leq 60$, 
\[
\tilde{f}^{+(60-i)}(\alpha_{\gamma}^{t})=D^{-i}(h)(\alpha,\gamma,t).
\]
Like in \eqref{notation01}, we denote:
\begin{align}\label{notation03h}
V^{k}_{h} =(D^{-60}h_1, D^{-60}h_2,D^{-59}h_1, D^{-59}h_2..., D^{-60+k}h_1, D^{-60+k}h_2).
\end{align}

Then we have the equation for $h$ by changing the corresponding $\tilde{f}^{+(i)}(\alpha_{\gamma}^{t})$ to $D^{-(60-i)}(h)(\alpha,\gamma,t)$ whenever $\gamma\neq 0$. We keep $\tilde{f}^{+}$ if $\gamma=0$ (on the real line). For $\alpha> -\tau(t)$, from \eqref{3mainequationchange}, we have
\begin{align}\label{extendedequation01}
     &\frac{dh(\alpha,\gamma,t)}{dt}=T(h)\\\nonumber
    %=&\lambda(\alpha_{\gamma}^{t})\frac{h_{\mu}(\alpha,\gamma,t)}{d\alpha}B_{D^{-60}(h)}(\alpha,\gamma,t)\\
    %&-\lambda(\alpha_{\gamma}^{t})\int_{-\pi}^{\pi}K(D^{-60}(h)(\alpha,\gamma, t)-D^{-60}(h)(\beta,\gamma, t)+\tilde{f}^{L}(\alpha_{\gamma}^{t})-\tilde{f}^{L}(\beta_{\gamma}^{t}))\frac{d h_{\mu}(\beta,\gamma, t))}{d\beta}d\beta\\
    %&+\lambda(\alpha_{\gamma}^{t})\sum_{i}O^{1,i,c}(h)+\lambda(\alpha_{\gamma}^{t})T_{fixed}(\alpha_{\gamma}^{t})\\
     &=M_{1,1}(h)+M_{1,2}(h)+F_2(h)+\sum_{i}O^{i}(h)+\lambda(\alpha_{\gamma}^{t})T_{fixed}(\alpha_{\gamma}^{t}).
\end{align}
with
\begin{align}\label{3M11chequ}
    (M_{1,1}(h))_{\mu}=\lambda(\alpha_{\gamma}^{t})(\kappa(t)+\tau'(t))\frac{dh_{\mu}(\alpha,\gamma,t)}{d\alpha},
\end{align}
\begin{align}\label{M1chequ}
    (M_{1,2}(h))_{\mu}=\lambda(\alpha_{\gamma}^{t})\frac{dh_{\mu}(\alpha,\gamma,t)}{d\alpha}B_{D^{-60}(h)}(\alpha,\gamma,t),
\end{align}
\begin{align*}
    &\quad B_{D^{-60}(h)}(\alpha,\gamma,t)\\\nonumber
    &=[\frac{1}{1+ic'(\alpha+\tau(t))\gamma t}(ic(\alpha+\tau(t))\gamma+\frac{\partial_{t}x}{\partial_{\alpha}x}(\alpha_{\gamma}^{t})\\&\qquad +\frac{1}{(\partial_{\alpha}x)(\alpha_{\gamma}^{t})}
    \cdot p.v.\int_{-\pi}^{\pi}K(D^{-60}(h)(\alpha,\gamma, t)-D^{-60}(h)(\beta,\gamma, t)+\tilde{f}^{L}(\alpha_{\gamma}^{t})-\tilde{f}^{L}(\beta_{\gamma}^{t}))\frac{dx(\beta_{\gamma}^{t})}{d\beta}d\beta
    -\frac{\partial_{t}x}{\partial_{\alpha}x}(0)\\&\qquad-\frac{1}{(\partial_{\alpha}x)(0)}p.v\int_{-\pi}^{\pi}K(D^{-60}(h)(-\tau(t),\gamma, t)-D^{-60}(h)(\beta,\gamma, t)+\tilde{f}^{L}(0)-\tilde{f}^{L}(\beta_{\gamma}^{t}))\frac{dx(\beta_{\gamma}^{t})}{d\beta}d\beta)]\\
    &\quad+(\frac{1}{1+ic'(\alpha+\tau(t))\gamma t} -1)\kappa(t).
\end{align*}
From \eqref{M2c}, the next term is
\begin{align}\label{M2c02}
&\quad(F_{2}(h))_{\mu}\\\nonumber
&=Term_1 + Term_2 + Term_3 + Term_4 + Term_{5,1} +Term_{5,2}\\\nonumber
&=\underbrace{-\lambda(\al_{\gamma}^{t})K(D^{-60}(h)(\alpha,\gamma, t)-D^{-60}(h)(2\alpha+\tau(t),\gamma, t)+\tilde{f}^{L}(\alpha_{\gamma}^{t})-\tilde{f}^{L}((2\alpha+\tau(t))_{\gamma}^{t}))}_{Term_{1}}\\\nonumber
&\qquad\cdot\underbrace{(h_{\mu}(2\alpha+\tau(t),\gamma,t)-h_{\mu}(\alpha,\gamma,t))}_{Term_{1}}\\\nonumber
&\quad+\underbrace{\lambda(\al_{\gamma}^{t})K(D^{-60}(h)(\alpha,\gamma, t)-D^{-60}(h)(-\tau(t),\gamma, t)+\tilde{f}^{L}(\alpha_{\gamma}^{t})-\tilde{f}^{L}(-\tau(t)))(h_{\mu}(-\tau(t),\gamma, t)-h_{\mu}(\alpha,\gamma, t))}_{Term_2}\\\nonumber
    &\quad\underbrace{+\lambda(\al_{\gamma}^{t})p.v.\int_{-\tau(t)}^{2\alpha+\tau(t)}\frac{d}{d\beta}(K(D^{-60}(h)(\alpha,\gamma, t)-D^{-60}(h)(\beta,\gamma, t)+\tilde{f}^{L}(\alpha_{\gamma}^{t})-\tilde{f}^{L}(\beta_{\gamma}^{t})))(h_{\mu}(\beta,\gamma,t)-h_{\mu}(\alpha,\gamma,t))d\beta}_{Term_3}\\\nonumber
    &\quad+\underbrace{\lambda(\al_{\gamma}^{t})K(D^{-60}(h)(\alpha,\gamma, t)-D^{-60}(h)(2\alpha+\tau(t),\gamma, t)+\tilde{f}^{L}(\alpha_{\gamma}^{t})-\tilde{f}^{L}((2\alpha+\tau(t))_{\gamma}^{t}))h_{\mu}(2\alpha+\tau(t),\gamma,t)}_{Term_4}\\\nonumber
    &\quad-\underbrace{\lambda(\al_{\gamma}^{t})\int_{0}^{\gamma}\frac{\frac{d}{d\beta}(K(D^{-60}(h)(\alpha,\gamma, t)-D^{-60}(h)(\beta,\eta, t)+\tilde{f}^{L}(\alpha_{\gamma}^{t})-\tilde{f}^{L}(\beta_{\eta}^{t})))|_{\beta=2\alpha+\tau(t)}}{1+ic'(2\alpha+2\tau(t))\eta t}h_{\mu}(2\alpha+\tau(t),\eta,t)}_{Term_{5,1}}\\\nonumber
    &\qquad\underbrace{\cdot ic(2\alpha+2\tau(t)) t d\eta}_{Term_{5,1}}\\\nonumber
    &\quad+\underbrace{\lambda(\al_{\gamma}^{t})\int_{2\alpha+\tau(t)}^{\pi}\frac{d}{d\beta}(K(D^{-60}(h)(\alpha,\gamma, t)-\tilde{f}^{+}(\beta_0^ t)+\tilde{f}^{L}(\alpha_{\gamma}^{t})-\tilde{f}^{L}(\beta_{0}^{t})))\tilde{f}_{\mu}^{+(60)}(\beta_0^ t)d\beta}_{Term_{5,2}}.
\end{align}

For the sake of later proof, we could rewrite $F_2(h)$.

We first separate the $Term_{3}$ into the singular part and the non-singular part. We also separate $Term_{5,2}$ into the part depending on h and the part not depending on h, which is 
\begin{equation}\label{3M21equationn}
    \begin{split}
    &\quad(Term_3)_{\mu}\\
    =&-\lambda(\al_{\gamma}^{t})\int_{-\tau(t)}^{2\alpha+\tau(t)}\frac{d}{d\beta}(K(D^{-60}(h)(\alpha,\gamma, t)-D^{-60}(h)(\beta,\gamma, t)+\tilde{f}^{L}(\alpha_{\gamma}^{t})-\tilde{f}^{L}(\beta_{\gamma}^{t})))(h_{\mu}(\alpha,\gamma,t)-h_{\mu}(\beta,\gamma,t))d\beta\\
    =&\underbrace{-\lambda(\al_{\gamma}^{t})\lim_{\beta\to\alpha}(\frac{d}{d\beta}(K(D^{-60}(h)(\alpha,\gamma, t)-D^{-60}(h)(\beta,\gamma, t)+\tilde{f}^{L}(\alpha_{\gamma}^{t})-\tilde{f}^{L}(\beta_{\gamma}^{t})))(\alpha-\beta)^2)}_{M_{2,1}}\\
    &\quad\underbrace{\cdot p.v.\int_{-\tau(t)}^{2\alpha+\tau(t)}\frac{h_{\mu}(\alpha,\gamma,t)-h_{\mu}(\beta,\gamma,t)}{(\alpha-\beta)^2}d\beta}_{M_{2,1}}\\
    &-\lambda(\al_{\gamma}^{t})\int_{-\tau(t)}^{2\alpha+\tau(t)}\frac{1}{\alpha-\beta}[\frac{d}{d\beta}(K(D^{-60}(h)(\alpha,\gamma, t)-D^{-60}(h)(\beta,\gamma, t)+\tilde{f}^{L}(\alpha_{\gamma}^{t})-\tilde{f}^{L}(\beta_{\gamma}^{t})))(\alpha-\beta)^2\\
    &\qquad-\lim_{\beta\to\alpha}(\frac{d}{d\beta}(K(D^{-60}(h)(\alpha,\gamma, t)-D^{-60}(h)(\beta,\gamma, t)+\tilde{f}^{L}(\alpha_{\gamma}^{t})-\tilde{f}^{L}(\beta_{\gamma}^{t})))(\alpha-\beta)^2)]\\&\quad\cdot\frac{h_{\mu}(\alpha,\gamma,t)-h_{\mu}(\beta,\gamma,t)}{(\alpha-\beta)}d\beta\\
    &=M_{2,1}+Term_{3,2}.
    \end{split}
\end{equation}
Then we further split $Term_{3,2}$, $Term_{5,2}$  and have
\begin{equation}
\begin{split}
&\quad Term_{3,2}\\
&=\lambda(\al_{\gamma}^{t})p.v.\int_{-\tau(t)}^{2\alpha+\tau(t)}\frac{1}{\alpha-\beta}[\frac{d}{d\beta}(K(D^{-60}(h)(\alpha,\gamma, t)-D^{-60}(h)(\beta,\gamma, t)+\tilde{f}^{L}(\alpha_{\gamma}^{t})-\tilde{f}^{L}(\beta_{\gamma}^{t})))(\alpha-\beta)^2\\
    &\qquad -\lim_{\beta\to\alpha}(\frac{d}{d\beta}(K(D^{-60}(h)(\alpha,\gamma, t)-D^{-60}(h)(\beta,\gamma, t)+\tilde{f}^{L}(\alpha_{\gamma}^{t})-\tilde{f}^{L}(\beta_{\gamma}^{t})))(\alpha-\beta)^2)]\cdot \frac{h_{\mu}(\alpha,\gamma,t)}{\alpha-\beta}d\beta\\
    &\quad-\lambda(\al_{\gamma}^{t})p.v.\int_{-\tau(t)}^{2\alpha+\tau(t)}\frac{1}{\alpha-\beta}[\frac{d}{d\beta}(K(D^{-60}(h)(\alpha,\gamma, t)-D^{-60}(h)(\beta,\gamma, t)+\tilde{f}^{L}(\alpha_{\gamma}^{t})-\tilde{f}^{L}(\beta_{\gamma}^{t})))(\alpha-\beta)^2\\
    &\qquad-\lim_{\beta\to\alpha}(\frac{d}{d\beta}(K(D^{-60}(h)(\alpha,\gamma, t)-D^{-60}(h)(\beta,\gamma, t)+\tilde{f}^{L}(\alpha_{\gamma}^{t})-\tilde{f}^{L}(\beta_{\gamma}^{t})))(\alpha-\beta)^2)]\frac{h_{\mu}(\beta,\gamma,t)}{(\alpha-\beta)}d\beta\\
    &=Term_{3,2,1}+Term_{3,2,2}.
    \end{split}
\end{equation}
\begin{align*}
&\quad Term_{5,2}\\&=
\lambda(\al_{\gamma}^{t})\int_{2\alpha+\tau(t)}^{\pi}\frac{d}{d\beta}(K(-\tilde{f}^{+}(\beta_0^ t)+\tilde{f}^{L}(\alpha_{\gamma}^{t})-\tilde{f}^{L}(\beta_{0}^{t})))\tilde{f}^{+(60)}_{\mu}(\beta_0^ t)d\beta\\
&\quad+\lambda(\al_{\gamma}^{t})\int_{0}^{1}d\zeta_1 \int_{2\alpha+\tau(t)}^{\pi}\frac{d}{d\zeta_1}\frac{d}{d\beta}(K(\zeta_1 D^{-60}(h)(\alpha,\gamma, t)-\tilde{f}^{+}(\beta_0^ t)+\tilde{f}^{L}(\alpha_{\gamma}^{t})-\tilde{f}^{L}(\beta_{0}^{t})))\tilde{f}_{\mu}^{+(60)}(\beta_0^ t)d\beta\\
&=\underbrace{Term_{5,2,1}}_{\text{does not depend on h}} +Term_{5,2,2}.
\end{align*}
We then call the sum of terms in $F_2(h)$ except the term $M_{2,1}$ "$O^{0}$" . Then we have
\begin{equation}\label{3f2new}
    \begin{split}
        F_2(h)&=M_{2,1}(h)+O^{0}\\
        &=M_{2,1}(h)+Term_1 + Term_2 +Term_{3,2,1}+Term_{3,2,2}+Term_{4}+Term_{5,1}+Term_{5,2,1}+Term_{5,2,2}.
    \end{split}
\end{equation}

 From \eqref{o1ic}, we then have
\begin{align}\label{O1icnewequation}
&\quad(O^{1,i}(h))_{\mu}\\\nonumber
&=Term_{1} + Term_{2,1} + Term_{2,2} + Term_{2,3}.\\\nonumber
&=\lambda(\al_{\gamma}^{t})p.v.\int_{-\pi}^{\pi}K_{-1}^{i}(V^{30}_{h}(\alpha,\gamma,t)-V^{30}_{h}(\beta,\gamma,t),V^{61}_{\tilde{f}^{L}}(\alpha_{\gamma}^{t})-V^{61}_{\tilde{f}^{L}}(\beta_{\gamma}^{t}),V_X(\alpha_{\gamma}^{t})-V_X(\beta_{\gamma}^{t}))X_{i'}(\beta_{\gamma}^{t})\\\nonumber
&\qquad D^{-60+b_i}(h_v)(\alpha,\gamma,t)(1+ic'(\beta+\tau(t))\gamma t)d\beta\\\nonumber
&\quad-\lambda(\al_{\gamma}^{t})p.v.\int_{-\tau(t)}^{2\alpha+\tau(t)}K_{-1}^{i}(V_{h}^{30}(\alpha,\gamma,t)-V_{h}^{30}(\beta,\gamma,t),V^{61}_{\tilde{f}^{L}}(\alpha_{\gamma}^{t})-V^{61}_{\tilde{f}^{L}}(\beta_{\gamma}^{t}),V_X(\alpha_{\gamma}^{t})-V_X(\beta_{\gamma}^{t}))X_{i'}(\beta_{\gamma}^{t})\\\nonumber
&\qquad(D^{-60+b_i}(h_v)(\beta,\gamma,t))(1+ic'(\beta+\tau(t))\gamma t)d\beta\\\nonumber
&\quad+\lambda(\al_{\gamma}^{t})\int_{0}^{\gamma}K_{-1}^{i}(V^{30}_h(\alpha,\gamma,t)-V^{30}_h(2\alpha+\tau(t),\eta,t),V^{61}_{\tilde{f}^{L}}(\alpha_{\gamma}^{t})-V^{61}_{\tilde{f}^{L}}((2\alpha+\tau(t))_{\eta}^{t}),V_{X}(\alpha_{\gamma}^{t})-V_{X}((2\alpha+\tau(t))_{\eta}^{t})\\\nonumber&\qquad X_{i'}((2\alpha+\tau(t))_{\eta}^{t})(D^{-60+b_i}(h_v)(2\alpha+\tau(t),\eta,t))(ic(2\alpha+2\tau(t))t)d\eta\\\nonumber
&\quad-\lambda(\al_{\gamma}^{t})\int_{2\alpha+\tau(t)}^{\pi}K_{-1}^{i}(V^{30}_h(\alpha,\gamma,t)-V^{30}(\beta_{0}^{t}),V^{61}_{\tilde{f}^{L}}(\alpha_{\gamma}^{t})-V^{61}_{\tilde{f}^{L}}(\beta_{0}^{t}),V_X(\alpha_{\gamma}^{t})-V_X(\beta_{0}^{t}))X_{i'}(\beta_{0}^{t})\\\nonumber
&\qquad(\partial_{\alpha}^{b_{i}}\tilde{f}_{v}^{+})(\beta_{0}^{t})d\beta.
\end{align}
We also separate the $Term_{2,3}$ as the $Term_{5,2}$ in $F_2$.
We have 
\begin{align*}
    &\quad Term_{2,3}\\
    &=-\lambda(\al_{\gamma}^{t})\int_{2\alpha+\tau(t)}^{\pi}K_{-1}^{i}(-V^{30}(\beta_{0}^{t}),V^{61}_{\tilde{f}^{L}}(\alpha_{\gamma}^{t})-V^{61}_{\tilde{f}^{L}}(\beta_{0}^{t}),V_X(\alpha_{\gamma}^{t})-V_X(\beta_{0}^{t}))X_{i'}(\beta_{0}^{t})(\partial_{\alpha}^{b_{i}}\tilde{f}_{v}^{+})(\beta_{0}^{t})d\beta\\
&\quad -\lambda(\al_{\gamma}^{t})\int_{0}^{1}d\zeta_0\int_{2\alpha+\tau(t)}^{\pi}\frac{d}{d\zeta_0}K_{-1}^{i}(\zeta_0V^{30}_h(\alpha,\gamma,t)-V^{30}(\beta_{0}^{t}),V^{61}_{\tilde{f}^{L}}(\alpha_{\gamma}^{t})-V^{61}_{\tilde{f}^{L}}(\beta_{0}^{t}),V_X(\alpha_{\gamma}^{t})-V_X(\beta_{0}^{t}))X_{i'}(\beta_{0}^{t})\\
&\quad(\partial_{\alpha}^{b_{i}}\tilde{f}_{v}^{+})(\beta_{0}^{t})d\beta\\
&=\underbrace{Term_{2,3,1}}_{\text{does not depend on h}} +Term_{2,3,2}.
\end{align*}
Hence we get
\begin{align}\label{O1icnew}
  (O^{1,i}(h))_{\mu}=Term_{1} + Term_{2,1} + Term_{2,2} + Term_{2,3,1}+Term_{2,3,2}.  
\end{align}
For $O^{2,i}$ type terms, from \eqref{3O2ic}, we have
\begin{align}\label{O2ic}
(O^{2,i}(h))_{\mu}=O^{1,\tilde{i}}(h)D^{-60+s_i}h_{\mu}(\alpha,\gamma,t),
\end{align}
with $s_i\leq 30$.
For $O^{4,i}$ type terms, from \eqref{3O4ic}, we have
\begin{align}\label{O4ic}
(O^{4,i}(h))_{\mu}=\int_{0}^{1}(O^{1,\tilde{i}}_{\zeta}(h))_{\mu}\zeta^{q_i}d\zeta.
\end{align}
Here $O^{1,\tilde{i}}$ means a $O^{1,i}$ type term and $O^{1,\tilde{i}}_{\zeta}(h)$ means a $O^{1,i}$ type term except changing kernel $K_{-1}^{i}$ to $K_{-1,\zeta}^{i}$, 
 and $q_i\geq 0$.

From \eqref{3O3ic}, \eqref{3O5ic}, we have
\begin{align}\label{O3ic}
&(O^{3,i}(h))_{\mu}=\lambda(\al_{\gamma}^{t})\int_{-\pi}^{\pi}K_{0}^{i}(V^{30}_{h}(\alpha, \gamma,t)-V^{30}_h(\beta,\gamma,t),V^{61}_{\tilde{f}^{L}}(\alpha_{\gamma}^{t})-V^{61}_{\tilde{f}^{L}}(\beta_{\gamma}^{t}),V_X(\alpha_{\gamma}^{t})-V_X(\beta_{\gamma}^{t}))X_{i'}(\beta_{\gamma}^{t})\\\nonumber
&(1+ic'(\beta+\tau(t))\gamma t ) d\beta D^{-60+b_i}(h_{\mu})(\alpha,\gamma,t),
\end{align}
\begin{align}\label{O5ic}
    &(O^{5,i}(h))_{\mu}=\lambda(\al_{\gamma}^{t})D^{-60+b_i}(h_{\mu})(\alpha,\gamma,t)\tilde{X}_i(\alpha_{\gamma}^t),
\end{align}
with $b_i\leq 60$.

In conclusion, we have for $-\tau(t)<\al<\pi$,
\begin{align}\label{modifiedequation}
    \frac{dh(\alpha,\gamma,t)}{dt}=T(h)=
       M_{1,1}(h)+M_{1,2}(h)+M_{2,1}(h)+O^{0}+\sum_{i}O^{i}(h)+\lambda(\alpha)T_{fixed}(\alpha_{\gamma}^{t}) 
\end{align}
with the initial value $h_{\mu}(\alpha,\gamma,0)=\tilde{f}_{\mu}^{+(60)}(\alpha,0)$.
 Notice that all the terms that do not depend on $h$ are well defined even without the analyticity assumption on f.
 Moreover, $\tilde{f}^{+(60)}(\alpha+\tau(t),t)$ satisfies the  $\eqref{modifiedequation}$ when $\gamma=0$.

 We also remark that \eqref{extendedequation01} and \eqref{modifiedequation} are the same equation but written in different ways. We can use either form given that $h$ satisfies one of them.
 
 Now we will separate the case into $\kappa(t)>0$
and $\kappa(t)<0$. %They are corresponding to different linear model in \cite{existence1theorem}.%
We first deal with $\kappa(t)>0$, in this case by our def \eqref{defkappa}, we have $\tau(t)=0$.
\section{Behavior of the modified equation for $\kappa(t)>0$}\label{kappa1sectiongene}
In this section we show $T$ in \eqref{modifiedequation} satisfying the following generalized equation \eqref{3GE01}
when $\kappa(t)>0$. 

\subsection{A generalized equation when $\kappa(t)>0$}
We introduce the following equation:
\begin{equation}\label{3GE01}
\begin{split}
&\frac{dh(\alpha,\g,t)}{dt}=T^{+}(h)\\
&=M_{1,1}(h)+M_{1,2}(h)+M_{2,1}(h)+ \sum_{i}B_i,
\end{split}
\end{equation}
with initial data satisfying
\begin{equation}\label{3g0}
h(\al,\g,0)=h(\al,\g',0).
\end{equation}
Here we abuse the notation and denote:
\begin{equation}\label{M11def}
\begin{split}
M_{1,1}(h)=\lambda(\alpha)\kappa(t)\partial_{\al} h(\al,\g,t),
\end{split}
\end{equation}
with $\kappa(t)>0$.
\begin{equation}\label{M12def}
\begin{split}
M_{1,2}(h)=\lambda(\alpha)L_1^{+}(h)(\alpha,\gamma,t)\partial_{\alpha}h(\alpha,\gamma,t),
\end{split}
\end{equation}
\begin{equation}
\begin{split}\label{M21def}
M_{2,1}(h)=\lambda(\alpha)L_2^{+}(h)(\alpha,\gamma,t)\int_{0}^{2\alpha}\frac{h(\alpha,\g,t)-h(\beta,\g,t)}{(\alpha-\beta)^2} d\beta,
\end{split}
\end{equation}
with $L_1^{+}(h)$, $L_2^{+}(h)$ ,$B_i(h)$ satisfying the following conditions \label{3conditionsA}:

Let 
\begin{equation}\label{3Xformular}
X^{k}=C_{\gamma}^{0}([-1,1],H_{\alpha}^{k}[0,\pi])\cap \{h|\supp h \subseteq  [0,\frac{\pi}{4}]\},
Y^{k+2}=C_{\gamma}^{0}([-1,1],C_{\alpha}^{k+2}[0,\frac{\pi}{4}])).
\end{equation}
%\begin{equation}\label{3Yformular}
%Y^{k}=C_{\gamma}^{0}([-1,1],C_{\alpha}^{k+2}[0,\pi])\cap\%{h|\supp h \subseteq  [0,\frac{\pi}{2}]\}.
%\end{equation}
For $\delta$ in \eqref{fLequation} sufficiently small, there exists $\delta_{s}>0$, $t_{s}>0$, sufficiently small such that for any $1\leq k \leq 12$, if $h,g,\tilde{g},t$ satisfy 
\[
h,g,\tilde{g} \in X^{k}
\]
\begin{equation}\label{3hsmall}
\|h(\alpha,\gamma,t)-h(\al,\g,0))\|_{X^{1}}\leq \delta_{s},
\end{equation}
\[0\leq t\leq t_{s},
\]
\[
h(\al,\g,0)=\tilde{f}^{+(60)}(\al,0),
\]
 then $L_{i}^{+}(h)$, $B_i(h)$ satisfying the following conditions: refined R-T conditions, vanishing conditions and smoothness conditions. 

\quad\textbf{Refined R-T conditions}:
\[
 \text{ when }\al\in \supp{\lambda}, \quad 18 |\Im L_{1}^{+}(h)(\al,\g,t)|+18 |\Im L_{2}^{+}(h)(\al,\g,t)|\leq -\Re L_2^{+}(h)(\al,\g,t). 
\]

\quad\textbf{Vanishing condition}:
\[
L_{i}^{+}(h)(0,\g,t)=0, 
\]
\quad\textbf{Smoothness conditions}:
\begin{enumerate}
    \item 
    \begin{equation}\label{boundestimateL}
    \|L_{i}^{+}(h)(\alpha,\g,t)\|_{Y^{k+2}}\lesssim 1,   
    \end{equation}
    %\item $\|\bar{\partial}_{\g}L_{i}^{+}(h)(\alpha,\g,t)\|_{Y^{k+2}}\lesssim 1$,  $\|\bar{\partial}_{\g}^4L_{i}^{+}(h)(\alpha,\g,t)\|_{Y^{k+2}}\lesssim 1$,\label{3Lgammaderi} 
    \item  $\|D_{h}L_{i}^{+}(h)(\alpha,\g,t)[g]\|_{Y^{k+2}}\lesssim \|g\|_{X^{k}}$, %$\|D_{h}^2L_{i}^{+}(h,\g,t)[g,\tilde{g}]\|_{Y^{k+2}}\lesssim \|g\|_{X^{k}}\|\tilde{g}\|_{X^{k}}$, \label{3Ldh2}
     \item $\|L_{i}^{+}(h)(\alpha,\gamma,t)-L_{i}^{+}(h)(\alpha,\gamma,t')\|_{Y^{k+2}}\lesssim \mathcal{O}(t-t')+\|h(\alpha,\gamma,t)-h(\alpha,\gamma,t')\|_{X^{k}},$
     
     with $\lim_{x\to 0}\mathcal{O}(x)=0$ independent of $h$.
    \item There exists operator $D_\gamma(L_i^{+}(h))$ such that for $w(\alpha,\gamma,t), \tilde{w}(\alpha,\gamma,t)\in X^{k}$, 
    \[
    D_{\g}(L_{i}^{+}(h))[w]\in Y^{k+2},
    \]
     and 
     \begin{align}\label{dgammaestimateLi}
     &\|\frac{L_{i}^{+}(h)(\al,\g,t)-L_{i}^{+}(h)(\al,\g',t)}{\g-\g'}-D_{\g}(L_{i}^{+}(h))[w](\al,\g,t)\|_{C_{\al}^{k+2}[0,\frac{\pi}{4}]}\\\nonumber
     &\lesssim \|\frac{h(\al,\g,t)-h(\al,\g',t)}{\g-\g'}-w(\al,\g,t)\|_{H_{\al}^{k}[0,\pi]}+|\g-\g'|+\|h(\al,\g)-h(\al,\g',t)\|_{H_{\al}^{k}[0,\pi]},
     \end{align}
     and when $\frac{d}{d\g}h \in X^{k}$, we have
     \[
     \frac{d}{d\g}L_{i}^{+}(h)(\al,\g,t)=D_{\g}(L_{i}^{+}(h))[\frac{d}{d\g}h].
     \]
     
    \item $\|D_{\g}(L_{i}^{+}(h))[w](\al,\g,t)-D_{\g}(L_{i}^{+}(h))[\tilde{w}](\al,\g,t)\|_{Y^{k+2}}\lesssim\|w(\al,\g,t)-\tilde{w}(\al,\g,t)\|_{X^{k}}$,
    \item \begin{align*}
    &\quad\|D_{\g}(L_{i}^{+}(h))[w](\al,\g',t)-D_{\g}(L_{i}^{+}(h))[w](\al,\g,t)\|_{C_{\al}^{k+2}[0,\frac{\pi}{4}]}\\
    &\lesssim\|w(\al,\g,t)-w(\al,\g',t)\|_{H_{\al}^{k}[0,\pi]}+|\g-\g'|+\|h(\al,\g,t)-h(\al,\g',t)\|_{H_{\al}^{k}[0,\pi]},
    \end{align*}
    \item 
    \begin{align*}
    &\quad\|D_{\g}(L_{i}^{+}(h))[w](\al,\g,t')-D_{\g}(L_{i}^{+}(h))[w](\al,\g,t)\|_{C_{\al}^{k+2}[0,\frac{\pi}{4}]}\\
    &\lesssim\|w(\al,\g,t)-w(\al,\g,t')\|_{H_{\al}^{k}[0,\pi]}+ \mathcal{O}(t-t')+\|h(\al,\g,t)-h(\al,\g,t')\|_{H_{\al}^{k}[0,\pi]},
    \end{align*}
     \item $\supp_{\al}B_i(h)\subset [0,\frac{\pi}{4}].$\label{3Bsupp}
\item \begin{equation}\label{boundestimateB}
    \|B_i(h)(\alpha,\g,t)\|_{X^{k}}\lesssim 1,   
    \end{equation}
    %\item $\|\bar{\partial}_{\g}B_i(h)(\alpha,\g,t)\|_{X^{k}}\lesssim 1$,  $\|\bar{\partial}_{\g}^4 B_{i}(h)(\al,\g,t)\|_{X^{k}}\lesssim 1$,\label{3Bgammaderi} 
    \item  $\|D_{h}B_i(h)(\alpha,\g,t)[g]\|_{X^{k}}\lesssim \|g\|_{X^{k}}$,  %$\|D_{h}^2 B_{i}(h)(\al,\g,t)[g,\tilde{g}]\|_{X^{k}}\lesssim \|g\|_{X^{k}}\|\tilde{g}\|_{X^{k}},$ \label{3Bdh2}
     %\item  $\|D_{h}\bar{\partial}_{\g}B_i(h)(\alpha,\g,t)[g]\|_{X^{k}}\lesssim \|g\|_{X^{k}}$,
         \item$
    \|B_{i}(h)(\alpha,\gamma,t)-B_{i}(h)(\alpha,\gamma,t')\|_{X^{k}}
    \lesssim \mathcal{O}(t-t')+\|h(\alpha,\gamma,t)-h(\alpha,\gamma,t')\|_{X^{k}}.$
     \item There exists operator $D_\gamma(B_i(h))$ such that for $w\in X^{k}$, 
    \[
    D_{\g}(B_i(h))[w]\in X^{k},
    \]
     and 
     \begin{align}\label{dgammaestimateBi}
     &\quad\|\frac{B_i(h)(\al,\g)-B_i(h)(\al,\g')}{\g-\g'}-D_{\g}(B_i(h))[w](\al,\g)\|_{H_{\al}^{k}[0,\frac{\pi}{4}]}\\\nonumber
     &\lesssim \|\frac{h(\al,\g)-h(\al,\g')}{\g-\g'}-w(\al,\g)\|_{H_{\al}^{k}[0,\pi]}+|\g-\g'|+\|h(\al,\g)-h(\al,\g')\|_{H_{\al}^{k}[0,\pi]},
     \end{align}
     and when $\frac{d}{d\g}h \in X^{k}$, we have
     \[
     \frac{d}{d\g}B_i(h)(\al,\g)=D_{\g}(B_i(h))[\frac{d}{d\g}h]
     \]
    \label{3Bdh2ga}
    \item$
    \|B_{i}(h)(\alpha,\gamma,t)-B_{i}(h)(\alpha,\gamma,t')\|_{X^{k}}
    \lesssim \mathcal{O}(t-t')+\|h(\alpha,\gamma,t)-h(\alpha,\gamma,t')\|_{X^{k}}.$
       \item $\|D_{\g}(B_{i}(h))[w](\al,\g,t)-D_{\g}(B_{i}(h))[\tilde{w}](\al,\g,t)\|_{X^{k}}\lesssim\|w(\al,\g,t)-\tilde{w}(\al,\g,t)\|_{X^{k}}$,
    \item \begin{align*}
    &\quad\|D_{\g}(B_{i}(h))[w](\al,\g',t)-D_{\g}(B_{i}(h))[w](\al,\g,t)\|_{H_{\al}^{k}[0,\frac{\pi}{4}]}\\
    &\lesssim\|w(\al,\g,t)-w(\al,\g',t)\|_{H_{\al}^{k}[0,\pi]}+|\g-\g'|+\|h(\al,\g,t)-h(\al,\g',t)\|_{H_{\al}^{k}[0,\pi]},
    \end{align*}
    \item 
    \begin{align*}
    &\quad\|D_{\g}(B_{i}(h))[w](\al,\g,t')-D_{\g}(B_{i}(h))[w](\al,\g,t)\|_{H_{\al}^{k}[0,\frac{\pi}{4}]}\\
    &\lesssim\|w(\al,\g,t)-w(\al,\g,t')\|_{H_{\al}^{k}[0,\pi]}+ \mathcal{O}(t-t')+\|h(\al,\g,t)-h(\al,\g,t')\|_{H_{\al}^{k}[0,\pi]},
    \end{align*}
\end{enumerate}

Moreover, when $\g=0$. we have $L_i^{+}(h)(\al,0,t)=\bar{L}_i^{+}(h(\al,0,t),t)$ maps $(H_{\al}^{k}[0,\pi]\cap\{h|\supp h\subset [0,\frac{\pi}{4}]\})\times [0,t_s] \to C_{\al}^{k+2}[0,\frac{\pi}{4}]$, $B_i(h)(\al,0,t)=\bar{B}_i(h(\al,0,t),t)$ maps $(H_{\al}^{k}[0,\pi]\cap\{h|\supp h\subset [0,\frac{\pi}{4}]\})\times [0,t_s] \to H_{\al}^{k}[0,\pi]$, and satisfy
\begin{enumerate}
\item $\|\bar{L}_i^{+}(h,t)\|_{C_{\al}^{k+2}[0,\frac{\pi}{4}]}\lesssim 1$,\label{3barl0}
\item $\|D_h\bar{L}_i^{+}(h,t)[g]\|_{C_{\al}^{k+2}[0,\frac{\pi}{4}]}\lesssim \|g\|_{H_{\al}^{k}[0,\pi]}$,\label{3barl0lip}
\item $\|\bar{B}_i^{+}(h,t)\|_{H_{\al}^{k}[0,\pi]}\lesssim 1$,\label{3barb0}
\item $\|D_h\bar{B}_i^{+}(h,t)[g]\|_{H_{\al}^{k}[0,\pi]}\lesssim \|g\|_{H_{\al}^{k}[0,\pi]}$\label{3barblip}.
\end{enumerate}

\begin{rem}
    Roughly speaking, those conditions show $L_i^{+}(h)$ vanishes at $\alpha=0$, satisfies the refined R-T condition, and are sufficiently smooth with respect to $\g,t,h$ in $C_{\al}^{k+2}$. $B_i(h)$ are sufficiently smooth with respect to $\g,t,h$ in $H_{\al}^{k}$. $L_i^{+}(h)|_{\g=0}$, $B_i(h)|_{\g=0}$ only depend on $h|_{\g=0}$.
\end{rem}

\subsection{Analysis of the modified equation}
In this section our goal is to show the following theorem:

\begin{theorem}\label{theoremgeneralizedequa}
The modified equation \eqref{modifiedequation} satisfies the generalized equation \eqref{3GE01} when $\kappa(t)>0$. $\delta$, $\delta_{s}$, $t_s$ are chosen to fit the arc-chord condition in lemma \ref{arcchord}, the refined R-T condition in lemma \ref{rfR-T}.
\end{theorem}

For the sake of simplicity, We use $\Delta g(\alpha,\gamma,t)$ to show $g(\alpha,\gamma,t)-g(\beta,\gamma,t)$ and also use notations in \eqref{notation01}, \eqref{notation02}, \eqref{Xifunction} and \eqref{notation03h}. If we take $A=\Delta V_h^{30}(\alpha,\gamma,t)$, $B=\Delta V_{\tilde{f}^{L}}^{61}(\alpha_{\gamma}^{t},t)$, $C=\Delta V_{X}(\alpha_{\gamma}^{t},t)$, we have 
\begin{align*}
    &K_{-\sigma}^{j}(\Delta V_h^{30}(\alpha,\gamma,t),\Delta V_{\tilde{f}^{L}}^{61}(\alpha_{\gamma}^{t},t),\Delta V_{X}(\alpha_{\gamma}^{t},t))=\\
    &K_{-\sigma}^{j}(V_h^{30}(\alpha,\gamma,t)-V_h^{30}(\beta,\gamma,t), V_{\tilde{f}^{L}}^{61}(\alpha_{\gamma}^{t},t)-V_{\tilde{f}^{L}}^{61}(\beta_{\gamma}^{t},t),V_{X}(\alpha_{\gamma}^{t},t)-V_{X}(\beta_{\gamma}^{t},t)).
\end{align*}
We also let $\bar{\partial}_{\gamma}$ be the partial derivative with respect to $\gamma$ assuming $h$ does not depend on $\gamma.$

Now we separate the terms in \eqref{modifiedequation} in several types. They correspond to the terms in the generalized equation \eqref{3GE01}. We note that terms with the same notation eg.($M_{1,1}(h), M_{1,2}(h)$) will always correspond to each other.
\begin{lemma}\label{3maintermtype}
When $\kappa(t)>0$, we have $\al_{\gamma}^{t}=\al+ic(\al)\gamma t.$ Moreover, $\lambda(\alpha_{\gamma}^{t})=\lambda(\alpha)$, $M_{1,1}(h)$ in \eqref{modifiedequation} satisfies:
\begin{align}\label{3M11chequG01}
    M_{1,1}(h)=\lambda(\al_{\gamma}^{t})\kappa(t)\frac{dh(\alpha,\gamma,t)}{d\alpha}.
\end{align}

$M_{1,2}$ satisfies
\begin{equation}\label{3M12chequG01}
    M_{1,2}(h)=\lambda(\al_{\gamma}^{t})L_1^{+}(h)(\al,\g,t)\partial_{\al}h(\al,\g,t),
\end{equation}

with  \begin{align}\label{3L1form01}
    &\quad L_{1}^{+}(h)(\alpha,\gamma,t)\\\nonumber
    &=B_{D^{-60}(h)}(\alpha,\gamma,t)=[\frac{1}{1+ic'(\alpha)\gamma t}(ic(\alpha)\gamma +\frac{\partial_{t}x}{\partial_{\alpha}x}(\alpha_{\gamma}^{t},t)\\\nonumber&\quad+\frac{1}{(\partial_{\alpha}x)(\alpha_{\gamma}^{t},t)}
    \cdot p.v.\int_{-\pi}^{\pi}K_{-1}^{i}(\Delta V_h^{30}(\alpha,\gamma,t),\Delta V_{\tilde{f}^{L}}^{61}(\alpha_{\gamma}^{t},t),\Delta V_{X}(\alpha_{\gamma}^{t},t))X_i(\beta_{\gamma}^{t},t)(1+ic'(\beta)\gamma t)d\beta
    -\frac{\partial_{t}x}{\partial_{\alpha}x}(0,t)\\\nonumber
    &\quad -\frac{1}{(\partial_{\alpha}x)(0,t)}p.v\int_{-\pi}^{\pi}K_{-1}^{i}(\Delta V_h^{30}(0,\gamma,t),\Delta V_{\tilde{f}^{L}}^{61}(0,t),\Delta V_{X}(0,t))X_i(\beta_{\gamma}^{t},t)(1+ic'(\beta)\gamma t)d\beta)]\\
       &\quad+(\frac{1}{1+ic'(\alpha)\gamma t} -1)\kappa(t),
\end{align}
for some $K_{-1}$ type kernel. $M_{2,1}$ satisfies
\begin{equation}\label{3M21chequG01}
    M_{2,1}(h)=\lambda(\al_{\gamma}^{t})L_2^{+}(h)(\al,\g,t)\int_{0}^{2\alpha}\frac{(h(\alpha,\gamma,t)-h(\beta,\gamma,t))}{(\alpha-\beta)^2}d\beta,
\end{equation}
with 
\begin{align*}
    L_2^{+}(h)(\alpha,\gamma,t)=-\lim_{\beta\to\alpha}(\frac{d}{d\beta}K_{-1}^{j}(\Delta V_h^{30}(\alpha,\gamma,t),\Delta V_{\tilde{f}^{L}}^{61}(\alpha_{\gamma}^{t},t),\Delta V_{X}(\alpha_{\gamma}^{t},t))(\alpha-\beta)^2)
\end{align*}
with particular $K_{-1}^{j}(A,B,C)=\frac{\sin(A_1+B_1)}{\cosh(A_2+B_2)-\cos(A_1+B_1)}$.

\end{lemma}
\begin{proof}
   It directly follows from the definition of $\alpha_{\gamma}^{t}$ \eqref{alphadefi}, $\tau(t)$ \eqref{defkappa}, $\lambda(\alpha)$ \eqref{lambdadefi}, $c(\alpha)$ \eqref{cdefi}, the $M_{1,2}$\eqref{M1chequ},  $M_{1,1}$\eqref{3M11chequ} and $M_{2,1}$\eqref{3M21equationn}, kernel $K$ \eqref{kdefi}.
\end{proof}

\begin{lemma}\label{3boundedtermtype}
When $\kappa(t)>0$, we have  $\al_{\gamma}^{\tau}=\al+ic(\al)\gamma t.$ Terms in  $O^{0}$, $O^{1,i}$, $O^{3,i}$ and  $O^{5,i}$  can be written as the sum of the following types with $b_i\leq 60$:

\textbf{The first type}:
\begin{align}\label{firsttypeestimate}
(B_1^{i}(h))_{\mu}=
 \lambda(\alpha_{\gamma}^{t})p.v.\int_{0}^{2\alpha}  \tilde{B}_{1}^{i}(h)(\alpha,\beta,\gamma,t)\frac{D^{-60+b_i}h_v(\beta,\gamma,t)}{(\alpha-\beta)}d\beta,
\end{align}

with $\tilde{B}_1^{i}(h)$ one of the following two types:
sub-type 1:
\begin{align*}
    &\tilde{B}_1^{i}(h)(\alpha,\beta,\gamma,t)=((\alpha-\beta)^2\frac{d}{d\beta}(K_{-1}^{i}(\Delta V_{h}^{30}(\alpha,\gamma,t),\Delta V_{\tilde{f}^{L}}^{61}(\alpha_{\gamma}^{t},t),\Delta V_{X}(\alpha_{\gamma}^{t},t)))X_i(\beta_{\gamma}^{t},t)\\
    &-\lim_{\beta\to\alpha}((\alpha-\beta)^2\frac{d}{d\beta}(K_{-1}^{i}(\Delta V_{h}^{30}(\alpha,\gamma,t),\Delta V_{\tilde{f}^{L}}^{61}(\alpha_{\gamma}^{t},t),\Delta V_{X}(\alpha_{\gamma}^{t},t)))X_i(\alpha_{\gamma}^{t},t))\frac{1}{\alpha-\beta}
\end{align*}
or the form sub-type 2
\begin{align*}
    &\tilde{B}_1^{i}(h)(\alpha,\beta,\gamma,t)=((\alpha-\beta)(K_{-1}^{i}(\Delta V_{h}^{30}(\alpha,\gamma,t),\Delta V_{\tilde{f}^{L}}^{61}(\alpha_{\gamma}^{t},t),\Delta V_{X}(\alpha_{\gamma}^{t},t)))X_i(\beta_{\gamma}^{t},t)(1+ic'(\beta)\gamma t).
\end{align*}

\textbf{The second type}:
\begin{align}\label{secondtypeestimate}
(B_2^{i}(h))_{\mu}=
 \lambda(\alpha_{\gamma}^{t})\tilde{B}_2^{i}(h)(\alpha,\gamma,t)D^{-60+b_i}h_v(\alpha,\gamma,t) \text{ or } \lambda(\alpha_{\gamma}^{t})\tilde{B}_2^{i}(h)(\alpha,\gamma,t)D^{-60+b_i}h_v(2\alpha,\gamma,t),\\\nonumber
 \text{ or } \lambda(\alpha_{\gamma}^{t})\tilde{B}_2^{i}(h)(\alpha,\gamma,t)D^{-60+b_i}h_v(0,\gamma,t)
\end{align}

with $\tilde{B}_2^{i}(h)$ having the following six forms:
sub-type 1:
\begin{align*}
  \tilde{B}_2^{i}(h)(\alpha,\gamma,t)=p.v. \int_{-\pi}^{\pi}K_{-1}^{i}(\Delta V_h^{30}(\alpha,\gamma,t),\Delta V_{\tilde{f}^{L}}^{61}(\alpha_{\gamma}^{t},t),\Delta V_{X}(\alpha_{\gamma}^{t},t))X_i(\beta_{\gamma}^{t},t)(1+ic'(\beta)\gamma t)d\beta,
\end{align*}
sub-type 2 (with particular $K_{-1}^{i}(A,B,C)=\frac{\sin(A_1+B_1)}{\cosh(A_2+B_2)-\cos(A_1+B_1)}$)
\begin{align*}
    \tilde{B}_2^{i}(h)(\alpha,\gamma,t)=K_{-1}^{i}(\Delta V_h^{30}(\alpha,\gamma,t),\Delta V_{\tilde{f}^{L}}^{61}(\alpha_{\gamma}^{t},t),\Delta V_{X}(\alpha_{\gamma}^{t},t)|_{\beta=2\alpha},
\end{align*}
or 
\begin{align*}
    \tilde{B}_2^{i}(h)(\alpha,\gamma,t)=K_{-1}^{i}(\Delta V_h^{30}(\alpha,\gamma,t),\Delta V_{\tilde{f}^{L}}^{61}(\alpha_{\gamma}^{t},t),\Delta V_{X}(\alpha_{\gamma}^{t},t)|_{\beta=0},
\end{align*}

sub-type 3
\begin{align*}
    & \tilde{B}_2^{i}(h)(\alpha,\gamma,t)=\int_{0}^{1}d\zeta\int_{2\alpha}^{\pi}(\frac{d}{d\beta}(K_{-2}^{i}( \zeta V_h^{30}(\alpha,\gamma,t)-V^{30}(\beta_0^{t},t),\Delta V_{\tilde{f}^{L}}^{61}(\alpha_{\gamma}^{t},t),\Delta V_{X}(\alpha_{\gamma}^{t},t)))X_{i}(\beta_{0}^{t},t)\tilde{f}^{+(b_i)}(\beta_0^{t})d\beta,
\end{align*}
sub-type 4
\begin{align*}
    &\tilde{B}_2^{i}(h)(\alpha,\gamma,t)=\int_{0}^{1}d\zeta\int_{2\alpha}^{\pi}(K_{-2}^{i}( \zeta V_h^{30}(\alpha,\gamma,t)-V^{30}(\beta_0^{t},t),\Delta V_{\tilde{f}^{L}}^{61}(\alpha_{\gamma}^{t},t),\Delta V_{X}(\alpha_{\gamma}^{t},t))X_{i}(\beta_{0}^{t},t)\tilde{f}^{+(b_i)}(\beta_0^{t})d\beta,
\end{align*}
sub-type 5
\begin{align*}
    &\tilde{B}_2^{i}(h)(\alpha,\gamma,t)=p.v.\int_{0}^{2\al}((\alpha-\beta)^2\frac{d}{d\beta}(K_{-1}^{i}(\Delta V_{h}^{30}(\alpha,\gamma,t),\Delta V_{\tilde{f}^{L}}^{61}(\alpha_{\gamma}^{t},t),\Delta V_{X}(\alpha_{\gamma}^{t},t)))\\
    &-\lim_{\beta\to\alpha}((\alpha-\beta)^2\frac{d}{d\beta}(K_{-1}^{i}(\Delta V_{h}^{30}(\alpha,\gamma,t),\Delta V_{\tilde{f}^{L}}^{61}(\alpha_{\gamma}^{t},t),\Delta V_{X}(\alpha_{\gamma}^{t},t))))\frac{1}{(\alpha-\beta)^2}d\beta,
\end{align*}
sub-type 6
\begin{align*}
\tilde{B}_2^{i}(h)(h)(\alpha,\gamma,t)=\tilde{X}_i(\alpha_{\gamma}^{t},t).
\end{align*}

\textbf{The third type}:
\begin{align}\label{thirdtypeestimate}
(B_{3}^{i}(h))_{\mu}=\lambda(\alpha_{\gamma}^{t})\int_{0}^{\gamma}D^{-60+b_i}h_v(2\alpha,\eta,t)\tilde{B}_3^{i}(h)(\alpha,\gamma,\eta,t)d\eta.
\end{align}
with $\tilde{B}_3^{i}(h)$ having the following two forms:
sub-type 1
\begin{align*}
    &\tilde{B}_3^{i}(h)(\alpha,\gamma,\eta,t)=\frac{d}{d\beta}(K_{-1}^{i}( V_h^{30}(\alpha,\gamma,t)-V_h^{30}(\beta,\eta,t), V_{\tilde{f}^{L}}^{61}(\alpha_{\gamma}^{t},t)-V_{\tilde{f}^{L}}^{61}(\beta_{\eta}^{t},t), V_{X}(\alpha_{\gamma}^{t},t)-V_{X}(\beta_{\eta}^{t},t))|_{\beta=2\alpha}\\
    &\times \frac{ic(2\alpha)t}{1+ic'(2\alpha)\eta t}.
\end{align*}
and sub-type 2
\begin{align*}
    &\tilde{B}_3^{i}(h)(\alpha,\gamma,\eta,t)=K_{-1}^{i}( V_h^{30}(\alpha,\gamma,t)-V_h^{30}(\beta,\eta,t), V_{\tilde{f}^{L}}^{61}(\alpha_{\gamma}^{t},t)-V_{\tilde{f}^{L}}^{61}(\beta_{\eta}^{t},t), V_{X}(\alpha_{\gamma}^{t},t)-V_{X}(\beta_{\eta}^{t},t))|_{\beta=2\alpha}\\
    &X_{i}((2\alpha)_{\eta}^t)\times ic(2\alpha)t.
\end{align*}

\textbf{The forth type}: terms does not depending on $h$:
\begin{align}\label{fourthtypeestimate}
B_{4}^{i}=\lambda(\alpha_{\gamma}^{t})\int_{2\alpha}^{\pi}\tilde{B}_4^{i}(\alpha_{\gamma}^{t},\beta,t)d\beta
\end{align}
with $\tilde{B}_4^{i}$ having the following two forms:
sub-type 1
\begin{align*}
    &(\tilde{B}_4^{i}(\alpha_{\gamma}^{t},\beta,t))_{\mu}=(\frac{d}{d\beta}(K_{-1}^{i}( -V^{30}(\beta_0^{t},t),\Delta V_{\tilde{f}^{L}}^{61}(\alpha_{\gamma}^{t},t),\Delta V_{X}(\alpha_{\gamma}^{t},t)))X_{i}(\beta_{0}^{t},t)\tilde{f}_{v}^{+(b_i)}(\beta_0^{t},t).
\end{align*}
sub-type 2
\begin{align*}
    &(\tilde{B}_4^{i}(\alpha_{\gamma}^{t},\beta,t))_{\mu}=(K_{-1}^{i}( -V^{30}(\beta_0^{t},t),\Delta V_{\tilde{f}^{L}}^{61}(\alpha_{\gamma}^{t},t),\Delta V_{X}(\alpha_{\gamma}^{t},t))X_{i}(\beta_0^{t},t)\tilde{f}_{v}^{+(b_i)}(\beta_0^{t},t).
\end{align*}
  \end{lemma}
  \begin{proof}
  From \eqref{3f2new}, we have 
\begin{align*}
O^{0}
=&\underbrace{Term_1}_{\text{a sum of terms of } B_2, \text{ sub-type }2}+  \underbrace{Term_2}_{\text{a sum of terms of } B_2, \text{ sub-type }2}+ \underbrace{Term_{3,2,1}}_{B_2\text{ sub-type }5}\\&\underbrace{+Term_{3,2,2}}_{B_1\text{ sub-type }1}+\underbrace{Term_4}_{B_2, \text{ sub-type }2}
+\underbrace{Term_{5,1}}_{B_3, \text{ sub-type }1} +\underbrace{Term_{ 5,2,1}}_{B_4^{FT}, \text{ sub-type }1}+\underbrace{Term_{5,2,2}}_{\text{a sum of terms of } B_2, \text{ sub-type }3}.
\end{align*}
From \eqref{O1icnew} we get
\begin{align*}
&O^{1,i}=\underbrace{Term_1}_{B_2, \text{ sub-type }1} +\underbrace{Term_{2,1}}_{B_1, \text{ sub-type }2} + \underbrace{Term_{ 2,2}}_{B_3, \text{ sub-type }2} + \underbrace{Term_{2,3,1}}_{B_4^{FT}, \text{ sub-type }2}+\underbrace{Term_{ 2,3,2}}_{\text{a sum of terms of } B_2, \text{ sub-type }4}.
\end{align*}
 From \eqref{O3ic} $O^{3,i}$ is type $B_2$ with sub-type 1.
From  \eqref{O5ic},  $O^{5,i}$ is in $B_2$ with sub type 6.
  \end{proof}
In conclusion, from  \eqref{modifiedequation} and above lemmas \ref{3maintermtype}, \ref{3boundedtermtype}, we have when $\kappa(t)>0$:
\begin{align}\label{conclu:modified equation}
    \frac{dh(\alpha,\gamma,t)}{dt}=T^{+}(h)=
       M_{1,1}(h)+M_{1,2}(h)+M_{2,1}(h)+O^{0}+\sum_{i}O^{i}(h)+\lambda(\alpha)T_{fixed}(\alpha_{\gamma}^{t}), 
\end{align}
with $M_{1,1}$, $M_{1,2}$, $M_{2,1}$ the structure as in lemma \ref{3maintermtype}. $O_{i}$ satisfies:
\begin{equation}\label{Oterms}
O^{0}+O^{1,i}+O^{3,i}+O^{5,i}=\sum_{\tilde{i}}\sum_{j=1}^{4}B^{\tilde{i}}_{j}(h).
\end{equation}
\subsubsection{The vanishing and smoothness conditions}
In the following lemmas, we aim to show the $L_i^{+}(h)$, $B_i(h)$ described in the lemmas \ref{3maintermtype}, \ref{3boundedtermtype}  satisfy the smoothness conditions and the vanishing condition in the generalized equation \eqref{3GE01}. The R-T condition will be postponed to the next section.

Due to the fact that $L_i^{+}(h)$, $\tilde{B}_{1}^i(h)$, $\tilde{B}_{2}^i(h)$, $\tilde{B}_{3}^i(h)$, $\tilde{B}_{4}^i(h)$ in each type only depend on $D^{-60+s_i}h$ with $s_i\leq 30$, the proof is tedious but the idea is simple: we only take less than 12th derivative but we have 30th derivative to lose.

We also remark that from the analyticity of kernel $K_{-\sigma}(A,B,C)$ in \eqref{k-sigma}, we only show the boundedness conditions (eg. $\|L_i^{+}(h)\|_{Y^{k+2}}\lesssim 1$, $\|B_i(h)\|_{X^{k}}\lesssim 1$ ) and other smoothness conditions follow similar way.

The estimate for $L_1^{+}(h)$ follows from lemma \ref{L1main} and \eqref{3L1form01}. From lemma \ref{Lh1lemma}, \eqref{3M21chequG01}, we get the properties of $L_2^{+}(h)$.  Lemmas \ref{Lh3lemma}, \ref{Lh4lemma}, \ref{Lh5lemma}, \ref{Lhjlemma} show the estimate for $B_{1}^{i}(h)$, $B_{2}^{i}(h)$, $B_{3}^{i}(h)$, and $B_{4}^{i}.$ From \eqref{Oterms},  $O^{0}$, $O^{1,i}$, $O^{3,i}$ and  $O^{5,i}$  can be bounded. From \eqref{O2ic}, \eqref{O4ic}, we claim $O^{4,i}$, $O^{2,i}$ can be estimated as $O^{1,i}$. $\lambda(\al)T_{fixed}(\alpha_{\gamma}^t,t)$ can be controlled by lemma \ref{Tfixed behaviour}.

In the following lemmas (\ref{Lh1lemma} to \ref{Tfixed behaviour}) we assume $h$, $g$, $\tilde{g}$ satisfies \eqref{3hsmall}, with $\delta_s$, $t_s$ sufficiently small from arc-chord condition in lemma \ref{arcchord}.
\begin{lemma}\label{L1main}
If \begin{align*}
    L(h)(\alpha,\gamma,t)=p.v. \int_{-\pi}^{\pi}K_{-1}^{i}(\Delta V_h^{30}(\alpha,\gamma,t),\Delta V_{\tilde{f}^{L}}^{61}(\alpha_{\gamma}^{t},t),\Delta V_{X}(\alpha_{\gamma}^{t},t))X_i(\beta_{\gamma}^{t},t)(1+ic'(\beta)\gamma t)d\beta,
\end{align*}
for some $K_{-1}$ type kernel,  then we have the vanishing condition \begin{equation}\label{continuitycondition}
lim_{\alpha\to 0^{+}}L(h)(\alpha,\gamma,t)=L(h)(0,\gamma,t).
\end{equation}
Moreover, it satisfies the smoothness conditions of $L_i^{+}(h)$ in \eqref{3GE01} with
\begin{align}\label{dgammaL1}
D_{\gamma}L(h)[w]=\bar{\partial}_{\gamma}L(h)+D_h{L(h)}[w].
\end{align}
Here we let $\bar{\partial}_{\gamma}$ be the partial derivative with respect to $\gamma$ assuming $h$ does not depend on $\gamma.$
\end{lemma}
\begin{proof}
We first show \eqref{continuitycondition}.
We have
\begin{align*}
   &\qquad L(h)(\alpha,\gamma,t)\\
   &=p.v. \int_{-\pi}^{\pi}K_{-1}^{i}(\Delta V_h^{30}(\alpha,\gamma,t),\Delta V_{\tilde{f}^{L}}^{61}(\alpha_{\gamma}^{t},t),\Delta V_{X}(\alpha_{\gamma}^{t},t))X_{i}(\beta_{\gamma}^{t},t)(1+ic'(\beta)\gamma t)d\beta\\
   &=\lim_{\beta\to\alpha}(K_{-1}^{i}( \Delta V_h^{30}(\alpha,\gamma,t),\Delta V_{\tilde{f}^{L}}^{61}(\alpha_{\gamma}^{t},t),\Delta V_{X}(\alpha_{\gamma}^{t},t))(\alpha-\beta))X_{i}(\alpha_{\gamma}^{t},t)(1+ic'(\alpha)\gamma t)p.v.\int_{-\pi}^{\pi}\frac{1}{\alpha-\beta}d\beta\\
   &\quad+\int_{-\pi}^{\pi}(K_{-1}^{i}(\Delta V_h^{30}(\alpha,\gamma,t),\Delta V_{\tilde{f}^{L}}^{61}(\alpha_{\gamma}^{t},t),\Delta V_{X}(\alpha_{\gamma}^{t},t))(\alpha-\beta)X_{i}(\beta_{\gamma}^{t},t)(1+ic'(\beta)\gamma t)\\
   &\qquad -\lim_{\beta\to\alpha}(K_{-1}^{i}( \Delta V_h^{30}(\alpha,\gamma,t),\Delta V_{\tilde{f}^{L}}^{61}(\alpha_{\gamma}^{t},t),\Delta V_{X}(\alpha_{\gamma}^{t},t))(\alpha-\beta))X_{i}(\alpha_{\gamma}^{t},t)(1+ic'(\alpha)\gamma t))\frac{1}{\alpha-\beta}d\beta\\
   &=Term_{0,1}+Term_{0,2}
\end{align*}
Since $c(\alpha)\in C^{1,1}(-\infty,\infty)$ (\eqref{cdefi}), from definition of $a_{\gamma}^{t}$ \eqref{alphadefi}, the space of $\tilde{f}^{L}$ and $X_i$ \eqref{Xispace}, we have when $s_l\leq 30$
 \begin{align}\label{spacelemma}
 D^{-60+s_l}h(\alpha,\gamma,t), \partial_{\alpha}^{d_i}\tilde{f}^{L}(\alpha_{\gamma}^{t}), X(\alpha_{\gamma}^{t}) \in C_{\gamma}^{0}([-1,1], C_{\alpha}^{1,1}[-\pi,\pi]).
 \end{align}
 Recall that $K_{-\sigma}^{j}(A,B,C)$ is of $-\sigma$ type if for $A$, $B$, $C$ in $R^n$, it has the form
\begin{align}\label{k-sigma02}
&c_j\frac{\sin(A_1+B_1)^{m_1}\cos(A_1+B_1)^{m_2}}{(\cosh(A_2+B_2)-\cos(A_1+B_1))^{m_0}}\\\nonumber
    &\times(\sinh(A_2+B_2))^{m_3}(\cosh(A_2+B_2))^{m_4}\Pi_{j=1}^{m_5}(A_{\lambda_{j}})\Pi_{j=1}^{m_6}(B_{\lambda_{j,2}})\Pi_{j=1}^{m_7}(C_{\lambda_{j,3}}),
\end{align}
with $m_1+m_3+m_5+m_6+m_7-2m_0\geq -\sigma$. $c_j$ is a constant.
Here we can take
\[
A_i,B_i,C_i\in\{\partial_{\alpha}^{d_i}\tilde{f}^{L}(\alpha_{\gamma}^{t},t)-\partial_{\alpha}^{d_i}\tilde{f}^{L}(\beta_{\gamma}^{t},t), D^{-60+s_i}h(\alpha,\gamma,t)-D^{-60+s_i}h(\beta,\gamma,t), X_i(\alpha_{\gamma}^{t},t)-X_i(\beta_{\gamma}^{t},t)\}
\]
 with $d_i\leq 61$, $s_{i}\leq 30$.
 Then each $\sin, \sinh, A_{j}, B_{j}, C_{j}$ gives an order of $(\alpha-\beta)$ and the denominator gives an order of $(\alpha-\beta)^{-2m_0}$. Hence from $m_1+m_3+m_5+m_6+m_7-2m_0\geq -\sigma$ and the arc-chord condition \eqref{arcchord4}, we have the estimate:
 \[
 \lim_{\beta\to\alpha}(K_{-1}^{i}( \Delta V_h^{30}(\alpha,\gamma,t),\Delta V_{\tilde{f}^{L}}^{61}(\alpha_{\gamma}^{t},t),\Delta V_{X}(\alpha_{\gamma}^{t},t))(\alpha-\beta))\in  C_{\gamma}^{0}([-1,1], C_{\alpha}^{0,1}[0,\pi]).
\]
Then we have $\lim_{\alpha\to 0^{+}}Term_{0,1}(\alpha,\gamma,t)=Term_{0,1}(0,\gamma,t)$.
For $Term_{0,2}$, to use Lebesgue's convergence theorem, we only need to show the uniform bound with respect to $\alpha,\beta$. From the structure of $K_{-\sigma}$ and \eqref{spacelemma}, it is enough to show for any $g(\alpha)\in C_{\alpha}^{1,1}[-\pi,\pi]$, the following estimates hold:
\begin{equation}\label{lemmaboundaryes01}
|\frac{g'(\alpha)-g'(\beta)}{\alpha-\beta}|\lesssim 1,
\end{equation}
and
\begin{equation}\label{lemmaboundaryes02}
|\frac{1}{\alpha-\beta}(\frac{g(\alpha)-g(\beta)}{\alpha-\beta}-\lim_{\beta\to\alpha}\frac{g(\alpha)-g(\beta)}{\alpha-\beta})|\lesssim 1.
\end{equation}
The first estimate \eqref{lemmaboundaryes01} holds directly because $g'\in C_{\alpha}^{0,1}$. For the second one \eqref{lemmaboundaryes02}, we use 
\begin{align*}
&\frac{1}{\alpha-\beta}(\frac{g(\alpha)-g(\beta)}{\alpha-\beta}-\lim_{\beta\to\alpha}\frac{g(\alpha)-g(\beta)}{\alpha-\beta})\\
&=\frac{1}{\alpha-\beta}(\int_{0}^{1}g'(\beta+\zeta(\alpha-\beta))d\zeta-g'(\alpha))\\
&=\int_{0}^{1}\frac{(g'(\beta+\zeta(\alpha-\beta))-g'(\alpha))}{\alpha-\beta}d\zeta.
\end{align*}
Then the bound \eqref{lemmaboundaryes02} holds because $g'\in C_{\alpha}^{0,1}$.
we could use the Lebesgue convergence theorem to get the limit.

Now we show the smoothness condition. We will only show the boundedness conditions and other conditions follow a similar way. We first separate the function in the integral into several parts in \eqref{separateterm3}. One part is analytic with respect to $\beta$ ($Term_{1,1}$), other parts has good boundary behavior ($Term_{1,2}$ and $Term_{2}$).
We first separate the function into one term depending on $h$ and another doesn't: \begin{align}\label{separationl1}
   &\quad L(h)(\alpha,\gamma,t)\\\nonumber
   &=p.v. \int_{-\pi}^{\pi}K_{-1}^{i}(0,\Delta V_{\tilde{f}^{L}}^{61}(\alpha_{\gamma}^{t},t),\Delta V_{X}(\alpha_{\gamma}^{t},t))X_{i}(\beta_{\gamma}^{t},t)(1+ic'(\beta)\gamma t)d\beta\\\nonumber
   &\quad+p.v.\int_{-\pi}^{\pi}\int_{0}^{1}\frac{d}{d\zeta}K_{-1}^{i}(\zeta \Delta V_h^{30}(\alpha,\gamma,t),\Delta V_{\tilde{f}^{L}}^{61}(\alpha_{\gamma}^{t},t),\Delta V_{X}(\alpha_{\gamma}^{t},t))X_{i}(\beta_{\gamma}^{t},t)(1+ic'(\beta)\gamma t)d\beta d\zeta\\\nonumber
   &=Term_1+Term_2.
\end{align}

Then like in \eqref{fLequation}, we further separate $\tilde{f}^{L}(\alpha,t)$ into the part that is polynomial in $[-\frac{\delta}{2},-\frac{\delta}{2}]$ and the part that vanishes up to 99th order at 0,
which is
\begin{equation}\label{fLequation2}
\tilde{f}^{-}(\alpha,t)=(\tilde{f}(\alpha,t)-\sum_{i=1}^{100}\frac{\tilde{f}^{<i>}(0)}{i!}\alpha^i)\lambda(\alpha)1_{\alpha\leq 0},
\end{equation}
and 
\begin{equation}\label{fLequation2new}
\tilde{f}^{L}(\alpha,t)=\tilde{f}^{L,L}(\alpha,t)+\tilde{f}^{-}(\alpha,t).
\end{equation}
Here 
\begin{align}\label{f-fLLspace}
&\tilde{f}^{-}(\alpha,t)\in C^1([0,t_0],C^{99}(-\infty,\infty)),\\\nonumber
&\partial_{\alpha}^{j}\tilde{f}^{-}(0,t)=0,&\text{when } j\leq 99,\\\nonumber
& \tilde{f}^{-}(\alpha,t)=0, &\text{when } \alpha>0, \\\nonumber
&\tilde{f}^{L,L}(\alpha,t) \text{ is a polynomial in } [-\frac{\delta}{2},\frac{\delta}{2}].
\end{align}
Moreover, we also separate $X_i$ into terms with respect to $\tilde{f}^{-}$  and $\tilde{f}^{L,L}$. From def of $X_i$ \eqref{Xifunction}, we have
\begin{equation}
X_i(\alpha,t)=(\partial_{\alpha}^{l_i}\tilde{f}^{L}(\alpha,t))^{\phi_{i}}(\partial_{\alpha}^{l_i'}(\frac{1}{\frac{dx}{d\alpha}(\alpha,t)}))^{\phi_{i}'}(\partial_{\alpha}^{l_i''}\frac{d}{d\alpha}x(\alpha,t))^{\phi_{i''}},
\end{equation}
with $l_i,l_i',l_1''\leq 61$, $\phi_i,\phi_{i}',\phi_{i''} \in \{0,1\}$.
Then we define
\begin{align}\label{XLfunction}
   X_i^{L}=\bigg\{\begin{array}{cc}
\partial_{\alpha}^{l_i}\tilde{f}^{L,L}(\alpha,t)(\partial_{\alpha}^{l_i'}(\frac{1}{\frac{dx}{d\alpha}(\alpha,t)}))^{\phi_{i}'}(\partial_{\alpha}^{l_i''}\frac{d}{d\alpha}x(\alpha,t))^{\phi_{i''}}&   \phi_i=1,\\
 (\frac{1}{\frac{dx}{d\alpha}(\alpha,t)})^{\phi_{i}'}(\partial_{\alpha}^{l_i''}\frac{d}{d\alpha}x(\alpha,t))^{\phi_{i''}}& \phi_i=0, 
      \end{array}
      \end{align}
\begin{align}\label{X-function}
   X_i^{-}=\bigg\{\begin{array}{cc}
\partial_{\alpha}^{l_i}\tilde{f}^{-}(\alpha,t)(\partial_{\alpha}^{l_i'}(\frac{1}{\frac{dx}{d\alpha}(\alpha,t)}))^{\phi_{i}'}(\partial_{\alpha}^{l_i''}\frac{d}{d\alpha}x(\alpha,t))^{\phi_{i''}}&   \phi_i=1,\\
0& \phi_i=0.
      \end{array}
      \end{align}
      Like in \eqref{notation01}, we denote:
\begin{align}\label{notation03f-}
V^{k}_{\tilde{f}^{-}} (\alpha,t) =(\tilde{f}^{-}_1(\alpha,t), \tilde{f}^{-}_2(\alpha,t), \partial_{\alpha}\tilde{f}^{-}_1(\alpha,t),\partial_{\alpha}\tilde{f}^{-}_2(\alpha,t), ...,\partial_{\alpha}^{k}\tilde{f}^{-}_1(\alpha,t),\partial_{\alpha}^{k}\tilde{f}^{-}_2(\alpha,t)).
\end{align}
\begin{align}\label{notation03fLL}
V^{k}_{\tilde{f}^{L,L}} (\alpha,t) =(\tilde{f}^{L,L}_1(\alpha,t), \tilde{f}^{L,L}_2(\alpha,t), \partial_{\alpha}\tilde{f}^{L,L}_1(\alpha,t),\partial_{\alpha}\tilde{f}^{L,L}_2(\alpha,t), ...,\partial_{\alpha}^{k}\tilde{f}^{L,L}_1(\alpha,t),\partial_{\alpha}^{k}\tilde{f}^{L,L}_2(\alpha,t)).
\end{align}
\begin{align}\label{notation03xL}
V^{k}_{X^{L}} (\alpha,t) =(X_i^{L}(\alpha,t)).
\end{align}
\begin{align}\label{notation03x-}
V^{k}_{X^{-}} (\alpha,t) =(X_i^{-}(\alpha,t)).
\end{align}
From the above definition, we also have
\begin{equation}\label{spacefLL}
V_{\tilde{f}^{L,L}}^{61}(\alpha_{\gamma}^{t},t), V_{X^L}(\alpha_{\gamma}^{t},t) \text{ are analytic in } \tilde{D}=\{\alpha+iy||\alpha|\leq \frac{\delta}{4}\}
\end{equation}
Then we rewrite $Term_1$, we have
\begin{align}\label{separationl2}
    &\quad Term_1\\\nonumber
    &=p.v. \int_{-\pi}^{\pi}K_{-1}^{i}(0,\Delta V_{\tilde{f}^{L,L}}^{61}(\alpha_{\gamma}^{t},t)+\Delta V_{\tilde{f}^{-}}^{61}(\alpha_{\gamma}^{t},t),\Delta V_{X^L}(\alpha_{\gamma}^{t},t)+\Delta V_{X^-}(\alpha_{\gamma}^{t},t))\\\nonumber
    &\qquad\cdot(X_{i}^{L}(\beta_{\gamma}^{t},t)+X_{i}^{-}(\beta_{\gamma}^{t},t))(1+ic'(\beta)\gamma t)d\beta\\\nonumber
    &=p.v. \int_{-\pi}^{\pi}K_{-1}^{i}(0,\Delta V_{\tilde{f}^{L,L}}^{61}(\alpha_{\gamma}^{t},t),\Delta V_{X^L}(\alpha_{\gamma}^{t},t))\cdot(X_{i}^{L}(\beta_{\gamma}^{t},t))(1+ic'(\beta)\gamma t)d\beta\\\nonumber
    &\quad+\int_0^{1}d\zeta p.v. \int_{-\pi}^{\pi}\frac{d}{d\zeta}(K_{-1}^{i}(0,\Delta V_{\tilde{f}^{L,L}}^{61}(\alpha_{\gamma}^{t},t)+\zeta\Delta V_{\tilde{f}^{-}}^{61}(\alpha_{\gamma}^{t},t),\Delta V_{X^L}(\alpha_{\gamma}^{t},t)+\zeta\Delta V_{X^-}(\alpha_{\gamma}^{t},t))\\\nonumber
    &\qquad\cdot(X_{i}^{L}(\beta_{\gamma}^{t},t)+\zeta X_{i}^{-}(\beta_{\gamma}^{t},t)))(1+ic'(\beta)\gamma t)d\beta \\\nonumber
    &=Term_{1,1}+Term_{1,2}.
\end{align}
Combining \eqref{separationl1} and \eqref{separationl2}, we have 
\begin{align}\label{separateterm3}
L(h)=Term_{1,1}+Term_{1,2}+Term_{2}. 
\end{align}
Now we start from  $Term_{1,1}$. From \eqref{spacefLL}
, we can change the contour of $Term_{1,1}$ when $\al>0$, $\beta<0$. 
Let \begin{align}\label{definitiontildec}
   \tilde{c}(\alpha)=\bigg\{\begin{array}{cc}
 c(\alpha)&   \alpha \geq 0,\\
 c(-\alpha)& \alpha\leq 0, 
      \end{array}
      \end{align}
and
\begin{align}\label{definitiontildecalpha}
    \tilde{\alpha}_{\gamma}^{t}=\alpha+i\tilde{c}(\alpha)\gamma t.
\end{align}
Since by the definition $c(\alpha)=\alpha^2$ when $\alpha$ is sufficiently small, we have $\tilde{c}(\alpha)\in C^{100}(-\infty,\infty)$.
By the arc-chord condition \eqref{arcchord3L}, there is no zero point of denominator when $\beta\leq 0$, $\alpha>0$, therefore we could change the contour, and have
\begin{align*}
    Term_{1,1}=p.v. \int_{-\pi}^{\pi}K_{-1}^{i}(0,\Delta V_{\tilde{f}^{L,L}}^{61}(\tilde{\alpha}_{\gamma}^{t},t),\Delta V_{X^L}(\tilde{\alpha}_{\gamma}^{t},t))X_{i}^{L}(\tilde{\beta}_{\gamma}^{t},t)(1+i\tilde{c}'(\beta)\gamma t)d\beta.
\end{align*}

Then from the kernel estimate \eqref{K-1changed}, \eqref{K-1changedt}, $Term_{1,1}$ satisfies the smoothness conditions.

For $Term_{1,2}$, from \eqref{X-function}, we have
\begin{align*}
    &\quad\frac{d}{d\zeta}(K_{-1}^{i}(0,\Delta V_{\tilde{f}^{L,L}}^{61}(\alpha_{\gamma}^{t},t)+\zeta\Delta V_{\tilde{f}^{-}}^{61}(\alpha_{\gamma}^{t},t),\Delta V_{X^L}(\alpha_{\gamma}^{t},t)+\zeta\Delta V_{X^-}(\alpha_{\gamma}^{t},t))(X_{i}^{L}(\beta_{\gamma}^{t},t)+\zeta X_{i}^{-}(\beta_{\gamma}^{t},t)))\\
    &=K_{-1}^{i}(0,\Delta V_{\tilde{f}^{L,L}}^{61}(\alpha_{\gamma}^{t},t)+\zeta\Delta V_{\tilde{f}^{-}}^{61}(\alpha_{\gamma}^{t},t),\Delta V_{X^L}(\alpha_{\gamma}^{t},t)+\zeta\Delta V_{X^-}(\alpha_{\gamma}^{t},t))\partial_{\beta}^{l_i}\tilde{f}^{-}(\beta_{\gamma}^{t},t)\\
    &\qquad\cdot \phi_i\cdot((\partial_{\beta}^{l_i'}\frac{1}{\partial_{\beta}x})(\beta_{\gamma}^{t},t))^{\phi_{i}'}((\partial_{\beta}^{l_i''}\partial_{\beta}x)(\beta_{\gamma}^{t},t))^{\phi_{i}''}\\
    &\quad+\sum_{j'}K_{-2}^{j'}(0,\Delta V_{\tilde{f}^{L,L}}^{61}(\alpha_{\gamma}^{t},t)+\zeta\Delta V_{\tilde{f}^{-}}^{61}(\alpha_{\gamma}^{t},t),\Delta V_{X^L}(\alpha_{\gamma}^{t},t)+\zeta\Delta V_{X^-}(\alpha_{\gamma}^{t},t))(\partial_{\alpha}^{l_j'}\tilde{f}^{-}(\alpha_{\gamma}^{t},t)-\partial_{\beta}^{l_j'}\tilde{f}^{-}(\beta_{\gamma}^{t},t))\\
    &\qquad\cdot (X_{i}^{L}(\beta_{\gamma}^{t},t)+\zeta X_{i}^{-}(\beta_{\gamma}^{t},t))\\
    &\quad+\sum_{n}K_{-2}^{n}(0,\Delta V_{\tilde{f}^{L,L}}^{61}(\alpha_{\gamma}^{t},t)+\zeta\Delta V_{\tilde{f}^{-}}^{61}(\alpha_{\gamma}^{t},t),\Delta V_{X^L}(\alpha_{\gamma}^{t},t)+\zeta\Delta V_{X^-}(\alpha_{\gamma}^{t},t))\phi_n\\
    &\qquad\cdot(\partial_{\alpha}^{l_n}\tilde{f}^{-}(\alpha_{\gamma}^{t},t)((\partial_{\alpha}^{l_n'}\frac{1}{\partial_{\alpha}x})(\alpha_{\gamma}^{t},t))^{\phi_{n}'}((\partial_{\alpha}^{l_n''}\partial_{\alpha}x)(\alpha_{\gamma}^{t},t))^{\phi_{n}''}-\partial_{\beta}^{l_n}\tilde{f}^{-}(\beta_{\gamma}^{t},t)((\partial_{\beta}^{l_n'}\frac{1}{\partial_{\beta}x})(\beta_{\gamma}^{t},t))^{\phi_{n}'}((\partial_{\beta}^{l_n''}\partial_{\beta}x)(\beta_{\gamma}^{t},t))^{\phi_{n}''})\\
    &\qquad\cdot (X_{i}^{L}(\beta_{\gamma}^{t},t)+\zeta X_{i}^{-}(\beta_{\gamma}^{t},t))\\
    &=KTerm_{1,2,1}+KTerm_{1,2,2}+KTerm_{1,2,3},
\end{align*}
for some $-2$ type functions $K_{-2}^{j'}$ and $K_{-2}^{n}$,  with all index of derivative less than 61, $\phi_{i}, \phi_{i}', \phi_{i}'', \phi_{n},\phi_{n}', \phi_{n}''\in \{0,1\}$.
Here the first term shows up when the derivative hits on $(X_{i}^{L}(\beta_{\gamma}^{t},t)+\zeta X_{i}^{-}(\beta_{\gamma}^{t},t))$, the second one shows up when it hits on $\zeta\Delta V_{\tilde{f}^{-}}^{61}(\alpha_{\gamma}^{t},t)$, the last one shows up when it hits on $\zeta\Delta V_{X^-}(\alpha_{\gamma}^{t},t)$.
Then from \eqref{separationl2} the term becomes
\begin{align*}
    &Term_{1,2}\\
    &=Term_{1,2,1}+Term_{1,2,2}+Term_{1,2,3}\\
    &=\int_{0}^{1}d\zeta p.v.\int_{-\pi}^{\pi}KTerm_{1,2,1}(1+ic'(\beta)\gamma t)d\beta+\int_{0}^{1}d\zeta p.v.\int_{-\pi}^{\pi}KTerm_{1,2,2}(1+ic'(\beta)\gamma t)d\beta\\
    &\quad+\int_{0}^{1}d\zeta p.v.\int_{-\pi}^{\pi}KTerm_{1,2,3}(1+ic'(\beta)\gamma t)d\beta.
\end{align*}
For $Term_{1,2,1}$, we have
\begin{align*}
    &\quad Term_{1,2,1}\\
    &=\int_{0}^{1}d\zeta p.v.\int_{-\pi}^{\pi}K_{-1}^{i}(0,\Delta V_{\tilde{f}^{L,L}}^{61}(\alpha_{\gamma}^{t},t)+\zeta\Delta V_{\tilde{f}^{-}}^{61}(\alpha_{\gamma}^{t},t),\Delta V_{X^L}(\alpha_{\gamma}^{t},t)+\zeta\Delta V_{X^-}(\alpha_{\gamma}^{t},t))\\ &\qquad\cdot\partial_{\beta}^{l_i}\tilde{f}^{-}(\beta_{\gamma}^{t},t)\cdot \phi_i\cdot((\partial_{\beta}^{l_i'}\frac{1}{\partial_{\beta}x})(\beta_{\gamma}^{t},t))^{\phi_{i}'}((\partial_{\beta}^{l_i''}\partial_{\beta}x)(\beta_{\gamma}^{t},t))^{\phi_{i}''}(1+ic'(\beta)\gamma t)d\beta\\
    &=\int_{0}^{1}d\zeta p.v.\int_{-\pi}^{0}K_{-1}^{i}(0,\Delta V_{\tilde{f}^{L,L}}^{61}(\alpha_{\gamma}^{t},t)+\zeta\Delta V_{\tilde{f}^{-}}^{61}(\alpha_{\gamma}^{t},t),\Delta V_{X^L}(\alpha_{\gamma}^{t},t)+\zeta\Delta V_{X^-}(\alpha_{\gamma}^{t},t))\\
   &\qquad\cdot\partial_{\beta}^{l_i}\tilde{f}^{-}(\beta_{\gamma}^{t},t)\cdot \phi_i\cdot((\partial_{\beta}^{l_i'}\frac{1}{\partial_{\beta}x})(\beta_{\gamma}^{t},t))^{\phi_{i}'}((\partial_{\beta}^{l_i''}\partial_{\beta}x)(\beta_{\gamma}^{t},t))^{\phi_{i}''}(1+ic'(\beta)\gamma t)d\beta,
\end{align*}
where we use $\tilde{f}^{-}(\beta_{\gamma}^{t},t)=0$ when $\beta>0$ \eqref{fLequation2}. 

From the definition of $\tilde{f}^{-}$ \eqref{fLequation2}, and $c(\beta)=0$ when $\beta<0$ \eqref{cdefi}, we have when $\beta<0$, $|\partial_{\beta}^{l_i}\tilde{f}^{-}(\beta_{\gamma}^{t},t)|=|\partial_{\beta}^{l_i}\tilde{f}^{-}(\beta,t)|\lesssim |\beta|^{30}\eqref{f-fLLspace}.$ Then from lemma \ref{alphabetafar-}, we have the result for $Term_{1,2,1}$ because the singularity from $K_{-1}$ can be compensated by $|\beta|^{30}.$

Since from \eqref{fLequation2}, $\tilde{f}^{-}(\alpha_{\gamma}^t,t)=0$ when $\alpha\geq 0$, then $\partial_{\alpha}^{l_i'}\tilde{f}^{-}(\alpha_{\gamma}^{t},t)$ and $\partial_{\alpha}^{l_n'}\tilde{f}^{-}(\alpha_{\gamma}^{t},t)$ disappear in $Term_{1,2,2}$ and  $Term_{1,2,3}$.
 Hence these two terms can be done in the same way as $Term_{1,2,1}$ by lemma \ref{alphabetafar-}.
\color{black}

For $Term_2$, from \eqref{separationl1}, we first claim that
 \begin{align*}
     &\quad\frac{d}{d\zeta}K_{-1}^{i}(\zeta \Delta V_h^{30}(\alpha,\gamma,t),\Delta V_{\tilde{f}^{L}}^{61}(\alpha_{\gamma}^{t},t),\Delta V_{X}(\alpha_{\gamma}^{t},t))\\
     &=\sum_{l}K_{-2}^{l}(\zeta \Delta V_h^{30}(\alpha,\gamma,t),\Delta V_{\tilde{f}^{L}}^{61}(\alpha_{\gamma}^{t},t),\Delta V_{X}(\alpha_{\gamma}^{t},t))(D^{-60+s_l}h(\alpha,\gamma,t)-D^{-60+s_l}h(\beta,\gamma,t)),
     \end{align*}
with $s_l\leq 30$ for some $-2$ type $K_{-2}^{l}$.
Therefore we have
\begin{align*}
    &\quad Term_2\\
    &=p.v.\int_{-\pi}^{\pi}\int_{0}^{1}\frac{d}{d\zeta}K_{-1}^{i}(\Delta \zeta V_h^{30}(\alpha,\gamma,t),\Delta V_{\tilde{f}^{L}}^{61}(\alpha_{\gamma}^{t},t),\Delta V_{X}(\alpha_{\gamma}^{t},t))X_{i}(\beta_{\gamma}^{t},t)(1+ic'(\beta)\gamma t)d\beta d\zeta\\
    &=\sum_{l}\int_{-\pi}^{0}\int_{0}^{1}K_{-2}^{l}(\zeta \Delta V_h^{30}(\alpha,\gamma,t),\Delta V_{\tilde{f}^{L}}^{61}(\alpha_{\gamma}^{t},t),\Delta V_{X}(\alpha_{\gamma}^{t},t))(1+ic'(\beta)\gamma t)X_{i}(\beta_{\gamma}^{t},t)d\beta d\zeta D^{-60+s_l}h(\alpha,\gamma,t)\\
    &\quad+\sum_{l}p.v.\int_{0}^{\pi}\int_{0}^{1}K_{-2}^{l}(\zeta \Delta V_h^{30}(\alpha,\gamma,t),\Delta V_{\tilde{f}^{L}}^{61}(\alpha_{\gamma}^{t},t),\Delta V_{X}(\alpha_{\gamma}^{t},t)) X_{i}(\beta_{\gamma}^{t},t)\\
    &\qquad\cdot (D^{-60+s_l}h(\alpha,\gamma,t)-D^{-60+s_l}h(\beta,\gamma,t))(1+ic'(\beta)\gamma t) d\beta d\zeta\\
    &=Term_{2,1}+\sum_{l}\int_{0}^{1}\int_{0}^{\pi}(K_{-2}^{l}(\zeta \Delta V_h^{30}(\alpha,\gamma,t),\Delta V_{\tilde{f}^{L}}^{61}(\alpha_{\gamma}^{t},t),\Delta V_{X}(\alpha_{\gamma}^{t},t))(\alpha-\beta)^2)\\
    &\qquad\frac{(D^{-60+s_l}h(\alpha,\gamma,t)-D^{-60+s_l}h(\beta,\gamma,t)-(\alpha-\beta)\frac{d}{d\alpha}D^{-60+s_l}h(\alpha,\gamma,t))}{(\alpha-\beta)^2}X_{i}(\beta_{\gamma}^{t},t)(1+ic'(\beta)\gamma t)d\beta d\zeta\\
    &\quad+\sum_{l}\frac{d}{d\alpha}D^{-60+s_l}h(\alpha,\gamma,t)\int_{0}^{1}p.v.\int_{0}^{\pi}(K_{-2}^{l}(\zeta \Delta V_h^{30}(\alpha,\gamma,t),\Delta V_{\tilde{f}^{L}}^{61}(\alpha_{\gamma}^{t},t),\Delta V_{X}(\alpha_{\gamma}^{t},t))(\alpha-\beta)^2)(1+ic'(\beta)\gamma t)\\
    &\qquad X_{i}(\beta_{\gamma}^{t},t)\frac{1}{\alpha-\beta}d\beta d\zeta\\
    &=Term_{2,1}+Term_{2,2}+Term_{2,3}.
\end{align*}
Here we separate the singular part in the second step. Since $s_l\leq 30$, from the definition of $D^{-60+s_{l}}$ \eqref{D-formularnew}, we have $|D^{-60+s_l}h(\alpha)|\lesssim |\alpha|^{30}.$
Then the singularity in $Term_{2,1}$ can be compensated for $\beta\leq 0$, $\alpha>0$ by corollary \ref{alphabetafarleft} (the behavior of $K_{-\sigma}$ when $\beta<0$) since $|\beta-\alpha|\geq |\alpha|$.
$Term_{2,2}$ can be controlled directly from kernel estimate lemma \ref{alphabetanear} (the behavior of $K^{-\sigma}$ when $\beta>0$) because there is no singularity.

For the last term, we further take the most singular term and get
\begin{align*}
    &\quad Term_{2,3}\\
    &=\sum_{l}\frac{d}{d\alpha}D^{-60+s_l}h(\alpha,\gamma,t)\int_{0}^{1}p.v.\int_{0}^{\pi}[(K_{-2}^{l}(\zeta \Delta V_h^{30}(\alpha,\gamma,t),\Delta V_{\tilde{f}^{L}}^{61}(\alpha_{\gamma}^{t},t),\Delta V_{X}(\alpha_{\gamma}^{t},t))(\alpha-\beta)^2)\\
    &\qquad\cdot(1+ic'(\beta)\gamma t)X_{i}(\beta_{\gamma}^{t},t)\\
    &\qquad-\lim_{\beta\to\alpha}((K_{-2}^{l}(\zeta \Delta V_h^{30}(\alpha,\gamma,t),\Delta V_{\tilde{f}^{L}}^{61}(\alpha_{\gamma}^{t},t),\Delta V_{X}(\alpha_{\gamma}^{t},t))(\alpha-\beta)^2))X_{i}(\alpha_{\gamma}^{t},t)(1+ic'(\alpha)\gamma t)]\frac{1}{\alpha-\beta} d\beta d\zeta\\
    &\quad+\sum_{l}\frac{d}{d\alpha}D^{-60+s_l}h(\alpha,\gamma,t)\int_{0}^{1}\lim_{\beta\to\alpha}(K_{-2}^{l}(\zeta \Delta V_h^{30}(\alpha,\gamma,t),\Delta V_{\tilde{f}^{L}}^{61}(\alpha_{\gamma}^{t},t),\Delta V_{X}(\alpha_{\gamma}^{t},t))(\alpha-\beta)^2)X_{i}(\alpha_{\gamma}^{t},t)\\
    &\qquad\cdot(1+ic'(\alpha)\gamma t)p.v.\int_{0}^{\pi}\frac{1}{\alpha-\beta}d\beta d\zeta\\
    &=Term_{2,3,1}+Term_{2,3,2}.
\end{align*}
$Term_{2,3,1}$ can be bounded in the similar way as in $Term_{2,2}$ from lemma \ref{alphabetanear}.
 For $Term_{2,3,2}$, the log singularity in $p.v.\int_{0}^{\pi}\frac{1}{\alpha-\beta}d\beta$ is compensated by 
 $|\frac{d}{d\alpha}D^{-60+s_l}h(\alpha,\gamma,t)|\lesssim |\alpha|^{29}$. Hence from lemma \ref{alphabetanear} (the behavior of $K^{-\sigma}$ when $\beta>0$),  $Term_{2,3,2}$  also satisfies the conditions.
 
For $D_\gamma L(h)[w]$, \eqref{dgammaL1} holds through direct calculation by finding the operator that satisfies
\[
D_\gamma L(h)[\frac{dh}{d\gamma}]=\frac{d}{d\gamma}L(h).
\]
Other smoothness conditions for $L_i^{+}$ holds similarly as in the boundedness estimate.
\end{proof}
\begin{lemma}\label{Lh1lemma}
 If
\begin{align*}
    L_2^{+}(h)(\alpha,\gamma,t)=-\lim_{\beta\to\alpha}(\frac{d}{d\beta}K_{-1}^{j}(\Delta V_h^{30}(\alpha,\g ,t),\Delta V_{\tilde{f}^{L}}^{61}(\alpha_{\gamma}^{t},t),\Delta V_{X}(\alpha_{\gamma}^{t},t))(\alpha-\beta)^2)
\end{align*}
with particular $K_{-1}^{j}(A,B,C)=\frac{\sin(A_1+B_1)}{\cosh(A_2+B_2)-\cos(A_1+B_1)}$, then it satisfies the vanishing condition :
\[
lim_{\alpha\to 0^{+}}L_{2}^{+}(h)(\alpha,\gamma,t)=0.
\]
and 
the smoothness conditions for $L_i^{+}$ in \eqref{3GE01} with
\begin{align}\label{dgammaL2}
D_{\gamma}L_2^{+}(h)[w]=\bar{\partial}_{\gamma}L_2^{+}(h)+D_h{L_{2}^{+}(h)}[w].
\end{align}
\end{lemma}
\begin{proof}
From sufficiently good smoothness of $V_{h}^{30}$ \eqref{D-formularnew}, \eqref{notation03h}, $\tilde{f}^{L}$ \eqref{fLspace}, $X$ and $\tilde{X}$ \eqref{Xispace},  we claim 
\begin{align*}
&\frac{d}{d\beta}(K_{-1}^{i}(\Delta V_h^{30}(\alpha,\gamma,t),\Delta V_{\tilde{f}^{L}}^{61}(\alpha_{\gamma}^{t},t),\Delta V_{X}(\alpha_{\gamma}^{t},t)))\\\nonumber
&=\sum_{l'}K_{-(2)}^{l'}(\Delta V_h^{30}(\alpha,\gamma,t),\Delta V_{\tilde{f}^{L}}^{61}(\alpha_{\gamma}^{t},t),\Delta V_{X}(\alpha_{\gamma}^{t},t))) G_{l'}(\alpha,\beta,\gamma, t),
\end{align*}
with  $G(\al,\beta,\g,l)\in C_{\g}^{0}([-1,1],C_{\al,\beta}^{k+4}([0,\frac{\pi}{4}]\times[0,\pi])). $ Then we could use lemma \ref{alphabetanear} with $-2$ type kernel
, to show 
\begin{equation}\label{kernelestimatenearorigin}
(\alpha-\beta)^2\frac{d}{d\beta}(K_{-1}^{i}(\Delta V_h^{30}(\alpha,\gamma,t),\Delta V_{\tilde{f}^{L}}^{61}(\alpha_{\gamma}^{t},t),\Delta V_{X}(\alpha_{\gamma}^{t},t)))\in C_{\gamma}^{0}([-1,1], C^{k+4}_{\alpha,\beta}([0,\frac{\pi}{4}]\times [0,\pi])),
\end{equation}
then the boundedness estimate \eqref{boundestimateL} follows.

For $D_\gamma L_2^{+}(h)[w]$, \eqref{dgammaL2} holds through direct calculation by finding the operator that satisfies
\[
D_\gamma L_2^{+}(h)[\frac{dh}{d\gamma}]=\frac{d}{d\gamma}L_2^{+}(h).
\]
Other smoothness conditions holds similarly as in the boundedness estimate.

The vanishing condition relies on the particular kernel. We have
\[
 L_{2}^{+}(h)(\alpha,\gamma,t)=-\lim_{\beta\to\alpha}(\frac{d}{d\beta}(\frac{\sin(A_1+B_1)}{\cosh(A_2+B_2)-\cos(A_1+B_1)})(\alpha-\beta)^2),
\]
with $A_\mu=D^{-60}h_\mu(\alpha,\gamma,t)-D^{-60}h_\mu(\beta,\gamma,t)$, $B_\mu=\tilde{f}^{L}_{\mu}(\alpha_\gamma^t,t)-\tilde{f}^{L}_{\mu}(\beta_\gamma^t,t)$.
Then
\begin{align*}
    &\frac{d}{d\beta}(\frac{\sin(A_1+B_1)}{\cosh(A_2+B_2)-\cos(A_1+B_1)})(\alpha-\beta)^2\\
    &=\frac{\cos(A_1+B_1)(\alpha-\beta)^2}{\cosh(A_1+B_1)-\cos(A_1+B_1)}\frac{d}{d\beta}(A_1+B_1)\\
    &-\frac{\sin(A_1+B_1)(\sinh(A_2+B_2)\frac{d}{d\beta}(A_2+B_2)+\sin(A_1+B_1)\frac{d}{d\beta}(A_1+B_1))}{(\cosh(A_1+B_1)-\cos(A_2+B_2))^2}(\alpha-\beta)^2.
\end{align*}
From \eqref{turnoverconditionnew}, \eqref{fLequation},  $\partial_{\alpha}\tilde{f}_1^{L}(0,t)=0$. Moreover, $\partial_{\alpha}D^{-60}h(0,\gamma,t)=0$. Then we have 
\[\lim_{\alpha\to 0}\lim_{\alpha\to\beta}\frac{\sin(A_1+B_1)}{(\alpha-\beta)}=\lim_{\alpha\to 0^{+}}\lim_{\alpha\to\beta}\frac{A_1+B_1}{(\alpha-\beta)}=\partial_{\alpha}\tilde{f}_1^{L}(0,t)=0,
\]
and the vanishing condition follows.

\end{proof}
\begin{lemma}\label{Lh3lemma}
The kernel of the first type \eqref{firsttypeestimate} $\tilde{B}_{1}^{i}(h)$ can be extended to:
\begin{align}\label{boundednesscontildeB1}
\tilde{B}_{1}^{i}(h) \in C_{\gamma}^{0}([-1,1], C_{\alpha,\beta}^{k+2}([0,\frac{\pi}{4}]\times [0,\frac{2}{3}\pi])).
\end{align}Moreover, 
$B_{1}^{i}(h)$ satisfies the smoothness conditions in \eqref{3GE01} with
\begin{align}\label{dBL1}
D_{\gamma}B_1^{i}(h)[w]=\bar{\partial}_{\gamma}B_1^{i}(h)+D_h{B_1^{i}(h)}[w].
\end{align}
\end{lemma}
\begin{proof}
We first show the boundedness condition \eqref{boundednesscontildeB1}.
For the first sub-type $\tilde{B}_{1}^{i}(h)$, we use \eqref{kernelestimatenearorigin} from the previous lemma. For the sub-type 2 $\tilde{B}_{1}^{i}(h)$, we use the smoothness of $X_i$ \eqref{Xispace} and the estimate for kernel $K_{-1}$ lemma \ref{alphabetanear} directly. For the boundedness condition of $B_1^i(h)$ \eqref{boundestimateB}, we only need to show when $b_i=60$, we have
$$
\|B_1^i(h)\|_{H_{\alpha}^{k}[0,\frac{\pi}{4}]}\lesssim\|\tilde{B}_1^i(h)\|_{C_{\alpha,\beta}^{k+2}([0,\frac{\pi}{4}]\times [0,\frac{2}{3}\pi])}\|h\|_{H_{\alpha}^{k}[0,\pi]}.
$$
In fact, since $\lambda$ is sufficiently smooth, we only need to control the integral. We have for $j\leq k$:
\begin{align*}
&\partial_{\alpha}^{j}(p.v.\int_{0}^{2\alpha} \tilde{B}_1^i(h)(\alpha,\beta,\gamma)\frac{h_v(\beta,\gamma,t)}{(\alpha-\beta)}d\beta)\\
&=\partial_{\alpha}^{j}(\int_{0}^{2\alpha}  (\tilde{B}_1^i(h)(\alpha,\beta,\gamma)-\tilde{B}_1^i(h)(\alpha,\alpha,\gamma))\frac{h_v(\beta,\gamma,t)}{(\alpha-\beta)}d\beta)\\&+\partial_{\alpha}^{j}(\tilde{B}_1^i(h)(\alpha,\alpha,\gamma) \int_{0}^{2\alpha} \frac{h_v(\beta,\gamma,t)-h_v(\alpha,\gamma,t)}{(\alpha-\beta)}d\beta)\\
&=Term_1+Term_2.
\end{align*}
For $Term_1$, since $\|\frac{\tilde{B}_1^i(h)(\alpha,\beta,\gamma)-\tilde{B}_1^i(h)(\alpha,\alpha,\gamma)}{\alpha-\beta}\|_{C_{\alpha,\beta}^{k+1}([0,\frac{\pi}{4}]\times [0,\frac{2}{3}\pi])}\lesssim \|\tilde{B}_1^i(h)\|_{C_{\alpha,\beta}^{k+2}([0,\frac{\pi}{4}]\times [0,\frac{2}{3}\pi])}$, we have the desired estimate. For $Term_2$, we only need to show
$$\|\partial_{\alpha}^{j}\int_{0}^{2\alpha} \frac{h_v(\beta,\gamma,t)-h_v(\alpha,\gamma,t)}{(\alpha-\beta)}d\beta\|_{L_{\alpha}^2[0,\frac{\pi}{4}]}\lesssim \|h\|_{X^{k}}.$$ We have
\begin{align*}
&\quad\partial_{\alpha}^{j}(\int_{0}^{2\alpha} \frac{h_v(\beta,\gamma,t)-h_v(\alpha,\gamma,t)}{(\alpha-\beta)}d\beta)\\
&=\partial_{\alpha}^{j}(\int_{-\alpha}^{\alpha} \frac{h_v(\alpha-\beta,\gamma,t)-h_v(\alpha,\gamma,t)}{\beta}d\beta)\\
&=\int_{-\alpha}^{\alpha} \frac{\partial_{\alpha}^{j}h_v(\alpha-\beta,\gamma,t)-\partial_{\alpha}^{j}h_v(\alpha,\gamma,t)}{\beta}d\beta\\
&\quad+\sum_{j'\leq j-1} C_{j,j'}\partial_{\alpha}^{j-j'-1}\frac{h_{v}^{(j')}(0,\gamma,t)-h_{v}^{(j')}(\alpha,\gamma,t)}{\alpha}+\sum_{j'\leq j-1} \tilde{C}_{j,j'}\partial_{\alpha}^{j-j'-1}\frac{h_{v}^{(j')}(2\alpha,\gamma,t)-h_{v}^{(j')}(\alpha,\gamma,t)}{\alpha}\\
&=Term_{2,1}+Term_{2,2}+Term_{2,3}.
\end{align*}
Then from Hilbert's inequality, $Term_{2,1}$ can be controlled. For $Term_{2,2}$, we can use the following estimate:
\begin{align*}
&\quad\|\partial_{\alpha}^{j-j'-1}\frac{h_{v}^{(j')}(0,\gamma,t)-h_{v}^{(j')}(\alpha,\gamma,t)}{\alpha}\|_{L^2_{\alpha}[0,\frac{\pi}{4}]}\\
&=\|\int_{0}^{1}h_{v}^{(j)}(\alpha\eta,\gamma,t)\eta^{j-j'-1}d\eta\|_{L^2_{\alpha}[0,\frac{\pi}{4}]}\\
&\lesssim \|h_{v}(\alpha,\gamma,t)\|_{H^{k}}.
\end{align*}

$Term_{2,3}$ can be controlled in a similar way as $Term_{2,2}.$

For $D_\gamma B_1^{+}(h)[w]$, \eqref{dBL1} holds through direct calculation by finding the operator that satisfies
\[
D_\gamma B_1^{i}(h)[\frac{dh}{d\gamma}]=\frac{d}{d\gamma}B_1^{i}(h).
\]
Other smoothness conditions of $B_1^{i}(h)$ hold similarly as in the boundedness estimate.
\end{proof}
\begin{lemma}\label{Lh4lemma}
The kernel $\tilde{B}_{2}^{i}(h)$ \eqref{secondtypeestimate} in second type term can be extended to 
\begin{equation}\label{boundednesscontildeB2}
\tilde{B}_{2}^{i}(h)\in \bar{Y}^{k+2}:= C_{\gamma}^{0}([-1,1], C_{\alpha}^{k+2}[0,\frac{1}{4}\pi]).
\end{equation}

Moreover, 
$B_{2}^{i}(h)$ satisfies the smoothness conditions for $B_i(h)$ in \eqref{3GE01}, with 
\begin{align}\label{dBL2}
D_{\gamma}B_2^{i}(h)[w]=\bar{\partial}_{\gamma}B_2^{i}(h)+D_h{B_2^{i}(h)}[w].
\end{align}
\end{lemma}
\begin{proof}
We first show the condition \eqref{boundednesscontildeB2}. The boundedness condition \eqref{boundestimateB} follows directly from  \eqref{boundednesscontildeB2} and \eqref{secondtypeestimate}.

The sub-type 1 has the same form as in lemma \ref{L1main}. Therefore we can use lemma \ref{L1main} to get the result.

The sub-type 2 relies on the particular $K_{-1}^{j}(A,B,C)=\frac{\sin(A_1+B_1)}{\cosh(A_2+B_2)-\cos(A_1+B_1)}$. We will show this for the first type and the other one follows the same way.
For this boundedness condition, the main problem comes from the $-1$ th order singularity given by the index of the kernel. This is fixed by the fact that, for this particular kernel,  $B_1=\tilde{f}^{L}(\al_{\g}^{t})$ gives one extra order of $\al$ in the numerator. 

More specifically,  we have 
\begin{align*}
&\quad A_1+B_1|_{\beta=2\alpha}\\
&=D^{-60}h_1(\alpha,\gamma,t)-D^{-60}h_1(2\alpha,\gamma,t)+\tilde{f}^{L}_{1}(\alpha_{\gamma}^{t},t)-\tilde{f}^{L}_{1}((2\alpha)_{\gamma}^{t},t)\\
    &=-(\alpha)^2\int_{0}^{1}\frac{D^{-59}h_1(\alpha+(\alpha)\zeta,\gamma,t)}{\alpha}(1+ic'((\alpha)(1+\zeta))\gamma t)d\zeta\\
    &\quad -(\alpha)^2\int_{0}^{1}\frac{(\partial_{\alpha}\tilde{f}^{L}_1)((\alpha)(1+\zeta)+ic((\alpha)(1+\zeta))\gamma t)}{(\alpha)}(1+ic'((\alpha)(1+\zeta))\gamma t)d\zeta .
\end{align*}
Since $D^{-59}h_1(\alpha+\zeta(\alpha),\gamma,t)|_{\alpha=0}=0$ and $\partial_{\alpha}\tilde{f}^{L}_{1}(0,t)=0$ \eqref{turnoverconditionnew}, \eqref{fLequation}, from the smoothness of $\tilde{f}^{L}$ \eqref{fLspace} we have
\begin{align*}
    \|\frac{\sin(A_1+B_1)|_{\beta=2\alpha}}{(\alpha)^2}\|_{\overline{Y}^{k+3}}\lesssim 1.
\end{align*}
Then the boundedness estimate follows from the natural upper bound for the denominator:
\begin{align*}
 \|\frac{\cosh(A_2+B_2)-\cos(A_1+B_1)|_{\beta=2\alpha}}{(\alpha)^2}\|_{\overline{Y}^{k+3}}\lesssim 1,
\end{align*}
and the lower bound comes from the R-T condition lemma \ref{arcchord}.
\begin{align*}
 |\frac{\cosh(A_2+B_2)-\cos(A_1+B_1)|_{\beta=2\alpha}}{(\alpha)^2}|\gtrsim 1.
\end{align*}

For sub-type 3 and sub-type 4, from the estimate for the kernel $K_{-\sigma}$ in this setting (corollary \ref{alphabetafarright}), we know the derivative $\partial_{\al}^{k+3}$ gives at most $(\frac{1}{\beta})^{k+6}$ th order singularity. This is compensated by  $|\tilde{f}^{+(b_i)}(\beta)|\lesssim |\beta^{38}|$ \eqref{fLspace}. For sub-type 5, we could use \eqref{kernelestimatenearorigin} from the previous lemma. 
Sub-type  6 is trivial from the smoothness condition of $\tilde{X}$ \eqref{Xispace}.

For $D_\gamma B_2^{+}(h)[w]$, \eqref{dBL2} holds through direct calculation by finding the operator that satisfies
\[
D_\gamma B_2^{i}(h)[\frac{dh}{d\gamma}]=\frac{d}{d\gamma}B_2^{i}(h).
\]
Other smoothness conditions hold similarly as in the boundedness estimate.
\end{proof}
\begin{lemma}\label{Lh5lemma}
The kernel $\tilde{B}_{3}^{i}(h)$ \eqref{thirdtypeestimate} in third type term  can be extended to
\begin{align}\label{boundednesscontildeB3}
\tilde{B}_{3}^{i}(h)\in C_{\gamma,\eta}^{0}([-1,1]\times[-1,1], C_{\alpha}^{k+2}[0,\frac{1}{4}\pi]).
\end{align}
Moreover, $B_{3}^{i}(h)$ satisfies the smoothness conditions for $B_{i}$ in \eqref{3GE01}.
\end{lemma}
Moreover, we have
\begin{align}\label{dBL3}
&\quad D_{\gamma}B_3^{i}(h)[w]\\\nonumber
&=\lambda(\alpha_{\gamma}^{t})D^{-60+b_i}h_v(2\alpha,\gamma,t)\tilde{B}_3^{i}(h)(\alpha,\gamma,\gamma,t)+\lambda(\alpha_{\gamma}^{t})\int_{0}^{\gamma}D^{-60+b_i}h_v(2\alpha,\eta,t)D_{\gamma}\tilde{B}_3^{i}(h)[w](\alpha,\gamma,\eta,t)d\eta.
\end{align}
Here for sub-type 1 
\begin{align*}
    &\quad D_{\gamma}\tilde{B}_3^{i}(h)[w](\alpha,\gamma,\eta,t)\\\nonumber
    &=\bar{\partial_{\gamma}}\tilde{B}_3^{i}(h)(\alpha,\gamma,\eta,t)\\
    &\quad+\frac{d}{d\zeta}\frac{d}{d\beta}(K_{-1}^{i}( V_h^{30}(\alpha,\gamma,t)+ \zeta V_w^{30}(\alpha,\gamma,t)-V_h^{30}(\beta,\eta,t), V_{\tilde{f}^{L}}^{61}(\alpha_{\gamma}^{t},t)-V_{\tilde{f}^{L}}^{61}(\beta_{\eta}^{t},t), V_{X}(\alpha_{\gamma}^{t},t)-V_{X}(\beta_{\eta}^{t},t))|_{\beta=2\alpha,\zeta=0}\\
    &\qquad\times \frac{ic(2\alpha)t}{1+ic'(2\alpha)\eta t}.
\end{align*}
and for sub-type 2
\begin{align*}
     &\quad D_{\gamma}\tilde{B}_3^{i}(h)[w](\alpha,\gamma,\eta,t)\\\nonumber
     &=\bar{\partial_{\gamma}}\tilde{B}_3^{i}(h)(\alpha,\gamma,\eta,t)+\\\nonumber
    &\frac{d}{d\zeta}(K_{-1}^{i}( V_h^{30}(\alpha,\gamma,t)+ \zeta V_w^{30}(\alpha,\gamma,t)-V_h^{30}(\beta,\eta,t), V_{\tilde{f}^{L}}^{61}(\alpha_{\gamma}^{t},t)-V_{\tilde{f}^{L}}^{61}(\beta_{\eta}^{t},t), V_{X}(\alpha_{\gamma}^{t},t)-V_{X}(\beta_{\eta}^{t},t))|_{\beta=2\alpha,\zeta=0}\\\nonumber
    &\qquad\times ic(2\alpha)t,
\end{align*}
with  $V_w^{30}=(D^{-60}w_1,D^{-60}w_2,D^{-59}w_1,D^{-59}w_2,....D^{-30}w_1, D^{-30}w_2)$ as in \eqref{notation03h}.
\begin{proof}
We first show the condition \eqref{boundednesscontildeB3}. The boundedness condition \eqref{boundestimateB} follows directly from \eqref{boundednesscontildeB3} and \eqref{thirdtypeestimate}.

From sufficiently good smoothness of $V_{h}^{30}$ \eqref{D-formularnew}, \eqref{notation03h}, $\tilde{f}^{L}$ \eqref{fLspace}, $X$ and $\tilde{X}$ \eqref{Xispace},  we claim
\begin{align*}
&\quad\frac{d}{d\beta}K_{-1}(V_h^{30}(\alpha,\gamma,t)-V_h^{30}(\beta,\eta,t), V_{\tilde{f}^{L}}^{61}(\alpha_{\gamma}^{t},t)-V_{\tilde{f}^{L}}^{61}(\beta_{\eta}^{t},t), V_{X}(\alpha_{\gamma}^{t},t)-V_{X}(\alpha_{\eta}^{t},t))\\
&=\sum_{l'}K_{-(2)}^{l'}(V_h^{30}(\alpha,\gamma,t)-V_h^{30}(\beta,\eta,t), V_{\tilde{f}^{L}}^{61}(\alpha_{\gamma}^{t},t)-V_{\tilde{f}^{L}}^{61}(\beta_{\eta}^{t},t), V_{X}(\alpha_{\gamma}^{t},t)-V_{X}(\alpha_{\eta}^{t},t))G_{l'}(\beta,\g,\eta,t),
\end{align*}
with $G_{l'}(
\beta,\g,\eta,t)\in C_{\g,\eta}^{0}([-1,1]\times[-1,1], C_{\beta}^{k+4}[0,\pi]).$
Then from the estimate of kernel of $K_{-1}$ and $K_{-2}$ (lemma \ref{gammachange}), we only need to show the singularity of $\frac{1}{\alpha^{2}}$ can be compensated. This is done from the definition of $c(\alpha)$ \eqref{cdefi} which satisfies:
\begin{align}\label{calphaspace}
\frac{ic(2\alpha)}{(\alpha)^2}\in C^{97}_{\alpha}[0,\pi].
\end{align}
For $D_\gamma B_3^{+}(h)[w]$, \eqref{dBL3} holds through direct calculation by finding the operator that satisfies
\[
D_\gamma B_3^{i}(h)[\frac{dh}{d\gamma}]=\frac{d}{d\gamma}B_3^{i}(h).
\]
Other smoothness conditions hold similarly as in the boundedness estimate.
\end{proof}

\begin{lemma}\label{Lhjlemma}
The $B_4^{i}$, kernel $\tilde{B}_{4}^{i}$ \eqref{fourthtypeestimate} in fourth type term can be extended to
\begin{align}\label{boundednesscontildeB4}
 \tilde{B}_{4}^{i}\in C_{\gamma}^{4}([-1,1], C_{\alpha,\beta}^{14}(\bar{\Omega})),
\end{align}
 with $\Omega=\{(\alpha,\beta)|\alpha\in (0,\frac{1}{4}\pi], \beta \in (0,\frac{3}{2}\pi], \beta \geq 2\alpha\}$, 
 \begin{equation}\label{B4iestimate}
 B_{4}^{i}\in C_{\gamma}^{4}([-1,1], C_{\alpha}^{14}[0,\frac{\pi}{4}]),
 \end{equation}
 and satisfies when $l''+j'\leq 13$, $l\leq 4$:
 \begin{align}\label{condition left term boundary0}
    \lim_{\alpha\to 0^{+}}(\frac{d}{d\gamma})^{l}(\frac{d}{d\alpha})^{l''}((\partial_{\alpha}^{j'}\tilde{B}_4^{i})(\alpha_{\gamma}^{t},2\alpha,t))=0.
\end{align}Moreover, $B_{4}^{i}$ satisfies the smoothness conditions  for $B_i$ in \eqref{3GE01} with
\begin{align}\label{dBL4}
D_{\gamma}B_4^{i}[w]=\partial_{\gamma}B_{4}^{i}=\bar{\partial}_{\gamma}B_{4}^{i}.
\end{align}
\end{lemma}
\begin{proof}
We first show the boundedness condition \eqref{boundednesscontildeB4}, \eqref{B4iestimate}.  \eqref{boundestimateB} follows direcly from \eqref{boundednesscontildeB4} and \eqref{fourthtypeestimate}.

From sufficiently good smoothness of  $\tilde{f}^{L}$ \eqref{fLspace}, $X$ and $\tilde{X}$ \eqref{Xispace}, 
We have
\begin{equation}\label{dbetaksigma4}
\begin{split}
&\frac{d}{d\beta}K_{-1}(-V^{30}(\beta_0^{t}),\Delta V_{\tilde{f}^{L}}^{61}(\alpha_{\gamma}^{t},t),\Delta V_{X}(\alpha_{\gamma}^{t},t))\\
&=\sum_{l'}K_{-2}^{l'}(-V^{30}(\beta_0^{t},t),\Delta V_{\tilde{f}^{L}}^{61}(\alpha_{\gamma}^{t},t),\Delta V_{X}(\alpha_{\gamma}^{t},t)) G_{l'}(\alpha, \beta,\gamma, t),
\end{split}
\end{equation}
with $G_{l'}(,\alpha, \beta,\gamma, t)\in C_{\gamma}^{4}([-1,1], C_{\alpha,\beta}^{16}([0,\pi]\times [0,\pi])).$
The same expression holds for $\frac{d}{d\alpha}$ and $\frac{d}{d\gamma}$.  
Then the estimates \eqref{boundednesscontildeB4}, \eqref{B4iestimate} for both sub-types follows from kernel estimate corollary \ref{alphabetafarright} and the fact that $|\tilde{f}^{+(b_i)}(\beta)|\lesssim |\beta|^{38}$ when $b_i\leq 60$.
\eqref{condition left term boundary0} also holds from the same argument.

For $D_\gamma B_4^{i}[w]$, \eqref{dBL4} holds through direct calculation by finding the operator that satisfies
\[
D_\gamma B_4^{i}[\frac{dh}{d\gamma}]=\frac{d}{d\gamma}B_4^{i}.
\]
We also use the fact that $B_{4}^{i}$ does not depend on $h$.

Other smoothness conditions hold similarly as in the boundedness estimate.

\end{proof}
\begin{lemma}\label{Tfixed behaviour}
 $\lambda(\al_{\gamma}^{t})T_{fixed}(\alpha_{\gamma}^t,t)\in {C_{\gamma}^{4}([-1,1],C^{16}_{\alpha}[0,\frac{\pi}{4}])}$ and satisfies the smoothness conditions in \eqref{3GE01}.
\end{lemma}
\begin{proof}
From \eqref{Tfixed}, we have
\begin{align}\label{Tfixedequation}
&\quad T_{fixed}(\alpha,t)\\\nonumber
&=\partial_{\alpha}^{60}(-\frac{d\tilde{f}_{\mu}^{L}(\alpha,t)}{dt}+\frac{d\tilde{f}_{\mu}^{L}}{d\alpha}(\alpha,t)(\frac{\frac{dx(\alpha,t)}{dt}}{\frac{dx(\alpha,t)}{d\alpha}})\\\nonumber&\quad+\int_{-\pi}^{\pi}K(\tilde{f}^{L}(\alpha,t)-\tilde{f}^{L}(\beta,t))(\frac{\partial_{\alpha}\tilde{f}_{\mu}^{L}(\alpha,t)}{\frac{dx(\alpha,t)}{d\alpha}}-\frac{\partial_{\beta}\tilde{f}_{\mu}^{L}(\beta,t)}{\frac{dx(\beta,t)}{d\beta}})(\frac{dx(\beta,t)}{d\beta})d\beta)\\\nonumber
&=T_{fixed}^1(\alpha,t)+T_{fixed}^2(\alpha,t)+T_{fixed}^3(\alpha,t).
\end{align}
We have $\tilde{f}^{L}(\alpha,t)\in C^{1}([0,t], C^{99}(-\infty,\infty))$ and analytic in $\{\alpha+iy|0<\alpha<\frac{\pi}{2}\}$ \eqref{fLspace}, then $T_{fixed}^1(\alpha_{\gamma}^{t},t)$, $T_{fixed}^2(\alpha_{\gamma}^{t},t)$ satisfy the conditions. 
For $T_{fixed}^3(\alpha,t)$, first from the arc-chord condition \eqref{arcchord1}, it is well defined. Moreover, we have
\begin{align}\label{Tfixedequation2}
   T_{fixed}^3(\alpha,t)= \sum_{j}\int_{-\pi}^{\pi}K_{0}^{j}(0,V^{61}_{\tilde{f}^{L}}(\alpha,t)-V^{61}_{\tilde{f}^{L}}(\beta,t),V_{X}(\alpha,t)-V_{X}(\beta,t))X_j(\beta,t)d\beta,
\end{align}
for some $K_{0}$ type kernels.
Then 
\begin{align}\label{Tfixedequation2}
   T_{fixed}^3(\alpha_{\gamma}^{t},t)= \sum_{j}\int_{-\pi}^{\pi}K^{j}_{0}(0,V^{61}_{\tilde{f}^{L}}(\alpha_{\gamma}^{t},t)-V^{61}_{\tilde{f}^{L}}(\beta_{\gamma}^{t},t),V_{X}(\alpha_{\gamma}^{t},t)-V_{X}(\beta_{\gamma}^{t},t))X_j(\beta_{\gamma}^{t},t)(1+ic'(\beta)\gamma t)d\beta,
\end{align}
for some $K_{0}$ type kernel $K_{0}^{j}$. Then the smoothness conditions follow from the estimate of the kernel (lemma \ref{alphabetanear}).
\end{proof}
\subsubsection{The Refined R-T condition when $\kappa(t)>0$}
In this section, we use the following lemma to show the refined R-T condition is satisfied. We will first show at $t=0$ the condition is satisfied. Then we use the smoothness condition of $L_i^{+}(h)$ to show if $t$ is sufficiently small, the sign will not change.
\begin{lemma}\label{rfR-T}
When $\alpha\in (0, \frac{\pi}{4}) \cap \supp{\lambda}(\alpha)$, for $h$, $t$
 satisfying
 \[
 h\in X^{1},
 \]
\[
\|h(\alpha,\gamma,t)-h(\al,\g,0)\|_{X^{1}}\lesssim\delta_s,
\]
\[
h(\al,\g,0)=\tilde{f}^{+(60)}(\al,0),
\]
and
\[
0\leq t\leq t_s,
\] 

there exists $\delta$ \eqref{lambda0defi}, $t_s$, $\delta_s$ sufficiently small, such that we have the following inequality:

\begin{equation}\label{Refined R-T conditon0}
-\Re L_2^{+}(h)(\alpha,\gamma,t)\geq 18 |\Im L_1^{+}(h)(\alpha,\gamma,t)|+18|\Im L_2^{+}(h)(\alpha,\gamma,t)|.
\end{equation}
\end{lemma}
\begin{proof}
From vanishing conditions of $L_1^{+}(h)$ and $L_2^{+}(h)$ (lemma \ref{Lh1lemma} and \ref{L1main}), we have 
\[
\Re L_2^{+}(h)(0,\gamma,t)=\Im L_1^{+}(h)(0,\gamma,t)=\Im L_1^{+}(h)(0,\gamma,t)=0.
\]
Since $\supp\tilde{\lambda}(\alpha_\gamma^t)\subset [-20\delta,20\delta]$, it is enough to show the following refined R-T condition for $\alpha\in (0, 20\delta]$:
\begin{equation}\label{Refined R-T conditon}
-\partial_{\alpha}\Re L_2^{+}(h)(\alpha,\gamma,t)-18|\partial_{\alpha}\Im L_1^{+}(h)(\al,\gamma,t)|-18|\partial_{\alpha}\Im L_2^{+}(h)(\al,\gamma,t)|>0.
\end{equation}
We first show a lower bound for $-\partial_{\alpha}\Re L_2^{+}(h)(\alpha,\gamma,t)$ when $t=0$:
\begin{lemma}\label{secondederivative}
When $t=0$, if $\delta$ is sufficiently small, there exists $C_0>0$ depending on $\delta$ such that
\begin{align*}
  -\partial_{\alpha}\Re L_2^{+}(h)(\alpha,\gamma,0)> C_{0}. 
    \end{align*}
 when $\alpha\in [0,20\delta]$.
\end{lemma}
\begin{proof}
Recall that from \eqref{3M21chequG01},
\[
-L_2^{+}(h)(\alpha,\gamma,t)=\lim_{\beta\to\alpha}(\frac{d}{d\beta}(\frac{\sin(A_1+B_1)}{\cosh(A_2+B_2)-\cos(A_1+B_1)})(\alpha-\beta)^2),
\]
where $A+B=D^{-60}(h)(\alpha,\gamma,t)-D^{-60}(h)(\beta,\gamma,t)+\tilde{f}^{L}(\alpha_{\gamma}^{t},t)-\tilde{f}^{L}(\beta_{\gamma}^{t},t)$.

When $t=0$, since $h(\alpha,\gamma,0)=\tilde{f}^{+(60)}(\alpha,0)$, from \eqref{cut function0}, \eqref{D-formularnew} we have
\[
A+B=\tilde{f}^{+}(\alpha,0)-\tilde{f}^{+}(\beta,0)+\tilde{f}^{L}(\alpha,0)-\tilde{f}^{L}(\beta,0)=\tilde{f}(\alpha,0)-\tilde{f}(\beta,0).
\]
Then 

\begin{align*}
    &-L_2^{+}(h)(\alpha,\gamma,0)=\lim_{\beta\to\alpha}(\frac{d}{d\beta}(\frac{\sin(\tilde{f}_1(\alpha,0)-\tilde{f}_1(\beta,0))}{\cosh(\tilde{f}_2(\alpha,0)-\tilde{f}_2(\beta,0))-\cos(\tilde{f}_1(\alpha,0)-\tilde{f}_1(\beta,0))})(\alpha-\beta)^2)\\
    &=\lim_{\beta\to\alpha}((\frac{-\cos(\tilde{f}_1(\alpha,0)-\tilde{f}_1(\beta,0))\tilde{f}_1^{'}(\beta,0)}{\cosh(\tilde{f}_2(\alpha,0)-\tilde{f}_2(\beta,0))-\cos(\tilde{f}_1(\alpha,0)-\tilde{f}_1(\beta,0))}-\\
    &\frac{\sin(\tilde{f}_1(\alpha,0)-\tilde{f}_1(\beta,0))(-\sinh(\tilde{f}_2(\alpha,0)-\tilde{f}_2(\beta,0))\tilde{f}_2^{'}(\beta,0)-\sin(\tilde{f}_1(\alpha,0)-\tilde{f}_1(\beta,0))\tilde{f}_1^{'}(\beta,0))}{(\cosh(\tilde{f}_2(\alpha,0)-\tilde{f}_2(\beta,0))-\cos(\tilde{f}_1(\alpha,0)-\tilde{f}_1(\beta,0)))^2})(\alpha-\beta)^2)\\
    &=-\frac{2\tilde{f}_1^{'}(\alpha,0)}{(\tilde{f}_1^{'}(\alpha,0))^2+(\tilde{f}_2^{'}(\alpha,0))^2}+\frac{4\tilde{f}_1^{'}(\alpha,0)}{(\tilde{f}_1^{'}(\alpha,0))^2+(\tilde{f}_2^{'}(\alpha,0))^2}\\
    &=2\frac{\tilde{f}_1^{'}(\alpha,0)}{(\tilde{f}_1^{'}(\alpha,0))^2+(\tilde{f}_2^{'}(\alpha,0))^2}
\end{align*}
By the assumption \eqref{assumption} we have $\frac{d^2}{d\alpha^2}\tilde{f}_1(\alpha,0)|_{\alpha=0}>0$, moreover, by the condition of turnover points, we have
$\frac{d}{d\alpha}\tilde{f}_1(\alpha,0)|_{\alpha=0}=0$. From the arc-chord condition \eqref{arcchord0}, we have $(\tilde{f}_1^{'}(\alpha,0))^2+(\tilde{f}_2^{'}(\alpha,0))^2>0$. Then
by letting $\delta$ sufficiently small, we get the result.
\end{proof}
\begin{lemma}
When $\delta$ is sufficiently small, 
\eqref{Refined R-T conditon} holds for small $\delta_s$, $t_s$.
\end{lemma}
\begin{proof}
From lemma \ref{secondederivative}, we know at $t=0$,
\begin{align*}
    -\partial_{\alpha}\Re L_2(\tilde{f}^{{+}(60)})(\alpha,\gamma,0)> C_0,
    \end{align*}
    when $\alpha\in [0,20\delta]$.
    $\partial_{\alpha}\Im L_1^{+}(h)(\al,\gamma,0)=\partial_{\alpha}\Im L_2^{+}(h)(\al,\gamma,0)=0$ because everything is real at $t=0$.
From smoothness condition of $L_{i}^{+}(h)$ lemmas, \ref{L1main}, \ref{Lh1lemma}, we also have
\begin{align*}
   & \quad \|L_i^{+}(h)(\alpha,\gamma,t)-L_i^{+}(h)(\alpha,\gamma,0)\|_{C_{\alpha}^{3}[0,\frac{\pi}{4}]}\\
    &\lesssim \|h(\alpha,\gamma,t)-h(\alpha,\gamma,0)\|_{X^{1}}+\mathcal{O}(t)\\
    &=\|h(\alpha,\gamma,t)-\tilde{f}^{{+}(60)}(\al,0)\|_{X^{1}}+\mathcal{O}(t)
\end{align*}
Therefore when $\alpha\in [0,20\delta]$, we have
\begin{align*}
&|\partial_{\alpha}\Re L_2^{+}(h)(\al,\g,t)-\partial_{\alpha}\Re L_2^{+}(h)(\al,\g,0)|+|\partial_{\alpha}\Im L_2^{+}(h)(\al,\g,t)-\partial_{\alpha}\Im L_2^{+}(h)(\al,\g,0)|\\
&+|\partial_{\alpha} \Im L_1^{+}(h)(\al,\g,t)-\partial_{\alpha} \Im L_1^{+}(h)(\al,\g,0)|\lesssim O(t)+\|h(\alpha,\gamma,t)-\tilde{f}^{+(60)}(\alpha,t)\|_{X^1}.
\end{align*}
%Moreover, when $0<\alpha< 0$, from \eqref{Lhcondition3}, we have
%\begin{align*}
    %&\|L_{h^{\epsilon,s}}^{j}(\alpha,\gamma,t)-L_{h^{\epsilon,s}}^{j}(-\tau(t'),\gamma,t)\|_{C^{0}_{\gamma}([-1,1],C^{1}_{\alpha}[0,-\tau(t')])}\\
%&\leq (\tau(t)-\tau(t'))\|L^{i}_{h^{\epsilon,s}}(\alpha,\gamma,t)\|_{C^{0}_{\gamma}([-1,1],C^{2}_{\alpha}[0,-\tau(t')]}\lesssim B(\|h\|_{U})(\tau(t)-\tau(t')).
%\end{align*}
%Then 
%\begin{align*}
    %&|\partial_{\alpha}d(\alpha,\gamma,t)-\partial_{\alpha}d(0,\gamma,t)|+|\partial_{\alpha}d_2(\alpha,\gamma,t)-\partial_{\alpha}d_2(0,\gamma,t)|+|\partial_{\alpha}d_3(\alpha,\gamma,t)-\partial_{\alpha}d_3(0,\gamma,t)|\lesssim B(\|h\|_{U})|\tau(t)|.
%\end{align*}
 Hence there exists small $\delta$, $t_{s}$, $\delta_{s}$, such that \eqref{Refined R-T conditon} holds.
 \end{proof}
\end{proof}

\section{Existence and uniqueness for a generalized equation for $\kappa(t)>0$ }\label{kappa1sectionexistence}
In this section, we aim to show the existence and uniqueness theorem \ref{existence1theorem}.
\begin{theorem}\label{existence1theorem}
    For any equation $T^{+}(h)$ and initial data $h(\al,\g,0)$ satisfying the generalized equation \eqref{3GE01}.  There exists sufficiently small time $\bar{t}$ and solution $h(\al,\g,t)$ such that  
    \begin{align}\label{hspaceconclusion}
\begin{cases}
    &h(\al,\g,t)\in C_{t}^1([0,\bar{t}], X^{5}), \frac{d}{d\g}h\in C_{t}^{1}([0,\bar{t}], X^{4}),\\
    &h(\al,\g,t)\in C_{\g}^{1}([-1,1],C_{\al}^{2}[0,\pi]),\\
    &\frac{d}{d\g}\frac{d}{dt}h=\frac{d}{dt}\frac{d}{d\g}h.
\end{cases}
\end{align}
Moreover, for all $g(\al,t)\in C_{t}^{1}([0,\bar{t}],H^{5}[0,\pi]))\cap\{h|\supp h \subset  [0,\frac{\pi}{4}]\}$ satisfying the equation \eqref{3GE01} when $\gamma=0$ and the initial data condition $g(\al,0)= h(\al,\gamma,0)=\tilde{f}^{+(60)}(\alpha,t)$, we have  when $t\leq \bar{t}$,
\begin{equation}\label{huniqueness01}
g(\al,t)=h(\al,0,t).
\end{equation}
\label{model1generalizedequation}\end{theorem}
\begin{rem}
We do not assume any periodic boundary conditions in space $H_{\alpha}^{k}[0,\pi]$ and $C_{\alpha}^{k+2}[0,\pi]$.
\end{rem}

\subsection{The existence of $h$}\label{hexsi}
We plan to show after suitable perturbation of $T$ to $T^{\epsilon}$, we have the energy estimate
\[
\Re<T^{\ep}(h),h>_{X^{k}}\leq C(\|h\|_{X^{k}}),
\]
for $1\leq k\leq 12$, with $C(\cdot)$ a bounded function.
Then we can apply the Picard theorem for $\epsilon>0$ and use the energy estimate to take the limit of $\epsilon$ to 0. 
\subsubsection{Perturbation}
We first perturb the main terms $M_{1,1}$, $M_{1,2}$ and $M_{2,1}$. The terms including $h^{(j)}(0)$ deal with the singularity caused by the boundary ($0$ and $2\alpha$ in the integral) when $\ep>0$.  Notice that this kind of singularity will not happen in the limit case when $\ep=0$.

%The perturbation includes 2 steps. After first step, the operator behaves well on $X\cap \{h|\partial_{\alpha}^{j}h(0,\gamma)=0, 0\leq j\leq k-1\}$. The second step 

%For the first step, we have
%\begin{equation}\label{3M11perturb}
%\begin{split}
%M_{1,1}^{\epsilon}(h)=\kappa(t)\lambda(\alpha)\frac{h(\alpha+\epsilon)-h(\alpha)}{\epsilon}.
%\end{split}
%\end{equation}
%\begin{equation}\label{3M12perturb}
%\begin{split}
%M_{1,2}^{\epsilon}(h)=\lambda(\alpha)\int_{0}^{2\alpha}\frac{(h(\alpha)-h(\beta))(\alpha-\beta)\epsilon}{((\alpha-\beta)^2+\epsilon^2)^2}d\beta \frac{2}{\pi}L_1^{+}(h).
%\end{split}
%\end{equation}

%\begin{equation}\label{3M22perturb}
%\begin{split}
%M_{2,1}^{\epsilon}&=-%\lambda(\alpha)L_2^{+}(h)\int_{0}^{2\al}\frac{h(\alpha)-h(\beta)}{(\alpha-\beta)^2+\epsilon^2}d\beta.
%\end{split}
%\end{equation}
For any $1\leq k \leq 12$, we let
\begin{equation}\label{3M11perturb2}
\begin{split}
\tilde{M}_{1,1}^{\epsilon}(h)=\lambda(\alpha)\frac{h(\alpha+\epsilon)-h(\alpha)}{\epsilon}\kappa(t),
\end{split}
\end{equation}
\begin{equation}\label{3M12perturb}
\begin{split}
\tilde{M}_{1,2}^{k,\epsilon}(h)&=\lambda(\alpha)\int_{0}^{2\alpha}\frac{((h(\alpha)-\sum_{j=0}^{k-1}\frac{\al^{j}}{j!}h^{(j)}(0))-(h(\beta)-\sum_{j=0}^{k-1}\frac{\beta^{j}}{j!}h^{(j)}(0)))(\alpha-\beta)\epsilon}{((\alpha-\beta)^2+\epsilon^2)^2}d\beta \frac{2}{\pi}L_1^{+}(h)\\
&+\lambda(\alpha)\sum_{j=1}^{k-1}\frac{\al^{j-1}}{(j-1)!}h^{(j)}(0)L_1^{+}(h),
\end{split}
\end{equation}
\begin{equation}\label{3M22perturb}
\begin{split}
\tilde{M}_{2,1}^{k,\epsilon}(h)&=\lambda(\alpha)L_2^{+}(h)\int_{0}^{2\al}\frac{(h(\alpha)-\sum_{j=0}^{k-1}\frac{\al^{j}}{j!}h^{(j)}(0))-(h(\beta)-\sum_{j=0}^{k-1}\frac{\beta^{j}}{j!}h^{(j)}(0))}{(\alpha-\beta)^2+\epsilon^2}d\beta\\
&+\lambda(\alpha)L_2^{+}(h)\sum_{j=0}^{k-1}h^{(j)}(0)p.v.\int_{0}^{2\alpha}\frac{\frac{\al^{j}}{(j)!}-\frac{\beta^{j}}{(j)!}}{(\alpha-\beta)^2}d\beta.
\end{split}
\end{equation}
Here the notation $\sum_{j=1}^{k-1}$ is taken to be $0$ when $k=1$.

We will keep other terms and let
\begin{equation}\label{perturbationterm}
T^{\epsilon}(h)=\tilde{M}_{1,1}^{\epsilon}+\tilde{M}_{1,2}^{k,\epsilon}+\tilde{M}_{2,1}^{k,\epsilon}+\sum_{i}B_i.
\end{equation}

Now we explain why the terms defined above are perturbation of original $M_{1,1}$, $M_{1,2}$, $M_{2,1}$.
\begin{lemma}\label{limitlemma01}
If $h\in H_{\alpha}^{k}[0,\pi]$ with $k\geq 3$, then we have
\begin{equation}\label{limitcase11}
\lim_{\epsilon\to 0}\tilde{M}_{1,1}^{\epsilon}(h)=M_{1,1}(h),
\end{equation}
\begin{equation}\label{limitcase12}
\lim_{\epsilon\to 0}\tilde{M}_{1,2}^{k,\epsilon}(h)=M_{1,2}(h),
\end{equation}
\begin{equation}\label{limitcase21}
\lim_{\epsilon\to 0}\tilde{M}_{2,1}^{k,\epsilon}(h)=M_{2,1}(h).
\end{equation}
The limit is in the sense of strong convergence in $L^{2}_{\alpha}[0,\pi].$
\end{lemma}
\begin{proof}
\eqref{limitcase11} is trivial from the definition \eqref{3M11perturb2} and \eqref{M11def}.
For \eqref{limitcase12}, we first rewrite the $\tilde{M}_{1,2}^{k,\epsilon}$, $\tilde{M}_{2,1}^{k,\epsilon}$.
Let
\begin{equation}\label{3bj1}
b_{1,j}^{\epsilon}(\al)=\begin{cases}\int_{0}^{2\al}\frac{(\frac{\al^{j}}{j!}-\frac{\beta^{j}}{j!})(\al-\beta)\epsilon}{((\al-\beta)^2+\epsilon^2)^2}d\beta\frac{2}{\pi}-\frac{\al^{j-1}}{(j-1)!} &j\geq 1,\\
0 &j=0,
\end{cases}
\end{equation}
\begin{equation}\label{3bj2}
b_{2,j}^{\epsilon}(\al)=\begin{cases}\int_{0}^{2\al}\frac{\frac{\al^{j}}{j!}-\frac{\beta^{j}}{j!}}{(\al-\beta)^2+\epsilon^2}d\beta-\int_{0}^{2\al}\frac{\frac{\al^{j}}{j!}-\frac{\beta^{j}}{j!}}{(\al-\beta)^2}\beta &j\geq 1\\
0 & j=0.
\end{cases}
\end{equation}
Then we have
\begin{equation}\label{3M12split}
\begin{split}
\tilde{M}_{1,2}^{k,\epsilon}&=\lambda(\alpha)\int_{0}^{2\alpha}\frac{(h(\alpha)-h(\beta))(\alpha-\beta)\epsilon}{((\alpha-\beta)^2+\epsilon^2)^2}d\beta \frac{2}{\pi}L_1^{+}(h)(\al,\g)-\lambda(\alpha)\sum_{j=0}^{k-1}b_{1,j}^{\epsilon}(\al)L_1^{+}(h)(\al,\g)\frac{2}{\pi}h^{(j)}(0,\g)\\
&=\tilde{M}_{1,2,I}^{\ep}(h)+\tilde{M}_{1,2,B}^{\ep}(h).
\end{split}
\end{equation}
\begin{equation}\label{3M21split}
\begin{split}
\tilde{M}_{2,1}^{k,\epsilon}&=\lambda(\alpha)L_2^{+}(h)(\al,\g)\int_{0}^{2\alpha}\frac{(h(\alpha)-h(\beta))}{(\alpha-\beta)^2+\epsilon^2}d\beta -\lambda(\alpha)\sum_{j=0}^{k-1}b_{2,j}^{\epsilon}(\al)L_2^{+}(h)(\al,\g)h^{(j)}(0,\g)\\
&=\tilde{M}_{2,1,I}^{\ep}(h)+\tilde{M}_{2,1,B}^{\ep}(h).
\end{split}
\end{equation}
Now we introduce a lemma to control the limit behavior
\begin{lemma}\label{limitlemme0}
For $g(\al)\in H^{3}_{\alpha}[0,\pi]$, we have
\begin{equation}\label{3b1g0}
\|\al\frac{2}{\pi}\int_{0}^{2\al}\frac{(g(\al)-g(\beta))(\al-\beta)\ep}{((\al-\beta)^2+\ep^2)^2}d\beta-g'(\al)\al\|_{H_{\al}^{1}[0,\frac{\pi}{4}]}\lesssim_{g}O(\epsilon),
\end{equation}
and
\begin{equation}\label{3b2g0}
\|\int_{0}^{2\al}\frac{(g(\al)-g(\beta))}{((\al-\beta)^2+\ep^2)}d\beta-p.v.\int_{0}^{2\al}\frac{(g(\al)-g(\beta))}{(\al-\beta)^2}d\beta\|_{H_{\al}^{1}[0,\frac{\pi}{4}]}\lesssim_{g}O(\epsilon).
\end{equation}
\end{lemma}
\begin{proof}
We have
\begin{equation}
\begin{split}
&\quad \frac{2}{\pi}\int_{0}^{2\al}\frac{(g(\al)-g(\beta))(\al-\beta)\ep}{((\al-\beta)^2+\ep^2)^2}d\beta\\
&=\frac{2}{\pi}\int_{0}^{2\al}\frac{(g(\al)-g(\beta)-g'(\al)(\al-\beta))(\al-\beta)\ep}{((\al-\beta)^2+\ep^2)^2}d\beta\\
&\quad+\frac{2}{\pi}g'(\al)\int_{0}^{2\al}\frac{(\al-\beta)^2\ep}{((\al-\beta)^2+\ep^2)^2}d\beta,
\end{split}
\end{equation}
and 
\begin{equation}
\begin{split}
    &\quad\frac{d}{d\al}\frac{2}{\pi}\int_{0}^{2\al}\frac{(g(\al)-g(\beta))(\al-\beta)\ep}{((\al-\beta)^2+\ep^2)^2}d\beta\\
    &=\frac{2}{\pi}\frac{d}{d\al}(\int_{-\al}^{\al}\frac{(g(\al)-g(\al-\beta)-g'(\al)\beta)(\beta)\ep}{(\beta^2+\ep^2)^2}d\beta)+\frac{2}{\pi}\frac{d}{d\al}(g'(\al)\int_{-\al}^{\al}\frac{\beta^2\ep}{(\beta^2+\ep^2)^2}d\beta)\\
    &=\frac{2}{\pi}\int_{-\al}^{\al}\frac{(g'(\al)-g'(\al-\beta)-g''(\al)\beta)(\beta)\ep}{(\beta^2+\ep^2)^2}d\beta+\frac{2}{\pi}\frac{g(\al)-g(0)-g'(\al)\al}{(\al^2+\ep^2)^2}\al\ep-\frac{2}{\pi}\frac{g(\al)-g(2\al)+g'(\al)\al}{(\al^2+\ep^2)^2}\al\ep\\
    &\quad-\frac{2}{\pi}g''(\al)\int_{-\al}^{\al}\frac{\beta^2\ep}{(\beta^2+\ep^2)^2}d\beta+\frac{2}{\pi}g'(\al)[\frac{\al^2\ep}{(\al^2+\ep^2)^2}+\frac{\al^2\ep}{(\al^2+\ep^2)^2}].
    \end{split}
\end{equation}
Since $\lim_{\ep\to 0}\frac{2}{\pi}\int_{0}^{2\al}\frac{(\al-\beta)^2\ep}{((\al-\beta)^2+\ep^2)^2}d\beta=\frac{2}{\pi} \int_{-\infty}^{\infty}\frac{\beta^2}{(\beta^2+1)^2}d\beta =1$ almost everywhere, and 
\[
|\int_{0}^{2\al}\frac{(\al-\beta)^2\ep}{((\al-\beta)^2+\ep^2)^2}d\beta|\lesssim 1,
\]
\[
\lim_{\ep\to 0}\frac{\al^2\ep}{(\al^2+\ep^2)^2}\al =0,
\]
\[
\frac{\al^2\ep}{(\al^2+\ep^2)^2}\al\lesssim 1,
\]
from Lebesgues' convergence theorem, we get the \eqref{3b1g0}.
For \eqref{3b2g0}, similarly, we have
\begin{equation}
\int_{0}^{2\al}\frac{(g(\al)-g(\beta))}{((\al-\beta)^2+\ep^2)}d\beta=\int_{-\al}^{\al}\frac{(g(\al)-g(\al-\beta)-g'(\al)\beta)}{(\beta^2+\ep^2)}d\beta,
\end{equation}
and
\begin{align*}
\frac{d}{d\al}\int_{0}^{2\al}\frac{(g(\al)-g(\beta))}{((\al-\beta)^2+\ep^2)}d\beta&=\int_{-\al}^{\al}\frac{(g'(\al)-g'(\al-\beta)-g''(\al)\beta)}{(\beta^2+\ep^2)}d\beta+\frac{g(\al)-g(0)-g'(\al)\al}{\al^2+\ep^2}\\
&\quad+\frac{g(\al)-g(2\al)+g'(\al)\al}{\al^2+\ep^2}.
\end{align*}
The Lebesgue's convergence theorem gives the result. 
\end{proof}

Then from \eqref{3M12split}, \eqref{3M21split}, we have
\begin{align*}
\|\tilde{M}_{1,2,B}^{\ep}(h)\|_{X^{0}}\leq \|\frac{1}{\alpha}L_{1}^{+}(h)(\alpha)\|_{C_{\gamma}^{0}([0,1],C_{\alpha}^{0}[0,\pi])}\sum_{j=0}^{k-1}\|b_{1,j}^{\ep}(\alpha)\alpha\|_{X^{0}}\|h\|_{X^{k}}.
\end{align*}
\begin{align*}
\|\tilde{M}_{2,1,B}^{\ep}(h)\|_{X^{0}}\leq \|\frac{1}{\alpha}L_{2}^{+}(h)(\alpha)\|_{C_{\gamma}^{0}([0,1],C_{\alpha}^{0}[0,\pi])}\sum_{j=0}^{k-1}\|b_{2,j}^{\ep}(\alpha)\alpha\|_{X^{0}}\|h\|_{X^{k}}.
\end{align*}
\begin{align*}
&\quad\|\tilde{M}_{1,2,I}^{\ep}(h)-M_{1,2}(h)\|_{X^{0}}=\|\tilde{M}_{1,2,I}^{\ep}(h)-\lambda(\alpha)h'(\alpha)L_{1}^{+}(h)\|_{X^{0}}\\&\leq \|\frac{1}{\alpha}L_{1}^{+}(h)(\alpha)\|_{C_{\gamma}^{0}([0,1],C_{\alpha}^{0}[0,\pi])}\|\lambda(\alpha)\alpha(\frac{2}{\pi}\int_{0}^{2\alpha}\frac{(h(\alpha)-h(\beta))(\alpha-\beta)\epsilon}{((\alpha-\beta)^2+\epsilon^2)^2}d\beta-h'(\alpha))\|_{X^{0}}\|h\|_{X^{k}}.
\end{align*}
\begin{align*}
&\quad\|\tilde{M}_{2,1,I}^{\ep}(h)-M_{2,1}(h)\|_{X^{0}}=\|\tilde{M}_{2,1,I}^{\ep}(h)-\lambda(\alpha)L_{2}^{+}(h)\int_{0}^{2\alpha}\frac{h(\alpha)-h(\beta)}{(\alpha-\beta)^2+\epsilon^2}d\beta\|_{X^{0}}\\&\leq \|\frac{1}{\alpha}L_{2}^{+}(h)(\alpha)\|_{C_{\gamma}^{0}([0,1],C_{\alpha}^{0}[0,\pi])}\|\lambda(\alpha)\alpha(\int_{0}^{2\alpha}\frac{h(\alpha)-h(\beta)}{(\alpha-\beta)^2+\epsilon^2}d\beta-p.v.\int_{0}^{2\alpha}\frac{h(\alpha)-h(\beta)}{(\alpha-\beta)^2}d\beta)\|_{X^{0}}\|h\|_{X^{k}}.
\end{align*}
From vanishing and smoothness conditions of $L_{i}^{+}(h)$, we have 
\[
\|\frac{1}{\alpha}L_{i}^{+}(h)(\alpha)\|_{C_{\gamma}^{0}([-1,1], C_{\alpha}^{0}[0,\pi])} \lesssim \|L_{i}^{+}(h)(\alpha)\|_{C_{\gamma}^{0}([-1,1], C_{\alpha}^{1}[0,\pi])}.
\]
By taking the $g$ as $h$ and $\frac{\alpha^{j}}{j!}$, we can control all four terms above from lemma \ref{limitlemme0}, and have the result.
\end{proof}

\subsubsection{Control for the good terms}
We first show some lemmas to separate the main terms. Roughly speaking, we want to show the only unbounded terms of $\partial_{\al}^{k}T^{\ep}(h)$ in $L^{2}_{\al}$ are those terms in $\tilde{M}_{1,1}^{k,\ep}+\tilde{M}_{1,2}^{k,\ep}+\tilde{M}_{2,1}^{k,\ep}$ with all derivatives hitting on $h$.  
\begin{lemma}\label{coro22k0100}
For $h\in H_{\al}^{1}[0,\pi]$, with $h(0)=0$, $\supp h \subset [0,\frac{3}{4}\pi]$, we have 
\begin{equation}\label{3coro03k011}
\|\int_{0}^{2\al}\frac{h(\al)-h(\beta)}{(\al-\beta)^2+\ep^2}d\beta\|_{L_{\al}^{2}[0,\frac{\pi}{4}]}\lesssim \|h\|_{H_{\al}^{1}[0,\pi]},
\end{equation}
and
\begin{equation}\label{3corom12k01}
\|\int_{0}^{2\al}\frac{(h(\al)-h(\beta))(\al-\beta)\ep}{((\al-\beta)^2+\ep^2)^2}d\beta\|_{L_{\al}^{2}[0,\frac{\pi}{4}]}\lesssim \|h\|_{H_{\al}^{1}[0,\pi]},
\end{equation}
\end{lemma}
\begin{proof}
For \eqref{3coro03k011}, we only need to show 
\begin{equation}\label{3coro0301}
\|\int_{2\al}^{\pi}\frac{h(\al)-h(\beta)}{(\al-\beta)^2+\ep^2}d\beta\|_{L_{\al}^{2}[0,\frac{\pi}{4}]}\lesssim \|h\|_{H_{\al}^{1}[0,\pi]},
\end{equation}
and
\begin{equation}\label{3coro0302}
\|\int_{0}^{\pi}\frac{h(\al)-h(\beta)}{(\al-\beta)^2+\ep^2}d\beta\|_{L_{\al}^{2}[0,\frac{\pi}{4}]}\lesssim \|h\|_{H_{\al}^{1}[0,\pi]}.
\end{equation}
For \eqref{3coro0301}, we have
 \begin{align}\label{newseparation16}
&\int_{2\al}^{\pi}\frac{h(\alpha)-h(\beta)}{(\alpha-\beta)^2+\epsilon^2}d\beta\\\nonumber
&=(\al)\int_{2\al}^{\pi}\frac{\frac{h(\alpha)}{\al}}{(\alpha-\beta)^2+\epsilon^2}d\beta-\int_{2\al}^{\pi}\frac{\beta\frac{h(\beta)}{\beta}}{(\alpha-\beta)^2+\epsilon^2}d\beta\\\nonumber
&=Term_1+Term_2.
\end{align}
Since 
\[
\al\int_{2\al}^{\pi}\frac{1}{(\alpha-\beta)^2+\epsilon^2}d\beta
\lesssim\int_{2\al}^{\pi}\frac{\al}{(\alpha-\beta)^2}d\beta \lesssim 1, \]
and
\begin{equation}\label{0l1bound}
\|\frac{h(\al)}{\al}\|_{L^2}=\|\frac{h(\al)-h(0)}{\al}\|_{L^2}\lesssim \|h\|_{H_{\al}^{1}[0,\pi]}.
\end{equation}
$Term_1$ can be bounded.
We also have
\begin{align*}
    |Term_2|\leq \int_{2\al}^{\pi}\frac{\frac{|h(\beta)|}{\beta}}{|\alpha-\beta|}d\beta\lesssim \int_{2\al}^{\pi}\frac{\frac{|h(\beta)|}{\beta}}{|\alpha+\beta|}d\beta.
\end{align*}
Then it can be bounded by Hilbert's inequality and we get
\begin{align}
 \|Term_2\|_{L^2_{\al}[0,\frac{\pi}{4}]}\lesssim\|\frac{h(\beta)}{\beta}\|_{L_{\al}^2[0,\pi]}\lesssim \|h\|_{H_{\al}^{1}[0,\pi]},
\end{align}
where we use \eqref{0l1bound} in the last step.
Hence we have \eqref{3coro0301}.

For \eqref{3coro0302}, we have
the following estimate:
\begin{align}\label{sinsingular}
\|\int_{0}^{\pi}\frac{h(\alpha)-h(\beta)}{(\alpha-\beta)^2+\epsilon^2}d\beta-\int_{0}^{\pi}\frac{h(\alpha)-h(\beta)}{\sin^2(\alpha-\beta)+\epsilon^2}d\beta\|_{L^2_{\al}[0,\frac{\pi}{4}]}\lesssim\|h\|_{H^{1}_{\al}[0,\pi]}.
\end{align}
In fact, since $\supp h\subset[0,\frac{3}{4}\pi]$, we have
\begin{align*}
    \|\int_{\frac{7\pi}{8}}^{\pi}\frac{h(\alpha)-h(\beta)}{(\alpha-\beta)^2+\epsilon^2}d\beta\|_{L^2_{\al}[0,\frac{\pi}{4}]}\lesssim\|h(\alpha)\|_{L_{\al}^2[0,\pi]}.
\end{align*}
Moreover,
\begin{align*}
    &\quad\|\int_{0}^{\frac{7\pi}{8}}\frac{h(\alpha)-h(\beta)}{(\alpha-\beta)^2+\epsilon^2}d\beta-\int_{0}^{\frac{7\pi}{8}}\frac{h(\alpha)-h(\beta)}{\sin^2(\alpha-\beta)+\epsilon^2}d\beta\|_{L^2_{\al}[0,\frac{\pi}{4}]}\\
    &\lesssim \|\int_{0}^{\frac{7\pi}{8}}|h(\alpha)-h(\beta)||\frac{\sin(\alpha-\beta)^2-(\alpha-\beta)^2}{(\sin^2(\alpha-\beta)+\epsilon^2)((\alpha-\beta)^2+\epsilon^2)}|d\beta\|_{L^2_{\al}[0,\frac{\pi}{4}]}\\
    &\lesssim\|h(\alpha)\|_{L_{\al}^2[0,\pi]}.
\end{align*}
We also have
\begin{align*}
    &\quad\|\int_{\frac{7\pi}{8}}^{\pi}\frac{h(\alpha)-h(\beta)}{\sin^2(\alpha-\beta)+\epsilon^2}d\beta\|_{L^2_{\al}[0,\frac{\pi}{4}]}\\
    &=\|\int_{\frac{7\pi}{8}}^{\pi}\frac{h(\alpha)}{\sin^2(\alpha-\beta)+\epsilon^2}d\beta\|_{L^2_{\al}[0,\frac{\pi}{4}]}\\
    &\lesssim\|h(\alpha)\frac{1}{(\al)}\|_{L^2_{\al}[0,\frac{\pi}{4}]}+\|h(\alpha)\|_{L^2_{\al}[0,\pi]}\\
    &\lesssim\|h(\alpha)\|_{H^1_{\al}[0,\pi]}.
\end{align*}
Here we use \eqref{0l1bound} in the last inequality. Therefore \eqref{sinsingular} is true.

Moreover, if $h$ $\in$ $H^{2}[\mathcal{T}]$, where $\mathcal{T}$ is the torus of length $\pi$, for any $f\in C^\infty(\mathcal{T})$, we have
\begin{align*}
    &\int_{0}^{\pi}f(\alpha)\int_{0}^{\pi}\frac{h(\alpha)-h(\beta)}{\sin^2(\alpha-\beta)+\epsilon^2}d\beta d\alpha\\
    &=\int_{0}^{\pi}\Lambda_{\mathcal{T}}^{-\frac{1}{2}}(f-\hat{f}(0))(\alpha)\int_{0}^{\pi}\frac{\Lambda_{\mathcal{T}}^{\frac{1}{2}}h(\alpha)-\Lambda_{\mathcal{T}}^{\frac{1}{2}}h(\beta)}{\sin^2(\alpha-\beta)+\epsilon^2}d\beta d\alpha\\
   & \lesssim(\int_{0}^{\pi}\int_{0}^{\pi}\frac{(\Lambda_{\mathcal{T}}^{-\frac{1}{2}}(f(\alpha)-\hat{f}(0))-\Lambda_{\mathcal{T}}^{-\frac{1}{2}}(f(\beta)-\hat{f}(0)))^2}{\sin^2(\alpha-\beta)}d\beta d\alpha)^{\frac{1}{2}}\\
    &\quad\cdot(\int_{0}^{\pi}\int_{0}^{\pi}\frac{(\Lambda_{\mathcal{T}}^{\frac{1}{2}}h(\alpha)-\Lambda_{\mathcal{T}}^{\frac{1}{2}}h(\beta))^2}{\sin^2(\alpha-\beta)}d\beta d\alpha)^{\frac{1}{2}}\\
    &\lesssim (\int_{0}^{\pi}\Lambda_{\mathcal{T}}^{-\frac{1}{2}}(f(\alpha)-\hat{f}(0))\Lambda_{\mathcal{T}}^{\frac{1}{2}}(f(\alpha)-\hat{f}(0))d\alpha)^{\frac{1}{2}}(\int_{0}^{\pi}\Lambda_{\mathcal{T}}^{\frac{1}{2}}(h(\alpha))\Lambda_{\mathcal{T}}^{\frac{3}{2}}(h(\alpha))d\alpha)^{\frac{1}{2}}\\
    &\lesssim \|f\|_{L^2}\|\Lambda_{\mathcal{T}} h(\alpha)\|_{L^2}\lesssim \|f\|_{L^2}\|h(\alpha)\|_{H^1}.
\end{align*}
Since  $h(0)=0$, $h=0$ when $\alpha>\frac{3\pi}{4}$ and $h\in H^1[0,\pi] $, we have $h\in H^{1}[\mathcal{T}]$. Then we could choose a sequence in $H^{2}$ to approximate it and get the result.
Therefore
\begin{equation}\label{H1}
\|\int_{0}^{\pi}\frac{h(\alpha)-h(\beta)}{\sin^2(\alpha-\beta)+\epsilon^2}d\beta\|_{L^2_{\al}[0,\frac{\pi}{4}]}\lesssim\|h\|_{H^{1}_{\al}[0,\pi]}.
\end{equation}
Then from \eqref{sinsingular}, \eqref{H1}, we have \eqref{3coro0302}.

Now we deduce \eqref{3corom12k01}. Similarly, we only need to show 
\begin{equation}\label{3coro0301m12}
\|\int_{2\al}^{\pi}\frac{(h(\al)-h(\beta))(\al-\beta)\ep}{((\al-\beta)^2+\ep^2)^2}d\beta\|_{L_{\al}^{2}[0,\frac{\pi}{4}]}\lesssim \|h\|_{H_{\al}^{1}[0,\pi]},
\end{equation}
and
\begin{equation}\label{3coro0302m12}
\|\int_{0}^{\pi}\frac{(h(\al)-h(\beta))(\al-\beta)\ep}{((\al-\beta)^2+\ep^2)^2}d\beta\|_{L_{\al}^{2}[0,\frac{\pi}{4}]}\lesssim \|h\|_{H_{\al}^{1}[0,\pi]}.
\end{equation}
 For \eqref{3coro0301m12}, we have 
\begin{align*}
&\int_{2\al}^{\pi}\frac{(h(\alpha)-h(\beta))(\alpha-\beta)\epsilon}{((\alpha-\beta)^2+\epsilon^2)^2}d\beta\\
&=(\al)\int_{2\al}^{\pi}\frac{(\frac{h(\alpha)}{\al})(\alpha-\beta)\epsilon}{((\alpha-\beta)^2+\epsilon^2)^2}d\beta-\int_{2\al}^{\pi}\frac{(\alpha-\beta)\beta\epsilon\frac{h(\beta)}{\beta}}{((\alpha-\beta)^2+\epsilon^2)^2}d\beta\\
&=Term_{1}+Term_{2}.
\end{align*}
Since
\[|\frac{(\alpha-\beta)\epsilon}{((\alpha-\beta)^2+\epsilon^2)^2}|\lesssim \frac{1}{(\alpha-\beta)^2+\epsilon^2},
\]
it can be bounded in the same way as in the corresponding terms in \eqref{newseparation16}.

For \eqref{3coro0302m12}, we get
\begin{align*}
    &\quad \|\int_{0}^{\pi}\frac{(h(\alpha)-h(\beta))(\alpha-\beta)\epsilon}{((\alpha-\beta)^2+\epsilon^2)^2}d\beta\|_{L^2_{\alpha}[0,\frac{\pi}{4}]}\\
    &=\|\int_{\alpha}^{-\pi+\alpha}\frac{\int_{0}^{1}h'(\alpha-\zeta(\beta))d\zeta\beta^2\epsilon}{(\beta^2+\epsilon^2)^2}d\beta\|_{L^2_{\alpha}[0,\frac{\pi}{4}]}\\
    &\lesssim \int_{-\pi}^{\pi}\frac{\|\int_{0}^{1}h'(\alpha-\zeta(\beta))d\zeta\|_{L^2_{\alpha}[0,\frac{\pi}{4}]}\beta^2\epsilon}{(\beta^2+\epsilon^2)^2}d\beta\\
    &\lesssim\|h\|_{H^1[0,\pi]}\int_{-\pi}^{\pi}\frac{\beta^2\epsilon}{(\beta^2+\epsilon^2)^2}d\beta\\
     &\lesssim\|h\|_{H^1[0,\pi]}.
\end{align*}
\end{proof}
In next corollary, we show that $h(0)=0$, $\supp h \subset [0,\frac{3}{4}\pi]$ conditions are not needed.
\begin{corollary}\label{coro22k01}
For $h\in H_{\al}^{1}[0,\pi]$, we have 
\begin{equation}\label{3coro03k01102}
\|\int_{0}^{2\al}\frac{h(\al)-h(\beta)}{(\al-\beta)^2+\ep^2}d\beta\|_{L_{\al}^{2}[0,\frac{\pi}{4}]}\lesssim \|h\|_{H_{\al}^{1}[0,\pi]},
\end{equation}
and
\begin{equation}\label{3corom12k0102}
\|\int_{0}^{2\al}\frac{(h(\al)-h(\beta))(\al-\beta)\ep}{((\al-\beta)^2+\ep^2)^2}d\beta\|_{L_{\al}^{2}[0,\frac{\pi}{4}]}\lesssim \|h\|_{H_{\al}^{1}[0,\pi]},
\end{equation}
\end{corollary}
\begin{proof}
Let $\tilde{\lambda}(\al)$ be a sufficiently smooth cut off function with $\supp \tilde{\lambda}\subset [0,\frac{3\pi}{4}]$, and $ \tilde{\lambda}(\al)=1$ when $\al \in [0,\frac{\pi}{2}]$. Then if we let
\[
\tilde{h}(\al)=\tilde{\lambda}(\al)(h(\al)-h(0)),
\]
when $\al \in [0,\frac{\pi}{4}]$, we have
\begin{equation}
\begin{split}
&\int_{0}^{2\al}\frac{h(\al)-h(\beta)}{(\al-\beta)^2+\ep^2}d\beta\\
&=\int_{0}^{2\al}\frac{(h(\al)-h(0))-(h(\beta)-h(0))}{(\al-\beta)^2+\ep^2}d\beta\\
&=\int_{0}^{2\al}\frac{\tilde{h}(\al)-\tilde{h}(\beta)}{(\al-\beta)^2+\ep^2}d\beta.
\end{split}
\end{equation}
Similarly,  when $\al \in [0,\frac{\pi}{4}]$, 
\begin{equation}
\begin{split}
&\int_{0}^{2\al}\frac{(h(\al)-h(\beta))(\al-\beta)\ep}{((\al-\beta)^2+\ep^2)^2}d\beta\\
&=\int_{0}^{2\al}\frac{(\tilde{h}(\al)-\tilde{h}(\beta))(\al-\beta)\ep}{((\al-\beta)^2+\ep^2)^2}d\beta,
\end{split}
\end{equation}
where we also use $\tilde{\lambda}(\beta)=1$ for $\beta \in [0,2\alpha]\subset[0,\frac{\pi}{2}].$
Moreover, $\tilde{h}(\al)$ satisfies
\[
\|\tilde{h}(\al)\|_{H_{\al}^{1}[0,\pi]}\lesssim \|h\|_{H_{\al}^{1}[0,\pi]}.
\]
Then we could use the previous lemma \ref{coro22k0100} for $\tilde{h}$ and have the result.
\end{proof}
For the higher-order derivatives, we first introduce a lemma:

\begin{lemma}\label{farboundarybound}
If $h(\alpha)\in H^{k}[0,\pi]$ with $\partial_\alpha^{l'} h(0)=0$ for $0\leq l'\leq k-1$, then for $j\leq k$ we have
\[
\|\frac{\partial_{\alpha}^{j}h(\alpha)}{(\alpha)^{k-j}}\|_{L_{\al}^2[0,\pi]}\lesssim \|h(\alpha)\|_{H^{k}[0,\pi]}.
\]
\end{lemma}
\begin{proof}
When $j=k$, it is trivial. When $j\leq k-1$,
we have
\begin{align*}
    &\quad\|\frac{\partial_{\alpha}^{j}h(\alpha)}{(\alpha)^{k-j}}\|_{L^2}= \|\frac{\int_{0}^{1}\partial_{\alpha}^{j+1}h(\alpha-\zeta(\alpha))d\zeta}{(\alpha)^{k-j-1}}\|_{L^2}\\
    &\leq \int_{0}^{1}\|\frac{\partial_{\alpha}^{j+1}h(\tilde{\alpha})}{(\tilde{\alpha})^{k-j-1}(\frac{1}{1-\zeta})^{k-j-1}}\|_{L^2}\frac{1}{\sqrt{1-\zeta}}d\zeta\\
    &\leq \int_{0}^{1}\|\frac{\partial_{\alpha}^{j+1}h(\tilde{\alpha})}{(\tilde{\alpha})^{k-j-1}}\|_{L^2}(1-\zeta)^{k-j-1}\frac{1}{\sqrt{1-\zeta}}d\zeta\\
    &\lesssim\|\frac{\partial_{\alpha}^{j+1}h(\tilde{\alpha})}{(\tilde{\alpha})^{k-j-1}}\|_{L^2}.
\end{align*}
We could repeat this process and get the result.
\end{proof}
Then we separate the main term for the $\tilde{M}_{2,1}^{k,\epsilon}$  term and control the good terms:
\begin{lemma}\label{coro22ep}
For any $k\geq 1$, if $h\in H_{\al}^{k}[0,\pi]$, we have 
\begin{equation}\label{3coro22ep02}
\begin{split}
&\quad \|\al\partial_{\al}^{k}\int_{0}^{2\al}\frac{(h(\al)-\sum_{j=0}^{k-1}\frac{\al^{j}}{j!}h^{(j)}(0))-(h(\beta)-\sum_{j=0}^{k-1}\frac{\beta^{j}}{j!}h^{(j)}(0))}{(\al-\beta)^2+\ep^2}d\beta-\al\int_{0}^{2\al}\frac{h^{(k)}(\al)-h^{(k)}(\beta)}{(\al-\beta)^2+\ep^2}d\beta\|_{L_{\al}^{2}[0,\frac{\pi}{4}]}\\
&\lesssim \|h\|_{H_{\al}^{k}[0,\pi]}.
\end{split}
\end{equation}
Moreover, for any $l\leq k-1$, we have
\begin{equation}\label{3coro22ep0202}
\begin{split}
&\quad\|\partial_{\al}^{l}\int_{0}^{2\al}\frac{(h(\al)-\sum_{j=0}^{k-1}\frac{\al^{j}}{j!}h^{(j)}(0))-(h(\beta)-\sum_{j=0}^{k-1}\frac{\beta^{j}}{j!}h^{(j)}(0))}{(\al-\beta)^2+\ep^2}d\beta\\
&\qquad-\int_{0}^{2\al}\frac{\partial_{\al}^{l}(h(\al)-\sum_{j=0}^{k-1}\frac{\al^{j}}{j!}h^{(j)}(0))-\partial_{\beta}^{l}(h(\beta)-\sum_{j=0}^{k-1}\frac{\beta^{j}}{j!}h^{(j)}(0))}{(\al-\beta)^2+\ep^2}d\beta\|_{L_{\al}^{2}[0,\frac{\pi}{4}]}\\
&\lesssim \|h\|_{H_{\al}^{k}[0,\pi]},
\end{split}
\end{equation}
and
\begin{equation}\label{3coro22ep01}
\|\partial_{\al}^{l}\int_{0}^{2\al}\frac{(h(\al)-\sum_{j=0}^{k-1}\frac{\al^{j}}{j!}h^{(j)}(0))-(h(\beta)-\sum_{j=0}^{k-1}\frac{\beta^{j}}{j!}h^{(j)}(0))}{(\al-\beta)^2+\ep^2}d\beta\|_{L_{\al}^{2}[0,\frac{\pi}{4}]}\lesssim \|h\|_{H_{\al}^{k}[0,\pi]}.
\end{equation}
\end{lemma}
\begin{proof}
We claim it it enough to show the case when $h$ satisfying $\partial_{\al}^{j}h(0)=0$ for $1\leq j\leq k-1$.

In fact, we can do the estimate for $\tilde{h}=h(\al)-\sum_{j=0}^{k-1}\frac{\al^{j}}{j!}h^{(j)}(0)$  and we have

\[
\partial_{\al}^{k}\tilde{h}(\al)=\partial_{\al}^{k}h(\al),
\]
and
\[
\|\tilde{h}(\al)\|_{H_{\al}^{k}[0,\pi]}\lesssim \|h\|_{H_{\al}^{k}[0,\pi]},
\]
from the Sobolev embedding theorem. 

In this case, for \eqref{3coro22ep02}, we have
\begin{align*}
    &\quad \partial_{\al}^{q}\int_{0}^{2\al}\frac{(h(\al)-\sum_{j=0}^{k-1}\frac{\al^{j}}{j!}h^{(j)}(0))-(h(\beta)-\sum_{j=0}^{k-1}\frac{\beta^{j}}{j!}h^{(j)}(0))}{(\al-\beta)^2+\ep^2}d\beta\\ &=\partial_{\alpha}^{q}\int_{0}^{2\alpha}\frac{h(\alpha)-h(\beta)}{(\alpha-\beta)^2+\epsilon^2}d\beta\\
    &=\partial_{\alpha}^{q}\int_{-(\alpha)}^{\alpha}\frac{h(\alpha)-h(\alpha-\beta)}{(\beta)^2+\epsilon^2}d\beta\\
    &=\int_{-(\alpha)}^{\alpha}\frac{\partial_{\alpha}^{q}h(\alpha)-\partial_{\alpha}^{q}h(\alpha-\beta)}{(\beta)^2+\epsilon^2}d\beta+\sum_{l+l'=q-1}C^{l,l'}\partial_{\alpha}^{l-l'}(\frac{\partial_{\alpha}^{l'}h(\alpha)-\partial_{\alpha}^{l'}h(0)}{(\alpha)^2+\epsilon^2})\\
    &+\sum_{l+l'=q-1}C^{l,l'}\partial_{\alpha}^{l-l'}(\frac{\partial_{\alpha}^{l'}h(\alpha)-\partial_{\alpha}^{l'}h(2\alpha)}{(\alpha)^2+\epsilon^2})\\
    &=\int_{-(\alpha)}^{\alpha}\frac{\partial_{\alpha}^{q}h(\alpha)-\partial_{\alpha}^{q}h(\alpha-\beta)}{(\beta)^2+\epsilon^2}d\beta+ Term\ 2 + Term \ 3.
\end{align*}
We have 
\[
|Term\ 2|+|Term\ 3|\lesssim \sum_{l\leq q-1, l'\leq l}\frac{|\partial_{\alpha}^{l'}h(2\alpha)|}{(\alpha)^{l-l'+2}}+\sum_{l \leq q-1,l'\leq l}\frac{|\partial_{\alpha}^{l'}h(\alpha)|}{(\alpha)^{l-l'+2}},
\]
and 
\[
(\alpha)(|Term\ 2|+|Term\ 3|)\lesssim \sum_{l\leq q-1, l'\leq l}\frac{|\partial_{\alpha}^{l'}h(2\alpha)|}{(\alpha)^{l-l'+1}}+\sum_{l\leq q-1, l'\leq l}\frac{|\partial_{\alpha}^{l'}h(\alpha)|}{(\alpha)^{l-l'+1}}.
\]
They can be bounded by lemma \ref{farboundarybound}:
\[
\|(\alpha)(|Term\ 2|+|Term\ 3|)\|_{L^{2}_{\alpha}[0,\frac{\pi}{4}]}\lesssim\|h\|_{H^{k}_{\alpha}[0,\pi]}.
\]Then we get \eqref{3coro22ep02} when taking $q=k$.

The second inequality \eqref{3coro22ep0202} can be done in the same way by taking $q=l$.
The last inequality \eqref{3coro22ep01} follows from \eqref{3coro22ep0202} and \eqref{3coro03k01102} in lemma \ref{coro22k01}.

\end{proof}
Similarly, for $M_{1,2}$ type term, we have
\begin{lemma}\label{coro22epm12}
For any $k\geq 1$, if $h\in H_{\al}^{k}[0,\pi]$, we have 
\begin{equation}\label{3coro22ep020102}
\begin{split}
&\quad\|\al\partial_{\al}^{k}\int_{0}^{2\al}\frac{((h(\al)-\sum_{j=0}^{k-1}\frac{\al^{j}}{j!}h^{(j)}(0))-(h(\beta)-\sum_{j=0}^{k-1}\frac{\beta^{j}}{j!}h^{(j)}(0)))(\al-\beta)\ep}{((\al-\beta)^2+\ep^2)^2}d\beta\\
&\qquad-\al\int_{0}^{2\al}\frac{(h^{(k)}(\al)-h^{(k)}(\beta))(\al-\beta)\ep}{((\al-\beta)^2+\ep^2)^2}d\beta\|_{L_{\al}^{2}[0,\frac{\pi}{4}]}\\
&\lesssim \|h\|_{H_{\al}^{k}[0,\pi]}.
\end{split}
\end{equation}
Moreover, for any $l\leq k-1$, we have
\begin{equation}\label{3coro22ep020202}
\begin{split}
&\quad\|\partial_{\al}^{l}\int_{0}^{2\al}\frac{((h(\al)-\sum_{j=0}^{k-1}\frac{\al^{j}}{j!}h^{(j)}(0))-(h(\beta)-\sum_{j=0}^{k-1}\frac{\beta^{j}}{j!}h^{(j)}(0)))(\al-\beta)\ep}{((\al-\beta)^2+\ep^2)^2}d\beta\\
&\qquad-\int_{0}^{2\al}\frac{(\partial_{\al}^{l}(h(\al)-\sum_{j=0}^{k-1}\frac{\al^{j}}{j!}h^{(j)}(0))-\partial_{\beta}^{l}(h(\beta)-\sum_{j=0}^{k-1}\frac{\beta^{j}}{j!}h^{(j)}(0)))(\al-\beta)\ep}{((\al-\beta)^2+\ep^2)^2}d\beta\|_{L_{\al}^{2}[0,\frac{\pi}{4}]}\\
&\lesssim \|h\|_{H_{\al}^{k}[0,\pi]},
\end{split}
\end{equation}
and
\begin{equation}\label{3coro22ep0102020}
\|\partial_{\al}^{l}\int_{0}^{2\al}\frac{((h(\al)-\sum_{j=0}^{k-1}\frac{\al^{j}}{j!}h^{(j)}(0))-(h(\beta)-\sum_{j=0}^{k-1}\frac{\beta^{j}}{j!}h^{(j)}(0)))(\alpha-\beta)\epsilon}{((\al-\beta)^2+\ep^2)^2}d\beta\|_{L_{\al}^{2}[0,\frac{\pi}{4}]}\lesssim \|h\|_{H_{\al}^{k}[0,\pi]},
\end{equation}
\begin{proof}
The proof is similar as in lemma \ref{coro22ep} by using \eqref{3corom12k0102} instead of \eqref{3coro03k01102} in the last step.
\end{proof}
\end{lemma}

\begin{lemma}\label{3lemmaotherterms}
For $1\leq k\leq 12$, $h\in X^{k}$, we have 
\begin{equation}\label{3boundedmain}
\begin{split}
\|&\partial_{\alpha}^{k}(\tilde{M}_{1,1}^{\epsilon}+\tilde{M}_{1,2}^{k,\epsilon}+\tilde{M}_{2,1}^{k,\epsilon}(h))-\lambda(\alpha)L_2^{+}(h)\int_{0}^{2\alpha}\frac{h^{(k)}(\alpha)-h^{(k)}(\beta)}{(\alpha-\beta)^2+\epsilon^2}d\beta-\lambda(\alpha)\kappa(t)\frac{h^{(k)}(\alpha+\epsilon)-h^{(k)}(\al)}{\epsilon}-\\
&\lambda(\alpha)\int_{0}^{2\alpha}\frac{(h^{(k)}(\alpha)-h^{(k)}(\beta))(\alpha-\beta)\epsilon}{((\alpha-\beta)^2+\epsilon^2)^2}d\beta \frac{2}{\pi}L_1^{+}(h)\|_{X^{0}}\lesssim C(\|h\|_{X^{k}}),
\end{split}
\end{equation}
\begin{equation}\label{3boundedB}
\|B_i(h)\|_{X^{k}}\lesssim C(\|h\|_{X^{k}}).
\end{equation}
For $0\leq j\leq k-1$, we have
\begin{equation}\label{3boundedlower}
\begin{split}
\|\tilde{M}_{1,1}^{\epsilon}+\tilde{M}_{1,2}^{k,\epsilon}+\tilde{M}_{2,1}^{k,\epsilon}\|_{X^{j}}\lesssim C(\|h\|_{X^{k}})\|h\|_{X^{k}},
\end{split}
\end{equation}
\end{lemma}
\begin{proof}
 \eqref{3boundedmain} follows from the equations \eqref{3M11perturb2}, \eqref{3M12perturb}, \eqref{3M22perturb}, the conditions of $L_i^{+}(h)$ in \eqref{3GE01} and the estimates \eqref{3coro22ep02}, \eqref{3coro22ep01} in lemma \ref{coro22ep}, estimates \eqref{3coro22ep020102} \eqref{3coro22ep0102020} in lemma \ref{coro22epm12}.
  
  \eqref{3boundedB} follows from the condition of $B_i$ in \eqref{3GE01} .
  \eqref{3boundedlower} follows from \eqref{3M11perturb2}, \eqref{3M12perturb}, \eqref{3M22perturb}, the conditions of $L_i^{+}(h)$ in \eqref{3GE01}  and the estimate \eqref{3coro22ep01} in lemma \ref{coro22ep} and estimate \eqref{3coro22ep0102020} in lemma \ref{coro22epm12}.
\end{proof}
We also introduce two similar corollaries to be used later. One is for  the operator replacing the $h$ by $g$ in $\tilde{M}_{1,2}^{k,\epsilon}$, $\tilde{M}_{2,1}^{k,\epsilon}(h)$, another is for the limit.
\begin{corollary}\label{3lemmaotherterms02}
For $1\leq k\leq 12$, $h\in X^{k}$, $g\in H_{\alpha}^{k}[0,\pi]$,  we have 
\begin{equation}\label{3boundedmain01g}
\begin{split}
\|&\partial_{\alpha}^{k}(\lambda(\alpha)\frac{2}{\pi}L_1^{+}(h)(\alpha,\gamma)\int_{0}^{2\alpha}\frac{(g(\alpha)-g(\beta))(\alpha-\beta)\epsilon}{((\alpha-\beta)^2+\epsilon^2)^2}d\beta+\sum_{j=0}^{k-1}b_{1,j}^{\epsilon}(\alpha)g^{<j>}(0))\\
&-\lambda(\al)\frac{2}{\pi}L_1^{+}(h)(\alpha,\gamma)\int_{0}^{2\al}\frac{(g^{<k>}(\al)-g^{<k>}(\beta))(\al-\beta)\ep}{((\al-\beta)^2+\ep^2)^2}d\beta\|_{X^{0}}\\
\lesssim &C(\|h\|_{X^{k}})\|g\|^{2}_{H_{\alpha}^{k}[0,\pi]},
\end{split}
\end{equation}
\begin{equation}\label{3M22perturbg}
\begin{split}
\|&\partial_{\alpha}^{k}(\lambda(\alpha)L_2^{+}(h)\int_{0}^{2\alpha}\frac{(g(\alpha)-g(\beta))}{(\alpha-\beta)^2+\epsilon^2}d\beta+\sum_{j=0}^{k-1}b_{2,j}^{\epsilon}(\alpha)g^{<j>}(0))\\
&-\lambda(\al)L_2^{+}(h)\int_{0}^{2\al}\frac{g^{<k>}(\al)-g^{<k>}(\beta)}{(\al-\beta)^2+\ep^2}d\beta\|_{X^{0}}\\
\lesssim & \|g\|^{2}_{H_{\alpha}^{k}[0,\pi]}C(\|h\|_{X^{k}}).
\end{split}
\end{equation} 
When $j\leq k-1,$ we have
\begin{equation*}
\begin{split}
\|&\partial_{\alpha}^{j}(\lambda(\alpha)\frac{2}{\pi}L_1^{+}(h)(\alpha,\gamma)\int_{0}^{2\alpha}\frac{(g(\alpha)-g(\beta))(\alpha-\beta)\epsilon}{((\alpha-\beta)^2+\epsilon^2)^2}d\beta+\sum_{j=0}^{k-1}b_{1,j}^{\epsilon}(\alpha)g^{<j>}(0))\|_{X^{0}}\\
\lesssim &C(\|h\|_{X^{k}})\|g\|^{2}_{H_{\alpha}^{k}[0,\pi]},
\end{split}
\end{equation*}
and
\begin{equation*}
\|\partial_{\alpha}^{j}(\lambda(\alpha)L_2^{+}(h)\int_{0}^{2\alpha}\frac{(g(\alpha)-g(\beta))}{(\alpha-\beta)^2+\epsilon^2}d\beta+\sum_{j=0}^{k-1}b_{2,j}^{\epsilon}(\alpha)g^{<j>}(0))\|_{X^{0}}
\lesssim  \|g\|^{2}_{H_{\alpha}^{k}[0,\pi]}C(\|h\|_{X^{k}}).
\end{equation*} 
\end{corollary}
\begin{proof}
The proof is similar as in lemma \ref{3lemmaotherterms}. Here we write the formula with expression of $\tilde{M}_{1,2}^{k,\epsilon}$,  $\tilde{M}_{2,1}^{k,\epsilon}$ in \eqref{3M12split}, \eqref{3M21split} instead of \eqref{3M12perturb} and \eqref{3M22perturb}.
\end{proof}
\begin{corollary}\label{3coro22}
For any $k\geq 0$, if $h\in H^{k}_{\alpha}[0,\pi],$ then for $0\leq j \leq k-1$, we have 
\begin{equation}\label{3coro22lim1}
\|\partial_{\al}^{j}p.v.\int_{0}^{2\al}\frac{h(\al)-h(\beta)}{(\al-\beta)^2}d\beta\|_{L_{\al}^{2}[0,\frac{\pi}{4}]}\lesssim \|h\|_{H_{\al}^{k}[0,\pi]},
\end{equation}
and for $h\in H^{k+1}_{\alpha}[0,\pi],$
\begin{equation}\label{3coro22lim2}
\|\al\partial_{\al}^{k}p.v.\int_{0}^{2\al}\frac{h(\al)-h(\beta)}{(\al-\beta)^2}d\beta-\al p.v.\int_{0}^{2\al}\frac{h^{(k)}(\al)-h^{(k)}(\beta)}{(\al-\beta)^2}d\beta\|_{L_{\al}^{2}[0,\frac{\pi}{4}]}\lesssim \|h\|_{H_{\al}^{k}[0,\pi]}.
\end{equation}
\end{corollary}
\begin{proof}

We still first show the case where $\partial_{\al}^{j}h(0)=0$ when $1\leq j\leq k-1$. 

For \eqref{3coro22lim1}, if $h(\al)$ is also in $C^{2}[0,\pi]$, we have 
$\int_{0}^{2\al}\frac{h(\al)-h(\beta)}{(\al-\beta)^2+\ep^2}d\beta$ weakly converges to $p.v.\int_{0}^{2\al}\frac{h(\al)-h(\beta)}{(\al-\beta)^2}d\beta$ in $L^2_{\al}[0,\frac{\pi}{4}]$. From the uniform bound \eqref{3coro22ep01} and the weakly compactness of $H^{k-1}_{\al}$, we have
\begin{equation}
\|p.v.\int_{0}^{2\al}\frac{h(\al)-h(\beta)}{(\al-\beta)^2}d\beta\|_{H_{\al}^{k-1}[0,\frac{\pi}{4}]}\lesssim \|h\|_{H_{\al}^{k}[0,\pi]}.
\end{equation}
For $h\in H^{k}$, we could use the approximation of the smooth function to get the result.

For \eqref{3coro22lim2}, if $k=0$, the bound is trivial. When $k\geq 1$, $h\in H^{k+2}$, we could use the weak convergence and the uniform bound \eqref{3coro22ep02} to get the result. For $h\in H^{k+1}$, we also use the approximation of smooth function.

When $h$ does not vanish at $0$ up to $k-1$ orders, for 
$\tilde{h}(\al)=h(\al)-\sum_{j=0}^{k-1}\frac{\al^{j}}{j!}h^{(j)}(0)$, \eqref{3coro22lim1} and \eqref{3coro22lim2} are satisfied. Moreover, we have
\[
p.v.\int_{0}^{2\al}\frac{\al^j-\beta^j}{(\al-\beta)^2}d\beta
\]
is in $C^{\infty}$ and 
\[
|\partial_{\al}^{j}h(0)|\lesssim \|h(\al)\|_{H_{\al}^{k}[0,\pi]},
\]
from Sobolev's embedding theorem. Then \eqref{3coro22lim1} \eqref{3coro22lim2}  also hold in this case.
\end{proof}
\begin{lemma}\label{323}
If $h\in C_{t}^{0}([0,t_0], X^{k})$ for $k\geq 1$, then we have
\[
\|T(h)\|_{C_{t}^{0}([0,t_0], X^{k-1})}\lesssim 1.
\]
\end{lemma}
\begin{proof}
The result follows from \eqref{3GE01}, the conditions of $L_i^{+}(h)$, $B_i(h)$ and lemma \ref{3coro22}.
\end{proof}
\subsubsection{Refined G\r{a}rding's inequality}
In this section we show an important lemma to control the main term in the energy estimate.
\begin{lemma}\label{Garding}
Let $c(\alpha),\tilde{c}(\alpha) \in C^2[0,\pi]$,$f,g\in L^{2}[0,\pi]$ be real functions, satisfying 
\[
|c(\alpha)|\lesssim \alpha, |\tilde{c}(\alpha)|\lesssim\alpha,\supp{c},\supp{\tilde{c}}, \supp{f}, \supp{g}\subset [0,\frac{\pi}{4}],
\]
there exists $C=9$, such that when $C|c(\alpha)|\leq \tilde{c}(\alpha)$, we have 
\begin{align}\label{ene1}
    &\quad\int_{0}^{\pi}c(\alpha)f(\alpha)\int_{0}^{\pi}\frac{(\alpha-\beta)\epsilon}{((\alpha-\beta)^2+\epsilon^2)^2}(g(\alpha)-g(\beta))d\beta\\\nonumber
    &\leq \int_{0}^{\pi}\tilde{c}(\alpha)f(\alpha)\int_{0}^{\pi}\frac{1}{(\alpha-\beta)^2+\epsilon^2}(f(\alpha)-f(\beta))d\beta d\alpha\\\nonumber
    &\quad+\int_{0}^{\pi}\tilde{c}(\alpha)g(\alpha)\int_{0}^{\pi}\frac{1}{(\alpha-\beta)^2+\epsilon^2}(g(\alpha)-g(\beta))d\beta d\alpha+\tilde{B.T}, 
\end{align}
and
\begin{align}\label{ene2}
    &\quad\int_{0}^{\pi}c(\alpha)f(\alpha)\int_{0}^{\pi}\frac{1}{(\alpha-\beta)^2+\epsilon^2}(g(\alpha)-g(\beta))d\beta d\alpha\\\nonumber
    &\leq \int_{0}^{\pi}\tilde{c}(\alpha)f(\alpha)\int_{0}^{\pi}\frac{1}{(\alpha-\beta)^2+\epsilon^2}(f(\alpha)-f(\beta))d\beta d\alpha\\\nonumber
    &\quad+\int_{0}^{\pi}\tilde{c}(\alpha)g(\alpha)\int_{0}^{\pi}\frac{1}{(\alpha-\beta)^2+\epsilon^2}(g(\alpha)-g(\beta))d\beta d\alpha+\tilde{B.T},
\end{align}
where
\[
\tilde{B.T}\lesssim (\|f\|_{L^2}^2+\|g\|_{L^2}^2)
\]independent of $\epsilon$. 
\end{lemma}
\begin{corollary}\label{3corogardingmain}
With same conditions as in lemma \ref{Garding}, we have
\begin{align}\label{ene102}
    &\int_{0}^{\pi}c(\alpha)f(\alpha)\int_{0}^{2\al}\frac{(\alpha-\beta)\epsilon}{((\alpha-\beta)^2+\epsilon^2)^2}(g(\alpha)-g(\beta))d\beta d\al\\\nonumber
    &\leq \int_{0}^{\pi}\tilde{c}(\alpha)f(\alpha)\int_{0}^{2\al}\frac{1}{(\alpha-\beta)^2+\epsilon^2}(f(\alpha)-f(\beta))d\beta d\alpha\\\nonumber
    &+\int_{0}^{\pi}\tilde{c}(\alpha)g(\alpha)\int_{0}^{2\al}\frac{1}{(\alpha-\beta)^2+\epsilon^2}(g(\alpha)-g(\beta))d\beta d\alpha+\tilde{B.T}, 
\end{align}
and
\begin{align}\label{ene202}
    &\int_{0}^{\pi}c(\alpha)f(\alpha)\int_{0}^{2\al}\frac{1}{(\alpha-\beta)^2+\epsilon^2}(g(\alpha)-g(\beta))d\beta d\alpha\\\nonumber
    &\leq \int_{0}^{\pi}\tilde{c}(\alpha)f(\alpha)\int_{0}^{2\al}\frac{1}{(\alpha-\beta)^2+\epsilon^2}(f(\alpha)-f(\beta))d\beta d\alpha\\\nonumber
    &+\int_{0}^{\pi}\tilde{c}(\alpha)g(\alpha)\int_{0}^{2\al}\frac{1}{(\alpha-\beta)^2+\epsilon^2}(g(\alpha)-g(\beta))d\beta d\alpha+\tilde{B.T},
\end{align}
\end{corollary}
\begin{proof}
It is enough to show
 \begin{align}\label{farbound2eq0}
\|(\al)\int_{2\alpha}^{\pi}\frac{|h(\alpha)-h(\beta)|}{(\alpha-\beta)^2+\epsilon^2}d\beta\|_{L^2_{\al}[0,\frac{\pi}{4}]}\lesssim\|h(\alpha)\|_{L_{\al}^2[0,\pi]}.
\end{align}
and 
\begin{align}\label{farbound2eq1}
    \|(\al)\int_{2\alpha}^{\pi}\frac{|(h(\alpha)-h(\beta))(\alpha-\beta)\epsilon|}{((\alpha-\beta)^2+\epsilon^2)^2}d\beta\|_{L^2_{\al}[0,\frac{\pi}{4}]}\lesssim\|h(\alpha)\|_{L_{\al}^2[0,\pi]}.
\end{align}

For \eqref{farbound2eq0}, we have  
\begin{align*}
&\quad|(\al)\int_{2\alpha}^{\pi}\frac{|h(\alpha)-h(\beta)|}{(\alpha-\beta)^2+\epsilon^2}d\beta|\\
&\leq |(\al)\int_{\alpha-\pi}^{-\al}\frac{1}{(\beta)^2+\epsilon^2}d\beta| |h(\alpha)|\\
&\quad+|(\al)\int_{2\alpha}^{\pi}\frac{|h(\beta)|}{(\alpha-\beta)^2+\epsilon^2}d\beta|\\
&\lesssim |h(\alpha)|+ Term_2(\alpha).
\end{align*}
Moreover, we have
\begin{align}\label{Hilbertinequalityapply}
    &Term_2(\alpha)=|(\alpha)\int_{2\alpha}^{\pi}\frac{|h(\beta)|}{(\alpha-\beta)^2+\epsilon^2}d\beta|\\\nonumber
    &\lesssim\int_{0}^{\pi}\frac{|h(\beta)|}{(\alpha+\beta)}d\beta.
\end{align}
Hence by Hilbert's Inequality, we have
\[
\|Term_2(\alpha)\|_{L^2_{\al}[0,\frac{\pi}{4}]}\lesssim\|h(\alpha)\|_{L_{\al}^2[0,\pi]}.
\]
Then we have the result.

For \eqref{farbound2eq1}, we have
\[|\frac{(\alpha-\beta)\epsilon}{((\alpha-\beta)^2+\epsilon^2)^2}|\lesssim \frac{1}{(\alpha-\beta)^2+\epsilon^2},
\]
and the same proof holds.
\end{proof}

\begin{corollary}\label{Garding limit}
If we also assume $f,g \in H^{1}$, we have

\begin{align*}
    &\quad \int_{0}^{\pi}c(\alpha)f(\alpha)\frac{\pi}{2}g'(\alpha)d\alpha\\
    &\leq \int_{0}^{\pi}\tilde{c}(\alpha)f(\alpha)p.v.\int_{0}^{2\al}\frac{1}{(\alpha-\beta)^2}(f(\alpha)-f(\beta))d\beta d\alpha\\
    &\quad+\int_{0}^{\pi}\tilde{c}(\alpha)g(\alpha)p.v.\int_{0}^{2\al}\frac{1}{(\alpha-\beta)^2}(g(\alpha)-g(\beta))d\beta d\alpha+\tilde{B.T}, 
\end{align*}
and
\begin{align*}
    &\quad\int_{0}^{\pi}c(\alpha)f(\alpha)p.v.\int_{0}^{2\alpha}\frac{1}{(\alpha-\beta)^2}(g(\alpha)-g(\beta))d\beta d\alpha\\
    &\leq \int_{0}^{\pi}\tilde{c}(\alpha)f(\alpha)p.v.\int_{0}^{2\alpha}\frac{1}{(\alpha-\beta)^2}(f(\alpha)-f(\beta))d\beta d\alpha\\
    &\quad+\int_{0}^{\pi}\tilde{c}(\alpha)g(\alpha)p.v.\int_{0}^{2\alpha}\frac{1}{(\alpha-\beta)^2}(g(\alpha)-g(\beta))d\beta d\alpha+\tilde{B.T}. 
\end{align*}
\end{corollary}
\begin{proof}
We can take the limit of the previous lemma and get the result.
\end{proof}
\begin{proof}
Now we start to show the lemma \ref{Garding}. First, we introduce some lemmas.

\begin{lemma}\label{le2}
Let $H_R$ be the Hilbert transform on the real line. We have
\begin{align*}
&\quad \int_{0}^{\pi}c(\alpha)f(\alpha)\int_{0}^{\pi}\frac{(\alpha-\beta)\epsilon}{((\alpha-\beta)^2+\epsilon^2)^2}g(\beta)d\beta d\alpha\\
&=-\frac{1}{2}\int_{0}^{\pi}c(\alpha)f(\alpha)\int_{0}^{\pi}(H_{\mathcal{R}}g(\beta)-H_{\mathcal{R}}g(\alpha))(\frac{1}{((\alpha-\beta)^2+\epsilon^2)}-\frac{2\epsilon^2}{((\alpha-\beta)^2+\epsilon^2)^2})d\beta d\alpha+\tilde{B.T}.
\end{align*}
\end{lemma}
\begin{proof}
Since 
\begin{align}\label{Hilbert inequality}
H_{\mathcal{R}}(\frac{\epsilon}{\beta^2+\epsilon^2})=\frac{\beta}{\beta^2+\epsilon^2},
\end{align}
By changing the order of the Hilbert transform and the derivative, we have
\begin{equation}\label{newhilbertderi}
H_{\mathcal{R}}(\frac{\beta \epsilon}{(\beta^2+\epsilon^2)^2})=\frac{1}{2}\frac{\beta^2-\epsilon^2}{(\beta^2+\epsilon^2)^2}=\frac{1}{2}[\frac{1}{\beta^2+\epsilon^2}-\frac{2\epsilon^2}{(\beta^2+\epsilon^2)^2}].\end{equation}

Now we extend $g(\beta)$ to be a function $g$ on $\mathbb{R}$ with $g(\beta)=0$ on $[0,\pi]^c$. By the property of Hilbert transform, we have 
\begin{align}\label{changeintegrand1}
   &\int_{0}^{\pi}c(\alpha)f(\alpha)\int_{0}^{\pi}\frac{(\alpha-\beta)\epsilon}{((\alpha-\beta)^2+\epsilon^2)^2}g(\beta)d\beta d\alpha\\\nonumber
   =&\int_{0}^{\pi}c(\alpha)f(\alpha)\int_{-\infty}^{\infty}\frac{(\alpha-\beta)\epsilon}{((\alpha-\beta)^2+\epsilon^2)^2}g(\beta)d\beta d\alpha\\\nonumber
   =&-\int_{0}^{\pi}c(\alpha)f(\alpha)\int_{-\infty}^{\infty}\frac{1}{2}[\frac{1}{(\alpha-\beta)^2+\epsilon^2}-\frac{2\epsilon^2}{((\alpha-\beta)^2+\epsilon^2)^2}]H_{\mathcal{R}}g(\beta)d\beta d\alpha
\end{align}
Now we claim
  \begin{align}\label{IE4}
  &\int_{0}^{\pi}c(\alpha)f(\alpha)\int_{-\infty}^{\infty}[\frac{1}{(\alpha-\beta)^2+\epsilon^2}-\frac{2\epsilon^2}{((\alpha-\beta)^2+\epsilon^2)^2}]H_{\mathcal{R}}g(\beta)d\beta d\alpha\\\nonumber
  =&\int_{0}^{\pi}c(\alpha)f(\alpha)\int_{0}^{\pi}[\frac{1}{(\alpha-\beta)^2+\epsilon^2}-\frac{2\epsilon^2}{((\alpha-\beta)^2+\epsilon^2)^2}]H_{\mathcal{R}}g(\beta)d\beta d\alpha+\tilde{B.T}.
  \end{align}
Let us verify it.  Since $|c(\alpha)|\lesssim |\alpha|$, we can use the Hilbert's Inequality and have 
 \begin{align*}
     \|f(\alpha)c(\alpha)\int_{-\infty}^{0}\frac{1}{(\alpha-\beta)^2+\epsilon^2}H_{\mathcal{R}}g(\beta)d\beta\|_{L^1[0,\pi]}&\lesssim\|\int_{0}^{\infty}\frac{1}{\alpha+\beta}|H_{\mathcal{R}}g(-\beta)|d\beta\|_{L_{\al}^2[0,\pi]} \|f\|_{L_{\al}^2[0,\pi]}\\
     &\lesssim \|H_{\mathcal{R}}(g)\|_{L^2[0,\infty)}\|f\|_{L_{\al}^2[0,\pi]}=\tilde{B.T}, \end{align*}
     and
      \begin{align*}
    \|f(\alpha)c(\alpha)\int_{-\infty}^{0}\frac{2\epsilon^2}{((\alpha-\beta)^2+\epsilon^2)^2}H_{\mathcal{R}}g(\beta)d\beta\|_{L^1[0,\pi]}&\lesssim\|\int_{0}^{\infty}\frac{1}{\alpha+\beta}|H_{\mathcal{R}}g(-\beta)|d\beta\|_{L_{\al}^2[0,\pi]} \|f\|_{L_{\al}^2[0,\pi]}\\
     &\lesssim \|H_{\mathcal{R}}(g)\|_{L^2[0,\infty)}\|f\|_{L_{\al}^2[0,\pi]}=\tilde{B.T}. \end{align*}
 For the integral from $
\pi$ to $\infty$, we use $\supp{c(\alpha)}\subset[0,\frac{\pi}{4}]$ and we have 
\begin{align*}
&\quad\|\int_{\pi}^{\infty}\frac{2\epsilon^2}{((\alpha-\beta)^2+\epsilon^2)^2}H_{\mathcal{R}}g(\beta)d\beta\|_{L_{\alpha}^{\infty}[0,\frac{\pi}{4}]}+\|\int_{\pi}^{\infty}\frac{1}{(\alpha-\beta)^2+\epsilon^2}H_{\mathcal{R}}g(\beta)d\beta\|_{L_{\alpha}^{\infty}[0,\frac{\pi}{4}]}\\
&\lesssim\|\int_{\pi}^{\infty}\frac{1}{(\beta-\frac{\pi}{4})^2}|H_{\mathcal{R}}g(\beta)|d\beta\|_{L_{\alpha}^{\infty}[0,\frac{\pi}{4}]}\lesssim\|H_{\mathcal{R}}(g)\|_{L^2(0,\infty)}.
\end{align*}
Then \eqref{IE4} follows.
Moreover
\begin{equation}\label{integralcancel}
    \int_{0}^{\pi}(\frac{1}{(\alpha-\beta)^2+\epsilon^2}-\frac{2\epsilon^2}{(\alpha-\beta)^2+\epsilon^2})d\beta=\int_{0}^{\pi}\frac{d}{d\beta}\frac{\alpha-\beta}{(\alpha-\beta)^2+\epsilon^2}d\beta=\frac{\alpha-\pi}{(\alpha-\pi)^2+\epsilon^2}-\frac{\alpha}{\alpha^2+\epsilon^2}.
\end{equation}
Since $|\frac{\alpha-\pi}{(\alpha-\pi)^2+\epsilon^2}-\frac{\alpha}{\alpha^2+\epsilon^2}|\cdot|c(\alpha)|\lesssim 1$ on $[0,\frac{\pi}{4}]$, we have
\begin{align}\label{IE5}
&\int_{0}^{\pi}c(\alpha)f(\alpha)\int_{0}^{\pi}[\frac{1}{(\alpha-\beta)^2+\epsilon^2}-\frac{2\epsilon^2}{((\alpha-\beta)^2+\epsilon^2)^2}]H_{\mathcal{R}}g(\beta)d\beta d\alpha\\\nonumber
=&\int_{0}^{\pi}c(\alpha)f(\alpha)\int_{0}^{\pi}[\frac{1}{(\alpha-\beta)^2+\epsilon^2}-\frac{2\epsilon^2}{((\alpha-\beta)^2+\epsilon^2)^2}](H_{\mathcal{R}}g(\beta)-H_{\mathcal{R}}g(\alpha))d\beta d\alpha+\tilde{B.T}.\\\nonumber
\end{align}
Combining \eqref{IE5}, \eqref{IE4} and \eqref{changeintegrand1}, we can finish the lemma.
\end{proof}
\begin{lemma}\label{le3}
We have
\begin{align*}
    &\quad\int_{0}^{\pi}\int_{0}^{\pi}c(\alpha)f(\alpha)(g(\alpha)-g(\beta))\frac{1}{(\alpha-\beta)^2+\epsilon^2}d\beta d\alpha\\
    &=\frac{1}{2}\int_{0}^{\pi}\int_{0}^{\pi}c(\alpha)(f(\alpha)-f(\beta))(g(\alpha)-g(\beta))(\frac{1}{(\alpha-\beta)^2+\epsilon^2}-\frac{1}{(\alpha+\beta)^2+\epsilon^2})d\beta d\alpha+\tilde{B.T},
\end{align*}
and
\begin{align*}
&\quad\int_{0}^{\pi}\int_{0}^{\pi}c(\alpha)f(\alpha)(g(\alpha)-g(\beta))\frac{\epsilon^2}{((\alpha-\beta)^2+\epsilon^2)^2}d\beta d\alpha\\
    &=\frac{1}{2}\int_{0}^{\pi}\int_{0}^{\pi}c(\alpha)(f(\alpha)-f(\beta))(g(\alpha)-g(\beta))(\frac{\epsilon^2}{((\alpha-\beta)^2+\epsilon^2)^2}-\frac{\epsilon^2}{((\alpha+\beta)^2+\epsilon^2)^2})d\beta d\alpha+\tilde{B.T}.
\end{align*}
We remark that we do not need the $\supp f, \supp g$ condition here.
\end{lemma}
\begin{proof}
For the first inequality, we interchange $\alpha$ and $\beta$ and have,
\begin{align*}
    &\int_{0}^{\pi}\int_{0}^{\pi}c(\alpha)f(\alpha)(g(\alpha)-g(\beta))\frac{1}{(\alpha-\beta)^2+\epsilon^2}d\beta d\alpha\\
    =&\frac{1}{2}\int_{0}^{\pi}\int_{0}^{\pi}c(\alpha)(f(\alpha)-f(\beta))(g(\alpha)-g(\beta))(\frac{1}{(\alpha-\beta)^2+\epsilon^2}-\frac{1}{(\alpha+\beta)^2+\epsilon^2})d\beta d\alpha\\
    &+\frac{1}{2}\int_{0}^{\pi}\int_{0}^{\pi}(c(\alpha)-c(\beta))f(\beta)(g(\alpha)-g(\beta))\frac{1}{(\alpha-\beta)^2+\epsilon^2}d\beta d\alpha\\
    &+\frac{1}{2}\int_{0}^{\pi}\int_{0}^{\pi}c(\alpha)(f(\alpha)-f(\beta))(g(\alpha)-g(\beta))(\frac{1}{(\alpha+\beta)^2+\epsilon^2})d\beta d\alpha\\
    =&I_1+I_2+I_3.
\end{align*}
We now show $I_2 +I_3$ are good terms.
\begin{align*}
    &I_2+I_3\\
    =&\frac{1}{2}\int_{0}^{\pi}\int_{0}^{\pi}(c(\alpha)-c(\beta))f(\beta)(g(\alpha)-g(\beta))\frac{1}{(\alpha-\beta)^2+\epsilon^2}d\beta d\alpha\\
    &+\frac{1}{2}\int_{0}^{\pi}\int_{0}^{\pi}c(\alpha)(f(\alpha)-f(\beta))(g(\alpha)-g(\beta))(\frac{1}{(\alpha+\beta)^2+\epsilon^2})d\beta d\alpha\\
    =&\frac{1}{2}\int_{0}^{\pi}(-\int_{0}^{\pi}\frac{(c(\alpha)-c(\beta))}{(\alpha-\beta)^2+\epsilon^2}d\alpha+\int_{0}^{\pi}\frac{c(\alpha)}{(\alpha+\beta)^2+\epsilon^2}d\alpha)f(\beta)g(\beta)d\beta \\
    &+\frac{1}{2}\int_{0}^{\pi}\int_0^{\pi}\frac{(c(\alpha)-c(\beta))g(\alpha)}{(\alpha-\beta)^2+\epsilon^2}d\alpha f(\beta) d\beta\\
    &+\frac{1}{2}\int_{0}^{\pi}\int_0^{\pi}\frac{1}{(\alpha+\beta)^2+\epsilon^2}d\beta c(\alpha)f(\alpha) g(\alpha) d\alpha\\
    &-\frac{1}{2}\int_{0}^{\pi}\int_0^{\pi}\frac{(c(\alpha))f(\beta)}{(\alpha+\beta)^2+\epsilon^2}d\beta g(\alpha) d\alpha\\
    &-\frac{1}{2}\int_{0}^{\pi}\int_0^{\pi}\frac{(c(\alpha))g(\beta)}{(\alpha+\beta)^2+\epsilon^2}d\beta f(\alpha) d\alpha\\
    =& J_1+J_2+J_3+J_4+J_5.
\end{align*}
For the first term, from lemma \ref{boundene}, we have that $-\int_{0}^{\pi}\frac{(c(\alpha)-c(\beta))}{(\alpha-\beta)^2+\epsilon^2}d\alpha+\int_{0}^{\pi}\frac{c(\alpha)}{(\alpha+\beta)^2+\epsilon^2}d\alpha$ is bounded. Then $J_1=\tilde{B.T}.$ 

For $J_2$, if we again extend $g(\beta)$ to be a function $g$ on $\mathbb{R}$ with $g(\beta)=0$ on $[0,\pi]^c$, from \eqref{Hilbert inequality},  we have 
\begin{align*}
    &\int_0^{\pi}\frac{(c'(\alpha)(\alpha-\beta))g(\alpha)}{(\alpha-\beta)^2+\epsilon^2}d\alpha\\
    &=\int_{-\infty}^{\infty}\frac{(c'(\alpha)(\alpha-\beta))g(\alpha)}{(\alpha-\beta)^2+\epsilon^2}d\alpha\\
    &=\int_{-\infty}^{\infty}\frac{\epsilon}{(\alpha-\beta)^2+\epsilon^2}H_{\mathcal{R}}(c'(\alpha)g(\alpha))d\alpha.
\end{align*}
Hence 
\begin{align*}
    &\|\int_0^{\pi}\frac{(c(\alpha)-c(\beta))g(\alpha)}{(\alpha-\beta)^2+\epsilon^2}d\alpha\|_{L^2}\lesssim\|g\|_{L^2}+\|\int_0^{\pi}\frac{(c(\alpha)-c(\beta)-c'(\alpha)(\alpha-\beta))g(\alpha)}{(\alpha-\beta)^2+\epsilon^2}d\alpha\|_{L^2}\\
    &\lesssim \|g\|_{L^2}.
\end{align*}
Then $J_2$ can be controlled by Holder inequality.

$J_3$ is bounded because the boundness of $\int_0^{\pi}\frac{1}{(\alpha+\beta)^2+\epsilon^2}d\beta c(\alpha)$, and $J_4, J_5$ can be controlled by the Hilbert's inequality and $|\frac{c(\alpha)}{(\alpha+\beta)^2+\epsilon^2}|\lesssim |\frac{1}{\alpha+\beta}|$. 

For the second inequality, we have
\begin{align*}
    &\int_{0}^{\pi}\int_{0}^{\pi}c(\alpha)f(\alpha)(g(\alpha)-g(\beta))\frac{\epsilon^2}{((\alpha-\beta)^2+\epsilon^2)^2}d\beta d\alpha\\
    =&\frac{1}{2}\int_{0}^{\pi}\int_{0}^{\pi}c(\alpha)(f(\alpha)-f(\beta))(g(\alpha)-g(\beta))(\frac{\epsilon^2}{((\alpha-\beta)^2+\epsilon^2)^2}-\frac{\epsilon^2}{((\alpha+\beta)^2+\epsilon^2)^2})d\beta d\alpha\\
    &+\frac{1}{2}\int_{0}^{\pi}\int_{0}^{\pi}(c(\alpha)-c(\beta))f(\beta)(g(\alpha)-g(\beta))\frac{\epsilon^2}{((\alpha-\beta)^2+\epsilon^2)^2}d\beta d\alpha\\
    &+\frac{1}{2}\int_{0}^{\pi}\int_{0}^{\pi}c(\alpha)(f(\alpha)-f(\beta))(g(\alpha)-g(\beta))\frac{\epsilon^2}{((\alpha+\beta)^2+\epsilon^2)^2}d\beta d\alpha\\
    =&\tilde{I}_1+\tilde{I}_2+\tilde{I}_3.
\end{align*}
Similarly,
 \begin{align*}
     &\tilde{I_2}+\tilde{I_3}\\
    =&\frac{1}{2}\int_{0}^{\pi}\int_{0}^{\pi}(c(\alpha)-c(\beta))f(\beta)(g(\alpha)-g(\beta))\frac{\epsilon^2}{((\alpha-\beta)^2+\epsilon^2)^2}d\beta d\alpha\\
    &+\frac{1}{2}\int_{0}^{\pi}\int_{0}^{\pi}c(\alpha)(f(\alpha)-f(\beta))(g(\alpha)-g(\beta))\frac{\epsilon^2}{((\alpha+\beta)^2+\epsilon^2)^2}d\beta d\alpha\\
    =&\frac{1}{2}\int_{0}^{\pi}(-\int_{0}^{\pi}\frac{(c(\alpha)-c(\beta))(\epsilon^2)}{((\alpha-\beta)^2+\epsilon^2)^2}d\alpha+\int_{0}^{\pi}\frac{c(\alpha)\epsilon^2}{((\alpha+\beta)^2+\epsilon^2)^2}d\alpha)f(\beta)g(\beta)d\beta \\
    &+\frac{1}{2}\int_{0}^{\pi}\int_0^{\pi}\frac{(c(\alpha)-c(\beta))g(\alpha)\epsilon^2}{((\alpha-\beta)^2+\epsilon^2)^2}d\alpha f(\beta) d\beta\\
    &+\frac{1}{2}\int_{0}^{\pi}\int_0^{\pi}\frac{\epsilon^2}{((\alpha+\beta)^2+\epsilon^2)^2}d\beta c(\alpha)f(\alpha) g(\alpha) d\alpha\\
    &-\frac{1}{2}\int_{0}^{\pi}\int_0^{\pi}\frac{(c(\alpha))f(\beta)\epsilon^2}{((\alpha+\beta)^2+\epsilon^2)^2}d\beta g(\alpha) d\alpha\\
    &-\frac{1}{2}\int_{0}^{\pi}\int_0^{\pi}\frac{(c(\alpha))g(\beta)\epsilon^2}{((\alpha+\beta)^2+\epsilon^2)^2}d\beta f(\alpha) d\alpha\\
    =& \tilde{J}_1+\tilde{J}_2+\tilde{J}_3+\tilde{J}_4+\tilde{J}_5.
\end{align*}
From lemma \ref{newlebound}, we have $\tilde{J}_1=\tilde{B.T}.$
For $\tilde{J}_2$, we have
\[
|\int_0^{\pi}\frac{(c(\alpha)-c(\beta))g(\alpha)\epsilon^2}{((\alpha-\beta)^2+\epsilon^2)^2}d\alpha|\lesssim |\int_0^{\pi}\frac{\epsilon |g(\alpha)|}{(\alpha-\beta)^2+\epsilon^2}d\alpha|.
\]
Then $\tilde{J}_2=\tilde{B.T}.$
$\tilde{J}_3$, $\tilde{J}_4$ and $\tilde{J}_5$ can be controlled in the same way as $J_3$, $J_4$ and $J_5$ by using 
\[
|\frac{\epsilon^2}{(\alpha+\beta)^2+\epsilon^2}|\lesssim 1.
\]
\end{proof}
\begin{lemma}\label{boundene}
We have $\sup_{0\leq \beta\leq \pi}|-\int_{0}^{\pi}\frac{c(\alpha)-c(\beta)}{(\alpha-\beta)^2+\epsilon^2}d\alpha+\int_{0}^{\pi}\frac{c(\alpha)}{(\alpha+\beta)^2+\epsilon^2}d\alpha|\lesssim 1$.
\end{lemma}
\begin{proof}
Since $\supp c$ and $\supp \tilde{c}$ $\subset [0,\frac{\pi}{4}]$, when $\beta\geq\frac{\pi}{3}$, we have
\begin{align*}
    &\quad|-\int_{0}^{\pi}\frac{c(\alpha)-c(\beta)}{(\alpha-\beta)^2+\epsilon^2}d\alpha+\int_{0}^{\pi}\frac{c(\alpha)}{(\alpha+\beta)^2+\epsilon^2}d\alpha|\\
    &=|-\int_{0}^{\frac{\pi}{4}}\frac{(c(\alpha))}{(\alpha-\beta)^2+\epsilon^2}d\alpha+\int_{0}^{\pi}\frac{c(\alpha)}{(\alpha+\beta)^2+\epsilon^2}d\alpha|\\
    &\lesssim 1.
\end{align*}
When $\beta\leq \frac{\pi}{3}$, we have
\begin{align*}
    &\quad-\int_{0}^{\pi}\frac{(c(\alpha)-c(\beta))}{(\alpha-\beta)^2+\epsilon^2}d\alpha+\int_{0}^{\pi}\frac{c(\alpha)}{(\alpha+\beta)^2+\epsilon^2}d\alpha\\
    &=-\int_{2\beta}^{\pi}\frac{(c(\alpha)-c(\beta))}{(\alpha-\beta)^2+\epsilon^2}d\alpha+\int_{0}^{\pi}\frac{c(\alpha)}{(\alpha+\beta)^2+\epsilon^2}d\alpha-\int_{0}^{2\beta}\frac{(c(\alpha)-c(\beta)-c'(\beta)(\alpha-\beta))}{(\alpha-\beta)^2+\epsilon^2}d\alpha\\
    &=-\int_{0}^{\pi-2\beta}\frac{c(\alpha+2\beta)}{(\alpha+\beta)^2+\epsilon^2}d\alpha+\int_{0}^{\pi-2\beta}\frac{c(\beta)}{(\alpha+\beta)^2+\epsilon^2}d\alpha+\int_{0}^{\pi-2\beta}\frac{c(\alpha)}{(\alpha+\beta)^2+\epsilon^2}d\alpha+\int_{\pi-2\beta}^{\pi}\frac{c(\alpha)}{(\alpha+\beta)^2+\epsilon^2}d\alpha\\
    &\quad-\int_{0}^{2\beta}\frac{(c(\alpha)-c(\beta)-c'(\beta)(\alpha-\beta))}{(\alpha-\beta)^2+\epsilon^2}d\alpha\\
    &=Term_1+Term_2+Term_3+Term_4+Term_5.
\end{align*}
Here
\[
|Term_1+Term_3|=|\int_{0}^{\pi-2\beta}\frac{c(\alpha+2\beta)-c(\alpha)}{(\alpha+\beta)^2+\epsilon^2}d\alpha|\lesssim 2\beta\int_{0}^{\pi-2\beta}\frac{1}{(\alpha+\beta)^2}d\alpha\lesssim 1,
\]
\[
|Term_2|\lesssim  \beta\int_{0}^{\pi}\frac{1}{(\alpha+\beta)^2}d\alpha\lesssim 1,
\]
and since $\beta\leq\frac{\pi}{3}$,
\[
|Term_4|\lesssim 1,
\]

\[
|Term_5|\lesssim  2\beta\lesssim 1.
\]
Then we have the result.
\end{proof}
\begin{lemma}\label{newlebound}
We have
$\sup_{0\leq \beta\leq \pi}|-\int_{0}^{\pi}\frac{(c(\alpha)-c(\beta))\epsilon^2}{((\alpha-\beta)^2+\epsilon^2)^2}d\alpha+\int_{0}^{\pi}\frac{c(\alpha)\epsilon^2}{((\alpha+\beta)^2+\epsilon^2)^2}d\alpha|\lesssim 1$.
\end{lemma}
\begin{proof}
The proof is similar as in lemma \ref{boundene}.
\end{proof}
\begin{lemma}\label{le4}
We have
\begin{align}\label{le4eq}
   &\int_{0}^{\pi}\int_{0}^{\pi}c(\alpha)f(\alpha)\frac{f(\alpha)-f(\beta)}{(\alpha-\beta)^2+\epsilon^2}d\beta d\alpha\\\nonumber
    &=\int_{0}^{\pi}\int_{0}^{\pi}c(\alpha)H_{\mathcal{R}}f(\alpha)(H_{\mathcal{R}}f(\alpha)-H_{\mathcal{R}}f(\beta))(\frac{1}{(\alpha-\beta)^2+\epsilon^2})d\beta d\alpha+\tilde{B.T}.
\end{align}
\end{lemma}
\begin{proof}
We first show 
\begin{align}\label{bd01}
    &\quad \int_{0}^{\pi}\int_{0}^{\pi}c(\alpha)f(\alpha)(f(\alpha)-f(\beta))(\frac{1}{(\alpha-\beta)^2+\epsilon^2}-\frac{2\epsilon^2}{((\alpha-\beta)^2+\epsilon^2)^2})d\beta d\alpha\\\nonumber
    &=\int_{0}^{\pi}\int_{0}^{\pi}c(\alpha)H_{\mathcal{R}}f(\alpha)(H_{\mathcal{R}}f(\alpha)-H_{\mathcal{R}}f(\beta))(\frac{1}{(\alpha-\beta)^2+\epsilon^2}-\frac{2\epsilon^2}{((\alpha-\beta)^2+\epsilon^2)^2})d\beta d\alpha+\tilde{B.T}.
\end{align}
 We have
\begin{align*}
    &\int_{0}^{\pi}\int_{0}^{\pi}c(\alpha)f(\alpha)(f(\alpha)-f(\beta))(\frac{1}{(\alpha-\beta)^2+\epsilon^2}-\frac{2\epsilon^2}{((\alpha-\beta)^2+\epsilon^2)^2})d\beta d\alpha\\
    =&_{\eqref{integralcancel}}-\int_{0}^{\pi}c(\alpha)f(\alpha)\int_{0}^{\pi}f(\beta)(\frac{1}{(\alpha-\beta)^2+\epsilon^2}-\frac{2\epsilon^2}{((\alpha-\beta)^2+\epsilon^2)^2})d\beta d\alpha+\tilde{B.T}\\
    =&_{\eqref{IE4}}-\int_{0}^{\pi}c(\alpha)f(\alpha)\int_{-\infty}^{\infty}f(\beta)(\frac{1}{(\alpha-\beta)^2+\epsilon^2}-\frac{2\epsilon^2}{((\alpha-\beta)^2+\epsilon^2)^2})d\beta d\alpha+\tilde{B.T}\\
    =&_{\eqref{newhilbertderi}}-2\int_{0}^{\pi}c(\alpha)f(\alpha)\int_{-\infty}^{\infty}H_{\mathcal{R}}f(\beta)\frac{(\alpha-\beta)\epsilon}{((\alpha-\beta)^2+\epsilon^2)^2}d\beta d\alpha+\tilde{B.T}\\
    =&-2\int_{0}^{\pi}f(\alpha)\int_{-\infty}^{\infty}c(\beta)H_{\mathcal{R}}f(\beta)\frac{(\alpha-\beta)\epsilon}{((\alpha-\beta)^2+\epsilon^2)^2}d\beta d\alpha+2\int_{0}^{\pi}f(\alpha)\int_{-\infty}^{\infty}H_{\mathcal{R}}f(\beta)\frac{(c(\beta)-c(\alpha))(\alpha-\beta)\epsilon}{((\alpha-\beta)^2+\epsilon^2)^2}d\beta d\alpha\\
    &\quad+\tilde{B.T}\\
    =&-2\int_{0}^{\pi}c(\beta)H_{\mathcal{R}}f(\beta)\int_{-\infty}^{\infty}f(\alpha)\frac{(\alpha-\beta)\epsilon}{((\alpha-\beta)^2+\epsilon^2)^2}d\alpha d\beta+\tilde{B.T}\\
    =&_{\eqref{newhilbertderi}}-\int_{0}^{\pi}c(\beta)H_{\mathcal{R}}f(\beta)\int_{-\infty}^{\infty}H_{\mathcal{R}}f(\alpha)(\frac{1}{(\alpha-\beta)^2+\epsilon^2}-\frac{2\epsilon^2}{((\alpha-\beta)^2+\epsilon^2)^2}) d\alpha d\beta+\tilde{B.T}\\
     =&-\int_{0}^{\pi}c(\alpha)H_{\mathcal{R}}f(\alpha)\int_{-\infty}^{\infty}H_{\mathcal{R}}f(\beta)(\frac{1}{(\alpha-\beta)^2+\epsilon^2}-\frac{2\epsilon^2}{((\alpha-\beta)^2+\epsilon^2)^2}) d\alpha d\beta+\tilde{B.T}\\
   =&_{\eqref{IE4}}-\int_{0}^{\pi}c(\alpha)H_{\mathcal{R}}f(\alpha)\int_{0}^{\pi}H_{\mathcal{R}}f(\beta)(\frac{1}{(\alpha-\beta)^2+\epsilon^2}-\frac{2\epsilon^2}{((\alpha-\beta)^2+\epsilon^2)^2}) d\alpha d\beta+\tilde{B.T}\\
    =&_{\eqref{integralcancel}}-\int_{0}^{\pi}c(\alpha)H_{\mathcal{R}}f(\alpha)\int_{0}^{\pi}(H_{\mathcal{R}}f(\beta)-H_{\mathcal{R}}f(\alpha))(\frac{1}{(\alpha-\beta)^2+\epsilon^2}-\frac{2\epsilon^2}{((\alpha-\beta)^2+\epsilon^2)^2}) d\beta d\alpha+\tilde{B.T},
\end{align*}
where we use Hilbert's inequality and the fact that $|c(\alpha)|\lesssim |\alpha|$, $\supp c(\alpha)\subset [0,\frac{\pi}{4}]$.

Then we only need to show
\begin{align}\label{bd02}
     &\quad \int_{0}^{\pi}\int_{0}^{\pi}c(\alpha)f(\alpha)(f(\alpha)-f(\beta))(\frac{\epsilon^2}{((\alpha-\beta)^2+\epsilon^2)^2})d\beta d\alpha\\\nonumber
    &=\int_{0}^{\pi}\int_{0}^{\pi}c(\alpha)H_{\mathcal{R}}f(\alpha)(H_{\mathcal{R}}f(\alpha)-H_{\mathcal{R}}f(\beta))(\frac{\epsilon^2}{((\alpha-\beta)^2+\epsilon^2)^2})d\beta d\alpha+\tilde{B.T}.
\end{align}
From \eqref{bd02} and \eqref{bd01}, we can get \eqref{le4eq}.

For the (LHS) of \eqref{bd02}, we  extend $f(\alpha)$, $c(\alpha)$ to be a function on $\mathbb{R}$ with $f(\alpha)=0$, $c(\alpha)=0$ on $[0,\pi]^c$. From similar estimate as in \eqref{IE4}, we have
\begin{align}\label{IE402}
 &\int_{0}^{\pi}c(\alpha)f(\alpha)\int_{-\infty}^{\infty}\frac{\epsilon^2}{((\alpha-\beta)^2+\epsilon^2)^2}f(\beta)d\beta d\alpha\\\nonumber
=&\int_{0}^{\pi}c(\alpha)f(\alpha)\int_{0}^{\pi}\frac{\epsilon^2}{((\alpha-\beta)^2+\epsilon^2)^2}f(\beta)d\beta d\alpha+\tilde{B.T}.
\end{align}
\begin{align}\label{IE4021}
 &\int_{0}^{\pi}c(\alpha)H_Rf(\alpha)\int_{-\infty}^{\infty}\frac{\epsilon^2}{((\alpha-\beta)^2+\epsilon^2)^2}H_{R}f(\beta)d\beta d\alpha\\\nonumber
=&\int_{0}^{\pi}c(\alpha)H_{R}f(\alpha)\int_{0}^{\pi}\frac{\epsilon^2}{((\alpha-\beta)^2+\epsilon^2)^2}H_{R}f(\beta)d\beta d\alpha+\tilde{B.T}.
\end{align}
Moreover, for $\frac{\pi}{4}>\alpha>0$, we have
\[
|\int_{-\infty}^{0}\frac{\epsilon^2}{((\alpha-\beta)^2+\epsilon^2)^2}d\beta c(\alpha)|\lesssim |\int^{\infty}_{0}\frac{1}{(\alpha+\beta)^2}d\beta|| c(\alpha)|\lesssim 1,
\]
and
\[
|\int_{\pi}^{\infty}\frac{\epsilon^2}{((\alpha-\beta)^2+\epsilon^2)^2}d\beta c(\alpha)|\lesssim |\int^{\infty}_{\pi}\frac{1}{(\frac{\pi}{4}-\beta)^2}d\beta|| c(\alpha)|\lesssim 1.
\]
Hence we have
\begin{align}\label{IE403}
 &\int_{0}^{\pi}c(\alpha)f^2(\alpha)\int_{-\infty}^{\infty}\frac{\epsilon^2}{((\alpha-\beta)^2+\epsilon^2)^2}d\beta d\alpha\\\nonumber
=&\int_{0}^{\pi}c(\alpha)f^2(\alpha)\int_{0}^{\pi}\frac{\epsilon^2}{((\alpha-\beta)^2+\epsilon^2)^2}d\beta d\alpha+\tilde{B.T},
\end{align}
and \begin{align}\label{IE40302}
 &\int_{0}^{\pi}c(\alpha)(H_{R}f(\alpha))^2\int_{-\infty}^{\infty}\frac{\epsilon^2}{((\alpha-\beta)^2+\epsilon^2)^2}d\beta d\alpha\\\nonumber
=&\int_{0}^{\pi}c(\alpha)(H_{R}f(\alpha))^2\int_{0}^{\pi}\frac{\epsilon^2}{((\alpha-\beta)^2+\epsilon^2)^2}d\beta d\alpha+\tilde{B.T}.
\end{align}
Therefore, since $\supp c\subset[0,\frac{\pi}{4}], \supp f\subset[0,\pi],$ we have
\begin{align*}
    &\int_{0}^{\pi}\int_{0}^{\pi}c(\alpha)f(\alpha)(f(\alpha)-f(\beta))(\frac{\epsilon^2}{((\alpha-\beta)^2+\epsilon^2)^2})d\beta d\alpha\\
    =&\eqref{IE402}, \eqref{IE403}\int_{0}^{\pi}c(\alpha)f(\alpha)\int_{-\infty}^{\infty}(f(\alpha)-f(\beta))\frac{\epsilon^2}{((\alpha-\beta)^2+\epsilon^2)^2}d\beta d\alpha+\tilde{B.T}.\\
    =&\int_{-\infty}^{\infty}c(\alpha)f(\alpha)\int_{-\infty}^{\infty}(f(\alpha)-f(\beta))\frac{\epsilon^2}{((\alpha-\beta)^2+\epsilon^2)^2}d\beta d\alpha+\tilde{B.T}.\\
    =&\int_{-\infty}^{\infty}H_{\mathcal{R}}(c(\alpha)f(\alpha))\int_{-\infty}^{\infty}(H_{\mathcal{R}}f(\alpha)-H_{\mathcal{R}}f(\beta))\frac{\epsilon^2}{((\alpha-\beta)^2+\epsilon^2)^2}d\beta d\alpha+\tilde{B.T}.\\
    =&\int_{-\infty}^{\infty}c(\alpha)H_{\mathcal{R}}f(\alpha)\int_{-\infty}^{\infty}(H_{\mathcal{R}}f(\alpha)-H_{\mathcal{R}}f(\beta))\frac{\epsilon^2}{((\alpha-\beta)^2+\epsilon^2)^2}d\beta d\alpha\\
    &\quad-\frac{1}{\pi}\int_{-\infty}^{\infty}\int_{0}^{\pi}\frac{(c(\alpha)-c(\gamma))f(\gamma)}{\alpha-\gamma}d\gamma\int_{-\infty}^{\infty}(H_{\mathcal{R}}f(\alpha)-H_{\mathcal{R}}f(\beta))\frac{\epsilon^2}{((\alpha-\beta)^2+\epsilon^2)^2}d\beta d\alpha+\tilde{B.T}\\
    &=\eqref{IE4021},\eqref{IE40302}\int_{0}^{\pi}c(\alpha)H_{\mathcal{R}}f(\alpha)\int_{0}^{\pi}(H_{\mathcal{R}}f(\alpha)-H_{\mathcal{R}}f(\beta))\frac{\epsilon^2}{((\alpha-\beta)^2+\epsilon^2)^2}d\beta d\alpha\\
    &\quad-\frac{1}{\pi}\int_{-\infty}^{\infty}\int_{0}^{\pi}\frac{(c(\alpha)-c(\gamma))f(\gamma)}{\alpha-\gamma}d\gamma\int_{-\infty}^{\infty}(H_{\mathcal{R}}f(\alpha)-H_{\mathcal{R}}f(\beta))\frac{\epsilon^2}{((\alpha-\beta)^2+\epsilon^2)^2}d\beta d\alpha+\tilde{B.T}.
\end{align*}
Then in order to show \eqref{bd02}, we are left to show 
\begin{equation}\label{equation2compliment}
    \int_{-\infty}^{\infty}\int_{0}^{\pi}\frac{(c(\alpha)-c(\gamma))f(\gamma)}{\alpha-\gamma}d\gamma\int_{-\infty}^{\infty}(H_{\mathcal{R}}f(\alpha)-H_{\mathcal{R}}f(\beta))\frac{\epsilon^2}{((\alpha-\beta)^2+\epsilon^2)^2}d\beta d\alpha=\tilde{B.T}.
\end{equation}

Here
\begin{align*}
     &\int_{-\infty}^{\infty}\int_{0}^{\pi}\frac{(c(\alpha)-c(\gamma))f(\gamma)}{\alpha-\gamma}d\gamma\int_{-\infty}^{\infty}(H_{\mathcal{R}}f(\alpha)-H_{\mathcal{R}}f(\beta))\frac{\epsilon^2}{((\alpha-\beta)^2+\epsilon^2)^2}d\beta d\alpha\\
     =& \int_{-\infty}^{\infty}\int_{0}^{\pi}\frac{(c(\alpha)-c(\gamma))f(\gamma)}{\alpha-\gamma}d\gamma\int_{-\infty}^{\infty}(H_{\mathcal{R}}f(\alpha))\frac{\epsilon^2}{((\alpha-\beta)^2+\epsilon^2)^2}d\beta d\alpha\\
     &-\int_{-\infty}^{\infty}\int_{0}^{\pi}\frac{(c(\alpha)-c(\gamma))f(\gamma)}{\alpha-\gamma}d\gamma\int_{-\infty}^{\infty}(H_{\mathcal{R}}f(\beta))\frac{\epsilon^2}{((\alpha-\beta)^2+\epsilon^2)^2}d\beta d\alpha\\
     =& \int_{-\infty}^{\infty}\int_{0}^{\pi}\frac{(c(\alpha)-c(\gamma))f(\gamma)}{\alpha-\gamma}d\gamma\int_{-\infty}^{\infty}(H_{\mathcal{R}}f(\alpha))\frac{\epsilon^2}{((\alpha-\beta)^2+\epsilon^2)^2}d\beta d\alpha\\
     &-\int_{-\infty}^{\infty}\int_{-\infty}^{\infty}\int_{0}^{\pi}\frac{(c(\beta)-c(\gamma))f(\gamma)}{\beta-\gamma}d\gamma(H_{\mathcal{R}}f(\alpha))\frac{\epsilon^2}{((\alpha-\beta)^2+\epsilon^2)^2}d\beta d\alpha\\
     =& \int_{-\infty}^{\infty}\int_{-\infty}^{\infty}\int_{0}^{\pi}(\frac{(c(\alpha)-c(\gamma))}{\alpha-\gamma}-\frac{(c(\beta)-c(\gamma))}{\beta-\gamma})f(\gamma)d\gamma\frac{\epsilon^2}{((\alpha-\beta)^2+\epsilon^2)^2}d\beta (H_{\mathcal{R}}f(\alpha))d\alpha.
\end{align*}
Then we separate the integral into 5 parts. For the integral over $\alpha<0,$ and  $\beta <0$,
\begin{align*}
    & |\int_{-\infty}^{0}\int_{-\infty}^{0}\int_{0}^{\pi}(\frac{(c(\alpha)-c(\gamma))}{\alpha-\gamma}-\frac{(c(\beta)-c(\gamma))}{\beta-\gamma})f(\gamma)d\gamma\frac{\epsilon^2}{((\alpha-\beta)^2+\epsilon^2)^2}d\beta (H_{\mathcal{R}}f(\alpha))d\alpha|\\
    = &| \int_{0}^{\infty}\int_{0}^{\infty}\int_{0}^{\pi}(\frac{c(\gamma)}{\alpha+\gamma}-\frac{c(\gamma)}{\beta+\gamma})f(\gamma)d\gamma\frac{\epsilon^2}{((\alpha-\beta)^2+\epsilon^2)^2}d\beta H_{\mathcal{R}}f(-\alpha)d\alpha|\\
    \leq&|\int_{0}^{\infty}\int_{0}^{\infty}\frac{\int_{0}^{\pi}\frac{|c(\gamma)|}{|(\alpha+\gamma)(\beta+\gamma)|}|f(\gamma)|d\gamma\epsilon^2|\alpha-\beta|}{((\alpha-\beta)^2+\epsilon^2)^2}d\beta |H_{\mathcal{R}}f(-\alpha)|d\alpha|\\
    \leq&|\int_{0}^{\infty}\int_{0}^{\infty}\frac{\int_{0}^{\pi}\frac{|f(\gamma)|}{|(\alpha+\gamma)|}d\gamma\epsilon^2|\alpha-\beta|}{((\alpha-\beta)^2+\epsilon^2)^2}d\beta |H_{\mathcal{R}}f(-\alpha)|d\alpha|\\
    \lesssim&|\int_{0}^{\infty}\int_{0}^{\pi}\frac{|f(\gamma)|}{|(\alpha+\gamma)|}d\gamma |H_{\mathcal{R}}f(-\alpha)|d\alpha|\\
    =&\tilde{B.T},
\end{align*}
where we use the Holder inequality and the Hilbert's inequality in the last step.

For the integral over $\alpha<0$, $\beta>0$,\[\frac{(c(\alpha)-c(\gamma))}{\alpha-\gamma}-\frac{(c(\beta)-c(\gamma))}{\beta-\gamma}=\frac{c(\alpha)-c(\beta)}{\alpha-\gamma}-\frac{(c(\beta)-c(\gamma))(\alpha-\beta)}{(\alpha-\gamma)(\beta-\gamma)}.
\]
 We have 
\begin{align*}
&|\int_{-\infty}^{0}\int_{0}^{\infty}\int_{0}^{\pi}\frac{c(\alpha)-c(\beta)}{\alpha-\gamma}f(\gamma)d\gamma\frac{\epsilon^2}{((\alpha-\beta)^2+\epsilon^2)^2}d\beta (H_{\mathcal{R}}f(\alpha))d\alpha|\\
\leq&|\int_{0}^{\infty}\int_{0}^{\infty}\int_{0}^{\pi}\frac{1}{\alpha+\gamma}|f(\gamma)|d\gamma\frac{\epsilon^2|c(-\alpha)-c(\beta)|}{((\alpha+\beta)^2+\epsilon^2)^2}d\beta |H_{\mathcal{R}}f(-\alpha)|d\alpha|\\
\lesssim&|\int_{0}^{\infty}\int_{0}^{\pi}\frac{|f(\gamma)|}{|(\alpha+\gamma)|}d\gamma |H_{\mathcal{R}}f(-\alpha)|d\alpha|\\
&=\tilde{B.T}.
\end{align*}
Here we use $|c(-\alpha)-c(\beta)|\lesssim |\alpha+\beta|$.
Moreover, 
\begin{align*}
&|\int_{-\infty}^{0}\int_{0}^{\infty}\int_{0}^{\pi}\frac{(c(\beta)-c(\gamma))(\alpha-\beta)}{(\alpha-\gamma)(\beta-\gamma)}f(\gamma)d\gamma\frac{\epsilon^2}{((\alpha-\beta)^2+\epsilon^2)^2}d\beta (H_{\mathcal{R}}f(\alpha))d\alpha|\\
\lesssim&|\int_{0}^{\infty}\int_{0}^{\infty}\int_{0}^{\pi}\frac{1}{\alpha+\gamma}|f(\gamma)|d\gamma\frac{\epsilon^2|\alpha+\beta|}{((\alpha+\beta)^2+\epsilon^2)^2}d\beta |H_{\mathcal{R}}f(-\alpha)|d\alpha|\\
\lesssim&|\int_{0}^{\infty}\int_{0}^{\pi}\frac{|f(\gamma)|}{|(\alpha+\gamma)|}d\gamma |H_{\mathcal{R}}f(-\alpha)|d\alpha|\\
&=\tilde{B.T},
\end{align*}
where we use $|\frac{(c(\beta)-c(\gamma))}{\beta-\gamma}|\lesssim 1$.
For the integral over $\alpha>0$, $\beta<0$,
\[\frac{(c(\alpha)-c(\gamma))}{\alpha-\gamma}-\frac{(c(\beta)-c(\gamma))}{\beta-\gamma}=\frac{(c(\alpha)-c(\gamma))(\beta-\alpha)}{(\alpha-\gamma)(\beta-\gamma)}+\frac{c(\alpha)-c(\beta)}{\beta-\gamma}.
\]
We have 
\begin{align*}
    &|\int_{0}^{\infty}\int_{-\infty}^{0}\int_{0}^{\pi}\frac{(c(\alpha)-c(\beta))}{(\beta-\gamma)}f(\gamma)d\gamma\frac{\epsilon^2}{((\alpha-\beta)^2+\epsilon^2)^2}d\beta (H_{\mathcal{R}}f(\alpha))d\alpha|\\
   = &|\int_{0}^{\infty}\int_{0}^{\infty}\int_{0}^{\pi}\frac{1}{(\beta+\gamma)}f(\gamma)d\gamma\frac{\epsilon^2(c(\alpha)-c(-\beta))}{((\alpha+\beta)^2+\epsilon^2)^2}d\beta (H_{\mathcal{R}}f(\alpha))d\alpha|\\
   \lesssim&|\int_{0}^{\infty}\int_{0}^{\infty}\int_{0}^{\pi}\frac{1}{(\beta+\gamma)}|f(\gamma)|d\gamma\frac{\epsilon^2|\alpha+\beta|}{((\alpha+\beta)^2+\epsilon^2)^2}d\beta |H_{\mathcal{R}}(f)(\alpha)|d\alpha|\\
   \lesssim&\|H_{\mathcal{R}}f\|_{L^2}\|\int_{-\infty}^{\infty}\int_{0}^{\pi}\frac{1}{(\beta+\gamma)}|f(\gamma)|1_{\beta\geq 0}d\gamma\frac{\epsilon^2|\alpha+\beta|}{((\alpha+\beta)^2+\epsilon^2)^2}d\beta\|_{L^2}\\
   \lesssim &\|H_{\mathcal{R}}(f)\|_{L^2}\|\int_{0}^{\pi}\frac{1}{(\beta+\gamma)}|f(\gamma)|d\gamma\|_{L^2}\int_{-\infty}^{\infty}\frac{\epsilon^2|\beta|}{((\beta)^2+\epsilon^2)^2}d\beta\\
   =&\tilde{B.T}.
\end{align*}
Moreover,
\begin{align*}
     &|\int_{0}^{\infty}\int_{-\infty}^{0}\int_{0}^{\pi}\frac{(c(\alpha)-c(\gamma))(\alpha-\beta)}{(\beta-\gamma)(\alpha-\gamma)}f(\gamma)d\gamma\frac{\epsilon^2}{((\alpha-\beta)^2+\epsilon^2)^2}d\beta H_{\mathcal{R}}f(\alpha)d\alpha|\\
      \leq&|\int_{0}^{\infty}\int_{0}^{\infty}\int_{0}^{\pi}\frac{1}{(\beta+\gamma)}|f(\gamma)|d\gamma\frac{\epsilon^2|\alpha+\beta|}{((\alpha+\beta)^2+\epsilon^2)^2}d\beta |H_{\mathcal{R}}f(\alpha)|d\alpha|,\\
\end{align*}
Then we can use the same way as in the previous inequality to bound it.

For the integral over $2\pi>\alpha>0$, $\beta>0$, we have
\[
|\frac{(c(\alpha)-c(\gamma))}{\alpha-\gamma}-\frac{(c(\beta)-c(\gamma))}{\beta-\gamma}|\lesssim |\alpha-\beta|.
\]
Hence
\begin{align*}
    &|\int_{0}^{2\pi}\int_{0}^{\infty}\int_{0}^{\pi}\frac{(c(\alpha)-c(\gamma))}{\alpha-\gamma}-\frac{(c(\beta)-c(\gamma))}{\beta-\gamma}f(\gamma)d\gamma\frac{\epsilon^2}{((\alpha-\beta)^2+\epsilon^2)^2}d\beta (H_{\mathcal{R}}f(\alpha))d\alpha|\\
    &\lesssim|\int_{0}^{2\pi}\int_{0}^{\infty}\int_{0}^{\pi}|f(\gamma)|d\gamma\frac{\epsilon^2|\alpha-\beta|}{((\alpha-\beta)^2+\epsilon^2)^2}d\beta |H_{\mathcal{R}}f(\alpha)|d\alpha|\\
    &\lesssim\int_{0}^{2\pi}|H_{\mathcal{R}}f(\alpha)|d\alpha\int_{0}^{\pi}|f(\gamma)|d\gamma\\
    &\lesssim \|f\|_{L^2}^2\\
    &=\tilde{B.T}.
\end{align*}
For the integral over $\beta>0, \alpha>2\pi$,
\[\frac{(c(\alpha)-c(\gamma))}{\alpha-\gamma}-\frac{(c(\beta)-c(\gamma))}{\beta-\gamma}=\frac{c(\alpha)-c(\beta)}{\alpha-\gamma}-\frac{(c(\beta)-c(\gamma))(\alpha-\beta)}{(\alpha-\gamma)(\beta-\gamma)}.
\]
We have
\begin{align*}
     &\quad|\int_{2\pi}^{\infty}\int_{0}^{\infty}\int_{0}^{\pi}\frac{c(\alpha)-c(\beta)}{\alpha-\gamma}f(\gamma)d\gamma\frac{\epsilon^2}{((\alpha-\beta)^2+\epsilon^2)^2}d\beta (H_{\mathcal{R}}f(\alpha))d\alpha|\\
     &\lesssim|\int_{2\pi}^{\infty}\int_{0}^{\infty}\int_{0}^{\pi}|f(\gamma)|d\gamma\frac{\epsilon^2|\alpha-\beta|}{((\alpha-\beta)^2+\epsilon^2)^2}d\beta \frac{|H_{\mathcal{R}}f(\alpha)|}{\alpha-\pi}d\alpha|\\
     &\lesssim\int_{2\pi}^{\infty}|H_{\mathcal{R}}f(\alpha)|\frac{1}{\alpha-\pi}d\alpha\int_{0}^{\pi}|f(\gamma)|d\gamma\\
     &\lesssim (\int_{2\pi}^{\infty}|H_{\mathcal{R}}f(\alpha)|^2d\alpha)^{\frac{1}{2}}\int_{0}^{\pi}|f(\gamma)|d\gamma\\
     &=\tilde{B.T},
\end{align*}
and
\begin{align*}
     &|\int_{2\pi}^{\infty}\int_{0}^{\infty}\int_{0}^{\pi}\frac{(c(\beta)-c(\gamma))(\alpha-\beta)}{(\alpha-\gamma)(\beta-\gamma)}f(\gamma)d\gamma\frac{\epsilon^2}{((\alpha-\beta)^2+\epsilon^2)^2}d\beta (H_{\mathcal{R}}f(\alpha))d\alpha|\\
      \lesssim&|\int_{2\pi}^{\infty}\int_{0}^{\infty}\int_{0}^{\pi}|f(\gamma)|d\gamma\frac{\epsilon^2|\alpha-\beta|}{((\alpha-\beta)^2+\epsilon^2)^2}d\beta \frac{|H_{\mathcal{R}}f(\alpha)|}{\alpha-\pi}d\alpha|\\
      =&\tilde{B.T}.
\end{align*}
Then we can use the same method as before to show the result. From the above estimates, we have \eqref{equation2compliment}. Thus \eqref{bd02} follows.
By \eqref{bd01} and \eqref{bd02}, we can show the lemma.
\end{proof}
Now we use show lemma \ref{Garding}. We first show \eqref{ene1}.  By lemma \ref{le2},
\begin{align}
    (LHS)=&\frac{1}{2}\int_{0}^{\pi}c(\alpha)f(\alpha)\int_{0}^{\pi}(H_{\mathcal{R}}g(\beta)-H_{\mathcal{R}}g(\alpha))(\frac{1}{((\alpha-\beta)^2+\epsilon^2)}-\frac{2\epsilon^2}{((\alpha-\beta)^2+\epsilon^2)^2})d\beta d\alpha\\\nonumber
    &+\int_{0}^{\pi}c(\alpha)f(\alpha)g(\alpha)\int_{0}^{\pi}\frac{(\alpha-\beta)\epsilon}{((\alpha-\beta)^2+\epsilon^2)^2}d\beta d\alpha+\tilde{B.T}\\\nonumber
    =&\frac{1}{2}\int_{0}^{\pi}c(\alpha)f(\alpha)\int_{0}^{\pi}(H_{\mathcal{R}}g(\beta)-H_{\mathcal{R}}g(\alpha))(\frac{1}{((\alpha-\beta)^2+\epsilon^2)}-\frac{2\epsilon^2}{((\alpha-\beta)^2+\epsilon^2)^2})d\beta d\alpha+\tilde{B.T}.\\\nonumber
    =&\frac{1}{2}\int_{0}^{\pi}c(\alpha)f(\alpha)\int_{0}^{\pi}(H_{\mathcal{R}}g(\beta)-H_{\mathcal{R}}g(\alpha))(\frac{1}{((\alpha-\beta)^2+\epsilon^2)})d\beta d\alpha\\\nonumber
    &-\frac{1}{2}\int_{0}^{\pi}c(\alpha)f(\alpha)\int_{0}^{\pi}(H_{\mathcal{R}}g(\beta)-H_{\mathcal{R}}g(\alpha))(\frac{2\epsilon^2}{((\alpha-\beta)^2+\epsilon^2)^2})d\beta d\alpha+\tilde{B.T}\\\nonumber
    =&L_1+L_2+\tilde{B.T}.
\end{align}
Next, we use lemma \ref{le3} and get
\begin{equation}
    L_1=\frac{1}{4}\int_{0}^{\pi}\int_{0}^{\pi}c(\alpha)(f(\alpha)-f(\beta))(H_{\mathcal{R}}g(\beta)-H_{\mathcal{R}}g(\alpha))(\frac{1}{(\alpha-\beta)^2+\epsilon^2}-\frac{1}{(\alpha+\beta)^2+\epsilon^2})d\beta d\alpha+\tilde{B.T},
\end{equation}
and
\begin{equation}
    L_2=-\frac{1}{2}\int_{0}^{\pi}\int_{0}^{\pi}c(\alpha)(f(\alpha)-f(\beta))(H_{\mathcal{R}}g(\beta)-H_{\mathcal{R}}g(\alpha))(\frac{\epsilon^2}{((\alpha-\beta)^2+\epsilon^2)^2}-\frac{\epsilon^2}{((\alpha+\beta)^2+\epsilon^2)^2})d\beta d\alpha+\tilde{B.T}.
\end{equation}
Since we have when $0\leq \alpha\leq \pi,0\leq \beta\leq \pi$, 
\[
\frac{1}{(\alpha-\beta)^2+\epsilon^2}-\frac{1}{(\alpha+\beta)^2+\epsilon^2}>0,
\]
\[
\frac{\epsilon^2}{((\alpha-\beta)^2+\epsilon^2)^2}-\frac{\epsilon^2}{((\alpha+\beta)^2+\epsilon^2)^2}>0,
\]
\[
\frac{\epsilon^2}{((\alpha-\beta)^2+\epsilon^2)^2}-\frac{\epsilon^2}{((\alpha+\beta)^2+\epsilon^2)^2}\leq 4(\frac{1}{(\alpha-\beta)^2+\epsilon^2}-\frac{1}{(\alpha+\beta)^2+\epsilon^2}).
\]
By Holder's inequality, we get 
\begin{align}\label{ene3}
    |(LHS)|\leq&\frac{9}{4}\int_{0}^{\pi}\int_{0}^{\pi}|c(\alpha)|(f(\alpha)-f(\beta))^2(\frac{1}{(\alpha-\beta)^2+\epsilon^2}-\frac{1}{(\alpha+\beta)^2+\epsilon^2})d\beta d\alpha\\\nonumber
    &+\frac{9}{4}\int_{0}^{\pi}\int_{0}^{\pi}|c(\alpha)|(H_{\mathcal{R}}g(\beta)-H_{\mathcal{R}}g(\alpha))^2(\frac{1}{(\alpha-\beta)^2+\epsilon^2}-\frac{1}{(\alpha+\beta)^2+\epsilon^2})d\beta d\alpha+\tilde{B.T}.
\end{align}
Moreover, in \eqref{ene1}, by lemma \ref{le3} and \ref{le4},
\begin{align}\label{ene4}
    (RHS)=&\frac{1}{4}\int_{0}^{\pi}\int_{0}^{\pi}\tilde{c}(\alpha)(f(\alpha)-f(\beta))^2(\frac{1}{(\alpha-\beta)^2+\epsilon^2}-\frac{1}{(\alpha+\beta)^2+\epsilon^2})d\beta d\alpha\\\nonumber
    &+\frac{1}{4}\int_{0}^{\pi}\int_{0}^{\pi}\tilde{c}(\alpha)(H_{\mathcal{R}}f(\beta)-H_{\mathcal{R}}f(\alpha))^2(\frac{1}{(\alpha-\beta)^2+\epsilon^2}-\frac{1}{(\alpha+\beta)^2+\epsilon^2})d\beta d\alpha\\\nonumber
    &+\frac{1}{4}\int_{0}^{\pi}\int_{0}^{\pi}\tilde{c}(\alpha)(g(\alpha)-g(\beta))^2(\frac{1}{(\alpha-\beta)^2+\epsilon^2}-\frac{1}{(\alpha+\beta)^2+\epsilon^2})d\beta d\alpha\\\nonumber
    &+\frac{1}{4}\int_{0}^{\pi}\int_{0}^{\pi}\tilde{c}(\alpha)(H_{\mathcal{R}}g(\beta)-H_{\mathcal{R}}g(\alpha))^2(\frac{1}{(\alpha-\beta)^2+\epsilon^2}-\frac{1}{(\alpha+\beta)^2+\epsilon^2})d\beta d\alpha\\\nonumber
    &+\tilde{B.T}.
\end{align}
By \eqref{ene3} and \eqref{ene4}, we can take $C=9$ and get \eqref{ene1}.
For \eqref{ene2}, by lemma \ref{le3}, we have
\begin{align*}
    &\quad(LHS)\\   
&=\frac{1}{2}\int_{0}^{\pi}\int_{0}^{\pi}c(\alpha)(f(\alpha)-f(\beta))(g(\alpha)-g(\beta))(\frac{1}{(\alpha-\beta)^2+\epsilon^2}-\frac{1}{(\alpha+\beta)^2+\epsilon^2})d\beta d\alpha+\tilde{B.T}.\\
    &\leq \frac{1}{4}\int_{0}^{\pi}\int_{0}^{\pi}c(\alpha)(f(\alpha)-f(\beta))^2(\frac{1}{(\alpha-\beta)^2+\epsilon^2}-\frac{1}{(\alpha+\beta)^2+\epsilon^2})d\beta d\alpha\\
    &\quad+\frac{1}{4}\int_{0}^{\pi}\int_{0}^{\pi}c(\alpha)(g(\alpha)-g(\beta))^2(\frac{1}{(\alpha-\beta)^2+\epsilon^2}-\frac{1}{(\alpha+\beta)^2+\epsilon^2})d\beta d\alpha+\tilde{B.T}.
\end{align*}
The (RHS) of \eqref{ene2} is same as \eqref{ene4}, since \eqref{ene1} and \eqref{ene2} have same (RHS). Then we can again take $C=9$.
\end{proof}

\subsubsection{Energy estimate}
Now we show the energy estimate. It is sufficient to show 
\begin{equation}\label{3eneryestimate}
\Re<h,T^{\epsilon}(h)>_{X^{k}}\lesssim C(\|h\|_{X^{k}}),
\end{equation}
for some bounded function $C(\cdot)$ independent of $\epsilon$. We will abuse the notation and use $C(\cdot)$ to show any bounded function and B.T to show any terms bounded by  $C(\|h\|_{X^{k}})$.

Then from lemma \ref{3lemmaotherterms}, conditions in \eqref{3GE01}, we have
\begin{equation}\label{3maintermscontrol}
\begin{split}
\Re<h,T^{\epsilon}(h)>_{X^{k}}&=\Re \int_{0}^{\pi}h^{(k)}(\alpha)\overline{\lambda(\alpha)\int_{0}^{2\alpha}\frac{(h^{(k)}(\alpha)-h^{(k)}(\beta))(\alpha-\beta)\epsilon}{((\alpha-\beta)^2+\epsilon^2)^2}d\beta \frac{2}{\pi}L_1^{+}(h)d\beta}\\
&\quad +\Re \int_{0}^{\pi}h^{(k)}(\alpha)\overline{\lambda(\alpha)L_2^{+}(h)\int_{0}^{2\alpha}\frac{h^{(k)}(\alpha)-h^{(k)}(\beta)}{(\alpha-\beta)^2+\epsilon^2}d\beta}\\
&\quad +\Re \int_{0}^{\pi}h^{(k)}(\alpha)\overline{\lambda(\alpha)\frac{h^{(k)}(\alpha+\epsilon)-h^{(k)}(\al)}{\epsilon}d\beta}\kappa(t)\\
&\quad+B.T.\\
&=E_1+E_2+E_3+B.T.
\end{split}
\end{equation}
with 
\[
B.T\lesssim C(\|h\|_{X^{k}}).
\]
Then in order to show \eqref{3eneryestimate}, from conditions in \eqref{3GE01}, we only need to show the following lemma \ref{30101}. In fact, from \eqref{3linearcoro01}, we have $E_1+E_2\lesssim C(\|h\|_{X^{k}})$ and from \eqref{3linearcoro02}, we have $E_3\lesssim C(\|h\|_{X^{k}}).$
\begin{lemma}\label{30101}
 If $g\in L_{\al}^{2}[0,\pi]\cap \{g(x)=0|x\geq \frac{\pi}{4}\}$, $\kappa(t)>0$, $L_1^{+}(h)$, $L_2^{+}(h)$ satisfying the following conditions:

\begin{equation}\label{RT2condition}
 \quad 18 |\Im L_{1}^{+}(h)(\al,\g,t)|+18 |\Im L_{2}^{+}(h)(\al,\g,t)|\leq -\Re L_2^{+}(h)(\al,\g,t), \text{ when }\al\in \supp{\lambda}, 
\end{equation}
\[
L_{i}^{+}(h)(0,\g,t)=0,
\]
\[
\|L_{i}^{+}(h)(\alpha,\g,t)\|_{C_{\alpha}^{2}[0,\frac{\pi}{4}]}\lesssim 1,
\]
 then we have
\begin{equation}\label{3linearcoro01}
\begin{split}
\Re &<g, \lambda(\al)L_1^{+}(h)(\al)\int_{0}^{2\al}\frac{(g(\al)-g(\beta))(\al-\beta)\ep}{((\al-\beta)^2+\ep^2)^2}d\beta\frac{2}{\pi}\\
&\quad+\lambda(\alpha)L_2^{+}(h)\int_{0}^{2\alpha}\frac{(g(\alpha)-g(\beta))}{(\alpha-\beta)^2+\epsilon^2}d\beta>_{L_{\al}^{2}[0,\pi]}\\
&\lesssim \|g\|_{L_{\al}^{2}[0,\pi]}^{2}.
\end{split}
\end{equation}
and
\begin{equation}\label{3linearcoro02}
\begin{split}
\Re &<g, \lambda(\alpha)\frac{g(\alpha+\epsilon)-g(\alpha)}{\epsilon}\kappa(t)>_{L_{\al}^{2}[0,\pi]}\lesssim \|g\|_{L_{\al}^{2}[0,\pi]}^{2}.
\end{split}
\end{equation}
\end{lemma}
\begin{proof}
We separate the real and imaginary part of $(LHS)$ of \eqref{3linearcoro01} and have
\begin{equation}\label{3E3split01}
\begin{split}
    I_1=&\Re \int_{0}^{\pi}g(\alpha)\overline{\lambda(\alpha)\frac{2}{\pi}L_1^{+}(h)(\al) \int_{0}^{2\al}\frac{g(\alpha)-g(\beta))(\alpha-\beta)\epsilon}{((\alpha-\beta)^2+\epsilon^2)^2}d\beta}d\alpha\\
=&\int_{0}^{\pi}\frac{2}{\pi}\lambda(\alpha)\Re L_1^{+}(h)(\al)\Re g(\alpha)\int_{0}^{2\al}\frac{(\Re g(\alpha)-\Re  g(\beta))(\alpha-\beta)\epsilon}{((\alpha-\beta)^2+\epsilon^2)^2}d\beta d\alpha\\
    &+ \int_{0}^{\pi}\frac{2}{\pi}\lambda(\alpha)\Re L_1^{+}(h)(\al)\Im g(\alpha)\int_{0}^{2\al}\frac{(\Im g(\alpha)-\Im g(\beta))(\alpha-\beta)\epsilon}{((\alpha-\beta)^2+\epsilon^2)^2}d\beta d\alpha\\
    &- \int_{0}^{\pi}\frac{2}{\pi}\lambda(\alpha)\Im L_1^{+}(h)(\al)\Re g(\alpha)\int_{0}^{2\al}\frac{(\Im g(\alpha)-\Im g(\beta))(\alpha-\beta)\epsilon}{((\alpha-\beta)^2+\epsilon^2)^2}d\beta d\alpha\\
    &+ \int_{0}^{\pi}\frac{2}{\pi}\lambda(\alpha)\Im L_1^{+}(h)(\al) \Im g(\alpha)\int_{0}^{2\al}\frac{(\Re g(\alpha)-\Re  g(\beta))(\alpha-\beta)\epsilon}{((\alpha-\beta)^2+\epsilon^2)^2}d\beta d\alpha
\end{split}
\end{equation}

\begin{equation}\label{3E1split}
\begin{split}
    I_2=&\quad \Re<g, \lambda(\alpha)L_2^{+}(h)(\alpha)\int_{0}^{2\alpha}\frac{(g(\alpha)-g(\beta))}{(\alpha-\beta)^2+\epsilon^2}d\beta>\\=&\int_{0}^{\pi}\lambda(\alpha)\Re L_2^{+}(h)(\al) \Re g(\alpha)\int_{0}^{2\al}\frac{\Re g(\alpha)-\Re g(\beta)}{(\alpha-\beta)^2+\epsilon^2}d\beta d\alpha\\
    &+ \int_{0}^{\pi}\lambda(\alpha)\Re L_2^{+}(h)(\al) \Im g(\alpha)\int_{0}^{2\al}\frac{\Im g(\alpha)-\Im g(\beta)}{(\alpha-\beta)^2+\epsilon^2}d\beta d\alpha\\
    &- \int_{0}^{\pi}\lambda(\alpha)\Im L_2^{+}(h)(\al) \Re  g(\alpha)\int_{0}^{2\al}\frac{\Im g(\alpha)-\Im g(\beta)}{(\alpha-\beta)^2+\epsilon^2}d\beta d\alpha\\
    &+\int_{0}^{2\al}\lambda(\alpha)\Im L_2^{+}(h)(\al)  \Im g(\alpha)\int_{0}^{2\al}\frac{\Re g(\alpha)-\Re g(\beta)}{(\alpha-\beta)^2+\epsilon^2}d\beta d\alpha
\end{split}
\end{equation}
For $I_1$, we need an extra lemma for the good term.
\begin{lemma}\label{integration by part control}
If $L_1^{+}(h)(0)=0$, $L_1^{+}(h)(\alpha)\in C^{1}_{\alpha}[0,\frac{\pi}{4}]$, $g\in L_{\al}^{2}[0,\pi]\cap \{g(x)=0|x\geq \frac{\pi}{2}\}$,  we have
\begin{align*}
    |\int_{0}^{\pi}\lambda(\al)L_1^{+}(h)(\al)g(\alpha)\int_{0}^{2\al}\frac{(g(\alpha)-g(\beta))(\alpha-\beta)\epsilon}{((\alpha-\beta)^2+\epsilon^2)^2}d\beta d\alpha|\lesssim \|L_{1}^{+}(h)(\alpha)\|_{C_{\al}^{1}[0,\frac{\pi}{4}]}\|g\|^2_{L_{\al}^2[0,\pi]}.
    \end{align*}
\begin{proof}
We separate $g(\alpha)$, $g(\beta)$ and get
\begin{align*}
     &|\int_{0}^{\pi}\lambda(\al)L_1^{+}(h)(\al)g(\alpha)\int_{0}^{2\al}\frac{(g(\alpha)-g(\beta))(\alpha-\beta)\epsilon}{((\alpha-\beta)^2+\epsilon^2)^2}d\beta d\alpha|\\
     &=
      |\int_{0}^{\pi}\lambda(\al)L_1^{+}(h)(\al)g(\alpha)\int_{0}^{2\al}\frac{(g(\beta))(\alpha-\beta)\epsilon}{((\alpha-\beta)^2+\epsilon^2)^2}d\beta d\alpha|\\
        &\lesssim |\int_{0}^{\pi}\lambda(\al)L_1^{+}(h)(\al)g(\alpha)\int_{2\al}^{\pi}\frac{(g(\beta))(\alpha-\beta)\epsilon}{((\alpha-\beta)^2+\epsilon^2)^2}d\beta d\alpha|+
      |\int_{0}^{\pi}\lambda(\al)L_1^{+}(h)(\al)g(\alpha)\int_{0}^{\pi}\frac{(g(\beta))(\alpha-\beta)\epsilon}{((\alpha-\beta)^2+\epsilon^2)^2}d\beta d\alpha|\\
     &\lesssim \int_{0}^{\frac{\pi}{4}}\|L_1^{+}(h)\|_{C^{1}_{\alpha}[0,\frac{\pi}{4}]}|g(\alpha)|\int_{2\al}^{\pi}\frac{|g(\beta)|}{\al+\beta}d\alpha\\
     &\quad+ |\frac{1}{2}\int_{0}^{\pi}g(\alpha)\int_{0}^{\pi}\frac{(\lambda(\al)L_1^{+}(h)(\al)-\lambda(\beta)(L_1^{+}(h)(\beta))}{\alpha-\beta}\frac{(g(\beta))(\alpha-\beta)^2\epsilon}{((\alpha-\beta)^2+\epsilon^2)^2}d\beta d\alpha|\\
     &\lesssim \|L_1^{+}(h)(\al)\|_{C_{\alpha}^{1}[0,\frac{\pi}{4}]}\|g\|^2_{L_{\al}^2[0,\pi]}\\
     &\quad+ |\int_{0}^{\pi}|g(\alpha)|\int_{0}^{\pi}\|L_1^{+}(h)(\al)\|_{C_{\alpha}^{1}[0,\frac{\pi}{4}]}\frac{|g(\beta)|\epsilon}{(\alpha-\beta)^2+\epsilon^2}d\beta d\alpha|\\
     &\lesssim \|L_1^{+}(h)(\al)\|_{C_{\alpha}^{1}[0,\frac{\pi}{4}]}\|g\|^2_{L_{\al}^2[0,\pi]},
\end{align*}
where we use the Young's inequality in the last step.
\end{proof}
\end{lemma}

Then from lemma \ref{integration by part control}, \eqref{3E3split01}, we can control the term with $\Re L_{1}^{+}(h)$ and have
\begin{equation}\label{3E3split}
\begin{split}
    I_1=&\Re \int_{0}^{\pi}g(\alpha)\overline{\lambda(\alpha)\frac{2}{\pi}L_1^{+}(h)(\al) \int_{0}^{2\al}\frac{g(\alpha)-g(\beta))(\alpha-\beta)\epsilon}{((\alpha-\beta)^2+\epsilon^2)^2}d\beta}d\alpha\\
    =&- \int_{0}^{\pi}\frac{2}{\pi}\lambda(\alpha)\Im L_1^{+}(h)(\al)\Re g(\alpha)\int_{0}^{2\al}\frac{(\Im g(\alpha)-\Im g(\beta))(\alpha-\beta)\epsilon}{((\alpha-\beta)^2+\epsilon^2)^2}d\beta d\alpha\\
    &+ \int_{0}^{\pi}\frac{2}{\pi}\lambda(\alpha)\Im L_1^{+}(h)(\al)\Im g(\alpha)\int_{0}^{2\al}\frac{(\Re g(\alpha)-\Re  g(\beta))(\alpha-\beta)\epsilon}{((\alpha-\beta)^2+\epsilon^2)^2}d\beta d\alpha\\
    &+B.T.^{0},
\end{split}
\end{equation}
where 
\[
B.T.^{0}\lesssim\|g\|^2_{L_{\al}^2[0,\pi]}.
\] 

Then from \eqref{3E1split} and \eqref{3E3split} and the conditions of $L_i^{+}(h)$ \eqref{RT2condition}, we could use the refined G\r{a}rding's inequality lemma \ref{Garding} to get \eqref{3linearcoro01}.

For \eqref{3linearcoro02} , the good sign $ \kappa(t)>0$ gives us the desired estimate. We have
\begin{equation}
\begin{split}
&\int_{0}^{\pi}g(\alpha)\overline{\lambda(\alpha)\frac{g(\alpha+\epsilon)-g(\al)}{\epsilon}}d\beta\\
=&\int_{0}^{\pi}\Re g(\alpha)\lambda(\alpha)\frac{\Re g(\alpha+\epsilon)-\Re g(\al)}{\epsilon}d\beta\\
&+\int_{0}^{\pi}\Im g(\alpha)\lambda(\alpha)\frac{\Im g(\alpha+\epsilon)-\Im g(\al)}{\epsilon}d\beta.
\end{split}
\end{equation}
Moreover,
\begin{align*}
&\int_{0}^{\pi}\Re g(\alpha)\lambda(\alpha)\frac{\Re g(\alpha+\epsilon)-\Re g(\al)}{\epsilon}d\beta\\
=&\frac{1}{\epsilon}\int_{0}^{\pi}\Re g(\alpha)\lambda(\alpha)\Re g(\alpha+\epsilon)d\beta-\frac{1}{2\epsilon}\int_{0}^{\pi} (\Re g(\alpha))^2\lambda(\alpha)d\al-\frac{1}{2\epsilon}\int_{0}^{\pi}(\Re g(\alpha+\epsilon))^2\lambda(\alpha+\epsilon)d\al\\
&-\frac{1}{2\epsilon}\int_{0}^{\epsilon}(\Re g(\alpha))^2\lambda(\alpha)d\al\\
=&-\frac{1}{\epsilon}\int_{0}^{\pi}\lambda(\alpha)(\Re g(\al)-\Re g(\al+\epsilon))^2d\al+\frac{1}{2\epsilon}\int_{0}^{\pi}(\Re g(\al+\epsilon))^2(\lambda(\alpha)-\lambda(\alpha+\epsilon))d\al\\
&-\frac{1}{\epsilon}\int_{0}^{\epsilon}(\Re g(\al))^2\lambda(\al)d\al\\
\leq &\int_{0}^{\pi}(\Re g(\alpha+\epsilon))^2\frac{|\lambda(\alpha)-\lambda(\alpha+\epsilon)|}{2\epsilon}d\alpha\\
\lesssim& \|\Re g(\alpha))\|_{L^2([0,\pi])}^{2}.
\end{align*}
Here we use the fact that $\supp \lambda \subset [0,\frac{\pi}{2}]$. The term of $\Im g$ can be treated in the same way.

Now we introduce several corollaries for later use.
\end{proof}
\begin{corollary}\label{301}
For any $12\geq k\geq 1$, $\kappa(t)>0$, if $g\in H^{k}_{\al}[0,\pi]\cap\{\supp g\subset[0,\frac{\pi}{4}]\}$, $h\in X^{k}$ with $h$, $L_1^{+}(h)$, $L_2^{+}(h)$ satisfying properties in \eqref{3GE01}, we have
\begin{equation}\label{3linearcoro0102}
\begin{split}
\Re &<g(\al), \lambda(\al)L_1^{+}(h)(\al,\gamma)\int_{0}^{2\al}\frac{(g(\al)-g(\beta))(\al-\beta)\ep}{((\al-\beta)^2+\ep^2)^2}d\beta\frac{2}{\pi}+\lambda(\al)L_1^{+}(h)(\al,\g)\sum_{j=0}^{k-1}b_{1,j}^{\ep}(\al)g^{<j>}(0)\frac{2}{\pi}\\
&+\lambda(\alpha)L_2^{+}(h)(\alpha,\g)\int_{0}^{2\alpha}\frac{(g(\alpha)-g(\beta))}{(\alpha-\beta)^2+\epsilon^2}d\beta -\lambda(\alpha)\sum_{j=0}^{k-1}b_{2,j}^{\epsilon}(\al)L_2^{+}(h)(\al,\g)g^{<j>}(0)>_{X^{k}}\\
&\lesssim \|g\|_{H^{k}_{\al}[0,\pi]}^{2}C(\|h\|_{X^{k}}).
\end{split}
\end{equation}
\begin{equation}\label{3linearcoro0202}
\begin{split}
\Re &<g,\lambda(\alpha)\frac{g(\alpha+\epsilon)-g(\alpha)}{\epsilon}\kappa(t)>_{H_{\al}^{k}[0,\pi]}\lesssim \|g\|_{H_{\al}^{k}[0,\pi]}^{2}C(\|h\|_{X^{k}}).
\end{split}
\end{equation}
\end{corollary}
\begin{proof}
We can use lemma \ref{3lemmaotherterms02} and lemma \ref{30101} to get the result. For \eqref{3linearcoro0102}, we can use lemma \ref{3lemmaotherterms02} to control the bounded terms, and \eqref{3linearcoro01} to control the cancellation of the higher-order term.
\eqref{3linearcoro0202} directly follows from \eqref{3linearcoro02}.
\end{proof}
When $g$ is sufficiently smooth,  we could let $\epsilon\to 0$ and have
\begin{corollary}\label{32}
For  $g\in H^{1}_{\al}[0,\pi]\cap\{\supp g\in[0,\frac{\pi}{2}]\}$, $h\in X^{1}$ with $L_1^{+}(h)$, $L_2^{+}(h)$ satisfying satisfying properties in \eqref{3GE01}, then we have for $k_1=0,1$, $\kappa(t)>0$,
\begin{equation}\label{M11estimateenergy}
\begin{split}
\Re &<g, \lambda(\alpha)g'(\alpha)\kappa(t)>_{H^{k_1}_{\alpha}[0,\pi]}\\
&\lesssim \|g\|_{H^{k_1}_{\alpha}[0,\pi]}^{2}C(\|h\|_{X^{1}}).
\end{split}
\end{equation}
\begin{equation}\label{L12estimateenergy}
\begin{split}
\Re &<g, \lambda(\al)L_1^{+}(h)(\al,\g)g'(\al,\gamma)+\lambda(\alpha)L_2^{+}(h)(\alpha,\gamma)p.v.\int_{0}^{2\alpha}\frac{(g(\alpha,\gamma)-g(\beta,\gamma))}{(\alpha-\beta)^2}d\beta>_{H^{k_1}_{\alpha}[0,\pi]}\\
&\lesssim \|g\|_{H^{k_1}_{\alpha}[0,\pi]}^{2}C(\|h\|_{X^{1}}).
\end{split}
\end{equation}
\end{corollary}
\begin{proof}
For \eqref{L12estimateenergy}, since  $b_{1,0}^{\epsilon}=b_{2,0}^{\epsilon}=0$ (\eqref{3bj1}, \eqref{3bj2}), we could take limit of \eqref{3linearcoro0102} in corollary \ref{301} when $k_1=1$, and of \eqref{3linearcoro01} in corollary \ref{30101} when $k_1=0$ . For \eqref{M11estimateenergy}, we could take limit of \eqref{3linearcoro0202} in corollary \ref{301} when $k_1=1$, and of \eqref{3linearcoro02} in corollary \ref{30101} when $k_1=0$.
\end{proof}

\subsubsection{The Picard theorem and the limit}
 Since there is no singularity when $\epsilon>0$, it is easy to see that when $k=12$, $T^\epsilon(h)$ is a Lipschitz map from $C_t^{0}{([0,t_{\epsilon}],X^{12})}$ to itself. From the Picard theorem, for any $\epsilon>0$, there exists $t_{\epsilon}$, $h^\ep\in C_{t}^{0}([0,t_{\ep}],X^{k})$ such that
\begin{equation}\label{picard1}
\begin{split}
&\frac{dh^{\ep}(\al,\g,t)}{dt}=T^{\ep}(h^{\ep})(\al,\g,t),\\
&h^{\ep}(\al,\g,0)=f^{+}(\al,0).
\end{split}
\end{equation}
We also have the integration equation
\begin{equation}\label{3picard2}
\begin{split}
&h^{\ep}(\al,\g,t)=f^{+}(\al,0)+\int_{0}^{t}T^{\ep}(h^{\ep})(\al,\g,s)ds.
\end{split}
\end{equation}
Moreover, by the energy estimate \eqref{3eneryestimate},   conditions of generalized equation \eqref{3GE01}, there exists a bound independent of $\epsilon$:
\begin{equation}\label{unitbound}
\|h^{\ep}(\al,\g,t)\|_{C_{t}([0,t_1],X^{12})}\lesssim 1 ,
\end{equation}
for some $t_1\leq t_s$.

Now we show that there exists $\ep_{n}$, such that $h^{\ep_{n}}$ converges strongly in $C_{t}^{0}([0,t_1], X^{11}).$ From the uniform bound in \eqref{unitbound}, we only need to show the strong convergence in $C_{t}^{0}([0,t_1], X^{1}).$ We first prove the following estimate:
%The main obstacle appears when we bound the difference of $h(\al,\g)-h(\al,\g')$. The best estimate we could get is the $H_{\al}^{k-1}$, but $h^{k-1}(0,\gamma)$ shown in \eqref{3bj1} and \eqref{3bj2} is not bounded by $\|h\|_{H_{\al}^{k-1}}$. 
%We will use the fact that $\al b_{j,1}^{\ep}$ and $b_{j,2}^{\ep}$ strongly tends to 0 in $H^1_{\al}[0,\pi]$ and Arzela-Ascoli theorem to show the convergence. 
\begin{lemma}
For $\gamma, \gamma' \in[-1,1],$ $t\in[0,t_1]$, when $k=12$, we have
\begin{equation}\label{3gammabound}
\Re<h^{\ep}(\al,\g)-h^{\ep}(\al,\g'),T^{\ep}(h^{\ep})(\al,\g)-T^{\ep}(h^{\ep})(\al,\g')>_{H_{\al}^{1}[0,\pi]}\lesssim |\g-\g'|^2+ O(\ep)+\|h^{\ep}(\al,\g)-h^{\ep}(\al,\g')\|_{H_{\al}^{1}[0,\pi]}^2,
\end{equation}
with $\lesssim$ does not depend on $\ep$, $\g$, $\g'$ or t.
\end{lemma}
\begin{proof}
 From the equation \eqref{perturbationterm}, when $k=12$, we have
\begin{equation}\label{Tsplit}
\begin{split}
&\quad T^{\ep}(h^{\ep})(\al,\g)-T^{\ep}(h^{\ep})(\al,\g')\\
&=\tilde{M}_{1,1}^{\ep}(h^{\ep})(\al,\g)-\tilde{M}_{1,1}^{\ep}(h^{\ep})(\al,\g')\\
&\quad+\tilde{M}_{1,2}^{12,\ep}(h^{\ep})(\al,\g)-\tilde{M}_{1,2}^{12,\ep}(h^{\ep})(\al,\g')\\
&\quad+\tilde{M}_{2,1}^{12,\ep}(h^{\ep})(\al,\g)-\tilde{M}_{2,1}^{12,\ep}(h^{\ep})(\al,\g')\\
&\quad+\sum_{i}(B_i(h^{\ep})(\al,\g)-B_i(h^{\ep})(\al,\g')).
\end{split}
\end{equation}
We can use the notation in \eqref{3M12split} \eqref{3M21split} and have
\begin{align}\label{Msplitsum}
&\quad \tilde{M}_{1,2}^{12,\ep}(h^{\ep})(\al,\g)-\tilde{M}_{1,2}^{12,\ep}(h^{\ep})(\al,\g')
+\tilde{M}_{2,1}^{12,\ep}(h^{\ep})(\al,\g)-\tilde{M}_{2,1}^{12,\ep}(h^{\ep})(\al,\g')\\\nonumber
&=\tilde{M}_{1,2,I}^{\ep}(h^{\ep})(\al,\g)-\tilde{M}_{1,2,I}^{\ep}(h^{\ep})(\al,\g')
+\tilde{M}_{2,1,I}^{\ep}(h^{\ep})(\al,\g)-\tilde{M}_{2,1,I}^{\ep}(h^{\ep})(\al,\g')\\\nonumber
&\quad-\underbrace{\lambda(\alpha)\sum_{j=0}^{11}b_{1,j}^{\epsilon}(\al)((L_1^{+}(h^{\ep})(\al,\g)\frac{2}{\pi}h^{\ep,(j)}(0,\g))-(L_1^{+}(h^{\ep})(\al,\g')\frac{2}{\pi}h^{\ep,(j)}(0,\g')))}_{Term_2}\\\nonumber
&\quad-\underbrace{\lambda(\alpha)\sum_{j=0}^{11}b_{2,j}^{\epsilon}(\al)(L_2^{+}(h^{\ep})(\al,\g)h^{\ep,(j)}(0,\g)-L_2^{+}(h^{\ep})(\al,\g')h^{\ep,(j)}(0,\g'))}_{Term_3}.
\end{align}
From the space of $h^{\ep}$ \eqref{unitbound}, we have for $j\leq 11$, $|h^{\ep,(j)}(0,\g)|\lesssim 1$. Then from estimates \eqref{3b1g0}, \eqref{3b2g0}, and vanishing condition of $L_i^{+}(h)$, by taking $g$ as $\frac{\alpha^{j}}{j!}$, we have
\[
\|Term_{2}\|_{H_{\al}^{1}[0,\pi]}+\|Term_{3}\|_{H_{\al}^{1}[0,\pi]}\lesssim O(\ep).
\]
Then from the equation \eqref{Tsplit}, \eqref{Msplitsum}, by using the conditions of $L_i^{+}(h)$, $B_i(h)$ in \eqref{3GE01} and the uniform bound of $h$ in $X^{12}$ \eqref{unitbound}, and  we have
\begin{equation}\label{Tdiffequation}
\begin{split}
&T^{\ep}(h^{\ep})(\al,\g)-T^{\ep}(h^{\ep})(\al,\g')=\lambda(\al)\kappa(t)\frac{h^{\ep}(\al+\ep,\g)-h^{\ep}(\al+\ep,\g')-(h^{\ep}(\al,\g)-h^{\ep}(\al,\g'))}{\ep}\\
&\quad+\lambda(\al)\int_{0}^{2\al}\frac{(h^{\ep}(\al,\g)-h^{\ep}(\al,\g'))-(h^{\ep}(\beta,\g)-h^{\ep}(\beta,\g'))(\al-\beta)\ep}{((\al-\beta)^2+\ep^2)^2}d\beta L_1^{+}(h^{\ep})(\al,\g)\frac{2}{\pi}\\
&\quad+\lambda(\al)\int_{0}^{2\al}\frac{(h^{\ep}(\al,\g)-h^{\ep}(\al,\g'))-(h^{\ep}(\beta,\g)-h^{\ep}(\beta,\g'))}{(\al-\beta)^2+\ep^2}d\beta L_2^{+}(h^{
\ep})(\al,\g)\\
&\quad+\lambda(\al)\int_{0}^{2\al}\frac{(h^{\ep}(\al,\g')-h^{\ep}(\beta,\g'))(\al-\beta)\ep}{((\al-\beta)^2+\ep^2)^2}d\beta (L_1^{+}(h^{\ep})(\al,\g)-L_1^{+}(h^{\ep})(\al,\g'))\frac{2}{\pi}\\
&\quad+\lambda(\al)\int_{0}^{2\al}\frac{(h^{\ep}(\al,\g')-h^{\ep}(\beta,\g'))}{(\al-\beta)^2+\ep^2}d\beta (L_2^{+}(h^{
\ep})(\al,\g)-L_2^{+}(h^{
\ep})(\al,\g'))\\
&\quad+B.T_{\g},\\
&=T_{\g,1}+T_{\g,2}+T_{\g,3}+T_{\g,4}+T_{\g,5}+B.T_{\g}.
\end{split}
\end{equation}
where $\|B.T_{\g}\|_{H_{\al}^{1}[0,\pi]}\lesssim O(\ep)+|\g-\g'|+\|h^{\ep}(\al,\g)-h^{\ep}(\al,\g')\|_{H_{\al}^1[0,\pi]}.$

From corollary \ref{coro22k01} and lemma \ref{coro22ep}, \ref{coro22epm12}, and the vanishing and smoothness conditions of $L_i^{+}(h)$ in \eqref{3GE01}, we have
\begin{align}\label{Tg45es}
&\quad\|T_{\g,4}+T_{\g,5}\|_{H_{\al}^1[0,\pi]}\\\nonumber
&\lesssim \|\frac{L_1^{+}(h^{\ep})(\al,\g)-L_1^{+}(h^{\ep})(\al,\g')}{\alpha}\|_{C_{\al}^1[0,\frac{\pi}{4}]}+\|\frac{L_2^{+}(h^{\ep})(\al,\g)-L_2^{+}(h^{\ep})(\al,\g')}{\alpha}\|_{C_{\al}^1[0,\frac{\pi}{4}]}\\\nonumber
&\lesssim \|L_1^{+}(h^{\ep})(\al,\g)-L_1^{+}(h^{\ep})(\al,\g')\|_{C_{\al}^2[0,\frac{\pi}{4}]}+\|L_2^{+}(h^{\ep})(\al,\g)-L_2^{+}(h^{\ep})(\al,\g')\|_{C_{\al}^2[0,\frac{\pi}{4}]}\\\nonumber
&\lesssim|\g-\g'|+\|h^{\ep}(\al,\g)-h^{\ep}(\al,\g')\|_{H_{\al}^1[0,\pi]}.
\end{align}

Since $b_{1,0}^{\ep}=b_{2,0}^{\ep}=0$ from \eqref{3bj1}, \eqref{3bj2}, by taking $k=1$ in the estimate \eqref{3linearcoro0102} in corollary \ref{301}, we have 
\begin{align}\label{Tg3es}
\quad\Re<h^{\ep}(\al,\g)-h^{\ep}(\al,\g'),T_{\g,2}+T_{\g,3}>_{H_{\al}^{1}[0,\pi]}\lesssim &\|h^{\ep}(\al,\g)-h^{\ep}(\al,\g')\|_{H_{\al}^1[0,\pi]}^{2}.
\end{align}
By using estimate \eqref{3linearcoro0202}, we have

\begin{align}\label{Tg1es}
\quad\Re<h^{\ep}(\al,\g)-h^{\ep}(\al,\g'),T_{\g,1}>_{H_{\al}^{1}[0,\pi]}\lesssim &\|h^{\ep}(\al,\g)-h^{\ep}(\al,\g')\|_{H_{\al}^1[0,\pi]}^{2}.
\end{align}
Combining the estimates \eqref{Tdiffequation}, \eqref{Tg45es}, \eqref{Tg3es}, \eqref{Tg1es}, we have the bound.
\end{proof}
 Therefore by using esimate \eqref{3gammabound}, from Gronwall's inequality, and $h^{\ep}(\al,\g)-h^{\ep}(\al,\g')|_{t=0}=0$ \eqref{picard1}, we have for $0\leq t\leq t_1$,
\begin{equation}\label{3hgammabound}
\|h^{\ep}(\al,\g)-h^{\ep}(\al,\g')\|_{H_{\al}^{1}[0,\pi]}\lesssim O(\ep)+|\gamma-\gamma'|.
\end{equation}
Moreover, from \eqref{3boundedlower}, \eqref{3boundedB}, we have
\begin{equation}\label{3tbound}
\|T^{\epsilon}(h^{\ep})(\al,\g,t)\|_{X^{1}}\lesssim 1.
\end{equation}
Combining \eqref{3picard2}, and \eqref{3hgammabound}, we have
\begin{equation}\label{3gammatdiff}
 \|h^{\ep}(\al,\g,t)-h^{\ep}(\al,\g',t')\|_{H_{\al}^{1}[0,\pi]}\lesssim |\g-\g'|+|t-t'|+O(\ep).
\end{equation}

We could use the Arzela-Ascoli theorem to show the convergence. In fact,
for each fixed $t\in [0,t_1], \g \in [-1,1]$, by the weak compactness, and Sobolev's embedding theorem, there exists $\ep_{n}^{t,\g}$, and $h(\al,\g,t)\in H_{\al}^{12}[0,\pi]$, with
\[
\lim_{n\to \infty}\epsilon_{n}^{t,\gamma}=0,
\]
\[
\|h^{\ep_{n}^{t,\g}}(\cdot,\g,t)-h(\cdot,\g,t)\|_{H_{\al}^{1}[0,\pi]}\to 0.
\]
Then there exists a sequence $\ep_{n}$ such that for a countable dense set ${\g_i}\subset [-1,1]$, ${t_j} \subset [0,t_1]$, we have \[
\lim_{n\to\infty}\|h^{\epsilon_{n}}(\al,\g_i,t_j)-h(\al,\g_i,t_j)\|_{H_{\al}^{1}[0,\pi]}= 0.
\]
From \eqref{3gammatdiff}, we also have
\[
\lim_{n\to \infty}\sup_{\gamma\in [-1,1],t\in [0,t_1]}\|h^{\epsilon_{n}}(\al,\g,t)-h(\al,\g,t)\|_{H_{\al}^{1}[0,\pi]}= 0.
\]
From the uniform bound \eqref{unitbound}, the interpolation theorem gives 
 \[
\lim_{n\to \infty}\sup_{\gamma\in [-1,1],t\in [0,t_1]}\|h^{\epsilon_{n}}(\al,\g,t)-h(\al,\g,t)\|_{H_{\al}^{11}[0,\pi]}= 0.
\] 
We then have 
\begin{equation}\label{limitspace01}
h(\al,\g,t)\in C_{t}^{0}([0,t_1], X^{11})
\end{equation}
and the uniform convergence.

Hence from \eqref{perturbationterm}, and lemma \ref{limitlemma01}, we could take the limit of \eqref{3picard2} and have 
\begin{equation}\label{3E2new}
\begin{split}
&h(\al,\g,t)=f^{+}(\al,0)+\int_{0}^{t}T(h)(\al,\g,s)ds.
\end{split}
\end{equation}
From lemma \ref{323}, and the space of $h$ \eqref{limitspace01}, we have
\[
\|T(h)(\al,\g,t)\|_{C_{t}^{0}([0,t_1], X^{10})}\lesssim C(\|h(\al,\g,t)\|_{C_{t}^{0}([0,t_1], X^{11})}).
\]
Therefore
\begin{equation}\label{3space1}
h(\al,\g,t)\in C_{t}^{1}([0,t_0], X^{10})\cap L_{t}^{\infty}([0,t_1], L_{\gamma}^{\infty}([-1,1],H_{\al}^{12}[0,\pi])).
\end{equation}
\subsection{The differentiability with respect to $\g$}
Now we show $h(\al,\g,t)$ is differentiable with respect to $\g$. We first show the Lipschitz condition :
\begin{lemma}
We have
\begin{equation}\label{3differenbound1}
\|h(\al,\g,t)-h(\al,\g',t)\|_{H_{\al}^{9}[0,\pi]}\lesssim |\g-\g'|.
\end{equation}
\end{lemma}
\begin{proof}
From equation \eqref{3GE01}, we get
\begin{equation}
\begin{split}
&\quad\frac{d}{dt}(h(\alpha,\g)-h(\alpha,\g'))\\
&=M_{1,1}(h)(\al,\g)-M_{1,1}(h)(\al,\g')+M_{1,2}(h)(\al,\g)-M_{1,2}(h)(\al,\g')\\
&\quad+M_{2,1}(h)(\al,\g)-M_{2,1}(h)(\al,\g')+\sum_{i}(B_i(h)(\al,\g)-B_i(h)(\al,\g')).
\end{split}
\end{equation}
From the conditions of $B_i(h)$, $L_i^{+}(h)$ in \eqref{3GE01}, we have
\begin{equation}
\begin{split}
&\frac{d}{dt}(h(\al,\g)-h(\al,\g'))\\
&=\lambda(\al)\kappa(t)\partial_{\al}(h(\al,\g)-h(\al,\g'))\\
&\quad+\lambda(\al)L_1^{+}(h)(\al,\g)\frac{\pi}{2}\partial_{\al}(h(\al,\g)-h(\al,\g'))\\
&\quad+\lambda(\al)L_2^{+}(h)(\al,\g)p.v.\int_{0}^{2\al}\frac{(h(\al,\g)-h(\al,\g))-(h(\beta,\g)-h(\beta,\g'))}{(\al-\beta)^2}d\beta\\
&\quad+BT_{\g},
\end{split}
\end{equation}
where 
\[
\|BT_{\g}\|_{H_{\al}^{9}[0,\pi]}\lesssim |\g-\g'|+\|h(\al,\g)-h(\al,\g')\|_{H_{\al}^{9}[0,\pi]}.
\]
Then from lemma \ref{3coro22}, we have
\begin{equation}
\begin{split}
&\quad\Re<h(\al,\g)-h(\al,\g'),\frac{d}{dt}(h(\al,\g)-h(\al,\g'))>_{H_{\al}^{9}[0,\pi]}\\
&\leq <h^{(9)}(\al,\g)-h^{(9)}(\al,\g'),\\
&\quad\lambda(\al)(h^{(10)}(\al,\g)-h^{(10)}(\al,\g'))\kappa(t)+\lambda(\al)L_1^{+}(h)(\al,\g)(h^{(10)}(\al,\g)-h^{(10)}(\al,\g'))\\
&\quad+\lambda(\al)L_2^{+}(h)(\al,\g) p.v.\int_{0}^{2\al}\frac{(h^{(9)}(\al,\g)-h^{(9)}(\al,\g'))-(h^{(9)}(\beta,\g)-h^{(9)}(\beta,\g'))}{(\al-\beta)^2}d\beta>_{L_{\al}^{2}[0,\pi]}\\
&\quad+C\|h(\al,\g)-h(\al,\g')\|_{H_{\al}^{9}[0,\pi]}^2+C|\g-\g'|^2.
\end{split}
\end{equation}
Therefore, from corollary \ref{32} with $k_1=0$, we have
\[
\Re<h(\al,\g)-h(\al,\g'),\frac{d}{dt}(h(\al,\g)-h(\al,\g'))>_{H_{\al}^{9}[0,\pi]}\lesssim \|h(\al,\g)-h(\al,\g')\|_{H_{\al}^{9}[0,\pi]}^2+|\g-\g'|^2.
\]
Finally, from \eqref{3g0}, $h(\al,\g)-h(\al,\g')|_{t=0}=0$, the Gronwall's inequality tells us \eqref{3differenbound1}.
\end{proof}
Then we deal with the differentiability.
Let $w(\al,\g,t)$  be the solution of the following equation
\begin{equation}\label{3eqw}
\begin{split}
\frac{dw(\al,\g)}{dt}&=\lambda(\al)w'(\al,\g)k(t)+\lambda(\al)L_1^{+}(h)(\al,\g)w'(\al,\g)+\lambda(\al)L_2^{+}(h)(\al,\g)p.v.\int_{0}^{2\al}\frac{w(\al,\g)-w(\beta,\g)}{(\al-\beta)^2}d\beta\\
&\quad+\lambda(\al)(D_{\g}L_1^{+}(h))[w](\al,\g)h'(\al,\g)+\lambda(\al)(D_{\g}L_2^{+}(h))[w](\al,\g)p.v.\int_{0}^{2\al}\frac{h(\al,\g)-h(\beta,\g)}{(\al-\beta)^2}d\beta\\
&\quad+\sum_{i}D_{\g}B_i[w](\al,\g),\\
&=\lambda(\al)w'(\al,\g)k(t)+\lambda(\al)L_1^{+}(h)(\al,\g)w'(\al,\g)+\lambda(\al)L_2^{+}(h)(\al,\g)p.v.\int_{0}^{2\al}\frac{w(\al,\g)-w(\beta,\g)}{(\al-\beta)^2}d\beta\\
&\quad+\sum_{j}\tilde{\tilde{B}}_{j}(w)(\al,\gamma),
\end{split}
\end{equation}
with $w(\al,\g)|_{t=0}=0$.
From the space of $h$ \eqref{3space1}, and condition of $L_i^{+}(h)$, $B_i(h)$ in \eqref{3GE01},  we claim in space $X^{6}$  we could repeat the similar process of the existence of $h$ and have for some $0<t_2\leq t_1$, 
\begin{equation}\label{3spacew}
w(\al,\g,t)\in C_{t}^{1}([0,t_2], X^{4})\cap L_{t}^{\infty}([0,t_2], L_{\g}^{\infty}([-1,1],H_{\al}^{6}[0,\pi])).
\end{equation}
 
 Let $v(\al,\g,t)=\frac{h(\g)-h(\g')}{\g-\g'}-w(\g)$. In order to get the differentiability, We want to show
 \begin{lemma}\label{differgamma01}
\begin{equation}\label{3diffgammalimit}
\lim_{\g\to\g'}\|v(\al,\g,\g')\|_{H_{\al}^{1}[0,\pi]}=0.
\end{equation}
\end{lemma}
\begin{proof}
Since we have
\begin{equation}
\begin{split}
&\quad \frac{L_1^{+}(h)(\al,\g)h'(\al,\g)-L_1^{+}(h)(\al,\g')h'(\al,\g')}{\g-\g'}-L_1^{+}(h)(\al,\g)w'(\al,\gamma)-D_{\g}L_1^{+}[w](\al,\g)h'(\al,\g)\\
&=L_1^{+}(h)(\al,\g)[\frac{h'(\al,\g)-h'(\al,\g')}{\g-\g'}-w'(\al, \g)]+(\frac{L_1^{+}(h)(\al,\g)-L_1^{+}(h)(\al,\g')}{\g-\g'}-D_{\g}L_1^{+}[w](\al,\g))h'(\al,\g')\\
&\quad+D_{\g}L_{1}^{+}[w](\al,\g)(h'(\al,\g')-h'(\al,\g)).
\end{split}
\end{equation}
and
\begin{equation}
\begin{split}
&\quad\frac{L_2^{+}(h)(\al,\g)p.v.\int_{0}^{2\al}\frac{h(\al,\g)-h(\beta,\g)}{(\al-\beta)^2}d\beta-L_2^{+}(h)(\al,\g')p.v.\int_{0}^{2\al}\frac{h(\al,\g')-h(\beta,\g')}{(\al-\beta)^2}d\beta}{\g-\g'}\\
&\qquad-L_2^{+}(h)(\al,\g)p.v.\int_{0}^{2\al}\frac{w(\al,\g)-w(\beta,\g)}{(\al-\beta)^2}d\beta-D_{\g}L_{2}^{+}[w](\al,\g)p.v.\int_{0}^{2\al}\frac{h(\al,\g)-h(\beta,\g)}{(\al-\beta)^2}d\beta\\
&=L_2^{+}(h)(\al,\g)p.v.\int_{0}^{2\al}\frac{(\frac{h(\al,\g)-h(\al,\g')}{\g-\g'}-w(\al,\g))-(\frac{h(\beta,\g)-h(\beta,\g')}{\g-\g'}-w(\beta,\g))}{(\al-\beta)^2}d\beta\\
&\quad+(\frac{L_2^{+}(h)(\al,\g)-L_2^{+}(h)(\al,\g')}{\g-\g'}-D_{\g}L_{2}^{+}[w](\al,\g))p.v.\int_{0}^{2\al}\frac{h(\al,\g')-h(\beta,\g')}{(\al-\beta)^2}d\beta\\
&\quad+D_{\g}L_{2}^{+}[w](\al,\g)p.v.\int_{0}^{2\al}\frac{(h(\al,\g')-h(\al,\g))-(h(\beta,\g')-h(\beta,\g))}{(\al-\beta)^2}d\beta.
\end{split}
\end{equation}
From \eqref{3GE01} and \eqref{3eqw}, we have
\begin{equation}
    \begin{split}
    \frac{d}{dt}v(\al,\g,\g')=&\lambda(\al)\kappa(t)\pa_{\al}v(\al,\g,\g')+\lambda(\al)L_1^{+}(h)(\al,\g)\pa_{\al}v(\al,\g,\g')\\
    &+\lambda(\al)p.v.\int_{0}^{2\al}\frac{v(\al,\g,\g')-v(\beta,\g,\g')}{(\al-\beta)^2}d\beta L_2^{+}(h)(\al,\g)\\
    &+\lambda(\al)(\frac{L_1^{+}(h)(\al,\g)-L_1^{+}(h)(\al,\g')}{\g-\g'}-D_{\g}L_{1}^{+}[w](\al,\g))h'(\al,\g')\\
    &+\lambda(\al)(\frac{L_2^{+}(h)(\al,\g)-L_2^{+}(h)(\al,\g')}{\g-\g'}-D_{\g}L_{2}^{+}[w](\al,\g))p.v.\int_{0}^{2\al}\frac{h(\al,\g')-h(\beta,\g')}{(\al-\beta)^2}d\beta\\
&+\lambda(\al)D_{\g}L_{1}^{+}[w](\al,\g)(h'(\al,\g')-h'(\al,\g))\\
&+\lambda(\al)D_{\g}L_{2}^{+}[w](\al,\g)p.v.\int_{0}^{2\al}\frac{(h(\al,\g')-h(\al,\g))-(h(\beta,\g')-h(\beta,\g))}{(\al-\beta)^2}d\beta\\
&+\sum_{j}[\frac{B_j(h)(\al,\g)-B_j(h)(\al,\g')}{\g-\g'}-D_{\g}B_{j}[w](\al,\g)].
    \end{split}
\end{equation}
Then from the properties of $L_i^{+}(h)$, $B_i(h)$ in \eqref{3GE01}, and the space of $h$, $w$ \eqref{3space1}, \eqref{3spacew}, the Lipschitz condition of $h$ with respect to $\g$ \eqref{3differenbound1}, and corollary \ref{3coro22}, we have
\begin{equation}
    \begin{split}
    <v(\al,\g,\g'),\frac{d}{dt} v(\al,\g,\g')>_{H_{\al}^{1}[0,\pi]}\leq&\\
    &<v(\al,\g,\g'),\lambda(\al)\kappa(t)v(\al,\g,\g')+\lambda(\al)L_1^{+}(h)(\al,\g)\pa_{\al}v(\al,\g,\g')\\
    &+\lambda(\al)p.v.\int_{0}^{2\al}\frac{v(\al,\g,\g')-v(\beta,\g,\g')}{(\al-\beta)^2}d\beta L_2^{+}(h)(\al,\g)>_{H_{\al}^{1}[0,\pi]}\\
    &+C \|v(\al,\g,\g')\|^2_{H_{\al}^{1}[0,\pi]}+\mathcal{O}(|\g-\g'|).
    \end{split}
\end{equation}
From corollary \ref{32}, we have
\begin{equation}
    <v(\al,\g,\g'),\frac{d}{dt}v(\al,\g,\g')>_{H_{\al}^{1}[0,\pi]}\lesssim \mathcal{O}|\g-\g'|+\|v(\al,\g,\g')\|^2_{H_{\al}^{1}[0,\pi]}.
\end{equation}
Then \eqref{3diffgammalimit} follows from Gronwall's inequality and 
initial value $v(\al,\g,\g')|_{t=0}=0$ (from \eqref{3g0}).
\end{proof}
Then by using the interpolation theorem, from \eqref{3differenbound1} and \eqref{3spacew}, we have
\begin{align*}
\lim_{\g\to\g'}\|v(\al,\g,\g')\|_{H_{\al}^{5}[0,\pi]=0}.
\end{align*}
Moreover, from the equation of $h$ \eqref{3GE01}, and equation of $w$ \eqref{3eqw}, we have
\[
\frac{d}{dt}w=\frac{d}{dt}\frac{d}{d\g}h=\frac{d}{d\g}\frac{d}{dt}h.
\]
In conclusion, we have
\begin{align}\label{hspaceconclusion01}
\begin{cases}
    &h(\al,\g,t)\in C_{t}^1([0,t_1], X^{9}), \frac{d}{d\g}h\in C_{t}^{1}([0,t_2], X^{4}),\\
    &h(\al,\g,t)\in C_{\g}^{1}([-1,1], H_{\al}^{5}[0,\pi]),\\
    &\frac{d}{d\g}\frac{d}{dt}h=\frac{d}{dt}\frac{d}{d\g}h.
\end{cases}
\end{align}
Thus \eqref{hspaceconclusion} is verified.
\subsection{Uniqueness}\label{huni}
In this section, we show \eqref{huniqueness01}: for all $g(\al,t)\in C_{t}^{1}([0,t_2],H^{5}_{\alpha}[0,\pi]))\cap\{h|\supp h \subset  [0,\frac{\pi}{4}]\}$ satisfying the equation \eqref{3GE01} when $\gamma=0$ and the initial data condition $g(\al,0)= h(\al,\gamma,0)$, we have
\[
g(\al,t)=h(\al,\g,t).
\]
 
 Since both $h(\al,0,t)$, $g(\al,t)$ satisfies the equation \eqref{3GE01}, we have
 \begin{equation}
\begin{split}
   \frac{dh(\al,0)}{dt}&=T^{+}(h)(\al,0)\\
   &=M_{1,1}(h)(\al,0)+M_{1,2}(h)(\al,0)+M_{2,1}(h)(\al,0)+\sum_{i}B_i(h)(\al,0),
\end{split}
\end{equation}
 \begin{equation}
\begin{split}
   \frac{dg(\al)}{dt}&=T^{+}(g)(\al,0)\\
   &=M_{1,1}(g)(\al,0)+M_{1,2}(g)(\al,0)+M_{2,1}(g)(\al,0)+\sum_{i}B_i(g)(\al,0).
\end{split}
\end{equation}
When $t=0$, $g(\al,0)-h(\al,0)|_{t=0}=0$. It is sufficient to show 
\begin{equation}\label{3unibound1}
\begin{split}
    \Re<h(\al,0)-g(\al), T(h)(\al,0)-T(g)(\al,0)>_{H_{\al}^{1}[0,\pi]}\lesssim \|g(\al)-h(\al,0)\|_{H_{\al}^{1}[0,\pi]}^2.
\end{split}
\end{equation}
We have 
\begin{equation}
    \begin{split}
    &\quad\frac{d(h(\al,0)-g(\al))}{dt}\\
    &=\lambda(\al)\kappa(t)(h'(\al,0)-g'(\al))+\lambda(\al)L_1^{+}(h)(\al,0)(h'(\al,0)-g'(\al))\\
    &\quad+\lambda(\al)L_2^{+}(h)(\al,0)\int_{0}^{2\al}\frac{(h(\al,0)-g(\al))-(h(\beta,0)-g(\beta))}{(\al-\beta)^2}d\beta\\
    &\quad+\lambda(\al)(L_1^{+}(h)(\al,0)-L_1^{+}(g)(\al,0))g'(\al)\\
    &\quad+\lambda(\al)(L_2^{+}(h)(\al,0)-L_2^{+}(g)(\al,0))\int_{0}^{2\al}\frac{g(\al)-g(\beta)}{(\al-\beta)^2}d\beta\\
    &\quad+\sum_{i}(B_i(h)(\al,0)-B_i(g)(\al,0)).
    \end{split}
\end{equation}
Then from conditions of $B_{i}(h)$, $L_{i}^{+}(h)$ in \eqref{3GE01}, we have
\begin{equation}
    \begin{split}
        &\quad\Re<h(\al,0)-g(\al), \frac{d}{dt}(h(\al,0)-g(\al))>_{H^1_{\al}[0,\pi]}\\
        &\leq \Re <h(\al,0)-g(\al),\lambda(\al)\kappa(t)(h'(\al,0)-g'(\al))+\lambda(\al)L_1^{+}(h)(\al,0)(h'(\al,0)-g'(\al))\\
        &\quad+\lambda(\al)L_2^{+}(h)(\al,0)\int_{0}^{2\al}\frac{(h(\al,0)-g(\al))-(h(\beta,0)-g(\beta))}{(\al-\beta)^2}d\beta>_{H^1_{\al}[0,\pi]}+C\|h(\al,0)-g(\al)\|_{H^1_{\al}[0,\pi]}^2.
    \end{split}
\end{equation}
We could use corollary \ref{32} to get \eqref{3unibound1}.

\section{Analyticity when $\kappa(t)>0$}\label{analyticity+section}
From the conditions of \eqref{3GE01}, previous theorems \ref{theoremgeneralizedequa}, \ref{existence1theorem}, we have the existence of solution $h$ to equation \eqref{modifiedequation} : there exists  $\bar{t}>0$ such that $h$ satisfies
\begin{align}\label{hspaceconclusion0}
\begin{cases}
    &h(\al,\g,t)\in C_{t}^1([0,\bar{t}], X^{5}), \frac{d}{d\g}h\in C_{t}^{1}([0,\bar{t}], X^{4}),\\
    &h(\al,\g,t)\in C_{\g}^{1}([-1,1], C_{\al}^{2}[0,\pi]),\\
    &\frac{d}{d\g}\frac{d}{dt}h=\frac{d}{dt}\frac{d}{d\g}h.
\end{cases}
\end{align}
The initial condition: 
\begin{equation}\label{initialdatacondition01}
h(\alpha,\gamma,0)=\tilde{f}^{+(60)}(\alpha,0).
\end{equation}
Moreover, from our derivation of \eqref{modifiedcondition}, $\tilde{f}^{+(60)}(\alpha,t)$ is also a equation of $\eqref{modifiedequation}$ when $\gamma=0$, $\kappa(t)>0$. Then from the uniqueness in theorem, \ref{existence1theorem}, we have for $t\leq \bar{t}$, 
\begin{equation}\label{3uniqh}
    h(\al,0,t)=\tilde{f}^{+(60)}(\al,t).
\end{equation}

Now we show the analyticity:
\begin{lemma}\label{analyticitylemmao1}
When $0\leq t\leq t_1$, $h(\alpha,0,t)$  can be analytically extended to region $\{(\alpha+iy)|0\leq\alpha\leq \frac{\pi}{4}, |y|\leq c(\alpha)t\}.$
\end{lemma}
\begin{proof}
We plan to use a similar way as in \cite{muskatregulatiryJia} by verifying Cauchy-Riemann condition. Let 
\[
A(h)=\frac{ic(\alpha)t}{1+ic'(\alpha)\gamma t}\partial_{\alpha}h-\partial_{\gamma}h.
\]
From lemma 9.1 in \cite{muskatregulatiryJia}, it is enough to show $\|A(h)\|_{X^1}=0.$
From the definition of $c(\alpha)$ \eqref{cdefi} and initial data \eqref{initialdatacondition01}, we get $\|A(h)\|_{X^1}|_{t=0}=0$. Then we can use the next lemma and Gronwall's inequality to get the result.
\end{proof}
\begin{lemma}\label{analyenergyestimate}
When $\kappa(t)>0$, we have the following estimate for $h$ satisfies \eqref{extendedequation01}, \eqref{hspaceconclusion0}, \eqref{initialdatacondition01}, \eqref{3uniqh}:
\[
\Re<A(h), \frac{dA(h)}{dt}>_{X^{1}}\lesssim \|A(h)\|^{2}_{X^{1}}.
\]
We remark that the proof till \eqref{analyticityestimateG} is not dependent on the sign of $\kappa(t)$. 

\end{lemma}
\begin{proof}
From \eqref{defkappa}, we have $\tau(t)=0$. Then by  \eqref{extendedequation01},  when  $0<\al<\pi$, the solution $h$ satisfies
\begin{align}\label{modifiedequation2}
    \frac{dh(\alpha,\gamma,t)}{dt}=T(h)=
       M_{1,1}(h)+M_{1,2}(h)+F_2(h)+\sum_{i}O^{i}(h)+\lambda(\alpha)T_{fixed}(\alpha_{\gamma}^{t}) 
\end{align}
with the initial value $h(\alpha,\gamma,0)=\tilde{f}^{+(60)}(\alpha,0)$.

First, we combine two main terms $M_{1,1}$, $M_{1,2}$ in \eqref{extendedequation01} and have:
\begin{align}\label{3anasplit}
&\quad M_{1,1}(h)+M_{1,2}(h)\\\nonumber
&=\bigg(\frac{ic(\al)\g}{1+ic'(\al)\g t}\frac{dh(\al,\g,t)}{d\al}\bigg)\\\nonumber
&\quad+\bigg(\lambda(\alpha_{\gamma}^{t})\frac{\frac{dh(\al,\g,t)}{d\al}}{1+ic'(\al)\g t}(\kappa(t)+\frac{\partial_{t}x}{\partial_{\alpha}x}(\alpha_{\gamma}^{t})  -\frac{\partial_{t}x}{\partial_{\alpha}x}(0)\\\nonumber&\quad+\int_{-\pi}^{\pi}K(D^{-60}(h)(\alpha,\gamma, t)-D^{-60}(h)(\beta,\gamma, t)+\tilde{f}^{L}(\alpha_{\gamma}^{t})-\tilde{f}^{L}(\beta_{\gamma}^{t}))(\frac{1}{(\partial_{\alpha}x)(\alpha_{\gamma}^{t})}-\frac{1}{(\partial_{\beta}x)(\beta_{\gamma}^{t})})\frac{dx(\beta_{\gamma}^{t})}{d\beta}d\beta\\\nonumber
  &\quad-\frac{1}{(\partial_{\alpha}x)(0)}p.v.\int_{-\pi}^{\pi}K(D^{-60}(h)(0,\gamma, t)-D^{-60}(h)(\beta,\gamma, t)+\tilde{f}^{L}(0)-\tilde{f}^{L}(\beta_{\gamma}^{t}))\frac{dx(\beta_{\gamma}^{t})}{d\beta}d\beta)\bigg)\\\nonumber
  &\quad+\bigg(\lambda(\alpha_{\gamma}^{t})\frac{\frac{dh(\al,\g,t)}{d\al}}{1+ic'(\al)\g t}p.v.\int_{-\pi}^{\pi}K(D^{-60}(h)(\alpha,\gamma, t)-D^{-60}(h)(\beta,\gamma, t)+\tilde{f}^{L}(\alpha_{\gamma}^{t})-\tilde{f}^{L}(\beta_{\gamma}^{t}))(1+ic'(\beta)\gamma t)d\beta\bigg)\\\nonumber
  &=N_1(h)+N_2(h)+N_3(h).
\end{align}
    Here we use the fact that $\lambda(\alpha_{\gamma}^{t})=\lambda(\alpha+\tau(t))=1$, $c(\alpha)$ is supported in where $\lambda(\alpha)=1$ and $\tau(t)=0$ since $\kappa(t)>0$.

    From lemma \ref{switch},
we have
\[
A(\frac{d}{dt}-\frac{ic(\al)\g}{1+ic'(\al)\g t }\frac{d}{d\al})h=(\frac{d}{dt}-\frac{ic(\al)\g}{1+ic'(\al)\g t}\frac{d}{d\al})A(h).
\]
Then from the trivial term $A(\lambda(\alpha_{\gamma}^{t}))=0$, we have,
\begin{equation}\label{3Ahsplit}
\frac{dA(h)(\alpha,\gamma,t)}{dt}=N_1(A(h))+
       A(N_2(h)+N_3(h)+F_2(h)+\sum_{i}O^{i}(h)+\lambda(\alpha)T_{fixed}(\alpha_{\gamma}^{t})). 
\end{equation}
We will calculate $A(N_2(h))$, $A(N_{3}(h)+F_2(h))$ and $A(O^{i}(h))$  and $A(\lambda(\alpha)T_{fixed}(\alpha_{\gamma}^{t})))$ separately. For $A(N_2(h))$, we first use the following lemma to remove the p.v. before the integral.
\begin{lemma}\label{AB24new}
We have
\begin{align}\label{estimatekernelanalyticity01}
K(D^{-60}(h)(0,\gamma, t)-D^{-60}(h)(\beta,\gamma, t)+\tilde{f}^{L}(0)-\tilde{f}^{L}(\beta_{\gamma}^{t}))\in C_{\gamma}^{1}([-1,1], L_{\beta}^{\infty}[-\pi,\pi]),
\end{align}

\begin{align}\label{estimatekernelanalyticity02}
K(D^{-60}(h)(0,\gamma, t)-D^{-60}(h)(\beta,\gamma, t)+\tilde{f}^{L}(0)-\tilde{f}^{L}(\beta_{\gamma}^{t}))c(\beta) \in C_{\gamma}^{1}([-1,1], C_{\beta}^{1}[-\pi,\pi]),
\end{align}
and
\begin{align}\label{estimatekernelanalyticity03}
(\nabla K)(D^{-60}(h)(0,\gamma, t)-D^{-60}(h)(\beta,\gamma, t)+\tilde{f}^{L}(0)-\tilde{f}^{L}(\beta_{\gamma}^{t}))c(\beta) \in C_{\gamma}^{1}([-1,1], L^{\infty}_{\beta}[-\pi,\pi]).
\end{align}
\end{lemma}
\begin{proof}
Although
\[
K(x_1,x_2)=\frac{\sin(x_1)}{\cos(x_1)-\cosh(x_2)}
\]
is an $K_{-1}$ type operator from def \eqref{k-sigma}, the property of $\tilde{f}^{L}(\beta_{\gamma}^{t})$ make sure there is no singularity in the integral. In fact, from the def of $D^{-60}$(\eqref{D-formularnew}) and $\tilde{f}^{L}(\beta_{\gamma}^{t})$ (\eqref{fLequation},\eqref{turnoverconditionnew}), we have \[
\frac{d}{d\beta}(D^{-60}(h))(0,\gamma,t)=0 \quad \frac{d}{d\beta}(\tilde{f}^{L})(0,t)=0. 
\]
Then from the smoothness of $h$ \eqref{hspaceconclusion} and $\tilde{f}^{L}(\beta_{\gamma}^{t})$ \eqref{fLequation},  we get
\[
\frac{\tilde{f}^L_1(0)-\tilde{f}^L_1(\beta_{\gamma}^{t})}{\beta^2},\frac{D^{-60}(h)(\beta,\gamma,t)-D^{-60}(h)(0,\gamma,t)}{\beta^2} \in C^{1}_{\gamma}([-1,1], L_{\beta}^{\infty}([-\pi,\pi])).
\]
Therefore from the arc-cord condition  \eqref{arcchord4} we have the result \eqref{estimatekernelanalyticity01}.
Additionally, we have $|\frac{c(\beta)}{\beta^2}|\lesssim 1$, $\nabla K$ is $K_{-2}$ type and 
\[
\frac{\tilde{f}^L_1(0)-\tilde{f}^L_1(\beta_{\gamma}^{t})}{\beta},\frac{D^{-60}(h)(\beta,\gamma,t)-D^{-60}(h)(0,\gamma,t)}{\beta} \in C^{1}_{\gamma}([-1,1], C_{\beta}^{1}([-\pi,\pi])).
\]
Then \eqref{estimatekernelanalyticity02}, \eqref{estimatekernelanalyticity03} follows.

\end{proof} 

From \eqref{initialchangevariable}, \eqref{fLequation},  $\tilde{f}^{L}(\alpha,t)$ , $\partial_{t}x(\alpha,t)$, and  $x(\alpha,t)$ are analytic in $\tilde{D}_A=\{\alpha+iy|0< \alpha<\frac{\delta }{4}\}$. By using lemma \ref{switch} to change $A$ and $\frac{\frac{d}{d\alpha}}{1+ic'(\alpha)\gamma t}$, we get
\begin{align}\label{3n21}
&\quad A(N_2(h))\\\nonumber
&=\lambda(\alpha_{\gamma}^{t})\frac{\frac{dA(h)}{d\al}(\al,\g,t)}{1+ic'(\al)\g t}[\kappa(t)+\frac{\partial_{t}x}{\partial_{\alpha}x}(\alpha_{\gamma}^{t})  -\frac{\partial_{t}x}{\partial_{\alpha}x}(0)\\\nonumber
&\quad+\int_{-\pi}^{\pi}K(D^{-60}(h)(\alpha,\gamma, t)-D^{-60}(h)(\beta,\gamma, t)+\tilde{f}^{L}(\alpha_{\gamma}^{t})-\tilde{f}^{L}(\beta_{\gamma}^{t}))(\frac{1}{(\partial_{\alpha}x)(\alpha_{\gamma}^{t})}-\frac{1}{(\partial_{\beta}x)(\beta_{\gamma}^{t})})\frac{dx(\beta_{\gamma}^{t})}{d\beta}d\beta\\\nonumber
  &\quad-\frac{1}{(\partial_{\alpha}x)(0)}\int_{-\pi}^{\pi}K(D^{-60}(h)(0,\gamma, t)-D^{-60}(h)(\beta,\gamma, t)+\tilde{f}^{L}(0)-\tilde{f}^{L}(\beta_{\gamma}^{t}))\frac{dx(\beta_{\gamma}^{t})}{d\beta}d\beta]\\\nonumber
  &\quad+[A(\int_{-\pi}^{\pi}K(D^{-60}(h)(\alpha,\gamma, t)-D^{-60}(h)(\beta,\gamma, t)+\tilde{f}^{L}(\alpha_{\gamma}^{t})-\tilde{f}^{L}(\beta_{\gamma}^{t}))(\frac{1}{(\partial_{\alpha}x)(\alpha_{\gamma}^{t})}-\frac{1}{(\partial_{\beta}x)(\beta_{\gamma}^{t})})\frac{dx(\beta_{\gamma}^{t})}{d\beta}d\beta)\\\nonumber
  &\quad-A(\frac{1}{(\partial_{\alpha}x)(0)}\int_{-\pi}^{\pi}K(D^{-60}(h)(0,\gamma, t)-D^{-60}(h)(\beta,\gamma, t)+\tilde{f}^{L}(0)-\tilde{f}^{L}(\beta_{\gamma}^{t}))\frac{dx(\beta_{\gamma}^{t})}{d\beta}d\beta)]\frac{\frac{dh}{d\al}(\al,\g,t)}{1+ic'(\al)\g t}\lambda(\alpha_{\gamma}^{t}).
\end{align}

By letting 
\begin{equation*}
\ddot{h}(\alpha)=(D^{-60}(h)(\alpha,\gamma,t),\tilde{f}^{L}(\alpha_{\gamma}^{t}),\frac{1}{(\partial_{\alpha}x)(\alpha_{\gamma}^{t})}),
\end{equation*}
\begin{equation*}
    \breve{K}(\ddot{h}(\alpha)-\ddot{h}(\beta))=K(D^{-60}(h)(\alpha,\gamma, t)-D^{-60}(h)(\beta,\gamma, t)+\tilde{f}^{L}(\alpha_{\gamma}^{t})-\tilde{f}^{L}(\beta_{\gamma}^{t}))(\frac{1}{(\partial_{\alpha}x)(\alpha_{\gamma}^{t})}-\frac{1}{(\partial_{\beta}x)(\beta_{\gamma}^{t})}),
\end{equation*}
and 
\begin{equation*}
    X(\beta,\gamma)=(\partial_{\beta}x)(\beta_{\gamma}^{t}),
\end{equation*}
from commuting lemma \ref{forM1}, and  corollary \ref{forD},  we have,
\begin{equation}\label{AB21new2}
\begin{split}
  &\quad A(\int_{-\pi}^{\pi}K(D^{-60}(h)(\alpha,\gamma, t)-D^{-60}(h)(\beta,\gamma, t)+\tilde{f}^{L}(\alpha_{\gamma}^{t})-\tilde{f}^{L}(\beta_{\gamma}^{t}))(\frac{1}{(\partial_{\alpha}x)(\alpha_{\gamma}^{t})}-\frac{1}{(\partial_{\beta}x)(\beta_{\gamma}^{t})})\frac{dx(\beta_{\gamma}^{t})}{d\beta}d\beta)\\
  &=D_{h}(\int_{-\pi}^{\pi}K(D^{-60}(h)(\alpha,\gamma, t)-D^{-60}(h)(\beta,\gamma, t)+\tilde{f}^{L}(\alpha_{\gamma}^{t})-\tilde{f}^{L}(\beta_{\gamma}^{t}))\\
  &\quad \cdot(\frac{1}{(\partial_{\alpha}x)(\alpha_{\gamma}^{t})}-\frac{1}{(\partial_{\beta}x)(\beta_{\gamma}^{t})})\frac{dx(\beta_{\gamma}^{t})}{d\beta}d\beta)[A(h)].
  \end{split}
\end{equation}
Similarly, from smoothness lemma \ref{AB24new}, commuting corollary \ref{forM12} and corollary \ref{forD}, we have
\begin{equation}\label{AB21new3}
\begin{split}
  &\quad A(\frac{1}{(\partial_{\alpha}x)(0)}\int_{-\pi}^{\pi}K(D^{-60}(h)(0,\gamma, t)-D^{-60}(h)(\beta,\gamma, t)+\tilde{f}^{L}(0)-\tilde{f}^{L}(\beta_{\gamma}^{t}))\frac{dx(\beta_{\gamma}^{t})}{d\beta}d\beta)\\
  &=(D_{h}(\frac{1}{(\partial_{\alpha}x)(0)}\int_{-\pi}^{\pi}K(D^{-60}(h)(0,\gamma, t)-D^{-60}(h)(\beta,\gamma, t)+\tilde{f}^{L}(0)-\tilde{f}^{L}(\beta_{\gamma}^{t}))\frac{dx(\beta_{\gamma}^{t})}{d\beta}d\beta)[A(h)],
  \end{split}
\end{equation}
where we also use $D^{-60}(h)(0,\gamma, t)=0$, $D^{-60}(A(h))(0,\gamma, t)=0$ from \eqref{D-formularnew}.

Therefore, from \eqref{3anasplit}, \eqref{3n21}, \eqref{AB21new2}, and \eqref{AB21new3}, we have
\begin{equation}\label{3anan2}
\begin{split}
A(N_2(h))=D_h(N_2(h))[A(h)].
\end{split}
\end{equation}

%The reason we do this separation is that the first term comes from the derivative when it hits on $\al+ic(\al)\gamma t$, which makes it play a different role from the other terms. 

Next, we calculate $A(N_3(h)+F_2(h))$ together to get rid of the principal value. 
We again rewrite $F_2$. From \eqref{M2c02}, by integration by parts, we have
\begin{equation}
\begin{split}
(F_2(h))_{\mu}&=-\lambda(\al_{\gamma}^{t})p.v.\int_{0}^{2\alpha}(K(D^{-60}(h)(\alpha,\gamma, t)-D^{-60}(h)(\beta,\gamma, t)+\tilde{f}^{L}(\alpha_{\gamma}^{t})-\tilde{f}^{L}(\beta_{\gamma}^{t}))\frac{d}{d\beta}(h_{\mu}(\beta,\gamma,t))d\beta\\
 &\quad+\lambda(\al_{\gamma}^{t})K(D^{-60}(h)(\alpha,\gamma, t)-D^{-60}(h)(2\alpha,\gamma, t)+\tilde{f}^{L}(\alpha_{\gamma}^{t})-\tilde{f}^{L}((2\alpha)_{\gamma}^{t}))h_{\mu}(2\alpha,\gamma,t)\\
 &\quad-\lambda(\al_{\gamma}^{t})\int_{0}^{\gamma}\frac{\frac{d}{d\beta}(K(D^{-60}(h)(\alpha,\gamma, t)-D^{-60}(h)(\beta,\eta, t)+\tilde{f}^{L}(\alpha_{\gamma}^{t})-\tilde{f}^{L}(\beta_{\eta}^{t}))|_{\beta=2\alpha}}{1+ic'(2\alpha)\eta t}h_{\mu}(2\alpha,\eta,t)ic(2\alpha) t d\eta\\
    &\quad+\lambda(\al_{\gamma}^{t})\int_{2\alpha}^{\pi}\frac{d}{d\beta}(K(D^{-60}(h)(\alpha,\gamma, t)-\tilde{f}^{+}(\beta_0^ t)+\tilde{f}^{L}(\alpha_{\gamma}^{t})-\tilde{f}^{L}(\beta_{0}^{t}))\tilde{f}_{\mu}^{+(60)}(\beta_{0}^{t})d\beta.
\end{split}
\end{equation}
From \eqref{3uniqh}, we can replace the $\tilde{f}^{+(60)}(\beta_{0}^{t})$ by $h(\beta,0,t)$ , and get
\begin{align*}
(F_2(h))_{\mu}&=-\lambda(\al_{\gamma}^{t})p.v.\int_{0}^{2\alpha}(K(D^{-60}(h)(\alpha,\gamma, t)-D^{-60}(h)(\beta,\gamma, t)+\tilde{f}^{L}(\alpha_{\gamma}^{t})-\tilde{f}^{L}(\beta_{\gamma}^{t}))\frac{d}{d\beta}(h_{\mu}(\beta,\gamma,t))d\beta\\
    &\quad+\lambda(\al_{\gamma}^{t})K(D^{-60}(h)(\alpha,\gamma, t)-D^{-60}(h)(2\alpha,\gamma, t)+\tilde{f}^{L}(\alpha_{\gamma}^{t})-\tilde{f}^{L}((2\alpha)_{\gamma}^{t}))h_{\mu}(2\alpha,\gamma,t)\\
 &\quad+\lambda(\al_{\gamma}^{t})\int_{0}^{\gamma}\nabla K(D^{-60}(h)(\alpha,\gamma, t)-D^{-60}(h)(2\alpha,\eta, t)+\tilde{f}^{L}(\alpha_{\gamma}^{t})-\tilde{f}^{L}((2\alpha)_{\eta}^{t}))\\
    &\qquad\cdot(D^{-59}(h)(2\alpha,\eta,t)+(\tilde{f}^{L})'((2\alpha)_{\eta}^{t}))h_{\mu}((2\alpha),\eta,t)ic(2\alpha) t d\eta\\
    &\quad-\lambda(\al_{\gamma}^{t})\int_{2\alpha}^{\pi}\nabla K(D^{-60}(h)(\alpha,\gamma, t)-D^{-60}(h)(\beta,0, t)+\tilde{f}^{L}(\alpha_{\gamma}^{t})-\tilde{f}^{L}(\beta_{0}^{t}))\\
    &\qquad\cdot(D^{-59}(h)(\beta,0,t)+(\tilde{f}^{L})'(\beta_{0}^{t}))h_{\mu}(\beta,0, t)d\beta.
\end{align*}
Then from the definition of  $N_3(h)$ \eqref{3anasplit}, we have
\begin{align*}
    &\quad A((N_3(h)+F_2(h))_{\mu})\\
    &=\lambda(\al_{\gamma}^{t})A(\int_{0}^{2\alpha}(K(D^{-60}(h)(\alpha,\gamma, t)-D^{-60}(h)(\beta,\gamma, t)+\tilde{f}^{L}(\alpha_{\gamma}^{t})-\tilde{f}^{L}(\beta_{\gamma}^{t}))\\
    &\qquad(\frac{\frac{d}{d\alpha}(h_{\mu})(\alpha,\gamma,t)}{1+ic'(\alpha)\gamma t}-\frac{\frac{d}{d\beta}(h_{\mu})(\beta,\gamma,t)}{1+ic'(\beta)\gamma t})(1+ic'(\beta)\gamma t)d\beta\\
    &\quad+K(D^{-60}(h)(\alpha,\gamma, t)-D^{-60}(h)(2\alpha,\gamma, t)+\tilde{f}^{L}(\alpha_{\gamma}^{t})-\tilde{f}^{L}((2\alpha)_{\gamma}^{t}))h_{\mu}(2\alpha,\gamma,t)\\
 &\quad+\int_{0}^{\gamma}\nabla K(D^{-60}(h)(\alpha,\gamma, t)-D^{-60}(h)(2\al,\eta, t)+\tilde{f}^{L}(\alpha_{\gamma}^{t})-\tilde{f}^{L}((2\al)_{\eta}^{t}))\\
    &\qquad\cdot(D^{-59}(h)(2\alpha,\eta,t)+(\tilde{f}^{L})'((2\alpha)_{\eta}^{t}))h_{\mu}(2\alpha,\eta,t)ic(2\alpha) t d\eta\\
    &\quad-\int_{2\alpha}^{\pi}\nabla K(D^{-60}(h)(\alpha,\gamma, t)-D^{-60}(h)(\beta,0, t)+\tilde{f}^{L}(\alpha_{\gamma}^{t})-\tilde{f}^{L}(\beta_{0}^{t}))\\
    &\qquad\cdot(D^{-59}(h)(\beta,0,t)+(\tilde{f}^{L})'(\beta_{0}^{t}))h_{\mu}(\beta,0, t)d\beta\\
    &\quad+\int_{[-\pi,0]\cup[2\alpha,\pi]}K(D^{-60}(h)(\alpha,\gamma, t)-D^{-60}(h)(\beta,\gamma, t)+\tilde{f}^{L}(\alpha_{\gamma}^{t})-\tilde{f}^{L}(\beta_{\gamma}^{t}))(1+ic'(\beta)\gamma t)d\beta\\
    &\qquad\cdot\frac{\frac{d}{d\alpha}(h_{\mu})(\alpha,\gamma,t)}{1+ic'(\alpha)\gamma t}).
\end{align*}
From commuting corollary \ref{forM2} and \ref{forD}, by letting
\[
\ddot{h}(\alpha,\gamma)=D^{-60}(h)(\alpha,\gamma,t)+\tilde{f}^{L}(\alpha_{\gamma}^{t}),
\]
\[
\breve{K}=K,
\]
and
\[
\breve{h}(\alpha,\gamma)=h_{\mu}(\alpha,\gamma,t),
\]
we have
\begin{align}\label{A-Mequation}
 &\quad A((N_3(h)+F_2(h))_{\mu})\\\nonumber&=\lambda(\al_{\gamma}^{t})\int_{0}^{2\alpha}(K(D^{-60}(h)(\alpha,\gamma, t)-D^{-60}(h)(\beta,\gamma, t)+\tilde{f}^{L}(\alpha_{\gamma}^{t})-\tilde{f}^{L}(\beta_{\gamma}^{t}))\\\nonumber
&\qquad(\underbrace{\frac{\frac{d}{d\alpha}A(h_{\mu})(\alpha,\gamma,t)}{1+ic'(\alpha)\gamma t}}_{A_{2,0}}\underbrace{-\frac{\frac{d}{d\beta}A(h_{\mu})(\beta,\gamma,t)}{1+ic'(\beta)\gamma t})}_{A_{2,1}}(1+ic'(\beta)\gamma t)d\beta\\\nonumber
&\quad+\lambda(\al_{\gamma}^{t})\int_{0}^{2\alpha}(D_{h}K(D^{-60}(h)(\alpha,\gamma, t)-D^{-60}(h)(\beta,\gamma, t)+\tilde{f}^{L}(\alpha_{\gamma}^{t})-\tilde{f}^{L}(\beta_{\gamma}^{t}))[A(h)])\\\nonumber
     &\qquad(\underbrace{\frac{\frac{d}{d\alpha}h_{\mu}(\alpha,\gamma,t)}{1+ic'(\alpha)\gamma t}}_{A_{2,2}}\underbrace{-\frac{\frac{d}{d\beta}h_{\mu}(\beta,\gamma,t)}{1+ic'(\beta)\gamma t})}_{A_{2,3}}(1+ic'(\beta)\gamma t)d\beta\\\nonumber
    &\quad\underbrace{+\lambda(\al_{\gamma}^{t})K(D^{-60}(h)(\alpha,\gamma, t)-D^{-60}(h)(2\alpha,\gamma, t)+\tilde{f}^{L}(\alpha_{\gamma}^{t})-\tilde{f}^{L}((2\alpha)_{\gamma}^{t}))(A(h))_{\mu}(2\alpha,\gamma,t)}_{A_{2,4}}\\\nonumber
    &\quad\underbrace{+\lambda(\al_{\gamma}^{t})D_hK(D^{-60}(h)(\alpha,\gamma, t)-D^{-60}(h)(2\alpha,\gamma, t)+\tilde{f}^{L}(\alpha_{\gamma}^{t})-\tilde{f}^{L}((2\alpha)_{\gamma}^{t}))[A(h)](h)_{\mu}(2\alpha,\gamma,t)}_{A_{2,5}}\\\nonumber
   &\quad\underbrace{+\lambda(\al_{\gamma}^{t})\int_{0}^{\gamma}\nabla K(D^{-60}(h)(\alpha,\gamma, t)-D^{-60}(h)(2\alpha,\eta, t)+\tilde{f}^{L}(\alpha_{\gamma}^{t})-\tilde{f}^{L}((2\alpha)_{\eta}^{t}))}_{A_{2,6}}\\\nonumber
    &\qquad\underbrace{\cdot(D^{-59}(h)(2\al,\eta,t)+(\tilde{f}^{L})'((2\alpha)_{\eta}^{t}))\frac{2ic(\alpha)t(1+ic'(2\alpha)\eta t)}{(1+ic'(\alpha)\gamma t)(ic(2\alpha) t)}A(h_{\mu})((2\alpha),\eta,t)ic(2\alpha) t d\eta}_{A_{2,6}}\\\nonumber
    &\quad+\underbrace{\lambda(\al_{\gamma}^{t})\int_{0}^{\gamma}\nabla K(D^{-60}(h)(\alpha,\gamma, t)-D^{-60}(h)(2\alpha,\eta, t)+\tilde{f}^{L}(\alpha_{\gamma}^{t})-\tilde{f}^{L}((2\alpha)_{\eta}^{t}))}_{A_{2,7}}\\\nonumber
    &\qquad\underbrace{\cdot D^{-59}(A(h))(2\alpha,\eta,t)h_{\mu}(2\alpha,\eta,t)\frac{2ic(\alpha)t(1+ic'(2\alpha)\eta t)}{(1+ic'(\alpha)\gamma t)(ic(2\alpha) t)}ic(2\alpha) t d\eta}_{A_{2,7}}\\\nonumber
     &\quad+\underbrace{\lambda(\al_{\gamma}^{t})\int_{0}^{\gamma}\nabla^2 K(D^{-60}(h)(\alpha,\gamma, t)-D^{-60}(h)(2\alpha,\eta, t)+\tilde{f}^{L}(\alpha_{\gamma}^{t})-\tilde{f}^{L}((2\alpha)_{\eta}^{t}))}_{A_{2,8}}\\\nonumber
     &\qquad\underbrace{[(D^{-60}(A(h))(\alpha,\gamma, t)-\frac{2ic(\alpha)t(1+ic'(2\alpha)\eta t)}{(1+ic'(\alpha)\gamma t)(ic(2\alpha) t)}D^{-60}(A(h))(2\alpha,\eta, t))}_{A_{2,8}},\\\nonumber
    &\qquad\underbrace{(D^{-59}(h)(2\al,\eta,t)+(\tilde{f}^{L})'((2\al)_{\eta}^{t}))]h_{\mu}(2\al,\eta,t)ic(2\alpha) t d\eta}_{A_{2,8}},\\\nonumber
     &\quad\underbrace{-\lambda(\al_{\gamma}^{t})\int_{2\al}^{\pi}\nabla^2 K(D^{-60}(h)(\alpha,\gamma, t)-D^{-60}(h)(\beta,0, t)+\tilde{f}^{L}(\alpha_{\gamma}^{t})-\tilde{f}^{L}(\beta_{0}^{t}))[D^{-60}(A(h))(\alpha,\gamma,t),}_{A_{2,9}}\\\nonumber
    &\qquad\underbrace{(D^{-59}(h)(\beta,0,t)+(\tilde{f}^{L})'(\beta_{0}^{t}))]h_{\mu}(\beta,0, t)d\beta}_{A_{2,9}}\\\nonumber
    &\quad\underbrace{+\lambda(\al_{\gamma}^{t})\int_{[-\pi,0]\cup[2\al,\pi]}D_h(K(D^{-60}(h)(\alpha,\gamma, t)-D^{-60}(h)(\beta,\gamma, t)+\tilde{f}^{L}(\alpha_{\gamma}^{t})-\tilde{f}^{L}(\beta_{\gamma}^{t})))[A(h)]}_{A_{2,10}}\\\nonumber
    &\qquad\underbrace{\cdot(1+ic'(\beta)\gamma t)d\beta\frac{\frac{d}{d\alpha}(h_{\mu})(\alpha,\gamma,t)}{1+ic'(\al)\gamma t}}_{A_{2,10}}\\\nonumber
    &\quad+\underbrace{\lambda(\al_{\gamma}^{t})\int_{[-\pi,0]\cup[2\al,\pi]}K(D^{-60}(h)(\alpha,\gamma, t)-D^{-60}(h)(\beta,\gamma, t)+\tilde{f}^{L}(\alpha_{\gamma}^{t})-\tilde{f}^{L}(\beta_{\gamma}^{t}))(1+ic'(\beta)\gamma t)d\beta}_{A_{2,11}}\\\nonumber
    &\qquad\underbrace{\cdot\frac{\frac{d}{d\alpha}A(h_{\mu})(\alpha,\gamma,t)}{1+ic'(\al)\gamma t}}_{A_{2,11}}\\\nonumber
    &=A_{2,0}+A_{2,1}+A_{2,2}+A_{2,3}+A_{2,4}+A_{2,5}+A_{2,6}+A_{2,7}+A_{2,8}+A_{2,9}+A_{2,10}+A_{2,11}.
\end{align}
Here some terms disappear because $A(\lambda(\al_{\gamma}^{t}))=0$ and $A(\tilde{f}^{L}(\alpha_{\gamma}^{t}))=0$. Here we also use
\[
\frac{\frac{d}{d\alpha}(D^{-60}(h)(\alpha,\gamma,t)+\tilde{f}^{L}(\alpha_{\gamma}^{t}))}{1+ic'(\alpha)\gamma t}=D^{-59}(h)(\alpha,\gamma,t)+(\tilde{f}^{L})'(\alpha_{\gamma}^{t}).
\]
From the def of $N_3(h)$ \eqref{3anasplit}, we have
\begin{equation}\label{3N3ah} A_{2,0}+A_{2,2}+A_{2,10}+A_{2,11}=D_{h}((N_3(h))_{\mu})[A(h)].
\end{equation}
For the term $A_{2,1}$, we again do the integration by parts and have
\begin{align}\label{3a21inte}
   &\quad A_{2,1}\\\nonumber
   &=-\lambda(\al_{\gamma}^{t})p.v.\int_{0}^{2\al}(K(D^{-60}(h)(\alpha,\gamma, t)-D^{-60}(h)(\beta,\gamma, t)+\tilde{f}^{L}(\alpha_{\gamma}^{t})-\tilde{f}^{L}(\beta_{\gamma}^{t}))\frac{d}{d\beta}(A(h_{\mu})(\beta,\gamma,t))d\beta\\\nonumber
   &=-\lambda(\al_{\gamma}^{t})K(D^{-60}(h)(\alpha,\gamma, t)-D^{-60}(h)(2\al,\gamma, t))+\tilde{f}^{L}(\alpha_{\gamma}^{t})-\tilde{f}^{L}((2\al)_{\gamma}^{t}))\\\nonumber
   &\qquad (A(h_{\mu})(2\al,\gamma,t)-A(h_{\mu})(\alpha,\gamma,t))\\\nonumber
   &\quad +\lambda(\al_{\gamma}^{t})K(D^{-60}(h)(\alpha,\gamma, t)-D^{-60}(h)(0,\gamma, t))+\tilde{f}^{L}(\alpha_{\gamma}^{t})-\tilde{f}^{L}((0)_{\gamma}^{t}))(A(h_{\mu})(0,\gamma,t)-A(h_{\mu})(\alpha,\gamma,t))\\\nonumber
   &\quad +\lambda(\al_{\gamma}^{t})p.v.\int_{0}^{2\al}\frac{d}{d\beta}(K(D^{-60}(h)(\alpha,\gamma, t)-D^{-60}(h)(\beta,\gamma, t)+\tilde{f}^{L}(\alpha_{\gamma}^{t})-\tilde{f}^{L}(\beta_{\gamma}^{t})))\\\nonumber
   &\qquad(A(h_{\mu})(\beta,\gamma,t)-A(h_{\mu})(\alpha,\gamma,t))d\beta\\\nonumber
   &=A_{2,1,1}+A_{2,1,2}\\\nonumber 
   &\quad +\lambda(\al_{\gamma}^{t})\lim_{\beta\to\alpha}(\frac{d}{d\beta}(K(D^{-60}(h)(\alpha,\gamma, t)-D^{-60}(h)(\beta,\gamma, t)+\tilde{f}^{L}(\alpha_{\gamma}^{t})-\tilde{f}^{L}(\beta_{\gamma}^{t})))(\alpha-\beta)^2)\\\nonumber
   &\qquad p.v.\int_{0}^{2\al}\frac{(A(h_{\mu})(\beta,\gamma,t)-A(h_{\mu})(\alpha,\gamma,t))}{(\alpha-\beta)^2}d\beta\\\nonumber
   &\quad -\lambda(\al_{\gamma}^{t})\int_{0}^{2\al}[\lim_{\beta\to\alpha}(\frac{d}{d\beta}(K(D^{-60}(h)(\alpha,\gamma, t)-D^{-60}(h)(\beta,\gamma, t)+\tilde{f}^{L}(\alpha_{\gamma}^{t})-\tilde{f}^{L}(\beta_{\gamma}^{t})))(\alpha-\beta)^2)\\\nonumber
   &\qquad -\frac{d}{d\beta}(K(D^{-60}(h)(\alpha,\gamma, t)-D^{-60}(h)(\beta,\gamma, t)+\tilde{f}^{L}(\alpha_{\gamma}^{t})-\tilde{f}^{L}(\beta_{\gamma}^{t})))(\alpha-\beta)^2]\\
   &\qquad\cdot\frac{(A(h_{\mu})(\beta,\gamma,t)-A(h_{\mu})(\alpha,\gamma,t))}{(\alpha-\beta)^2}d\beta\\\nonumber
   &=A_{2,1,1}+A_{2,1,2}+A_{2,1,M}\\\nonumber
   &\quad -\lambda(\al_{\gamma}^{t})p.v.\int_{0}^{2\al}[\lim_{\beta\to\alpha}(\frac{d}{d\beta}(K(D^{-60}(h)(\alpha,\gamma, t)-D^{-60}(h)(\beta,\gamma, t)+\tilde{f}^{L}(\alpha_{\gamma}^{t})-\tilde{f}^{L}(\beta_{\gamma}^{t})))(\alpha-\beta)^2)\\\nonumber
   &\quad -\frac{d}{d\beta}(K(D^{-60}(h)(\alpha,\gamma, t)-D^{-60}(h)(\beta,\gamma, t)+\tilde{f}^{L}(\alpha_{\gamma}^{t})-\tilde{f}^{L}(\beta_{\gamma}^{t})))(\alpha-\beta)^2]\frac{(A(h_{\mu})(\beta,\gamma,t))}{(\alpha-\beta)^2}d\beta\\\nonumber
   &\quad +\lambda(\al_{\gamma}^{t})p.v.\int_{0}^{2\al}[\lim_{\beta\to\alpha}(\frac{d}{d\beta}(K(D^{-60}(h)(\alpha,\gamma, t)-D^{-60}(h)(\beta,\gamma, t)+\tilde{f}^{L}(\alpha_{\gamma}^{t})-\tilde{f}^{L}(\beta_{\gamma}^{t})))(\alpha-\beta)^2)\\\nonumber
   &\quad -\frac{d}{d\beta}(K(D^{-60}(h)(\alpha,\gamma, t)-D^{-60}(h)(\beta,\gamma, t)+\tilde{f}^{L}(\alpha_{\gamma}^{t})-\tilde{f}^{L}(\beta_{\gamma}^{t})))(\alpha-\beta)^2]\frac{1}{(\alpha-\beta)^2}d\beta(A(h_{\mu})(\alpha,\gamma,t))\\\nonumber
   &=A_{2,1,1}+A_{2,1,2}+A_{2,1,M}+A_{2,1,4}+A_{2,1,5}.
\end{align}
Combining \eqref{3a21inte}, \eqref{A-Mequation} and \eqref{3N3ah}, we have
\begin{equation}\label{3a21main}
\begin{split}
&\quad A((N_3(h)+M_{2,1}(h))_{\mu})\\
&=A_{2,1,M}+A_{2,1,1}+A_{2,1,2}+A_{2,1,4}+A_{2,1,5}+\sum_{j=3}^{9}A_{2,j}+D_{h}((N_3(h))_{\mu})[A(h)]\\
&=\lambda(\al_{\gamma}^{t})(L_2^{+}(h))p.v.\int_{0}^{2\al}\frac{(A(h_{\mu})(\alpha,\gamma,t)-A(h_{\mu})(\beta,\gamma,t))}{(\alpha-\beta)^2}d\beta\\
&\quad +A_{2,1,1}+A_{2,1,2}+A_{2,1,4}+A_{2,1,5}+\sum_{j=3}^{9}A_{2,j}+D_{h}((N_3(h))_{\mu})[A(h)],
\end{split}
\end{equation}
where in the last step we use notation in \eqref{3M12chequG01}.
Now we show except $A_{2,1,M}$ and $D_{h}((N_3(h))_{\mu})[A(h)]$, other terms are bounded term  $B.T^{0}$, where $B.T^{0}$ denote the terms satisfying 
\[
\|B.T^{0}\|_{X^1}\lesssim\|A(h)\|_{X^1}.
\]
\begin{lemma}\label{AM2L2}
All $X^1$ norms of $A_{2,1,1}$ to $A_{2,1,5}$ and $A_{2,3}$ to $A_{2,9}$ are bounded by $\|A(h)\|_{X^1}.$
\end{lemma}
\begin{proof}
For
$A_{2,1,1}$, $A_{2,1,2}$ and $A_{2,1,5}$, from definition of $K$ \eqref{kdefi} and $K_{-\sigma}$ type operator \eqref{k-sigma}, by using notation in \eqref{firsttypeestimate}, \eqref{secondtypeestimate}, we have 
\[
A_{2,1,1}=\lambda(\al_{\gamma}^{t})\tilde{B}_{2}^{j_1}(h)(\alpha,\gamma,t)(A(h)(2\alpha,\gamma,t)-A(h)(\alpha,\gamma,t)),
\]
\[
A_{2,1,2}=\lambda(\al_{\gamma}^{t})\tilde{B}_{2}^{j_2}(h)(\alpha,\gamma,t)(A(h)(0,\gamma,t)-A(h)(\alpha,\gamma,t)),
\]
and
\[
A_{2,1,4}=-\lambda(\al_{\gamma}^{t})p.v.\int_{0}^{2\alpha}\tilde{B}^{j_3}_{1}(h)(\alpha,\beta,\gamma,t)\frac{A(h)(\beta,\gamma,t)}{(\alpha-\beta)}d\beta,
\]
\[
A_{2,1,5}=\lambda(\al_{\gamma}^{t})\tilde{B}_{2}^{j_4}(h)(\alpha,\gamma,t)A(h)(\alpha,\gamma,t).
\]

The first two are of sub-type 2 from notation \eqref{secondtypeestimate}, the third one is the sub-type one from notation \eqref{firsttypeestimate}, and the last one is of sub-type 5 in \eqref{secondtypeestimate}. 
Therefore, by kernel estimate for $\tilde{B}_{1}^{j}$ and $\tilde{B}_{2}^{j}$ in lemma \ref{Lh4lemma}, \ref{Lh3lemma},  and Hilbert's inequality, we have the bound for $A_{2,1,1}$ to $A_{2,1,5}$.

For $A_{2,3}$, from  definition of $K$ \eqref{kdefi}, the notation in \eqref{k-sigma}, we can treat $K(D^{-60}(h)(\alpha,\gamma, t)-D^{-60}(h)(\beta,\gamma, t)+\tilde{f}^{L}(\alpha_{\gamma}^{t})-\tilde{f}^{L}(\beta_{\gamma}^{t}))$ as a kernel of type $\sigma=-1$ . Therefore from lemma \ref{alphabetanear}, and space of $h$ \eqref{hspaceconclusion0} we have the bound.

For  $A_{2,4}$ and $A_{2,5}$, we have  $K(D^{-60}(h)(\alpha,\gamma, t)-D^{-60}(h)(2\alpha,\gamma, t)+\tilde{f}^{L}(\alpha_{\gamma}^{t})-\tilde{f}^{L}((2\alpha)_{\gamma}^{t}))$ is a term of sub-type 2 of $\tilde{B}_{2}^{j}(h)$ in \eqref{secondtypeestimate}. Then we have $A_{2,4}+A_{2,5}=D_{h}(B_{2}^{j}(h))[A(h)]$ with $b_i=60.$ Thus by smoothness conditions in \eqref{3GE01} proved in lemma \ref{Lh4lemma}  ,  we can also control these two terms.

For $A_{2,6}$, and $A_{2,7}$, we have the component of $\nabla K(D^{-60}(h)(\alpha,\gamma, t)-D^{-60}(h)(2\alpha,\eta, t)+\tilde{f}^{L}(\alpha_{\gamma}^{t})-\tilde{f}^{L}((2\alpha)_{\eta}^{t}))$ has the form $K_{-2}$ in kernel estimate lemma \ref{gammachange}.
Moreover, from \eqref{cdefi}, we have
\[
|\frac{ic(\alpha)t}{(\alpha)^2}|\lesssim 1.
\]
Then from kernel estimate lemma \ref{gammachange}, we have the result.
For $A_{2,8}$, we have the component of $\nabla^2 K(D^{-60}(h)(\alpha,\gamma, t)-D^{-60}(h)(2\alpha,\eta, t)+\tilde{f}^{L}(\alpha_{\gamma}^{t})-\tilde{f}^{L}((2\alpha)_{\eta}^{t}))$ has the form $K_{-3}^{j}$ in kernel estimate lemma \ref{gammachange}.

Moreover, from \eqref{cdefi}, \eqref{D-formularnew}, and lemma \ref{farboundarybound} we get
\[
|\frac{ic(\alpha)t}{(\alpha)^2}|\lesssim 1,
\]

\[
\|\frac{D^{-60}(A(h))(\alpha,\gamma,t)}{\alpha}\|_{C^{0}_{\gamma}([-1,1], H^1_{\alpha}[0,\frac{\pi}{4}])}\lesssim \|A(h)\|_{C^{0}_{\gamma}([-1,1], L^2_{\alpha}[0,\pi])},
\]

\[
\|\frac{D^{-60}(A(h))(2\alpha,\eta,t)}{\alpha}\|_{C^{0}_{\gamma}([-1,1], H^1_{\alpha}[0,\frac{\pi}{4}])}\lesssim \|A(h)\|_{C^{0}_{\gamma}([-1,1], L^2_{\alpha}[0,\pi])}.
\]
Then from lemma \ref{gammachange}, we have the bound.

For $A_{2,9}$, we still use the component of $\nabla^2 K(D^{-60}(h)(\alpha,\gamma, t)-D^{-60}(h)(\beta,0, t)+\tilde{f}^{L}(\alpha_{\gamma}^{t})-\tilde{f}^{L}(\beta_{0}^{t}))$ has the form $K_{-3}$ in corollary \ref{alphabetafarright}.
Moreover, from uniqueness \eqref{initialdatacondition01} and choice of $\tilde{f}^{+}$ \eqref{cut function0} we have
\[
\|\frac{h(\beta,0,t)}{(\beta)^3}\|_{L^{\infty}_{\beta}[0,\pi]}=\|\frac{\tilde{f}^{+(60)}(\beta,t)}{(\beta)^3}\|_{L^{\infty}_{\beta}[0,\pi]}\lesssim 1.
\]
Then from kernel estimate corollary \ref{alphabetafarright}, we can bound this term.
%therefore
%\begin{align*}
   % &\|A_{2,8}\|_{C^{0}_{\gamma}([-1,1], L^2_{\alpha}(0,\frac{\pi}{4}))}\\
   % &\lesssim\|D^{-60}(A(h))\|_{C^{0}_{\gamma}([-1,1], L^2_{\alpha}(0,\frac{\pi}{4}))}\|\nabla^2 K(D^{-60}(h)(\alpha,\gamma, t)-D^{-60}(h)(\beta,0, t)+\tilde{f}^{L}(\alpha_{\gamma}^{t})-\tilde{f}^{L}(\beta_{0}^{t})\beta^3\|_{C^{0}_{\gamma}([-1,1], L^2_{\alpha}(0,\frac{\pi}{4}))}\\
  %  &\lesssim\|D^{-60}A(h)\|_{C^{0}_{\gamma}([-1,1], L^2_{\alpha}(0,\frac{\pi}{4}))}\\
   % &\lesssim\|A(h)\|_{C^{0}_{\gamma}([-1,1], L^2_{\alpha}(0,\frac{\pi}{4}))}.
    %\end{align*}
\end{proof}
Now we show $A(O^{j,i})=B.T^{0}$ with $1\leq j\leq 5$ and $\lambda(\al_{\gamma}^{t})T_{fixed}(\alpha_{\gamma}^{t},t)=B.T^{0}$.
\begin{lemma}\label{3anao1}
We have 
\[
\|A(O^{1,i})\|_{X^1}\lesssim\|A(h)\|_{X^1}.
\]
\end{lemma}
\begin{proof}
First, we still rewrite the $O^{1,i}$ by changing the $\tilde{f}^{+(60)}(\beta,t)$ by $h(\beta,0,t)$ and also write singular terms together to cancel the $p.v.$. From \eqref{O1icnewequation}, we have
\begin{align*}
&\quad(O^{1,i}(h))_{\mu}\\
&=\lambda(\al_{\gamma}^{t})p.v.\int_{-\pi}^{\pi}K_{-1}^{i}(V^{30}_{h}(\alpha,\gamma,t)-V^{30}_{h}(\beta,\gamma,t),V^{61}_{\tilde{f}^{L}}(\alpha_{\gamma}^{t})-V^{61}_{\tilde{f}^{L}}(\beta_{\gamma}^{t}),V_X(\alpha_{\gamma}^{t})-V_X(\beta_{\gamma}^{t}))X_{i'}(\beta_{\gamma}^{t})\\
&\qquad D^{-60+b_i}(h_v)(\alpha,\gamma,t)(1+ic'(\beta)\gamma t)d\beta\\
&\quad-\lambda(\al_{\gamma}^{t})p.v.\int_{0}^{2\alpha}K_{-1}^{i}(V_{h}^{30}(\alpha,\gamma,t)-V_{h}^{30}(\beta,\gamma,t),V^{61}_{\tilde{f}^{L}}(\alpha_{\gamma}^{t})-V^{61}_{\tilde{f}^{L}}(\beta_{\gamma}^{t}),V_X(\alpha_{\gamma}^{t})-V_X(\beta_{\gamma}^{t}))X_{i'}(\beta_{\gamma}^{t})\\
&\qquad(D^{-60+b_i}(h_v)(\beta,\gamma,t))(1+ic'(\beta)\gamma t)d\beta\\
&\quad+\lambda(\al_{\gamma}^{t})\int_{0}^{\gamma}K_{-1}^{i}(V^{30}_h(\alpha,\gamma,t)-V^{30}_h(2\alpha,\eta,t),V^{61}_{\tilde{f}^{L}}(\alpha_{\gamma}^{t})-V^{61}_{\tilde{f}^{L}}((2\alpha)_{\eta}^{t}),V_{X}(\alpha_{\gamma}^{t})-V_{X}((2\alpha)_{\eta}^{t})\\&\qquad X_{i'}((2\alpha)_{\eta}^{t})(D^{-60+b_i}(h_v)(2\alpha,\eta,t))(ic(2\alpha)t)d\eta\\
&\quad-\lambda(\al_{\gamma}^{t})\int_{2\alpha}^{\pi}K_{-1}^{i}(V^{30}_h(\alpha,\gamma,t)-V_{h}^{30}(\beta,0,t),V^{61}_{\tilde{f}^{L}}(\alpha_{\gamma}^{t})-V^{61}_{\tilde{f}^{L}}(\beta_{0}^{t}),V_X(\alpha_{\gamma}^{t})-V_X(\beta_{0}^{t}))X_{i'}(\beta_{0}^{t})\\
&\qquad (D^{-60+b_i}h_v)(\beta,0,t)d\beta\\
&=\lambda(\al_{\gamma}^{t})\int_{0}^{2\alpha}K_{-1}^{i}(V_{h}^{30}(\alpha,\gamma,t)-V_{h}^{30}(\beta,\gamma,t),V^{61}_{\tilde{f}^{L}}(\alpha_{\gamma}^{t})-V^{61}_{\tilde{f}^{L}}(\beta_{\gamma}^{t}),V_X(\alpha_{\gamma}^{t})-V_X(\beta_{\gamma}^{t}))X_{i'}(\beta_{\gamma}^{t})\\
&\qquad(D^{-60+b_i}(h_v)(\alpha,\gamma,t)-D^{-60+b_i}(h_v)(\beta,\gamma,t))(1+ic'(\beta)\gamma t)d\beta\\
&\quad+\lambda(\al_{\gamma}^{t})\int_{0}^{\gamma}K_{-1}^{i}(V^{30}_h(\alpha,\gamma,t)-V^{30}_h(2\alpha,\eta,t),V^{61}_{\tilde{f}^{L}}(\alpha_{\gamma}^{t})-V^{61}_{\tilde{f}^{L}}((2\alpha)_{\eta}^{t}),V_{X}(\alpha_{\gamma}^{t})-V_{X}((2\alpha)_{\eta}^{t})\\&\qquad X_{i'}((2\alpha)_{\eta}^{t})(D^{-60+b_i}(h_v)(2\alpha,\eta,t))(ic(2\alpha)t)d\eta\\
&\quad-\lambda(\al_{\gamma}^{t})\int_{2\alpha}^{\pi}K_{-1}^{i}(V^{30}_h(\alpha,\gamma,t)-V_{h}^{30}(\beta,0,t),V^{61}_{\tilde{f}^{L}}(\alpha_{\gamma}^{t})-V^{61}_{\tilde{f}^{L}}(\beta_{0}^{t}),V_X(\alpha_{\gamma}^{t})-V_X(\beta_{0}^{t}))X_{i'}(\beta_{0}^{t})\\
&\qquad(D^{-60+b_i}h_v)(\beta,0,t)d\beta\\
&\quad+\lambda(\al_{\gamma}^{t})\int_{[-\pi,0]\cup[2\alpha,\pi]}K_{-1}^{i}(V^{30}_{h}(\alpha,\gamma,t)-V^{30}_{h}(\beta,\gamma,t),V^{61}_{\tilde{f}^{L}}(\alpha_{\gamma}^{t})-V^{61}_{\tilde{f}^{L}}(\beta_{\gamma}^{t}),V_X(\alpha_{\gamma}^{t})-V_X(\beta_{\gamma}^{t}))X_{i'}(\beta_{\gamma}^{t})\\
&\qquad D^{-60+b_i}(h_v)(\alpha,\gamma,t)(1+ic'(\beta)\gamma t)d\beta.
\end{align*}
Then by commuting corollary \ref{forO1},  by letting
\[
\breve{K}=K_{-1}^{i},
\]
\[
\ddot{h}(\alpha,\gamma)=(V_{h}^{30}(\alpha,\gamma,t), V_{\tilde{f}^{L}}^{61}(\alpha_{\gamma}^{t}), V_{X}(\alpha_{\gamma}^{t})),
\]
\[
X(\beta,\gamma)=X_{i'}(\beta_{\gamma}^{t}),
\]
\[
\breve{h}(\alpha,\gamma)
=D^{-60+b_i}h_v(\alpha,\gamma,t),\]
we have
\begin{align*}
  &\quad  A(O^{1,i}(h))\\
  &=\underbrace{\lambda(\al_{\gamma}^{t})p.v.\int_{-\pi}^{\pi}D_{h}K_{-1}^{i}(V^{30}_{h}(\alpha,\gamma,t)-V^{30}_{h}(\beta,\gamma,t),V^{61}_{\tilde{f}^{L}}(\alpha_{\gamma}^{t})-V^{61}_{\tilde{f}^{L}}(\beta_{\gamma}^{t}),V_X(\alpha_{\gamma}^{t})-V_X(\beta_{\gamma}^{t}))[A(h)]X_{i'}(\beta_{\gamma}^{t})}_{AO_{1,1}}\\
&\qquad \underbrace{(1+ic'(\beta)\gamma t)d\beta D^{-60+b_i}(h_v)(\alpha,\gamma,t)}_{AO_{1,1}}\\
&\quad\underbrace{+\lambda(\al_{\gamma}^{t})p.v.\int_{-\pi}^{\pi}K_{-1}^{i}(V^{30}_{h}(\alpha,\gamma,t)-V^{30}_{h}(\beta,\gamma,t),V^{61}_{\tilde{f}^{L}}(\alpha_{\gamma}^{t})-V^{61}_{\tilde{f}^{L}}(\beta_{\gamma}^{t}),V_X(\alpha_{\gamma}^{t})-V_X(\beta_{\gamma}^{t}))X_{i'}(\beta_{\gamma}^{t})}_{AO_{1,2}}\\
&\qquad\underbrace{D^{-60+b_i}(A(h_v))(\alpha,\gamma,t)(1+ic'(\beta)\gamma t)d\beta}_{AO_{1,2}}\\
&\quad \underbrace{-\lambda(\al_{\gamma}^{t})p.v.\int_{0}^{2\alpha}D_{h}K_{-1}^{i}(V_{h}^{30}(\alpha,\gamma,t)-V_{h}^{30}(\beta,\gamma,t),V^{61}_{\tilde{f}^{L}}(\alpha_{\gamma}^{t})-V^{61}_{\tilde{f}^{L}}(\beta_{\gamma}^{t}),V_X(\alpha_{\gamma}^{t})-V_X(\beta_{\gamma}^{t}))[A(h)]X_{i'}(\beta_{\gamma}^{t})}_{AO_{1,3}}\\
&\qquad\underbrace{(D^{-60+b_i}(h_v)(\beta,\gamma,t))(1+ic'(\beta)\gamma t)d\beta}_{AO_{1,3}}\\
&\quad\underbrace{-\lambda(\al_{\gamma}^{t})p.v.\int_{0}^{2\alpha}K_{-1}^{i}(V_{h}^{30}(\alpha,\gamma,t)-V_{h}^{30}(\beta,\gamma,t),V^{61}_{\tilde{f}^{L}}(\alpha_{\gamma}^{t})-V^{61}_{\tilde{f}^{L}}(\beta_{\gamma}^{t}),V_X(\alpha_{\gamma}^{t})-V_X(\beta_{\gamma}^{t}))X_{i'}(\beta_{\gamma}^{t})}_{AO_{1,4}}\\
&\qquad\underbrace{(D^{-60+b_i}(A(h_v))(\beta,\gamma,t))(1+ic'(\beta)\gamma t)d\beta}_{AO_{1,4}}\\
&\quad\underbrace{+\lambda(\al_{\gamma}^{t})\int_{0}^{\gamma}\nabla_1K_{-1}^{i}(V^{30}_h(\alpha,\gamma,t)-V^{30}_h(2\alpha,\eta,t),V^{61}_{\tilde{f}^{L}}(\alpha_{\gamma}^{t})-V^{61}_{\tilde{f}^{L}}((2\alpha)_{\eta}^{t}),V_{X}(\alpha_{\gamma}^{t})-V_{X}((2\alpha)_{\eta}^{t})}_{AO_{1,5}}\\
&\qquad\underbrace{\cdot(V^{30}_{A(h)}(\alpha,\gamma,t)-\frac{2ic(\alpha)t(1+ic'(2\alpha)\eta t)}{(1+ic'(\alpha)\gamma t)(ic(2\alpha) t)}V^{30}_{A(h)}(2\alpha,\eta,t))}_{AO_{1,5}}\\
&\qquad\underbrace{X_{i'}((2\alpha)_{\eta}^{t})(D^{-60+b_i}(h_v)(2\alpha,\eta,t))(ic(2\alpha)t)d\eta}_{AO_{1,5}}\\
&\quad\underbrace{+\lambda(\al_{\gamma}^{t})\int_{0}^{\gamma}K_{-1}^{i}(V^{30}_h(\alpha,\gamma,t)-V^{30}_h(2\alpha,\eta,t),V^{61}_{\tilde{f}^{L}}(\alpha_{\gamma}^{t})-V^{61}_{\tilde{f}^{L}}((2\alpha)_{\eta}^{t}),V_{X}(\alpha_{\gamma}^{t})-V_{X}((2\alpha)_{\eta}^{t})}_{AO_{1,6}}\\&\qquad\underbrace{\cdot X_{i'}((2\alpha)_{\eta}^{t})\frac{2ic(\alpha)t(1+ic'(2\alpha)\eta t)}{(1+ic'(\alpha)\gamma t)(ic(2\alpha) t)}(D^{-60+b_i}(A(h_v))(2\alpha,\eta,t))(ic(2\alpha)t)d\eta}_{AO_{1,6}}\\
&\quad\underbrace{-\lambda(\al_{\gamma}^{t})\int_{2\alpha}^{\pi}\nabla_1 K_{-1}^{i}(V^{30}_h(\alpha,\gamma,t)-V_{h}^{30}(\beta,0,t)),V^{61}_{\tilde{f}^{L}}(\alpha_{\gamma}^{t})-V^{61}_{\tilde{f}^{L}}(\beta_{0}^{t}),V_X(\alpha_{\gamma}^{t})-V_X(\beta_{0}^{t}))\cdot V^{30}_{A(h)}(\alpha,\gamma,t)X_{i'}(\beta_{0}^{t})}_{AO_{1,7}}\\
&\qquad\underbrace{\cdot( D^{-60+b_i}h_v)(\beta,0,t)d\beta}_{AO_{1,7}}\\
&=AO_{1,1}+AO_{1,2}+AO_{1,3}+AO_{1,4}+AO_{1,5}+AO_{1,6}+AO_{1,7}.
\end{align*}
Here $AO_{1,1}$ is a sum of $FA_9$ and a term in $FA_1$ in commuting corollary \ref{forO1}. $AO_{1,2}$ is a sum of $FA_{11}$ and a term in $FA_2$. Here some terms disappear because  $A(V_{\tilde{f}^{L}}(\alpha_{\gamma}^{t}))=A(V_{X}(\alpha_{\gamma}^{t}))=0$ and $A(X_{i'}(\alpha_{\gamma}^{t}))=0$. We also use the notation 
\[
V_{A(h)}^{30}=(D^{-60}A(h_1),D^{-60}A(h_2),D^{-59}A(h_1),D^{-59}A(h_2),....D^{-30}A(h_1), D^{-30}A(h_2))
\]as in \eqref{notation03h}.

For $AO_{1,1}$ and $AO_{1,2}$,  
\[
p.v.\int_{-\pi}^{\pi}K_{-1}^{i}(V^{30}_{h}(\alpha,\gamma,t)-V^{30}_{h}(\beta,\gamma,t),V^{61}_{\tilde{f}^{L}}(\alpha_{\gamma}^{t})-V^{61}_{\tilde{f}^{L}}(\beta_{\gamma}^{t}),V_X(\alpha_{\gamma}^{t})-V_X(\beta_{\gamma}^{t}))X_{i'}(\beta_{\gamma}^{t})(1+ic'(\beta)\gamma t)d\beta
\]
is of the same form as the sub-type 1 of $\tilde{B}_{2}^{i}(h)$  in \eqref{secondtypeestimate}. Then $AO_{1,1}+AO_{1,2}=D_{h}(B_{2}^{i}(h))[A(h)]$. From lemma \ref{Lh4lemma}, we can control $AO_{1,1}+AO_{1,2}$.

For $AO_{1,3}$ and $AO_{1,4}$,
$K_{-1}^{i}(V_{h}^{30}(\alpha,\gamma,t)-V_{h}^{30}(\beta,\gamma,t),V^{61}_{\tilde{f}^{L}}(\alpha_{\gamma}^{t})-V^{61}_{\tilde{f}^{L}}(\beta_{\gamma}^{t}),V_X(\alpha_{\gamma}^{t})-V_X(\beta_{\gamma}^{t}))X_{i'}(\beta_{\gamma}^{t})(1+ic'(\beta)\gamma t)(\alpha-\beta)$ is of the form sub-type 2 of $\tilde{B}_{1}^{i}(h)$ in the lemma \ref{Lh3lemma}.  Then $AO_{1,3}+AO_{1,4}=D_{h}(B_{1}^{i}(h))[A(h)]$. From lemma \ref{Lh3lemma},  it can be bounded. 

For $AO_{1,5}$,  the component of $\nabla_1K_{-1}^{i}(V^{30}_h(\alpha,\gamma,t)-V^{30}_h(2\alpha,\eta,t),V^{61}_{\tilde{f}^{L}}(\alpha_{\gamma}^{t})-V^{61}_{\tilde{f}^{L}}((2\alpha)_{\eta}^{t}),V_{X}(\alpha_{\gamma}^{t})-V_{X}((2\alpha)_{\eta}^{t})$ has the form $K_{-2}^{j}$ in lemma \ref{gammachange}.
Moreover,
\[
|\frac{ic(\alpha)t}{(\alpha)^2}|\lesssim 1.
\]
Then from lemma \ref{gammachange}, we have the result.

For $AO_{1,6}$, the component of $K_{-1}^{i}(V^{30}_h(\alpha,\gamma,t)-V^{30}_h(2\alpha,\eta,t),V^{61}_{\tilde{f}^{L}}(\alpha_{\gamma}^{t})-V^{61}_{\tilde{f}^{L}}((2\alpha)_{\eta}^{t}),V_{X}(\alpha_{\gamma}^{t})-V_{X}((2\alpha)_{\eta}^{t})$ has the form $K_{-1}^{j}$ in lemma \ref{gammachange}. Then it can be bounded in the same way as in $AO_{1,5}$.

For $AO_{1,7}$, we have the component of $\nabla_1 K_{-1}^{i}(V^{30}_h(\alpha,\gamma,t)-V_{h}^{30}(\beta,0,t)),V^{61}_{\tilde{f}^{L}}(\alpha_{\gamma}^{t})-V^{61}_{\tilde{f}^{L}}(\beta_{0}^{t}),V_X(\alpha_{\gamma}^{t})-V_X(\beta_{0}^{t}))$ has the form $K_{-2}^{j}$ in lemma \ref{alphabetafarright}.
Moreover, from the uniqueness,
\[
\|\frac{h(\beta,0,t)}{(\beta)^2}\|_{L^{\infty}_{\beta}[0,\pi]}=\|\frac{\tilde{f}^{+(60)}(\beta,t)}{(\beta)^2}\|_{L^{\infty}_{\beta}[0,\pi]}\lesssim 1,
\]
Then from lemma \ref{alphabetafarright}, we can bound this term.
\end{proof}
\begin{corollary}\label{3anao2}
We have
\[
\|A(O^{2,i})\|_{X^{1}}\lesssim\|A(h)\|_{X^{1}},
\]
and
\[
\|A(O^{4,i})\|_{X^{1}}\lesssim\|A(h)\|_{X^{1}}.
\]
\end{corollary}
\begin{proof}
From \eqref{O2ic} and \eqref{O4ic}, we have $O^{2,i}$, $O^{4,i}$ behaves similarly as $O^{1,i}$. They can be bounded in the same way as $A(O^{1,i})$.
\end{proof}
\begin{lemma}\label{3anao3}
We have
\[
\|A(O^{3,i})\|_{X^1}\lesssim\|A(h)\|_{X^1},
\]
\[
\|A(O^{5,i})\|_{X^1}\lesssim\|A(h)\|_{X^1},
\]
and 
\[
A(\lambda(\al_{\gamma}^{t})T_{fixed}(\alpha_{\gamma}^{t},t))=0.
\]
\end{lemma}
\begin{proof}
From \eqref{O3ic}, we have
\begin{align}
&O^{3,i}(h)=\lambda(\al_{\gamma}^{t})\int_{-\pi}^{\pi}K_{0}^{i}(V^{30}_{h}(\alpha,\gamma,t)-V^{30}_h(\beta,\gamma,t),V^{61}_{\tilde{f}^{L}}(\alpha_{\gamma}^{t})-V^{61}_{\tilde{f}^{L}}(\beta_{\gamma}^{t}),V_X(\alpha_{\gamma}^{t})-V_X(\beta_{\gamma}^{t}))X_{i'}(\beta_{\gamma}^{t})\\\nonumber
&(1+ic'(\beta)\gamma t ) d\beta D^{-60+b_i}(h_{\mu})(\alpha,\gamma,t).
\end{align}
Then 
\begin{align*}
  &\quad A(O^{3,i}(h))\\
  &=\lambda(\al_{\gamma}^{t})A(\int_{-\pi}^{\pi}K_{0}^{i}(V^{30}_{h}(\alpha,\gamma,t)-V^{30}_h(\beta,\gamma,t),V^{61}_{\tilde{f}^{L}}(\alpha_{\gamma}^{t})-V^{61}_{\tilde{f}^{L}}(\beta_{\gamma}^{t}),V_X(\alpha_{\gamma}^{t})-V_X(\beta_{\gamma}^{t}))X_{i'}(\beta_{\gamma}^{t})\\
&\qquad(1+ic'(\beta)\gamma t ) d\beta) D^{-60+b_i}(h_{\mu})(\alpha,\gamma,t)\\
&\quad+\lambda(\al_{\gamma}^{t})\int_{-\pi}^{\pi}K_{0}^{i}(V^{30}_{h}(\alpha_{\gamma}^{t})-V^{30}_h(\beta_{\gamma}^{t}),V^{61}_{\tilde{f}^{L}}(\alpha_{\gamma}^{t})-V^{61}_{\tilde{f}^{L}}(\beta_{\gamma}^{t}),V_X(\alpha_{\gamma}^{t})-V_X(\beta_{\gamma}^{t}))X_{i'}(\beta_{\gamma}^{t})(1+ic'(\beta)\gamma t ) d\beta\\
&\qquad D^{-60+b_i}A(h_{\mu})(\alpha,\gamma,t)\\
&=AO_{3,1}+AO_{3,2}.
\end{align*}
Since from \eqref{k-sigma}, there is no singularity in $K_{0}^{i}$, from commuting lemma \ref{forM1}, when $\alpha\geq 0$, we have
\begin{align*}
AO_{3,1}&=D_{h}(\int_{-\pi}^{\pi}K_{0}^{i}(V^{30}_{h}(\alpha,\gamma,t)-V^{30}_h(\beta,\gamma,t),V^{61}_{\tilde{f}^{L}}(\alpha_{\gamma}^{t})-V^{61}_{\tilde{f}^{L}}(\beta_{\gamma}^{t}),V_X(\alpha_{\gamma}^{t})-V_X(\beta_{\gamma}^{t}))X_{i'}(\beta_{\gamma}^{t})\\
&\quad(1+ic'(\beta)\gamma t ) d\beta)[A(h)] D^{-60+b_i}(h_{\mu})(\alpha,\gamma,t),
\end{align*}
where we use $A(\tilde{f}^{L}(\alpha_{\gamma}^{t},t))=0$ and $A(X_i(\alpha_{\gamma}^{t},t))=0$. 

Since from \eqref{k-sigma}, a $K_{0}$ type kernel is also a $K_{-1}$ type kernel,  the \[
\int_{-\pi}^{\pi}K_{0}^{i}(V^{30}_{h}(\alpha,\gamma,t)-V^{30}_h(\beta,\gamma,t),V^{61}_{\tilde{f}^{L}}(\alpha_{\gamma}^{t})-V^{61}_{\tilde{f}^{L}}(\beta_{\gamma}^{t}),V_X(\alpha_{\gamma}^{t})-V_X(\beta_{\gamma}^{t}))X_{i'}(\beta_{\gamma}^{t})(1+ic'(\beta)\gamma t ) d\beta\]
is of the same form as $\tilde{B}_{2}^{i}(h)$ as the sub-type 1 in \eqref{secondtypeestimate}. Then $AO_{3,1}+AO_{3,2}=D_hB_{2}^{i}(h)[A(h)]$ with $b_i=60$. From lemma \ref{Lh4lemma} we can control $AO_{3,1}+AO_{3,2}$.

From \eqref{O5ic}, and commuting corollary \ref{forD}, it is easy to get
\[
\|A(O^{5,i})\|_{X^1}\lesssim\|A(h)\|_{X^1}.
\]
From \eqref{kdefi}, \eqref{Tfixedequation}, lemma \ref{forM1}, since $\tilde{f}^{L}(\alpha,t)$, $x(\alpha,t)$ is analytic in the region $\{\alpha+iy|0< |\alpha|<\frac{\pi}{4}\}$, we have $A(T_{fixed}^{1}(\alpha_{\gamma}^{t},t))=0$ and  $A(T_{fixed}^{2}(\alpha_{\gamma}^{t},t))=0$ and $A(T_{fixed}^{3}(\alpha_{\gamma}^{t},t))=0$. Therefore $A(\lambda(\al_{\gamma}^{t})T_{fixed}(\alpha_{\gamma}^{t},t))=0$.
\end{proof}
Then from \eqref{3Ahsplit}, \eqref{3anan2}, \eqref{3a21main} and lemmas \ref{AM2L2}, \ref{3anao1}, \ref{3anao2}, \ref{3anao3}, we have
\begin{align}\label{3Ahsplit2}
&\quad \frac{dA(h)(\alpha,\gamma,t)}{dt}\\\nonumber
&=N_1(A(h))+
       D_hN_2(h)[A(h)]+D_hN_3(h)[A(h)]\\\nonumber
       &\quad+ \lambda(\al_{\gamma}^{t})(L_2^{+}(h))p.v.\int_{0}^{2\al}\frac{(A(h_{\mu})(\alpha,\gamma,t)-A(h_{\mu})(\beta,\gamma,t))}{(\alpha-\beta)^2}d\beta\\\nonumber
       &\quad+B.T^0. 
\end{align}
From \eqref{3anasplit}, and equation of $M_{1,1}$, $M_{1,2}$ \eqref{3M11chequ}, \eqref{M1chequ}, equation of $N_{1}$, $N_{2}$, $N_{3}$ \eqref{3anasplit}, we have
\begin{equation}
\begin{split} 
    &\quad N_1(A(h))+
       D_hN_2(h)[A(h)]+D_hN_3(h)[A(h)]\\
       &=D_hN_1(A(h))+
       D_hN_2(h)[A(h)]+D_hN_3(h)[A(h)]\\
       &=D_h(M_{1,1}+M_{1,2})[A(h)]\\
       &=M_{1,1}(A(h))+\lambda(\al_{\gamma}^{t})(L_1^{+}(h))\frac{d}{d\al}A(h)+\lambda(\al_{\gamma}^{t})D_h(L_1^{+}(h))[A(h)]\frac{d}{d\al}h.
\end{split}
\end{equation}
Here we also use notation in \eqref{3L1form01}.

From smoothness condition of $L_i^{+}(h)$ in \eqref{3GE01}, we have $\lambda(\al_{\gamma}^{t})D_h(L_1^{+}(h))[A(h)]\frac{d}{d\al}h=B.T^{0}.$
In conclusion, we get 
\begin{align}\label{analyticityestimateG}
&\quad\frac{dA(h)(\alpha,\gamma,t)}{dt}\\\nonumber
&=M_{1,1}(A(h))+\lambda(\al_{\gamma}^{t})(L_1^{+}(h)(\alpha))\frac{d}{d\al}A(h)+\lambda(\al_{\gamma}^{t})(L_2^{+}(h)(\alpha))p.v.\int_{0}^{2\alpha}\frac{A(h)(\beta)-A(h)(\alpha)}{(\alpha-\beta)^2}d\beta+ B.T^{0}.
\end{align}
Since $\tau(t)=0$ \eqref{defkappa}, from \eqref{3M11chequ}, we have
\[
M_{1,1}(A(h))=\lambda(\alpha_{\gamma}^{t})\kappa(t)\frac{dh_{\mu}(\alpha,\gamma,t)}{d\alpha}.\]
Since $\kappa(t)>0$, we use energy estimate corollary \ref{32} and have the estimate.
\end{proof}
\section{Behavior of the modified equation for $\kappa(t)<0$}\label{kappa2sectiongene}

Now we move on to the case when $\kappa(t)<0$.
We will show it fit another generalized equation \eqref{3Tequation2}. In this case $\tau(t)>0$ and is increasing (from \eqref{con:tautneg}).
\begin{align}\label{3Tequation2}
\frac{dh(\al,\g,t)}{dt}=T^{-}(h)=\begin{cases}
    M_{1,2}(h)+M_{2,1}(h)+\sum_{i}B_{M,i}+\sum_{i=1}B_i &\al>-\tau(t)\\
    0 &\al\leq -\tau(t),
\end{cases}
\end{align}
with initial data satisfying $h(\al,\g,0)=h(\al,\g',0).$
Here 
\[
M_{1,2}(h)=\lambda(\al+\tau(t))L_1^{-}(h)(\al,\g,t)\partial_{\al}h(\al,\g,t),
\]
\[
M_{2,1}(h)=\lambda(\al+\tau(t))L_2^{-}(h)(\al,\g,t)\int_{-\tau(t)}^{2\al+\tau(t)}\frac{h(\al,\g,t)-h(\beta,\g,t)}{(\al-\beta)^2}d\beta,
\]
\[
B_{M,i}(h)=\lambda(\al+\tau(t))\int_{-\tau(t)}^{2\al+\tau(t)}L_{B,i}(h)(\al,\beta,\g,t)\frac{D^{-60+b_i}(h)(\beta,\g,t)}{(\al-\beta)}d\beta,
\]
with $b_i\leq 60$, $D^{-i}$ the same definition as \eqref{D-formularnew02}: \begin{align}\label{D-formularnew02}
&D^{-i}(\tilde{h})(\alpha,\gamma,t)\\\nonumber
&=\bigg\{\begin{array}{cc}
         \int_{-\tau(t)}^{\alpha}(1+ic'(\alpha_1+\tau(t))\gamma t) ...\int_{-\tau(t)}^{\alpha_{i-1}}(1+ic'(\alpha_i+\tau(t))\gamma t)\tilde{h}(\alpha_i)d\alpha_i d\alpha_{i-1}...d\alpha_1 &  -\tau(t)< \alpha \leq \pi\\\nonumber
          0 & -\pi\leq\alpha\leq -\tau(t),
    \end{array}
\end{align}
with $L_i^{-}(h)$, $B_i(h)$ satisfying the following conditions.

Let 
\[
C_{\gamma}^{i_0}([-1,1], H^{k}[-\pi,\pi])=W^{i_0,k}.
\]
For $\delta$ in \eqref{fLequation} sufficiently small, there exists $\delta_s,t_s>0$ sufficiently small such that for any $1\leq k\leq 12$, $0\leq i_0\leq 4$, if $h,g,\tilde{g},t$ satisfy
\begin{align}\label{hspacew}
h,g\in W^{i_0,k}\cap \{h| \supp h\in [ -\tau(t),\frac{\pi}{8}]\},
\end{align}
\[ 
\sup_{\gamma,t}\|h(\alpha,\gamma,t)-h(\alpha,\gamma,0)\|_{W^{0,1}_{\alpha}[-\tau(t),\pi]}\lesssim \delta_s,
  \]
  \[
  0\leq t\leq t_s,
  \]
\[
h(\alpha,\gamma,0)=\tilde{f}^{+(60)}(\alpha,0),
\]
then $L_{i}^{-}(h)(\al,\g,t)$ $L_{B,i}(h)$, $B_i(h)$ satisfies the following inequalities:

\textbf{Refined R-T conditions:}
\begin{align}\label{3Lbound02}
  18 |\Im L_{1}^{-}(h)(\al,\g,t)|+18 |\Im L_{2}^{-}(h)(\al,\g,t)|\leq -\Re L_2^{+}(h)(\al,\g,t),\text{ when } \lambda(\alpha+\tau(t))\neq 0.
\end{align}
\textbf{Vanishing conditions:}
\begin{align}\label{3L002}
L_{i}^{-}(h)(-\tau(t),\g,t)=0, \text{(R-T coefficient vanishes at 0)}.
\end{align}
\textbf{Boundary conditions:}
\begin{enumerate}
    % \item For any $0\leq t'< t\leq t_*$, $\lim_{t'\to t}\|B_i(h)(\al,\g,t)\|_{{C_{\gamma}^{i_0}([-1,1], H^{k}[-\tau(t),\tau(t')])}}=0,$ and 
     
   % $\lim_{t\to t'}\|B_i(h)(\al,\g,t)\|_{{C_{\gamma}^{i_0}([-1,1], H^{k}[-\tau(t),\tau(t')])}}=0.$
     \item  For $0\leq t\leq t_s$,  we have $\lim_{\alpha\to -\tau(t)^{+}}\sum_{i}\partial_{\al}^{j}\partial_{\g}^{j'}B_i(h)(\alpha,\g,t)=0$, when $0\leq j\leq k-1$, $0\leq j'\leq i_0$.
     \item  For $0\leq t'\leq t\leq t_s$, if $h\in C_{t}([0,t_s], W^{i_0,k})$, we have\[
\lim_{t'\to t}\|\sum_{i}B_i(h)(\alpha,\gamma,t)\|_ { C_{\gamma}^{i_{0}}([-1,1],H^k_{\alpha}[-\tau(t),\tau(t')])}=0,
\] and \[
\lim_{t\to t'}\|\sum_{i}B_i(h)(\alpha,\gamma,t)\|_ { C_{\gamma}^{i_{0}}([-1,1],H^k_{\alpha}[-\tau(t),\tau(t')])}=0.
\]
\end{enumerate}
\textbf{Smoothness conditions:}
\begin{enumerate}
    \item $\sup_{t}\|L_i^{-}(h)(\alpha,\gamma,t)\|_{C_{\gamma}^{i_0}([-1,1], C^{k+2}_{\alpha}[-\tau(t),\frac{\pi}{4}])}\lesssim 1$,\label{Lhcondition02bound}
    \item $\sup_{t}\|D_{h}L_i^{-}(h)\alpha,\gamma,t)[g]\|_{C_{\gamma}^{i_0}([-1,1], C^{k+2}_{\alpha}[-\tau(t),\frac{\pi}{4}])}\lesssim \|g\|_{W^{i_0,k}}$,\label{Lhcondition02lip}
    \item For any $0\leq t'< t\leq t_s$,\label{Lhcondition02continuitytime}
    \[
    \|L_{i}^{-}(h)(\alpha,\gamma,t)-L_{i}^{-}(h)(\alpha,\gamma,t')\|_{C_{\gamma}^{i_0}([-1,1], C_{\alpha}^{k+1}[-\tau(t'),\frac{\pi}{4}])}\\\nonumber
    \lesssim \mathcal{O}(t-t')+\|h(\alpha,\gamma,t)-h(\alpha,\gamma,t')\|_{W^{i_0,k}}.
    \]
     \item $\sup_{t}\|L_{B,i}(h)(\alpha,\beta, \gamma,t)\|_{C_{\gamma}^{i_0}([-1,1], C^{k+2}_{\alpha,\beta}[-\tau(t),\frac{\pi}{4}]\times[-\tau(t),\frac{2}{3}\pi])}\lesssim 1$,\label{LBhcondition02bound}
    \item $\sup_{t}\|D_{h}L_{B,i}(h)(\al,\beta, \g,t)[g]\|_{C_{\gamma}^{i_0}([-1,1], C^{k+2}_{\alpha,\beta}[-\tau(t),\frac{\pi}{4}]\times[-\tau(t),\frac{2}{3}\pi])}\lesssim \|g\|_{W^{i_0,k}}$,\label{LBhcondition02lip}
    \item For any $0\leq t'< t\leq t_s$,\label{LBhcondition02continuitytime}
    \begin{align*}
&\|L_{B,i}(h)(\alpha,\gamma,t)-L_{B,i}(h)(\alpha,\gamma,t')\|_{C_{\gamma}^{i_0}([-1,1], C^{k+2}_{\alpha,\beta}[-\tau(t'),\frac{\pi}{4}]\times[-\tau(t'),\frac{2}{3}\pi]))}\\\nonumber
   & \lesssim \mathcal{O}(t-t')+\|h(\alpha,\gamma,t)-h(\alpha,\gamma,t')\|_{W^{i_0,k}}.
    \end{align*}
     \item $\sup_{t}\|B_i(h)(\alpha,\gamma,t)\|_{C_{\gamma}^{i_0}([-1,1], H^{k}_{\alpha}[-\tau(t),\pi])}\lesssim 1$,\label{Bcondition02bound}
    \item $\sup_{t}\|D_{h}B_i(h)\alpha,\gamma,t)[g]\|_{C_{\gamma}^{i_0}([-1,1], H^{k}_{\alpha}[-\tau(t),\pi])}\lesssim \|g\|_{W^{i_0,k}}$,\label{Bcondition02lip}
    \item For any $0\leq t'< t\leq t_s$,\label{Bcondition02continuitytime}
    \[
    \|B_{i}(h)(\alpha,\gamma,t)-B_{i}(h)(\alpha,\gamma,t')\|_{C_{\gamma}^{i_0}([-1,1], H^{k}[-\tau(t'),\pi])}\\\nonumber
    \lesssim \mathcal{O}(t-t')+\|h(\alpha,\gamma,t)-h(\alpha,\gamma,t')\|_{W^{i_0,k}}.
    \]
     \item $\supp_{\al}\tilde{B}_i(h)\subset [-\tau(t),\frac{\pi}{4}].$\label{3Bsupp02}

     Moreover, when $\g=0$. we have $L_i^{-}(h)(\al,0,t)=\bar{L}_i^{-}(h,t)|$, $L_{B,i}(h)(\al,0,t)=\bar{L}_{B,i}(h,t)$, $B_i(h)(\al,0,t)=\bar{B}_{i}(h,t)$ only depends on $h|_{\g=0}$, and satisfy
     \item $\sup_{t}\|\bar{L}_i^{-}(h,t)\|_{C_{\al}^{k+2}[-\tau(t),\frac{\pi}{4}]}\lesssim 1$,\label{302barl0}
\item $\sup_{t}\|D_h\bar{L}_i^{-}(h,t)[g]\|_{C_{\al}^{k+2}[-\tau(t),\frac{\pi}{4}]}\lesssim \|g|_{\g=0}\|_{H_{\al}^{k}[-\pi,\pi]}$.\label{302barl0lip}
\item $\sup_{t}\|\bar{L}_{B,i}(h,t)\|_{C_{\al}^{k+2}[-\tau(t),\frac{\pi}{4}]}\lesssim 1$,\label{302LBbarl0}
%\item $\sup_{t}\|D_h\bar{L}_{B,i}(h,t)[g]\|_{C_{\al}^{k+2}[-\tau(t),\frac{\pi}{4}]}\lesssim \|g|_{\g=0}\|_{H_{\al}^{k}[-\pi,\pi]}$.\label{302LBbarl0lip}
\item $\sup_{t}\|D_h\bar{L}_{B,i}(h,t)[g]\|_{C_{\al}^{k+2}[-\tau(t),\frac{\pi}{4}]}\lesssim \|g|_{\g=0}\|_{H_{\al}^{k}[-\pi,\pi]}$.\label{302barl0lip}
\item $\sup_{t}\|\bar{B}_{i}(h,t)\|_{H_{\al}^{k}[-\tau(t),\frac{\pi}{4}]}\lesssim 1$,\label{302LBbarl0}
\item $\sup_{t}\|D_h\bar{B}_{i}(h,t)[g]\|_{H_{\al}^{k}[-\tau(t),\frac{\pi}{4}]}\lesssim \|g|_{\g=0}\|_{H_{\al}^{k}[-\pi,\pi]}$.\label{302LBbarl0lip}
\end{enumerate}
Here $|\mathcal{O}(t-t')|\to 0$ as $t\to t'$ independent of $h$. $D_h$ is the Gateaux derivative,

\subsubsection{A generalized equation for $\kappa(t)<0$}
In this section, we will focus on the $\kappa(t)<0$ case. We still aim to show the following theorem:

\begin{theorem}\label{theoremgeneralizedequa02}
The modified equation \eqref{modifiedequation} satisfies the generalized equation \eqref{3Tequation2} when $\kappa(t)<0$. $\delta$, $\delta_{s}$, $t_s$ are chosen to fit the arc-chord condition in lemma \ref{arcchord}, the refined R-T condition in lemma \ref{RTlemma02}.
\end{theorem}We start by separating the terms into different types. We use the same classification as in lemmas \ref{3maintermtype}, \ref{3boundedtermtype} when $\kappa(t)>0$ because the only difference between two cases is a translation from $\alpha$ to $\alpha+\tau(t)$. For the sake of simplicity, we use the same notation except change the notation of $B_{1, i}(h)$ in lemma \ref{3boundedtermtype} to $B_{M,i}(h)$.
\begin{lemma}\label{3maintermtype02}
When $\kappa(t)<0$, we have  $\al_{\gamma}^{\tau}=\al+\tau(t)+ic(\al+\tau(t))\gamma t.$ $M_{1,1}(h)$ in \eqref{modifiedequation} satisfies:
\begin{align}\label{3M11chequ02}
    M_{1,1}(h)=0.
\end{align}

$M_{1,2}$ satisfies
\begin{equation}\label{3M12chequ02}
    M_{1,2}(h)=\lambda(\al+\tau(t))L_1^{-}(h)(\al,\g,t)\partial_{\al}h(\al,\g,t),
\end{equation}

with  \begin{equation}\label{3L1form}
    \begin{split}
    &\quad L_{1}^{-}(h)(\alpha,\gamma,t)=B_{D^{-60}(h)}(\alpha,\gamma,t)\\&=[\frac{1}{1+ic'(\alpha+\tau(t))\gamma t}(ic(\alpha+\tau(t))\gamma +\frac{\partial_{t}x}{\partial_{\alpha}x}(\alpha_{\gamma}^{t},t)\\&
\qquad+\frac{1}{(\partial_{\alpha}x)(\alpha_{\gamma}^{t},t)}
    \cdot p.v.\int_{-\pi}^{\pi}K_{-1}^{i}(\Delta V_h^{30}(\alpha,\gamma,t),\Delta V_{\tilde{f}^{L}}^{61}(\alpha_{\gamma}^{t},t),\Delta V_{X}(\alpha_{\gamma}^{t},t))X_i(\beta_{\gamma}^{t},t)(1+ic'(\beta+\tau(t))\gamma t)d\beta
    \\&\qquad-\frac{\partial_{t}x}{\partial_{\alpha}x}(0,t)-\frac{1}{(\partial_{\alpha}x)(0,t)}p.v\int_{-\pi}^{\pi}K_{-1}^{i}(\Delta V_h^{30}(0,\gamma,t),\Delta V_{\tilde{f}^{L}}^{61}(0,t),\Delta V_{X}(0,t))X_i(\beta_{\gamma}^{t},t)(1+ic'(\beta+\tau(t))\gamma t)d\beta)] \\  &\quad+(\frac{1}{1+ic'(\alpha+\tau(t))\gamma t} -1)\kappa(t),
    \end{split}
\end{equation}
for some $K_{-1}$ type kernel. $M_{2,1}$ satisfies
\begin{equation}\label{3M21chequ02}
    M_{2,1}(h)=\lambda(\al+\tau(t))L_2^{-}(h)(\al,\g,t)\int_{-\tau(t)}^{2\alpha+\tau(t)}\frac{(h(\alpha,\gamma,t)-h(\beta,\gamma,t))}{(\alpha-\beta)^2+\epsilon^2}d\beta,
\end{equation}
with 
\begin{align*}
    L_2^{-}(h)(\alpha,\gamma,t)=-\lim_{\beta\to\alpha}(\frac{d}{d\beta}K_{-1}^{j}(\Delta V_h^{30}(\alpha,\gamma,t),\Delta V_{\tilde{f}^{L}}^{61}(\alpha_{\gamma}^{t},t),\Delta V_{X}(\alpha_{\gamma}^{t},t))(\alpha-\beta)^2)
\end{align*}
with particular $K_{-1}^{j}(A,B,C)=\frac{\sin(A_1+B_1)}{\cosh(A_2+B_2)-\cos(A_1+B_1)}$. 
\end{lemma}
\begin{proof}
   It directly follows from the definition of $\alpha_{\gamma}^{t}$ \eqref{alphadefi}, $\tau(t)$ \eqref{defkappa}, $\lambda(\alpha)$ \eqref{lambdadefi}, $c(\alpha)$ \eqref{cdefi},  the $M_{1,2}$\eqref{M1chequ},  $M_{1,1}$,\eqref{3M11chequ} and $M_{2,1}$\eqref{3M21equationn}.
\end{proof}
\begin{lemma}\label{3boundedtermtype02}
When $\kappa(t)<0$, terms in  $O^{0}$, $O^{1,i}$, $O^{3,i}$ and  $O^{5,i}$  can be written as the sum of the following types with $b_i\leq 60$:

The first type:
 
 \begin{align}\label{BMichequ02}
(B_M^{i}(h))_{\mu}=
 \lambda(\alpha_{\gamma}^{t})p.v.\int_{-\tau(t)}^{2\alpha+\tau(t)}  L_{B,i}(h)(\alpha,\beta,\gamma,t)\frac{D^{-60+b_i}h_v(\beta,\gamma,t)}{(\alpha-\beta)}d\beta,
\end{align}
with $L_{B,i}(h)$ one of the following two types:
sub-type 1:
\begin{align*}
    &L_{B,i}(h)(\alpha,\beta,\gamma,t)=((\alpha-\beta)^2\frac{d}{d\beta}(K_{-1}^{i}(\Delta V_{h}^{30}(\alpha,\gamma,t),\Delta V_{\tilde{f}^{L}}^{61}(\alpha_{\gamma}^{t},t),\Delta V_{X}(\alpha_{\gamma}^{t},t)))X_i(\beta_{\gamma}^{t},t)\\
    &-\lim_{\beta\to\alpha}((\alpha-\beta)^2\frac{d}{d\beta}(K_{-1}^{i}(\Delta V_{h}^{30}(\alpha,\gamma,t),\Delta V_{\tilde{f}^{L}}^{61}(\alpha_{\gamma}^{t},t),\Delta V_{X}(\alpha_{\gamma}^{t},t)))X_i(\alpha_{\gamma}^{t},t))\frac{1}{\alpha-\beta}
\end{align*}
or the form sub-type 2
\begin{align*}
    &L_{B,i}(h)(\alpha,\beta,\gamma,t)=((\alpha-\beta)(K_{-1}^{i}(\Delta V_{h}^{30}(\alpha,\gamma,t),\Delta V_{\tilde{f}^{L}}^{61}(\alpha_{\gamma}^{t},t),\Delta V_{X}(\alpha_{\gamma}^{t},t)))X_i(\beta_{\gamma}^{t},t)(1+ic'(\beta+\tau(t))\gamma t).
\end{align*}

The second type:
\begin{align}\label{B2hequation}
(B_2^{i}(h))_{\mu}=
 & \lambda(\alpha_{\gamma}^{t})\tilde{B}_2^{i}(h)(\alpha,\gamma,t)D^{-60+b_i}h_v(\alpha,\gamma,t) \text{ or } \lambda(\alpha_{\gamma}^{t})\tilde{B}_2^{i}(h)(\alpha,\gamma,t)D^{-60+b_i}h_v(2\alpha+\tau(t),\gamma,t)
\end{align}

with $\tilde{B}_2^{i}(h)$ having the following six forms:
sub-type 1:
\begin{align*}
  \tilde{B}_2^{i}(h)(\alpha,\gamma,t)=p.v. \int_{-\pi}^{\pi}K_{-1}^{i}(\Delta V_h^{30}(\alpha,\gamma,t),\Delta V_{\tilde{f}^{L}}^{61}(\alpha_{\gamma}^{t},t),\Delta V_{X}(\alpha_{\gamma}^{t},t))X_i(\beta_{\gamma}^{t},t)(1+ic'(\beta+\tau(t))\gamma t)d\beta,
\end{align*}
sub-type 2
\begin{align*}
    \tilde{B}_2^{i}(h)(\alpha,\gamma,t)=K_{-1}^{i}(\Delta V_h^{30}(\alpha,\gamma,t),\Delta V_{\tilde{f}^{L}}^{61}(\alpha_{\gamma}^{t},t),\Delta V_{X}(\alpha_{\gamma}^{t},t)|_{\beta=2\alpha+\tau(t)},
\end{align*}
or
\begin{align*}
    \tilde{B}_2^{i}(h)(\alpha,\gamma,t)=K_{-1}^{i}(\Delta V_h^{30}(\alpha,\gamma,t),\Delta V_{\tilde{f}^{L}}^{61}(\alpha_{\gamma}^{t},t),\Delta V_{X}(\alpha_{\gamma}^{t},t)|_{\beta=-\tau(t)},
\end{align*}
with particular $K_{-1}^{j}(A,B,C)=\frac{\sin(A_1+B_1)}{\cosh(A_2+B_2)-\cos(A_1+B_1)}$,

sub-type 3
\begin{align*}
    & \tilde{B}_2^{i}(h)(\alpha,\gamma,t)=\int_{0}^{1}d\zeta\int_{2\alpha+\tau(t)}^{\pi}(\frac{d}{d\beta}(K_{-2}^{i}( \zeta V_h^{30}(\alpha,\gamma,t)-V^{30}(\beta_0^{t},t),\Delta V_{\tilde{f}^{L}}^{61}(\alpha_{\gamma}^{t},t),\Delta V_{X}(\alpha_{\gamma}^{t},t)))X_{i}(\beta_{0}^{t},t)\tilde{f}^{+(b_i)}(\beta_0^{t})d\beta,
\end{align*}
sub-type 4
\begin{align*}
    &\tilde{B}_2^{i}(h)(\alpha,\gamma,t)=\int_{0}^{1}d\zeta\int_{2\alpha+\tau(t)}^{\pi}(K_{-2}^{i}( \zeta V_h^{30}(\alpha,\gamma,t)-V^{30}(\beta_0^{t},t),\Delta V_{\tilde{f}^{L}}^{61}(\alpha_{\gamma}^{t},t),\Delta V_{X}(\alpha_{\gamma}^{t},t))X_{i}(\beta_{0}^{t},t)\tilde{f}^{+(b_i)}(\beta_0^{t})d\beta,
\end{align*}
sub-type 5
\begin{align*}
    &\tilde{B}_2^{i}(h)(\alpha,\gamma,t)=p.v.\int_{-\tau(t)}^{2\al+\tau(t)}((\alpha-\beta)^2\frac{d}{d\beta}(K_{-1}^{i}(\Delta V_{h}^{30}(\alpha,\gamma,t),\Delta V_{\tilde{f}^{L}}^{61}(\alpha_{\gamma}^{t},t),\Delta V_{X}(\alpha_{\gamma}^{t},t)))\\
    &-\lim_{\beta\to\alpha}((\alpha-\beta)^2\frac{d}{d\beta}(K_{-1}^{i}(\Delta V_{h}^{30}(\alpha,\gamma,t),\Delta V_{\tilde{f}^{L}}^{61}(\alpha_{\gamma}^{t},t),\Delta V_{X}(\alpha_{\gamma}^{t},t))))\frac{1}{(\alpha-\beta)^2}d\beta,
\end{align*}
sub-type 6
\begin{align*}
\tilde{B}_2^{i}(h)(\alpha,\gamma,t)=\tilde{X}_i(\alpha_{\gamma}^{t},t).
\end{align*}

The third type:
\begin{align}\label{B3hequation}
(B_{3}^{i}(h))_{\mu}=\lambda(\alpha_{\gamma}^{t})\int_{0}^{\gamma}D^{-60+b_i}h_v(2\alpha+\tau(t),\eta,t)\tilde{B}_3^{i}(h)(\alpha+\tau(t),\gamma,\eta,t)d\eta,
\end{align}
with $\tilde{B}_3^{i}(h)$ having the following two forms:
sub-type 1
\begin{align*}
    &\tilde{B}_3^{i}(h)(\alpha,\gamma,\eta,t)=\frac{d}{d\beta}(K_{-1}^{i}( V_h^{30}(\alpha,\gamma,t)-V_h^{30}(\beta,\eta,t), V_{\tilde{f}^{L}}^{61}(\alpha_{\gamma}^{t},t)-V_{\tilde{f}^{L}}^{61}(\beta_{\eta}^{t},t), V_{X}(\alpha_{\gamma}^{t},t)-V_{X}(\beta_{\eta}^{t},t))|_{\beta=2\alpha+\tau(t)}\\
    &\times \frac{ic(2\alpha+\tau(t))t}{1+ic'(2\alpha+2\tau(t))\eta t}.
\end{align*}
and sub-type 2
\begin{align*}
    &\tilde{B}_3^{i}(h)(\alpha,\gamma,\eta,t)=K_{-1}^{i}( V_h^{30}(\alpha,\gamma,t)-V_h^{30}(\beta,\eta,t), V_{\tilde{f}^{L}}^{61}(\alpha_{\gamma}^{t},t)-V_{\tilde{f}^{L}}^{61}(\beta_{\eta}^{t},t), V_{X}(\alpha_{\gamma}^{t},t)-V_{X}(\beta_{\eta}^{t},t))|_{\beta=2\alpha+\tau(t)}\\
    &X_{j}((2\alpha+\tau(t))_{\eta}^t)\times ic(2\alpha+2\tau(t))t.
\end{align*}

The fourth type are terms not depending on $h$:
\begin{align}\label{B4equation}
B_{4}^{i}=\lambda(\alpha_{\gamma}^{t})\int_{2\alpha+\tau(t)}^{\pi}\tilde{B}_4^{i}(\alpha_{\gamma}^{t},\beta,t)d\beta
\end{align}
with $\tilde{B}_4^{i}$ having the following two forms:
sub-type 1
\begin{align}\label{B4subtype1}
    &(\tilde{B}_4^{i}(\alpha_{\gamma}^{t},\beta,t))_{\mu}=(\frac{d}{d\beta}(K_{-1}^{j}( -V^{30}(\beta_0^{t},t),\Delta V_{\tilde{f}^{L}}^{61}(\alpha_{\gamma}^{t},t),\Delta V_{X}(\alpha_{\gamma}^{t},t)))X_{j}(\beta_{0}^{t},t)\tilde{f}_{v}^{+(b_i)}(\beta_0^{t},t).
\end{align}
sub-type 2
\begin{align}\label{B4subtype2}
    &(\tilde{B}_4^{i}(\alpha_{\gamma}^{t},\beta,t))_{\mu}=(K_{-1}^{j}( -V^{30}(\beta_0^{t},t),\Delta V_{\tilde{f}^{L}}^{61}(\alpha_{\gamma}^{t},t),\Delta V_{X}(\alpha_{\gamma}^{t},t))X_{j}(\beta_0^{t},t)\tilde{f}_{v}^{+(b_i)}(\beta_0^{t},t).
\end{align}
  \end{lemma}
  \begin{proof}
% We also use $h(-\tau(t),\gamma,t)=0$ ( since $h\in W^{i_0,k}$) to cancel a term in \eqref{B3hequation} that corresponds to $\lambda(\alpha_{\gamma}^{t})\tilde{B}_2^{i}(h)(\alpha,\gamma,t)D^{-60+b_i}h_v(0,\gamma,t)$ term in \eqref{secondtypeestimate}
The proof is same as in \eqref{3boundedtermtype} except we change the notation of $B_{1}^{i}$ to $B_{M}^{i}$. In $B_2^{i}(h)$, we use $D^{-60+b_i}(h)(-\tau(t),\gamma,t)=0$ from the space of $h$ \eqref{hspacew}.
  \end{proof}
  In conclusion, from \eqref{modifiedequation} and above lemmas \ref{3maintermtype02}, \ref{3boundedtermtype02}, we have when $\kappa(t)<0$, $\alpha>-\tau(t)$:
\begin{align}\label{conclu:modified equation02}
    \frac{dh(\alpha,\gamma,t)}{dt}=T^{-}(h)=
      M_{1,2}(h)+M_{2,1}(h)+O^{0}+\sum_{i}O^{i}(h)+\lambda(\alpha)T_{fixed}(\alpha_{\gamma}^{t}), 
\end{align}
with $M_{1,1}$, $M_{1,2}$, $M_{2,1}$ the similar structure as in lemma \ref{3maintermtype}. $O_{i}$ satisfies:
\begin{align}\label{conclu:modified equation02O}
O^{0},O^{1,i},O^{3,i},O^{5,i}=\sum_{\tilde{i}}\sum_{j=2}^{3}\underbrace{B^{\tilde{i}}_{j}(h)}_{B_{i} \text{ term in generalized equation}}+\sum_{\tilde{i}}\underbrace{B^{\tilde{i}}_{4}}_{B_{i} \text{ term in generalized equation}}+\sum_{\tilde{i}}\underbrace{B^{\tilde{i}}_{M}(h)}_{B_{M,i} \text{ term in generalized equation}}.
\end{align}
From \eqref{O2ic}, \eqref{O4ic}, we have
\begin{align}\label{O2inewetimate}
O^{2,i}=\sum_{\tilde{i}}\sum_{j=2}^{3}\underbrace{B^{\tilde{i}}_{j} (h)D^{-60+s_i}h(\alpha,\gamma,t)}_{B_{i} \text{ term in generalized equation}}+\sum_{\tilde{i}}\underbrace{B^{\tilde{i}}_{4}\cdot D^{-60+s_i}h(\alpha,\gamma,t)}_{B_{i} \text{ term in generalized equation}}+\sum_{\tilde{i}}\underbrace{B^{\tilde{i}}_{M}(h)D^{-60+s_i}h(\alpha,\gamma,t)}_{B_{M,i} \text{ term in generalized equation}}, \text{ with } s_i\leq 30.
\end{align}
\begin{align}\label{O4inewetimate}
O^{4,i}=\sum_{\tilde{i}}\sum_{j=2}^{3}\underbrace{\int_{0}^{1}B^{\tilde{i}}_{j,\zeta} (h)\zeta^{q_{\tilde{i}}}d\zeta}_{B_{i} \text{ term in generalized equation}}+\sum_{\tilde{i}}\underbrace{\int_{0}^{1}B^{\tilde{i}}_{4,\zeta} \zeta^{q_{\tilde{i}}}d\zeta}_{B_{i} \text{ term in generalized equation}}+\sum_{\tilde{
i}}\underbrace{\int_{0}^{1}B^{\tilde{i}}_{M,\zeta}(h)\zeta^{q_{\tilde{i}}}d\zeta}_{B_{M,i} \text{ term in generalized equation}},
\end{align}
 $B^{\tilde{i}}_{M,\zeta}(h), B^{\tilde{i}}_{j,\zeta} (h)$ are same type of term  as $B^{i}_{M}(h), B^{i}_{j} (h)$ in lemma \ref{3boundedtermtype02} except that kernel $K_{-1}^{i}$ are replaced by  $K_{-1,\zeta}^{i}$ \eqref{k-sigmazeta}. 
\subsubsection{Smoothness and vanishing conditions for $\kappa(t)<0$}\label{smoothvani02}
We first work on the vanishing conditions and the smoothness conditions. We claim each term can be treated as a translation of corresponding terms in the case when $\kappa(t)>0$ \eqref{conclu:modified equation} and we could use a similar method to show these two conditions,  with $\delta_s$, $t_s$ sufficiently small from arc-chord condition in lemma \ref{arcchord}. 

More precisely, $M_{1,2}$ corresponds to the term with the same notation in \eqref{3M12chequG01}. Then we could use a similar way as in lemma  \ref{L1main}.
$M_{2,1}(h)$ $B_{2}^{i}(h)$, $B_{3}^{i}(h)$, $B_4^{i, FT}(h)$,  $\lambda(\al_{\gamma}^{t})T_{fixed}(\alpha_{\gamma}^t,t)$ also corresponds to the term with same notation in lemma \ref{3boundedtermtype}, and we could use lemma \ref{Lh1lemma}, \ref{Lh4lemma}, \ref{Lh5lemma}, \ref{Lhjlemma}, lemma \ref{Tfixed behaviour}. 
For $B_{M,i}(h)$, it corresponds to  $B_{1}^{i}(h)$ in lemma \ref{3boundedtermtype}. Then we could use similar way as in lemma \ref{Lh3lemma} to control it.
Similarly as in lemmas  \ref{Lh3lemma}, \ref{Lh4lemma}, \ref{Lh5lemma}, \ref{Lhjlemma}, lemma \ref{Tfixed behaviour}, we also have the kernels  $\tilde{B}_2^{i}$, $\tilde{B}_{3}^{i}$, $\tilde{B}_{4}^{i}$ satisfying:
\begin{align}\label{newestimatekernal0}
\tilde{B}_{1}^{i}(h) \in C_{\gamma}^{i_0}([-1,1], C_{\alpha,\beta}^{k+2}([-\tau(t),\frac{\pi}{4}]\times [-\tau(t),\frac{2}{3}\pi])).
\end{align}
\begin{align}\label{newestimatekernel1}
\tilde{B}_{2}^{i}(h)\in C_{\gamma}^{i_0}([-1,1], C_{\alpha}^{k+2}[-\tau(t),\frac{1}{4}\pi]),
\end{align}
\begin{align}\label{newestimatekernel2}
\tilde{B}_{3}^{i}(h)\in C_{\gamma,\eta}^{i_0}([-1,1], C_{\alpha}^{k+2}[-\tau(t),\frac{1}{4}\pi]),
\end{align}
and
\begin{align}\label{newestimatekernel03}
\tilde{B}_{4}^{i}\in C_{\gamma}^{4}([-1,1], C_{\alpha,\beta}^{14}(\bar{\Omega})),
\end{align}
\begin{align}\label{newestimatekernel04}
B_{4}^{i}\in C_{\gamma}^{4}([-1,1], C_{\alpha}^{14}[-\tau(t),\frac{\pi}{4}]),
\end{align}
where $\Omega=\{(\alpha,\beta)|\alpha\in (-\tau(t),\frac{1}{4}\pi], \beta \in (-\tau(t),\frac{3}{2}\pi], \beta \geq 2\alpha+\tau(t)\}.$
Moreover,  for $j'+l''\leq 13$, $l\leq 4$, we have
\begin{align}\label{condition left term boundary}
    \lim_{\alpha\to -\tau(t)^{+}}(\frac{d}{d\gamma})^{l}(\frac{d}{d\alpha})^{l''}((\partial_{\alpha}^{j'}\tilde{B}_4^{i})(\alpha_{\gamma}^{t},2\alpha+\tau(t),t))=0.
\end{align}
\begin{align}\label{newestimatekernel4}
\lambda(\alpha_{\gamma}^{t})T_{fixed}(\alpha_{\gamma}^{t},t)\in C_{\gamma}^{4}([-1,1],C^{16}_{\alpha}[-\tau(t),\frac{\pi}{4}]).
\end{align}
Then terms in $M_{1,2},M_{2,1},O^{0},O^{1,i},O^{3,i},O^{5,i}, \lambda(\al_{\gamma}^{t})T_{fixed}(\alpha_{\gamma}^t,t)$ satisfies the smoothness conditions and vanishing conditions. 
From \eqref{O2inewetimate}, \eqref{O4inewetimate}, we claim $O^{4,i}$, $O^{2,i}$ can be estimated similarly.
\subsubsection{First boundary condition for $\kappa(t)<0$}
Now we work on the first boundary condition. Since the smoothness and vanishing conditions have been proved, we will also make use of them.

\begin{lemma}\label{boundary102}
We have $\lim_{\alpha\to -\tau(t)^{+}}\sum_{i}\partial_{\al}^{j}\partial_{\g}^{j'}B_i(h)(-\tau(t),\g,t)=0$, when $0\leq j\leq k-1$, $0\leq j'\leq i_0$.
\end{lemma}
\begin{proof}
The result follows from the following lemmas \ref{FTbound}, \ref{gamma0estimatelemma}, \ref{gammaestimate}.
\end{proof}
From \eqref{conclu:modified equation02}, \eqref{conclu:modified equation02O}, \eqref{O2inewetimate}, \eqref{O4inewetimate} we have 
\begin{align}\label{sumbiconclu}
\sum_{i}B_i(h)=&\sum_{i}\sum_{j=2}^{3}B_{j}^{i}(h)+\sum_{i}\sum_{j=2}^{3}
\int_{0}^{1}B_{j,\zeta}^{i}(h)\zeta^{q_{i}}d\zeta+\sum_{i}\sum_{j=2}^{3}B_{j}^{i}(h)D^{-60+s_{i}}(h)+\sum_{i}B_{4}^{i}D^{-60+s_{i}}(h)\\\nonumber
&+\sum_{\tilde{i}}B_{4}^{i}+\sum_{i}
\int_{0}^{1}B_{4,\zeta}^{i}\zeta^{q_{i}}d\zeta+\lambda(\alpha_{\gamma}^{t})T_{fixed}(\alpha_{\gamma}^{t}). 
\end{align}
Let $FT(\alpha,\gamma,t)$ be sum of terms in $\sum_{i}B_{i}(h)$\ that does not dependent on $h$, which is 
\begin{align}\label{FTdef}
FT(\alpha,\gamma,t)=\sum_{i}B_{4}^{i}+\sum_{i}
\int_{0}^{1}B_{4,\zeta}^{i}\zeta^{q_{i}}d\zeta+\lambda(\alpha_{\gamma}^{t})T_{fixed}(\alpha_{\gamma}^{t}).
\end{align}

We first show except the $FT$, other terms in $\sum_i B_i(h)$ satisfies this boundary condition:
\begin{lemma}\label{FTbound}
When $0\leq j\leq k-1, 0\leq j'\leq i_0$, we have $\sum_{i}B_i(h)-FT$ satisfies 
\begin{align*}
\lim_{\alpha\to -\tau(t)^{+}}\partial_{\al}^{j}\partial_{\g}^{j'}(\sum_{i}B_i(h)-FT)(\alpha,\gamma,t)=0.
\end{align*}
\end{lemma}
\begin{proof}
First, for $0\leq j \leq k-1$, $0\leq j'\leq i_0$, from the space of $h,g\in W^{i_0,k}\cap S$ \eqref{hspacew}, we have
\[
\lim_{\alpha\to-\tau(t)}\partial_{\alpha}^{j}\partial_{\gamma}^{j'} h=0.
\]
From the definition of $D^{-60+s_i}$ \eqref{D-formularnew}, since $s_i\leq 30$, we have 
\begin{align*}
\lim_{\alpha\to-\tau(t)}\partial_{\al}^{j}\partial_{\g}^{j'}D^{-60+s_i}(h)=0.
\end{align*}
 Then from the equation of $B_{2}^{i}(h)$ \eqref{B2hequation}, $B_3^{i}(h)$ 
 \eqref{B3hequation}, $B_{4}^{i}$ 
 \eqref{B4equation} and estimate \eqref{newestimatekernel1}, \eqref{newestimatekernel2},  \eqref{newestimatekernel03}, we also have for $l=2,3$
\begin{align*}
\lim_{\alpha\to-\tau(t)^{+}} \partial_{\alpha}^{j}\partial_{\gamma}^{j'} B_{l}^{i}(h)=0,
\end{align*}
\begin{align*}
\lim_{\alpha\to-\tau(t)^{+}}\partial_{\alpha}^{j}\partial_{\gamma}^{j'}(B_{l}^{i}(h)D^{-60+s_{i}}(h))=0.
\end{align*}
\begin{align*}
\lim_{\alpha\to-\tau(t)^{+}}\partial_{\alpha}^{j}\partial_{\gamma}^{j'}(B_{4}^{\tilde{i}}D^{-60+s_{i}}(h))=0.
\end{align*}
Similarly,  for $l=2,3$, because $B_{l,\zeta}^{i}(h)$ only differs from $B_{l}^{i}(h)$ by changing kernel $K_{-\sigma}^{i}$ to  $K_{-\sigma,\zeta}^{i}$, we claim by similar proof, we have
\begin{align*}
\lim_{\alpha\to-\tau(t)^{+}}\partial_{\alpha}^{j}\partial_{\gamma}^{j'} (\int_{0}^{1}B_{l,\zeta}^{i}(h)\zeta^{q_{i}}d\zeta)=0.
\end{align*}

Then from \eqref{sumbiconclu} \eqref{FTdef}, every term is controlled and we have the result.
\end{proof}
Therefore, we only need to show 
\begin{align}\label{boundaryestimate01}
\quad \lim_{\alpha\to-\tau(t)^{+}}\partial_{\alpha}^{j}\partial_{\gamma}^{j'}FT(\alpha,\gamma,t)=0.
\end{align}
Note that $FT$ is not dependent on $h$. It only depends on the given solution $f(\alpha,t)$. From our derivation of modified equation \eqref{3mainequationchange}, we know when $\gamma=0$, $\tilde{f}^{+(60)}(\alpha+\tau(t))$  also satisfies the modified equation \eqref{modifiedequation}.  Hence we have
\begin{align}\label{Tequationnew}
\frac{d\tilde{f}^{+(60)}(\alpha+\tau(t))}{dt}&=T(\tilde{f}^{+(60)}(\alpha+\tau(t)))\\\nonumber
&=M_{1,1}(\tilde{f}^{+(60)}(\alpha+\tau(t)))+M_{1,2}(\tilde{f}^{+(60)}(\alpha+\tau(t)))+M_{2,1}(\tilde{f}^{+(60)}(\alpha+\tau(t)))\\
&\quad+\sum_{i}B_{M,i}(\tilde{f}^{+(60)}(\alpha+\tau(t)))+\sum_{i}B_{i}(\tilde{f}^{+(60)}(\alpha+\tau(t))).
\end{align}
Now we show 
\begin{lemma}\label{gamma0estimatelemma}
When $0\leq j\leq 12$,
\begin{equation}\label{gamma0estimate}
\lim_{\alpha\to-\tau(t)^{+}}\partial_{\alpha}^{j}FT(\alpha,0,t)=0.
\end{equation}

\end{lemma}
\begin{proof}
From \eqref{Tequationnew}, it is enough to show the following two inequalities:
\begin{equation}\label{RHSestimate}
 \lim_{\alpha\to -\tau(t)^{+}}\partial_{\alpha}^{j}\frac{ d\tilde{f}^{+(60)}(\alpha+\tau(t),t)}{dt}=0
\end{equation}
\begin{align}\label{RHSestimate02}
\lim_{\alpha\to -\tau(t)^{+}}\partial_{\al}^{j}(T(\tilde{f}^{+(60)}(\alpha+\tau(t),t))-FT(\alpha,0,t))=0.
\end{align}By the def of $\tilde{f}^{+}$ \eqref{fLequation}, then $\eqref{RHSestimate}$ holds. We have $\partial_{\alpha}^{j}\tilde{f}^{+(60)}(0,t)=0$ for $j\leq 38$. Hence from the similar proof as in lemma \ref{FTbound} ,for $0\leq j \leq 12$,
  \begin{align*}
\lim_{\alpha\to -\tau(t)^{+}}\partial_{\al}^{j}(\sum_{i}B_i(\tilde{f}^{+(60)}(\alpha+\tau(t)))-FT)(\alpha,t)=0.
\end{align*}

%\begin{align*}
%&\quad\sum_{i}B_i(\tilde{f}^{+(60)})\\
%&=\sum_{\tilde{i}}\sum_{j=2}^{4}B_{j}^{\tilde{i}}(\tilde{f}^{+(60)})+\sum_{\tilde{i}}\sum_{j=2}^{4}
%\int_{0}^{1}B_{j,\zeta}^{\tilde{i}}(\tilde{f}^{+(60)})\zeta^{q_{\tilde{i}}}d\zeta+\sum_{\tilde{i}}\sum_{j=2}^{4}B_{j}^{\tilde{i}}(\tilde{f}^{+(60)})D^{-60+s_{\tilde{i}}}(\tilde{f}^{+(60)})+\lambda(\alpha_{\gamma}^{t})T_{fixed}(\alpha_{\gamma}^{t})\\
%&=\underbrace{\sum_{\tilde{i}}\sum_{j=2}^{3}B_{j}^{\tilde{i}}(\tilde{f}^{+(60)})+\sum_{\tilde{i}}\sum_{j=2}^{3}
%\int_{0}^{1}B_{j,\zeta}^{\tilde{i}}(\tilde{f}^{+(60)})\zeta^{q_{\tilde{i}}}d\zeta+\sum_{\tilde{i}}\sum_{j=2}^{3}B_{j}^{\tilde{i}}(\tilde{f}^{+(60)})D^{-60+s_{\tilde{i}}}(\tilde{f}^{+(60)})}_{BT(\alpha,t)}\\
%&\quad\underbrace{+(\sum_{\tilde{i}}B_{4}^{\tilde{i}}(\tilde{f}^{+(60)})+\sum_{\tilde{i}}
%\int_{0}^{1}B_{4,\zeta}^{\tilde{i}}(\tilde{f}^{+(60)})\zeta^{q_{\tilde{i}}}d\zeta+\lambda(\alpha_{\gamma}^{t})T_{fixed}(\alpha_{\gamma}^{t}))}_{FT(\alpha,0,t)}.
%\end{align*}
Moreover, from the structure of $M_{1,1}$\eqref{3M11chequ02}, $M_{1,2}$ \eqref{3M12chequ02},  $M_{2,1}$ \eqref{3M21chequ02}, $B_{M,i}$ \eqref{BMichequ02}, and smoothness conditions for $L_{1}^{-}(h)$, $L_{2}^{-}(h)$, $L_{B,i}(h)$ in \eqref{3Tequation2} (proved in previous section \ref{smoothvani02}), we have for $0\leq j\leq 12$, 
\[
\lim_{\alpha\to -\tau(t)^{+}}\partial_{\alpha}^{j} (M_{1,2}(\tilde{f}^{+(60)}(\alpha+\tau(t)))+M_{2,1}(\tilde{f}^{+(60)}(\alpha+\tau(t)))+\sum_{i}B_{M,i}(\tilde{f}^{+(60)}(\alpha+\tau(t)))=0.
\]
Then \eqref{RHSestimate02} holds as well.
\end{proof}
Now we show this condition holds for all $\gamma\in [-1,1]$.
\begin{lemma}\label{gammaestimate}
For all $\gamma\in [-1,1]$, $0\leq j\leq 12$, $0\leq j' \leq 4$ we have
\begin{align}\label{boundaryestimate02}
\quad \lim_{\alpha\to-\tau(t)^{+}}\partial_{\alpha}^{j}FT(\alpha,\gamma,t)=0.
\end{align}
and 
\begin{align}\label{boundaryestimate003}
\quad \lim_{\alpha\to-\tau(t)^{+}}\partial_{\alpha}^{j}\partial_{\gamma}^{j'}FT(\alpha,\gamma,t)=0.
\end{align}
\end{lemma}
\begin{proof}
Recall that from \eqref{B4equation} and \eqref{newestimatekernel03}, we have

\begin{align*}\label{B4equation}
B_{4}^{i}=\lambda(\alpha_{\gamma}^{t})\int_{2\alpha+\tau(t)}^{\pi}\tilde{B}_4^{i}(\alpha_{\gamma}^{t},\beta,t)d\beta, 
\end{align*}
with 
\begin{equation}\label{newestimatekernel0302}
\tilde{B}_4^{i}\in C_{\gamma}^{4}([-1,1], C_{\alpha,\beta}^{14}(\bar{\Omega})),
\end{equation}
and
\begin{equation}\label{newestimatekernelboundary}
    \lim_{\alpha\to -\tau(t)^{+}}(\frac{d}{d\gamma})^{l'}(\frac{d}{d\alpha})^{l''}((\partial_{\alpha}^{j'}\tilde{B}_4^{i})(\alpha_{\gamma}^{t},2\alpha+\tau(t),t))=0.
\end{equation}
 Here $\Omega=\{(\alpha,\beta)|\alpha\in (-\tau(t),\frac{1}{4}\pi], \beta \in (-\tau(t),\frac{3}{2}\pi], \beta \geq 2\alpha+\tau(t)\}$, $j'+l''\leq 13$, $l'\leq 4$. 
 
 We first show $\int_{0}^{1}B_{4,\zeta}^{i}\zeta^{q_{i}}d\zeta$ can be written as the same form of $B_{4}^{i}$. From the structure of $B_4^{i}$ \eqref{B4equation}, and the fact that  $B^{i}_{j,\zeta} (h)$ is same type of term  as $ B^{i}_{j} (h)$ except that kernel $K_{-1}^{i}$ are replaced by  $K_{-1,\zeta}^{i}$ \eqref{O4inewetimate}, we get 
\[
\int_{0}^{1}B_{4,\zeta}^{i}\zeta^{q_{i}}d\zeta=\int_{0}^{1}\lambda(\alpha_{\gamma}^{t})\int_{2\alpha+\tau(t)}^{\pi}\tilde{B}_{4,\zeta}^{i}(\alpha_{\gamma}^{t},\beta,t)d\beta d\zeta
\]
with $\tilde{B}_4^{i,\zeta}$ having the following two forms:
sub-type 1
\begin{align*}
    &(\tilde{B}_{4,\zeta}^{i}(\alpha_{\gamma}^{t},\beta,t))_{\mu}=(\frac{d}{d\beta}(K_{-1,\zeta}^{j}( - V^{30}(\beta_0^{t},t),\Delta V_{\tilde{f}^{L}}^{61}(\alpha_{\gamma}^{t},t),\Delta V_{X}(\alpha_{\gamma}^{t},t)))X_{j}(\beta_{0}^{t},t)\tilde{f}_{v}^{+(b_i)}(\beta_0^{t},t).
\end{align*}
sub-type 2
\begin{align*}
    &(\tilde{B}_{4,\zeta}^{i}(\alpha_{\gamma}^{t},\beta,t))_{\mu}=(K_{-1,\zeta}^{j}( - V^{30}(\beta_0^{t},t),\Delta V_{\tilde{f}^{L}}^{61}(\alpha_{\gamma}^{t},t),\Delta V_{X}(\alpha_{\gamma}^{t},t))X_{j}(\beta_0^{t},t)\tilde{f}_{v}^{+(b_i)}(\beta_0^{t},t).
\end{align*}
As in \eqref{newestimatekernel03}, we claim the extra $\zeta$ does not affect the smoothness and we have 
\begin{align}\label{newestimatekernel302}
\tilde{B}_{4,\zeta}^{i}\in C_{\gamma}^{4}([-1,1], C_{\alpha,\beta}^{14}(\bar{\Omega})),
\end{align}
\[
  \lim_{\alpha\to -\tau(t)^{+}}(\frac{d}{d\gamma})^{l'}(\frac{d}{d\alpha})^{l''}((\partial_{\alpha}^{j'}\tilde{B}_{4,\zeta}^{i})(\alpha_{\gamma}^{t},2\alpha+\tau(t),t))=0,
  \]
where $\Omega=\{(\alpha,\beta)|\alpha\in (-\tau(t),\frac{1}{4}\pi], \beta \in (-\tau(t),\frac{3}{2}\pi], \beta \geq 2\alpha+\tau(t)\}.$

Therefore, by changing the integration order and abusing the notation, we could also write 
\[
\int_{0}^{1}B_{4,\zeta}^{i}\zeta^{q_{i}}d\zeta=\lambda(\alpha_{\gamma}^{t})\int_{2\alpha+\tau(t)}^{\pi}\tilde{B}_{4}^{i}(\alpha_{\gamma}^{t},\beta,t)d\beta,
\]
with some kernel satisfies conditions \eqref{newestimatekernel0302}, \eqref{newestimatekernelboundary}.
Then from \eqref{FTdef}, we can write
 \begin{align}\label{FTdef02}
&FT(\alpha,\gamma,t)=\sum_{i}\lambda(\alpha_{\gamma}^{t})\int_{2\alpha+\tau(t)}^{\pi}\tilde{B}_4^{i}(\alpha_{\gamma}^{t},\beta,t)d\beta+\lambda(\alpha_{\gamma}^{t})T_{fixed}(\alpha_{\gamma}^{t}).
\end{align}

 For $\gamma\neq 0$, $0\leq l\leq 12$,
we have
\begin{align*}
    &\lim_{\alpha\to -\tau(t)^{+}}\partial_{\alpha}^{l}(\sum_{i}\int_{2\alpha+\tau(t)}^{\pi}\tilde{B}_{4}^{i}(\alpha_{\gamma}^{t},\beta,t)d\beta+ T_{fixed}(\alpha_{\gamma}^{t},t))\\
    &=_{\eqref{newestimatekernelboundary}}\lim_{\alpha\to -\tau(t)^{+}}\sum_{i}\int_{2\alpha+\tau(t)}^{\pi}\partial_{\alpha}^{l}(\tilde{B}_{4}^{i}(\alpha_{\gamma}^{t},\beta,t))d\beta+\partial_{\alpha}^{l}(T_{fixed}(\alpha_{\gamma}^{t},t))\\
    &=\lim_{\alpha\to -\tau(t)^{+}}\sum_{i}\int_{2\alpha+\tau(t)}^{\pi}\sum_{l'}C_{l,l'}(\partial_{\alpha}^{l'}\tilde{B}_{4}^{i})(\alpha_{\gamma}^{t},\beta,t)\Pi_{n=1}^{l'}\partial_{\alpha}^{\sigma(l,l',n)}(1+ic'(\alpha+\tau(t))\gamma t)d\beta\\
    &\quad+\sum_{l'}C_{l,l'}(\partial_{\alpha}^{l'}T_{fixed})(\alpha_{\gamma}^{t},t)\Pi_{n=1}^{l'}\partial_{\alpha}^{\sigma(l,l',n)}(1+ic'(\alpha+\tau(t))\gamma t)\\
    &=\lim_{\alpha\to -\tau(t)^{+}}\sum_{l'}C_{l,l'}\Pi_{n=1}^{l'}\partial_{\alpha}^{\sigma(l,l',n)}(1+ic'(\alpha+\tau(t))\gamma t)(\sum_i\int_{2\alpha+\tau(t)}^{\pi}(\partial_{\alpha}^{l'}\tilde{B}_{4}^{i})(\alpha_{\gamma}^{t},\beta,t)d\beta+(\partial_{\alpha}^{l'}T_{fixed})(\alpha_{\gamma}^{t},t))\\
    &=\sum_{l'}C_{l,l'}\lim_{\alpha\to -\tau(t)^{+}}(\Pi_{n=1}^{l'}\partial_{\alpha}^{\sigma(l,l',n)}(1+ic'(\alpha+\tau(t))\gamma t))\lim_{\alpha\to-\tau(t)^{+}}(\sum_i\int_{2\alpha+\tau(t)}^{\pi}(\partial_{\alpha}^{l'}\tilde{B}_{4}^{i})(\alpha_{\gamma}^{t},\beta,t)d\beta\\
    &\quad+(\partial_{\alpha}^{l'}T_{fixed})(\alpha_{\gamma}^{t},t)).
\end{align*}
where we use condition \eqref{newestimatekernelboundary} in the second equality.
Moreover, we have 
\begin{align*}
&\lim_{\alpha\to-\tau(t)^{+}}(\sum_{i}\int_{2\alpha+\tau(t)}^{\pi}(\partial_{\alpha}^{l'}\tilde{B}_{4}^{i})(\alpha_{\gamma}^{t},\beta,t)d\beta+(\partial_{\alpha}^{l'}T_{fixed})(\alpha_{\gamma}^{t},t))\\
=&\eqref{newestimatekernel0302},\eqref{newestimatekernel4}\lim_{\alpha\to-\tau(t)^{+}}(\sum_{i}\int_{2\alpha+\tau(t)}^{\pi}(\partial_{\alpha}^{l'}\tilde{B}_{4}^{i})(\alpha_{0}^{t},\beta,t)d\beta+(\partial_{\alpha}^{l'}T_{fixed})(\alpha_{0}^{t},t))\\
=&\eqref{newestimatekernelboundary}\lim_{\alpha\to-\tau(t)^{+}}\partial_{\alpha}^{l'}(\sum_{i}\int_{2\alpha+\tau(t)}^{\pi}(\tilde{B}_{4}^{i})(\alpha_{0}^{t},\beta,t)d\beta+T_{fixed}(\alpha_{0}^{t},t))\\
=&\eqref{FTdef02}\lim_{\alpha\to-\tau(t)^{+}}\partial_{\alpha}^{l'}FT(\alpha,0,t)\\
=&\eqref{gamma0estimate}0,
\end{align*}
where in the second equality we also use $\lim_{\alpha\to -\tau(t)^{+}}\alpha_{\gamma}^{t}=0$ from \eqref{cdefi},\eqref{alphadefi}.

Then we have \eqref{boundaryestimate02}. Moreover, from the form of $FT$ \eqref{FTdef02}, and estimates \eqref{newestimatekernel0302}, \eqref{newestimatekernel4}, we get \eqref{boundaryestimate003} from \eqref{boundaryestimate02}.
\end{proof}
\subsubsection{Second boundary condition for $\kappa(t)<0$}
\begin{lemma}\label{boundary103}
  For $0\leq t'\leq t\leq t_s$, if $h\in C_{t}([0,t_s], W^{i_0,k})$, we have\[
\lim_{t'\to t}\|\sum_{i}B_i(h)(\alpha,\gamma,t)\|_ { C_{\gamma}^{i_{0}}([-1,1],H^k_{\alpha}[-\tau(t),\tau(t')])}=0,
\] and \[
\lim_{t\to t'}\|\sum_{i}B_i(h)(\alpha,\gamma,t)\|_ { C_{\gamma}^{i_{0}}([-1,1],H^k_{\alpha}[-\tau(t),\tau(t')])}=0.
\]
\end{lemma}
We first introduce two lemmas:
\begin{lemma}\label{continuitytl2}
If $h(\alpha,\gamma,t)\in C_{t}^{0}([0,t_s], C_{\gamma}^{0}([-1,1],L^2_{\alpha}[-\pi,\pi])$, then for $0\leq t$, $t'\leq t_s$, we have
\[
\lim_{t'\to t}\|h(\alpha,\gamma,t)\|_ { C_{\gamma}^{0}([-1,1],L^2_{\alpha}[-\tau(t),-\tau(t')])}=0.
\]
\end{lemma}
\begin{proof}
For any $\epsilon>0$, there exists $\delta$ such that when $|\gamma-\gamma'|\leq \delta$, we have 
\[
\|h(\alpha,\gamma,t)-h(\alpha,\gamma',t)\|_{L^2_{\alpha}[-\pi,\pi]}\leq\epsilon.
\]
Let $-1\leq \gamma_1<\gamma_2<...<\gamma_n\leq 1$, with $|\gamma_i-\gamma_j|<\delta$. By the property of $L^2$ integral, there exists $\delta_{t}$ such that when $|t-t'|\leq \delta_{t}$, we have
\[
\sup_{i}\|h(\alpha,\gamma_i,t)\|_{L^2_{\alpha}[-\tau(t),\tau(t')]}\leq \epsilon.
\]
Then when $|t-t'|\leq \delta_{t}$, for all $\gamma \in [-1,1]$, we have $\|h(\alpha,\gamma,t)\|_{L^2_{\alpha}[-\tau(t),-\tau(t')]}\leq 2\epsilon$.
\end{proof}
\begin{corollary}\label{continuitytl202}
If $h(\alpha,\gamma,t)\in C_{t}^{0}([0,t_s], C_{\gamma}^{0}([-1,1],L^2_{\alpha}[-\pi,\pi])$, then for $0\leq t,t'\leq t_s$, we have
\[
\lim_{t\to t'}\|h(\alpha,\gamma,t)\|_ { C_{\gamma}^{0}([-1,1],L^2_{\alpha}[-\tau(t),-\tau(t')])}=0.
\]
\end{corollary}
\begin{proof}
For any $\epsilon>0$, there exists $\tilde{\delta}>0$, such that when $|t-t'|\leq \tilde{\delta}$, we have
\[
\sup_{\gamma}\|h(\alpha,\gamma,t)-h(\alpha,\gamma,t')\|_{L^2_{\alpha}[-\pi,\pi]}<\epsilon.
\]
Since from the lemma \ref{continuitytl2}, $\lim_{t\to t'}\|h(\alpha,\gamma,t')\|_ { C_{\gamma}^{0}([-1,1],L^2_{\alpha}[-\tau(t),\tau(t')])}=0$,
then 
\[
\overline{\lim_{t\to t'}}\sup_{\gamma}\|h(\alpha,\gamma,t)\|_{L^2_{\alpha}[-\tau(t),\tau(t')]}\leq \lim_{t\to t'}\|h(\alpha,\gamma,t')\|_ { C_{\gamma}^{0}([-1,1],L^2_{\alpha}[-\tau(t),\tau(t')])}+\epsilon=\epsilon.
\]
\end{proof}
Since $k\leq 12$, from the form of $FT$ \eqref{FTdef02} and estimates \eqref{newestimatekernel0302}, \eqref{newestimatekernel4}, it is straightforward that
\begin{align}\label{FTboundary1}
\lim_{t\to t'}, \lim_{t'\to t}\|FT\|_ { C_{\gamma}^{i_{0}}([-1,1],H^k_{\alpha}[-\tau(t),\tau(t')])}=0.
\end{align}

Now we show this condition also holds for other terms.

\begin{lemma}\label{FTbound02}
When $0\leq t' \leq t \leq t_s$, $h\in C_{t}([0,t_s], W^{i_0,k})$, we have $\sum_{i}B_i(h)-FT$ satisfies 
\begin{align*}
\lim_{t\to t'},\lim_{t'\to t}\|\sum_{i}B_i(h)-FT\|_ { C_{\gamma}^{i_{0}}([-1,1],H^k_{\alpha}[-\tau(t),\tau(t')])}=0.
\end{align*}
\end{lemma}
\begin{proof}
First, for from lemmas \ref{continuitytl2}, \ref{continuitytl202}, and the space of $h$, we have
\[
\lim_{t'\to t}\|h\|_ { C_{\gamma}^{i_{0}}([-1,1],H^k_{\alpha}[-\tau(t),\tau(t')])}=0,
\]  \[
\lim_{t\to t'}\|h\|_ { C_{\gamma}^{i_{0}}([-1,1],H^k_{\alpha}[-\tau(t),\tau(t')])}=0.
\]
and for $0\leq j\leq k-1$, $0\leq j\leq i_0$,

\[
\lim_{\alpha\to-\tau(t)}\partial_{\alpha}^{j}\partial_{\gamma}^{j'} h=0.
\]
From the property of $D^{-60+s_i}$ \eqref{D-formularnew}, \eqref{3d-ipro}, since $s_i\leq 30$, we have
\[
\lim_{t'\to t}\|D^{-60+s_i}(h)\|_ { C_{\gamma}^{i_{0}}([-1,1],H^k_{\alpha}[-\tau(t),\tau(t')])}=0,
\] and \[
\lim_{t\to t'}\|D^{-60+s_i}(h)\|_ { C_{\gamma}^{i_{0}}([-1,1],H^k_{\alpha}[-\tau(t),\tau(t')])}=0.
\]

 Then from the equation of $B_{2}^{i}(h)$\eqref{B2hequation}, $B_3^{i}(h)$ \eqref{B3hequation}, $B_{4}^{i}$ \eqref{B4equation} and estimate \eqref{newestimatekernel1}, \eqref{newestimatekernel2}, \eqref{newestimatekernel03}, we also have when $j=2,3$
\begin{align*}
\lim_{t\to t'},\lim_{t'\to t}\|B_{j}^{i}(h)\|_ { C_{\gamma}^{i_{0}}([-1,1],H^k_{\alpha}[-\tau(t),\tau(t')])}=0,
\end{align*}
and
\begin{align*}
\lim_{t\to t'},\lim_{t'\to t}\|B_{4}^{i}D^{-60+s_i}(h)\|_ { C_{\gamma}^{i_{0}}([-1,1],H^k_{\alpha}[-\tau(t),\tau(t')])}=0,
\end{align*}
Similarly,  for $j=2,3$, because $B_{j,\zeta}^{i}(h)$ only differs from $B_{j}^{i}(h)$ by changing kernel $K_{-\sigma}^{i}$ to  $K_{-\sigma,\zeta}^{i}$, we claim by similar proof, we have for $j=2,3$
\begin{align*}
\lim_{t\to t'},\lim_{t'\to t}\|\int_{0}^{1}B_{j}^{i,\zeta}(h)\zeta^{q_i}d\zeta\|_ { C_{\gamma}^{i_{0}}([-1,1],H^k_{\alpha}[-\tau(t),\tau(t')])}=0,
\end{align*}
From  \eqref{sumbiconclu} \eqref{FTdef}, every term in $\sum_{i}B_i(h)$ except $FT$ is controlled.
\end{proof}
Then lemma \ref{boundary103} follows from \eqref{FTboundary1}, and lemma \ref{FTbound02}.
\subsubsection{R-T conditions for $\kappa(t)<0$}
Now we show for $\kappa(t)<0$, the refined R-T condition holds. As in the $\tau(t)>0$ case, we first show at $t=0$ the condition is satisfied. Then we use the smoothness condition of $L_i^{-}(h)$ to show if $t$ is sufficiently small, the sign will not change.
\begin{lemma}\label{RTlemma02}
For $h$, $t_s$,
 satisfying
 \[
 h\in W^{0,1},
 \]
\[
\|h(\alpha,\gamma,t)-h(\al,\g,0)\|_{W^{0,1}}\lesssim\delta_s,
\]
\[
h(\al,\g,0)=\tilde{f}^{+(60)}(\al,0),
\]
and
\[
0\leq t\leq t_s,
\] 

there exists $t_s$, $\delta_s$ sufficiently small, such that when $\alpha\in (-\tau(t), \frac{\pi}{4}) \cap \supp{\lambda}(\alpha_{\gamma}^{t})$, we have

\begin{equation}\label{Refined R-T conditon0202}
-\Re L_2^{-}(h)(\alpha,\gamma,t)\geq 18 |\Im L_1^{-}(h)(\alpha,\gamma,t)|+18|\Im L_2^{-}(h)(\alpha,\gamma,t)|.
\end{equation}
\end{lemma}
\begin{proof}
From vanishing conditions of $L_1^{-}(h)$ and $L_2^{-}(h)$ in \eqref{3GE01} (proved in section \ref{smoothvani02}), we have 
\[
\Re L_2^{-}(h)(-\tau(t),\gamma,t)=\Im L_1^{-}(h)(-\tau(t),\gamma,t)=\Im L_2^{-}(h)(-\tau(t),\gamma,t)=0.
\]
Since $\supp\tilde{\lambda}(\alpha_\gamma^t)\subset [-20\delta,20\delta]$, it is enough to show the following refined R-T condition for $\alpha\in (-\tau(t), 20\delta]$:
\begin{equation}\label{Refined R-T conditon02}
-\partial_{\alpha}\Re L_2^{-}(h)(\alpha,\gamma,t)-18|\partial_{\alpha}\Im L_1^{-}(h)(\al,\gamma,t)|-18|\partial_{\alpha}\Im L_2^{-}(h)(\al,\gamma,t)|>0.
\end{equation}
When $t=0$, we have $\tau(0)=0$ from \eqref{defkappa}. Since $h(\alpha,\gamma,0)=\tilde{f}^{+(60)}(\alpha,0)$, we claim from our construction $L_{2}^{-}(h)|_{t=0}$ is same function as in  lemma \ref{secondederivative}  $L_{2}^{+}(h)|_{t=0}$.  Therefore we could use lemma \ref{secondederivative} directly and have when $\alpha\in [-\tau(0),20\delta]=[0,20\delta]$,
\begin{align*}
  -\partial_{\alpha}\Re L_2^{-}(h)(\alpha,\gamma,0)> C_{0}, 
    \end{align*} for some fixed $C_0> 0$.
 Since everything is real at $t=0$, 
  $\partial_{\alpha}\Im L_1^{-}(h)(\al,\gamma,0)=\partial_{\alpha}\Im L_2^{-}(h)(\al,\gamma,0)=0$. 
Then \eqref{Refined R-T conditon02} holds when $t=0$:
\begin{equation}\label{Refined R-T conditon03}
-\partial_{\alpha}\Re L_2^{-}(h)(\alpha,\gamma,0)-18|\partial_{\alpha}\Im L_1^{-}(h)(\al,\gamma,0)|-18|\partial_{\alpha}\Im L_2^{-}(h)(\al,\gamma,0)|>C_0.
\end{equation}

 Now we use the following lemma to show \eqref{Refined R-T conditon0202}  also holds for small $t_{s}$, $\delta_s$
 
 \begin{lemma}
When $\delta$ is sufficiently small, 
\eqref{Refined R-T conditon0202} holds for small $\delta_{s}$, $t_{s}$.
\end{lemma}
\begin{proof}
From smoothness condition of $L_{i}^{-}(h)$, when $0\leq t'\leq t_s$, we also have
\begin{align*}
   &\quad \|L_i^{-}(h)(\alpha,\gamma,t)-L_i^{-}(h)(\alpha,\gamma,0)\|_{C_{\alpha}^{3}[0,20\delta]}\\
    &\lesssim \|h(\alpha,\gamma,t)-h(\alpha,\gamma,0)\|_{W^{0,1}}+\mathcal{O}(t-t')\\
    &=\|h(\alpha,\gamma,t)-\tilde{f}^{{+}(60)}(\al,0)\|_{W^{0,1}}+\mathcal{O}(t-t').
\end{align*}
Therefore when $20\delta\geq \alpha\geq 0$, we have
\begin{align}\label{RTsmalltime1}
&|\partial_{\alpha}\Re L_2^{-}(h)(\al,\g,t)-\partial_{\alpha}\Re L_2^{-}(h)(\al,\g,0)|+|\partial_{\alpha}\Im L_2^{-}(h)(\al,\g,t)-\partial_{\alpha}\Im L_2^{-}(h)(\al,\g,0)|\\\nonumber
&+|\partial_{\alpha} \Im L_1^{-}(h)(\al,\g,t)-\partial_{\alpha} \Im L_1^{-}(h)(\al,\g,0)|\lesssim O(t)+\|h(\alpha,\gamma,t)-\tilde{f}^{+(60)}(\alpha,t)\|_{W^{0,1}}\lesssim O(\delta_s+t_s).
\end{align}
For $\alpha\in (-\tau(t),0],$ still from the smoothness of $L_i^{-}(h)$, we have
\begin{align*}
    &\|L_{i}^{-}(h)(\alpha,\gamma,t)-L_{i}^{-}(h)(0,\gamma,t)\|_{C^{0}_{\gamma}([-1,1],C^{1}_{\alpha}[-\tau(t),-\tau(t')])}\\\nonumber
&\leq (\tau(t))\|L_{i}^{-}(h)(\alpha,\gamma,t)\|_{C^{0}_{\gamma}([-1,1],C^{2}_{\alpha}[-\tau(t),0]}\lesssim (\tau(t)).
\end{align*}
Then $\alpha\in (-\tau(t),0],$
\begin{align}\label{RTsmalltime2}
    &\quad|\partial_{\alpha}\Re L_{2}^{-}(h)(\alpha,\gamma,t)-\partial_{\alpha}\Re L_{2}^{-}(h)(0,\gamma,t)|+|\partial_{\alpha}\Im L_2^{-}(h)(\alpha,\gamma,t)-\partial_{\alpha}\Im L_2^{-}(h)(0,\gamma,t)|\\\nonumber
    &\qquad+|\partial_{\alpha}\Im L_1^{-}(h)(\alpha,\gamma,t)-\partial_{\alpha}\Im L_1^{-}(h)(0,\gamma,t)|\\\nonumber
    &\lesssim |\tau(t)|\lesssim O(t_s).
\end{align}
Then from \eqref{RTsmalltime1}, \eqref{RTsmalltime2}, and \eqref{Refined R-T conditon03}, the lemma holds.
\end{proof}
\end{proof}
\section{Existence and uniqueness for a generalized equation for $\kappa(t)<0$ }\label{kappa2sectionexistence}
In this section, we aim to show the existence and uniqueness theorem \ref{existence1theorem2}.
\begin{theorem}\label{existence1theorem2}
    For any equation $T^{-}(h)$ and initial data $h(\al,\g,0)$ satisfying the conditions in \eqref{3Tequation2}.  There exists sufficiently small time $\bar{t}$ and solution $h(\al,\g,t)$ such that  
    \begin{align}\label{hspaceconclusion02}
h\in C_{t}^{1}([0,\bar{t}], C_{\gamma}^{1}([-1,1],H_{\al}^{5}[-\pi,\pi]))\cap \{h|\supp_{\alpha}h\in[-\tau(t),\frac{\pi}{4}]\}.
\end{align}
Moreover, for all $g(\al,t)\in C_{t}^{1}([0,\bar{t}],H^{5}[-\pi,\pi]))\cap\{h|\supp_{\alpha} h \subset  [-\tau(t),\frac{\pi}{4}]\}$ satisfying the equation \eqref{3Tequation2} when $\gamma=0$ and the initial data condition $g(\alpha,0)=h(\alpha,\gamma,0)=\tilde{f}^{+(60)}(\alpha,t)$, we have when $t\leq \bar{t}$, 
\begin{equation}\label{uniquenesskappanegi}
g(\al,t)=h(\al,0,t).
\end{equation}
\end{theorem}
\subsection{The existence of $h$}
\label{model1generalizedequation}
\subsubsection{Perturbation}
In order to use the Picard theorem, we first perturb terms $M_{1,2}$, $M_{2,1}$ and $B_{M,i}^{\epsilon}$ by desingularize the denominator. Let
\begin{equation}\label{02M12perturb}
 M_{1,2}^{\epsilon}(h)=\frac{2}{\pi}\lambda(\al+\tau(t))L_1^{-}(h)(\al,\g,t)\int_{-\tau(t)}^{2\al+\tau(t)}\frac{(h(\al)-h(\beta))(\al-\beta)\epsilon}{((\al-\beta)^2+\epsilon^2)^2}d\beta,   
\end{equation}
\begin{equation}\label{02M21perturb}
    M_{2,1}^{\epsilon}(h)=\lambda(\al+\tau(t))L_2^{-}(h)(\al,\g,t)\int_{-\tau(t)}^{2\al+\tau(t)}\frac{h(\al,\g,t)-h(\beta,\g,t)}{(\al-\beta)^2+\epsilon^2}d\beta,
\end{equation}

\begin{equation}\label{02BMiperturb}
B_{M,i}^{\epsilon}(h)=\lambda(\al+\tau(t))\int_{-\tau(t)}^{2\al+\tau(t)}L_{B,i}(h)(\al,\beta,\g,t)\frac{D^{-60+b_i}(h)(\beta,\g,t)(\al-\beta)}{(\al-\beta)^2+\epsilon^2}d\beta,
\end{equation}
with $b_i\leq 60$.
Then the perturbed operator becomes 
\begin{align*}
\frac{dh(\al,\g,t)}{dt}=T^{\epsilon}(h)=\begin{cases}
    M^{\epsilon}_{1,2}(h)+M^{\epsilon}_{2,1}(h)+\sum_{i}B_{M,i}^{\epsilon}+\sum_{i=1}B_i &\al>-\tau(t)\\
    0 &\al\leq -\tau(t).
\end{cases}
\end{align*}
We also extend the $M_{1,2}^{\ep}$, $M^{\epsilon}_{2,1}$, $B_{M,i}^{\epsilon}$ and $\sum_{i}B_i$ to $[-\pi,\pi]$ by letting them be 0 when $\al\leq -\tau(t)$.
Now we show the perturbed operator is smooth enough. Notice that the bound in this lemma is a very weak bound that can depend on $\epsilon.$

\begin{lemma}\label{Tepsilonbehavior}
Under the assumptions in \eqref{3Tequation2}, if additionally $h,g\in  C_{t}^{0}([0,t_s], W^{i_0,k})$, then we have
    \begin{equation}\label{perturbT1}
    T^{\epsilon}(h)\in C_{t}^{0}([0,t_s], W^{i_0,k})\cap S.
    \end{equation}
    Here \[S=\{h|\supp_{\alpha}h\in[-\tau(t),\frac{\pi}{8}]\}.
    \]

    Moreover, the Lipschitz condition holds;
    \begin{equation}\label{perturbT2}
        \|T^{\epsilon}(h)-T^{\epsilon}(g)\|_{W^{i_0,k}}\lesssim_{\epsilon}\|h-g\|_{W^{i_0,k}},
    \end{equation}
   Here $\lesssim$ depends on $\epsilon$, $\delta$, $\delta_s$, $t_s$  $\|h\|_{W^{i_0,k}}$, $\|g\|_{W^{i_0,k}}$. 
\end{lemma}
\begin{proof}
First, for $M_{1,2}^{\epsilon}$, since there is no singularity in the denominator, from the conditions of $L_i^{-}(h)$, $L_{B,i}(h)$ and $B_{i}(h)$, the following two bound holds:
\begin{equation}\label{Tepslioncont2}
    \sup_{t}\|T^{\epsilon}(h)\|_{C_{\gamma}^{i_0}([-1,1],H^{k}[-\tau(t),\pi])}\lesssim_{\epsilon} 1,
\end{equation}
\begin{equation}\label{Tepsiloncont3}
    \sup_{t}\|D_hT^{\epsilon}[g]\|_{C_{\gamma}^{i_0}([-1,1],H^{k}[-\tau(t),\pi])}\lesssim_{\epsilon} \|h-g\|_{W^{i_0,k}}.
\end{equation}
Moreover, by taking $\delta$ sufficiently small in \eqref{lambdadefi}, we have $\supp_{\alpha}\lambda(\al)\in [0,\frac{\pi}{16}]$.From $\tau(t)>0$ in \eqref{defkappa}, we have $\supp_{\alpha}\lambda(\alpha+\tau(t))\subset [-\tau(t),\frac{\pi}{16}].$  

Hence  
\[
T^{\epsilon}(h) \in S.
\]

Now we show 
\begin{equation}\label{limittepsilon}
\partial_{\al}^{j}\partial_{\g}^{j'}T^{\epsilon}(h)(-\tau(t),\g,t)=0,
\end{equation}
when $0\leq j\leq k-1$ and $0\leq k'\leq i_0$. From the boundary conditions of $\sum_i B_i(h)$ in \eqref{3Tequation2}, we only need to show the corresponding condition for $M_{1,2}^{\epsilon}(h)$, $M_{2,1}^{\epsilon}(h)$ and $B_{M,i}^{\epsilon}(h)$. From $h\in W^{i_0,k}$ \eqref{hspacew}, we have
 for $0\leq j\leq k-1$, $0\leq j\leq i_0$,
\begin{equation}\label{hboundray}
\lim_{\alpha\to-\tau(t)}\partial_{\alpha}^{j}\partial_{\gamma}^{j'} h=0.
\end{equation}
Since terms  $M_{1,2}^{\epsilon}(h)$, $M_{2,1}^{\epsilon}(h)$ and $B_{M,i}^{\epsilon}(h)$ does not have singularity when $\epsilon>0$ \eqref{02M12perturb},\eqref{02M21perturb}, \eqref{02BMiperturb}, the condition follows directly.

Then we have \eqref{limittepsilon}. Combining \eqref{Tepslioncont2}, \eqref{Tepsiloncont3}, and definition of $W^{i_0,k}$ \eqref{hspacew}, we have $T^{\epsilon}(h)\in W^{i_0,k}$ and \eqref{perturbT2} holds.

Finally, we show the continuity of $T^{\epsilon}(h)$ with respect to $t$ in $W^{i_0,k}$. It is enough to show the continuity for $M_{1,2}^{\epsilon}$, $M_{2,1}^{\epsilon}$, $B_{M,i}^{\epsilon}$ and $\sum_{i}B_i$ given that they are extended to $\al\in [-\pi,\pi]$ by defining the value to be $0$ when $\alpha\leq -\tau(t).$

When $0\leq t<t'\leq t_s$, $\alpha>-\tau(t')$, we have
\begin{align*}
    &M_{2,1}^{\epsilon}(h)(t)- M_{2,1}^{\epsilon}(h)(t')=\lambda(\alpha+\tau(t))L_{2}^{-}(h)(\alpha,\gamma,t)\int_{-\tau(t)}^{2\alpha+\tau(t)}\frac{h(\alpha,\gamma,t)-h(\beta,\gamma,t)}{(\alpha-\beta)^2+\epsilon^2}d\beta\\
    &\quad-\lambda(\alpha+\tau(t'))L_{2}^{-}(h)(\alpha,\gamma,t')\int_{-\tau(t')}^{2\alpha+\tau(t')}\frac{h(\alpha,\gamma,t')-h(\beta,\gamma,t')}{(\alpha-\beta)^2+\epsilon^2}d\beta\\
    &=(\lambda(\alpha+\tau(t))-\lambda(\alpha+\tau(t')))L_{2}^{-}(h)(\alpha,\gamma,t)\int_{-\tau(t)}^{2\alpha+\tau(t)}\frac{h(\alpha,\gamma,t)-h(\beta,\gamma,t)}{(\alpha-\beta)^2+\epsilon^2}d\beta\\
    &+\lambda(\alpha+\tau(t'))(L_{2}^{-}(h)(\alpha,\gamma,t)-L_2^{-}(h)(\alpha,\gamma,t'))\int_{-\tau(t)}^{2\alpha+\tau(t)}\frac{h(\alpha,\gamma,t)-h(\beta,\gamma,t)}{(\alpha-\beta)^2+\epsilon^2}d\beta\\
    &+\lambda(\alpha+\tau(t'))L_{2}^{-}(h)(\alpha,\gamma,t')(\int_{-\tau(t)}^{2\alpha+\tau(t)}\frac{h(\alpha,\gamma,t)-h(\beta,\gamma,t)}{(\alpha-\beta)^2+\epsilon^2}d\beta-\int_{-\tau(t')}^{2\alpha+\tau(t')}\frac{h(\alpha,\gamma,t')-h(\beta,\gamma,t')}{(\alpha-\beta)^2+\epsilon^2}d\beta)\\
    &=Term_{1}+Term_{2}+Term_{3}.
\end{align*}
Since $\lambda(\alpha)\in C^{100}[-\pi,\pi]$, from the smoothness conditions of $L_{2}^{-}(h)$ in \eqref{3Tequation2} ,  we have
\begin{align*}
    \|Term_1\|_{C_{\gamma}^{i_0}([-1,1],H^{k}[-\tau(t'),\pi])}\lesssim_{\epsilon} \mathcal{O}(t-t').
\end{align*}
and
\begin{align*}
    \|Term_2\|_{C_{\gamma}^{i_0}([-1,1],H^{k}[-\tau(t'),\pi])}\lesssim_{\epsilon} \mathcal{O}(t-t')+\|h(\alpha,\gamma,t)-h(\alpha,\gamma,t')\|_{W^{i_0,k}}.
\end{align*}
For the last term, we have
\begin{align*}
    &Term_3=\lambda(\alpha+\tau(t'))L_{2}^{-}(h)(\alpha,\gamma,t')(\int_{-\tau(t')}^{2\alpha+\tau(t')}\frac{h(\alpha,\gamma,t)-h(\alpha,\gamma,t')-(h(\beta,\gamma,t)-h(\beta,\gamma,t'))}{(\alpha-\beta)^2+\epsilon^2}d\beta\\
    &+\lambda(\alpha+\tau(t'))L_{2}^{-}(h)(\alpha,\gamma,t')(\int_{-\tau(t)}^{-\tau(t')}+\int_{2\alpha+\tau(t')}^{2\alpha+\tau(t)})\frac{h(\alpha,\gamma,t)-h(\beta,\gamma,t)}{(\alpha-\beta)^2+\epsilon^2}d\beta.
\end{align*}
Then we have
\begin{align*}
    \|Term_3\|_{C_{\gamma}^{i_0}([-1,1],H^{k}[-\tau(t'),\pi])}\lesssim_{\epsilon} &\|h(\alpha,\gamma,t)-h(\alpha,\gamma,t')\|_{W^{i_0,k}}\\
    &+\sqrt{\tau(t)-\tau(t')}\|h(\alpha,\gamma,t)\|_{W^{i_0,k}}.
\end{align*}
For $-\tau(t)\leq \alpha \leq-\tau(t') $, we have
\[
M_{2,1}^{\epsilon}(h)(t)-M_{2,1}^{\epsilon}(h)(t')=M_{2,1}^{\epsilon}(h)(t).
\]
Moreover,
\begin{align*}
    &\quad\|M_{2,1}^{\epsilon}(h)(t)\|_{C_{\gamma}^{i_0}([-1,1],H^{k}[-\tau(t),-\tau(t')])}\\
    &=\|\lambda(\alpha+\tau(t))L_{2}^{-}(h)(\alpha,\gamma,t)\int_{-\tau(t)}^{2\alpha+\tau(t)}\frac{h(\alpha,\gamma,t)-h(\beta,\gamma,t)}{(\alpha-\beta)^2+\epsilon^2}d\beta\|_{C_{\gamma}^{i_0}([-1,1],H^{k}[-\tau(t),-\tau(t')])}\\
    &\lesssim_{\epsilon} \|h(\alpha,\gamma,t)\|_{C_{\gamma}^{i_0}([-1,1],H^{k}[-\tau(t),-\tau(t)+2(\tau(t)-\tau(t'))])}.
\end{align*}
Therefore, we have
\begin{align*}
    &\qquad\|M_{2,1}^{\epsilon}(h)(t)-M_{2,1}^{\epsilon}(h)(t')\|_{W^{i_0,k}}\\
    &\lesssim_{\epsilon} \mathcal{O}(t-t')+\sqrt{\tau(t)-\tau(t')}\|h(\alpha,\gamma,t)-h(\alpha,\gamma,t')\|_{W^{i_0,k}}+\|h(\alpha,\gamma,t)\|_{C_{\gamma}^{i_0}([-1,1],H^{k}[-\tau(t),-\tau(t)+2(\tau(t)-\tau(t'))])}.
\end{align*}
From lemma \ref{continuitytl2} and corollary \ref{continuitytl202}, it tends to 0 as $t \to t'$ or $t'\to t $. Hence $M_{2,1}^{\epsilon}(h)$ is continuous in $t$ in space $W^{i_0,k}$.
The similar proof holds for $M_{1,2}^{\epsilon}(h)$ and $B_{M_i}^{\ep}(h)$. For $\sum_{i}B_i$, when $0\leq t\leq t'\leq t_s$, $\alpha>-\tau(t')$, we could use smoothness condition of $B_i(h)$ in \eqref{3Tequation2} and have
\begin{align}\label{Bicontinuous}
    &\quad \|\sum_{i}B_i(h)(\alpha,\gamma,t)-\sum_{i}B_i(h)(\alpha,\gamma,t')\|_{C_{\gamma}^{i_0}([-1,1],H^{k}[-\tau(t'),\pi])}\\\nonumber
    &\lesssim_{\epsilon} \mathcal{O}(t-t')+\|h(\alpha,\gamma,t)-h(\alpha,\gamma,t')\|_{W^{i_0,k}}.
\end{align}
When $-\tau(t)\leq \alpha\leq -\tau(t')$, we can use the second boundary conditions of $\sum_{i}B_i(h)$ in \eqref{3Tequation2} and have:
\[
\lim_{t'\to t}\|\sum_{i}B_i(h)(\alpha,\gamma,t)\|_ { C_{\gamma}^{i_{0}}([-1,1],H^k_{\alpha}[-\tau(t),\tau(t')])}=0,
\] and \[
\lim_{t\to t'}\|\sum_{i}B_i(h)(\alpha,\gamma,t)\|_ { C_{\gamma}^{i_{0}}([-1,1],H^k_{\alpha}[-\tau(t),\tau(t')]}=0.
\]
Then $\sum_{i}B_i(h)$ is continuous with respect to $t$ in $W^{i_0,k}.$ 
\end{proof}

\subsubsection{Control the good terms}
Let
\begin{align}\label{setU}
U=W^{0,k}\cap W^{1,k-1}...\cap W^{3,k-3}\cap{S},
\end{align}
with $\|h\|_{U}=\|h\|_{W^{0,k}}+\|h\|_{W^{1,k-1}}+\|h\|_{W^{2,k-1}}+\|h\|_{W^{3,k-3}}.$

Now we show the energy estimate in $U$ when $k=12$. In this section, our estimates are all independent of $\epsilon$.

From the structure of the perturbed operator $T^{\ep}(h)$, only $B_{M,i}^{\ep}$, $M_{1,1}^{\ep},$  and $M_{1,2}^{\ep}$ are dependent on $\ep$. For $B_{M,i}^{\epsilon}$, we have the following lemmas: 
\begin{lemma}\label{3T3goodbehaviour}
When $h\in W^{i_0,k}$, $0\leq j\leq i_0$, $0\leq k'\leq k$ we have 
\begin{align*}
\|\partial_{\al}^{k'}\partial_{\g}^{j} B_{M,i}^{\epsilon}(h)\|_{C^{0}_{\g}([-1,1],L_{\al}^{2}[-\tau(t),\pi])}\lesssim C(\|h\|_{W^{i_0,k}}).
\end{align*}
\end{lemma}
\begin{proof}
We only show the $j=0$ case. Other cases will be similar by putting more $\g$ derivatives on $D^{-60+b_i}(h)$ and $L_{B,i}$. 
From the definition in \eqref{02BMiperturb}, \begin{equation}
B_{M,i}^{\epsilon}(h)=\lambda(\al+\tau(t))\int_{-\tau(t)}^{2\al+\tau(t)}L_{B,i}(h)(\al,\beta,\g,t)\frac{D^{-60+b_i}(h)(\beta,\g,t)(\al-\beta)}{(\al-\beta)^2+\epsilon^2}d\beta.
\end{equation}
For $b_i=60$, we first consider $k'=0$,  we have
\begin{align*}
     &\quad\int_{-\tau(t)}^{2\alpha+\tau(t)}L_{B,i}(h)(\alpha,\beta,\gamma,t)\frac{h(\beta,\gamma,t)(\alpha-\beta)}{(\alpha-\beta)^2+\epsilon^2}d\beta\\
     &=\int_{-\tau(t)}^{2\alpha+\tau(t)}(L_{B,i}(h)(\alpha,\beta,\gamma,t)-L_{B,i}(h)(\alpha,\alpha,\gamma,t))\frac{h(\beta,\gamma,t)(\alpha-\beta)}{(\alpha-\beta)^2+\epsilon^2}d\beta\\
     &\quad+L_{B,i}(h)(\alpha,\alpha,\gamma,t)\int_{-\tau(t)}^{2\alpha+\tau(t)}\frac{h(\beta,\gamma,t)(\alpha-\beta)}{(\alpha-\beta)^2+\epsilon^2}d\beta\\
     &=Term_{1}+Term_{2}.
 \end{align*}
By the smoothness of $L_{B,i}(h)$ as in \eqref{3Tequation2}, we claim
\begin{align*}
    \|\frac{(L_{B,i}(h)(\alpha,\beta,\gamma,t)-L_{B,i}(h)(\alpha,\alpha,\gamma,t))(\alpha-\beta)}{(\alpha-\beta)^2+\epsilon^2}\|_{C_{\gamma}^{i_0}([-1,1], C^{0}((-\tau(t),\frac{\pi}{4})\times(-\tau(t),\frac{2\pi}{3})))}\lesssim 1.
\end{align*}
Then we have
\begin{align*}
    \|Term_1\|_{C_{\gamma}^{i_0}([-1,1],L^2[-\tau(t),\frac{\pi}{4}])}\lesssim 1.
\end{align*}
By the following lemma \ref{3HilbertL2}, we also get the similar bound for $Term_2$:
\begin{align*}
    \|Term_2\|_{C_{\gamma}^{i_0}([-1,1],L^2[-\tau(t),\frac{\pi}{4}])}\lesssim 1.
\end{align*}

Then we take $k$th derivative and have
\begin{align*}
    &\partial_{\alpha}^{k}\int_{-\tau(t)}^{2\alpha+\tau(t)}L_{B,i}(h)(\alpha,\beta,\gamma,t)\frac{h(\beta,\gamma,t)(\alpha-\beta)}{(\alpha-\beta)^2+\epsilon^2}d\beta\\
    &=\sum_{k'\leq k}\int_{-\alpha-\tau(t)}^{\alpha+\tau(t)}\partial_{\alpha}^{k'}(L_{B,i}(h)(\alpha,\alpha-\beta,\gamma,t))\frac{\partial_{\alpha}^{k-k'}h(\alpha-\beta,\gamma,t)(\beta)}{(\beta)^2+\epsilon^2}d\beta\\
    &+\sum_{k_1+k_2+k_3\leq k-1}C_{k_1,k_2,k_3}\partial_{\alpha}^{k_3}(\partial_{\alpha}^{k_1}(L_{B,i}(h)(\alpha,\alpha-\beta,\gamma,t))|_{\beta=\alpha+\tau(t)})\partial_{\alpha}^{k_2}(\frac{(\partial_{\alpha}^{k-1-k_1-k_2-k_3}h)(-\tau(t),\gamma,t)(\alpha+\tau(t))}{(\alpha+\tau(t))^2+\epsilon^2})\\
    &-\sum_{k_1+k_2+k_3\leq k-1}C'_{k_1,k_2,k_3}\partial_{\alpha}^{k_3}(\partial_{\alpha}^{k_1}(L_{B,i}(h)(\alpha,\alpha-\beta,\gamma,t))|_{\beta=-\alpha-\tau(t)})\\
    &\quad\cdot\partial_{\alpha}^{k_2}(\frac{(\partial_{\alpha}^{k-1-k_1-k_2-k_3}h)(2\alpha+\tau(t),\gamma,t)(\alpha+\tau(t))}{(\alpha+\tau(t))^2+\epsilon^2})\\
    &=Term_1^{k}+Term_2^{k}+Term_3^{k}.
\end{align*}
 Still from the conditions of $L_{B,i}(h)$, we have for $j+j'\leq k$,
\begin{align*}
    \|\frac{(\partial_{\alpha}^{j}\partial_{\beta}^{j'}L_{B,i}(h)(\alpha,\beta,\gamma,t)-\partial_{\alpha}^{j}\partial_{\beta}^{j'}L_{B,i}(h)(\alpha,\alpha,\gamma,t))(\alpha-\beta)}{(\alpha-\beta)^2+\epsilon^2}\|_{C_{\gamma}^{i_0}([-1,1], C^{0}((-\tau(t),\frac{\pi}{4})\times(-\tau(t),\frac{2\pi}{3})))}\lesssim 1.
\end{align*}
then $Term_{1}^{k}$ can be controlled in the same way as in the $L^2$ case.
$Term_{2}^{k}$ is 0 from \eqref{hboundray}.  $Term_{3}^{k}$ can be controlled by lemma \ref{farboundarybound}.

When $b_i<60$, we can use a similar proof because of the following estimate from \eqref{D-formularnew}: 

\begin{align}\label{3Thresult3}
D^{-60+b_i}(h(\alpha,\gamma,t))\in W^{i_0,k+(60-b_i)}.
\end{align}

\end{proof}
\begin{lemma}\label{3HilbertL2}
If $h(\alpha)\in L^{2}[-\pi,\pi]$, with $h(\alpha)=0$ when $\alpha<-\tau(t)$, we have
\begin{align*}
    \|\int_{-\tau(t)}^{2\alpha+\tau(t)}\frac{h(\beta)(\alpha-\beta)}{(\alpha-\beta)^2+\epsilon^2}d\beta\|_{L^{2}[-\tau(t),\frac{\pi}{4}]}\lesssim\|h(\alpha,\gamma,t)\|_{L^{2}[-\tau(t),\pi]}.
\end{align*}
\end{lemma}

\begin{proof}
Let
\begin{align*}
    h_{ex}(\alpha)=\bigg\{\begin{array}{cc}
 h(\alpha)&  -\pi\leq \alpha \leq \pi,\\
          0 & \alpha<-\pi \text{ or } \alpha>\pi,
      \end{array}
\end{align*}
\begin{align*}
G(\alpha)=\int_{2\alpha+\tau(t)}^{\infty}\frac{h(\beta)(\alpha-\beta)}{(\alpha-\beta)^2+\epsilon^2}d\beta.
\end{align*}
Then 
\begin{align*}
    |G(\alpha-\tau(t))|=|\int_{2\alpha}^{\infty}\frac{h(\beta-\tau(t))(\alpha-\beta)}{(\alpha-\beta)^2+\epsilon^2}d\beta|\lesssim \int_{2\alpha}^{\infty}|\frac{h_{ex}(\beta-\tau(t))}{(\alpha+\beta)}|d\beta.
\end{align*}
By  Hilbert's Inequality, we have
\begin{align*}
    \|G(\alpha-\tau(t))\|_{L^2[0,\infty)}\lesssim \|h_{ex}(\alpha)\|_{L^2[-\tau(t),\infty)}.
\end{align*}
Moreover, since 
\[
H_{\mathcal{R}}(\frac{\epsilon}{\beta^2+\epsilon^2})=\frac{\beta}{\beta^2+\epsilon^2},
\]
we have
\[
\int_{-\infty}^{\infty}\frac{h_{ex}(\beta)(\alpha-\beta)}{(\alpha-\beta)^2+\epsilon^2}d\beta=-\int_{-\infty}^{\infty}H_{\mathcal{R}}(h_{ex})(\beta)\frac{\epsilon}{(\alpha-\beta)^2+\epsilon^2}d\beta.
\]
Here $H_{\mathcal{R}}$ is the Hilbert transform on the real line. Hence
\begin{align*}
   \|\int_{-\infty}^{\infty}\frac{h_{ex}(\beta)(\alpha-\beta)}{(\alpha-\beta)^2+\epsilon^2}d\beta\|_{L^2[-\tau(t),\frac{\pi}{4}]}\lesssim \|h_{ex}\|_{L^2(-\infty,\infty)}\lesssim\|h\|_{L^2[-\tau(t),\pi]}.
\end{align*}
Since 
\begin{align*}
    &\int_{-\infty}^{\infty}\frac{h_{ex}(\beta)(\alpha-\beta)}{(\alpha-\beta)^2+\epsilon^2}- G(\alpha)=\int_{-\tau(t)}^{2\alpha+\tau(t)}\frac{h(\beta)(\alpha-\beta)}{(\alpha-\beta)^2+\epsilon^2}d\beta,
\end{align*}
we get the result.
\end{proof}

Now we introduce some corollaries concerning $M_{1,2}^{\ep}(h)$ and $M_{2,1}^{\ep}(h).$ 
As in the lemmas \ref{coro22epm12} and \ref{coro22ep}, \ref{30101} for the case when $\kappa(t)>0$, we have the following estimates:
\begin{lemma}\label{coro22epmodel2}
For any $k\geq 1$, if $h\in W^{0,k}$, we have 
\begin{equation}\label{3coro22ep02model2}
\begin{split}
&\|(\al+\tau(t))\partial_{\al}^{k}\int_{-\tau(t)}^{2\al+\tau(t)}\frac{h(\al)-h(\beta)}{(\al-\beta)^2+\ep^2}d\beta-(\al+\tau(t))\int_{-\tau(t)}^{2\al+\tau(t)}\frac{h^{(k)}(\al)-h^{(k)}(\beta)}{(\al-\beta)^2+\ep^2}d\beta\|_{L_{\al}^{2}[-\tau(t),\frac{\pi}{4}]}\lesssim \|h\|_{H_{\al}^{k}[-\tau(t),\pi]}.
\end{split}
\end{equation}
Moreover, for any $l\leq k-1$, we have
\begin{equation}\label{3coro22ep0202model2}
\begin{split}
&\|\partial_{\al}^{l}\int_{-\tau(t)}^{2\al+\tau(t)}\frac{h(\al)-h(\beta)}{(\al-\beta)^2+\ep^2}d\beta-\int_{-\tau(t)}^{2\al+\tau(t)}\frac{\partial_{\al}^{l}h(\al)-\partial_{\beta}^{l}h(\beta)}{(\al-\beta)^2+\ep^2}d\beta\|_{L_{\al}^{2}[-\tau(t),\frac{\pi}{4}]}\lesssim \|h\|_{H_{\al}^{k}[-\tau(t),\pi]},
\end{split}
\end{equation}
and
\begin{equation}\label{3coro22ep01model2}
\|\partial_{\al}^{l}\int_{-\tau(t)}^{2\al+\tau(t)}\frac{h(\al)-h(\beta)}{(\al-\beta)^2+\ep^2}d\beta\|_{L_{\al}^{2}[-\tau(t),\frac{\pi}{4}]}\lesssim \|h\|_{H_{\al}^{k}[-\tau(t),\pi]}.
\end{equation}
\end{lemma}
\begin{proof}
This follows the same proof as in lemma \ref{coro22ep}. Since $h\in W^{0,k}$, we have $\partial_{\alpha}^{j}h|_{\al=-\tau(t)}=0$ when $j \leq k-1$.
\end{proof}
\begin{lemma}\label{coro22epm12model2}
For any $k\geq 1$, if $h\in W^{0,k}$, we have 
\begin{equation}\label{3coro22ep020102model2}
\begin{split}
&\|(\al+\tau(t))\partial_{\al}^{k}\int_{-\tau(t)}^{2\al+\tau(t)}\frac{(h(\al)-h(\beta))(\al-\beta)\ep}{((\al-\beta)^2+\ep^2)^2}d\beta-(\al+\tau(t))\int_{-\tau(t)}^{2\al+\tau(t)}\frac{(h^{(k)}(\al)-h^{(k)}(\beta))(\al-\beta)\ep}{((\al-\beta)^2+\ep^2)^2}d\beta\|_{L_{\al}^{2}[-\tau(t),\frac{\pi}{4}]}\\
&\lesssim \|h\|_{H_{\al}^{k}[-\tau(t),\pi]}.
\end{split}
\end{equation}
Moreover, for any $l\leq k-1$, we have
\begin{equation}\label{3coro22ep020202}
\begin{split}
&\|\partial_{\al}^{l}\int_{-\tau(t)}^{2\al+\tau(t)}\frac{(h(\al)-h(\beta))(\al-\beta)\ep}{((\al-\beta)^2+\ep^2)^2}d\beta-\int_{-\tau(t)}^{2\al+\tau(t)}\frac{(\partial_{\al}^{l}h(\al)-\partial_{\beta}^{l}h(\beta))(\al-\beta)\ep}{((\al-\beta)^2+\ep^2)^2}d\beta\|_{L_{\al}^{2}[-\tau(t),\frac{\pi}{4}]}\\
&\lesssim \|h\|_{H_{\al}^{k}[-\tau(t),\pi]},
\end{split}
\end{equation}
and
\begin{equation}\label{3coro22ep010202}
\|\partial_{\al}^{l}\int_{-\tau(t)}^{2\al+\tau(t)}\frac{(h(\al)-h(\beta))(\alpha-\beta)\epsilon}{((\al-\beta)^2+\ep^2)^2}d\beta\|_{L_{\al}^{2}[-\tau(t),\frac{\pi}{4}]}\lesssim \|h\|_{H_{\al}^{k}[-\tau(t),\pi]},
\end{equation}
\end{lemma}
\begin{proof}
This follows the same proof as in lemma \ref{coro22epm12}. Since $h\in W^{0,k}$, we have $\partial_{\alpha}^{j}h|_{\al=-\tau(t)}=0$ when $j \leq k-1$.
\end{proof}
\begin{lemma}\label{3lemmaotherterms03}
For $1\leq k\leq 12$, $h\in W^{0,k}$, $g\in H_{\alpha}^{k}[-\pi,\pi]\cap\{\supp g\in[-\tau(t),\frac{\pi}{8}]\}$, we have 
\begin{align}\label{3boundedmain02}
\|&\partial_{\alpha}^{k}(\lambda(\alpha+\tau(t))(L_2^{-}(h)(\alpha)\int_{-\tau(t)}^{2\alpha+\tau(t)}\frac{g(\alpha)-g(\beta)}{(\alpha-\beta)^2+\epsilon^2}d\beta+
\int_{-\tau(t)}^{2\alpha+\tau(t)}\frac{(g(\alpha)-g(\beta))(\alpha-\beta)\epsilon}{((\alpha-\beta)^2+\epsilon^2)^2}d\beta \frac{2}{\pi}L_1^{-}(h)(\alpha)))\\\nonumber
&-\lambda(\alpha+\tau(t))L_2^{-}(h)(\alpha)\int_{-\tau(t)}^{2\alpha+\tau(t)}\frac{g^{(k)}(\alpha)-g^{(k)}(\beta)}{(\alpha-\beta)^2+\epsilon^2}d\beta\\\nonumber
&-\lambda(\alpha+\tau(t))\int_{-\tau(t)}^{2\alpha+\tau(t)}\frac{(g^{(k)}(\alpha)-g^{(k)}(\beta))(\alpha-\beta)\epsilon}{((\alpha-\beta)^2+\epsilon^2)^2}d\beta \frac{2}{\pi}L_1^{-}(h)(\alpha))\|_{W^{0,0}}
\leq C(\|h\|_{W^{0,k}})\|g\|_{ H_{\alpha}^{k}[-\pi,\pi]},
\end{align}
and
\begin{equation}\label{3boundedB2}
\|\sum_{i}B_i(h)\|_{W^{0,k}}\lesssim C(\|h\|_{W^{0,k}}), \quad \|\sum_{i}B_{M,i}(h)\|_{W^{0,k}}\lesssim C(\|h\|_{W^{0,k}}) .
\end{equation}
For $0\leq j\leq k-1$, we have
\begin{equation}\label{3boundedlower02}
\begin{split}
&\quad\|\lambda(\alpha+\tau(t))(L_2^{-}(h)(\alpha)\int_{-\tau(t)}^{2\alpha+\tau(t)}\frac{g(\alpha)-g(\beta)}{(\alpha-\beta)^2+\epsilon^2}d\beta+
\int_{-\tau(t)}^{2\alpha+\tau(t)}\frac{(g(\alpha)-g(\beta))(\alpha-\beta)\epsilon}{((\alpha-\beta)^2+\epsilon^2)^2}d\beta \frac{2}{\pi}L_1^{-}(h)(\alpha))\|_{W^{0,j}}\\
&\lesssim C(\|h\|_{W^{0,k}})\||g\|_{ H_{\alpha}^{k}[-\pi,\pi]}.
\end{split}
\end{equation} 
\end{lemma}
\begin{proof}
\eqref{3boundedlower02} ,\eqref{3boundedmain02} follows from the conditions of $L_i^{-}(h)$ and the estimates \eqref{3coro22ep02model2}, \eqref{3coro22ep01model2} in lemma \ref{coro22epmodel2}, estimates \eqref{3coro22ep020102model2}, \eqref{3coro22ep010202} in lemma \ref{coro22epm12model2}.
  
  \eqref{3boundedB2} follows from the condition of $B_i$ in \eqref{3Tequation2} and estimates for $B_{M,i}$ in lemma \ref{3T3goodbehaviour}.
\end{proof}
\begin{corollary}\label{3010102}
 If $k\geq 1$, $g\in L_{\al}^{2}[-\tau(t),\pi]\cap \{g(x)=0|x\geq \frac{\pi}{2}\}$,  $L_1^{-}(h)$, $L_2^{-}(h)$ satisfying conditions the following conditions:

\[
 \quad 18 |\Im L_{1}^{-}(h)(\al,\g,t)|+18 |\Im L_{2}^{-}(h)(\al,\g,t)|\leq -\Re L_2^{-}(h)(\al,\g,t), \text{ when } \lambda(\alpha+\tau(t))\neq 0, 
\]
\[
L_{i}^{-}(h)(-\tau(t),\g,t)=0,
\]
\[
\|L_{i}^{-}(h)(\alpha,\g,t)\|_{C_{\alpha}^{2}[-\tau(t),\frac{\pi}{4}]}\lesssim 1,
\]
 then we have
\begin{equation}\label{3linearcoro0103}
\begin{split}
\Re &<g, \lambda(\al+\tau(t))L_1^{-}(h)(\al)\int_{-\tau(t)}^{2\al+\tau(t)}\frac{(g(\al)-g(\beta))(\al-\beta)\ep}{((\al-\beta)^2+\ep^2)^2}d\beta\frac{2}{\pi}\\
&+\lambda(\alpha+\tau(t))L_2^{-}(h)(\alpha)\int_{-\tau(t)}^{2\alpha+\tau(t)}\frac{(g(\alpha)-g(\beta))}{(\alpha-\beta)^2+\epsilon^2}d\beta>_{L_{\al}^{2}[-\tau(t),\pi]}\\
&\lesssim \|g\|_{L_{\al}^{2}[-\tau(t),\pi]}^{2}.
\end{split}
\end{equation}
\end{corollary}
\begin{proof}
We could use lemma \ref{30101}. By letting $g^{\tau}(\alpha)=g(\alpha-\tau(t))$, $L_{i}^{-,\tau}(h)(\alpha)=L_{i}^{-}(h)(\alpha-\tau(t))$, we have
\begin{equation*}
\begin{split}
\Re &<g^{\tau}, \lambda(\al)L_1^{-,\tau}(h)(\al)\int_{0}^{2\al}\frac{(g^{\tau}(\al)-g^{\tau}(\beta))(\al-\beta)\ep}{((\al-\beta)^2+\ep^2)^2}d\beta\frac{2}{\pi}\\
&+\lambda(\alpha)L_2^{-,\tau}(h)(\al)\int_{0}^{2\alpha}\frac{(g^{\tau}(\alpha)-g^{\tau}(\beta))}{(\alpha-\beta)^2+\epsilon^2}d\beta>_{L_{\al}^{2}[0,\pi]}\\
&\lesssim \|g^{\tau}\|_{L_{\al}^{2}[0,\pi]}^{2}.
\end{split}
\end{equation*}
 Since $\supp \lambda(\alpha+\tau(t))\subset [-\tau(t),\frac{\pi}{2}]$ when $\delta$ \eqref{lambdadefi} and $t$ \eqref{defkappa} are sufficiently small,  we can change $\alpha$ to $\alpha+\tau(t)$ to get \eqref{3linearcoro0102}.
\end{proof}
\begin{corollary}\label{3010102cor}
 If $k\geq 1$ , $g\in H_{\alpha}^{k}[-\pi,\pi]\cap\{\supp g\in[-\tau(t),\frac{\pi}{8}]\}$, $h\in W^{0, k}$ with $L_1^{-}(h)$, $L_2^{-}(h)$ satisfying conditions in \eqref{3Tequation2},  we have
\begin{align*}
\Re &<h, \lambda(\al+\tau(t))L_1^{-}(h)(\alpha)\int_{-\tau(t)}^{2\al+\tau(t)}\frac{(g(\al)-g(\beta))(\al-\beta)\ep}{((\al-\beta)^2+\ep^2)^2}d\beta\frac{2}{\pi}\\
&-\lambda(\alpha+\tau(t))L_2^{-}(h)(\alpha)\int_{-\tau(t)}^{2\alpha+\tau(t)}\frac{(g(\alpha)-g(\beta))}{(\alpha-\beta)^2+\epsilon^2}d\beta \frac{2}{\pi}>_{H_{\al}^{k}[-\tau(t),\pi]}\\
&\lesssim \|g\|_{H^{k}[-\tau(t),\pi]}^2 C(\|h\|_{C^{0}_{\g}([-1,1], H^{k}[-\tau(t),\pi])}).
\end{align*}
\end{corollary}

\begin{proof}
We can use corollary \ref{3lemmaotherterms03} and \ref{3010102}, conditions of $L_i^{-}(h)$ in \eqref{3Tequation2} to get the result.
\end{proof}
\begin{corollary}\label{3010102limit}
For $g\in H_{\alpha}^{1}[-\pi,\pi]\cap\{\supp g\in[-\tau(t),\frac{\pi}{8}]\}$, $h\in W^{0,1}$ with $L_1^{-}(h)$, $L_2^{-}(h)$ satisfying conditions in \eqref{3Tequation2},  we have for $k_1=0,1$
\begin{align*}
\Re &<h, \lambda(\al+\tau(t))L_1^{-}(h)(\alpha)\frac{dg(\alpha)}{d\alpha}\\
&\quad-\lambda(\alpha+\tau(t))L_2^{-}(h)(\alpha)\int_{-\tau(t)}^{2\alpha+\tau(t)}\frac{(g(\alpha)-g(\beta))}{(\alpha-\beta)^2}d\beta \frac{2}{\pi}>_{H_{\al}^{k_1}[-\tau(t),\pi]}\\
&\lesssim \|g\|_{H^{k_1}[-\pi,\pi]}^2 C(\|h\|_{W^{0,1}}).
\end{align*}
\begin{proof}
    We could take the limit of corollary \ref{3010102}.
\end{proof}
\end{corollary}
\subsubsection{Energy estimate}
Now we do the energy estimate.
\begin{lemma}
For $h\in U$ \eqref{setU}, we have
\[
\Re<h,T^{\epsilon}(h)>_{U}\lesssim C(\|h\|_{U}).
\]
\end{lemma}
\begin{proof}
From estimate of $B_{M,i}$ (lemma \ref{3T3goodbehaviour})  and conditions of $B_i(h)$ \eqref{3Tequation2}, we only need to consider $M_{1,1}^{\epsilon}(h)$ and $M_{1,2}^{\epsilon}(h)$:

\[
\Re<h,M_{1,1}^{\ep}(h)+M_{1,2}^{\ep}(h)>_{U}\lesssim C(\|h\|_{U}).
\] 
Recall that we have
\[
U=W^{0,k}\cap W^{1,k-1}...\cap W^{3,k-3}\cap{S}.
\]
For $W^{0,k}$, from corollary \ref{3010102}, lemma \ref{3lemmaotherterms03}, we have

%\Re<h,M_{1,1}^{\ep}(h)+M_{1,2}^{\ep}(h)>_{U}\lesssim C(\|h\|_{U}).
%\]

\begin{align}\label{zeroenerygestimate}
\Re<h,M_{1,1}^{\ep}(h)+M_{1,2}^{\ep}(h)>_{W^{0,k}}\lesssim C(\|h\|_{W^{0,k}}).
\end{align}
Then we consider the estimate in $U$. This follows from its structure.

In fact, for $h\in W^{i,k-i}$ with $1\leq i\leq 3$, we have
\begin{align*}
&\Re<\partial_{\g}^{i}h,\partial_{\g}^{i}(M_{1,1}^{\ep}(h)+M_{1,2}^{\ep}(h))>_{W^{0,k-i}}\\
&\lesssim \Re <\partial_{\g}^{i}h,\lambda(\al+\tau(t))L_1^{-}(h)(\al,\g)\int_{-\tau(t)}^{2\al+\tau(t)}\frac{(\partial_{\g}^{i}h(\al)-\partial_{\g}^{i}h(\beta))(\al-\beta)\ep}{((\al-\beta)^2+\ep^2)^2}d\beta\frac{2}{\pi}\\
&\qquad-\lambda(\al+\tau(t))L_2^{-}(h)(\al,\g)\int_{-\tau(t)}^{2\al+\tau(t)}\frac{(\partial_{\g}^{i}h(\al)-\partial_{\g}^{i}h(\beta))\ep}{(\al-\beta)^2+\ep^2}d\beta)>_{W^{0,k}}\\
&\quad+ \sum_{1\leq j\leq i}\Re <\partial_{\g}^{i}h,(\lambda(\al+\tau(t))(\partial_{\g}^{j}L_1^{-}(h)(\alpha,\gamma))\int_{-\tau(t)}^{2\al+\tau(t)}\frac{(\partial_{\g}^{i-j}h(\al)-\partial_{\g}^{i-j}h(\beta))(\al-\beta)\ep}{((\al-\beta)^2+\ep^2)^2}d\beta\frac{2}{\pi}\\
&\qquad-\lambda(\al+\tau(t))(\partial_{\g}^{j}L_2^{-}(h)(\alpha,\gamma))\int_{-\tau(t)}^{2\al+\tau(t)}\frac{(\partial_{\g}^{i-j}h(\al)-\partial_{\g}^{i-j}h(\beta))\ep}{(\al-\beta)^2+\ep^2}d\beta)>_{W^{0,k-1}}\\
&=E_{i,k,1}+E_{i,k,2}.
\end{align*}
For $E_{i,k,1}$, it can be controlled similarly as in corollary \ref{3010102},  lemma \ref{3lemmaotherterms03} by replacing $h$ with $\partial_{\g}^{i}h$. For $E_{i,k,2}$, since at most $i-1$ th $\g$ derivative hit on $h$, we have
\[
|E_{i,k,2}|\lesssim \|h\|_{W^{i-1,k-i+1}}^2C(\|h\|_{U}),
\]
where we use \eqref{3coro22ep010202}, and \eqref{3coro22ep01model2} by taking $k$ as $k-i+1$ and $h$ as $\partial_{\g}^{i-j}h$ and smoothness conditions of $L_i^{-}(h)$ in \eqref{3Tequation2}.
Therefore, we have
\begin{align}\label{gammeenerygestimate}
\Re <h, M_{1,1}^{\epsilon}(h)+M_{1,2}^{\epsilon}(h)>_{W^{i,k-i}}\lesssim C(\|h\|_{U}),
\end{align}
for $1\leq i \leq 3$.
Then from \eqref{zeroenerygestimate} and \eqref{gammeenerygestimate} we have
\begin{equation}\label{energyestimatefinal}
\frac{d}{dt}\|h\|_{U}^{2}\lesssim C(\|h\|_{U}),
\end{equation}
independent of $\epsilon$.
\end{proof}
\subsubsection{Picard's theorem and the limit}
\begin{lemma}
 For $\epsilon$ sufficiently small, there exists $t^{\epsilon}$, and $h^{\epsilon} \in  U\cap{S_{\delta_s}}$, with 
\[
S_{\delta_s}=\{h|\supp h \in [-\tau(t),\frac{\pi}{8}]\}\cap\{\|h(\alpha,\gamma,t)-h(\alpha,\gamma,0)\|_{W^{0,1}[-\tau(t),\frac{\pi}{8}]}
\leq \delta_s\}.
\]
such that  $h^{\epsilon}(\alpha,\gamma,t)\in C^{1}([0,t_{\epsilon}], U)$, and
\begin{align*}
&\bigg\{\begin{array}{cc}
\frac{d}{dt}h^{\epsilon}(\alpha,\gamma,t)=T^{\epsilon}(h^{\epsilon})(\alpha,\gamma,t), \\
        h^{\epsilon}(\alpha,\gamma,0)=\tilde{f}^{+(60)}(\alpha,0).
      \end{array}
\end{align*}
\begin{proof}
We use Picard's theorem. Let
\begin{align*}
&\bigg\{\begin{array}{cc}
h_{n+1}^{\epsilon}(\alpha,\gamma,t)=\int_{0}^{t}T^{\epsilon}(h_n^{\epsilon})(\alpha,\gamma,\tau)d\tau+\tilde{f}^{+}(\alpha,0), \\
        h_0^{\epsilon}(\alpha,\gamma,t)=0.
      \end{array}
\end{align*}
By  lemma \ref{Tepsilonbehavior}, the only non-trivial step is to show when $h_{n}^{\epsilon}\in S=\{h|\supp h \in [-\tau(t),\frac{\pi}{8}]\}$, so does $h_{n+1}^{\epsilon}.$  Since from \eqref{perturbT1}, we have $T^{\epsilon}(h_n^{\epsilon})(\alpha,\gamma,t) \in  \{h|\supp h \in [-\tau(t),\frac{\pi}{8}]\}$. Then $h_{n+1}^{\epsilon}\in S$ follows from the fact that $-\tau(t)$ is decreasing when $\kappa(t)<0$ \eqref{defkappa}. 
\end{proof}
\end{lemma}

From the energy estimate \eqref{energyestimatefinal}, $h^{\epsilon}(\alpha,\gamma,t)$ has a uniform bound in $U$. By using compactness argument, there exists sufficiently small time $\bar{t}$ such that $h\in C^{1}_{t}([0,\bar{t}], C^{1}_{\gamma}([-1,1], H^5_{\alpha}[-\pi,\pi]))\cap S_{\delta_s}$ with
 \begin{align*}
    &\bigg\{\begin{array}{cc} \frac{dh}{dt}(\alpha,\gamma,t)=T^{-}(h).\\
    h(\alpha,\gamma,0)=\tilde{f}^{+(60)}(\alpha,0).
    \end{array}
 \end{align*}
where $T^{-}(h)$ is from \eqref{3Tequation2}.
\subsection{Uniqueness}\label{huni}
We now show the uniqueness \eqref{uniquenesskappanegi}: 
\begin{equation}\label{3uniqh02}
    h(\al,0,t)=g(\alpha,t),
\end{equation}
 when $0\leq t\leq \bar{t}$.
We have
\begin{align}
\frac{dh(\al,0)}{dt}=T^{-}(h)(\al,0)=\begin{cases}
    M_{1,2}(h)(\al,0)+M_{2,1}(h)(\al,0)+\sum_{i}B_{M,i}(h)(\al,0)+\sum_{i=1}B_i(h)(\al,0) &\al>-\tau(t)\\
    0 &\al\leq -\tau(t),
\end{cases}
\end{align}
and \begin{align}
\frac{dg(\al)}{dt}=T^{-}(g)(\al,0)=\begin{cases}
    M_{1,2}(g)(\al,0)+M_{2,1}(g)(\al,0)+\sum_{i}B_{M,i}(g)(\al,0)+\sum_{i=1}B_i(\al,0) &\al>-\tau(t)\\
    0 &\al\leq -\tau(t),
\end{cases}
\end{align}
Here $T^{-}(g)(\al,0)$ is well-defined because when $\g=0$, from conditions in \eqref{3Tequation2}, $T^{-}(g)$ only depends on the value of $g$ on the real line.

Since when $t=0$, $g(\al)-h(\al,0)|_{t=0}=0$, it is sufficient to show 
\begin{equation}\label{3unibound102}
\begin{split}
    \Re<h(\al,0)-g(\al), \frac{d}{dt}(h(\al,0)-g(\al))>_{H_{\al}^{1}[-\pi,\pi]}\lesssim \|g(\al)-h(\al,0)\|_{H_{\al}^{1}[-\pi,\pi]}^2.
\end{split}
\end{equation}
We have when $\al\geq -\tau(t),$
\begin{equation}
    \begin{split}
    &\frac{d(h(\al,0)-g(\al))}{dt}=\lambda(\al+\tau(t))L_1^{-}(h)(\al,0)(h'(\al,0)-g'(\al))\\
    &+\lambda(\al+\tau(t))L_2^{-}(h)(\al,0)\int_{-\tau(t)}^{2\al+\tau(t)}\frac{(h(\al,0)-g(\al))-(h(\beta,0)-g(\beta)))}{(\al-\beta)^2}d\beta\\
    &+\lambda(\al+\tau(t))(L_1^{-}(h)(\al,0)-L_1^{-}(g)(\al,0))g'(\al)\\
    &+\lambda(\al+\tau(t))(L_2^{-}(h)(\al,0)-L_2^{-}(g)(\al,0))\int_{-\tau(t)}^{2\al+\tau(t)}\frac{g(\al)-g(\beta)}{(\al-\beta)^2}d\beta\\
      &+\sum_{i}(B_{M,i}(h)(\al,0)-B_{M,i}(g)(\al,0))\\
    &+\sum_{i}(B_i(h)(\al,0)-B_i(g)(\al,0)).
    \end{split}
\end{equation}
Then from conditions in section \eqref{3Tequation2}, we have
\begin{equation}\label{3uniqueness1}
    \begin{split}
        &\quad\Re<h(\al,0)-g(\al), \frac{d}{dt}(h(\al,0)-g(\al))>_{H^1_{\al}[-\tau(t),\pi]}\lesssim\\
        &\Re<h(\al,0)-g(\al),\lambda(\al+\tau(t))L_1^{-}(h)(\al,0)(h'(\al,0)-g'(\al))\\
        &+\lambda(\al+\tau(t))L_2^{-}(h)(\al,0)\int_{-\tau(t)}^{2\al+\tau(t)}\frac{(h(\al,0)-g(\al))-(h(\beta,0)-g(\beta))}{(\al-\beta)^2}d\beta>_{H^1_{\al}[-\tau(t),\pi]}\\
        &+\|h(\al,0)-g(\al)\|_{H^1_{\al}[-\tau(t),\pi]}^2.
    \end{split}
\end{equation}
By using lemma \ref{3010102limit}, we have \eqref{3unibound102}. Then by Gronwall's inequality, we have the uniqueness \eqref{3uniqh02}.

\section{Analyticity when $\kappa(t)<0$}\label{analyticity-section}
  From \ref{theoremgeneralizedequa02}, \ref{existence1theorem2}, we have the existence of solution $h$ to equation \eqref{modifiedcondition} : there exists $\bar{t}$ such that $h$ satisfies
\begin{align}\label{hspaceconclusion2}
    &h(\al,\g,t)\in C^{1}_{t}([0,\bar{t}], C^{1}_{\gamma}([-1,1], H^5_{\alpha}[-\pi,\pi]))\cap\{h|\supp_{\alpha}h\subset[-\tau(t),\frac{\pi}{4}]\}
\end{align}
with initial data $h(\alpha,\gamma,0)=\tilde{f}^{+(60)}(\alpha,0).$
Moreover, from our derivation of \eqref{modifiedcondition}, $\tilde{f}^{+(60)}(\alpha+\tau(t),t)$ is also a equation of $\eqref{modifiedequation}$ when $\gamma=0$. Then from the uniqueness in theorem, \ref{existence1theorem2}, we have for $t\leq \bar{t}$, 
\begin{equation}\label{3uniqh2}
    h(\al,0,t)=\tilde{f}^{+(60)}(\al+\tau(t),t).
\end{equation}

Now we show the analyticity in the region $\{\alpha+iy|-\tau(t)\leq\alpha\leq\frac{\pi}{4}-\tau(t), |y|\leq C(\alpha+\tau(t))t\}$. Let  $h_{\tau}(\alpha,\gamma,t)=h(\alpha-\tau(t),\gamma,t)$. As in lemma \ref{analyticitylemmao1}, it is enough to show 

\begin{align}\label{analyticityenegyestimate02}
\Re<A(h_{\tau}), \frac{dA(h_{\tau})}{dt}>_{X^{1}}\lesssim \|A(h_{\tau})\|^{2}_{X^{1}}.
\end{align}
From our construction in \eqref{extendedequation01} and definition in \eqref{defkappa}, we claim for $\kappa(t)<0$, $h_{\tau}$ satisfies the same equation as the case $\kappa(t)>0$ except the sign of $\kappa(t)$ is different. Moreover, $h_{\tau}$ satisfies the conditions \eqref{hspaceconclusion}, \eqref{initialdatacondition01}, \eqref{3uniqh}. Then from the same proof as in lemma \ref{analyenergyestimate}, we have:

\begin{align}\label{analyticityestimateG2}
&\quad\frac{dA(h_{\tau})(\alpha,\gamma,t)}{dt}\\\nonumber
&=\kappa(t)\lambda(\al_{\gamma}^{t})\frac{dA(h_{\tau})(\alpha)}{d\alpha}+\lambda(\al_{\gamma}^{t})(L_1^{+}(h_{\tau})(\alpha))\frac{dA(h_{\tau})(\alpha)}{d\al}\\
&\quad +\lambda(\al_{\gamma}^{t})(L_2^{+}(h_{\tau})(\alpha))p.v.\int_{0}^{2\alpha}\frac{A(h_{\tau})(\beta)-A(h_{\tau})(\alpha)}{(\alpha-\beta)^2}d\beta+ B.T^{0},
\end{align}
where 
\[
\|B.T^{0}\|_{X^{1}}\lesssim \|A(h_{\tau})\|_{X^{1}}.
\]
Different from $\kappa(t)>0$ case, we can not use \eqref{M11estimateenergy}. But from the space of $h$ \eqref{hspaceconclusion02}, we have the boundary value: 
\[
\partial_{\alpha}^{j}h_{\tau}(0,\gamma,t)=0, \quad \partial_{\alpha}^{j}\partial_{\gamma}h_{\tau}(0,\gamma,t)=0,
\]
and 
\[
\partial_{\alpha}^{j}h_{\tau}(\pi,\gamma,t)=0, \quad \partial_{\alpha}^{j}\partial_{\gamma}h_{\tau}(\pi,\gamma,t)=0,
\]
when $0\leq j \leq 2$.
Then from integration by parts, since both $\kappa(t)$ \eqref{kappadefi}, $\lambda(\al_{\gamma}^{t})$ \eqref{lambda0defi} are real, we have
\[
\Re<A(h_\tau),\kappa(t)\lambda(\al_{\gamma}^{t})\frac{dh_{\tau}(\alpha)}{d\alpha}>_{X^{1}} =B.T^{0}.
\]
Then from \eqref{L12estimateenergy}, we have the estimate.
\section{Appendix}

The following lemmas hold for both signs of $\kappa(t)$. $\tau(t)$ is defined as in \eqref{defkappa}.
 \subsection{Arc-chord condition}
  We first introduce a lemma to guarantee the Arc-chord condition, when $\|h\|_{C^2_{\alpha}}\lesssim \delta_s$, with $\delta_s$ sufficiently small. 
\begin{lemma}\label{arcchord}
 If   $\|g\|_{C^2_{\alpha}}$, t, $\delta$ in \eqref{lambda0defi} are sufficiently small, $g(\alpha)$ satisfying $\partial_{\alpha}^{i}g(-\tau(t))=0$ for $i=0,1$ and $g(\beta)=0$ when $\beta\leq -\tau(t)$, we have the following inequalities. Here $\tau(t)$ is defined in \eqref{defkappa}, $\alpha_{\gamma}^{t}$ is defined in \eqref{alphadefi}, $\tilde{\alpha}_{\gamma}^{t}$ is defined in \eqref{definitiontildecalpha}, $\tilde{f}^{+}, \tilde{f}^{L}$ are defined in \eqref{cut function0},  \eqref{fLequation2}, $\tilde{f}^{L,L},\tilde{f}^{-}$ are defined in \eqref{fLequation2new}.
 
 When $-1\leq \gamma\leq 1$, $|\eta|\leq |\gamma|$,  $\frac{\pi}{4}\geq \alpha> -\tau(t)$, $\beta\in [-\pi,\pi],$
 \begin{align}\label{arcchord0}
    &|\cosh(\tilde{f}_2(\alpha)-\tilde{f}_2(\beta))-\cos(\tilde{f}_1(\alpha)-\tilde{f}_1(\beta))\gtrsim(\alpha-\beta)^2.
\end{align}
\begin{equation}\label{arcchord1}
|\cosh(\tilde{f}^{L}_2(\al)-\tilde{f}^{L}_2(\beta))-\cos(\tilde{f}^{L}_1(\al)-\tilde{f}^{L}_1(\beta))|\gtrsim
(\alpha-\beta)^2.
\end{equation}
\begin{equation}\label{arcchord1L}
|\cosh(\tilde{f}^{L,L}_2(\al)-\tilde{f}^{L,L}_2(\beta))-\cos(\tilde{f}^{L,L}_1(\al)-\tilde{f}^{L,L}_1(\beta))|\gtrsim
(\alpha-\beta)^2.
\end{equation}
\begin{equation}\label{arcchord2}
|\cosh(\tilde{f}^{L}_2(\alpha_{\gamma}^{t})-\tilde{f}^{L}_2(\beta_{\gamma}^{t}))-\cos(\tilde{f}^{L}_1(\alpha_{\gamma}^{t})-\tilde{f}^{L}_1(\beta_{\gamma}^t))|\gtrsim
(\alpha-\beta)^2.
\end{equation}
\begin{equation}\label{arcchord2L}
|\cosh(\tilde{f}^{L,L}_2(\alpha_{\gamma}^{t})-\tilde{f}^{L,L}_2(\beta_{\gamma}^t)))-\cos(\tilde{f}^{L,L}_1(\alpha_{\gamma}^{t})-\tilde{f}^{L,L}_1(\beta_{\gamma}^t))|\gtrsim
(\alpha-\beta)^2.
\end{equation}
\begin{equation}\label{arcchord3L}
|\cosh(\tilde{f}^{L,L}_2(\tilde{\alpha}_{\gamma}^{t})-\tilde{f}^{L,L}_2(\tilde{\beta}_{\gamma}^{t}))-\cos(\tilde{f}^{L,L}_1(\tilde{\alpha}_{\gamma}^{t})-\tilde{f}^{L,L}_1(\tilde{\beta}_{\gamma}^t))|\gtrsim
(\alpha-\beta)^2.
\end{equation}
\begin{equation}\label{arcchord4}
|\cosh(g_2(\alpha)-g_2(\beta)+\tilde{f}^{L}_2(\alpha_{\gamma}^{t})-\tilde{f}^{L}_2(\beta_{\gamma}^{t}))-\cos(g_1(\alpha)-g_1(\beta)+\tilde{f}^{L}_1(\alpha_{\gamma}^{t})-\tilde{f}^{L}_1(\beta_{\gamma}^{t}))|\gtrsim (\alpha-\beta)^2.
\end{equation}
\begin{equation}\label{arcchord501}
|\cosh(\tilde{f}^{L}_2(\alpha_{\gamma}^{t})-\tilde{f}^{L}_2((2\al+\tau(t))_{\eta}^{t}))-\cos(\tilde{f}^{L}_1(\alpha_{\gamma}^{t})-\tilde{f}^{L}_1((2\al+\tau(t))_{\eta}^{t}))|\gtrsim (\al+\tau(t))^2.
\end{equation}
\begin{equation}\label{arcchord502}
|\cosh(g_2(\alpha)+\tilde{f}^{L}_2(\alpha_{\gamma}^{t})-\tilde{f}^{L}_2((2\al+\tau(t))_{\eta}^{t}))-\cos(g_1(\alpha)+\tilde{f}^{L}_1(\alpha_{\gamma}^{t})-\tilde{f}^{L}_1((2\al+\tau(t))_{\eta}^{t}))|\gtrsim (\al+\tau(t))^2.
\end{equation}
\begin{equation}\label{arcchord50102}
|\cosh(\tilde{f}^{L}_2(\alpha_{\gamma}^{t})-\tilde{f}^{L}_2((-\tau(t))_{\eta}^{t}))-\cos(\tilde{f}^{L}_1(\alpha_{\gamma}^{t})-\tilde{f}^{L}_1((-\tau(t))_{\eta}^{t}))|\gtrsim (\al+\tau(t))^2.
\end{equation}
\begin{equation}\label{arcchord50202}
|\cosh(g_2(\alpha)+\tilde{f}^{L}_2(\alpha_{\gamma}^{t})-\tilde{f}^{L}_2((-\tau(t))_{\eta}^{t}))-\cos(g_1(\alpha)+\tilde{f}^{L}_1(\alpha_{\gamma}^{t})-\tilde{f}^{L}_1((-\tau(t))_{\eta}^{t}))|\gtrsim (\al+\tau(t))^2.
\end{equation}
Moreover, when additionally $\beta\geq 2\al+\tau(t)$ we have
\begin{equation}\label{arcchord601}
|\cosh(\tilde{f}^{L}_2(\alpha_{\gamma}^{t})-\tilde{f}^{L}_2(\beta_0^{t}))-\cos(\tilde{f}^{L}_1(\alpha_{\gamma}^{t})-\tilde{f}^{L}_1(\beta_0^{t}))|\gtrsim(\alpha-\beta)^2.
\end{equation}
\begin{equation}\label{arcchord6}
|\cosh(-\tilde{f}^{+}_2(\beta_0^{t})+\tilde{f}^{L}_2(\alpha_{\gamma}^{t})-\tilde{f}^{L}_2(\beta_0^{t}))-\cos(-\tilde{f}^{+}_1(\beta_0^{t})+\tilde{f}^{L}_1(\alpha_{\gamma}^{t})-\tilde{f}^{L}_1(\beta_{0}^{t}))|\gtrsim (\alpha-\beta)^2.
\end{equation}
\begin{equation}\label{arcchord7}
|\cosh(g_2(\alpha)-\tilde{f}^{+}_2(\beta_0^{t})+\tilde{f}^{L}_2(\alpha_{\gamma}^{t})-\tilde{f}^{L}_2(\beta_0^{t}))-\cos(g_1(\alpha)-\tilde{f}^{+}_1(\beta_0^{t})+\tilde{f}^{L}_1(\alpha_{\gamma}^{t})-\tilde{f}^{L}_1(\beta_0^{t}))|\gtrsim (\alpha-\beta)^2.
\end{equation}
When additionally $\beta\leq -\tau(t)$, $0\leq \zeta\leq 1$ we have
\begin{equation}\label{arcchord8}
|\cosh(g_2(\alpha)+\tilde{f}^{L}_2(\alpha_{\gamma}^{t})-\tilde{f}^{L}_2(\beta_{\gamma}^{t}))-\cos(g_1(\alpha)+\tilde{f}^{L}_1(\alpha_{\gamma}^{t})-\tilde{f}^{L}_1(\beta_{\gamma}^{t}))|\gtrsim (\alpha-\beta)^2.
\end{equation}
\begin{equation}\label{arcchord9}
|\cosh(\tilde{f}^{L}_2(\alpha_{\gamma}^{t})-\tilde{f}^{L}_2(\beta_{0}^{t}))-\cos(\tilde{f}^{L}_1(\alpha_{\gamma}^{t})-\tilde{f}^{L}_1(\beta_{0}^t))|\gtrsim
(\alpha-\beta)^2.
\end{equation}
\begin{equation}\label{arcchord601L}
|\cosh(\tilde{f}^{L,L}_2(\alpha_{\gamma}^{t})-\tilde{f}^{L,L}_2(\beta_0^{t}))-\cos(\tilde{f}^{L,L}_1(\alpha_{\gamma}^{t})-\tilde{f}^{L,L}_1(\beta_0^{t}))|\gtrsim(\alpha-\beta)^2,
\end{equation}
\begin{equation}\label{arcchord6L}
|\cosh(-\zeta\tilde{f}^{-}_2(\beta_0^{t})+\tilde{f}^{L,L}_2(\alpha_{\gamma}^{t})-\tilde{f}^{L,L}_2(\beta_0^{t}))-\cos(-\zeta\tilde{f}^{-}_1(\beta_0^{t})+\tilde{f}^{L,L}_1(\alpha_{\gamma}^{t})-\tilde{f}^{L,L}_1(\beta_{0}^{t}))|\gtrsim (\alpha-\beta)^2.
\end{equation}
\end{lemma}
\begin{proof}
For $A, \tilde{A}, B,\tilde{B} \in \mathbb{R}$, we have
\begin{align}\label{estimate1}
&|\cosh(A)-\cos(\tilde{A})-(\cosh(B)-\cos(\tilde{B}))|\\\nonumber
&\leq |\sinh(A+\tau(B-A))||A-B|+|\sin(\tilde{A}+\tau(\tilde{B}-\tilde{A}))||\tilde{A}-\tilde{B}|\\\nonumber
&\leq C(|A|+|B|)(|A|+|B|)(|A-B|)+C(|\tilde{A}|+|\tilde{B}|)(|\tilde{A}|+|\tilde{B}|)(|\tilde{A}-\tilde{B}|).
\end{align}
with some increasing function $C(x)$.
Since $f(\alpha,t)$ satisfies the arc-chord condition \eqref{arcchord00}, we have
\[
|\cosh(f_2(\alpha)-f_2(\beta))-\cos(f_1(\alpha)-f_1(\beta))|\gtrsim (\alpha-\beta)^2.
\]
Because $\tilde{f}(\alpha,t)=f(x(\alpha,t),t)$ from \eqref{tildefdefi}, we have
\begin{align*}
    &|\cosh(\tilde{f}_2(\alpha)-\tilde{f}_2(\beta))-\cos(\tilde{f}_1(\alpha)-\tilde{f}_1(\beta))|\\
    &=|\cosh(f_2(x(\alpha,t))-f_2(x(\beta,t))-\cos(f_1(x(\alpha,t))-f_1(x(\beta,t)))|\gtrsim (x(\alpha,t)-x(\beta,t))^2\gtrsim(\alpha-\beta)^2.
\end{align*}
Here in the last inequality, we use \eqref{initialchangevariable} with $t$ sufficiently small. Then \eqref{arcchord0} holds.

Then we only need to bound the difference, by \eqref{estimate1}, we have
\begin{align*}
    &|\cosh(\tilde{f}^{L}_2(\alpha)-\tilde{f}^{L}_2(\beta))-\cos(\tilde{f}^{L}_1(\alpha)-\tilde{f}^{L}_2(\beta))|\\
    &\geq|\cosh(\tilde{f}_2(\alpha)-\tilde{f}_2(\beta))-\cos(\tilde{f}_1(\alpha)-\tilde{f}_2(\beta))|-(\alpha-\beta)^2C(\|\tilde{f}(\alpha)\|_{C^1_{\alpha}}+\|\tilde{f}^{+}(\alpha)\|)_{C^1_{\alpha}}\|\tilde{f}^{+}(\alpha)\|_{C^1_{\alpha}}\\
    &\gtrsim(\alpha-\beta)^2-(\alpha-\beta)^2C(\|\tilde{f}(\alpha)\|_{C^1_{\alpha}}+\|\tilde{f}^{+}(\alpha)\|_{C^1_{\alpha}})\|\tilde{f}^{+}(\alpha)\|_{C^1_{\alpha}},
\end{align*}
where we use $f= \tilde{f}^{L}+\tilde{f}^{+}$ \eqref{cut function0}.
By \eqref{fLequation}, we can take $\|\tilde{f}^{+}(\alpha)\|_{C^1_{\alpha}}$ to be sufficiently small if we let $\delta$ sufficiently small. Therefore we have \eqref{arcchord1}. 
Similarly, from \eqref{fLequation2}, $\|\tilde{f}^{-}(\alpha)\|_{C^1_{\alpha}}$ can also be controlled to be sufficiently small by choosing $\delta$. Therefore we have \eqref{arcchord1L}.

We also have
\[
|\tilde{f}_i^{L}(\alpha)-\tilde{f}_i^{L}(\beta)-(\tilde{f}_i^{L}(\alpha+ic(\alpha)\gamma t)-\tilde{f}_i^{L}(\beta+ic(\beta)\gamma t))|\leq |\alpha-\beta|\|\tilde{f}_i^{L}(\alpha)-\tilde{f}_i^{L}(\alpha+ic(\alpha)\gamma t)\|_{C^1_{\alpha}}.
\]
As $t\to 0$, $\|\tilde{f}_i^{L}(\alpha)-\tilde{f}_i^{L}(\alpha+ic(\alpha)\gamma t)\|_{C^1_{\alpha}}$ can be taken sufficiently small. Hence we get
\eqref{arcchord2}.  \eqref{arcchord2L} follows the similar way as \eqref{arcchord2}.  From \eqref{definitiontildecalpha}, $
\|\tilde{f}_1^{L,L}(\tilde{\alpha}_{\gamma}^{t})-\tilde{f}_1^{L,L}(\alpha_{\gamma}^{t})\|_{C^1_{\alpha}}$ tends to 0 as $t\to 0$, then from \eqref{arcchord2L} we have \eqref{arcchord3L}. Since $\|g\|_{C^2_{\alpha}}$ is sufficiently small, we have \eqref{arcchord4} from \eqref{arcchord2}.

We take $\beta=2\al+\tau(t)$ in \eqref{arcchord2} to get
\[
|\cosh(\tilde{f}^{L}_2(\alpha_{\gamma}^{t})-\tilde{f}^{L}_2((2\al+\tau(t))_{\gamma}^{t}))-\cos(\tilde{f}^{L}_1(\alpha_{\gamma}^{t})-\tilde{f}^{L}_1((2\al+\tau(t))_{\gamma}^{t})))|\gtrsim (\al+\tau(t))^2.
\]
By \eqref{estimate1}, $|\tilde{f}^{L}(\alpha_{\gamma}^{t})-\tilde{f}^{L}(\alpha_{\eta}^{t})|\lesssim t(\al+\tau(t))^2$ and $|g(\alpha)|\leq \|g\|_{C^2}(\al+\tau(t))^2$, we get \eqref{arcchord501}, \eqref{arcchord502}. \eqref{arcchord50102} and \eqref{arcchord50202} are similar by taking $\beta=-\tau(t)$ instead of $2\al+\tau(t)$.

When $\beta\geq 2\al+\tau(t)$ , from \eqref{arcchord2}, we have
\[
|\cosh(\tilde{f}^{L}_2(\alpha_{\gamma}^ t)-\tilde{f}^{L}_2(\beta_{\gamma}^{ t}))-\cos(\tilde{f}^{L}_1(\alpha_{\gamma}^{ t})-\tilde{f}^{L}_1(\beta_{\gamma}^{ t}))|\gtrsim (\alpha-\beta)^2.
\]
By $|\tilde{f}^{L}(\beta_{\gamma}^{t})-\tilde{f}^{L}(\beta_{0}^{t})|\lesssim t(\beta+\tau(t))^2$,  $|\tilde{f}^{+}(\beta)|\leq\|\tilde{f}^{+}\|_{C^2}(\beta+\tau(t))^2$,  $|g(\beta)|\leq\|g\|_{C^2}(\beta+\tau(t))^2$, and $(\al+\tau(t))^2\lesssim(\alpha-\beta)^2$, $(\beta+\tau(t))^2\lesssim(\alpha-\beta)^2$, we have \eqref{arcchord6} and \eqref{arcchord7}.

When $\beta\leq -\tau(t)$, we still use \eqref{arcchord2} to get
\[
|\cosh(\tilde{f}^{L}_2(\alpha_{\gamma}^ t)-\tilde{f}^{L}_2(\beta_{\gamma}^{ t}))-\cos(\tilde{f}^{L}_1(\alpha_{\gamma}^{ t})-\tilde{f}^{L}_1(\beta_{\gamma}^{ t}))|\gtrsim (\alpha-\beta)^2.
\]
Since when $\beta\leq -\tau(t)$, we  have $\beta_{\gamma}^{t}=\beta_{0}^{t}$. Then we have \eqref{arcchord9}. Moreover,  $g(\beta)=0$ when $\beta\leq -\tau(t)$, $|g(\alpha)|\leq \|g\|_{C^2}(\al+\tau(t))^2$, and $(\al+\tau(t))^2\lesssim (\alpha-\beta)^2$. Hence we get \eqref{arcchord8}.

Similarly, from \eqref{fLequation2}, since  $|\tilde{f}^{-}(\beta_0^t)|\leq\|\tilde{f}^{-}\|_{C^2}(\beta+\tau(t))^2$, and $(\beta+\tau(t))^2\lesssim (\alpha-\beta)^2$, we can use \eqref{arcchord2L} instead of \eqref{arcchord2} to get \eqref{arcchord601L} and \eqref{arcchord6L}.
\end{proof}
\subsection{Analysis of some kernels}
Here we let $\bar{\partial}_{\gamma}$ be the partial derivative with respect to $\gamma$ assuming $h$ does not depend on $\gamma$.
\begin{lemma}\label{K-1wholebehavior}
Let $\sigma\in \mathcal{N}$, $k\leq 12$, $j\leq 4$,  $t$, $\delta$ in \eqref{lambda0defi}  be sufficiently small satisfying conditions in lemma \ref{arcchord}. We have
\begin{align}\label{K-1changed}
  &\|K_{-\sigma}(0, V_{\tilde{f}^{L,L}}^{61}(\tilde{\alpha}_{\gamma}^{t},t)-V_{\tilde{f}^{L,L}}^{61}(\tilde{\beta}_{\gamma}^{t},t),V_{X^{L}}(\tilde{\alpha}_{\gamma}^{t},t)-V_{X^{L}}(\tilde{\beta}_{\gamma}^{t},t))(\alpha-\beta)^{\sigma}\|_{ C_{\gamma}^{j}([-1,1],C^{k+4}([-\tau(t),\frac{\pi}{4}]\times[-\pi,\pi]))}\lesssim 1,
\end{align}
and for $0\leq t'<t\leq t_0$, 
\begin{align}\label{K-1changedt}
 &\|K_{-\sigma}(0, V_{\tilde{f}^{L,L}}^{61}(\tilde{\alpha}_{\gamma}^{t},t)-V_{\tilde{f}^{L,L}}^{61}(\tilde{\beta}_{\gamma}^{t},t),V_{X^{L}}(\tilde{\alpha}_{\gamma}^{t},t)-V_{X^{L}}(\tilde{\beta}_{\gamma}^{t},t))(\alpha-\beta)^{\sigma}\\\nonumber
 &-K_{-\sigma}(0, V_{\tilde{f}^{L,L}}^{61}(\tilde{\alpha}_{\gamma}^{t'},t')-V_{\tilde{f}^{L,L}}^{61}(\tilde{\beta}_{\gamma}^{t'},t'),V_{X^{L}}(\tilde{\alpha}_{\gamma}^{t'},t')-V_{X^{L}}(\tilde{\beta}_{\gamma}^{t'},t'))(\alpha-\beta)^{\sigma}\|_{ C_{\gamma}^{j}([-1,1],C^{k+3}([-\tau(t'),\frac{\pi}{4}]\times[-\pi,\pi]))}\\\nonumber
  &\lesssim \mathcal{O}(t-t').
  \end{align}

  \end{lemma}
  \begin{proof}
  We only show the boundedness argument \eqref{K-1changed} and other estimates follow a similar way because the $C^{\infty}$ smoothness of $K_{-\sigma}(A,B,C)$ with respect to $(A,B,C)$.

  Recall that $K_{-\sigma}^{j}(A,B,C)$ is of $-\sigma$ type if for $A$, $B$, $C$ in $R^n$, it has the form
\begin{align}\label{k-sigma02}
&c_j\frac{\sin(A_1+B_1)^{m_1}\cos(A_1+B_1)^{m_2}}{(\cosh(A_2+B_2)-\cos(A_1+B_1))^{m_0}}\\\nonumber
    &\times(\sinh(A_2+B_2))^{m_3}(\cosh(A_2+B_2))^{m_4}\Pi_{j=1}^{m_5}(A_{\lambda_{j}})\Pi_{j=1}^{m_6}(B_{\lambda_{j,2}})\Pi_{j=1}^{m_7}(C_{\lambda_{j,3}}),
\end{align}
with $m_1+m_3+m_5+m_6+m_7-2m_0\geq -\sigma$. $c_j$ is a constant.
Here we can take
\[
B_i,C_i\in\{\partial_{\alpha}^{d_i}\tilde{f}^{L,L}(\tilde{\alpha}_{\gamma}^{t},t)-\partial_{\alpha}^{d_i}\tilde{f}^{L,L}(\tilde{\beta}_{\gamma}^{t},t), X_i^{L}(\tilde{\alpha}_{\gamma}^{t},t)-X_i(\tilde{\beta}_{\gamma}^{t},t)\}
\]
 with $d_i\leq 61$, $s_{i}\leq 30$. From definition of  $\tilde{\alpha}_{\gamma}^{t}$ \eqref{definitiontildecalpha}, the space of $\tilde{f}^{L,L}$ and $X_i^{L}$ \eqref{spacefLL}, we have
 \[
\partial_{\alpha}^{d_i}\tilde{f}^{L,L}(\tilde{\alpha}_{\gamma}^{t}), X_i^{L}(\tilde{\alpha}_{\gamma}^{t}) \in C_{\gamma}^{4}([-1,1], C_{\alpha}^{29}[-\pi,\pi]).
 \]
 Then each $\sin, \sinh, A_{j}, B_{j}, C_{j}$ gives an order of $(\alpha-\beta)$ and the denominator gives an order of $(\alpha-\beta)^{-2m_0}$. Then from $m_1+m_3+m_5+m_6+m_7-2m_0\geq -\sigma$ and the arc-chord condition \eqref{arcchord3L} we have the boundedness estimate.

\end{proof}
\begin{lemma}\label{alphabetanear}
Let $\sigma\in \mathcal{N}$, $k\leq 12$, $j,j'\leq 4$, $t$, $\delta$ in \eqref{lambda0defi}  be sufficiently small satisfying conditions in lemma \ref{arcchord}, $Y^{k}=C_{\gamma}^{0}([-1,1],C^{k}([-\tau(t),\frac{\pi}{4}]\times[-\tau(t),\pi)))$, and
\[
\tilde{K}_{-\sigma}(\alpha,\beta,\gamma,t)=K_{-\sigma}(V_h^{30}(\al,\g,t)-V_h^{30}(\beta,\g,t), V_{\tilde{f}^{L}}^{61}(\alpha_{\gamma}^{t},t)-V_{\tilde{f}^{L}}^{61}(\beta_{\gamma}^{t},t),V_{X}(\alpha_{\gamma}^{t},t)-V_{X}(\beta_{\gamma}^{t},t))(\alpha-\beta)^{\sigma}.
\]
For $h$, $g_{j
'}$ in 
 $C_{\gamma}^{0}([-1,1],L_{\al}^{2}[-\tau(t),\pi])\cap \{h|\supp h \subset [-\tau(t),\frac{\pi}{4}]\}$,  with $\|h\|_{L_{\al}^{2}[-\tau(t),\pi]}$ sufficiently small such that $D^{-60}(h)$ satisfying the condition in lemma \ref{arcchord}, we have the following inequalities: 
\begin{align}\label{near1}
&\|\bar{\partial}_{\gamma}^{j}\tilde{K}_{-\sigma}(\alpha,\beta,\gamma,t)\|_{ Y^{k+4}}\lesssim 1,
\end{align}
 \begin{align}\label{near2}
   &\|\bar{\partial}_{\gamma}^{j}D_{h}^{j'}\tilde{K}_{-\sigma}[g_1,...,g_{j'}](\alpha,\beta,\gamma,t)\|_{ Y^{k+4}}\lesssim \prod_{l=1}^{j'}\|g_l\|_{L_{\al}^2[0,\pi]},
   \end{align}
  and for $0\leq t’\leq t$,
  \begin{align}\label{neart}
  &\quad\|\bar{\partial}_{\gamma}^{j}(\tilde{K}_{-\sigma}(\alpha,\beta,\gamma,t)-\tilde{K}_{-\sigma}(\alpha,\beta,\gamma,t'))\|_{ Y^{k+3}}\\\nonumber
  &\lesssim \mathcal{O}(t-t')+\|h(\alpha,\gamma,t)-h(\alpha,\gamma,t')\|_{ C_{\gamma}^{0}([-1,1],L_{\al}^{2}[-\tau(t'),\pi])},
  \end{align}
  and
    \begin{align}\label{neart2}
  &\quad\|(\bar{\partial}_{\gamma}^{j}D_h^{j'}\tilde{K}_{-\sigma}[g_1,g_2...g_{j'}](\alpha,\beta,\gamma,t)-\bar{\partial}_{\gamma}^{j}D_h^{j'}\tilde{K}_{-\sigma}[g_1,g_2...g_{j'}](\alpha,\beta,\gamma,t'))\|_{ Y^{k+3}}\\\nonumber
  &\lesssim \mathcal{O}(t-t')+\|h(\alpha,\gamma,t)-h(\alpha,\gamma,t')\|_{ C_{\gamma}^{0}([-1,1],L_{\al}^{2}[-\tau(t'),\pi])}+\sum_{l=1}^{j'}\|g_l(\alpha,\gamma,t)-g_l(\alpha,\gamma,t')\|_{ C_{\gamma}^{0}([-1,1],L_{\al}^{2}[-\tau(t'),\pi])}.
  \end{align}
\end{lemma}
\begin{proof}
  When $s_i\leq 30$, $0\leq d_i\leq 61$, $0\leq j \leq 4$, from definition of $a_{\gamma}^{t}$ \eqref{alphadefi}, the space of $\tilde{f}^{L}$ and $X_i$ \eqref{Xispace}, the operator $D^{-i}$ \eqref{D-formularnew} we have 
\[
\bar{\partial}_{\gamma}^{j}D^{-60+s_i}h(\alpha,\gamma,t),\  \bar{\partial}_{\gamma}^{j}\partial_{\alpha}^{d_i}\tilde{f}^{L}(\alpha_{\gamma}^{t},t),\  \bar{\partial}_{\gamma}^{j}\partial_{\alpha}^{d_i}X_i(\alpha_{\gamma}^{t},t) \in C_{\gamma}^{0}([-1,1], C_{\alpha}^{29}[-\tau(t),\pi]).
\]
Since $V^{30}(h)$ only contain terms $D^{-60+s_i}h$ with $s_i\leq 30$ \eqref{notation03h}, we can use similar proof as in lemma \ref{K-1wholebehavior}.
\end{proof}
\begin{lemma}\label{alphabetafar}
Let $\sigma\in \mathcal{N}$,  $l+l'\leq 15$, $j',j\leq  4$,  $t$, $\delta$ in \eqref{lambda0defi}  be sufficiently small satisfying conditions in lemma \ref{arcchord},
and 
\begin{align*}
&\quad\tilde{K}_{-\sigma}(\alpha,\beta,\zeta,\gamma,t)\\
&=K_{-\sigma}(\zeta V_h^{30}(\al,\g,t)-V^{30}(\beta_0^{t},t), V_{\tilde{f}^{L}}^{61}(\alpha_{\gamma}^{t},t)-V_{\tilde{f}^{L}}^{61}(\beta_{0}^{t},t),V_{X}(\alpha_{\gamma}^{t},t)-V_{X}(\beta_{0}^{t},t)).
\end{align*} For $h$, $g_{j'}$ in $h\in C_{\gamma}^{0}([-1,1],L_{\al}^{2}[-\tau(t),\pi])$ with $\|h\|_{L_{\al}^{2}[-\tau(t),\pi]}$ sufficiently small such that $D^{-60} h$ satisfying the condition in lemma \ref{arcchord},   when $-\tau(t)<\alpha\leq \frac{\pi}{4}$, $\zeta\in[0,1]$, we have the following four inequalities:

\begin{align}\label{farawayuniform}
  &\sup_{\beta\in [-\pi,-\tau(t)]\cup  [2\alpha+\tau(t),\pi]}(|(\partial_{\alpha}^{l}\partial_{\beta}^{l'}\bar{\partial}_{\gamma}^{j})\tilde{K}_{-\sigma}(\alpha,\beta,\zeta,\gamma,t)||\alpha-\beta|^{l+l'+j+\sigma})\lesssim 1,
\end{align}
%\begin{align}\label{farawayhg}
 % &\sup_{\beta\in [-\pi,0]\cup  [2\alpha,\pi]}|(\partial_{\alpha}^{l}\partial_{\beta}^{l'}\partial_{\gamma}^{j})(K_{-\sigma}(\zeta V_h^{30}(\al,\g,t)-V^{30}(\beta_0^{t},t), V_{\tilde{f}^{L}}^{61}(\alpha_{\gamma}^{t},t)-V_{\tilde{f}^{L}}^{61}(\beta_{0}^{t},t),V_{X}(\alpha_{\gamma}^{t},t)-V_{X}(\beta_{0}^{t},t))\\\nonumber
 % &-K_{-\sigma}(\zeta V_g^{30}(\alpha)-V^{30}(\beta_0^{t},t), V_{\tilde{f}^{L}}^{61}(\alpha_{\gamma}^{t},t)-V_{\tilde{f}^{L}}^{61}(\beta_{0}^{t},t),V_{X}(\alpha_{\gamma}^{t},t)-V_{X}(\beta_{0}^{t},t)))|\\\nonumber
 % &\lesssim \|h-g\|_{W^{i_0,k}}(\frac{1}{|\beta-\alpha|})^{l+l'+1+j+\sigma},
 % \end{align}
  \begin{align}\label{farawayhg2}
  \sup_{\beta\in [-\pi,-\tau(t)]\cup  [2\alpha+\tau(t),\pi]}&(|(\partial_{\alpha}^{l}\partial_{\beta}^{l'}\bar{\partial}_{\gamma}^{j})D_{h}^{j'}\tilde{K}_{-\sigma}[g_1,g_2,...g_{j'}](\alpha,\beta,\zeta,\gamma,t)||\beta-\alpha|)^{l+l'+1+j+\sigma})\lesssim \prod_{j''=1}^{j'}\|g_{j''}\|_{L_{\al}^{2}},
  \end{align}

  and for $0\leq t'<t$, 
  \begin{align}\label{farawaytt'}
 &\quad\sup_{\beta\in [-\pi,-\tau(t)]\cup  [2\alpha+\tau(t),\pi]}(|(\partial_{\alpha}^{l}\partial_{\beta}^{l'}\bar{\partial}_{\gamma}^{j})(\tilde{K}_{-\sigma}(\alpha,\beta,\zeta,\gamma,t)-\tilde{K}_{-\sigma}(\alpha,\beta,\zeta,\gamma,t'))||\beta-\alpha|^{l+l'+1+j+\sigma})\\\nonumber
  &\lesssim \mathcal{O}(t-t')+\|h(\alpha,\gamma,t)-h(\alpha,\gamma,t')\|_{C_{\gamma}^{0}([-1,1],L_{\al}^{2}[-\tau(t'),\pi])}.
  \end{align}
    \begin{align}\label{farawaytt'dh}
 \quad\sup_{\beta\in [-\pi,-\tau(t)]\cup  [2\alpha+\tau(t),\pi]}&(|(\partial_{\alpha}^{l}\partial_{\beta}^{l'}\partial_{\gamma}^{j})(D_h^{j'}\tilde{K}_{-\sigma}[g_1,g_2...g_{j'}](\alpha,\beta,\zeta,\gamma,t)-D_h^{j'}\tilde{K}_{-\sigma}[g_1,g_2...g_{j'}](\alpha,\beta,\zeta,\gamma,t'))|\\\nonumber
 &\qquad \cdot|\beta-\alpha|^{l+l'+1+j+\sigma})\\\nonumber
  \qquad\lesssim \mathcal{O}(t-t')+\|h(\alpha,\gamma,t)-h&(\alpha,\gamma,t')\|_{C_{\gamma}^{0}([-1,1],L_{\al}^{2}[-\tau(t'),\pi])}+\sum_{j'=1}^{j''}\|g_l(\alpha,\gamma,t)-g_l(\alpha,\gamma,t')\|_{ C_{\gamma}^{0}([-1,1],L_{\al}^{2}[-\tau(t'),\pi])}.
  \end{align}
\end{lemma}
\begin{proof}
From definition of $a_{\gamma}^{t}$ \eqref{alphadefi}, the space of $\tilde{f}^{L}$ and $X_i$ \eqref{Xispace}, the operator $D^{-i}$ \eqref{D-formularnew} we have when $d_l\leq 61$, $s_l\leq 30$, $0\leq j\leq 4$
 \[
\bar{\partial}_{\gamma}^{j} D^{-60+s_l}h(\alpha,\gamma,t), \bar{\partial}_{\gamma}^{j}\partial_{\alpha}^{d_i}\tilde{f}^{L}(\alpha_{\gamma}^{t}), \bar{\partial}_{\gamma}^{j}X(\alpha_{\gamma}^{t}) \in C_{\gamma}^{0}([-1,1], C_{\alpha}^{29}[-\tau(t),\pi]).
 \]
 \[
\bar{\partial}_{\gamma}^{j} D^{-60+s_l}h(\beta,\gamma,t), \bar{\partial}_{\gamma}^{j}\partial_{\alpha}^{d_i}\tilde{f}^{L}(\beta_{0}^{t}),\bar{\partial}_{\gamma}^{j} X(\beta_{0}^{t}) \in C_{\gamma}^{0}([-1,1], C_{\alpha}^{29}[-\pi,\pi]).
 \]
For  bound \eqref{farawayuniform}, we claim
\begin{align*}
&\qquad\partial_{\alpha}^{l}\partial_{\beta}^{l'}\bar{\partial}_{\gamma}^{j}K_{-\sigma}(\zeta V_h^{30}(\al,\g,t)-V^{30}(\beta_0^{t},t), V_{\tilde{f}^{L}}^{61}(\alpha_{\gamma}^{t},t)-V_{\tilde{f}^{L}}^{61}(\beta_{0}^{t},t),V_{X}(\alpha_{\gamma}^{t},t)-V_{X}(\beta_{0}^{t},t))\\
    &=\sum_{l}K_{-(\sigma+l+l'+j)}^{l}(\zeta V_h^{30}(\al,\g,t)-V^{30}(\beta_0^{t},t), V_{\tilde{f}^{L}}^{61}(\alpha_{\gamma}^{t},t)-V_{\tilde{f}^{L}}^{61}(\beta_{0}^{t},t),V_{X}(\alpha_{\gamma}^{t},t)-V_{X}(\beta_{0}^{t},t))G^{l}(\alpha,\beta,\gamma,t)
\end{align*}
with $|G^{l}(\alpha,\beta,\gamma,t)|\lesssim 1$.
Then we can only consider $\tilde{\sigma}=\sigma+l+l'+j$ without derivative.
Recall that $K_{\tilde{\sigma}}$ has the form \eqref{k-sigma}, by arc-chord conditions \eqref{arcchord7} and \eqref{arcchord8}, we have
\[
|\frac{1}{denominator}|\lesssim \frac{1}{(\alpha-\beta)^{2m_0}}.
\]
Moreover, from definition of $V^{30}_{h}$ \eqref{notation03h}, $V^{30}$ \eqref{notation01}, $V^{61}_{\tilde{f}^{L}}$\eqref{notation02},  $V_{X}$ \eqref{notation03}, we have
\begin{align*}
&|\zeta V_h^{30}(\al,\g,t)-V^{30}(\beta_0^{t},t)|\leq | V_h^{30}(\alpha,\gamma,t)|+| V^{30}(\alpha_{0}^{t},t)|+|V^{30}(\alpha_0^{t},t)-V^{30}(\beta_0^{t},t)|\\
&\lesssim |\alpha+\tau(t)|+|\alpha+\tau(t)|+|\alpha-\beta|,
\end{align*}
\begin{align*}
  |V_{\tilde{f}^{L}}^{61}(\alpha_{\gamma}^{t},t)-V_{\tilde{f}^{L}}^{61}(\beta_{0}^{t},t)|\leq |V_{\tilde{f}^{L}}^{61}(\alpha_{0}^{t},t)-V_{\tilde{f}^{L}}^{61}(\beta_{0}^{t},t)|+|V_{\tilde{f}^{L}}^{61}(\alpha_{0}^{t},t)-V_{\tilde{f}^{L}}^{61}(\alpha_{\gamma}^{t},t)|\lesssim |\alpha-\beta|+|\alpha+\tau(t)|\lesssim|\alpha-\beta|,
\end{align*}
and
\begin{align*}
  |V_{X}(\alpha_{\gamma}^{t},t)-V_{X}(\beta_{0}^{t},t)|\leq |V_{X}(\alpha_{0}^{t},t)-V_{X}(\beta_{0}^{t},t)|+|V_{X}(\alpha_{0}^{t},t)-V_{X}(\alpha_{\gamma}^{t},t)|\lesssim |\alpha-\beta|+|\alpha+\tau(t)|\lesssim|\alpha-\beta|.
\end{align*}
Hence we have 
\[
|numerator|\lesssim |\alpha-\beta|^{m_1+m_3+m_5+m_6}.
\]
Then we get the first inequality \eqref{farawayuniform}. \eqref{farawayhg2} can be done in a similar way. For the difference in $t$ estimate, we could write 
\[
\frac{d}{d\zeta_1}K_{-\sigma}(\zeta_1\zeta V_h^{30}(\alpha,\gamma,t)-\zeta_1V^{30}(\beta_0^{t},t), \zeta_1V_{\tilde{f}^{L}}^{61}(\alpha_{\gamma}^{t},t)-\zeta_1V_{\tilde{f}^{L}}^{61}(\beta_{0}^{t},t),\zeta_1V_{X}(\alpha_{\gamma}^{t},t)-\zeta_1V_{X}(\beta_{0}^{t},t))
\]
and
\begin{align*}
&\frac{d}{d\zeta_1}D_{h}^{j'}K_{-\sigma}(\zeta_1\zeta V_h^{30}(\alpha,\gamma,t)-\zeta_1V^{30}(\beta_0^{t},t), \zeta_1V_{\tilde{f}^{L}}^{61}(\alpha_{\gamma}^{t},t)-\zeta_1V_{\tilde{f}^{L}}^{61}(\beta_{0}^{t},t),\zeta_1V_{X}(\alpha_{\gamma}^{t},t)-\zeta_1V_{X}(\beta_{0}^{t},t))[g_1,g_2,...g_{j'}]
\end{align*}
 as sums of the product of $K_{-(\sigma+1)}$ terms and other non-singular terms to bound the difference.
\end{proof}
\begin{corollary}\label{alphabetafarleft}
When $\beta<-\tau(t)$ in the previous lemma, \ref{alphabetafar}, the $\frac{1}{|\alpha-\beta|}$ bound can be replaced by $\frac{1}{\alpha+\tau(t)}$. Moreover, here
\begin{equation}\label{3corochange0}
\begin{split}
    &\quad K_{-\sigma}(\zeta V_h^{30}(\alpha,\gamma,t)-V^{30}(\beta_0^{t},t), V_{\tilde{f}^{L}}^{61}(\alpha_{\gamma}^{t},t)-V_{\tilde{f}^{L}}^{61}(\beta_{0}^{t},t),V_{X}(\alpha_{\gamma}^{t},t)-V_{X}(\beta_{0}^{t},t))\\
    &=K_{-\sigma}(\zeta V_h^{30}(\alpha,\gamma,t), V_{\tilde{f}^{L}}^{61}(\alpha_{\gamma}^{t},t)-V_{\tilde{f}^{L}}^{61}(\beta_{\gamma}^{t},t),V_{X}(\alpha_{\gamma}^{t},t)-V_{X}(\beta_{\gamma}^{t},t))\\
    &=K_{-\sigma}(\zeta (V_h^{30}(\alpha,\gamma,t)-V_h^{30}(\beta,\gamma,t), V_{\tilde{f}^{L}}^{61}(\alpha_{\gamma}^{t},t)-V_{\tilde{f}^{L}}^{61}(\beta_{\gamma}^{t},t),V_{X}(\alpha_{\gamma}^{t},t)-V_{X}(\beta_{\gamma}^{t},t)).
    \end{split}
\end{equation}
\end{corollary}
\begin{proof}
We use $\frac{1}{|\alpha-\beta|}\lesssim \frac{1}{\alpha+\tau(t)}$  when $\beta<-\tau(t)$ . For the equation \eqref{3corochange0} we use the fact that when $\beta\leq -\tau(t)$, $\beta_{\gamma}^{t}=\beta_{0}^{t}$ \eqref{alphadefi} and
$V^{30}(\beta_0^{t},t)=V_{h}^{30}(\beta,\gamma,t)=0$ (from \eqref{notation01}, \eqref{notation03h},\eqref{D-formularnew}).
\end{proof}

\begin{corollary}\label{alphabetafarright}
When $\beta>2\alpha+\tau(t)$ in the previous lemma, the $\frac{1}{|\alpha-\beta|}$ bound can be replaced by $\frac{1}{\beta+\tau(t)}$.
\end{corollary}

\begin{lemma}\label{alphabetafar-}
Let
\[
\tilde{K}_{-\sigma}(\alpha,\beta,\gamma,\zeta, t)=K_{-\sigma}^{i}(0,\Delta V_{\tilde{f}^{L,L}}^{61}(\alpha_{\gamma}^{t},t)+\zeta\Delta V_{\tilde{f}^{-}}^{61}(\alpha_{\gamma}^{t},t),\Delta V_{X^L}(\alpha_{\gamma}^{t},t)+\zeta\Delta V_{X^-}(\alpha_{\gamma}^{t},t)).
\]For $0\leq \zeta\leq 1$, $l+l'\leq 15$, $j\leq 4$, when $-\tau(t)<\alpha\leq \frac{\pi}{4}$, $\beta<-\tau(t)$,  we have the following two inequalities:

\begin{align}\label{farawayuniform-}
  &\quad \sup_{\beta\in [-\pi,-\tau(t)]}(|(\partial_{\alpha}^{l}\partial_{\beta}^{l'}\partial_{\gamma}^{j})\tilde{K}_{-\sigma}(\alpha,\beta,\gamma,\zeta, t)||\beta+\tau(t)|^{l+l'+j+\sigma})\\
  &\lesssim \sup_{\beta\in [-\pi,-\tau(t)]}(|(\partial_{\alpha}^{l}\partial_{\beta}^{l'}\partial_{\gamma}^{j})\tilde{K}_{-\sigma}(\alpha,\beta,\gamma,\zeta, t)||\beta-\alpha|^{l+l'+j+\sigma})\lesssim 1,
\end{align}
  and for $0\leq t'<t\leq t$, $\alpha>-\tau(t')$, 
  \begin{align}\label{farawaytt'-}
 &\quad\sup_{\beta\in [-\pi,-\tau(t)]}(|(\partial_{\alpha}^{l}\partial_{\beta}^{l'}\partial_{\gamma}^{j})(\tilde{K}_{-\sigma}(\alpha,\beta,\gamma,\zeta, t)-\tilde{K}_{-\sigma}(\alpha,\beta,\gamma,\zeta, t'))||\beta+\tau(t)|^{l+l'+1+j+\sigma})\\\nonumber
  &\lesssim\sup_{\beta\in [-\pi,-\tau(t)]}(|(\partial_{\alpha}^{l}\partial_{\beta}^{l'}\partial_{\gamma}^{j})(\tilde{K}_{-\sigma}(\alpha,\beta,\gamma,\zeta, t)-\tilde{K}_{-\sigma}(\alpha,\beta,\gamma,\zeta, t'))||\beta-\alpha|^{l+l'+1+j+\sigma})\\\nonumber
  &\lesssim \mathcal{O}(t-t').
  \end{align}
\end{lemma}
\begin{proof}
 From definition of $V^{30}_{\tilde{f}^{-}}$ \eqref{notation03f-},  $V^{30}_{\tilde{f}^{L,L}}$ \eqref{notation03fLL} and space of $\tilde{f}^{-}$, $\tilde{f}^{L,L}$  \eqref{f-fLLspace}, definition of $X^{L}$ \eqref{XLfunction}, $X^{-}$ \eqref{X-function}. When $\beta<-\tau(t)$, $\alpha>-\tau(t)$,  we have 
\begin{align*}
   &\quad |V_{\tilde{f}^{L,L}}^{61}(\alpha_{\gamma}^{t},t)-V_{\tilde{f}^{L,L}}^{61}(\beta_{\gamma}^{t},t)|=|V_{\tilde{f}^{L,L}}^{61}(\alpha_{\gamma}^{t},t)-V_{\tilde{f}^{L,L}}^{61}(\beta_{0}^{t},t)|\\
   &\leq |V_{\tilde{f}^{L,L}}^{61}(\alpha_{0}^{t},t)-V_{\tilde{f}^{L,L}}^{61}(\beta_{0}^{t},t)|+|V_{\tilde{f}^{L,L}}^{61}(\alpha_{0}^{t},t)-V_{\tilde{f}^{L,L}}^{61}(\alpha_{\gamma}^{t},t)|\\
  &\lesssim |\alpha-\beta|+|\alpha+\tau(t)|\lesssim|\alpha-\beta|,
\end{align*}
\begin{align*}
    &|V_{\tilde{f}^{-}}^{61}(\alpha_{\gamma}^{t},t)|=0,
\end{align*}
\begin{align*}
    &|V_{\tilde{f}^{-}}^{61}(\beta_{\gamma}^{t},t)|=|V_{\tilde{f}^{-}}^{61}(\beta_{0}^{t},t)|\lesssim |\beta+\tau(t)|\lesssim |\alpha-\beta|,
\end{align*}
\begin{align*}
  &\quad|V_{X^{L}}(\alpha_{\gamma}^{t},t)-V_{X^{L}}(\beta_{\gamma}^{t},t)|=|V_{X^{L}}(\alpha_{\gamma}^{t},t)-V_{X^{L}}(\beta_{0}^{t},t)|\\
  &\leq |V_{X^{L}}(\alpha_{0}^{t},t)-V_{X^{L}}(\beta_{0}^{t},t)|+|V_{X^{L}}(\alpha_{0}^{t},t)-V_{X^{L}}(\alpha_{\gamma}^{t},t)|\\
  &\lesssim |\alpha-\beta|+|\alpha+\tau(t)|\\
  &\lesssim|\alpha-\beta|,
\end{align*}
\begin{align*}
   &|V_{X^{-}}(\alpha_{\gamma}^{t},t)|=0,\\
  &|V_{X^{-}}(\beta_{\gamma}^{t},t)|=|V_{X^{-}}(\beta_{0}^{t},t)|\lesssim |\beta+\tau(t)|\lesssim |\alpha-\beta|.
\end{align*}
Then we can use a similar proof as in lemma \ref{alphabetafar} by arc-chord condition \eqref{arcchord6L} instead of \eqref{arcchord8}.
\end{proof}
\begin{lemma}\label{gammachange}
Let $\sigma\in \mathcal{N}$, $k\leq 12$, $j,j'\leq 4$,  $t$, $\delta$ in \eqref{lambda0defi}  be sufficiently small satisfying conditions in lemma \ref{arcchord}, $\tilde{Y}^{k}(t)=C_{\gamma,\eta}^{0}([-1,1]\times[-1,1],C^{k}_{\alpha}([-\tau(t),\frac{\pi}{4}])$ and 
 \begin{align*}
&\quad\tilde{K}_{-\sigma}(\alpha,\g, \eta,t)(\al)
\\&= K_{-\sigma}(V_h^{30}(\alpha,\gamma,t)-V_h^{30}(2\alpha+\tau(t),\eta,t), V_{\tilde{f}^{L}}^{61}(\alpha_{\gamma}^{t},t)-V_{\tilde{f}^{L}}^{61}((2\alpha+\tau(t))_{\eta}^{t},t),V_{X}(\alpha_{\gamma}^{t},t)-V_{X}((2\alpha+\tau(t))_{\eta}^{t},t))
  \\
  &\qquad\cdot(\alpha+\tau(t))^{\sigma}.
 \end{align*}For $h$, $g_{j'}$ in 
\[
h\in C_{\gamma}^{0}([-1,1],L_{\al}^{2}[-\tau(t),\pi])\cap \{h|\supp h \subset [-\tau(t),\frac{\pi}{4}]\},
\] with $\|h\|_{L_{\al}^{2}[-\tau(t),\pi]}$ sufficiently small such that $D^{-60}(h)$ satisfying the condition in lemma \ref{arcchord}, we have $\tilde{K}$
satisfying the following inequalities:
\begin{align}\label{gammachangeestimate1}
  &\|\bar{\partial}^{j}\tilde{K}_{-\sigma}(h,\g, \eta,t)\|_{ \tilde{Y}^{k+4}}\lesssim 1,
\end{align}
  \begin{align}\label{gammachangeestimate3}
  &\|\bar{\partial}^{j}D_h^{j'}\tilde{K}_{-\sigma}(h,\g, \eta,t)[g_1,g_2...g_{j'}]\|_{ \tilde{Y}^{k+4}}\lesssim \prod_{l=1}^{j'}\|g_l\|_{L_{\al}^2},
  \end{align}
  and for $0\leq t'<t$, 
\begin{align}\label{gammachangeestimate4}
&\|\bar{\partial}^{j}(\tilde{K}_{-\sigma}(h,\g, \eta,t)-\tilde{K}_{-\sigma}(h,\g, \eta,t'))\|_{ \tilde{Y}^{k+3}}\lesssim O(t-t')+\|h(\alpha,\gamma,t)-h(\alpha,\gamma,t')\|_{C_{\gamma}^{0}([-1,1],L_{\al}^{2}[-\tau(t'),\pi])},
  \end{align}
  \begin{align}\label{gammachangeestimate5}
  &\quad\|\bar{\partial}^{j}(D_h^{j'}\tilde{K}_{-\sigma}(h,\g, \eta,t)[g_1,g_2...,g_{j'}]-D_h\tilde{K}_{-\sigma}(h,\g, \eta,t')[g_1,g_2...,g_{j'}])\|_{ \tilde{Y}^{k+3}}\\\nonumber
  &\lesssim O(t-t')+\|h(\alpha,\gamma,t)-h(\alpha,\gamma,t')\|_{C_{\gamma}^{0}([-1,1],L_{\al}^{2}[-\tau(t'),\pi])}+\sum_{l=1}^{j'}\|g_l(\alpha,\gamma,t)-g_l(\alpha,\gamma,t')\|_{ C_{\gamma}^{0}([-1,1],L_{\al}^{2}[-\tau(t'),\pi])}.
  \end{align}
\end{lemma}
\begin{proof}
We can use the similar proof as in lemma \ref{alphabetanear}, by using arc-chord condition \eqref{arcchord502}. We only need to additionally show $V_h^{30}(\alpha,\gamma,t)-V_h^{30}(2\alpha+\tau(t),\eta,t), V_{\tilde{f}^{L}}^{61}(\alpha_{\gamma}^{t},t)-V_{\tilde{f}^{L}}^{61}((2\alpha+\tau(t))_{\eta}^{t},t),V_{X}(\alpha_{\gamma}^{t},t)-V_{X}((2\alpha+\tau(t))_{\eta}^{t})$ vanish at least of order $O(\alpha+\tau(t))$ when $\alpha\to \tau(t)^{-}$. From \eqref{notation03h} and \eqref{D-formularnew} we have
\begin{align*}
    \|\frac{V_h^{30}(\alpha,\gamma,t)}{(\alpha+\tau(t))}\|_{ \tilde{Y}^{k+4}}\lesssim \|h\|_{L_{\al}^2},
\end{align*}
\begin{align*}
    \|\frac{V_h^{30}(2\alpha+\tau(t),\eta,t)}{(\alpha+\tau(t))}\|_{ \tilde{Y}^{k+4}}\lesssim \|h\|_{L_{\al}^2}.
\end{align*}
Moreover, we have
\begin{align*}
&V_{\tilde{f}^{L}}^{61}(\alpha_{\gamma}^{t},t)-V_{\tilde{f}^{L}}^{61}((2\alpha+\tau(t))_{\eta}^{t},t)=\\ &V_{\tilde{f}^{L}}^{61}(\alpha_{\gamma}^{t},t)-V_{\tilde{f}^{L}}^{61}((2\alpha+\tau(t))_{\gamma}^{t},t)+V_{\tilde{f}^{L}}^{61}((2\alpha+\tau(t))_{\gamma}^{t},t)-V_{\tilde{f}^{L}}^{61}((2\alpha+\tau(t))_{\eta}^{t},t),
\end{align*}

and $\alpha_{\gamma}^{t}-\alpha_{\eta}^{t}=ic(\alpha+\tau(t))(\gamma-\eta)t$ with $\frac{c(\alpha+\tau(t))}{(\alpha+\tau(t))^2} \in C^{90}[-\tau(t),\pi]$. From the smoothness of $\tilde{f}^{L}$ \eqref{fLspace}, we get
\begin{align*}
    \|\frac{V_{\tilde{f}^{L}}^{61}(\alpha_{\gamma}^{t},t)-V_{\tilde{f}^{L}}^{61}((2\alpha+\tau(t))_{\eta}^{t},t)}{(\alpha+\tau(t))}\|_{ \tilde{Y}^{k+4}} \lesssim 1.
\end{align*}
Similarly, from \eqref{Xispace}, we have

\begin{align*}
    \|\frac{V_{X}(\alpha_{\gamma}^{t},t)-V_{X}((2\alpha+\tau(t))_{\eta}^{t},t)}{\alpha+\tau(t)}\|_{ \tilde{Y}^{k+4}}\lesssim 1.
\end{align*}

\end{proof}
\subsection{Commuting lemmas for Cauchy-Riemann operator}
Let
\begin{align*}
    A(h)(\alpha,\gamma)=(\frac{ic(\alpha)t}{1+ic'(\alpha)\gamma t}\frac{d}{d\alpha}-\frac{d}{d\gamma})h(\alpha,\gamma).
\end{align*}
Now we prove the following lemmas for $A$.
In those lemmas, the $\ddot{h}$, $\breve{h}$ we use are vector functions satisfying the conditions in the lemmas correspondingly.
\begin{lemma}\label{switch}
If all the derivatives are well-defined and $\partial_{\alpha}\partial_{\gamma}h(\alpha,\gamma,t)=\partial_{\gamma}\partial_{\alpha}h(\alpha,\gamma,t)$,  we have
\begin{align*}
A\circ\frac{\partial_{\alpha}h(\alpha,\gamma)}{1+ic'(\alpha)\gamma t}=\frac{\partial_{\alpha}}{1+ic'(\alpha)\gamma t}\circ A(h)(\alpha,\gamma),
\end{align*}
\end{lemma}
\begin{proof}
See lemma 9.2 in \cite{muskatregulatiryJia}.
\end{proof}

\begin{lemma}\label{analyticity 01}
If all the derivatives are well-defined,
$\frac{d}{d\alpha}\frac{d}{dt}h=\frac{d}{dt}\frac{d}{d\alpha}h$, $ \frac{d}{dt}\frac{d}{d\gamma}h=\frac{d}{d\gamma}\frac{d}{dt}h$, $\frac{d}{d\alpha}\frac{d}{d\gamma}h=\frac{d}{d\gamma}\frac{d}{d\alpha}h$,
then we have
\[
A(\frac{d}{dt}-\frac{ic(\al)\g}{1+ic'(\al)\g t }\frac{d}{d\al})h=(\frac{d}{dt}-\frac{ic(\al)\g}{1+ic'(\al)\g t}\frac{d}{d\al})A(h).
\]
\end{lemma}
\begin{proof}
See lemma 9.3 in \cite{muskatregulatiryJia}.
\end{proof}
\begin{lemma}\label{forM1}
Let $\tilde{K}$ be meromorphic. $X(\alpha,\gamma)\in C_{\gamma}^{1}([-1,1],C_{\alpha}^{1}[-\pi,\pi])$, $h(\alpha,\gamma)$ are well-defined vector functions with components in $C_{\gamma}^{1}([-1,1],C_{\alpha}^{1}[-\pi,\pi])$. If for fixed $\alpha$, there is no singular point in the integrals below and $c(\pi)=c(-\pi)=0$, $c(\alpha)\in W^{2,\infty}$, then we have
\begin{align*}
    &\quad A(\int_{-\pi}^{\pi}\tilde{K}(h(\alpha,\gamma)-h(\beta,\gamma))X(\beta,\gamma)(1+ic'(\beta)\gamma t)d\beta)\\
    &=\int_{-\pi}^{\pi}\nabla \tilde{K}(h(\alpha,\gamma)-h(\beta,\gamma))\cdot(A(h)(\alpha,\gamma)-A(h)(\beta,\gamma))X(\beta,\gamma)(1+ic'(\beta)\gamma t)d\beta\\
    &\quad+\int_{-\pi}^{\pi}\tilde{K}(h(\alpha,\gamma)-h(\beta,\gamma))A(X)(\beta,\gamma)(1+ic'(\beta)\gamma t)d\beta\\
    &=D_{h}(\int_{-\pi}^{\pi}\tilde{K}(h(\alpha,\gamma)-h(\beta,\gamma))X(\beta,\gamma)(1+ic'(\beta)\gamma t)d\beta)[A(h)]\\
    &\quad+\int_{-\pi}^{\pi}\tilde{K}(h(\alpha,\gamma)-h(\beta,\gamma))A(X)(\beta,\gamma)(1+ic'(\beta)\gamma t)d\beta.
\end{align*}
Here $D_h$ is the Gateaux derivative.
\end{lemma}
\begin{proof}
We have
\begin{align*}
 &A(\int_{\alpha-\pi}^{\alpha+\pi}\tilde{K}(h(\alpha,\gamma)-h(\alpha-\beta,\gamma))X(\alpha-\beta,\gamma)(1+ic'(\alpha-\beta)\gamma t)d\beta)=
    \\
    &\underbrace{\frac{ic(\alpha)t}{1+ic'(\alpha)\gamma t}(\tilde{K}(h(\alpha,\gamma)-h(-\pi,\gamma))X(-\pi,\gamma)(1+ic'(-\pi)\gamma t)-\tilde{K}(h(\alpha,\gamma)-h(\pi,\gamma))X(\pi,\gamma)(1+ic'(\pi)\gamma t))}_{Term_1}\\
    &+\underbrace{\int_{\alpha-\pi}^{\alpha+\pi}\nabla \tilde{K}(h(\alpha,\gamma)-h(\alpha-\beta,\gamma))\cdot((\frac{ic(\alpha)t}{1+ic'(\alpha)\gamma t}\partial_{\alpha}-\partial_{\gamma})h(\alpha,\gamma)-(\frac{ic(\alpha)t}{1+ic'(\alpha)\gamma t}\partial_{\alpha}-\partial_{\gamma})h(\alpha-\beta,\gamma,t))}_{Term_2}\\
    &\underbrace{X(\alpha-\beta,\gamma)(1+ic'(\alpha-\beta)\gamma t)d\beta}_{Term_2}\\
    &+\underbrace{\int_{\alpha-\pi}^{\alpha+\pi}\tilde{K}(h(\alpha,\gamma)-h(\alpha-\beta,\gamma))(\frac{ic(\alpha)t}{1+ic'(\alpha)\gamma t}\partial_{\alpha}-\partial_{\gamma})X(\alpha-\beta,\gamma)(1+ic'(\alpha-\beta)\gamma t)d\beta}_{Term_3}\\
    &+\underbrace{\int_{\alpha-\pi}^{\alpha+\pi}\tilde{K}(h(\alpha,\gamma)-h(\alpha-\beta,\gamma))X(\alpha-\beta,\gamma)(\frac{ic(\alpha)t}{1+ic'(\alpha)\gamma t}ic''(\alpha-\beta)\gamma t-ic'(\alpha-\beta) t)d\beta}_{Term_4}\\
    &=Term 1 +\\     
    &+\underbrace{\int_{\alpha-\pi}^{\alpha+\pi}\nabla \tilde{K}(h(\alpha,\gamma)-h(\alpha-\beta,\gamma))\cdot((\frac{ic(\alpha)t}{1+ic'(\alpha)\gamma t}\partial_{\alpha}-\partial_{\gamma})h(\alpha,\gamma)-(\frac{ic(\alpha-\beta)t}{1+ic'(\alpha-\beta)\gamma t}\partial_{\alpha}-\partial_{\gamma})h(\alpha-\beta,\gamma,t))}_{Term_{2,1}}\\
    &\underbrace{X(\alpha-\beta,\gamma)(1+ic'(\alpha-\beta)\gamma t)d\beta}_{Term_{2,1}}\\
    &+\underbrace{\int_{\alpha-\pi}^{\alpha+\pi}\nabla \tilde{K}(h(\alpha,\gamma)-h(\alpha-\beta,\gamma))\cdot((\frac{ic(\alpha-\beta)t}{1+ic'(\alpha-\beta)\gamma t}\partial_{\alpha}-\frac{ic(\alpha)t}{1+ic'(\alpha)\gamma t}\partial_{\alpha})h(\alpha-\beta,\gamma,t))}_{Term_{2,2}}\\
    &\underbrace{X(\alpha-\beta,\gamma)(1+ic'(\alpha-\beta)\gamma t)d\beta}_{Term_{2,2}}\\
    &+\underbrace{\int_{\alpha-\pi}^{\alpha+\pi}\tilde{K}(h(\alpha,\gamma)-h(\alpha-\beta,\gamma))(\frac{ic(\alpha-\beta)t}{1+ic'(\alpha-\beta)\gamma t}\partial_{\alpha}-\partial_{\gamma})X(\alpha-\beta,\gamma)(1+ic'(\alpha-\beta)\gamma t)d\beta}_{Term_{3,1}}\\
    &+\underbrace{\int_{\alpha-\pi}^{\alpha+\pi}\tilde{K}(h(\alpha,\gamma)-h(\alpha-\beta,\gamma))(\frac{ic(\alpha)t}{1+ic'(\alpha)\gamma t}\partial_{\alpha}-\frac{ic(\alpha-\beta)t}{1+ic'(\alpha-\beta)\gamma t}\partial_{\alpha})X(\alpha-\beta,\gamma)(1+ic'(\alpha-\beta)\gamma t)d\beta}_{Term_{3,2}}\\
     &+\underbrace{\int_{\alpha-\pi}^{\alpha+\pi}\tilde{K}(h(\alpha,\gamma)-h(\alpha-\beta,\gamma))X(\alpha-\beta,\gamma)(\frac{ic(\alpha)t}{1+ic'(\alpha)\gamma t}ic''(\alpha-\beta)\gamma t-ic'(\alpha-\beta) t)d\beta}_{Term_4}\\
    &=Term_1+Term_2+Term_{2,2}+Term_{3,1}+Term_{3,2}+Term_4.
\end{align*}
Moreover,
\begin{align*}
    &Term_{2,2}+Term_{3,2}+Term_4=\\
    &=\int_{\alpha-\pi}^{\alpha+\pi}\frac{d}{d\beta}(\tilde{K}(h(\alpha,\gamma)-h(\alpha-\beta,\gamma))X(\alpha-\beta,\gamma)(ic(\alpha-\beta)t-\frac{ic(\alpha)t}{1+ic'(\alpha)\gamma t}(1+ic'(\alpha-\beta)\gamma t)))d\beta\\
    &=\tilde{K}(h(\alpha,\gamma)-h(-\pi,\gamma))X(-\pi,\gamma)(ic(-\pi)t-\frac{ic(\alpha)t}{1+ic'(\alpha)\gamma t}(1+ic'(-\pi)\gamma t))\\
    &-\tilde{K}(h(\alpha,\gamma)-h(\pi,\gamma))X(\pi,\gamma)(ic(\pi)t-\frac{ic(\alpha)t}{1+ic'(\alpha)\gamma t}(1+ic'(\pi)\gamma t)).
\end{align*}
We use the condition that $c(-\pi)=c(\pi)=0$ and we could get the $Term_1+Term_{2,2}+Term_{3,2}+Term_4=0$. Then we have the result.
\end{proof}

\begin{lemma}\label{forO1}
Let $\breve{K}$ be meromorphic, $\ddot{h}(\alpha,\gamma), X(\alpha,\gamma) \in C_{\gamma}^{1}([-1,1],C_{\alpha}^{1}[-\pi,\pi])$, $\breve{h}(\alpha,\gamma) \in C_{\gamma}^{1}([-1,1],C_{\alpha}^{1}[0,\pi]).$ If for fixed $\alpha>0$, there is no singular point in the integrals of $O(\breve{h}, \ddot{h}, X)$ below and $c(0)=c(\pi)=c(-\pi)=0$, $c(\alpha)\in W^{2,\infty}$, then when $0<\alpha< \frac{\pi}{2}$,
for
\begin{align}\label{Oequationapp}
    &O(\breve{h}, \ddot{h}, X)=\int_{0}^{2\alpha}\breve{K}(\ddot{h}(\alpha,\gamma)-\ddot{h}(\beta,\gamma))(\breve{h}(\alpha,\gamma)-\breve{h}(\beta,\gamma))(X(\beta,\gamma))(1+ic'(\beta)\gamma t)d\beta\\\nonumber
    &+\int_{0}^{\gamma}\breve{h}(2\alpha,\eta) \breve{K}(\ddot{h}(\alpha,\gamma)-\ddot{h}(2\alpha,\eta))ic(2\alpha)t (X(2\alpha,\eta))d\eta\\\nonumber
    &-\int_{2\alpha}^{\pi}\breve{h}(\beta,0) \breve{K}(\ddot{h}(\alpha,\gamma)-\ddot{h}(\beta,0))(X(\beta,0))d\beta\\\nonumber
    &+\int_{[-\pi,0]\cup[2\alpha,\pi] } \breve{K}(\ddot{h}(\alpha,\gamma)-\ddot{h}(\beta,\gamma))(X(\beta,\gamma))(1+ic'(\beta)\gamma t)d\beta \breve{h}(\alpha,\gamma)\\\nonumber
    &=O_1+O_2+O_3+O_4,
\end{align}
we have
\begin{align*}
    &A(O(\breve{h}, \ddot{h},X))=\underbrace{\int_{0}^{2\alpha}\nabla \breve{K}(\ddot{h}(\alpha,\gamma)-\ddot{h}(\beta,\gamma))\cdot([A(\ddot{h})(\alpha,\gamma)]-[A(\ddot{h})(\beta,\gamma)])(1+ic'(\beta)\gamma t)}_{FA_1}\\
    &\underbrace{(\breve{h}(\alpha,\gamma)-\breve{h}(\beta,\gamma))X(\beta,\gamma)d\beta}_{FA_1}\\
&+\underbrace{\int_{0}^{2\alpha}\breve{K}(\ddot{h}(\alpha,\gamma)-\ddot{h}(\beta,\gamma))([A(\breve{h})(\alpha,\gamma)]-[A(\breve{h})(\beta,\gamma)])(X(\beta,\gamma))(1+ic'(\beta)\gamma t)d\beta}_{FA_2}\\
&+\underbrace{\int_{0}^{2\alpha}\breve{K}(\ddot{h}(\alpha,\gamma)-\ddot{h}(\beta,\gamma))(\breve{h}(\alpha,\gamma)-\breve{h}(\beta,\gamma))[A(X)(\beta,\gamma)](1+ic'(\beta)\gamma t)d\beta}_{FA_3}\\
       &\underbrace{+\int_{0}^{\gamma}\breve{h}(2\alpha,\eta) \nabla \breve{K}(\ddot{h}(\alpha,\gamma)-\ddot{h}(2\alpha,\eta))\cdot [A(\ddot{h})(\alpha,\gamma)]X(2\alpha,\eta)ic(2\alpha)t d\eta}_{FA_4}\\
       &\underbrace{-\int_{0}^{\gamma}\breve{h}(2\alpha,\eta) \nabla \breve{K}(\ddot{h}(\alpha,\gamma)-\ddot{h}(2\alpha,\eta))\cdot(\frac{2ic(\alpha)t}{1+ic'(\alpha)\gamma t})(\frac{1+ic'(2\alpha)\eta t}{ic(2\alpha)t})[A(\ddot{h})(2\alpha,\eta)]X(2\alpha,\eta)ic(2\alpha)td\eta}_{FA_5}\\
       &+\underbrace{\int_{0}^{\gamma}  \breve{h}(2\alpha,\eta)\breve{K}(\ddot{h}(\alpha,\gamma)-\ddot{h}(2\alpha,\eta))(\frac{2ic(\alpha)t}{1+ic'(\alpha)\gamma t})(\frac{1+ic'(2\alpha)\eta t}{ic(2\alpha)t})[A(X)(2\alpha,\eta)]ic(2\alpha)td\eta}_{FA_6}\\
       &+\underbrace{\int_{0}^{\gamma} (\frac{2ic(\alpha)t}{1+ic'(\alpha)\gamma t})(\frac{1+ic'(2\alpha)\eta t}{ic(2\alpha)t})[A(\breve{h})(2\alpha,\eta)] \breve{K}(\ddot{h}(\alpha,\gamma)-\ddot{h}(2\alpha,\eta))X(2\alpha,\eta)ic(2\alpha)td\eta}_{FA_7}\\
    &\underbrace{-\int_{2\alpha}^{\pi}\breve{h}(\beta,0) \nabla \breve{K}(\ddot{h}(\alpha,\gamma)-\ddot{h}(\beta,0))\cdot {[A(\ddot{h})(\alpha,\gamma)]}(X(\beta,0))d\beta}_{FA_8}\\
    &+\underbrace{\int_{[-\pi,0]\cup[2\alpha,\pi] } \nabla \breve{K}(\ddot{h}(\alpha,\gamma)-\ddot{h}(\beta,\gamma))\cdot([A(\ddot{h})(\alpha,\gamma)]-[A(\ddot{h})(\beta,\gamma)])(X(\beta,\gamma))(1+ic'(\beta)\gamma t))d\beta \breve{h}(\alpha,\gamma)}_{FA_{9}}\\
       &+\underbrace{\int_{[-\pi,0]\cup[2\alpha,\pi] }  \breve{K}(\ddot{h}(\alpha,\gamma)-\ddot{h}(\beta,\gamma))[A(X)(\beta,\gamma)](1+ic'(\beta)\gamma t))d\beta \breve{h}(\alpha,\gamma)}_{FA_{10}}\\
        &+\underbrace{\int_{[-\pi,0]\cup[2\alpha,\pi] }  \breve{K}(\ddot{h}(\alpha,\gamma)-\ddot{h}(\beta,\gamma))X(\beta,\gamma)(1+ic'(\beta)\gamma t))d\beta [A(\breve{h})(\alpha,\gamma)]}_{FA_{11}}.
\end{align*}
\end{lemma}
\begin{proof}
First, consider $O_1$,
\begin{align}\label{term01A}
&A(O_1)=\\\nonumber
&A(\int_{-\alpha}^{\alpha}\breve{K}(\ddot{h}(\alpha,\gamma)-\ddot{h}(\alpha-\beta,\gamma))(\breve{h}(\alpha,\gamma)-\breve{h}(\alpha-\beta,\gamma))X(\alpha-\beta,\gamma)(1+ic'(\alpha-\beta)\gamma t)d\beta)\\\nonumber
&=\underbrace{\frac{ic(\alpha)t}{1+ic'(\alpha)\gamma t}[\breve{K}(\ddot{h}(\alpha,\gamma)-\ddot{h}(0,\gamma))(\breve{h}(\alpha,\gamma)-\breve{h}(0,\gamma))X(0,\gamma)(1+ic'(0)\gamma t)}_{Term_1}\\\nonumber
&\underbrace{+\breve{K}(\ddot{h}(\alpha,\gamma)-\ddot{h}(2\alpha,\gamma))(\breve{h}(\alpha,\gamma)-\breve{h}(2\alpha,\gamma))X(2\alpha,\gamma)(1+ic'(2\alpha)\gamma t)]}_{Term_1}\\\nonumber
&\underbrace{+\int_{-\alpha}^{\alpha} \nabla \breve{K}(\ddot{h}(\alpha,\gamma)-\ddot{h}(\alpha-\beta,\gamma))\cdot([(\frac{ic(\alpha)t}{1+ic'(\alpha)\gamma t}\partial_{\alpha}-\partial_{\gamma})\ddot{h}(\alpha,\gamma)]-[(\frac{ic(\alpha)t}{1+ic'(\alpha)\gamma t}\partial_{\alpha}-\partial_{\gamma})\ddot{h}(\alpha-\beta,\gamma)])}_{Term_2}\\\nonumber
&\underbrace{(1+ic'(\alpha-\beta)\gamma t)(\breve{h}(\alpha,\gamma)-\breve{h}(\alpha-\beta,\gamma))X(\alpha-\beta,\gamma)d\beta}_{Term_2}\\\nonumber
&\underbrace{+\int_{-\alpha}^{\alpha}\breve{K}(\ddot{h}(\alpha,\gamma)-\ddot{h}(\alpha-\beta,\gamma))([(\frac{ic(\alpha)t}{1+ic'(\alpha)\gamma t}\partial_{\alpha}-\partial_{\gamma})\breve{h}(\alpha,\gamma)]-[(\frac{ic(\alpha)t}{1+ic'(\alpha)\gamma t}\partial_{\alpha}-\partial_{\gamma})\breve{h}(\alpha-\beta,\gamma)])}_{Term_3}\\\nonumber
&\underbrace{X(\alpha-\beta,\gamma)(1+ic'(\alpha-\beta)\gamma t)d\beta}_{Term_3}\\\nonumber
&\underbrace{+\int_{-\alpha}^{\alpha}\breve{K}(\ddot{h}(\alpha,\gamma)-\ddot{h}(\alpha-\beta,\gamma))(\breve{h}(\alpha,\gamma)-\breve{h}(\alpha-\beta,\gamma))[(\frac{ic(\alpha)t}{1+ic'(\alpha)\gamma t)}\partial_{\alpha}-\partial_{\gamma})X(\alpha-\beta,\gamma)]}_{Term_4}\\\nonumber
&\underbrace{(1+ic'(\alpha-\beta)\gamma t)d\beta}_{Term_4}\\\nonumber
&\underbrace{+\int_{-\alpha}^{\alpha}\breve{K}(\ddot{h}(\alpha,\gamma)-\ddot{h}(\alpha-\beta,\gamma))(\breve{h}(\alpha,\gamma)-\breve{h}(\alpha-\beta,\gamma))X(\alpha-\beta,\gamma)[\frac{ic(\alpha)t}{1+ic'(\alpha)\gamma t}(ic''(\alpha-\beta)\gamma t)-ic'(\alpha-\beta)t]d\beta}_{Term_5}\\\nonumber
&=Term_1+Term_2+Term_3+Term_4+Term_5.
\end{align}

From $Term_2$ to $Term_5$ we separate each term into the part depending on $A(\cdot)$ and the other part, and have
\begin{align*}
&Term_2+Term_3+Term_4+Term_5\\
&=\underbrace{\int_{-\alpha}^{\alpha} \nabla \breve{K}(\ddot{h}(\alpha,\gamma)-\ddot{h}(\alpha-\beta,\gamma))\cdot([(\frac{ic(\alpha)t}{1+ic'(\alpha)\gamma t}\partial_{\alpha}-\partial_{\gamma})\ddot{h}(\alpha,\gamma)]-[(\frac{ic(\alpha-\beta)t}{1+ic'(\alpha-\beta)\gamma t}\partial_{\alpha}-\partial_{\gamma})\ddot{h}(\alpha-\beta,\gamma)])}_{Term_{2,1}}\\
&\underbrace{(1+ic'(\alpha-\beta)\gamma t)X(\alpha-\beta,\gamma)(\breve{h}(\alpha,\gamma)-\breve{h}(\alpha-\beta,\gamma))d\beta}_{Term_{2,1}}\\
&\underbrace{+\int_{-\alpha}^{\alpha} \nabla \breve{K}(\ddot{h}(\alpha,\gamma)-\ddot{h}(\alpha-\beta,\gamma))\cdot[(\frac{ic(\alpha-\beta)t}{1+ic'(\alpha-\beta)\gamma t}-\frac{ic(\alpha)t}{1+ic'(\alpha)\gamma t})\partial_{\alpha}\ddot{h}(\alpha-\beta,\gamma)]}_{Term_{2,2}}\\
&\underbrace{(1+ic'(\alpha-\beta)\gamma t)X(\alpha-\beta,\gamma)(\breve{h}(\alpha,\gamma)-\breve{h}(\alpha-\beta,\gamma))d\beta}_{Term_{2,2}}\\
&\underbrace{+\int_{-\alpha}^{\alpha}\breve{K}(\ddot{h}(\alpha,\gamma)-\ddot{h}(\alpha-\beta,\gamma))([(\frac{ic(\alpha)t}{1+ic'(\alpha)\gamma t}\partial_{\alpha}-\partial_{\gamma})\breve{h}(\alpha,\gamma)]-[(\frac{ic(\alpha-\beta)t}{1+ic'(\alpha-\beta)\gamma t}\partial_{\alpha}-\partial_{\gamma})\breve{h}(\alpha-\beta,\gamma)])}_{Term_{3,1}}\\
&\underbrace{X(\alpha-\beta,\gamma)(1+ic'(\alpha-\beta)\gamma t)d\beta}_{Term_{3,1}}\\
&\underbrace{+\int_{-\alpha}^{\alpha}\breve{K}(\ddot{h}(\alpha,\gamma)-\ddot{h}(\alpha-\beta,\gamma))[(\frac{ic(\alpha-\beta)t}{1+ic'(\alpha-\beta)\gamma t}-\frac{ic(\alpha)t}{1+ic'(\alpha)\gamma t})\partial_{\alpha})\breve{h}(\alpha-\beta,\gamma)]}_{Term_{3,2}}\\
&\underbrace{X(\alpha-\beta,\gamma)(1+ic'(\alpha-\beta)\gamma t)d\beta}_{Term_{3,2}}\\
&\underbrace{+\int_{-\alpha}^{\alpha}\breve{K}(\ddot{h}(\alpha,\gamma)-\ddot{h}(\alpha-\beta,\gamma))(\breve{h}(\alpha,\gamma)-\breve{h}(\alpha-\beta,\gamma))[(\frac{ic(\alpha-\beta)t}{1+ic'(\alpha-\beta)\gamma t}\partial_{\alpha}-\partial_{\gamma})X(\alpha-\beta,\gamma)]}_{Term_{4,1}}\\
&\underbrace{(1+ic'(\alpha-\beta)\gamma t)d\beta}_{Term_{4,1}}\\
&\underbrace{+\int_{-\alpha}^{\alpha}\breve{K}(\ddot{h}(\alpha,\gamma)-\ddot{h}(\alpha-\beta,\gamma))(\breve{h}(\alpha,\gamma)-\breve{h}(\alpha-\beta,\gamma))[(\frac{ic(\alpha)t}{1+ic'(\alpha)\gamma t}\partial_{\alpha}-\frac{ic(\alpha-\beta)t}{1+ic'(\alpha-\beta)\gamma t}\partial_{\alpha})X(\alpha-\beta,\gamma)]}_{Term_{4,2}}\\
&\underbrace{(1+ic'(\alpha-\beta)\gamma t)d\beta}_{Term_{4,2}}\\
&+\underbrace{\int_{-\alpha}^{\alpha}\breve{K}(\ddot{h}(\alpha,\gamma)-\ddot{h}(\alpha-\beta,\gamma))(\breve{h}(\alpha,\gamma)-\breve{h}(\alpha-\beta,\gamma))X(\alpha-\beta,\gamma)}_{Term_5}\\
&\underbrace{[\frac{ic(\alpha)t}{1+ic'(\alpha)\gamma t}(ic''(\alpha-\beta)\gamma t)-ic'(\alpha-\beta)t]d\beta}_{Term_5}\\
&=Term_{2,1}+Term_{2,2}+Term_{3,1}+Term_{3,2}+Term_{4,1}+Term_{4,2}+Term_{5}.
\end{align*}
We could write the parts which does not depend on $A(\cdot)$ together, and have
\begin{align*}
&Term_{2,2}+Term_{3,2}+Term_{4,2}+Term_5\\    
&=-\int_{-\alpha}^{\alpha}\frac{d}{d\beta}[\breve{K}(\ddot{h}(\alpha,\gamma)-\ddot{h}(\alpha-\beta,\gamma))(\frac{ic(\alpha)t}{1+ic'(\alpha)\gamma t}-\frac{ic(\alpha-\beta)t}{1+ic'(\alpha-\beta)\gamma t})(1+ic'(\alpha-\beta)\gamma t)\\
&X(\alpha-\beta,\gamma)(\breve{h}(\alpha,\gamma)-\breve{h}(\alpha-\beta,\gamma))]d\beta\\
&=-\breve{K}(\ddot{h}(\alpha,\gamma)-\ddot{h}(0,\gamma))(\frac{ic(\alpha)t}{1+ic'(\alpha)\gamma t}-\frac{ic(0)t}{1+ic'(0)\gamma t})(1+ic'(0)\gamma t)(\breve{h}(\alpha,\gamma)-\breve{h}(0,\gamma))X(0,\gamma)\\
&+\breve{K}(\ddot{h}(\alpha,\gamma)-\ddot{h}(2\alpha,\gamma))(\frac{ic(\alpha)t}{1+ic'(\alpha)\gamma t}-\frac{ic(2\alpha)t}{1+ic'(2\alpha)\gamma t})(\breve{h}(\alpha,\gamma)-\breve{h}(2\alpha,\gamma))X(2\alpha,\gamma)(1+ic'(2\alpha)\gamma t).
\end{align*}
Therefore we have
\begin{align*}
    A(O_1)=&\underbrace{\frac{ic(\alpha)t}{1+ic'(\alpha)\gamma t}[\breve{K}(\ddot{h}(\alpha,\gamma)-\ddot{h}(0,\gamma))(\breve{h}(\alpha,\gamma)-\breve{h}(0,\gamma))X(0,\gamma)(1+ic'(0)\gamma t)}_{B_1}\\\nonumber
&+\underbrace{\breve{K}(\ddot{h}(\alpha,\gamma)-\ddot{h}(2\alpha,\gamma))(\breve{h}(\alpha,\gamma)-\breve{h}(2\alpha,\gamma))X(2\alpha,\gamma)(1+ic'(2\alpha)\gamma t)]}_{B_1}\\\nonumber
    &+\underbrace{Term_{2,1}}_{FA_1}+\underbrace{Term_{3,1}}_{FA_2}+\underbrace{Term_{4,1}}_{FA_3}\\\nonumber
    &\underbrace{-\breve{K}(\ddot{h}(\alpha,\gamma)-\ddot{h}(0,\gamma))(\frac{ic(\alpha)t}{1+ic'(\alpha)\gamma t}-\frac{ic(0)t}{1+ic'(0)\gamma t})(1+ic'(0)\gamma t)(\breve{h}(\alpha,\gamma)-\breve{h}(0,\gamma))X(0,\gamma)}_{B_2}\\\nonumber
&+\underbrace{\breve{K}(\ddot{h}(\alpha,\gamma)-\ddot{h}(2\alpha,\gamma))(\frac{ic(\alpha)t}{1+ic'(\alpha)\gamma t}-\frac{ic(2\alpha)t}{1+ic'(2\alpha)\gamma t})(\breve{h}(\alpha,\gamma)-\breve{h}(2\alpha,\gamma))X(2\alpha,\gamma)(1+ic'(2\alpha)\gamma t)}_{B_3}\\\nonumber
&=\underbrace{Term_{2,1}}_{FA_1}+\underbrace{Term_{3,1}}_{FA_2}+\underbrace{Term_{4,1}}_{FA_3}\\\nonumber
&\underbrace{+\frac{ic(\alpha)t}{1+ic'(\alpha)\gamma t}[\breve{K}(\ddot{h}(\alpha,\gamma)-\ddot{h}(0,\gamma))(\breve{h}(\alpha,\gamma))X(0,\gamma)(1+ic'(0)\gamma t)}_{B_{1,1}}\\\nonumber
&+\underbrace{\breve{K}(\ddot{h}(\alpha,\gamma)-\ddot{h}(2\alpha,\gamma))(\breve{h}(\alpha,\gamma))X(2\alpha,\gamma)(1+ic'(2\alpha)\gamma t)]}_{B_{1,1}}\\\nonumber&\underbrace{+\frac{ic(\alpha)t}{1+ic'(\alpha)\gamma t}[\breve{K}(\ddot{h}(\alpha,\gamma)-\ddot{h}(0,\gamma))(-\breve{h}(0,\gamma))X(0,\gamma)(1+ic'(0)\gamma t)}_{B_{1,2}}\\\nonumber
&+\underbrace{\breve{K}(\ddot{h}(\alpha,\gamma)-\ddot{h}(2\alpha,\gamma))(-\breve{h}(2\alpha,\gamma))X(2\alpha,\gamma)(1+ic'(2\alpha)\gamma t)]}_{B_{1,2}}\\\nonumber
 &\underbrace{-\breve{K}(\ddot{h}(\alpha,\gamma)-\ddot{h}(0,\gamma))(\frac{ic(\alpha)t}{1+ic'(\alpha)\gamma t}-\frac{ic(0)t}{1+ic'(0)\gamma t})(1+ic'(0)\gamma t)(\breve{h}(\alpha,\gamma))X(0,\gamma)}_{B_{2,1}}\\\nonumber
  &\underbrace{-\breve{K}(\ddot{h}(\alpha,\gamma)-\ddot{h}(0,\gamma))(\frac{ic(\alpha)t}{1+ic'(\alpha)\gamma t}-\frac{ic(0)t}{1+ic'(0)\gamma t})(1+ic'(0)\gamma t)(-\breve{h}(0,\gamma))X(0,\gamma)}_{B_{2,2}}\\\nonumber
&+\underbrace{\breve{K}(\ddot{h}(\alpha,\gamma)-\ddot{h}(2\alpha,\gamma))(\frac{ic(\alpha)t}{1+ic'(\alpha)\gamma t}-\frac{ic(2\alpha)t}{1+ic'(2\alpha)\gamma t})(\breve{h}(\alpha,\gamma))X(2\alpha,\gamma)(1+ic'(2\alpha)\gamma t)}_{B_{3,1}}\\\nonumber
&+\underbrace{\breve{K}(\ddot{h}(\alpha,\gamma)-\ddot{h}(2\alpha,\gamma))(\frac{ic(\alpha)t}{1+ic'(\alpha)\gamma t}-\frac{ic(2\alpha)t}{1+ic'(2\alpha)\gamma t})(-\breve{h}(2\alpha,\gamma))X(2\alpha,\gamma)(1+ic'(2\alpha)\gamma t)}_{B_{3,2}}.
\end{align*}
Here $B_1=B_{1,1}+B_{1,2}$, $B_2=B_{2,1}+B_{2,2}$, $B_{3}=B_{3,1}+B_{3,2}$.

Then we consider $A(O_2)$, we have
\begin{align*}
    &A(O_2)=A(\int_{0}^{\gamma}\breve{h}(2\alpha,\eta) \breve{K}(\ddot{h}(\alpha,\gamma)-\ddot{h}(2\alpha,\eta))X(2\alpha,\eta)ic(2\alpha)t d\eta)\\
    &=\underbrace{\int_{0}^{\gamma}(\frac{2ic(\alpha)t}{1+ic'(\alpha)\gamma t}(\partial_{\alpha}h)(2\alpha,\eta))X(2\alpha,\eta) \breve{K}(\ddot{h}(\alpha,\gamma)-\ddot{h}(2\alpha,\eta))ic(2\alpha)t d\eta}_{Term_1}\\
   &+ \underbrace{\int_{0}^{\gamma}\breve{h}(2\alpha,\eta)(\frac{2ic(\alpha)t}{1+ic'(\alpha)\gamma t}(\partial_{\alpha}X)(2\alpha,\eta)) \breve{K}(\ddot{h}(\alpha,\gamma)-\ddot{h}(2\alpha,\eta))ic(2\alpha)t d\eta}_{Term_2}\\
    &\underbrace{-\breve{h}(2\alpha,\gamma)X(2\alpha,\gamma)\breve{K}(\ddot{h}(\alpha,\gamma)-\ddot{h}(2\alpha,\gamma))ic(2\alpha)t}_{Term_3}\\
    &\underbrace{+\int_{0}^{\gamma}\breve{h}(2\alpha,\eta) X(2\alpha,\eta)\nabla \breve{K}(\ddot{h}(\alpha,\gamma)-\ddot{h}(2\alpha,\eta))\cdot(\frac{ic(\alpha)t}{1+ic'(\alpha)\gamma t}(\partial_{\alpha}\ddot{h})(\alpha,\gamma)-(\partial_{\gamma}\ddot{h})(\alpha,\gamma))ic(2\alpha)t d\eta}_{Term_4=FA_4}\\
    &\underbrace{-\int_{0}^{\gamma}\breve{h}(2\alpha,\eta)X(2\alpha,\eta)\nabla \breve{K}(\ddot{h}(\alpha,\gamma)-\ddot{h}(2\alpha,\eta))\cdot(\frac{2ic(\alpha)t}{1+ic'(\alpha)\gamma t}(\partial_{\alpha}h)(2\alpha,\eta))ic(2\alpha)t d\eta}_{Term_5}\\
    &\underbrace{+\int_{0}^{\gamma}\breve{h}(2\alpha,\eta)X(2\alpha,\eta) \breve{K}(\ddot{h}(\alpha,\gamma)-\ddot{h}(2\alpha,\eta))\frac{i2c(\alpha)t}{1+ic'(\alpha)\gamma t}ic'(2\alpha)t d\eta}_{Term_6}\\
    &=Term_1+Term_2+Term_3+Term_4+Term_5+Term_6
    \end{align*}
       We can again do the change to the terms 1, 2, 5 by writing the terms in $A(\cdot)$.
    \begin{align*}
    &Term_1+Term_2+Term_5=\\
    &\underbrace{\int_{0}^{\gamma}ic(2\alpha)t (\frac{2ic(\alpha)t}{1+ic'(\alpha)\gamma t})[(\frac{1}{2}\frac{d}{d\alpha}-\frac{1+ic'(2\alpha)\eta t}{ic(2\alpha)t}\frac{d}{d\eta})\breve{h}(2\alpha,\eta)]X(2\alpha,\eta) \breve{K}(\ddot{h}(\alpha,\gamma)-\ddot{h}(2\alpha,\eta))d\eta}_{Term_{1,1}}\\
  &\underbrace{+\int_{0}^{\gamma}ic(2\alpha)t (\frac{2ic(\alpha)t}{1+ic'(\alpha)\gamma t})[(\frac{1+ic'(2\alpha)\eta t}{ic(2\alpha)t}\frac{d}{d\eta})\breve{h}(2\alpha,\eta)]X(2\alpha,\eta) \breve{K}(\ddot{h}(\alpha,\gamma)-\ddot{h}(2\alpha,\eta))d\eta}_{Term_{1,2}}\\
   &\underbrace{+\int_{0}^{\gamma}ic(2\alpha)t (\frac{2ic(\alpha)t}{1+ic'(\alpha)\gamma t})[(\frac{1}{2}\frac{d}{d\alpha}-\frac{1+ic'(2\alpha)\eta t}{ic(2\alpha)t}\frac{d}{d\eta})X(2\alpha,\eta)]\breve{h}(2\alpha,\eta) \breve{K}(\ddot{h}(\alpha,\gamma)-\ddot{h}(2\alpha,\eta))d\eta}_{Term_{2,1}}\\
  &\underbrace{+\int_{0}^{\gamma}ic(2\alpha)t (\frac{2ic(\alpha)t}{1+ic'(\alpha)\gamma t})[(\frac{1+ic'(2\alpha)\eta t}{ic(2\alpha)t}\frac{d}{d\eta})X(2\alpha,\eta)]\breve{h}(2\alpha,\eta) \breve{K}(\ddot{h}(\alpha,\gamma)-\ddot{h}(2\alpha,\eta))d\eta}_{Term_{2,2}}\\
   &\underbrace{-\int_{0}^{\gamma}\breve{h}(2\alpha,\eta)X(2\alpha,\eta)\nabla \breve{K}(\ddot{h}(\alpha,\gamma)-\ddot{h}(2\alpha,\eta))\cdot(\frac{2ic(\alpha)t}{1+ic'(\alpha)\gamma t})[(\frac{1}{2}\frac{d}{d\alpha}-\frac{1+ic'(2\alpha)\eta t}{ic(2\alpha)t}\frac{d}{d\eta})\ddot{h}(2\alpha,\eta)]ic(2\alpha)t d\eta}_{Term_{5,1}}\\
      &\underbrace{-\int_{0}^{\gamma}\breve{h}(2\alpha,\eta)X(2\alpha,\eta)\nabla \breve{K}(\ddot{h}(\alpha,\gamma)-\ddot{h}(2\alpha,\eta))\cdot(\frac{2ic(\alpha)t}{1+ic'(\alpha)\gamma t})[(\frac{1+ic'(2\alpha)\eta t}{ic(2\alpha)t}\frac{d}{d\eta})\ddot{h}(2\alpha,\eta)]ic(2\alpha)t d\eta}_{Term_{5,2}}\\\
 &=Term_{1,1}+Term_{1,2}+Term_{2,1}+Term_{2,2}+Term_{5,1}+Term_{5,2}.
 \end{align*}
For those terms not depending on $A(\cdot)$,  we have
 \begin{align*}
     &Term_{1,2}+Term_{2,2}+Term_{5,2}+Term_6\\
     &=\int_{0}^{\gamma}ic(2\alpha)t (\frac{2ic(\alpha)t}{1+ic'(\alpha)\gamma t})[(\frac{1+ic'(2\alpha)\eta t}{ic(2\alpha)t}\frac{d}{d\eta})\breve{h}(2\alpha,\eta)]X(2\alpha,\eta) \breve{K}(\ddot{h}(\alpha,\gamma)-\ddot{h}(2\alpha,\eta))d\eta\\
     &+\int_{0}^{\gamma}ic(2\alpha)t (\frac{2ic(\alpha)t}{1+ic'(\alpha)\gamma t})[(\frac{1+ic'(2\alpha)\eta t}{ic(2\alpha)t}\frac{d}{d\eta})X(2\alpha,\eta)]\breve{h}(2\alpha,\eta) \breve{K}(\ddot{h}(\alpha,\gamma)-\ddot{h}(2\alpha,\eta))d\eta\\
     &-\int_{0}^{\gamma}\breve{h}(2\alpha,\eta)X(2\alpha,\eta)\nabla \breve{K}(\ddot{h}(\alpha,\gamma)-\ddot{h}(2\alpha,\eta))\cdot(\frac{2ic(\alpha)t}{1+ic'(\alpha)\gamma t})[(\frac{1+ic'(2\alpha)\eta t}{ic(2\alpha)t}\frac{d}{d\eta})\ddot{h}(2\alpha,\eta)]ic(2\alpha)t d\eta\\
     &+\int_{0}^{\gamma}\breve{h}(2\alpha,\eta)X(2\alpha,\eta) \breve{K}(\ddot{h}(\alpha,\gamma)-\ddot{h}(2\alpha,\eta))\frac{i2c(\alpha)t}{1+ic'(\alpha)\gamma t}ic'(2\alpha)t d\eta\\
     &=\int_{0}^{\gamma}\frac{d}{d\eta}[\breve{h}(2\alpha,\eta)X(2\alpha,\eta)\breve{K}(\ddot{h}(\alpha,\gamma)-\ddot{h}(2\alpha,\eta))\frac{i2c(\alpha)t}{1+ic'(\alpha)\gamma t}(ic'(2\alpha)\eta t+1)]d\eta\\
     &=(\breve{h}(2\alpha,\gamma)X(2\alpha,\gamma)\breve{K}(\ddot{h}(\alpha,\gamma)-\ddot{h}(2\alpha,\gamma))\frac{i2c(\alpha)t}{1+ic'(\alpha)\gamma t}(ic'(2\alpha)\gamma t+1)) \\
    &-\breve{h}(2\alpha,0)X(2\alpha,0)\breve{K}(\ddot{h}(\alpha,\gamma)-\ddot{h}(2\alpha,0))\frac{i2c(\alpha)t}{1+ic'(\alpha)\gamma t}.
   \end{align*}
   Moreover, we have
   \begin{align*}
       &Term_{1,1}=\int_{0}^{\gamma}ic(2\alpha)t (\frac{2ic(\alpha)t}{1+ic'(\alpha)\gamma t})[(\frac{1}{2}\frac{d}{d\alpha}-\frac{1+ic'(2\alpha)\eta t}{ic(2\alpha)t}\frac{d}{d\eta})\breve{h}(2\alpha,\eta)]X(2\alpha,\eta) \breve{K}(\ddot{h}(\alpha,\gamma)-\ddot{h}(2\alpha,\eta))d\eta\\
       &=\int_{0}^{\gamma}ic(2\alpha)t (\frac{2ic(\alpha)t}{1+ic'(\alpha)\gamma t})(\frac{1+ic'(2\alpha)\eta t}{ic(2\alpha)t})[(\frac{ic(2\alpha)t}{2(1+ic'(2\alpha)\eta t)}\frac{d}{d\alpha}-\frac{d}{d\eta})\breve{h}(2\alpha,\eta)]X(2\alpha,\eta) \\
       &\qquad \cdot\breve{K}(\ddot{h}(\alpha,\gamma)-\ddot{h}(2\alpha,\eta))d\eta\\
       &=\underbrace{\int_{0}^{\gamma}ic(2\alpha)t (\frac{2ic(\alpha)t}{1+ic'(\alpha)\gamma t})(\frac{1+ic'(2\alpha)\eta t}{ic(2\alpha)t})[A(\breve{h})(2\alpha,\eta)] \breve{K}(\ddot{h}(\alpha,\gamma)-\ddot{h}(2\alpha,\eta))X(2\alpha,\eta)d\eta}_{FA_7},
   \end{align*}
   \begin{align*}
       &Term_{2,1}=\int_{0}^{\gamma}ic(2\alpha)t (\frac{2ic(\alpha)t}{1+ic'(\alpha)\gamma t})[(\frac{1}{2}\frac{d}{d\alpha}-\frac{1+ic'(2\alpha)\eta t}{ic(2\alpha)t}\frac{d}{d\eta})X(2\alpha,\eta)]\breve{h}(2\alpha,\eta) \breve{K}(\ddot{h}(\alpha,\gamma)-\ddot{h}(2\alpha,\eta))d\eta\\
       &=\int_{0}^{\gamma}ic(2\alpha)t (\frac{2ic(\alpha)t}{1+ic'(\alpha)\gamma t})(\frac{1+ic'(2\alpha)\eta t}{ic(2\alpha)t})[(\frac{ic(2\alpha)t}{2(1+ic'(2\alpha)\eta t)}\frac{d}{d\alpha}-\frac{d}{d\eta})X(2\alpha,\eta)]\breve{h}(2\alpha,\eta)\\
       &\qquad \cdot \breve{K}(\ddot{h}(\alpha,\gamma)-\ddot{h}(2\alpha,\eta))d\eta\\
       &=\underbrace{\int_{0}^{\gamma}ic(2\alpha)t (\frac{2ic(\alpha)t}{1+ic'(\alpha)\gamma t})(\frac{1+ic'(2\alpha)\eta t}{ic(2\alpha)t})[A(X)(2\alpha,\eta)] \breve{K}(\ddot{h}(\alpha,\gamma)-\ddot{h}(2\alpha,\eta))\breve{h}(2\alpha,\eta)d\eta}_{FA_6},
   \end{align*}
   and
   \begin{align*}
       &Term_{5,1}=-\int_{0}^{\gamma}\breve{h}(2\alpha,\eta)X(2\alpha,\eta)\nabla \breve{K}(\ddot{h}(\alpha,\gamma)-\ddot{h}(2\alpha,\eta))\cdot(\frac{2ic(\alpha)t}{1+ic'(\alpha)\gamma t})[(\frac{1}{2}\frac{d}{d\alpha}-\frac{1+ic'(2\alpha)\eta t}{ic(2\alpha)t}\frac{d}{d\eta})\ddot{h}(2\alpha,\eta)]\\
       &ic(2\alpha)t d\eta\\
       &=-\int_{0}^{\gamma}  \breve{h}(2\alpha,\eta)X(2\alpha,\eta)\nabla \breve{K}(\ddot{h}(\alpha,\gamma)-\ddot{h}(2\alpha,\eta))\cdot(\frac{2ic(\alpha)t}{1+ic'(\alpha)\gamma t})(\frac{1+ic'(2\alpha)\eta t}{ic(2\alpha)t})\\
       &\cdot[(\frac{ic(2\alpha)t}{2(1+ic'(2\alpha)\eta t)}\frac{d}{d\alpha}-\frac{d}{d\eta})\ddot{h}(2\alpha,\eta)]ic(2\alpha)td\eta\\
       &=\underbrace{-\int_{0}^{\gamma}  \breve{h}(2\alpha,\eta)X(2\alpha,\eta)\nabla \breve{K}(\ddot{h}(\alpha,\gamma)-\ddot{h}(2\alpha,\eta))\cdot(\frac{2ic(\alpha)t}{1+ic'(\alpha)\gamma t})(\frac{1+ic'(2\alpha)\eta t}{ic(2\alpha)t})[A(\ddot{h})(2\alpha,\eta)]ic(2\alpha)td\eta}_{FA_5}.
   \end{align*}
   Therefore we have
   \begin{align*}
       &A(O_2)=Term_{1,1}+Term_{2,1}+Term_{5,1}+Term_4+Term_{3}+(Term_{1,2}+Term_{2,2}+ Term_{5,2}+Term_{6})\\
       &=\underbrace{\int_{0}^{\gamma}ic(2\alpha)t (\frac{2ic(\alpha)t}{1+ic'(\alpha)\gamma t})(\frac{1+ic'(2\alpha)\eta t}{ic(2\alpha)t})[A(\breve{h})(2\alpha,\eta)] \breve{K}(\ddot{h}(\alpha,\gamma)-\ddot{h}(2\alpha,\eta))X(2\alpha,\eta)d\eta}_{FA_7}\\
       &+\underbrace{\int_{0}^{\gamma}ic(2\alpha)t (\frac{2ic(\alpha)t}{1+ic'(\alpha)\gamma t})(\frac{1+ic'(2\alpha)\eta t}{ic(2\alpha)t})[A(X)(2\alpha,\eta)] \breve{K}(\ddot{h}(\alpha,\gamma)-\ddot{h}(2\alpha,\eta))\breve{h}(2\alpha,\eta)d\eta}_{FA_6}\\
       &\underbrace{-\int_{0}^{\gamma}  \breve{h}(2\alpha,\eta)X(2\alpha,\eta)\nabla \breve{K}(\ddot{h}(\alpha,\gamma)-\ddot{h}(2\alpha,\eta))\cdot(\frac{2ic(\alpha)t}{1+ic'(\alpha)\gamma t})(\frac{1+ic'(2\alpha)\eta t}{ic(2\alpha)t})[A(\ddot{h})(2\alpha,\eta)]ic(2\alpha)td\eta}_{FA_5}\\
       &\underbrace{+\int_{0}^{\gamma}\breve{h}(2\alpha,\eta) X(2\alpha,\eta)\nabla \breve{K}(\ddot{h}(\alpha,\gamma)-\ddot{h}(2\alpha,\eta))\cdot[(A(\ddot{h})(\alpha,\gamma)]ic(2\alpha)t d\eta}_{FA_4}\\
       &\underbrace{-\breve{h}(2\alpha,\gamma)X(2\alpha,\gamma)\breve{K}(\ddot{h}(\alpha,\gamma)-\ddot{h}(2\alpha,\gamma))ic(2\alpha)t}_{B_4}\\
    &\underbrace{+\breve{h}(2\alpha,\gamma)X(2\alpha,\gamma)\breve{K}(\ddot{h}(\alpha,\gamma)-\ddot{h}(2\alpha,\gamma))\frac{i2c(\alpha)t}{1+ic'(\alpha)\gamma t}(ic'(2\alpha)\gamma t+1)}_{B_5}\\
    &\underbrace{-\breve{h}(2\alpha,0)X(2\alpha,0)\breve{K}(\ddot{h}(\alpha,\gamma)-\ddot{h}(2\alpha,0))\frac{i2c(\alpha)t}{1+ic'(\alpha)\gamma t}}_{B_6} 
   \end{align*}
 For $O_3$, we have
   \begin{align*}
    &A(O_3)=A(-\int_{2\alpha}^{\pi}X(\beta,0)\breve{h}(\beta,0) \breve{K}(\ddot{h}(\alpha,\gamma)-\ddot{h}(\beta,0))d\beta)\\
    &=\underbrace{\frac{2ic(\alpha)t}{1+ic'(\alpha)\gamma t}\breve{h}(2\alpha,0)X(2\alpha,0) \breve{K}(\ddot{h}(\alpha,\gamma)-\ddot{h}(2\alpha,0))}_{B_7}\\
    &\underbrace{-\int_{2\alpha}^{\pi}X(\beta,0)\breve{h}(\beta,0) \nabla \breve{K}(\ddot{h}(\alpha,\gamma)-\ddot{h}(\beta,0))\cdot(\frac{ic(\alpha)t}{1+ic'(\alpha)\gamma t}\partial_{\alpha}-\partial_{\gamma})\ddot{h}(\alpha,\gamma)d\beta}_{FA_8}.
\end{align*}
Finally, for $O_4$, we have
\begin{align*}
  &A(O_4)=A(\int_{[\alpha-\pi,-\alpha]\cup[\alpha,\alpha+\pi] } \breve{K}(\ddot{h}(\alpha,\gamma)-\ddot{h}(\alpha-\beta,\gamma))(X(\alpha-\beta,\gamma))(1+ic'(\alpha-\beta)\gamma t)d\beta \breve{h}(\alpha,\gamma))\\
&=\underbrace{\int_{[-\pi,0]\cup[2\alpha,\pi] } \breve{K}(\ddot{h}(\alpha,\gamma)-\ddot{h}(\beta,\gamma))(X(\beta,\gamma))(1+ic'(\beta)\gamma t)d\beta A(\breve{h})(\alpha,\gamma)}_{FA_{11}}\\
    &\underbrace{+\frac{ic(\alpha)t}{1+ic'(\alpha)\gamma t}\breve{K}(\ddot{h}(\alpha,\gamma)-\ddot{h}(-\pi,\gamma))X(-\pi,\gamma)(1+ic'(-\pi)\gamma t)\breve{h}(\alpha,\gamma)}_{Term_{1,1}, B_{8,1}}\\
      &\underbrace{-\frac{ic(\alpha)t}{1+ic'(\alpha)\gamma t}\breve{K}(\ddot{h}(\alpha,\gamma)-\ddot{h}(0,\gamma))X(0,\gamma)(1+ic'(0)\gamma t)\breve{h}(\alpha,\gamma)}_{Term_{1,2}, B_{8,2}}\\  &\underbrace{-\frac{ic(\alpha)t}{1+ic'(\alpha)\gamma t}\breve{K}(\ddot{h}(\alpha,\gamma)-\ddot{h}(2\alpha,\gamma))X(2\alpha,\gamma)(1+ic'(2\alpha)\gamma t)\breve{h}(\alpha,\gamma)}_{Term_{1,3},B_{8,3}}\\
     &\underbrace{-\frac{ic(\alpha)t}{1+ic'(\alpha)\gamma t}\breve{K}(\ddot{h}(\alpha,\gamma)-\ddot{h}(\pi,\gamma))X(\pi,\gamma)(1+ic'(\pi)\gamma t)\breve{h}(\alpha,\gamma)}_{Term_{1,4},B_{8,4}}\\
    &+\underbrace{\int_{[\alpha-\pi,-\alpha]\cup[\alpha,\alpha+\pi] }X(\alpha-\beta,\gamma)\nabla \breve{K}(\ddot{h}(\alpha,\gamma)-\ddot{h}(\alpha-\beta,\gamma))}_{Term_2}\\
    &\underbrace{\cdot([(\frac{ic(\alpha)t}{1+ic'(\alpha)\gamma t}\partial_{\alpha}-\partial_{\gamma})\ddot{h}(\alpha,\gamma)]-[(\frac{ic(\alpha)t}{1+ic'(\alpha)\gamma t}\partial_{\alpha}-\partial_{\gamma})\ddot{h}(\alpha-\beta,\gamma,t)])
(1+ic'(\alpha-\beta)\gamma t)d\beta \breve{h}(\alpha,\gamma)}_{Term_2}\\
    &+\underbrace{\int_{[\alpha-\pi,-\alpha]\cup[\alpha,\alpha+\pi] }\breve{K}(\ddot{h}(\alpha,\gamma)-\ddot{h}(\alpha-\beta,\gamma))[(\frac{ic(\alpha)t}{1+ic'(\alpha)\gamma t}\partial_{\alpha}-\partial_{\gamma})X(\alpha-\beta,\gamma)](1+ic'(\alpha-\beta)\gamma t)d\beta \breve{h}(\alpha,\gamma)}_{Term_3}\\
    &+\underbrace{\int_{[\alpha-\pi,-\alpha]\cup[\alpha,\alpha+\pi] }\breve{K}(\ddot{h}(\alpha,\gamma)-\ddot{h}(\alpha-\beta,\gamma))X(\alpha-\beta,\gamma)(\frac{ic(\alpha)t}{1+ic'(\alpha)\gamma t}ic''(\alpha-\beta)\gamma t-ic'(\alpha-\beta) t)d\beta \breve{h}(\alpha,\gamma)}_{Term_4}\\
    &=FA_{11}+Term_{1,1}+Term_{1,2}+Term_{1,3}+Term_{1,4}+\\     
  &+\underbrace{\int_{[\alpha-\pi,-\alpha]\cup[\alpha,\alpha+\pi] }X(\alpha-\beta,\gamma)\nabla \breve{K}(\ddot{h}(\alpha,\gamma)-\ddot{h}(\alpha-\beta,\gamma))}_{Term_{2,1}}\\
    &\underbrace{\cdot([(\frac{ic(\alpha)t}{1+ic'(\alpha)\gamma t}\partial_{\alpha}-\partial_{\gamma})\ddot{h}(\alpha,\gamma)]-[(\frac{ic(\alpha-\beta)t}{1+ic'(\alpha-\beta)\gamma t}\partial_{\alpha}-\partial_{\gamma})\ddot{h}(\alpha-\beta,\gamma)])
(1+ic'(\alpha-\beta)\gamma t)d\beta \breve{h}(\alpha,\gamma)}_{Term_{2,1}}\\
    &+\underbrace{\int_{[\alpha-\pi,-\alpha]\cup[\alpha,\alpha+\pi] }\nabla \breve{K}(\ddot{h}(\alpha,\gamma)-\ddot{h}(\alpha-\beta,\gamma))\cdot[(\frac{ic(\alpha-\beta)t}{1+ic'(\alpha-\beta)\gamma t}\partial_{\alpha}-\frac{ic(\alpha)t}{1+ic'(\alpha)\gamma t}\partial_{\alpha})\ddot{h}(\alpha-\beta,\gamma)]}_{Term_{2,2}}\\
    &\underbrace{X(\alpha-\beta,\gamma)(1+ic'(\alpha-\beta)\gamma t)d\beta \breve{h}(\alpha,\gamma)}_{Term_{2,2}}\\
    &+\underbrace{\int_{[\alpha-\pi,-\alpha]\cup[\alpha,\alpha+\pi] }\breve{K}(\ddot{h}(\alpha,\gamma)-\ddot{h}(\alpha-\beta,\gamma))[(\frac{ic(\alpha-\beta)t}{1+ic'(\alpha-\beta)\gamma t}\partial_{\alpha}-\partial_{\gamma})X(\alpha-\beta,\gamma)](1+ic'(\alpha-\beta)\gamma t)d\beta \breve{h}(\alpha,\gamma)}_{Term_{3,1}}\\
    &+\underbrace{\int_{[\alpha-\pi,-\alpha]\cup[\alpha,\alpha+\pi] }\breve{K}(\ddot{h}(\alpha,\gamma)-\ddot{h}(\alpha-\beta,\gamma))[(\frac{ic(\alpha)t}{1+ic'(\alpha)\gamma t}\partial_{\alpha}-\frac{ic(\alpha-\beta)t}{1+ic'(\alpha-\beta)\gamma t}\partial_{\alpha})X(\alpha-\beta,\gamma)]}_{Term_{3,2}}\\
     &\underbrace{(1+ic'(\alpha-\beta)\gamma t)d\beta \breve{h}(\alpha,\gamma)}_{Term_{3,2}}\\
     &+\underbrace{\int_{[\alpha-\pi,-\alpha]\cup[\alpha,\alpha+\pi] }\breve{K}(\ddot{h}(\alpha,\gamma)-\ddot{h}(\alpha-\beta,\gamma))X(\alpha-\beta,\gamma)(\frac{ic(\alpha)t}{1+ic'(\alpha)\gamma t}ic''(\alpha-\beta)\gamma t-ic'(\alpha-\beta) t)d\beta \breve{h}(\alpha,\gamma)}_{Term_4}\\
    &=FA_{11}+Term_{1,1}+Term_{1,2}+Term_{1,3}+Term_{1,4}+Term_{2,1}+Term_{2,2}+Term_{3,1}+Term_{3,2}+Term_4.
\end{align*}
Then we have
\begin{align*}
    &Term_{2,2}+Term_{3,2}+Term_4=\\
    &=\int_{[\alpha-\pi,-\alpha]\cup[\alpha,\alpha+\pi] }\frac{d}{d\beta}[\breve{K}(\ddot{h}(\alpha,\gamma)-\ddot{h}(\alpha-\beta,\gamma))X(\alpha-\beta,\gamma)(ic(\alpha-\beta)t-\frac{ic(\alpha)t}{1+ic'(\alpha)\gamma t}(1+ic'(\alpha-\beta)\gamma t))]d\beta\\
  &\cdot \breve{h}(\alpha,\gamma)\\
    &=\underbrace{\breve{K}(\ddot{h}(\alpha,\gamma)-\ddot{h}(-\pi,\gamma))X(-\pi,\gamma)(ic(-\pi)t-\frac{ic(\alpha)t}{1+ic'(\alpha)\gamma t}(1+ic'(-\pi)\gamma t)) \breve{h}(\alpha,\gamma)}_{B_{9,1}}\\
    &\underbrace{-\breve{K}(\ddot{h}(\alpha,\gamma)-\ddot{h}(0,\gamma))X(0,\gamma)(ic(0)t-\frac{ic(\alpha)t}{1+ic'(\alpha)\gamma t}(1+ic'(0)\gamma t)) \breve{h}(\alpha,\gamma)}_{B_{9,2}}\\
    &\underbrace{+\breve{K}(\ddot{h}(\alpha,\gamma)-\ddot{h}(2\alpha,\gamma))X(2\alpha,\gamma)(ic(2\alpha)t-\frac{ic(\alpha)t}{1+ic'(\alpha)\gamma t}(1+ic'(2\alpha)\gamma t)) \breve{h}(\alpha,\gamma)}_{B_{9,3}}\\
    &\underbrace{-\breve{K}(\ddot{h}(\alpha,\gamma)-\ddot{h}(\pi,\gamma))X(\pi,\gamma)(ic(\pi)t-\frac{ic(\alpha)t}{1+ic'(\alpha)\gamma t}(1+ic'(\pi)\gamma t))\breve{h}(\alpha,\gamma).}_{B_{9,4}}
\end{align*}
Then 
\begin{align*}
    &A(O_4)=FA_{11}+\underbrace{Term_{1,1}+Term_{1,2}+Term_{1,3}+Term_{1,4}}_{B_8}+\underbrace{Term_{2,1}}_{FA_9}\\
    &+\underbrace{Term_{3,1}}_{FA_{10}}+\underbrace{Term_{2,2}+Term_{3,2}+Term_4}_{B_9}.
\end{align*}
Here we use $B_8$ to show the sum $B_{8,1}+B_{8,2}+B_{8,3}+B_{8,4}$, and $B_9$ to show the sum $B_{9,1}+B_{9,2}+B_{9,3}+B_{9,4}=B_{9}$.

Then we are left to show all B terms are cancelled.
\begin{align*}
&B_1+B_2+B_3+B_4+B_5+B_6+B_7+B_8+B_9\\
    &=\underbrace{\frac{ic(\alpha)t}{1+ic'(\alpha)\gamma t}\breve{K}(\ddot{h}(\alpha,\gamma)-\ddot{h}(0,\gamma))(\breve{h}(\alpha,\gamma))X(0,\gamma)(1+ic'(0)\gamma t)}_{B_{1,1}}\\
&+\underbrace{\frac{ic(\alpha)t}{1+ic'(\alpha)\gamma t}\breve{K}(\ddot{h}(\alpha,\gamma)-\ddot{h}(2\alpha,\gamma))(\breve{h}(\alpha,\gamma))X(2\alpha,\gamma)(1+ic'(2\alpha)\gamma t)}_{B_{1,1}}\\&\underbrace{+\frac{ic(\alpha)t}{1+ic'(\alpha)\gamma t}\breve{K}(\ddot{h}(\alpha,\gamma)-\ddot{h}(0,\gamma))(-\breve{h}(0,\gamma))X(0,\gamma)(1+ic'(0)\gamma t)}_{B_{1,2}}\\
&+\underbrace{\frac{ic(\alpha)t}{1+ic'(\alpha)\gamma t}\breve{K}(\ddot{h}(\alpha,\gamma)-\ddot{h}(2\alpha,\gamma))(-\breve{h}(2\alpha,\gamma))X(2\alpha,\gamma)(1+ic'(2\alpha)\gamma t)}_{B_{1,2}}\\
 &\underbrace{-\breve{K}(\ddot{h}(\alpha,\gamma)-\ddot{h}(0,\gamma))(\frac{ic(\alpha)t}{1+ic'(\alpha)\gamma t}-\frac{ic(0)t}{1+ic'(0)\gamma t})(1+ic'(0)\gamma t)(\breve{h}(\alpha,\gamma))X(0,\gamma)}_{B_{2,1}}\\
  &\underbrace{-\breve{K}(\ddot{h}(\alpha,\gamma)-\ddot{h}(0,\gamma))(\frac{ic(\alpha)t}{1+ic'(\alpha)\gamma t}-\frac{ic(0)t}{1+ic'(0)\gamma t})(1+ic'(0)\gamma t)(-\breve{h}(0,\gamma))X(0,\gamma)}_{B_{2,2}}\\
&+\underbrace{\breve{K}(\ddot{h}(\alpha,\gamma)-\ddot{h}(2\alpha,\gamma))(\frac{ic(\alpha)t}{1+ic'(\alpha)\gamma t}-\frac{ic(2\alpha)t}{1+ic'(2\alpha)\gamma t})(\breve{h}(\alpha,\gamma))X(2\alpha,\gamma)(1+ic'(2\alpha)\gamma t)}_{B_{3,1}}\\
&+\underbrace{\breve{K}(\ddot{h}(\alpha,\gamma)-\ddot{h}(2\alpha,\gamma))(\frac{ic(\alpha)t}{1+ic'(\alpha)\gamma t}-\frac{ic(2\alpha)t}{1+ic'(2\alpha)\gamma t})(-\breve{h}(2\alpha,\gamma))X(2\alpha,\gamma)(1+ic'(2\alpha)\gamma t)}_{B_{3,2}}\\
&\underbrace{-\breve{h}(2\alpha,\gamma)X(2\alpha,\gamma)\breve{K}(\ddot{h}(\alpha,\gamma)-\ddot{h}(2\alpha,\gamma))ic(2\alpha)t}_{B_4}\\
    &\underbrace{+\breve{h}(2\alpha,\gamma)X(2\alpha,\gamma)\breve{K}(\ddot{h}(\alpha,\gamma)-\ddot{h}(2\alpha,\gamma))\frac{i2c(\alpha)t}{1+ic'(\alpha)\gamma t}(ic'(2\alpha)\gamma t+1)}_{B_5}\\
    &\underbrace{-\breve{h}(2\alpha,0)X(2\alpha,0)\breve{K}(\ddot{h}(\alpha,\gamma)-\ddot{h}(2\alpha,0))\frac{i2c(\alpha)t}{1+ic'(\alpha)\gamma t}}_{B_6}\\ 
   & \underbrace{+\frac{2ic(\alpha)t}{1+ic'(\alpha)\gamma t}\breve{h}(2\alpha,0)X(2\alpha,0) \breve{K}(\ddot{h}(\alpha,\gamma)-\ddot{h}(2\alpha,0))}_{B_7}\\
    &\underbrace{+\frac{ic(\alpha)t}{1+ic'(\alpha)\gamma t}\breve{K}(\ddot{h}(\alpha,\gamma)-\ddot{h}(-\pi,\gamma))X(-\pi,\gamma)(1+ic'(-\pi)\gamma t)\breve{h}(\alpha,\gamma)}_{B_{8,1}}\\
      &\underbrace{-\frac{ic(\alpha)t}{1+ic'(\alpha)\gamma t}\breve{K}(\ddot{h}(\alpha,\gamma)-\ddot{h}(0,\gamma))X(0,\gamma)(1+ic'(0)\gamma t)\breve{h}(\alpha,\gamma)}_{B_{8,2}}\\  &\underbrace{-\frac{ic(\alpha)t}{1+ic'(\alpha)\gamma t}\breve{K}(\ddot{h}(\alpha,\gamma)-\ddot{h}(2\alpha,\gamma))X(2\alpha,\gamma)(1+ic'(2\alpha)\gamma t)\breve{h}(\alpha,\gamma)}_{B_{8,3}}\\
    &\underbrace{-\frac{ic(\alpha)t}{1+ic'(\alpha)\gamma t}\breve{K}(\ddot{h}(\alpha,\gamma)-\ddot{h}(\pi,\gamma))X(\pi,\gamma)(1+ic'(\pi)\gamma t)\breve{h}(\alpha,\gamma)}_{B_{8,4}}\\
     &\underbrace{+\breve{K}(\ddot{h}(\alpha,\gamma)-\ddot{h}(-\pi,\gamma))X(-\pi,\gamma)(ic(-\pi)t-\frac{ic(\alpha)t}{1+ic'(\alpha)\gamma t}(1+ic'(-\pi)\gamma t)) \breve{h}(\alpha,\gamma)}_{B_{9,1}}\\
    &\underbrace{-\breve{K}(\ddot{h}(\alpha,\gamma)-\ddot{h}(0,\gamma))X(0,\gamma)(ic(0)t-\frac{ic(\alpha)t}{1+ic'(\alpha)\gamma t}(1+ic'(0)\gamma t)) \breve{h}(\alpha,\gamma)}_{B_{9,2}}\\
    &\underbrace{+\breve{K}(\ddot{h}(\alpha,\gamma)-\ddot{h}(2\alpha,\gamma))X(2\alpha,\gamma)(ic(2\alpha)t-\frac{ic(\alpha)t}{1+ic'(\alpha)\gamma t}(1+ic'(2\alpha)\gamma t)) \breve{h}(\alpha,\gamma)}_{B_{9,3}}\\
    &\underbrace{-\breve{K}(\ddot{h}(\alpha,\gamma)-\ddot{h}(\pi,\gamma))X(\pi,\gamma)(ic(\pi)t-\frac{ic(\alpha)t}{1+ic'(\alpha)\gamma t}(1+ic'(\pi)\gamma t))\breve{h}(\alpha,\gamma)}_{B_{9,4}}\\
    &=\underbrace{\frac{ic(\alpha)t}{1+ic'(\alpha)\gamma t}\breve{K}(\ddot{h}(\alpha,\gamma)-\ddot{h}(0,\gamma))(\breve{h}(\alpha,\gamma))X(0,\gamma)(1+ic'(0)\gamma t)}_{C_1}\\
&\underbrace{+\frac{ic(\alpha)t}{1+ic'(\alpha)\gamma t}\breve{K}(\ddot{h}(\alpha,\gamma)-\ddot{h}(2\alpha,\gamma))(\breve{h}(\alpha,\gamma))X(2\alpha,\gamma)(1+ic'(2\alpha)\gamma t)}_{C_2}\\&\underbrace{+\frac{ic(\alpha)t}{1+ic'(\alpha)\gamma t}\breve{K}(\ddot{h}(\alpha,\gamma)-\ddot{h}(0,\gamma))(-\breve{h}(0,\gamma))X(0,\gamma)(1+ic'(0)\gamma t)}_{C_3}\\
&\underbrace{+\frac{ic(\alpha)t}{1+ic'(\alpha)\gamma t}\breve{K}(\ddot{h}(\alpha,\gamma)-\ddot{h}(2\alpha,\gamma))(-\breve{h}(2\alpha,\gamma))X(2\alpha,\gamma)(1+ic'(2\alpha)\gamma t)}_{C_4}\\
 &\underbrace{-\breve{K}(\ddot{h}(\alpha,\gamma)-\ddot{h}(0,\gamma))(\frac{ic(\alpha)t}{1+ic'(\alpha)\gamma t}-\frac{ic(0)t}{1+ic'(0)\gamma t})(1+ic'(0)\gamma t)(\breve{h}(\alpha,\gamma))X(0,\gamma)}_{C_5}\\
  &\underbrace{-\breve{K}(\ddot{h}(\alpha,\gamma)-\ddot{h}(0,\gamma))(\frac{ic(\alpha)t}{1+ic'(\alpha)\gamma t}-\frac{ic(0)t}{1+ic'(0)\gamma t})(1+ic'(0)\gamma t)(-\breve{h}(0,\gamma))X(0,\gamma)}_{C_3}\\
&\underbrace{+\breve{K}(\ddot{h}(\alpha,\gamma)-\ddot{h}(2\alpha,\gamma))(\frac{ic(\alpha)t}{1+ic'(\alpha)\gamma t}-\frac{ic(2\alpha)t}{1+ic'(2\alpha)\gamma t})(\breve{h}(\alpha,\gamma))X(2\alpha,\gamma)(1+ic'(2\alpha)\gamma t)}_{C_6}\\
&+\breve{K}(\ddot{h}(\alpha,\gamma)-\ddot{h}(2\alpha,\gamma))(\underbrace{\frac{ic(\alpha)t}{1+ic'(\alpha)\gamma t}}_{C_4}\underbrace{-\frac{ic(2\alpha)t}{1+ic'(2\alpha)\gamma t}}_{C_7})(-\breve{h}(2\alpha,\gamma))X(2\alpha,\gamma)(1+ic'(2\alpha)\gamma t)\\
&\underbrace{-\breve{h}(2\alpha,\gamma)X(2\alpha,\gamma)\breve{K}(\ddot{h}(\alpha,\gamma)-\ddot{h}(2\alpha,\gamma))ic(2\alpha)t}_{C_7}\\
     &\underbrace{+\breve{h}(2\alpha,\gamma)X(2\alpha,\gamma)\breve{K}(\ddot{h}(\alpha,\gamma)-\ddot{h}(2\alpha,\gamma))\frac{i2c(\alpha)t}{1+ic'(\alpha)\gamma t}(ic'(2\alpha)\gamma t+1)}_{C_4}\\
    &\underbrace{-\breve{h}(2\alpha,0)X(2\alpha,0)\breve{K}(\ddot{h}(\alpha,\gamma)-\ddot{h}(2\alpha,0))\frac{i2c(\alpha)t}{1+ic'(\alpha)\gamma t}}_{C_8}\\ 
   & \underbrace{+\frac{2ic(\alpha)t}{1+ic'(\alpha)\gamma t}\breve{h}(2\alpha,0)X(2\alpha,0) \breve{K}(\ddot{h}(\alpha,\gamma)-\ddot{h}(2\alpha,0))}_{C_8}\\
    &\underbrace{+\frac{ic(\alpha)t}{1+ic'(\alpha)\gamma t}\breve{K}(\ddot{h}(\alpha,\gamma)-\ddot{h}(-\pi,\gamma))X(-\pi,\gamma)(1+ic'(-\pi)\gamma t)\breve{h}(\alpha,\gamma)}_{C_9}\\
      &\underbrace{-\frac{ic(\alpha)t}{1+ic'(\alpha)\gamma t}\breve{K}(\ddot{h}(\alpha,\gamma)-\ddot{h}(0,\gamma))X(0,\gamma)(1+ic'(0)\gamma t)\breve{h}(\alpha,\gamma)}_{C_1}\\  &\underbrace{-\frac{ic(\alpha)t}{1+ic'(\alpha)\gamma t}\breve{K}(\ddot{h}(\alpha,\gamma)-\ddot{h}(2\alpha,\gamma))X(2\alpha,\gamma)(1+ic'(2\alpha)\gamma t)\breve{h}(\alpha,\gamma)}_{C_2}\\
     &\underbrace{-\frac{ic(\alpha)t}{1+ic'(\alpha)\gamma t}\breve{K}(\ddot{h}(\alpha,\gamma)-\ddot{h}(\pi,\gamma))X(\pi,\gamma)(1+ic'(\pi)\gamma t)\breve{h}(\alpha,\gamma)}_{C_{10}}\\
     &\underbrace{+\breve{K}(\ddot{h}(\alpha,\gamma)-\ddot{h}(-\pi,\gamma))X(-\pi,\gamma)(ic(-\pi)t-\frac{ic(\alpha)t}{1+ic'(\alpha)\gamma t}(1+ic'(-\pi)\gamma t)) \breve{h}(\alpha,\gamma)}_{C_{9}}\\
    &\underbrace{-\breve{K}(\ddot{h}(\alpha,\gamma)-\ddot{h}(0,\gamma))X(0,\gamma)(ic(0)t-\frac{ic(\alpha)t}{1+ic'(\alpha)\gamma t}(1+ic'(0)\gamma t)) \breve{h}(\alpha,\gamma)}_{C_{5}}\\
    &\underbrace{+\breve{K}(\ddot{h}(\alpha,\gamma)-\ddot{h}(2\alpha,\gamma))X(2\alpha,\gamma)(ic(2\alpha)t-\frac{ic(\alpha)t}{1+ic'(\alpha)\gamma t}(1+ic'(2\alpha)\gamma t)) \breve{h}(\alpha,\gamma)}_{C_{6}}\\
    &\underbrace{-\breve{K}(\ddot{h}(\alpha,\gamma)-\ddot{h}(\pi,\gamma))X(\pi,\gamma)(ic(\pi)t-\frac{ic(\alpha)t}{1+ic'(\alpha)\gamma t}(1+ic'(\pi)\gamma t))\breve{h}(\alpha,\gamma)}_{C_{10}}\\
    &=0.
\end{align*}
Here we write the same $C_i$ under terms that cancel with each other.
\end{proof}
\begin{lemma}\label{forM2}
Let $K$ be meromorphic, $\ddot{h}(\alpha,\gamma) \in C_{\gamma}^{1}([-1,1],C_{\alpha}^{2}[-\pi,\pi])$, $\breve{h}(\alpha,\gamma) \in C_{\gamma}^{1}([-1,1],C_{\alpha}^{2}[0,\pi])$. If for fixed $\alpha>0$, there is no singular point in the integrals in $M(\breve{h}, \ddot{h})$ below and $c(0)=c(\pi)=c(-\pi)=0$, $c(\alpha)\in W^{2,\infty}$, then when $0<\alpha< \frac{\pi}{2}$, for
\begin{align*}
    M(\breve{h}, \ddot{h})=&\int_{0}^{2\alpha}\breve{K}(\ddot{h}(\alpha,\gamma)-\ddot{h}(\beta,\gamma))(\frac{(\frac{d}{d\alpha}\breve{h})(\alpha,\gamma)}{1+ic'(\alpha)\gamma t}-\frac{\frac{d}{d\beta}\breve{h}(\beta,\gamma)}{1+ic'(\beta)\gamma t})(1+ic'(\beta)\gamma t)d\beta\\
   &+\breve{K}(\ddot{h}(\alpha,\gamma)-\ddot{h}(2\alpha,\gamma))\breve{h}(2\alpha,\gamma)\\
   &+\int_{0}^{\gamma}\breve{h}(2\alpha,\eta)\nabla \breve{K}(\ddot{h}(\alpha,\gamma)-\ddot{h}(2\alpha,\eta))\cdot(\frac{(\frac{d}{d\alpha}\ddot{h})(2\alpha,\eta)}{1+ic'(2\alpha)\eta t})ic(2\alpha)t d\eta\\
   &-\int_{2\alpha}^{\pi}\breve{h}(\beta,0)\nabla \breve{K}(\ddot{h}(\alpha,\gamma)-\ddot{h}(\beta,0))\cdot\frac{d\ddot{h}}{d\beta}(\beta,0)d\beta\\
    &+\int_{[-\pi,0]\cup[2\alpha,\pi] } \breve{K}(\ddot{h}(\alpha,\gamma)-\ddot{h}(\beta,\gamma))(1+ic'(\beta)\gamma t)d\beta \frac{(\frac{d}{d\alpha}\breve{h})(\alpha,\gamma)}{1+ic'(\alpha)\gamma t}\\
&=M_{1}+M_{2}+M_{3}+M_{4}+M_5.
   \end{align*}

we have 
\begin{align*}
    A(M)=&\underbrace{\int_{0}^{2\alpha}D_{\ddot{h}}(\breve{K}(\ddot{h}(\alpha,\gamma)-\ddot{h}(\beta,\gamma)))[A(\ddot{h})](\frac{(\frac{d}{d\alpha}\breve{h})(\alpha,\gamma)}{1+ic'(\alpha)\gamma t}-\frac{\frac{d}{d\beta}\breve{h}(\beta,\gamma)}{1+ic'(\beta)\gamma t})(1+ic'(\beta)\gamma t)d\beta}_{FM_{1,1}}\\
    &+\underbrace{\int_{0}^{2\alpha}\breve{K}(\ddot{h}(\alpha,\gamma)-\ddot{h}(\beta,\gamma))(\frac{\frac{d}{d\alpha}A(\breve{h})(\alpha,\gamma)}{1+ic'(\alpha)\gamma t}-\frac{\frac{d}{d\beta}A(\breve{h})(\beta,\gamma)}{1+ic'(\beta)\gamma t})(1+ic'(\beta)\gamma t)d\beta}_{FM_{1,2}}\\
    &\underbrace{+D_{\ddot{h}}\breve{K}(\ddot{h}(\alpha,\gamma)-\ddot{h}(2\alpha,\gamma))[A(\ddot{h})]\breve{h}(2\alpha,\gamma)}_{FM_{2,1}}\\
    &\underbrace{+\breve{K}(\ddot{h}(\alpha,\gamma)-\ddot{h}(2\alpha,\gamma))A(\breve{h})(2\alpha,\gamma)}_{FM_{2,2}}\\
    &\underbrace{+\int_{0}^{\gamma}\breve{h}(2\alpha,\eta)\nabla^2 \breve{K}(\ddot{h}(\alpha,\gamma)-\ddot{h}(2\alpha,\eta))[(A[\ddot{h}](\alpha,\gamma)),\frac{(\frac{d}{d\alpha}\ddot{h})(2\alpha,\eta)}{1+ic'(2\alpha)\eta t}]ic(2\alpha)t d\eta}_{FM_{3,1}}\\
    &\underbrace{-\int_{0}^{\gamma}\breve{h}(2\alpha,\eta)\nabla^2 \breve{K}(\ddot{h}(\alpha,\gamma)-\ddot{h}(2\alpha,\eta))[(\frac{2ic(\alpha)t(1+ic'(2\alpha)\eta t)}{(1+ic'(\alpha)\gamma t)(ic(2\alpha)t)}A(\ddot{h})(2\alpha,\eta)),\frac{(\frac{d}{d\alpha}\ddot{h})(2\alpha,\eta)}{1+ic'(2\alpha)\eta t}]ic(2\alpha)t d\eta}_{FM_{3,2}}\\
    &\underbrace{+\int_{0}^{\gamma}\breve{h}(2\alpha,\eta)\nabla \breve{K}(\ddot{h}(\alpha,\gamma)-\ddot{h}(2\alpha,\eta))\cdot\frac{2ic(\alpha)t(1+ic'(2\alpha)\eta t)}{(1+ic'(\alpha)\gamma t)(ic(2\alpha) t)}(\frac{(\frac{d}{d\alpha}A(\ddot{h}))(2\alpha,\eta)}{1+ic'(2\alpha)\eta t})ic(2\alpha)t d\eta}_{FM_{3,3}}\\
    &\underbrace{+\int_{0}^{\gamma}A(\breve{h})(2\alpha,\eta)\frac{2ic(\alpha)t(1+ic'(2\alpha)\eta t)}{(1+ic'(\alpha)\gamma t)(ic(2\alpha) t)}\nabla \breve{K}(\ddot{h}(\alpha,\gamma)-\ddot{h}(2\alpha,\eta))\cdot(\frac{(\frac{d}{d\alpha}\ddot{h})(2\alpha,\eta)}{1+ic'(2\alpha)\eta t})ic(2\alpha)t d\eta}_{FM_{3,4}}\\
    &\underbrace{-\int_{2\alpha}^{\pi}\breve{h}(\beta,0)\nabla^2\breve{K}(\ddot{h}(\alpha,\gamma)-\ddot{h}(\beta,0))[A(\ddot{h})(\alpha,\gamma),\frac{d\ddot{h}}{d\beta}(\beta,0)]d\beta}_{FM_{4}}\\
      &+\underbrace{\int_{[-\pi,0]\cup[2\alpha,\pi] } \nabla \breve{K}(\ddot{h}(\alpha,\gamma)-\ddot{h}(\beta,\gamma))\cdot(A(\ddot{h})(\alpha,\gamma)-A(\ddot{h})(\beta,\gamma))(1+ic'(\beta)\gamma t))d\beta \frac{(\frac{d}{d\alpha}\breve{h})(\alpha,\gamma)}{1+ic'(\alpha)\gamma t}}_{FM_{5,1}}\\
        &+\underbrace{\int_{[-\pi,0]\cup[2\alpha,\pi] }  \breve{K}(\ddot{h}(\alpha,\gamma)-\ddot{h}(\beta,\gamma))(1+ic'(\beta)\gamma t))d\beta \frac{\frac{d}{d\alpha}A(\breve{h})(\alpha,\gamma)}{1+ic'(\alpha)\gamma t}}_{FM_{5,2}}.
\end{align*}
\end{lemma}
\begin{proof}
First, we do the $M_1$

First, we compare $O_1$ in \eqref{term01A} with $M_1$. Since in $M_1$, we only change $\breve{h}(\alpha,\gamma)-\breve{h}(\beta,\gamma)$ by $\frac{\frac{d\breve{h}(\alpha,\gamma)}{d\alpha}}{1+ic'(\alpha)\gamma t}-\frac{\frac{d\breve{h}(\beta,\gamma)}{d\beta}}{1+ic'(\beta)\gamma t}$ and letting $X(\beta,\gamma)=1$,  we have
\begin{align*}
 A(M_1)=\\
&A(\int_{-\alpha}^{\alpha}\breve{K}(\ddot{h}(\alpha,\gamma)-\ddot{h}(\alpha-\beta,\gamma))(\frac{(\frac{d}{d\alpha}\breve{h})(\alpha,\gamma)}{1+ic'(\alpha)\gamma t}-\frac{\frac{d}{d\alpha}\breve{h}(\alpha-\beta,\gamma)}{1+ic'(\alpha-\beta)\gamma t})(1+ic'(\alpha-\beta)\gamma t)d\beta)\\
=&\underbrace{\frac{ic(\alpha)t}{1+ic'(\alpha)\gamma t}[\breve{K}(\ddot{h}(\alpha,\gamma)-\ddot{h}(0,\gamma))(\frac{(\frac{d}{d\alpha}\breve{h})(\alpha,\gamma)}{1+ic'(\alpha)\gamma t}-\frac{\frac{d}{d\alpha}\breve{h}(0,\gamma)}{1+ic'(0)\gamma t})(1+ic'(0)\gamma t)}_{B_1}\\
&+\underbrace{\breve{K}(\ddot{h}(\alpha,\gamma)-\ddot{h}(2\alpha,\gamma))(\frac{(\frac{d}{d\alpha}\breve{h})(\alpha,\gamma)}{1+ic'(\alpha)\gamma t}-\frac{(\frac{d}{d\alpha}\breve{h})(2\alpha,\gamma)}{1+ic'(2\alpha)\gamma t})(1+ic'(2\alpha)\gamma t)]}_{B_1}\\
  & \underbrace{+\int_{0}^{2\alpha}\nabla \breve{K}(\ddot{h}(\alpha,\gamma)-\ddot{h}(\beta,\gamma))\cdot(A(\ddot{h})(\alpha,\gamma)-A(\ddot{h})(\beta,\gamma))(1+ic'(\beta)\gamma t)(\frac{(\frac{d}{d\alpha}\breve{h})(\alpha,\gamma)}{1+ic'(\alpha)\gamma t}-\frac{\frac{d}{d\beta}\breve{h}(\beta,\gamma)}{1+ic'(\beta)\gamma t})d\beta}_{FM_{1,1}}\\
&+\underbrace{\int_{0}^{2\alpha}\breve{K}(\ddot{h}(\alpha,\gamma)-\ddot{h}(\beta,\gamma))(A(\frac{(\frac{d}{d\alpha}\breve{h})(\alpha,\gamma)}{1+ic'(\alpha)\gamma t})(\alpha,\gamma)-A(\frac{(\frac{d}{d\alpha}\breve{h})(\alpha,\gamma)}{1+ic'(\alpha)\gamma t})(\beta,\gamma))(1+ic'(\beta)\gamma t)d\beta }_{FM_{1,2}}\\
    &\underbrace{-\breve{K}(\ddot{h}(\alpha,\gamma)-\ddot{h}(0,\gamma))(\frac{ic(\alpha)t}{1+ic'(\alpha)\gamma t}-\frac{ic(0)t}{1+ic'(0)\gamma t})(1+ic'(0)\gamma t)(\frac{(\frac{d}{d\alpha}\breve{h})(\alpha,\gamma)}{1+ic'(\alpha)\gamma t}-\frac{\frac{d}{d\alpha}\breve{h}(0,\gamma)}{1+ic'(0)\gamma t})}_{B_2}\\
&+\underbrace{\breve{K}(\ddot{h}(\alpha,\gamma)-\ddot{h}(2\alpha,\gamma))(\frac{ic(\alpha)t}{1+ic'(\alpha)\gamma t}-\frac{ic(2\alpha)t}{1+ic'(2\alpha)\gamma t})(\frac{(\frac{d}{d\alpha}\breve{h})(\alpha,\gamma)}{1+ic'(\alpha)\gamma t}-\frac{(\frac{d}{d\alpha}\breve{h})(2\alpha,\gamma)}{1+ic'(2\alpha)\gamma t})(1+ic'(2\alpha)\gamma t)}_{B_3}\\
&=FM_{1,1}+FM_{1,2}\\
&\underbrace{+\frac{ic(\alpha)t}{1+ic'(\alpha)\gamma t}[\breve{K}(\ddot{h}(\alpha,\gamma)-\ddot{h}(0,\gamma))(\frac{(\frac{d}{d\alpha}\breve{h})(\alpha,\gamma)}{1+ic'(\alpha)\gamma t})(1+ic'(0)\gamma t)}_{B_{1,1}}\\&\underbrace{+\breve{K}(\ddot{h}(\alpha,\gamma)-\ddot{h}(2\alpha,\gamma))(\frac{(\frac{d}{d\alpha}\breve{h})(\alpha,\gamma)}{1+ic'(\alpha)\gamma t})(1+ic'(2\alpha)\gamma t)]}_{B_{1,1}}\\
&\underbrace{+\frac{ic(\alpha)t}{1+ic'(\alpha)\gamma t}[\breve{K}(\ddot{h}(\alpha,\gamma)-\ddot{h}(0,\gamma))(-\frac{\frac{d}{d\alpha}\breve{h}(0,\gamma)}{1+ic'(0)\gamma t})(1+ic'(0)\gamma t)}_{B_{1,2}}\\&\underbrace{+\breve{K}(\ddot{h}(\alpha,\gamma)-\ddot{h}(2\alpha,\gamma))(-\frac{(\frac{d}{d\alpha}\breve{h})(2\alpha,\gamma)}{1+ic'(2\alpha)\gamma t})(1+ic'(2\alpha)\gamma t)]}_{B_{1,2}}\\
 &\underbrace{-\breve{K}(\ddot{h}(\alpha,\gamma)-\ddot{h}(0,\gamma))(\frac{ic(\alpha)t}{1+ic'(\alpha)\gamma t}-\frac{ic(0)t}{1+ic'(0)\gamma t})(1+ic'(0)\gamma t)(\frac{(\frac{d}{d\alpha}\breve{h})(\alpha,\gamma)}{1+ic'(\alpha)\gamma t})}_{B_{2,1}}\\
  &\underbrace{-\breve{K}(\ddot{h}(\alpha,\gamma)-\ddot{h}(0,\gamma))(\frac{ic(\alpha)t}{1+ic'(\alpha)\gamma t}-\frac{ic(0)t}{1+ic'(0)\gamma t})(1+ic'(0)\gamma t)(-\frac{\frac{d}{d\alpha}\breve{h}(0,\gamma)}{1+ic'(0)\gamma t})}_{B_{2,2}}\\
&+\underbrace{\breve{K}(\ddot{h}(\alpha,\gamma)-\ddot{h}(2\alpha,\gamma))(\frac{ic(\alpha)t}{1+ic'(\alpha)\gamma t}-\frac{ic(2\alpha)t}{1+ic'(2\alpha)\gamma t})(\frac{(\frac{d}{d\alpha}\breve{h})(\alpha,\gamma)}{1+ic'(\alpha)\gamma t})(1+ic'(2\alpha)\gamma t)}_{B_{3,1}}\\
&+\underbrace{\breve{K}(\ddot{h}(\alpha,\gamma)-\ddot{h}(2\alpha,\gamma))(\frac{ic(\alpha)t}{1+ic'(\alpha)\gamma t}-\frac{ic(2\alpha)t}{1+ic'(2\alpha)\gamma t})(-\frac{(\frac{d}{d\alpha}\breve{h})(2\alpha,\gamma)}{1+ic'(2\alpha)\gamma t})(1+ic'(2\alpha)\gamma t)}_{B_{3,2}}.
\end{align*}
Here we use lemma \ref{switch} in the $FM_{1,2}$ term.
%\begin{align*}
%(\frac{ic(\alpha)t}{1+ic'(\alpha)\gamma %t}\partial_{\alpha}-\partial_{\gamma})(\frac{(\frac{d}{d%\alpha})\breve{h}(\alpha,\gamma)}{1+ic'(\alpha)\eta %t})=\frac{(\frac{d}{d\alpha})}{1+ic'(\alpha)\gamma %t}(\frac{ic(\alpha)t}{1+ic'(\alpha)\gamma %t}\partial_{\alpha}-\partial_{\gamma})\breve{h}(\alpha,\gamma).
%\end{align*}
For $M_2$, we have
\begin{align*}
   &A(M_2)=A(\breve{K}(\ddot{h}(\alpha,\gamma)-\ddot{h}(2\alpha,\gamma))\breve{h}(2\alpha,\gamma))\\
   =&\breve{K}(\ddot{h}(\alpha,\gamma)-\ddot{h}(2\alpha,\gamma))(\frac{2ic(\alpha)t}{1+ic'(\alpha)\gamma t}(\partial_{\alpha}\breve{h})(2\alpha,\gamma)-(\partial_{\gamma}\breve{h})(2\alpha,\gamma))\\
   &+\nabla \breve{K}(\ddot{h}(\alpha,\gamma)-\ddot{h}(2\alpha,\gamma))\cdot(\frac{ic(\alpha)t}{1+ic'(\alpha)\gamma t}(\partial_{\alpha}\ddot{h})(\alpha,\gamma)-(\partial_{\gamma}\ddot{h})(\alpha,\gamma)\\
   &\qquad-(\frac{2ic(\alpha)t}{1+ic'(\alpha)\gamma t}(\partial_{\alpha}\ddot{h})(2\alpha,\gamma)-(\partial_{\gamma}\ddot{h})(2\alpha,\gamma)))\breve{h}(2\alpha,\gamma).
   \end{align*}
   Then we can write each term into $A(\breve{h})$ and the remaining terms. We get 
   \begin{align*}
   &A(M_2)=\underbrace{\breve{K}(\ddot{h}(\alpha,\gamma)-\ddot{h}(2\alpha,\gamma))(\frac{2ic(\alpha)t}{1+ic'(\alpha)\gamma t}-\frac{ic(2\alpha)t}{1+ic'(2\alpha)\gamma t})(\partial_{\alpha}\breve{h})(2\alpha,\gamma)}_{B_4}\\
      &\underbrace{+\breve{K}(\ddot{h}(\alpha,\gamma)-\ddot{h}(2\alpha,\gamma))(\frac{ic(2\alpha)t}{1+ic'(2\alpha)\gamma t}(\partial_{\alpha}\breve{h})(2\alpha,\gamma)-(\partial_{\gamma}\breve{h})(2\alpha,\gamma))}_{FM_{2,2}}\\
   &\underbrace{- \nabla \breve{K}(\ddot{h}(\alpha,\gamma)-\ddot{h}(2\alpha,\gamma))\cdot[(\frac{2ic(\alpha)t}{1+ic'(\alpha)\gamma t}-\frac{ic(2\alpha)t}{1+ic'(2\alpha)\gamma t})(\partial_{\alpha}\ddot{h})(2\alpha,\gamma)]\breve{h}(2\alpha,\gamma)}_{B_5}\\
   &\underbrace{+ \nabla \breve{K}(\ddot{h}(\alpha,\gamma)-\ddot{h}(2\alpha,\gamma))\cdot[(\frac{ic(\alpha)t}{1+ic'(\alpha)\gamma t})(\partial_{\alpha}\ddot{h})(\alpha,\gamma)-(\partial_{\gamma}\ddot{h})(\alpha,\gamma)-((\frac{ic(2\alpha)t}{1+ic'(2\alpha)\gamma t})(\partial_{\alpha}\ddot{h})(2\alpha,\gamma)-(\partial_{\gamma}\ddot{h})(2\alpha,\gamma))]}_{FM_{2,1}}\\
   &\underbrace{\breve{h}(2\alpha,\gamma)}_{FM_{2,1}}.
   \end{align*}
    For the third term, we have
\begin{align*}
    &A(M_3)=A(\int_{0}^{\gamma}\breve{h}(2\alpha,\eta) \nabla \breve{K}(\ddot{h}(\alpha,\gamma)-\ddot{h}(2\alpha,\eta))\cdot(\frac{(\frac{d}{d\alpha}\ddot{h})(2\alpha,\eta)}{1+ic'(2\alpha)\eta t})ic(2\alpha)t d\eta)\\
    &=\underbrace{-\breve{h}(2\alpha,\gamma)\nabla \breve{K}(\ddot{h}(\alpha,\gamma)-\ddot{h}(2\alpha,\gamma))\cdot(\frac{(\frac{d}{d\alpha}\ddot{h})(2\alpha,\gamma)}{1+ic'(2\alpha)\gamma t})ic(2\alpha)t}_{Term_1}\\
    &\underbrace{+\int_{0}^{\gamma}(\frac{2ic(\alpha)t}{1+ic'(\alpha)\gamma t}(\partial_{\alpha}h)(2\alpha,\eta))  \nabla \breve{K}(\ddot{h}(\alpha,\gamma)-\ddot{h}(2\alpha,\eta))\cdot(\frac{(\frac{d}{d\alpha}\ddot{h})(2\alpha,\eta)}{1+ic'(2\alpha)\eta t})ic(2\alpha)t d\eta}_{Term_2}\\
    &\underbrace{+\int_{0}^{\gamma}\breve{h}(2\alpha,\eta) \nabla^2\breve{K}(\ddot{h}(\alpha,\gamma)-\ddot{h}(2\alpha,\eta))[(\frac{ic(\alpha)t}{1+ic'(\alpha)\gamma t}(\partial_{\alpha}\ddot{h})(\alpha,\gamma)-(\partial_{\gamma}\ddot{h})(\alpha,\gamma)),(\frac{(\frac{d}{d\alpha}\ddot{h})(2\alpha,\eta)}{1+ic'(2\alpha)\eta t})]ic(2\alpha)t d\eta}_{Term_3, FM_{3,1}}\\
    &\underbrace{-\int_{0}^{\gamma}\breve{h}(2\alpha,\eta) \nabla^2 \breve{K}(\ddot{h}(\alpha,\gamma)-\ddot{h}(2\alpha,\eta))[(\frac{2ic(\alpha)t}{1+ic'(\alpha)\gamma t}(\partial_{\alpha}\ddot{h})(2\alpha,\eta)),\frac{(\frac{d}{d\alpha}\ddot{h})(2\alpha,\eta)}{1+ic'(2\alpha)\eta t}]ic(2\alpha)t d\eta}_{Term_4}\\
    &\underbrace{+\int_{0}^{\gamma}\breve{h}(2\alpha,\eta) \nabla \breve{K}(\ddot{h}(\alpha,\gamma)-\ddot{h}(2\alpha,\eta))\cdot [(\frac{ic(\alpha)t}{1+ic'(\alpha)\gamma t}\frac{d}{d\alpha})(\frac{(\frac{d}{d\alpha}\ddot{h})(2\alpha,\eta)}{1+ic'(2\alpha)\eta t})]ic(2\alpha)t d\eta}_{Term_5}\\
    &\underbrace{+\int_{0}^{\gamma}\breve{h}(2\alpha,\eta) \nabla \breve{K}(\ddot{h}(\alpha,\gamma)-\ddot{h}(2\alpha,\eta))\cdot(\frac{(\frac{d}{d\alpha}\ddot{h})(2\alpha,\eta)}{1+ic'(2\alpha)\eta t})\frac{i2c(\alpha)t}{1+ic'(\alpha)\gamma t}ic'(2\alpha)t d\eta}_{Term_6}\\
    &=Term_1+Term_2+FM_{3,1}+Term_4+Term_5+Term_6.
    \end{align*}
    Now we do the change to the terms by writing the terms in $Term_2+Term_4+Term_5$ as $A(\cdot)$.
    We have
    \begin{align*}
    &Term_2+Term_4+Term_5\\
    =&\underbrace{\int_{0}^{\gamma}ic(2\alpha)t (\frac{2ic(\alpha)t}{1+ic'(\alpha)\gamma t})[(\frac{1}{2}\frac{d}{d\alpha}-\frac{1+ic'(2\alpha)\eta t}{ic(2\alpha)t}\frac{d}{d\eta})\breve{h}(2\alpha,\eta)] \nabla \breve{K}(\ddot{h}(\alpha,\gamma)-\ddot{h}(2\alpha,\eta))\cdot(\frac{(\frac{d}{d\alpha}\ddot{h})(2\alpha,\eta)}{1+ic'(2\alpha)\eta t})d\eta}_{Term_{2,1}=FM_{3,4}}\\
  &\underbrace{+\int_{0}^{\gamma}ic(2\alpha)t (\frac{2ic(\alpha)t}{1+ic'(\alpha)\gamma t})[(\frac{1+ic'(2\alpha)\eta t}{ic(2\alpha)t}\frac{d}{d\eta})\breve{h}(2\alpha,\eta)]  \nabla \breve{K}(\ddot{h}(\alpha,\gamma)-\ddot{h}(2\alpha,\eta))\cdot(\frac{(\frac{d}{d\alpha}\ddot{h})(2\alpha,\eta)}{1+ic'(2\alpha)\eta t})d\eta}_{Term_{2,2}}\\
   &\underbrace{-\int_{0}^{\gamma}\breve{h}(2\alpha,\eta) \nabla^2 \breve{K}(\ddot{h}(\alpha,\gamma)-\ddot{h}(2\alpha,\eta))[(\frac{2ic(\alpha)t}{1+ic'(\alpha)\gamma t})(\frac{1}{2}\frac{d}{d\alpha}-\frac{1+ic'(2\alpha)\eta t}{ic(2\alpha)t}\frac{d}{d\eta})\ddot{h}(2\alpha,\eta),(\frac{(\frac{d}{d\alpha}\ddot{h})(2\alpha,\eta)}{1+ic'(2\alpha)\eta t})]}_{Term_{4,1}=FM_{3,2}}\\
   &\underbrace{ic(2\alpha)t d\eta}_{Term_{4,1}=FM_{3,2}}\\
      &\underbrace{-\int_{0}^{\gamma}\breve{h}(2\alpha,\eta) \nabla^2 \breve{K}(\ddot{h}(\alpha,\gamma)-\ddot{h}(2\alpha,\eta))[(\frac{2ic(\alpha)t}{1+ic'(\alpha)\gamma t})(\frac{1+ic'(2\alpha)\eta t}{ic(2\alpha)t}\frac{d}{d\eta})\ddot{h}(2\alpha,\eta),(\frac{(\frac{d}{d\alpha}\ddot{h})(2\alpha,\eta)}{1+ic'(2\alpha)\eta t})]ic(2\alpha)t d\eta}_{Term_{4,2}}\\
   &\underbrace{+\int_{0}^{\gamma}\breve{h}(2\alpha,\eta) \nabla \breve{K}(\ddot{h}(\alpha,\gamma)-\ddot{h}(2\alpha,\eta))\cdot(\frac{2ic(\alpha)t}{1+ic'(\alpha)\gamma t})[(\frac{1}{2}\frac{d}{d\alpha}-\frac{1+ic'(2\alpha)\eta t}{ic(2\alpha)t}\frac{d}{d\eta})(\frac{(\frac{d}{d\alpha}\ddot{h})(2\alpha,\eta)}{1+ic'(2\alpha)\eta t})]ic(2\alpha)t d\eta}_{Term_{5,1}=FM_{3,3}}\\
      &\underbrace{+\int_{0}^{\gamma}\breve{h}(2\alpha,\eta) \nabla \breve{K}(\ddot{h}(\alpha,\gamma)-\ddot{h}(2\alpha,\eta))\cdot(\frac{2ic(\alpha)t}{1+ic'(\alpha)\gamma t})[(\frac{1+ic'(2\alpha)\eta t}{ic(2\alpha)t}\frac{d}{d\eta})(\frac{(\frac{d}{d\alpha}\ddot{h})(2\alpha,\eta)}{1+ic'(2\alpha)\eta t})]ic(2\alpha)t d\eta}_{Term_{5,2}}.
   \end{align*}
   Here when showing $Term_{5,1}=FM_{3,3}$, we used 
   \begin{align*}
(\frac{\frac{1}{2}ic(2\alpha)t}{1+ic'(2\alpha)\eta t}\frac{d}{d\alpha}-\frac{d}{d\g})(\frac{(\partial_{\alpha}\ddot{h})(2\alpha,\eta)}{1+ic'(2\alpha)\eta t})=A(\frac{(\partial_{\alpha})\ddot{h}(\alpha,\eta)}{1+ic'(\alpha)\eta t})(2\alpha,\eta)=\frac{(\partial_{\alpha}A(\ddot{h}))(2\alpha,\eta)}{1+ic'(2\alpha)\eta t},
\end{align*}
which follows from lemma  \ref{switch}:
\begin{align*}
A\circ \frac{\partial_{\alpha}\ddot{h}(\alpha,\eta)}{1+ic'(\alpha)\eta t}=\frac{(\partial_{\alpha})}{1+ic'(\alpha)\eta t}\circ A (\ddot{h})(\alpha,\eta).
\end{align*}
   We also have
\begin{align*}
    &Term_{2,2}+Term_{4,2}+Term_{5,2}+Term_6\\
    &=\underbrace{\int_{0}^{\gamma}ic(2\alpha)t (\frac{2ic(\alpha)t}{1+ic'(\alpha)\gamma t})[(\frac{1+ic'(2\alpha)\eta t}{ic(2\alpha)t}\frac{d}{d\eta})\breve{h}(2\alpha,\eta)]  \nabla \breve{K}(\ddot{h}(\alpha,\gamma)-\ddot{h}(2\alpha,\eta))\cdot(\frac{(\frac{d}{d\alpha}\ddot{h})(2\alpha,\eta)}{1+ic'(2\alpha)\eta t})d\eta}_{Term_{2,2}}\\
     &\underbrace{-\int_{0}^{\gamma}\breve{h}(2\alpha,\eta) \nabla^2 \breve{K}(\ddot{h}(\alpha,\gamma)-\ddot{h}(2\alpha,\eta))[(\frac{2ic(\alpha)t}{1+ic'(\alpha)\gamma t})(\frac{1+ic'(2\alpha)\eta t}{ic(2\alpha)t}\frac{d}{d\eta})\ddot{h}(2\alpha,\eta),(\frac{(\frac{d}{d\alpha}\ddot{h})(2\alpha,\eta)}{1+ic'(2\alpha)\eta t})]ic(2\alpha)t d\eta}_{Term_{4,2}}\\
           &\underbrace{+\int_{0}^{\gamma}\breve{h}(2\alpha,\eta) \nabla \breve{K}(\ddot{h}(\alpha,\gamma)-\ddot{h}(2\alpha,\eta))\cdot(\frac{2ic(\alpha)t}{1+ic'(\alpha)\gamma t})[(\frac{1+ic'(2\alpha)\eta t}{ic(2\alpha)t}\frac{d}{d\eta})(\frac{(\frac{d}{d\alpha}\ddot{h})(2\alpha,\eta)}{1+ic'(2\alpha)\eta t})]ic(2\alpha)t d\eta}_{Term_{5,2}}\\
           &\underbrace{+\int_{0}^{\gamma}\breve{h}(2\alpha,\eta) \nabla \breve{K}(\ddot{h}(\alpha,\gamma)-\ddot{h}(2\alpha,\eta))\cdot(\frac{(\frac{d}{d\alpha}\ddot{h})(2\alpha,\eta)}{1+ic'(2\alpha)\eta t})\frac{i2c(\alpha)t}{1+ic'(\alpha)\gamma t}ic'(2\alpha)t d\eta}_{Term_6}\\
           &=\int_{0}^{\gamma}\frac{d}{d\eta}[\breve{h}(2\alpha,\eta)\nabla \breve{K}(\ddot{h}(\alpha,\gamma)-\ddot{h}(2\alpha,\eta))\cdot(\frac{(\frac{d}{d\alpha}\ddot{h})(2\alpha,\eta)}{1+ic'(2\alpha)\eta t})\frac{i2c(\alpha)t}{1+ic'(\alpha)\gamma t}(ic'(2\alpha)\eta t+1)] d\eta\\
           &=(\breve{h}(2\alpha,\gamma)\nabla \breve{K}(\ddot{h}(\alpha,\gamma)-\ddot{h}(2\alpha,\gamma))\cdot(\frac{(\frac{d}{d\alpha}\ddot{h})(2\alpha,\gamma)}{1+ic'(2\alpha)\gamma t})\frac{i2c(\alpha)t}{1+ic'(\alpha)\gamma t}(ic'(2\alpha)\gamma t+1)) \\
    &-(\breve{h}(2\alpha,0)\nabla \breve{K}(\ddot{h}(\alpha,\gamma)-\ddot{h}(2\alpha,0))\cdot(\frac{d}{d\alpha}\ddot{h})(2\alpha,0)\frac{i2c(\alpha)t}{1+ic'(\alpha)\gamma t}.
\end{align*}
Finally, we have
  
   \begin{align*}
   A(M_3)&=Term_1+Term_{2,1}+FM_{3,1}+Term_{4,1}+Term_{5,1}+(Term_{2,2}+Term_{4,2}+Term_{5,2}+Term_6)\\
   &=Term_1+FM_{3,4}+FM_{3,1}+FM_{3,2}+FM_{3,3}+(Term_{2,2}+Term_{4,2}+Term_{5,2}+Term_6)\\
   &=\underbrace{-\breve{h}(2\alpha,\gamma)\nabla \breve{K}(\ddot{h}(\alpha,\gamma)-\ddot{h}(2\alpha,\gamma))\cdot(\frac{(\frac{d}{d\alpha}\ddot{h})(2\alpha,\gamma)}{1+ic'(2\alpha)\gamma t})ic(2\alpha)t}_{B_6}\\
   &+FM_{3,4}+FM_{3,1}+FM_{3,2}+FM_{3,3}\\
   &\underbrace{+\breve{h}(2\alpha,\gamma)\nabla \breve{K}(\ddot{h}(\alpha,\gamma)-\ddot{h}(2\alpha,\gamma))\cdot(\frac{(\frac{d}{d\alpha}\ddot{h})(2\alpha,\gamma)}{1+ic'(2\alpha)\gamma t})\frac{i2c(\alpha)t}{1+ic'(\alpha)\gamma t}(ic'(2\alpha)\gamma t+1)}_{B_7} \\
    &\underbrace{-\breve{h}(2\alpha,0)\nabla \breve{K}(\ddot{h}(\alpha,\gamma)-\ddot{h}(2\alpha,0))\cdot(\frac{d}{d\alpha}\ddot{h})(2\alpha,0)\frac{i2c(\alpha)t}{1+ic'(\alpha)\gamma t}}_{B_8}.
\end{align*}
Then we consider $M_4$
\begin{align*}
    &A(M_{4})=A(-\int_{2\alpha}^{\pi}\breve{h}(\beta,0)\nabla \breve{K}(\ddot{h}(\alpha,\gamma)-\ddot{h}(\beta,0))\cdot\frac{d\ddot{h}}{d\beta}(\beta,0)d\beta)\\
    &=\underbrace{\frac{2ic(\alpha)t}{1+ic'(\alpha)\gamma t}\breve{h}(2\alpha,0) \nabla \breve{K}(\ddot{h}(\alpha,\gamma)-\ddot{h}(2\alpha,0))\cdot(\frac{d\ddot{h}}{d\alpha})(2\alpha,0)}_{B_9}\\
    &-\underbrace{\int_{2\alpha}^{\pi}\breve{h}(\beta,0) \nabla^2 \breve{K}(\ddot{h}(\alpha,\gamma)-\ddot{h}(\beta,0))[(\frac{ic(\alpha)t}{1+ic'(\alpha)\gamma t}\partial_{\alpha}-\partial_{\gamma})\ddot{h}(\alpha,\gamma),\frac{d\ddot{h}}{d\beta}(\beta,0)]d\beta}_{FM_{4}}.
\end{align*}
For $M_5$, we still compare it with $O_4$ in \eqref{Oequationapp}. We only change $\breve{h}(\alpha,\gamma)$ to $\frac{(\frac{d}{d\alpha}\breve{h})(\alpha,\gamma)}{1+ic'(\alpha)\gamma t}$ and set $X(\beta,\gamma)=1$. Recall that we have
\begin{align*}
    &A(O_4)=\\
    &\int_{[-\pi,0]\cup[2\alpha,\pi] } \breve{K}(\ddot{h}(\alpha,\gamma)-\ddot{h}(\beta,\gamma))(X(\beta,\gamma))(1+ic'(\beta)\gamma t)d\beta A(\breve{h})(\alpha,\gamma)\\
    &+\int_{[\alpha-\pi,-\alpha]\cup[\alpha,\alpha+\pi] }X(\alpha-\beta,\gamma)\nabla \breve{K}(\ddot{h}(\alpha,\gamma)-\ddot{h}(\alpha-\beta,\gamma))\\
&\cdot((\frac{ic(\alpha)t}{1+ic'(\alpha)\gamma t}\partial_{\alpha}-\partial_{\gamma})\ddot{h}(\alpha,\gamma)-(\frac{ic(\alpha-\beta)t}{1+ic'(\alpha-\beta)\gamma t}\partial_{\alpha}-\partial_{\gamma})\ddot{h}(\alpha-\beta,\gamma,t))
(1+ic'(\alpha-\beta)\gamma t)d\beta \breve{h}(\alpha,\gamma)\\
&+\int_{[\alpha-\pi,-\alpha]\cup[\alpha,\alpha+\pi] }\breve{K}(\ddot{h}(\alpha,\gamma)-\ddot{h}(\alpha-\beta,\gamma))(\frac{ic(\alpha-\beta)t}{1+ic'(\alpha-\beta)\gamma t}\partial_{\alpha}-\partial_{\gamma})X(\alpha-\beta,\gamma)(1+ic'(\alpha-\beta)\gamma t)d\beta \breve{h}(\alpha,\gamma)\\
&+\frac{ic(\alpha)t}{1+ic'(\alpha)\gamma t}\breve{K}(\ddot{h}(\alpha,\gamma)-\ddot{h}(-\pi,\gamma))X(-\pi,\gamma)(1+ic'(-\pi)\gamma t)\breve{h}(\alpha,\gamma)\\
      &-\frac{ic(\alpha)t}{1+ic'(\alpha)\gamma t}\breve{K}(\ddot{h}(\alpha,\gamma)-\ddot{h}(0,\gamma))X(0,\gamma)(1+ic'(0)\gamma t)\breve{h}(\alpha,\gamma)\\  &-\frac{ic(\alpha)t}{1+ic'(\alpha)\gamma t}\breve{K}(\ddot{h}(\alpha,\gamma)-\ddot{h}(2\alpha,\gamma))X(2\alpha,\gamma)(1+ic'(2\alpha)\gamma t)\breve{h}(\alpha,\gamma)\\
     &-\frac{ic(\alpha)t}{1+ic'(\alpha)\gamma t}\breve{K}(\ddot{h}(\alpha,\gamma)-\ddot{h}(\pi,\gamma))X(\pi,\gamma)(1+ic'(\pi)\gamma t)\breve{h}(\alpha,\gamma)\\
     &+\breve{K}(\ddot{h}(\alpha,\gamma)-\ddot{h}(-\pi,\gamma))X(-\pi,\gamma)(ic(-\pi)t-\frac{ic(\alpha)t}{1+ic'(\alpha)\gamma t}(1+ic'(-\pi)\gamma t)) \breve{h}(\alpha,\gamma)\\
    &-\breve{K}(\ddot{h}(\alpha,\gamma)-\ddot{h}(0,\gamma))X(0,\gamma)(ic(0)t-\frac{ic(\alpha)t}{1+ic'(\alpha)\gamma t}(1+ic'(0)\gamma t)) \breve{h}(\alpha,\gamma)\\
    &+\breve{K}(\ddot{h}(\alpha,\gamma)-\ddot{h}(2\alpha,\gamma))X(2\alpha,\gamma)(ic(2\alpha)t-\frac{ic(\alpha)t}{1+ic'(\alpha)\gamma t}(1+ic'(2\alpha)\gamma t)) \breve{h}(\alpha,\gamma)\\
    &-\breve{K}(\ddot{h}(\alpha,\gamma)-\ddot{h}(\pi,\gamma))X(\pi,\gamma)(ic(\pi)t-\frac{ic(\alpha)t}{1+ic'(\alpha)\gamma t}(1+ic'(\pi)\gamma t))\breve{h}(\alpha,\gamma).
\end{align*}
Then, we have
\begin{align*}
    &A(M_5)  =\\
    &\underbrace{\int_{[-\pi,0]\cup[2\alpha,\pi] } \breve{K}(\ddot{h}(\alpha,\gamma)-\ddot{h}(\beta,\gamma))(1+ic'(\beta)\gamma t)d\beta \frac{\frac{d}{d\alpha}A(\breve{h})(\alpha,\gamma)}{1+ic'(\alpha)\gamma t}}_{FM_{5,2}}\\
    &\underbrace{+\int_{[\alpha-\pi,-\alpha]\cup[\alpha,\alpha+\pi] }\nabla \breve{K}(\ddot{h}(\alpha,\gamma)-\ddot{h}(\alpha-\beta,\gamma))}_{FM_{5,1}}\\
&\underbrace{\cdot((\frac{ic(\alpha)t}{1+ic'(\alpha)\gamma t}\partial_{\alpha}-\partial_{\gamma})\ddot{h}(\alpha,\gamma)-(\frac{ic(\alpha-\beta)t}{1+ic'(\alpha-\beta)\gamma t}\partial_{\alpha}-\partial_{\gamma})\ddot{h}(\alpha-\beta,\gamma,t))
(1+ic'(\alpha-\beta)\gamma t)d\beta \frac{(\frac{d}{d\alpha}\breve{h})(\alpha,\gamma)}{1+ic'(\alpha)\gamma t}}_{FM_{5,1}}\\
&\underbrace{+\frac{ic(\alpha)t}{1+ic'(\alpha)\gamma t}\breve{K}(\ddot{h}(\alpha,\gamma)-\ddot{h}(-\pi,\gamma))(1+ic'(-\pi)\gamma t)\frac{(\frac{d}{d\alpha}\breve{h})(\alpha,\gamma)}{1+ic'(\alpha)\gamma t}}_{B_{10}}\\
      &\underbrace{-\frac{ic(\alpha)t}{1+ic'(\alpha)\gamma t}\breve{K}(\ddot{h}(\alpha,\gamma)-\ddot{h}(0,\gamma))(1+ic'(0)\gamma t)\frac{(\frac{d}{d\alpha}\breve{h})(\alpha,\gamma)}{1+ic'(\alpha)\gamma t}}_{B_{11}}\\  &\underbrace{-\frac{ic(\alpha)t}{1+ic'(\alpha)\gamma t}\breve{K}(\ddot{h}(\alpha,\gamma)-\ddot{h}(2\alpha,\gamma))(1+ic'(2\alpha)\gamma t)\frac{(\frac{d}{d\alpha}\breve{h})(\alpha,\gamma)}{1+ic'(\alpha)\gamma t}}_{B_{12}}\\
     &\underbrace{-\frac{ic(\alpha)t}{1+ic'(\alpha)\gamma t}\breve{K}(\ddot{h}(\alpha,\gamma)-\ddot{h}(\pi,\gamma))(1+ic'(\pi)\gamma t)\frac{(\frac{d}{d\alpha}\breve{h})(\alpha,\gamma)}{1+ic'(\alpha)\gamma t}}_{B_{13}}\\
     &\underbrace{+\breve{K}(\ddot{h}(\alpha,\gamma)-\ddot{h}(-\pi,\gamma))(ic(-\pi)t-\frac{ic(\alpha)t}{1+ic'(\alpha)\gamma t}(1+ic'(-\pi)\gamma t)) \frac{(\frac{d}{d\alpha}\breve{h})(\alpha,\gamma)}{1+ic'(\alpha)\gamma t}}_{B_{14}}\\
    &\underbrace{-\breve{K}(\ddot{h}(\alpha,\gamma)-\ddot{h}(0,\gamma))(ic(0)t-\frac{ic(\alpha)t}{1+ic'(\alpha)\gamma t}(1+ic'(0)\gamma t))\frac{(\frac{d}{d\alpha}\breve{h})(\alpha,\gamma)}{1+ic'(\alpha)\gamma t}}_{B_{15}}\\
    &\underbrace{+\breve{K}(\ddot{h}(\alpha,\gamma)-\ddot{h}(2\alpha,\gamma))(ic(2\alpha)t-\frac{ic(\alpha)t}{1+ic'(\alpha)\gamma t}(1+ic'(2\alpha)\gamma t)) \frac{(\frac{d}{d\alpha}\breve{h})(\alpha,\gamma)}{1+ic'(\alpha)\gamma t}}_{B_{16}}\\
    &\underbrace{-\breve{K}(\ddot{h}(\alpha,\gamma)-\ddot{h}(\pi,\gamma))(ic(\pi)t-\frac{ic(\alpha)t}{1+ic'(\alpha)\gamma t}(1+ic'(\pi)\gamma t))\frac{(\frac{d}{d\alpha}\breve{h})(\alpha,\gamma)}{1+ic'(\alpha)\gamma t}}_{B_{17}}.
\end{align*}
In conclusion, we only need to show the sum of $B$ terms are 0. We have 
\begin{align*}
   &\sum_{i=1}^{17}B_i\\
   &=\underbrace{\frac{ic(\alpha)t}{1+ic'(\alpha)\gamma t}\breve{K}(\ddot{h}(\alpha,\gamma)-\ddot{h}(0,\gamma))(\frac{(\frac{d}{d\alpha}\breve{h})(\alpha,\gamma)}{1+ic'(\alpha)\gamma t})(1+ic'(0)\gamma t)}_{B_{1,1}}\\&\underbrace{+\frac{ic(\alpha)t}{1+ic'(\alpha)\gamma t}\breve{K}(\ddot{h}(\alpha,\gamma)-\ddot{h}(2\alpha,\gamma))(\frac{(\frac{d}{d\alpha}\breve{h})(\alpha,\gamma)}{1+ic'(\alpha)\gamma t})(1+ic'(2\alpha)\gamma t)}_{B_{1,1}}\\
&\underbrace{+\frac{ic(\alpha)t}{1+ic'(\alpha)\gamma t}\breve{K}(\ddot{h}(\alpha,\gamma)-\ddot{h}(0,\gamma))(-\frac{\frac{d}{d\alpha}\breve{h}(0,\gamma)}{1+ic'(0)\gamma t})(1+ic'(0)\gamma t)}_{B_{1,2}}\\&\underbrace{+\frac{ic(\alpha)t}{1+ic'(\alpha)\gamma t}\breve{K}(\ddot{h}(\alpha,\gamma)-\ddot{h}(2\alpha,\gamma))(-\frac{(\frac{d}{d\alpha}\breve{h})(2\alpha,\gamma)}{1+ic'(2\alpha)\gamma t})(1+ic'(2\alpha)\gamma t)}_{B_{1,2}}\\
 &\underbrace{-\breve{K}(\ddot{h}(\alpha,\gamma)-\ddot{h}(0,\gamma))(\frac{ic(\alpha)t}{1+ic'(\alpha)\gamma t}-\frac{ic(0)t}{1+ic'(0)\gamma t})(1+ic'(0)\gamma t)(\frac{(\frac{d}{d\alpha}\breve{h})(\alpha,\gamma)}{1+ic'(\alpha)\gamma t})}_{B_{2,1}}\\
  &\underbrace{-\breve{K}(\ddot{h}(\alpha,\gamma)-\ddot{h}(0,\gamma))(\frac{ic(\alpha)t}{1+ic'(\alpha)\gamma t}-\frac{ic(0)t}{1+ic'(0)\gamma t})(1+ic'(0)\gamma t)(-\frac{\frac{d}{d\alpha}\breve{h}(0,\gamma)}{1+ic'(0)\gamma t})}_{B_{2,2}}\\
&+\underbrace{\breve{K}(\ddot{h}(\alpha,\gamma)-\ddot{h}(2\alpha,\gamma))(\frac{ic(\alpha)t}{1+ic'(\alpha)\gamma t}-\frac{ic(2\alpha)t}{1+ic'(2\alpha)\gamma t})(\frac{(\frac{d}{d\alpha}\breve{h})(\alpha,\gamma)}{1+ic'(\alpha)\gamma t})(1+ic'(2\alpha)\gamma t)}_{B_{3,1}}\\
&+\underbrace{\breve{K}(\ddot{h}(\alpha,\gamma)-\ddot{h}(2\alpha,\gamma))(\frac{ic(\alpha)t}{1+ic'(\alpha)\gamma t}-\frac{ic(2\alpha)t}{1+ic'(2\alpha)\gamma t})(-\frac{(\frac{d}{d\alpha}\breve{h})(2\alpha,\gamma)}{1+ic'(2\alpha)\gamma t})(1+ic'(2\alpha)\gamma t)}_{B_{3,2}}\\
&\underbrace{+\breve{K}(\ddot{h}(\alpha,\gamma)-\ddot{h}(2\alpha,\gamma))(\frac{2ic(\alpha)t}{1+ic'(\alpha)\gamma t}-\frac{ic(2\alpha)t}{1+ic'(2\alpha)\gamma t})(\partial_{\alpha}\breve{h})(2\alpha,\gamma)}_{B_4}\\
   &\underbrace{- \nabla \breve{K}(\ddot{h}(\alpha,\gamma)-\ddot{h}(2\alpha,\gamma))\cdot[(\frac{2ic(\alpha)t}{1+ic'(\alpha)\gamma t}-\frac{ic(2\alpha)t}{1+ic'(2\alpha)\gamma t})(\partial_{\alpha}\ddot{h})(2\alpha,\gamma)]\breve{h}(2\alpha,\gamma)}_{B_5}\\
    &\underbrace{-\breve{h}(2\alpha,\gamma)\nabla \breve{K}(\ddot{h}(\alpha,\gamma)-\ddot{h}(2\alpha,\gamma))\cdot(\frac{(\frac{d}{d\alpha}\ddot{h})(2\alpha,\gamma)}{1+ic'(2\alpha)\gamma t})ic(2\alpha)t}_{B_6}\\
   &\underbrace{+(\breve{h}(2\alpha,\gamma)\nabla \breve{K}(\ddot{h}(\alpha,\gamma)-\ddot{h}(2\alpha,\gamma))\cdot(\frac{(\frac{d}{d\alpha}\ddot{h})(2\alpha,\gamma)}{1+ic'(2\alpha)\gamma t})\frac{i2c(\alpha)t}{1+ic'(\alpha)\gamma t}(ic'(2\alpha)\gamma t+1))}_{B_7} \\
    &\underbrace{-\breve{h}(2\alpha,0)\nabla \breve{K}(\ddot{h}(\alpha,\gamma)-\ddot{h}(2\alpha,0))\cdot(\frac{d}{d\alpha}\ddot{h})(2\alpha,0)\frac{i2c(\alpha)t}{1+ic'(\alpha)\gamma t}}_{B_8}\\
    &+\underbrace{\frac{2ic(\alpha)t}{1+ic'(\alpha)\gamma t}\breve{h}(2\alpha,0) \nabla \breve{K}(\ddot{h}(\alpha,\gamma)-\ddot{h}(2\alpha,0))\cdot(\frac{d\ddot{h}}{d\alpha})(2\alpha,0)}_{B_9}\\
    &\underbrace{+\frac{ic(\alpha)t}{1+ic'(\alpha)\gamma t}\breve{K}(\ddot{h}(\alpha,\gamma)-\ddot{h}(-\pi,\gamma))(1+ic'(-\pi)\gamma t)\frac{(\frac{d}{d\alpha}\breve{h})(\alpha,\gamma)}{1+ic'(\alpha)\gamma t}}_{B_{10}}\\
      &\underbrace{-\frac{ic(\alpha)t}{1+ic'(\alpha)\gamma t}\breve{K}(\ddot{h}(\alpha,\gamma)-\ddot{h}(0,\gamma))(1+ic'(0)\gamma t)\frac{(\frac{d}{d\alpha}\breve{h})(\alpha,\gamma)}{1+ic'(\alpha)\gamma t}}_{B_{11}}\\  &\underbrace{-\frac{ic(\alpha)t}{1+ic'(\alpha)\gamma t}\breve{K}(\ddot{h}(\alpha,\gamma)-\ddot{h}(2\alpha,\gamma))(1+ic'(2\alpha)\gamma t)\frac{(\frac{d}{d\alpha}\breve{h})(\alpha,\gamma)}{1+ic'(\alpha)\gamma t}}_{B_{12}}\\
     &\underbrace{-\frac{ic(\alpha)t}{1+ic'(\alpha)\gamma t}\breve{K}(\ddot{h}(\alpha,\gamma)-\ddot{h}(\pi,\gamma))(1+ic'(\pi)\gamma t)\frac{(\frac{d}{d\alpha}\breve{h})(\alpha,\gamma)}{1+ic'(\alpha)\gamma t}}_{B_{13}}\\
     &\underbrace{+\breve{K}(\breve{h}(\alpha,\gamma)-\breve{h}(-\pi,\gamma))(ic(-\pi)t-\frac{ic(\alpha)t}{1+ic'(\alpha)\gamma t}(1+ic'(-\pi)\gamma t)) \frac{(\frac{d}{d\alpha}\breve{h})(\alpha,\gamma)}{1+ic'(\alpha)\gamma t}}_{B_{14}}\\
    &\underbrace{-\breve{K}(\ddot{h}(\alpha,\gamma)-\ddot{h}(0,\gamma))(ic(0)t-\frac{ic(\alpha)t}{1+ic'(\alpha)\gamma t}(1+ic'(0)\gamma t))\frac{(\frac{d}{d\alpha}\breve{h})(\alpha,\gamma)}{1+ic'(\alpha)\gamma t}}_{B_{15}}\\
    &\underbrace{+\breve{K}(\ddot{h}(\alpha,\gamma)-\ddot{h}(2\alpha,\gamma))(ic(2\alpha)t-\frac{ic(\alpha)t}{1+ic'(\alpha)\gamma t}(1+ic'(2\alpha)\gamma t)) \frac{(\frac{d}{d\alpha}\breve{h})(\alpha,\gamma)}{1+ic'(\alpha)\gamma t}}_{B_{16}}\\
    &\underbrace{-\breve{K}(\ddot{h}(\alpha,\gamma)-\ddot{h}(\pi,\gamma))(ic(\pi)t-\frac{ic(\alpha)t}{1+ic'(\alpha)\gamma t}(1+ic'(\pi)\gamma t))\frac{(\frac{d}{d\alpha}\breve{h})(\alpha,\gamma)}{1+ic'(\alpha)\gamma t}}_{B_{17}}\\
    &=\underbrace{\frac{ic(\alpha)t}{1+ic'(\alpha)\gamma t}\breve{K}(\ddot{h}(\alpha,\gamma)-\ddot{h}(0,\gamma))(\frac{(\frac{d}{d\alpha}\breve{h})(\alpha,\gamma)}{1+ic'(\alpha)\gamma t})(1+ic'(0)\gamma t)}_{C_1}\\&\underbrace{+\frac{ic(\alpha)t}{1+ic'(\alpha)\gamma t}\breve{K}(\ddot{h}(\alpha,\gamma)-\ddot{h}(2\alpha,\gamma))(\frac{(\frac{d}{d\alpha}\breve{h})(\alpha,\gamma)}{1+ic'(\alpha)\gamma t})(1+ic'(2\alpha)\gamma t)}_{C_2}\\
&\underbrace{+\frac{ic(\alpha)t}{1+ic'(\alpha)\gamma t}\breve{K}(\ddot{h}(\alpha,\gamma)-\ddot{h}(0,\gamma))(-\frac{\frac{d}{d\alpha}\breve{h}(0,\gamma)}{1+ic'(0)\gamma t})(1+ic'(0)\gamma t)}_{C_3}\\&\underbrace{+\frac{ic(\alpha)t}{1+ic'(\alpha)\gamma t}\breve{K}(\ddot{h}(\alpha,\gamma)-\ddot{h}(2\alpha,\gamma))(-\frac{(\frac{d}{d\alpha}\breve{h})(2\alpha,\gamma)}{1+ic'(2\alpha)\gamma t})(1+ic'(2\alpha)\gamma t)}_{C_4}\\
 &\underbrace{-\breve{K}(\ddot{h}(\alpha,\gamma)-\ddot{h}(0,\gamma))(\frac{ic(\alpha)t}{1+ic'(\alpha)\gamma t}-\frac{ic(0)t}{1+ic'(0)\gamma t})(1+ic'(0)\gamma t)(\frac{(\frac{d}{d\alpha}\breve{h})(\alpha,\gamma)}{1+ic'(\alpha)\gamma t})}_{C_5}\\
  &\underbrace{-\breve{K}(\ddot{h}(\alpha,\gamma)-\ddot{h}(0,\gamma))(\frac{ic(\alpha)t}{1+ic'(\alpha)\gamma t}-\frac{ic(0)t}{1+ic'(0)\gamma t})(1+ic'(0)\gamma t)(-\frac{\frac{d}{d\alpha}\breve{h}(0,\gamma)}{1+ic'(0)\gamma t})}_{C_3}\\
&\underbrace{+\breve{K}(\ddot{h}(\alpha,\gamma)-\ddot{h}(2\alpha,\gamma))(\frac{ic(\alpha)t}{1+ic'(\alpha)\gamma t}-\frac{ic(2\alpha)t}{1+ic'(2\alpha)\gamma t})(\frac{(\frac{d}{d\alpha}\breve{h})(\alpha,\gamma)}{1+ic'(\alpha)\gamma t})(1+ic'(2\alpha)\gamma t)}_{C_6}\\
&+\breve{K}(\ddot{h}(\alpha,\gamma)-\ddot{h}(2\alpha,\gamma))(\underbrace{\frac{ic(\alpha)t}{1+ic'(\alpha)\gamma t}}_{C_4}\underbrace{-\frac{ic(2\alpha)t}{1+ic'(2\alpha)\gamma t}}_{C_7})(-\frac{(\frac{d}{d\alpha}\breve{h})(2\alpha,\gamma)}{1+ic'(2\alpha)\gamma t})(1+ic'(2\alpha)\gamma t)\\
&+\breve{K}(\ddot{h}(\alpha,\gamma)-\ddot{h}(2\alpha,\gamma))(\underbrace{\frac{2ic(\alpha)t}{1+ic'(\alpha)\gamma t}}_{C_4}\underbrace{-\frac{ic(2\alpha)t}{1+ic'(2\alpha)\gamma t}}_{C_7})(\partial_{\alpha}h)(2\alpha,\gamma)\\
   &- \nabla \breve{K}(\ddot{h}(\alpha,\gamma)-\ddot{h}(2\alpha,\gamma))\cdot[(\underbrace{\frac{2ic(\alpha)t}{1+ic'(\alpha)\gamma t}}_{C_8}\underbrace{-\frac{ic(2\alpha)t}{1+ic'(2\alpha)\gamma t}}_{C_9})(\partial_{\alpha}\ddot{h})(2\alpha,\gamma)]\breve{h}(2\alpha,\gamma)\\
    &\underbrace{-\breve{h}(2\alpha,\gamma)\nabla \breve{K}(\ddot{h}(\alpha,\gamma)-\ddot{h}(2\alpha,\gamma))\cdot(\frac{(\frac{d}{d\alpha}\ddot{h})(2\alpha,\gamma)}{1+ic'(2\alpha)\gamma t})ic(2\alpha)t}_{C_9}\\
   &+\underbrace{(\breve{h}(2\alpha,\gamma)\nabla \breve{K}(\ddot{h}(\alpha,\gamma)-\ddot{h}(2\alpha,\gamma))\cdot(\frac{(\frac{d}{d\alpha}\ddot{h})(2\alpha,\gamma)}{1+ic'(2\alpha)\gamma t})\frac{i2c(\alpha)t}{1+ic'(\alpha)\gamma t}(ic'(2\alpha)\gamma t+1))}_{C_{8}}\\
    &\underbrace{-\breve{h}(2\alpha,0)\nabla \breve{K}(\ddot{h}(\alpha,\gamma)-\ddot{h}(2\alpha,0))\cdot(\frac{d}{d\alpha}\ddot{h})(2\alpha,0)\frac{i2c(\alpha)t}{1+ic'(\alpha)\gamma t}}_{C_{10}}\\
    &\underbrace{+\frac{2ic(\alpha)t}{1+ic'(\alpha)\gamma t}\breve{h}(2\alpha,0) \nabla \breve{K}(\ddot{h}(\alpha,\gamma)-\ddot{h}(2\alpha,0))\cdot(\frac{d\ddot{h}}{d\alpha})(2\alpha,0)}_{C_{10}}\\
    &\underbrace{+\frac{ic(\alpha)t}{1+ic'(\alpha)\gamma t}\breve{K}(\ddot{h}(\alpha,\gamma)-\ddot{h}(-\pi,\gamma))(1+ic'(-\pi)\gamma t)\frac{(\frac{d}{d\alpha}\breve{h})(\alpha,\gamma)}{1+ic'(\alpha)\gamma t}}_{C_{11}}\\
      &\underbrace{-\frac{ic(\alpha)t}{1+ic'(\alpha)\gamma t}\breve{K}(\ddot{h}(\alpha,\gamma)-\ddot{h}(0,\gamma))(1+ic'(0)\gamma t)\frac{(\frac{d}{d\alpha}\breve{h})(\alpha,\gamma)}{1+ic'(\alpha)\gamma t}}_{C_{1}}\\  &\underbrace{-\frac{ic(\alpha)t}{1+ic'(\alpha)\gamma t}\breve{K}(\ddot{h}(\alpha,\gamma)-\ddot{h}(2\alpha,\gamma))(1+ic'(2\alpha)\gamma t)\frac{(\frac{d}{d\alpha}\breve{h})(\alpha,\gamma)}{1+ic'(\alpha)\gamma t}}_{C_{2}}\\
     &\underbrace{-\frac{ic(\alpha)t}{1+ic'(\alpha)\gamma t}\breve{K}(\ddot{h}(\alpha,\gamma)-\ddot{h}(\pi,\gamma))(1+ic'(\pi)\gamma t)\frac{(\frac{d}{d\alpha}\breve{h})(\alpha,\gamma)}{1+ic'(\alpha)\gamma t}}_{C_{12}}\\
     &\underbrace{+\breve{K}(\ddot{h}(\alpha,\gamma)-\ddot{h}(-\pi,\gamma))(ic(-\pi)t-\frac{ic(\alpha)t}{1+ic'(\alpha)\gamma t}(1+ic'(-\pi)\gamma t)) \frac{(\frac{d}{d\alpha}\breve{h})(\alpha,\gamma)}{1+ic'(\alpha)\gamma t}}_{C_{11}}\\
    &\underbrace{-\breve{K}(\ddot{h}(\alpha,\gamma)-\ddot{h}(0,\gamma))(ic(0)t-\frac{ic(\alpha)t}{1+ic'(\alpha)\gamma t}(1+ic'(0)\gamma t))\frac{(\frac{d}{d\alpha}\breve{h})(\alpha,\gamma)}{1+ic'(\alpha)\gamma t}}_{C_{5}}\\
    &\underbrace{+\breve{K}(\ddot{h}(\alpha,\gamma)-\ddot{h}(2\alpha,\gamma))(ic(2\alpha)t-\frac{ic(\alpha)t}{1+ic'(\alpha)\gamma t}(1+ic'(2\alpha)\gamma t)) \frac{(\frac{d}{d\alpha}\breve{h})(\alpha,\gamma)}{1+ic'(\alpha)\gamma t}}_{C_{6}}\\
    &\underbrace{-\breve{K}(\ddot{h}(\alpha,\gamma)-\ddot{h}(\pi,\gamma))(ic(\pi)t-\frac{ic(\alpha)t}{1+ic'(\alpha)\gamma t}(1+ic'(\pi)\gamma t))\frac{(\frac{d}{d\alpha}\breve{h})(\alpha,\gamma)}{1+ic'(\alpha)\gamma t}}_{C_{12}}\\
    &=0.
\end{align*}
Here we use $c(\pi)=c(-\pi)=c(0)=0$.
\end{proof}

\begin{lemma} \label{forM12}
Let $\breve{K}$ be meromorphic, $\breve{h}(\alpha,\gamma) \in C_{\gamma}^{1}([-1,1],C_{\alpha}^{1}[-\pi,\pi])$. If $c(\pi)=c(-\pi)=0$, $c(\alpha)\in W^{2,\infty}$, 
\[
\breve{K}(\breve{h}(\beta,\gamma))\in C_{\gamma}^{1}([-1,1],L_{\beta}^{\infty}[-\pi,\pi]),
\]
\[
\nabla\breve{K}(\breve{h}(\beta,\gamma)) c(\beta)\in C_{\gamma}^{0}([-1,1],L_{\beta}^{\infty}[-\pi,\pi]),
\]
and 
\[
\breve{K}(\breve{h}(\beta,\gamma)) c(\beta)\in C_{\gamma}^{0}([-1,1],C_{\beta}^{1}[-\pi,\pi]),\]

then 
we have
\begin{align*}
&A(\int_{-\pi}^{\pi}\breve{K}(\breve{h}(\beta,\gamma))(1+ic'(\beta)\gamma t)d\beta)=\int_{-\pi}^{\pi} D_{h}\breve{K}(\breve{h}(\beta,\gamma))[A(\breve{h})(\beta,\gamma)](1+ic'(\beta)\gamma t)d\beta
\end{align*}
\end{lemma}
\begin{proof}
\begin{align*}
&A(\int_{-\pi}^{\pi}\breve{K}(\breve{h}(\beta,\gamma))(1+ic'(\beta)\gamma t))d\beta\\
&=-\int_{-\pi}^{\pi}\nabla \breve{K}(\breve{h}(\beta,\gamma))\cdot \frac{d\breve{h}(\beta,\gamma)}{d\gamma}(1+ic'(\beta)\gamma t)d\beta -\int_{-\pi}^{\pi}\breve{K}(\breve{h}(\beta,\gamma))ic'(\beta) t d\beta\\
&=\int_{-\pi}^{\pi}\nabla \breve{K}(\breve{h}(\beta,\gamma))\cdot[(\frac{ic(\beta)t}{1+ic'(\beta)\gamma t}\partial_{\beta}-\partial_{\gamma})\breve{h}(\beta,\gamma)](1+ic'(\beta) t) d\beta\\
&\quad-\int_{-\pi}^{\pi}\nabla \breve{K}(\breve{h}(\beta,\gamma))\cdot\partial_{\beta}\breve{h}(\beta,\gamma)ic(\beta)td\beta-\int_{-\pi}^{\pi}\breve{K}(\breve{h}(\beta,\gamma))ic'(\beta) t d\beta\\
&=\int_{-\pi}^{\pi} D_{h}\breve{K}(\breve{h}(\beta,\g))[A(\breve{h})(\beta,\gamma)](1+ic'(\beta)\gamma t)d\beta.
\end{align*}
\end{proof}
\begin{corollary}\label{forM12c}
Let $\breve{K}$, $\breve{h}(\alpha,\gamma)$, $c(\alpha)$ satisfies the conditions in lemma \ref{forM12}, $X(\alpha,\gamma) \in C_{\gamma}^{1}([-1,1],C_{\alpha}^{1}[-\pi,\pi])$, then
we have
\begin{align*}
&A(\int_{-\pi}^{\pi}\breve{K}(\breve{h}(\beta,\gamma))X(\beta,\gamma)(1+ic'(\beta)\gamma t)d\beta)\\
&=\int_{-\pi}^{\pi} D_h\breve{K}(\breve{h}(\beta,\gamma))[A(\breve{h})(\beta,\gamma)]X(\beta,\gamma)(1+ic'(\beta)\gamma t)d\beta\\
&\quad +\int_{-\pi}^{\pi}\breve{K}(\breve{h}(\beta,\gamma))A(X)(\beta,\gamma)(1+ic'(\beta)\gamma t))d\beta.
\end{align*}
\end{corollary}
\begin{proof}
We can use lemma \ref{forM12}, by letting the new $\breve{K}$ as $\breve{K}(\breve{h}(\beta,\gamma))X(\beta,\gamma)$, and new $\breve{h}$ as $(\breve{h},X).$
\end{proof}
\begin{lemma}\label{forD}
If $\breve{h} \in C^{1}_{\gamma}([-1,1],H^1_{\alpha}[0,\pi])$ with  $c(\alpha)\in W^{2,\infty}$ and $c(0)=0$, then we have
\begin{equation}
    A(D_0^{-i}(\breve{h}))=D_0^{-i}(A(\breve{h})),
\end{equation}
where $D_0^{-1}$ is the same operator as $D^{-1}$ in \eqref{D-formularnew} with $\tau(t)=0:$
\begin{align*}
&\quad D_0^{-1}(\breve{h})(\alpha,\gamma,t)=\bigg\{\begin{array}{cc}
         \int_{0}^{\alpha}(1+ic'(\alpha_1)\gamma t) \breve{h}(\alpha_1,\gamma)d\alpha_1 &  0< \alpha \leq \pi\\\nonumber
          0 & -\pi\leq\alpha\leq 0.
    \end{array}
\end{align*}
\[
D_0^{-i}=\underbrace{D_0^{-1}\circ D_0^{-1}....\circ D_0^{-1}}_{i \text{ times}}(\breve{h}).
\]
\end{lemma}
\begin{proof}
When $\alpha> 0$,  we have
\begin{align*}
    A(D_0^{-1}(\breve{h}))&=A(\int_{0}^{\alpha}\breve{h}(\alpha_1,\gamma)(1+ic'(\alpha_1)\gamma t)d\alpha_1)\\
    &=ic(\alpha)t\breve{h}(\alpha,\gamma)-\int_{0}^{\alpha}\partial_{\gamma}\breve{h}(\alpha_1,\gamma)(1+ic'(\alpha_1)\gamma t)d\alpha_1-\int_{0}^{\alpha}\breve{h}(\alpha_1,\gamma)ic'(\alpha_1)td\alpha_1\\
    &=ic(\alpha)t\breve{h}(\alpha,\gamma)+\int_{0}^{\alpha}(\frac{ic(\alpha_1)t}{1+ic'(\alpha_1)\gamma t}\partial_{\alpha_1}-\partial_{\gamma})\breve{h}(\alpha_1,\gamma)(1+ic'(\alpha_1)\gamma t))d\alpha_1\\
    &\quad-\int_{0}^{\alpha}ic(\alpha_1)t\partial_{\alpha_1}\breve{h}(\alpha_1,\gamma)+\breve{h}(\alpha_1,\gamma)ic'(\alpha_1)td\alpha_1
\end{align*}
Then we can do integration by parts to the second term and have
\begin{align*}
    A(D_0^{-1}(\breve{h}))&=ic(\alpha)t\breve{h}(\alpha,\gamma)+\int_{0}^{\alpha}(\frac{ic(\alpha_1)t}{1+ic'(\alpha_1)\gamma t}\partial_{\alpha_1}-\partial_{\gamma})\breve{h}(\alpha_1,\gamma)(1+ic'(\alpha_1)\gamma t))d\alpha_1\\
    &\quad-ic(\alpha_1)t\partial_{\alpha}\breve{h}(\alpha_1,\gamma)|_{0}^{\alpha}\\
    &=D_{0}^{-1}(A(\breve{h})).
\end{align*}

When $\alpha<0$, $D^{-1}_0(A(h))=0=A(D_0^{-1}(h))$ from definition. 

Moreover, for $\breve{h} \in C^{1}_{\gamma}([-1,1],H^1_{\alpha}[0,\pi])$, we have $A(h)\in  C^{0}_{\gamma}([-1,1],L^2_{\alpha}[0,\pi])$,  $D_0^{-1}(h)\in  C^{1}_{\gamma}([-1,1],H^2_{\alpha}[0,\pi])$. Then $D_0^{-1}(A(h)), A(D_0^{-1}(h)) \in C^{0}_{\gamma}([-1,1],H^1_{\alpha}[0,\pi]).$ Hence $D^{-1}_0(A(h))=A(D_0^{-1}(h))$.

For $i\geq 1$, we could do this process i times and get the result.
\end{proof}

\section*{Acknowledgements}
 The author sincerely thanks Charles Fefferman for introducing this problem and for all the helpful discussions. This material is partly based upon work while the author studied at Princeton University. JS has been partially supported by NSF through Grant NSF DMS-1700180, by the European Research Council through ERC-StG-852741-CAPA. JS has been partially supported by the MICINN (Spain) research grant number PID2021–125021NA–I00. JS has been partially supported by the AMS-Simons Travel Grant. 
\printbibliography 

@Preamble{
"\def\cprime{$'$} "
}

@Article{Cordoba-GomezSerrano-Zlatos:stability-shifting-muskat,
  Title                    = {A note on stability shifting for the {M}uskat problem},
  Author                   = {C{\'o}rdoba, Diego and G{\'o}mez-Serrano, Javier and Zlato{\v s}, Andrej},
  Journal                  = {Philosophical Transactions of the Royal Society of London A: Mathematical, Physical and Engineering Sciences},
  Year                     = {2015},
  Number                   = {2050},
  Pages                    = {20140278, 10},
  Volume                   = {373},

  Doi                      = {10.1098/rsta.2014.0278},
  ISSN                     = {1364-503X},
  Mrclass                  = {76S05 (35B35 35Q31 35Q35 35R35)},
  Mrnumber                 = {3393318},
  Publisher                = {The Royal Society}
}

@Article{Cordoba-GomezSerrano-Zlatos:stability-shifting-muskat-II,
  Title                    = {A note on stability shifting for the {M}uskat problem, {II}: {F}rom stable to unstable and back to stable},
  Author                   = {C\'ordoba, Diego and G\'omez-Serrano, Javier and Zlato{\v s}, Andrej},
  Journal                  = {Anal. PDE},
  Year                     = {2017},
  Number                   = {2},
  Pages                    = {367--378},
  Volume                   = {10},

  Doi                      = {10.2140/apde.2017.10.367},
  Fjournal                 = {Analysis \& PDE},
  ISSN                     = {2157-5045},
  Mrclass                  = {35Q35 (35R35 76D03 76S05)},
  Mrnumber                 = {3619874},
  Url                      = {http://dx.doi.org/10.2140/apde.2017.10.367}
}

@Article{Castro-Cordoba-Fefferman-Gancedo-LopezFernandez:rayleigh-taylor-breakdown,
  Title                    = {Rayleigh-Taylor breakdown for the {M}uskat problem with applications to water waves},
  Author                   = {Castro, {\'A}. and C{\'o}rdoba, D. and Fefferman, C. and Gancedo, F. and L{\'o}pez-Fern{\'a}ndez, M.},
  Journal                  = {Ann. of Math. (2)},
  Year                     = {2012},
  Pages                    = {909--948},
  Volume                   = {175},

  Fjournal                 = {Annals of Mathematics. Second Series}
}

@Article{Matioc:local-existence-muskat-hs,
  Title                    = {The {M}uskat problem in 2D: equivalence of formulations, well-posedness, and regularity results},
  Author                   = {Matioc, Bogdan-Vasile},
 year = {2016},
pages = {},
title = {The Muskat problem in 2D: equivalence of formulations, well-posedness, and regularity results},
volume = {12},
month = {10},
journal = {Analysis and PDE},
doi = {10.2140/apde.2019.12.281}
}

@Article{MuskatPhysics,
author = {Muskat,Morris},
title = {Two Fluid Systems in Porous Media. The Encroachment of Water into an Oil Sand},
journal = {Physics},
volume = {5},
number = {9},
pages = {250-264},
year = {1934}
}

@article{Castro-Cordoba-Fefferman-Gancede,
	doi = {10.1007/s00205-013-0616-x},
  
	url = {https://doi.org/10.1007%2Fs00205-013-0616-x},
  
	year = 2013,
	month = {apr},
  
	publisher = {Springer Science and Business Media {LLC}
},
  
	volume = {208},
  
	number = {3},
  
	pages = {805--909},
  
	author = {{\'{A}}ngel Castro and Diego C{\'{o}}rdoba and Charles Fefferman and Francisco Gancedo},
  
	title = {Breakdown of Smoothness for the Muskat Problem},
  
	journal = {Archive for Rational Mechanics and Analysis}
}

@article{GANCEDO2019552,
title = {On the Muskat problem with viscosity jump: Global in time results},
journal = {Advances in Mathematics},
volume = {345},
pages = {552-597},
year = {2019},
issn = {0001-8708},
doi = {https://doi.org/10.1016/j.aim.2019.01.017},
url = {https://www.sciencedirect.com/science/article/pii/S0001870819300428},
author = {F. Gancedo and E. Garc{\' i}a-Juárez and N. Patel and R.M. Strain}
}

@article{Ambrose2004WellposednessOT,
  title={Well-posedness of two-phase Hele–Shaw flow without surface tension},
  author={David M. Ambrose},
  journal={European Journal of Applied Mathematics},
  year={2004},
  volume={15},
  pages={597 - 607}
}

@article{muskatc12021chenquocxu,
author = {Chen, Ke and Quoc-Hung Nguyen and Xu, Yiran},
year = {2021},
month = {09},
pages = {},
title = {The Muskat problem with $\mathcal{C}^1$ data},
journal = {Transactions of the American Mathematical Society},
doi = {10.1090/tran/8559}
}

@article{NGUYEN2022108122,
title = {Global solutions for the Muskat problem in the scaling invariant Besov space $ 
\dot{B}^1_{\infty, 1}$},
journal = {Advances in Mathematics},
volume = {394},
pages = {108122},
year = {2022},
issn = {0001-8708},
doi = {https://doi.org/10.1016/j.aim.2021.108122},
url = {https://www.sciencedirect.com/science/article/pii/S0001870821005612},
author = {Huy Q. Nguyen}
}

@article{abels_matioc_2022,
title={Well-posedness of the Muskat problem in subcritical Lp-Sobolev spaces}, 
volume={33}, 
DOI={10.1017/S0956792520000480}, 
number={2}, 
journal={European Journal of Applied Mathematics}, 
publisher={Cambridge University Press}, 
author={Abels, H. and Matioc, B.-V.}, 
year={2022}, 
pages={224–266}}

@article{Alazard-Nguyennonlipshitzmuskat,
author = {Thomas Alazard and Quoc-Hung Nguyen},
title = {On the Cauchy problem for the Muskat equation with non-Lipschitz initial data},
journal = {Communications in Partial Differential Equations},
volume = {46},
number = {11},
pages = {2171-2212},
year  = {2021},
publisher = {Taylor & Francis},
doi = {10.1080/03605302.2021.1928700},

URL = { 
        https://doi.org/10.1080/03605302.2021.1928700
    
},
eprint = { 
        https://doi.org/10.1080/03605302.2021.1928700
    
}

}

@article{alazardthomascritialspacemuskat,
author = {Alazard, Thomas and Quoc-Hung Nguyen},
year = {2021},
month = {06},
pages = {},
title = {On the Cauchy Problem for the Muskat Equation. II: Critical Initial Data},
volume = {7},
journal = {Annals of PDE},
doi = {10.1007/s40818-021-00099-x}
}

@article{Alazard2020EndpointST,
  title={Endpoint Sobolev theory for the Muskat equation},
  author={Thomas Alazard and Quoc-Hung Nguyen},
  journal={arXiv: Analysis of PDEs},
  year={2020}
}

@ARTICLE{alazardomarparalinearizationmuskat,
       author = {{Alazard}, Thomas and {Lazar}, Omar},
        title = "{Paralinearization of the Muskat Equation and Application to the Cauchy Problem}",
      journal = {Archive for Rational Mechanics and Analysis},
         year = 2020,
        month = mar,
       volume = {237},
       number = {2},
        pages = {545-583},
          doi = {10.1007/s00205-020-01514-6},
archivePrefix = {arXiv},
       eprint = {1907.02138},
 primaryClass = {math.AP},
       adsurl = {https://ui.adsabs.harvard.edu/abs/2020ArRMA.237..545A}
}

@article{NguyenBenoitparadifferentialmuskat,
author = {Nguyen, Huy Q. and Pausader, Benoît},
year = {2020},
month = {07},
pages = {},
title = {A Paradifferential Approach for Well-Posedness of the Muskat Problem},
volume = {237},
journal = {Archive for Rational Mechanics and Analysis},
doi = {10.1007/s00205-020-01494-7}
}

@article{DiegoOmarmuskatexistence,
author = {Córdoba, Diego and Lazar, Omar},
year = {2018},
month = {03},
pages = {},
title = {Global well-posedness for the 2D stable Muskat problem in $H^{3/2}$},
volume = {54},
journal = {Annales scientifiques de l'École Normale Supérieure},
doi = {10.24033/asens.2483}
}

@article{Deng2016OnTT,
  title={On the Two‐Dimensional Muskat Problem with Monotone Large Initial Data},
  author={Fan Deng and Zhen Lei and Fanghua Lin},
  journal={Communications on Pure and Applied Mathematics},
  year={2016},
  volume={70}
}

@article{stephenmuskatexistence,
author = {Cameron, Stephen},
year = {2018},
month = {01},
pages = {997-1022},
title = {Global well-posedness for the two-dimensional Muskat problem with slope less than 1},
volume = {12},
journal = {Analysis and PDE},
doi = {10.2140/apde.2019.12.997}
}

@article{CONSTANTIN20171041,
title = {Global regularity for 2D Muskat equations with finite slope},
journal = {Annales de l'Institut Henri Poincaré C, Analyse non linéaire},
volume = {34},
number = {4},
pages = {1041-1074},
year = {2017},
issn = {0294-1449},
doi = {https://doi.org/10.1016/j.anihpc.2016.09.001},
url = {https://www.sciencedirect.com/science/article/pii/S0294144916300646},
author = {Peter Constantin and Francisco Gancedo and Roman Shvydkoy and Vlad Vicol}
}

@article{constantincordobagancedostrainmuskatexistence,
author = {Constantin, Peter and Córdoba, Diego and Gancedo, Francisco and Strain, Robert},
year = {2013},
month = {10},
pages = {},
title = {On the Muskat problem: Global in time results in 2D and 3D},
volume = {138},
journal = {American Journal of Mathematics},
doi = {10.1353/ajm.2016.0044}
}

@article{CHENG201632,
title = {Well-posedness of the Muskat problem with H2 initial data},
journal = {Advances in Mathematics},
volume = {286},
pages = {32-104},
year = {2016},
issn = {0001-8708},
doi = {https://doi.org/10.1016/j.aim.2015.08.026},
url = {https://www.sciencedirect.com/science/article/pii/S0001870815003357},
author = {C.H. Arthur Cheng and Rafael Granero-Belinchón and Steve Shkoller}
}

@article{CordobaCordobaGancedomuskatexistence,
 ISSN = {0003486X, 19398980},
 URL = {http://www.jstor.org/stable/29783205},
 author = {Antonio Córdoba and Diego Córdoba and Francisco Gancedo},
 journal = {Annals of Mathematics},
 number = {1},
 pages = {477--542},
 publisher = {[Annals of Mathematics, Trustees of Princeton University on Behalf of the Annals of Mathematics, Mathematics Department, Princeton University]},
 title = {Interface evolution: the Hele-Shaw and Muskat problems},
 urldate = {2022-11-06},
 volume = {173},
 year = {2011}
}

@article{Constantin2013OnTG,
  title={On the global existence for the Muskat problem},
  author={Peter Constantin and Diego C{\'o}rdoba and Francisco Gancedo and Robert M. Strain},
  journal={Journal of the European Mathematical Society},
  year={2013},
  volume={15},
  pages={201-227}
}

@article{CordobaGancedomuskatmaximum,
author = {Córdoba, Diego and Gancedo, Francisco},
year = {2007},
month = {12},
pages = {},
title = {A Maximum Principle for the Muskat Problem for Fluids with Different Densities},
volume = {286},
journal = {Communications in Mathematical Physics},
doi = {10.1007/s00220-008-0587-1}
}

@article{CordobaGancedocontourdynamicsmuskat,
author = {Córdoba, Diego and Gancedo, Francisco},
year = {2007},
month = {06},
pages = {445-471},
title = {Contour Dynamics of Incompressible 3-D Fluids in a Porous Medium with Different Densities},
volume = {273},
journal = {Communications in Mathematical Physics},
doi = {10.1007/s00220-007-0246-y}
}

@article{MichaelRusselHowisonmuskat,
author = {Siegel, Michael and Caflisch, Russel and Howison, S.},
year = {2004},
month = {10},
pages = {1374 - 1411},
title = {Global existence, singular solutions, and ill‐posedness for the Muskat problem},
volume = {57},
journal = {Communications on Pure and Applied Mathematics},
doi = {10.1002/cpa.20040}
}

@article{YI2003442,
author = {Fahuai Yi},
title = {Global classical solution of Muskat free boundary problem},
journal = {Journal of Mathematical Analysis and Applications},
volume = {288},
number = {2},
pages = {442-461},
year = {2003},
issn = {0022-247X},
doi = {https://doi.org/10.1016/j.jmaa.2003.09.003},
url = {https://www.sciencedirect.com/science/article/pii/S0022247X03006656}
}

@article{Yifahuaimuskatlocalexistence,
author = {Yi, Fahuai},
year = {2003},
month = {12},
pages = {442-461},
title = { Local classical solution of Muskat free boundary problem},
volume = {288},
journal = {Journal of Mathematical Analysis and Applications - J MATH ANAL APPL},
doi = {10.1016/j.jmaa.2003.09.003}
}

@article{muskatregulatiryJia,
author={Shi, Jia},
title={Regularity of solutions to the Muskat equation },
journal = {Arch Rational Mech Anal 247, 36 
},
year={2023},
doi=
{ https://doi.org/10.1007/s00205-023-01862-z}
}

@article{2023desingularizationGGHP,
  title={Desingularization of small moving corners for the Muskat equation},
 
author={Eduardo Garc{\'i}a-Ju{\'a}rez and Javier G{\'o}mez-Serrano and Susanna V. Haziot and Beno{\^i}t Pausader},

  journal={arXiv:2305.05046},
  year={2023}
}

@article{2021selfsimilar,
     title = {Self-similar solutions for the Muskat equation},
journal = {Advances in Mathematics},
volume = {399},
pages = {108-294},
year = {2022},
issn = {0001-8708},
doi = {https://doi.org/10.1016/j.aim.2022.108294},
url = {https://www.sciencedirect.com/science/article/pii/S0001870822001104},
author = {Eduardo García-Juárez and Javier Gómez-Serrano and Huy Q. Nguyen and Benoît Pausader}
}

@Article{shaw1898motion,
  title={On the motion of a viscous fluid between two parallel plates},
  author={Hele Shaw and H.S.},
  journal={Nature},
  volume={58},
  pages={34--36},
  year={1898}
  }

@article{cordoba2011lack,
  title={Lack of uniqueness for weak solutions of the incompressible porous media equation},
  author={Cordoba, Diego and Faraco, Daniel and Gancedo, Francisco},
  journal={Archive for rational mechanics and analysis},
  volume={200},
  pages={725--746},
  year={2011},
  publisher={Springer}
}

@inproceedings{szekelyhidi2012relaxation,
  title={Relaxation of the incompressible porous media equation},
  author={Sz{\'e}kelyhidi Jr, L{\'a}szl{\'o}},
  booktitle={Annales scientifiques de l'Ecole normale sup{\'e}rieure},
  volume={45},
  number={3},
  pages={491--509},
  year={2012}
}

@article{forster2018piecewise,
  title={Piecewise constant subsolutions for the Muskat problem},
  author={F{\"o}rster, Clemens and Sz{\'e}kelyhidi, L{\'a}szl{\'o}},
  journal={Communications in Mathematical Physics},
  volume={363},
  pages={1051--1080},
  year={2018},
  publisher={Springer}
}

@article{noisette2021mixing,
  title={Mixing solutions for the Muskat problem with variable speed},
  author={Noisette, Florent and Sz{\'e}kelyhidi, L{\'a}szl{\'o}},
  journal={Journal of Evolution Equations},
  volume={21},
  pages={3289--3312},
  year={2021},
  publisher={Springer}
}

@article{castro2022localized,
  title={Localized mixing zone for Muskat bubbles and turned interfaces},
  author={Castro, {\'A}ngel and Faraco, Daniel and Mengual, Francisco},
  journal={Annals of PDE},
  volume={8},
  number={1},
  pages={7},
  year={2022},
  publisher={Springer}
}

@unknown{castro2019degraded,
author = {Castro,  {\'A}ngel and Faraco, Daniel and Mengual, Francisco},
year = {2018},
month = {05},
pages = {},
title = {Degraded mixing solutions for the Muskat problem}
}

@article{castro2016mixing,
author = {Castro, {\'A}ngel and Córdoba, Diego and Faraco, Daniel},
year = {2021},
month = {10},
pages = {},
title = {Mixing solutions for the Muskat problem},
volume = {226},
journal = {Inventiones mathematicae},
doi = {10.1007/s00222-021-01045-1}
}

 \begin{tabular}{l}
\textbf{Jia Shi}\\
{Department of Mathematics}\\
{Massachusetts Institute of Technology} \\
{Simons Building (Building 2), Room 157}\\
{Cambridge, MA 02139, USA}\\
{e-mail: jiashi@mit.edu}\\ \\
\end{tabular}
\end{document}